\documentclass[a5paper,twoside,russian,10pt]{article}

\usepackage[intlimits]{amsmath}
\usepackage{amsthm,amsfonts}
\usepackage{amssymb}
\usepackage{mathrsfs}
\usepackage[final]{graphicx,epsfig} 
\usepackage{longtable}
\usepackage{indentfirst}
\usepackage[utf8]{inputenc}
\usepackage[T2A]{fontenc}
\usepackage[english]{babel}
\usepackage[usenames]{color}
\usepackage{esint}
\usepackage{bbm}
\usepackage{seqsplit} 
\usepackage{hyperref}

\usepackage{xfrac}
\DeclareMathOperator*{\argmax}{arg\,max}
\DeclareMathOperator*{\argmin}{arg\,min}

\usepackage{tikz}
\usepackage{ifthen}
\usepackage{subcaption}

\usepackage{caption}
\usepackage{accents}
\usepackage{watermark}

\usepackage{setspace}

\usepackage[most]{tcolorbox}
\definecolor{block-gray}{gray}{0.85}
\newtcolorbox{myquote}{colback=block-gray,grow to right by=-0mm,grow to left by=-0mm,
boxrule=0pt,boxsep=0pt,breakable}

\usepackage{wrapfig} 

\usepackage{url}
\makeatletter
\g@addto@macro{\UrlBreaks}{\UrlOrds}
\makeatother

\usepackage{tocloft}

\usepackage{wasysym}
\usepackage{verbatim}

\usepackage{enumitem}
\RequirePackage{enumitem}
\renewcommand{\alph}[1]{\asbuk{#1}} 
\setenumerate[1]{label=\alph*), fullwidth, itemindent=\parindent, 
  listparindent=\parindent} 
\setenumerate[2]{label=\arabic*), fullwidth, itemindent=\parindent, 
  listparindent=\parindent, leftmargin=\parindent}
  
\usepackage{titlesec}
\titlelabel{\thetitle.\quad}
\usepackage[dotinlabels]{titletoc}

\usepackage{lscape}

\usepackage[titletoc]{appendix}

\def \epsilon {\varepsilon}

\def  \R {\mathbb R}

\def \NN {\mathbb N}
\def \PP {\mathbb P}
\def \EE {\mathbb E}

\newcommand{\blockmatrix}[9]{
  \draw[draw=#4,fill=#5] (0,0) rectangle( #1,#2);
  \ifthenelse{\equal{#6}{true}}
  {
    \draw[draw=#7,fill=#8] (0,#2) -- (#9,#2) -- ( #1,#9) -- ( #1,0) -- ( #1 - #9,0) -- (0,#2 -#9) -- cycle;
  }
  {}
  \draw ( #1/2, #2/2) node { #3};
}

\newcommand{\rightparen}[1]{
  \begin{tikzpicture} 
    \draw (0,#1/2) arc (0:30:#1);
    \draw (0,#1/2) arc (0:-30:#1);
  \end{tikzpicture}
}

\newcommand{\leftparen}[1]{
  \begin{tikzpicture} 
    \draw (0,#1/2) arc (180:150:#1);
    \draw (0,#1/2) arc (180:210:#1);
  \end{tikzpicture}
}

\newcommand{\mblockmatrix}[4][none]{
  \begin{tikzpicture} 
  \ifthenelse{\equal{#1}{none}}
  {
    \blockmatrix{#2}{#3}{#4}{none}{none}{false}{none}{none}{0.0}
  }
  {
    \definecolor{fillcolor}{rgb}{#1}
    \blockmatrix{#2}{#3}{#4}{none}{fillcolor}{false}{none}{none}{0.0}
  }
  \end{tikzpicture}
}

\newcommand{\fblockmatrix}[4][none]{
  \begin{tikzpicture} 
  \ifthenelse{\equal{#1}{none}}
  {
    \blockmatrix{#2}{#3}{#4}{black}{none}{false}{none}{none}{0.0}
  }
  {
    \definecolor{fillcolor}{rgb}{#1}
    \blockmatrix{#2}{#3}{#4}{black}{fillcolor}{false}{none}{none}{0.0}
  }
  \end{tikzpicture}
}

\newcommand{\dblockmatrix}[4][none]{
  \begin{tikzpicture} 
  \ifthenelse{\equal{#1}{none}}
  {
    \blockmatrix{#2}{#3}{#4}{black}{none}{true}{black}{none}{0.35cm}
  }
  {
    \definecolor{fillcolor}{rgb}{#1}
    \blockmatrix{#2}{#3}{#4}{black}{none}{true}{black}{fillcolor}{0.35cm}
  }
  \end{tikzpicture}
}

\newcommand{\diagonalblockmatrix}[5][none]{
  \begin{tikzpicture} 

  \ifthenelse{\equal{#1}{none}}
  {
    \blockmatrix{#2}{#3}{#4}{black}{none}{true}{black}{none}{#5}
  }
  {
    \definecolor{fillcolor}{rgb}{#1}
    \blockmatrix{#2}{#3}{#4}{black}{none}{true}{black}{fillcolor}{#5}
  }

  \end{tikzpicture}
}

\newcommand{\valignbox}[1]{
  \vtop{\null\hbox{#1}}
}

\newenvironment{blockmatrixtabular}
{
  \begin{tabular}{
  @{}l@{}l@{}l@{}l@{}l@{}l@{}l@{}l@{}l@{}l@{}l@{}l@{}l@{}l@{}l@{}l@{}l@{}l@{}l
  @{}l@{}l@{}l@{}l@{}l@{}l@{}l@{}l@{}l@{}l@{}l@{}l@{}l@{}l@{}l@{}l@{}l@{}l@{}l
  @{}l@{}l@{}l@{}l@{}l@{}l@{}l@{}l@{}l@{}l@{}l@{}l@{}l@{}l@{}l@{}l@{}l@{}l@{}l
  @{}
  }
}
{
  \end{tabular}
}

\renewcommand{\Re}{\mathrm{Re}\,}
\renewcommand{\Im}{\mathrm{Im}\,}

\newcommand{\const}{\ensuremath{\mathop{\mathrm{const}}}\nolimits}
\newcommand{\Var}{\ensuremath{\mathrm{{\mathbb D}}}}
\newcommand{\Exp}{\ensuremath{\mathrm{{\mathbb E}}}}
\newcommand{\PR}{\ensuremath{\mathrm{{\mathbb P}}}}
\newcommand{\Be}{\ensuremath{\mathrm{Be}}}
\newcommand{\Po}{\ensuremath{\mathrm{Po}}}

\newcommand{\Bi}{\ensuremath{\mathrm{Bi}}}

\newcommand{\cov}{\ensuremath{\mathrm{cov}}}

\newcommand{\N}{\ensuremath{\mathcal{N}}}

\newcommand{\Nbb}{\mathbb{N}}
\newcommand{\Rbb}{\mathbb{R}}

\newcommand{\Ibb}{\mathbb{I}}
\newcommand{\F}{\mathcal{F}}
\newcommand{\liminmean}{\text{l.i.m.}}
\newcommand{\B}{\mathcal{B}}
\newcommand{\eps}{\varepsilon}
\newcommand\mean[1]{\accentset{\rule{0.8em}{.7pt}}{#1}}

\newtheoremstyle{theoremdd}
  {0pt}          
  {0pt}          
  {\normalfont}  
  {0pt}          
  {\bfseries}    
  {.}            
  { }            
  {\thmname{#1}\thmnumber{ #2}\thmnote{#3}}

\theoremstyle{theoremdd}

\newtheorem{theorem}{\indent Теорема}[section]
\newtheorem{example}{\indent Пример}[section]
\newtheorem{definition}{\indent Определение}[section]

\newcommand{\EndEx}{$\triangle$}
\newcommand{\EndProof}{$\square$}

\renewenvironment{proof}{{\bfseries Доказательство.}}{}
\newenvironment{solution}{{\bfseries Решение.}}{}

\renewcommand{\refname}{Литература}

\definecolor{myorange}{RGB}{255, 136, 0}

\def\gav#1{{\color{black}#1}}    
\def\shmaxg#1{{\color{black}#1}} 
\def\elena#1{{\color{black}#1}}  
\def\eduard#1{{\color{black}#1}} 
\def\egor#1{{\color{black}#1}}   
\def\ag#1{{\color{black}#1}}     
\def\eg#1{{\color{black}#1}}     

\input glyphtounicode.tex 
\input glyphtounicode-cmr.tex 
\begin{document}

\selectlanguage{russian}
\def\contentsname{Оглавление}
\def\figurename{\small Рис.}

\captionsetup[figure]{labelsep = period, font = small, 
}

\begin{titlepage}
\begin{center}
	{\small
	Министерство науки и высшего образования  Российской Федерации\\[3pt]
	
	Федеральное государственное автономное\\ 
	
	образовательное учреждение высшего образования\\
	
	<<Московский физико-технический институт\\
	
	(национальный исследовательский университет)>>
	
}

\vspace{24mm}


\vspace{7mm}

{\huge {\bfseries{\HugeЛЕКЦИИ}\\[12pt] {ПО СЛУЧАЙНЫМ\\[8pt]   ПРОЦЕССАМ\\
[25pt]}} 
  }

\vspace{9.5mm}

\begin{center}
	{\largeПод редакцией\ \ А.\,В.~Гасникова}
\end{center}

\vfill



\vspace{49mm} МОСКВА \\ МФТИ \\ 2019
\end{center}
\end{titlepage}

\thispagestyle{empty}
{\small
\begin{flushleft}
{\bf УДК 519.21(075)\\
ББК 22.317я73} \\
\hspace{27pt}{\bf Л43}
\end{flushleft}}



\begin{center}
	{\normalsize А\,в\,т\,о\,р\,ы: \\
		А.\,В.~Гасников, Э.\,А.~Горбунов, С.\,А.~Гуз, Е.\,О.~Черноусова, М.\,Г.~Широбоков,  Е.\,В.~Шульгин} 
\end{center}

\begin{center}
Рецензенты: \\ 
{\small 
Институт проблем передачи информации \\
профессор  \emph{В.\,Г.~Спокойный} }\\[3pt]
{\small 
НИУ Высшая школа экономики \\ 
доцент \emph{А.\,А.~Наумов}}
\end{center}




\vspace*{5mm}

\noindentЛ43\hspace{22pt}{\bfЛекции по случайным процессам~:} учебное пособие\,/\\ 
\hspace*{27,6pt}А.\,В.\,Гасников, Э.\,А.\,Горбунов, С.\,А.\,Гуз и др.\,; под ред.\\
\hspace*{27,6pt}А.\,В.~Гасникова. --~Москва~: МФТИ, 2019. --~\ag{285}~с.

\noindent\hspace{40pt}ISBN 978-5-7417-0710-4  



\vspace{5mm}
    

{\small
Учебное пособие содержит конспект лекций и материалы семинарских занятий по курсу <<Случайные процессы>>
, десятилетиями формировавшемуся сотрудниками кафедры математических основ управления \mbox{ФУПМ} \mbox{МФТИ} А.\,А.\;Натаном, С.\,А.\;Гузом и О.\,Г.\;Горбачевым. 


Написано на достаточно элементарном языке и рассчитано на~широкую аудитори\eg{ю} (не только студентов МФТИ). 

}


{\small
\hspace{220pt}{\bf УДК 519.21(075)}

\hspace{220pt}{\bf ББК 22.317я73}

}

\vfill

{\it\smallПечатается по решению Редакционно-издательского совета Московс\-кого физико-технического института {\rm(}национального исследовательского университета{\rm)}}


\vfill

{\scriptsize \parindent=0mm  
	
	{\bf ISBN 978-5-7417-0710-4}  \hspace{1.3em} \copyright\ Гасников~А.\,В., Горбунов~Э.\,А., Гуз~С.\,А.,
	
	
	\hspace{15.2em}\ Черноусова~Е.\,О.,Широбоков~М.\,Г., Шульгин Е\,В., 2019
	
	\hspace{13.65em} \copyright\  Федеральное государственное автономное
	
	\hspace{15.2em}\ образовательное учреждение высшего образования 
	
	\hspace{15.2em}\ <<Московский физико-технический институт
	
	\hspace{15.2em}\ (национальный исследовательский университет)>>, 2019
	
}\newpage  













\setcounter{page}{3}
\tableofcontents

\section*{\centering Список обозначений}\addcontentsline{toc}{section}{Список обозначений\vspace{-1mm}}


\noindent $( \Omega, \mathcal{F}, \PR )$ -- вероятностное пространство ($\Omega$ -- множество исходов, \linebreak $\mathcal{F}$ -- $\sigma$-алгебра, $\PR$ -- вероятностная мера). \\
$\Exp X, \Exp(X)$ --    математическое ожидание случайной величины $X$.\\
$\Var X, \Var(X)$ -- дисперсия случайной величины $X$. \\
$\mathrm{cov}(X,Y)$ -- корреляционный момент случайных величин $X$ и $Y$.\\
$\Be(p)$ -- распределение Бернулли. \\
$\Po(\lambda)$ -- распределение Пуассона. \\
$\mathrm{R}(a,b)$ -- непрерывное равномерное распределение на отрезке $[a,b]$. \\
$\mathrm{N}(\mu,\sigma^2)$ -- нормальное распределение. \\
$\mathrm{Exp}(\lambda)$ -- показательное распределение с параметром $\lambda$. \\
$\mathrm{B}(\alpha,\beta)$ -- бета-распределение. \\
$\Phi(x)$ -- функция стандартного нормального распределения $\mathrm{N}(0,1)$. \\
$\overset{d}{\longrightarrow}$  -- сходимость по распределению.\\
$\overset{p}{\longrightarrow}$, $\overset{\PR}{\longrightarrow}$ -- сходимость по вероятности. \\
$\overset{\text{п.н.}}{\longrightarrow}$ -- сходимость с вероятностью 1 (почти наверное).\\
$\overset{d}{=}$  -- равенство по распределению.\\
с.к. -- в среднем квадратичном.\\
ЗБЧ -- закон больших чисел. \\
х.ф. -- характеристическая функция. \\
ЦПТ -- центральная предельная теорема. \\
i.i.d. -- independent identically distributed random variables. \\
$\langle \cdot, \cdot\rangle$ -- скалярное произведение. \\
Индикаторная функция:
\[
\mathsf{I}(\text{true}) = [\text{true}] = 1, \quad \mathsf{I}(\text{false}) = [\text{false}] = 0.
\]
$\|x\|_2^2 = \sqrt{\sum\limits_{k=1}^n x_k^2}$, где $x\in\Rbb^n$.\\
$L_p(X,\mu)$ -- пространство интегрируемых функций по мере $\mu$ (если $\mu$  не указана, то рассматривается мера Лебега на множестве $X$). Норма в $L_p(X,\mu)$ определяется следующим образом:
\[
    \|f\|_{L_p} = \left(\int_X |f(x)|^pd\mu(x) \right)^{1/p}.
\]
$\Rbb_+$ -- неотрицательные вещественные числа. \\
\ag{$\forall$ -- квантор <<для всех>>}\\
\ag{$\exists$ -- квантор <<существует>>}


\rm{\section*{\centering\rm{\rm\small ОБЯЗАТЕЛЬНАЯ ЧАСТЬ ПРОГРАММЫ УЧЕБНОГО КУРСА}\\ «Случайные процессы»}}\addcontentsline{toc}{section}{Обязательная часть программы учебного курса\\ «Случайные процессы»\vspace{-1mm}} 

 
Определение понятия «случайный процесс». Система конечномерных распределений случайного процесса, ее свойства. Моментные функции случайного процесса. Корреляционная и взаимная корреляционная функции случайных процессов, их свойства. Преобразования случайных процессов.
Непрерывность случайного процесса в среднем квадратическом, ее необходимое и достаточное условие. Непрерывность случайного процесса по вероятности и с вероятностью единица. Производная случайного процесса в среднем квадратическом, необходимое и достаточное условие ее существования. Интеграл от случайного процесса в среднем квадратическом, необходимое и достаточное условие его существования.
Стационарный случайный процесс. Строгая и слабая стационарность случайного процесса. Взаимная стационарность случайных процессов. Эргодичность случайного процесса по математическому ожиданию в среднем квадратическом. Условия эргодичности по математическому ожиданию.
Спектральное представление стационарного случайного процесса. Теорема Хинчина о спектральном представлении корреляционной функции случайного процесса. Спектральная функция и спектральная плотность случайного процесса, их свойства и приложение. Случайный процесс типа «белый шум».
Пуассоновский случайный процесс. Сложный пуассоновский процесс.
Гауссовский (нормальный) случайный процесс, его свойства. Винеровский случайный процесс. Процессы Леви.
Марковский случайный процесс. Дискретная марковская цепь. Переходные вероятности. Уравнения Колмогорова–Чепмена. Однородные дискретные марковские цепи. Классификация состояний дискретной марковской цепи, теорема о «солидарности» их свойств.
Управляемые марковские процессы. Уравнение Вальда--Беллмана. Задача о разборчивой невесте.
Асимптотическое поведение дискретной марковской цепи. Предельное и стационарное распределения вероятностей состояний дискретной марковской цепи. Теоремы об эргодичности дискретных марковских цепей.
Марковская цепь с непрерывным аргументом. Прямое и обратное уравнения Колмогорова–Феллера. Примеры приложения теории марковских цепей (модель Эренфестов, модель ранжирования web-страниц). Концепция равновесия макросистемы на примере задачи поиска вектора Page Rank.
Непрерывный марковский процесс. Обобщенное уравнение Маркова. Уравнения Колмогорова и Колмогорова–Фоккера–Планка. 

\newpage

\section*{Введение}
\addcontentsline{toc}{section}{Введение}

Курс теории случайных процессов читается студентам на факультете управления и прикладной математики  МФТИ в весеннем семестре третьего курса. В осеннем семестре студенты изучают теорию вероятностей, в осеннем семестре четвертого курса -- математическую статистику.
Данный цикл вероятностных  дисциплин входит в обязательную программу бакалавриата ФУПМ МФТИ на протяжении нескольких десятков лет~\cite{NGG1, NGG2, NGG3}. 
Изначально в курсе случайных процессов большой акцент делался на задачах, связанных с исследованием операций~\cite{Ventsel1, Ventsel2} 
(системы массового обслуживания~\cite{GnedenkoKovalenko}, динамическое программирование~\cite{Bertsekas}, затем курс стал немного смещаться в сторону приложений к страхованию (актуарной математике~\cite{Falin1, Falin2}) 
и финансовой математике~\cite{Shiryaev2}. 
Примеры, собранные в ряде основных учебников, определенно свидетельствуют об этом~\cite{BulinskyShiryaev2005,  Rozanov1, Feller, ShiryaevT1}. 

Однако в последнее время, прежде всего в связи с бурным развитием\ag{,} анализа данных в программу обучения магистров ФУПМ стали входить:  дополнительные главы математической статистики, статистическая теория обучения и онлайн оптимизация, стохастическая оптимизация, обучение с подкреплением, моделирование интернета (случайные графы и модели их роста), стохастические дифференциальные уравнения. Все эти предметы оказались сильно завязаны на курс случайных процессов. Возникла необходимость в адаптации курса. К сожалению, большая нагрузка студентов третьего курса ФУПМ не позволила существенно изменить объем излагаемого материала. Тем не менее, некоторые акценты было решено немного изменить.

Предлагаемое пособие во многом написано на базе курса (пособия)~\cite{NGG2}, 
однако содержит больше примеров и материала, отражающего некоторые современные приложения теории случайных процессов. Такие, например, как Метод Монте-Карло для марковских цепей (Markov Chain Monte Carlo). 

Материалы, собранные 
разделах~\ref{sec:PageRank} и~\ref{bandit_problem},
разрабатывались в ходе кооперации с компаниями Яндекс и Хуавей.

Авторы хотели бы выразить благодарность своим коллегам по кафедре Математических основ управления за помощь в создании данного пособия: Олегу Геннадьевичу Горбачеву и Ольге Сергеевне Федько. Также хотели бы выразить признательность студентам МФТИ,
помогавшим с набором данного курса лекций.


Замечания и пожелания просьба присылать на адрес электронной почты кафедры математических основ управления физтех-школы прикладной математики и информатики МФТИ mou@mail.mipt.ru.

\section{Случайные процессы, примеры}
\label{random_processes}

\subsection{Базовые понятия и примеры}

Цель данного раздела -- ввести новое для читателя понятие: случайный процесс. По своей сути, случайный процесс -- это некоторое обобщение понятия случайной величины из курса Теории вероятностей. А именно, это параметризованное семейство случайных величин. Параметр может быть скалярным (случай, описываемый в этой книге), векторным (тогда говорят о \textit{случайных полях}) или даже функцией (\textit{обобщенные случайные процессы}). Вводится понятие {\it реализации случайного процесса}, его многомерной функции распределения. Появляется и качественно новое для читателя понятие -- {\it сечение случайного процесса}.

Кроме того, здесь же появляются определения самых важных с точки зрения последующего изложения процессов -- винеровского процесса, пуассоновского процесса и процесса Леви.

\begin{definition}
\textit{Случайным процессом} называется семейство случайных величин $X(\omega,t)$, ${\omega\in\Omega}$, заданных на одном вероятностном пространстве $\left(\Omega,\mathcal{F},\mathbb{P}\right)$ и зависящих от параметра $t$, принимающего значения из некоторого множества ${T\subseteq\mathbb{R}}$. Параметр $t$ обычно называют \textit{временем}.
\end{definition}

Если множество $T$ состоит из одного элемента, то случайный процесс состоит из одной случайной величины и, стало быть, является случайной величиной. Если $T$ содержит $n\ge1$ элементов, то случайный процесс представляет собой случайный вектор с $n$ компонентами. Если ${T=\mathbb{N}}$, то случайный процесс представляет собой случайную последовательность $(X_1(\omega),X_2(\omega), \dots,X_n(\omega),\dots)$. Если ${T=[a,b]}$, ${-\infty\le a < b \le +\infty}$, то такой процесс называют \textit{случайной функцией}.

\begin{definition}
При фиксированном времени ${t=t_0}$ случайная величина $X(\omega,t_0)$ называется \textit{сечением процесса} в точке $t_0$. При фиксированном исходе ${\omega=\omega_0}$ функция времени $X(\omega_0,t)$ называется \textit{траекторией {\rm (}реализацией, выборочной функцией{\rm )}} процесса.
\end{definition}

К случайному процессу всегда следует относиться как к функции двух переменных: исхода $\omega$ и времени $t$. Это независимые переменные. Заметим, впрочем, что в записи случайного процесса $X(\omega,t)$ аргумент $\omega$ обычно опускают и пишут просто $X(t)$. Точно так же поступают и со случайными величинами -- вместо $X(\omega)$ пишут $X$. Надо помнить, однако, что во всех случаях зависимость от $\omega$ подразумевается. Записи $X(\omega,t)$ или $X(\omega)$ используются в тех случаях, когда необходимо явно подчеркнуть функциональную зависимость случайного процесса или случайной величины от исхода $\omega$. Иногда для случайного процесса используют обозначение ${\{X(t),t\in T\}}$, чтобы явно подчеркнуть множество времени $T$, на котором определен процесс.

\begin{example}\label{ex:Ex3.1}
Рассмотрим процесс ${\xi(\omega,t)=t\cdot\eta(\omega)}$, где ${\eta\in \mathrm{R}(0,1)}, \\  {t\in[0,1]=T}$. В качестве вероятностного пространства здесь рассмотрим тройку $\left(\Omega,\mathcal{F},\mathbb{P}\right)$, где пространство исходов ${\Omega=[0,1]}$, $\sigma$-алгебра является борелевской ${\mathcal{F} = \mathcal{B(\mathbb{R})}}$, а вероятностная мера совпадает с мерой Лебега на отрезке $[0,1]$. В этом случае случайная величина ${\eta(\omega)=\omega}$, а случайный процесс $\xi(\omega,t)=t\cdot\omega$. Тогда при фиксированном ${\omega_0\in[0,1]}$ траектория представляет собой линейную функцию ${\xi(\omega_0,t)=\omega_0 t}$, принимающую значение ноль в момент ${t=0}$ и значение $\omega_0$ в момент $t=1$. Если же зафиксировать ${t_0\in(0,1]}$, то получится случайная величина ${\xi(\omega,t_0)=t_0\omega}$, равномерно распределенная на отрезке $[0,t_0]$, т.е. ${\xi(\omega,t_0)\in \mathrm{R}(0,t_0)}$. При  ${t_0=0}$ получается вырожденная случайная величина: $\xi(\omega,0)=0$ при всех $\omega\in[0,1]$. \EndEx
\end{example}

Мы обращаем внимание читателя на то, что в данном примере траектория процесса однозначно восстанавливается по ее части: достаточно узнать наклон траектории на любом интервале времени. На практике же более типичная ситуация -- это когда предыстория не позволяет однозначно восстановить траекторию процесса и узнать ее будущее. В~таких случаях различным исходам $\omega_1$ и $\omega_2$ может соответствовать одна и та же наблюдаемая до текущего момента времени предыстория, но будущее этих траекторий может отличаться.

В приложениях случайные величины часто задаются не как явные функции исходов, а через вероятностное распределение. Вероятностное распределение случайных величин однозначно определяется функцией распределения. Случайные векторы также имеют функцию распределения: $$F_X(x_1,\dots,x_n)=\mathbb{P}(X_1<x_1,\dots,X_n<x_n)$$ для случайного вектора ${X=(X_1,\dots,X_n)}$ с $n$ компонентами. Определим теперь функцию распределения для случайного процесса.

Пусть даны $n \ge 1$ сечений $X(t_1)$, $\dots$, $X(t_n)$ случайного процесса $\{X(t)$, $t\in T\}$; ${t_1,\dots,t_n\in T}$. Так как по определению случайного процесса сечения заданы на одном вероятностном пространстве, то они образуют случайный вектор $(X(t_1),\dots,X(t_n))$. Функция распределения такого вектора
\begin{equation}\label{eq:nDimCDF}
F_X(x_1,\dots,x_n;t_1,\dots,t_n)=\mathbb{P}(X(t_1)<x_1,\dots,X(t_n)<x_n)
\end{equation}
зависит от $n$ и параметров $t_1$, $\dots$, $t_n$ из множества $T$.

\begin{definition}
Функция распределения из~\eqref{eq:nDimCDF} называется $n$\textit{-мер\-ной функцией распределения случайного процесса} $X(t)$. Совокупность функ\-ций рас\-пре\-де\-ле\-ния~\eqref{eq:nDimCDF} для различных ${n\ge 1}$ и всех возможных моментов времени $t_1$, $\dots$, $t_n$ из множества $T$ называется \textit{семейством конечномерных распределений случайного процесса} $X(t)$.
\end{definition}

\begin{example}
Рассмотрим случайный процесс ${\xi(t)=t\eta}$, ${t\in(0,1)}$, где ${\eta\in \mathrm{R}[0,1]}$. Для произвольных ${n\ge 1}$ и ${t_1,\dots,t_n\in(0,1)}$ найдем его $n$-мерную функцию распределения:
$$F_{\xi}(x_1,\dots,x_n;t_1,\dots,t_n) = \mathbb{P}(\xi(t_1)<x_1,\dots,\xi(t_n)<x_n)=$$
$$=\mathbb{P}(t_1\eta<x_1,\dots,t_n\eta<x_n)=\mathbb{P}\left( \eta < \frac{x_1}{t_1},\dots,\eta<\frac{x_n}{t_n} \right)=$$
$$=\mathbb{P}\left( \eta < \min\limits_{1\le k \le n}\frac{x_k}{t_k} \right)=F_{\eta}\left( \min\limits_{1\le k \le n}\frac{x_k}{t_k} \right),$$
где функция распределения случайной величины $\eta$ равна $$F_{\eta}(y)=\left\{ {\begin{array}{*{20}{l}}
	0,&y\le 0, \\ 
	y,&y\in(0,1], \\
	1,&y> 1.
	\end{array}} \right. \text{ \EndEx}$$ 
\end{example}
Семейство конечномерных распределений -- это основная характеристика вероятностных свойств случайного процесса. Мы будем считать, что случайный процесс задан, если задано его семейство конечномерных распределений или если явно задано вероятностное пространство и соответствующая функция двух переменных -- исхода и времени.
\begin{example}\label{ex:SimpProb3Sec}
На вероятностном пространстве
$\left([0,1], \mathcal{B}_{[0,1]}, \lambda_{[0,1]}\right)$, состоящем из пространства элементарных исходов $[0,1]$, борелевской $\sigma$-алгебры $\mathcal{B}_{[0,1]}$ подмножеств множества $[0,1]$ и меры Лебега $\lambda_{[0,1]}$, определен случайный процесс $$\xi(\omega,t)=\left\{ {\begin{array}{*{20}{l}}
	1,& t \le \omega, \\ 
	0,& t > \omega. \\
	\end{array}} \right.$$ 
Время определено на множестве ${T=[0,1]}$. Найти траектории, сечения и двумерное распределение процесса. Исследовать сечения на попарную независимость.
\end{example}

\begin{solution}
Чтобы определить траектории процесса, зафиксируем произвольный исход $\omega_0\in[0,1]$ и изобразим на графике зависимость $\xi(\omega_0,t)$, см. рис.~\ref{fig:fig1}. По условию задачи ${\xi(\omega_0,t)=1}$ при ${t\le\omega_0}$, а при ${t\in(\omega_0,1]}$ получаем $\xi(\omega_0,t)=0$.

\begin{figure}[ht]
	\centering
	\includegraphics[scale=0.20]{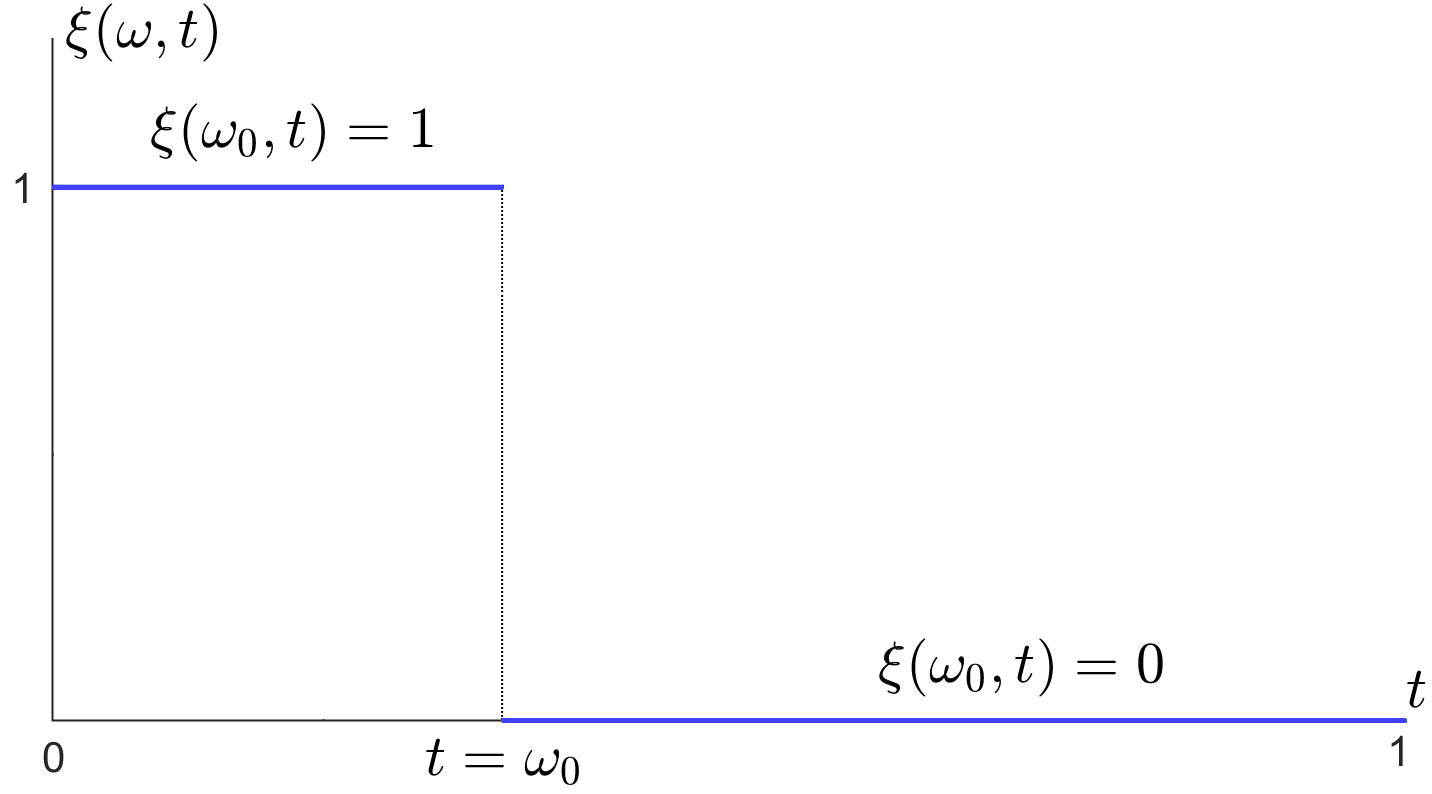}
	\caption{Типичная траектория процесса $\xi(\omega,t)$ из примера~\ref{ex:SimpProb3Sec}}
	\label{fig:fig1}
\end{figure}

Чтобы определить сечения процесса, зафиксируем время $t_0\in[0,1]$. Случайная величина $\xi(\omega,t_0)$ может принимать только два значения: 0 и 1. Следовательно, она имеет распределение Бернулли $\mathrm{Be}(p)$ с некоторым параметром $p$. Напоминаем, что $p$ есть вероятность принятия случайной величиной значения 1. Но случайная величина $\xi(\omega,t_0)$ принимает значение 1 только при ${t_0\le\omega\le 1}$. Вероятностная мера на отрезке $[0,1]$ по условию задачи является мерой Лебега, причем такой, что $\lambda_{[0,1]}([a,b])=b-a$ для любого отрезка $[a,b]\subset[0,1]$. Значит, искомая вероятность $$p=\mathbb{P}(\{\omega:\xi(\omega,t_0)=1\})=\mathbb{P}(\{\omega:t_0\le\omega\le 1\})=1-t_0,$$ поэтому произвольное сечение $\xi(t)\in\mathrm{Be}(1-t)$, $t\in[0,1]$.

Запишем по определению формулу двумерной функции распределения: $$F_{\xi}(x_1,x_2;t_1,t_2)=\mathbb{P}(\xi(t_1)<x_1,\xi(t_2)<x_2).$$ Вычислим эту вероятность, рассмотрев все возможные значения $x_1$ и~$x_2$, считая, что $t_1$ и $t_2$ принадлежат отрезку ${T=[0,1]}$. Начнем с~крайних случаев, когда хотя бы один из $x_1$ или $x_2$ превысит единицу или меньше нуля. Анализируя по отдельности случаи, можно прийти к~выражению
$$F_{\xi}(x_1,x_2;t_1,t_2) = \left\{ {\begin{array}{*{20}{l}}
	1,&x_1>1\text{ и }x_2>1, \\ 
	0,&x_1\le 0\text{ или }x_2\le 0, \\
	t_2,&x_1>1\text{ и }x_2\in(0,1],\\
	t_1,&x_2>1\text{ и }x_1\in(0,1],\\
	\min(t_1,t_2),&x_1\in(0,1]\text{ и }x_2\in(0,1].
	\end{array}} \right.$$
Здесь мы учли, что сечение $\xi(t)$ принимает значение 1 с вероятностью ${1-t}$ и значение 0 с вероятностью $t$. При подсчете первых четырех выражений достаточно было информации о том, что сечения $\xi(t)$ имеют распределение Бернулли. При вычислении последнего выражения этой информации недостаточно, требуется воспользоваться конкретным видом функции $\xi(\omega,t)$:
\begin{gather*}
\begin{split}
\mathbb{P}(\xi(t_1)<x_1,\xi(t_2)<x_2)&=\mathbb{P}(\{\omega:\xi(\omega,t_1)=0,\xi(\omega,t_2)=0\}),\\&=\mathbb{P}(\{\omega:0\le\omega<t_1,0\le\omega<t_2\}),\\&=\mathbb{P}(\{\omega:0\le\omega<\min(t_1,t_2)\}),\\&=\min(t_1,t_2).
\end{split}
\end{gather*}
Теперь исследуем произвольные два сечения $\xi(t_1)$ и $\xi(t_2)$ процесса на независимость. Для этого проверим наличие равенства в выражении
\begin{equation*}
\mathbb{P}(\xi(t_1)<x_1,\xi(t_2)<x_2)=\mathbb{P}(\xi(t_1)<x_1)\,\mathbb{P}(\xi(t_2)<x_2)
\end{equation*}
для всех $x_1$ и $x_2$ из $\mathbb{R}$. Если ${x_1\notin(0,1]}$ или $x_2\notin(0,1]$, то это равенство очевидно выполнено, причем для любых $t_1$ и $t_2$. При ${x_1\in(0,1]}$ и~${x_2\in(0,1]}$ это равенство равносильно
\begin{equation*}
\min(t_1,t_2)=t_1t_2,
\end{equation*}
что верно лишь в случаях, когда хотя бы одно из значений $t_1$ или $t_2$ равно нулю или единице. В таких случаях $\xi(t_1)$ и $\xi(t_2)$ являются независимыми случайными величинами. В остальных случаях, т.е. когда $t_1\ne0,1$ и $t_2\ne0,1$, сечения $\xi(t_1)$ и $\xi(t_2)$ являются зависимыми случайными величинами. \EndEx
\end{solution}

\begin{example}
Пусть дан тот же самый процесс, что и в предыдущем примере, и мы можем его наблюдать. Пусть на отрезке времени $[0,t_0]$ значение процесса все еще равно единице. С какой вероятностью скачок до нуля произойдет на интервале времени $[t_0,t_0+\Delta t]$, где $\Delta t < 1-t_0$?
\end{example}

\begin{solution}
Нам доступна информация о том, что на отрезке $[0,t_0]$ траектория процесса ${\xi(\omega,t_0)=1}$. Ясно, что скачок может произойти почти сразу после $t_0$ или позже. Таким образом, данная нам предыстория процесса не позволяет однозначно определить поведение траектории в будущем. Однако \textit{вероятность} скачка на наперед заданном интервале времени можно рассчитать точно. Воспользуемся для этого аппаратом условных вероятностей. Пусть событие $A$ состоит в том, что скачок произойдет в интервале времени $[t_0,t_0+\Delta t]$, а событие $B$ состоит в том, что на интервале времени $[0,t_0]$ скачка нет. Тогда по условию задачи необходимо вычислить условную вероятность $$\mathbb{P}(A \, | \, B)=\frac{\mathbb{P}(A\cap B)}{\mathbb{P}(B)}.$$ Так как из события $A$ следует событие $B$, то ${A\cap B=A}$. Далее, $$\mathbb{P}(A)=\mathbb{P}(\{\omega:t_0\le\omega\le t_0+\Delta t\})=\Delta t,$$ $$\mathbb{P}(B)=\mathbb{P}(\{\omega:t_0\le\omega\le 1\})=1-t_0.$$ Отсюда ответ: $$\mathbb{P}(A \, | \, B)=\frac{\Delta t}{1-t_0}.$$ Это и есть вероятность того, что скачок произойдет в ближайшее время $\Delta t$ при условии, что до момента $t_0$ скачка не было. Отметим, что если скачок до момента $t_0$ произошел, то будущее процесса определено однозначно -- его значение так и останется равным нулю. \EndEx
\end{solution}

Произвольная $n$-мерная функция распределения случайного процесса $\xi(t)$ обладает следующими свойствами:
\begin{enumerate}
    \item $0 \le F_{\xi}(x_1,\dots,x_n;t_1,\dots,t_n)\le 1$.
    \item Функции $F_{\xi}(x_1,\dots,x_n;t_1,\dots,t_n)$ непрерывны слева по каждой переменной $x_i$.
    \item Если хотя бы одна из переменных ${x_i\to-\infty}$, то $$F_{\xi}(x_1,\dots,x_n;t_1,\dots,t_n)\to0,$$
    если же все переменные $x_i\to+\infty$, то $$F_{\xi}(x_1,\dots,x_n;t_1,\dots,t_n)\to1.$$
    \item Функции $F_{\xi}(x_1,\dots,x_n;t_1,\dots,t_n)$ монотонны в следующем смысле:
    $$\Delta_1\dots\Delta_n F_{\xi}(x_1,\dots,x_n;t_1,\dots,t_n)\ge0,$$ где $\Delta_i$ -- оператор конечной разности по переменной $x_i$
    $$\Delta_i F = F(x_1,\dots,x_{i-1},x_i+h_i,x_{i+1},\dots,x_n;t_1,\dots,t_n)-$$
    $$-F(x_1,\dots,x_{i-1},x_i,x_{i+1},\dots,x_n;t_1,\dots,t_n)$$
    а $h_1\ge0$, \dots, $h_n\ge0$ произвольны.
    \item Для любой перестановки $\{k_1,\dots,k_n\}$ индексов $\{1,\dots,n\}$
    $$F_{\xi}(x_1,\dots,x_n;t_1,\dots,t_n)=F_{\xi}(x_{k_1},\dots,x_{k_n};t_{k_1},\dots,t_{k_n}).$$
    \item Для любых $1 \le k < n$ и $x_1,\dots,x_k\in\mathbb{R}$
    $$F_{\xi}(x_1,\dots,x_k;t_1,\dots,t_k)=F_{\xi}(x_1,\dots,x_k,+\infty,\dots,+\infty;t_1,\dots,t_n).$$
\end{enumerate}

Пусть дано некоторое семейство функций  $F(x_1,\dots, x_n; t_1, \dots, t_n)$, обладающее свойствами а)--е). Всегда ли можно подобрать случайный процесс с соответствующими $n$-мерными функциями распределения? Положительный ответ на этот вопрос дает теорема Колмогорова.

\begin{theorem}[ (Колмогоров)]
{\it Пусть задано некоторое семейство функций 
$$F=\{F(x_1,\dots,x_n;t_1,\dots,t_n), \ x_i\in\mathbb{R}, \ t_i\in T, \ i=1,\dots,n, \ n\ge 1\},$$
удовлетворяющих условиям а{\rm)}--е{\rm)}. Тогда существуют вероятностное пространство $(\Omega,\mathcal{F},\mathbb{P})$ и случайный процесс $\{\xi(t),t\in T\}$, \shmaxg{определенный на этом вероятностном пространстве}, такие, что семейство конечномерных распределений $F_{\xi}$ случайного процесса $\xi(t)$ совпадает с $F$.}
\end{theorem}

Напомним читателю, что функция распределения не задает случайную величину однозначно как функцию: может оказаться, что на одном и том же вероятностном пространстве существуют две случайные величины $\xi(\omega)$ и $\eta(\omega)$ с совпадающими функциями распределения. Точно так же семейство конечномерных распределений может неоднозначно определять случайный процесс. Это нас приводит к необходимости ввести понятие \textit{эквивалентности случайных процессов}.

\begin{definition}
Пусть ${\{\xi(t),t\in T\}}$ и ${\{\eta(t),t\in T\}}$ -- два случайных процесса, определенные на одном и том же вероятностном пространстве $(\Omega,\mathcal{F},\mathbb{P})$ и принимающие значения в одном и том же измеримом пространстве. Если для любого $t\in T$ $$\mathbb{P}(\xi(t)=\eta(t))=1,$$ то эти процессы называются \textit{стохастически эквивалентными} или \\ просто \textit{эквивалентными}. При этом $\xi(t)$ называется \textit{модификацией} \\ процесса $\eta(t)$, а $\eta(t)$ -- \textit{модификацией процесса} $\xi(t)$.
\end{definition}

\shmaxg{Легко показать, что конечномерные распределения эквивалентных процессов совпадают}. В теории случайных процессов понятие эквивалентности процессов чрезвычайно важно, так как позволяет обосновать возможность расчета вероятности каких-нибудь событий, связанных с процессом, переходом к другому, эквивалентному, процессу. Например, пусть на вероятностном пространстве $(\Omega,\mathcal{F},\mathbb{P})$ дан какой-нибудь процесс $\{\xi(t),t\ge0\}$ с непрерывным временем и непрерывным множеством значений на $\mathbb{R}$ и допустим, что нам требуется вычислить следующую вероятность: $$\mathbb{P}\left(\sup\limits_{t\in[0,1]}\xi(t) \le x\right).$$ Вероятность определена лишь для элементов сигма-алгебры $\mathcal{F}$, но принадлежит ли \shmaxg{событие, записанное под вероятностью,} этой сигма-алгеб\-ре? Супремум процесса не превосходит $x$ на отрезке $[0,1]$ тогда и только тогда, когда в каждый момент времени ${t\in[0,1]}$ выполнено ${\xi(t)\le x}$. И хотя \shmaxg{событие} $\{\xi(t)\le x\}$ является элементом сигма-алгебры $\mathcal{F}$, несчетное пересечение таких событий для разных ${t\in[0,1]}$ \shmaxg{может не быть} элементом сигма-алгебры. Однако если все или почти все траектории процесса $\xi(t)$ непрерывны на $[0,1]$, то супремум по $[0,1]$ можно заменить супремумом по рациональным точкам отрезка $[0,1]$, а это счетное пересечение событий сигма-алгебры. На самом деле траектории этого процесса не обязаны быть непрерывными, достаточно лишь, чтобы у процесса существовала его непрерывная модификация (модификация с почти всеми непрерывными траекториями). Установить существование непрерывной модификации процесса помогает следующая теорема.

\begin{theorem}[ (Колмогоров~\cite{VentselAD})] 
\label{Kolmog cont modific}
\textit{Пусть $\{\xi(t)$, $t\in T\}$ -- случайный процесс и $T=[a,b]$. Если существуют $\alpha>0$, $\beta>0$, $c<\infty$, такие, что при всех $t$, $t+h\in[a,b]$ выполнено неравенство $$\mathbb{E} \left|\xi(t+h)-\xi(t) \right|^\alpha \le c|h|^{1+\beta},$$ то $\xi(t)$ имеет непрерывную модификацию.}
\end{theorem}

Если непрерывная модификация процесса $\xi(t)$ существует, то всегда можно так определить вероятностное пространство и сигму-алгеб\-ру $\mathcal{F}$ (с сохранением семейства конечномерных распределений), на которых задан этот процесс, что условия на процесс в несчетном числе точек можно заменить на условия в счетном числе точек, и множества типа $\{\sup_{t\in[0,1]}\xi(t) \le x\}$ становятся измеримыми, вероятность их определена.



Если непрерывной модификации процесса не существует или по каким-то причинам трудно установить ее существование, то на помощь может прийти следующее понятие.

\begin{definition}
Случайный процесс ${\{\xi(t),t\in T\}}$ называется \textit{сепарабельным}, если множество $T$ содержит всюду плотное счетное подмножество $S$, такое, что для любого интервала $I\subset T$ справедливо
$$\mathbb{P}\left( \sup\limits_{t\in I \cap S} \xi(t) = \sup\limits_{t\in I}\xi(t), \ \inf\limits_{t\in I \cap S} \xi(t) = \inf\limits_{t\in I}\xi(t) \right) = 1.$$
\end{definition}

Оказывается, что любой процесс с непрерывным временем имеет сепарабельную модификацию, и всегда можно построить вероятностное пространство и сигму-алгебру, на которых определен случайный процесс, таким образом, что множества, связанные со значениями процесса в несчетном числе точек, станут измеримыми. Впрочем, этот серьезный математический факт остается во многих случаях чисто теоретическим -- для практических расчетов все равно потребуется знать множество сепарабельности $S$, которое следует дополнительно искать. Ситуация здесь упрощается в случае, когда исходный процесс $\xi(t)$ является стохастически непрерывным.

\begin{definition}
Случайный процесс $\{\xi(t),t\in T\}$ называется \textit{стохастически непрерывным}, если в каждый момент времени $t$ $$\forall\varepsilon > 0 \ \lim\limits_{h\to0}\mathbb{P}(|\xi(t+h)-\xi(t)|>\varepsilon)=0.$$
\end{definition}

Оказывается, что если случайный процесс является стохастически непрерывным, то в качестве множества сепарабельности $S$ его сепарабельной модификации можно брать любое всюду плотное счетное подмножество $S$, например, множество рациональных чисел, поэтому поиском такого множества в этом случае заниматься не приходится.

Несколько слов добавим вообще про способы задания случайного процесса. С практической точки зрения удобно задавать и определять процессы через семейства конечномерных распределений. Именно такой подход и принят в настоящем пособии. Однако задание семейства конечномерных распределений неоднозначно определяет случайный процесс как функцию двух переменных; так же, как и функция распределения неоднозначно определяет случайную величину как функцию исхода. В литературе можно встретить случаи, когда вместе с процессом определяется вероятностное пространство, которому процесс принадлежит. Пространством исходов такого вероятностного пространства часто выступает какое-то множество функций, которому принадлежат траектории процесса. Например, это может быть множество всех функций, заданных на множестве времени $T$ (для этого обычно используется обозначение $\mathbb{R}^T$), или множество непрерывных функций на $T$ (обозначение может быть, например, $C(T)$). Кроме того, много внимания уделяется сигма-алгебрам на таких пространствах исходов. Обычно сигма-алгебра определяется как минимальная сигма-алгебра, которая содержит какой-то набор более простых множеств. Например, в качестве такого набора может выступать набор \textit{цилиндрических множеств}: множеств вида $$\{X(t_1)\in B_1,\dots,X(t_n)\in B_n\},$$ где $B_1,\dots,B_n$ -- борелевские множества из $\mathbb{R}$. Название оправдано следующей геометрической аналогией: если $T=\{1,2,3\}$, то для случайного вектора $(X_1,X_2,X_3)$ множество ${C=\{X_1\in B_1, X_2 \in B_2\}}$ представляет собой цилиндр с образующей, параллельной третьей оси. Вероятности на цилиндрических множествах определяют вероятность произвольных множеств из сигма-алгебры ими порожденной. Причем задание вероятности на цилиндрических множествах равносильно заданию семейства конечномерных распределений.

\subsection{Моментные характеристики процессов}

Теперь, когда все формальные сведения о распределении случайных процессов даны, перейдем к их моментным характеристикам. Мы дадим определение математическому ожиданию, а также корреляционной и ковариационной функции случайного процесса. Когда мы будем давать эти определения, мы будем предполагать, что дан случайный процесс $\{X(t),t\in T\}$, каждое сечение которого на $T$ имеет конечный второй момент ${\mathbb{E}X^2(t) < \infty}$. Такие процессы называются \textit{случайными процессами второго порядка}, а множество таких процессов обозначают $L_2$.

\begin{definition}
\textit{Математическим ожиданием случайного процесса} ${X(t)}$ называется функция ${m_X: T \to \Rbb}$, значение которой в каждый момент времени равно математическому ожиданию соответствующего сечения, т.е. $\forall t\in T \ m_X(t) = \mathbb{E} X(t)$.
\end{definition}

\begin{definition}
\textit{Корреляционной функцией случайного процесса} $X(t)$ называется функция двух переменных ${R_X: T\times T \to \Rbb}$, которая каждой паре моментов времени сопоставляет корреляционный момент соответствующих сечений процесса, т.е. $$R_X(t_1,t_2) = \Exp \overset{\circ}{X}(t_1)\overset{\circ}{X}(t_2)=\mathbb{E}X(t_1)X(t_2) - \mathbb{E}X(t_1)\mathbb{E}X(t_2).$$
\end{definition}

\textbf{Замечание}. \textit{Корреляционной} эту функцию называют также в  уче\-бнике~\cite{NGG2}. 
В некоторых других источниках, например -- в книге~\cite{MillerPankov}, 
ее называют \textit{ковариационной}. В учебнике~\cite{NGG2} 
тоже есть понятие \textit{ковариационной} функции: ${K_{X}(t_1,t_2)=\mathbb{E}X(t_1)X(t_2)}$. В настоящем пособии мы будем придерживаться терминологии учебника~\cite{NGG2}. 

\begin{definition}
\textit{Ковариационной функцией случайного процесса} ${X(t)}$ называется функция двух переменных ${K_X: T\times T \to \Rbb}$, такая, что $$ K_X(t_1,t_2) = \Exp X(t_1)X(t_2) .$$
\end{definition}

Отметим, что корреляционная и ковариационная функции связаны соотношением ${K_X(t_1,t_2)=R_X(t_1,t_2)+ m_X(t_1)m_X(t_2)}$.

\begin{definition}
\textit{Дисперсией случайного процесса} ${X(t)}$ называется функция ${D_X: T \to \Rbb}$, в каждый момент времени равная дисперсии соответствующего сечения процесса, т.е. $$D_X(t) = \Exp \overset{\circ}{X}{}^2(t)=\mathbb{E}X^2(t) - (\mathbb{E}X(t))^2.$$ Легко видеть, что дисперсия процесса связана с корреляционной функцией по формуле $D_{X}(t)=R_X(t,t)$.
\end{definition}

Введем также важное понятие взаимной корреляционной функции двух процессов.

\begin{definition}
\textit{Взаимной корреляционной функцией} процессов ${X(t)}$ и ${Y(t)}$ из $L_2$ называется функция двух переменных $R_{X,Y}:T\times T\to\mathbb{R}$ такая, что $$R_{X,Y}(t_1,t_2)=\mathbb{E}\overset{\circ}{X}(t_1)\overset{\circ}{Y}(t_2).$$
\end{definition}

\begin{definition}
\textit{Характеристической функцией} процесса \\ $X(t)$ называется функция 
$$\varphi_{X(t)}(s)=\mathbb{E}\exp(isX(t)), \ s\in\mathbb{R}.$$
\end{definition}

Рассмотрим несколько задач на расчет введенных численных характеристик случайных процессов.

\begin{example}\label{ex:Ex6.1}
Вычислить математическое ожидание, корреляционную функцию и дисперсию случайного процесса ${\xi(t)=X\cos{(t+Y)}}$, где ${t\in\mathbb{R}}$, а случайные величины ${X \in \mathrm{N}(0,1)}$ и ${Y\in \mathrm{R}(-\pi,\pi)}$ независимы.
\end{example}

\begin{solution}
Сначала найдем математическое ожидание $\xi(t)$. Для этого зафиксируем момент времени $t\in\mathbb{R}$ и вычислим математическое ожидание сечения $$\mathbb{E}\xi(t)=\mathbb{E} \left(X\cos{(t+Y)}\right).$$ По условию задачи случайные величины $X$ и $Y$ независимы, значит, $$\mathbb{E}\left(X\cos{(t+Y)}\right) = \mathbb{E}X\cdot\mathbb{E}\cos{(t+Y)}.$$ Так как ${X\in\mathrm{N}(0,1)}$, то ${\mathbb{E}X=0}$, откуда ${\mathbb{E}\xi(t)=0}$ для всех $t\in\mathbb{R}$. Найдем теперь корреляционную функцию $$R_{\xi}(t_1,t_2)=\mathbb{E}\xi(t_1)\xi(t_2) - \mathbb{E}\xi(t_1)\mathbb{E}\xi(t_2).$$ Из $\mathbb{E}\xi(t)=0$ для всех $t\in\mathbb{R}$ следует, что $$R_{\xi}(t_1,t_2)=\mathbb{E} \xi(t_1)\xi(t_2)=\mathbb{E} \left(X^2\cos{(t_1+Y)}\cos{(t_2+Y)} \right).$$ Снова воспользуемся независимостью $X$ и $Y$, и получим $$R_{\xi}(t_1,t_2)=\mathbb{E}X^2\cdot\mathbb{E} \left(\cos{(t_1+Y)}\cos{(t_2+Y)}\right).$$ Второй момент посчитаем через дисперсию: $$\mathbb{E}X^2=\mathbb{D}X + (\mathbb{E}X)^2 = 1.$$ Второй множитель вычислим напрямую:
$$\mathbb{E}\cos{(t_1+Y)}\cos{(t_2+Y)}=\int\limits_{-\infty}^{+\infty}\cos{(t_1+y)}\cos{(t_2+y)}f_{Y}(y)dy,$$ где функция плотности равномерного распределения $$f_{Y}(y)=\left\{ {\begin{array}{*{20}{l}}
	1/(2\pi),& y\in(-\pi,\pi), \\ 
	0,& y\notin(-\pi,\pi). \\
	\end{array}} \right.$$
Отсюда получаем $$R_{\xi}(t_1,t_2)=\frac{1}{2\pi}\int\limits_{-\pi}^{\pi}\cos{(t_1+y)}\cos{(t_2+y)}dy=\frac{1}{2}\cos{(t_1-t_2)}.$$
Дисперсию можно вычислить по формуле $\mathbb{D}\xi(t)=R_{\xi}(t,t)=1/2$.
Итак, в результате имеем ${m_{\xi}(t)\equiv 0}$, ${D_{\xi}(t)\equiv 1/2}$,  ${R_{\xi}(t_1,t_2)=\frac{1}{2} \cos{(t_1-t_2)}}$. \EndEx
\end{solution}

\begin{example}\label{ex:Ex6.3}
Найти математическое ожидание, дисперсию и корреляционную функцию случайного процесса из примера~\ref{ex:SimpProb3Sec}.
\end{example}

\begin{solution}
Мы выяснили, что любое сечение ${\xi(t) \in \mathrm{Be}(1-t)}$. \shmaxg{Значит, математическое ожидание и дисперсия равны
$$m_{\xi}(t)=1-t, \ D_{\xi}(t)=t(1-t).$$}
Корреляционная функция равна $$R_{\xi}(t_1,t_2)=\mathbb{E}\xi(t_1)\xi(t_2)-\mathbb{E}\xi(t_1)\mathbb{E}\xi(t_2).$$ Для вычисления первого слагаемого удобно воспользоваться общей формулой для математического ожидания через интеграл Лебега: $$\mathbb{E}\xi(t_1)\xi(t_2)=\int\limits_{0}^1\xi(\omega,t_1)\xi(\omega,t_2)\,d\omega=1-\max{(t_1,t_2)}.$$ Другой подход состоит в том, чтобы заметить, что случайная величина $\xi(t_1)\xi(t_2)$ имеет распределение Бернулли $\mathrm{Be}(1-\max{(t_1,t_2)})$, стало быть, ее математическое ожидание равно $1-\max{(t_1,t_2)}$. В любом случае получаем $$R_{\xi}(t_1,t_2)=1-\max{(t_1,t_2)}-(1-t_1)(1-t_2)=\min{(t_1,t_2)}-t_1 t_2. \text{ \EndEx}$$ 
\end{solution}

\subsection{Винеровский процесс}\label{Winer}

Перейдём теперь к более содержательным примерам случайных процессов, повсеместно встречающихся в приложениях. Для этого дадим сначала определение одного важного класса случайных процессов.

\begin{definition}
Случайная функция $\{X(t), t\ge 0\}$ называется \textit{процессом с независимыми приращениями}, если для любого ${n\ge 1}$ и любых моментов времени $0\le t_1 \le t_2 \le \dots \le t_n$ случайные величины $$X(0), \ X(t_1)-X(0), \ X(t_2)-X(t_1), \ \dots, \ X(t_n)-X(t_{n-1})$$ независимы в совокупности.
\end{definition}

\begin{definition}\label{def:WienerDef1}
\textit{Винеровским процессом с параметром ${\sigma>0}$} на\-зы\-вается случайная функция ${\{W(t),t\ge0\}}$, удовлетворяющая условиям:
\begin{enumerate}
    \item $W(0)=0$ почти всюду.
    \item $W(t)$ -- процесс с независимыми приращениями.
    \item Для любых $t,s\ge0$ выполнено $W(t)-W(s)\in\mathrm{N}(0,\sigma^2|t-s|)$.
\end{enumerate}
\end{definition}
Для простоты винеровские процессы с параметром ${\sigma}$ будем иногда называть просто \textit{винеровскими процессами} там, где значение $\sigma$ не играет принципиальную роль.

Найдем плотность $n$-мерного распределения винеровского процесса. Для этого составим вектор $Y=(W(t_1),\dots,W(t_n))$ из сечений процесса в точках $t_1<\dots<t_n$. Сечения винеровского процесса являются зависимыми величинами, а приращения на непересекающихся интервалах -- независимые. Найдем поэтому сначала распределение вектора $$Z=(W(t_1),W(t_2)-W(t_1),\dots,W(t_n)-W(t_{n-1})).$$ Так как этот вектор состоит из независимых в совокупности нормальных случайных величин, определенных на одном вероятностном пространстве, то он является нормальным случайным вектором. Плотность его распределения выражается по формуле $$f_Z(z)=\prod\limits_{k=1}^n \frac{1}{\sqrt{2\pi\sigma^2(t_k-t_{k-1})}}\exp{\left( -\frac{z_k^2}{2\sigma^2(t_k-t_{k-1})} \right)}, \ z=(z_1,\dots,z_n),$$ где $t_0=0$. Остается заметить, что векторы $Y$ и $Z$ связаны линейным преобразованием $Y=AZ$, где 
\begin{equation}\label{eq:AFormulaGauss}
   A = \left[ {\begin{array}{*{20}{c}}
1&0&0&0& \ldots &0\\
1&1&0&0& \ldots &0\\
1&1&1&0& \ldots &0\\
 \vdots & \vdots & \vdots & \vdots & \ddots & \vdots \\
1&1& \ldots &1&1&1
\end{array}} \right]. 
\end{equation}
Плотность распределения $Y$ связана с плотностью распределения $Z$ по формуле
$$f_Y(y)=\frac{1}{|\det{A}|}f_Z(A^{-1}y), \ y=(y_1,\dots,y_n).$$ После преобразований получаем окончательно

$$
f_Y(y)=\prod_{k=1}^n \frac{1}{\sqrt{2 \pi \sigma^2(t_k-t_{k-1})}} \exp{\left(-\frac{(y_k - y_{k-1})^2}{2\sigma^2 (t_k - t_{k-1})}\right)}, \ t_0=0, \ y_0 = 0.
$$

Напомним, что линейное преобразование переводит нормальные случайные векторы в нормальные случайные векторы (см. раздел~\ref{gauss}).
Принимая во внимание выкладки выше, мы можем заключить, что любой вектор, составленный из сечений винеровского процесса, является нормальным случайным вектором, причем $$\mathbb{E}Y=\mathbb{E}(AZ)=A\mathbb{E}Z=0,$$ $$R_{Y}=R_{AZ} = AR_{Z}A^T=\left\|\min(t_i,t_j)\right\|_{i,j=1}^n,$$ где $R_Y$ и $R_Z$ -- ковариационные матрицы векторов $Y$ и $Z$ соответственно. Винеровский процесс таким образом относится к более широкому классу случайных процессов -- гауссовским процессам.

\begin{definition}
Случайный процесс $\{X(t),t\ge0\}$ называется \textit{гауссовским}, если для любого $n\ge1$ и точек $0 \le t_1 < \dots < t_n$ вектор $(X(t_1)$, $\dots$, $X(t_n))$ является нормальным случайным вектором.
\end{definition}

\begin{example}
Найти математическое ожидание, дисперсию, корреляционную и ковариационную функции винеровского процесса.
\end{example}

\begin{solution}
Воспользуемся тем, что ${W(t)=W(t)-W(0)\in\mathrm{N}(0,\sigma^2 t)}$:  $$m_W(t)=\mathbb{E}W(t)=0\text{ для любого } t\ge 0,$$
$$D_W(t)=\mathbb{D}W(t)=\sigma^2 t\text{ для любого } t\ge 0.$$
Для расчета корреляционной функции $$R_W(t,s)=\mathbb{E} \left(W(t)W(s)\right) - \mathbb{E} W(t)\mathbb{E}W(s) = \mathbb{E} \left(W(t)W(s)\right)$$ сначала рассмотрим случай ${t>s}$. В этом случае приращения\linebreak $W(t)-W(s)$ и $W(s)-W(0)=W(s)$ являются независимыми. Значит,
\begin{gather*}
\begin{split}
R_W(t,s)&=\mathbb{E}(W(t)-W(s))W(s)+\mathbb{E}W^2(s)\\&=\mathbb{E}(W(t)-W(s))\mathbb{E}W(s) + \mathbb{E}W^2(s)\\&=\mathbb{E}W^2(s)=\mathbb{D}W(s)+(\mathbb{E}W(s))^2=\sigma^2 s.
\end{split}
\end{gather*}
Если же ${t<s}$, то аналогично получим ${R_W(t,s)=\sigma^2 t}$. Случай ${t=s}$ рассматривается отдельно, получается ${R_W(t,t)=\mathbb{E}W^2(t)=\sigma^2 t}$. Значит, в общем случае корреляционная функция винеровского процесса $$R_W(t,s)=\sigma^2\min(t,s).$$ Так как математическое ожидание винеровского процесса равно нулю, то ковариационная функция равна $$K_W(t,s)=R_W(t,s)+\mathbb{E}W(t)\mathbb{E}W(s) = R_W(t,s) = \sigma^2\min(t,s).\text{ \EndEx}$$ 
\end{solution}

В связи с определением гауссовкого процесса можно дать равносильное определение винеровского процесса.

\begin{theorem}
Процесс $\{W(t),t\ge0\}$ является винеровским тогда и только тогда, когда
\begin{enumerate}[topsep=3pt,itemsep=0pt]
    \item $W(t)$ -- гауссовский процесс;
    \item $\mathbb{E}W(t) = 0$ для любого $t \ge 0$;
    \item $R_W(t,s) = \min(t,s)$ для любых $t,s \ge 0$.
\end{enumerate}
\end{theorem}

\textbf{Доказательство}. Пусть процесс является винеровским. Ранее мы уже установили принадлежность винеровского процесса к гауссовским процессам, а в предыдущем примере вычислили его математическое ожидание и корреляционную функцию.

Пусть теперь дан гауссовский процесс $X(t)$ с численными характеристиками $\mathbb{E}X(t)\!=\!0$ и $R_X(t,s)\!=\!\min(t,s)$. Из \mbox{$R_X(0,0)\!=\!\mathbb{D}X(0)\!=\!0$}\linebreak следует, что ${X(0)=\const}$ п.н., а из того, что ${\mathbb{E}X(0)=0}$ получаем\linebreak $X(0)=0$~п.н.

Далее, так как процесс $X(t)$ гауссовский, то и произвольное приращение $X(t)-X(s)$ является нормальной случайной величиной. Математическое ожидание $\mathbb{E}(X(t)-X(s))=0$, а дисперсия
$$\mathbb{D}(X(t)-X(s))=\mathbb{D}X(t) - 2\cov(X(t),X(s)) + \mathbb{D}X(s)=$$
$$=R_X(t,t)-2R_X(t,s)+R_X(s,s)=t+s-\min(t,s)=|t-s|.$$
Так как процесс гауссовский, то любой вектор из его сечений является гауссовским, причем математическое ожидание этого вектора -- нулевой вектор, а ковариационная матрица $\left\|R_{ij}\right\|$ состоит из элементов $R_{ij}=\min(t_i,t_j)$. Если взять преобразование, обратное к $A$ из формулы~\eqref{eq:AFormulaGauss}, то получится, что вектор $(X(t_1),X(t_2)-X(t_1),\dots,X(t_n)-X(t_{n-1}))$ состоит из некоррелированных случайных величин, а раз вектор гауссовский, то -- из независимых случайных величин. \EndProof

Согласно теореме~\ref{Kolmog cont modific} Колмогорова для винеровского процесса существует непрерывная модификация (для этого в теореме можно положить ${\alpha = 4}$, ${\beta = 1}$, ${c = 3 \sigma^4}$). Поэтому всегда подразумевается, что траектории винеровского процесса непрерывны.

Мы дали аксиоматическое определение винеровского процесса, т.е. путем перечисления его свойств. Доказательство существования винеровского процесса можно выполнить различными способами~\cite{BulinskyShiryaev2005}, 
основанными на теоремах Колмогорова, разложениях в ряды, слабой сходимости мер в $C(0,1)$ и центральной предельной теореме. Здесь важно отметить, что винеровский процесс можно получать (и определять) как в некотором смысле предел случайных блужданий на вещественной оси. Поясним сказанное на модельном примере.

Допустим имеется некоторая акциия, цена которой формируется скачками в моменты времени $1/N$, $2/N$, $3/N$ и т.д., ${N \ge 1}$. Пусть в момент времени ${t > 0}$ цена акции определяется формулой $$X_N(t)=\sum\limits_{i=1}^{[Nt]} \Delta X_i,$$ где 
$$\Delta X_i = \begin{cases}
	\sigma/\sqrt{N}, & \text{с вероятностью } $1/2$,\\
	-\sigma/\sqrt{N}, & \text{с вероятностью } $1/2$,
	\end{cases}$$
причем $\{\Delta X_i\}$ -- независимые одинаково распределенные случайные величины. Заметим, что ${\mathbb{E}\Delta X_i = 0}$, ${\mathbb{D}\Delta X_i=\sigma^2/N}$. Тогда из центральной предельной теоремы следует, что для $t>0$
$$\frac{X_N(t) - 0}{\sigma/\sqrt{N}\cdot\sqrt{[Nt]}} \xrightarrow[N\to\infty]{d} Z \in \mathrm{N}(0,1).$$ Но так как $[Nt]/N \to t$ при $N\to\infty$, то $$X_N(t) \xrightarrow[N\to\infty]{d} X(t) \in \mathrm{N}(0, \sigma^2t), \ t > 0.$$

Оказывается, что предел $X(t)$ в этом выражении -- это винеровский процесс $W(t)$ с параметром $\sigma$, но чтобы это доказать, потребуется обратиться к понятию \textit{слабой сходимости случайных процессов}, или \textit{слабой сходимости мер}.

Подробное изложении теории слабой сходимости случайных процессов, включая доказательство сходимости процесса выше к винеровскому процессу, можно найти в монографии~\cite{Billingsley}. Мы же приведем лишь краткое пояснение.

Вообще в функциональном анализе под слабой сходимостью абстрактной последовательности $x_n$ к абстрактному элементу $x$ в линейном топологическом пространстве $E$ понимается сходимость числовой последовательности $f(x_n)$ к $f(x)$ для любого непрерывного линейного функционала $f$ из сопряженного пространства $E^*$. В теории вероятностей, когда говорят о слабой сходимости случайной последовательности $\xi_n$ к величине $\xi$, обычно для простоты имеют в виду поточечную сходимость $F_n(x)\to F(x)$ функций распределения  $\xi_n$ к функции распределения $\xi$ в точках непрерывности $F(x)$. На самом деле же слабая сходимость $\xi_n$ к $\xi$ означает слабую сходимость последовательности мер $\mathbb{P}_n$, порождаемой функциями распределения $F_n(x)$, к мере $\mathbb{P}$, порождаемой функцией распределения $F(x)$, причем меры эти являются элементами некоторого линейного топологического пространства мер. Не важно, как строится такое топологическое пространство, отметим лишь, что слабую сходимость мер можно определить несколькими равносильными способами. Например, под слабой сходимостью мер $\mathbb{P}_n$ к мере $\mathbb{P}$ (определенных на некотором измеримом пространстве $(\mathcal{S},\mathcal{F})$) можно понимать сходимость числовой последовательности 
$$\mathbb{E}_n f \to \mathbb{E} f, \ n\to\infty$$ для каждой непрерывной ограниченной функции $f$, определенной на элементах метрического пространства $\mathcal{S}$; здесь $\mathbb{E}_n$ и $\mathbb{E}$ обозначают математические ожидания, подсчитанные относительно мер $\mathbb{P}_n$ и $\mathbb{P}$, соответственно.

А теперь вернемся к случайным процессам. Траектории случайного процесса $X_N(t)$ и винеровского процесса $W(t)$ непрерывны с вероятностью 1, поэтому можно считать, что они заданы на метрическом пространстве $\mathcal{S}=C[0,T]$, оснащенном стандартной равномерной метрикой и некоторой сигма-алгеброй. Конечномерные распределения этих процессов порождают на $\mathcal{S}$ некоторые меры: $\mathbb{P}_N$ и $\mathbb{W}$ (последняя еще называется \textit{мерой Винера}). Оказывается, что $\mathbb{P}_N$ слабо сходится к $\mathbb{W}$. Доказывать это можно по-разному, одно из достаточных условий сходимости описывается в монографии~\cite{Billingsley}:
1) сначала надо доказать, что при каждом $n\ge1$ и каждом фиксированном наборе $(t_1,\dots,t_n)$ конечномерные распределения процесса $X_N(t)$ сходятся к соответствующим конечномерным распределениям процесса $W(t)$, а~затем надо  показать, что последовательность $\mathbb{P}_N$ \textit{полна}, т.е. <<достаточно хорошая>>, чтобы вычислить пределы по ${N\to\infty}$ для произвольных, нефиксированных наборов $(t_1,\dots,t_n)$ (грубо говоря, чтобы имела место <<равномерная сходимость по всем сечениям>>). Мы же в примере выше показали только сходимость одномерных распределений $X_N(t)$ к $W(t)$ для фиксированного $t$.

Отметим, наконец, что так как винеровский процесс является пределом некоторых других, более простых процессов, то его числовые характеристики можно аппроксимировать теми же числовыми характеристики простых процессов, и наоборот.

На рис.~\ref{fig:TypicalWiener} показаны типичные траектории винеровского процесса. На~рис.~\ref{fig:ManyTrajsWiener} показаны 1000 реализаций процесса. Хорошо видно, что они с большой вероятностью попадают в область между графиками функций ${y=\pm 3\sqrt{t}}$, которые отвечают правилу трех сигм нормального распределения: если ${\xi\in\mathrm{N}(0,\sigma^2)}$, то ${\mathbb{P}(-3\sigma \le \xi \le 3\sigma)\approx 0.9973}$. В~нашем случае $\xi=W(t)$, $\sigma^2=t$.

\begin{figure}[p]
	\centering
	\begin{subfigure}[h]{.495\linewidth}
		\centering
		\includegraphics[width=1.0\textwidth]{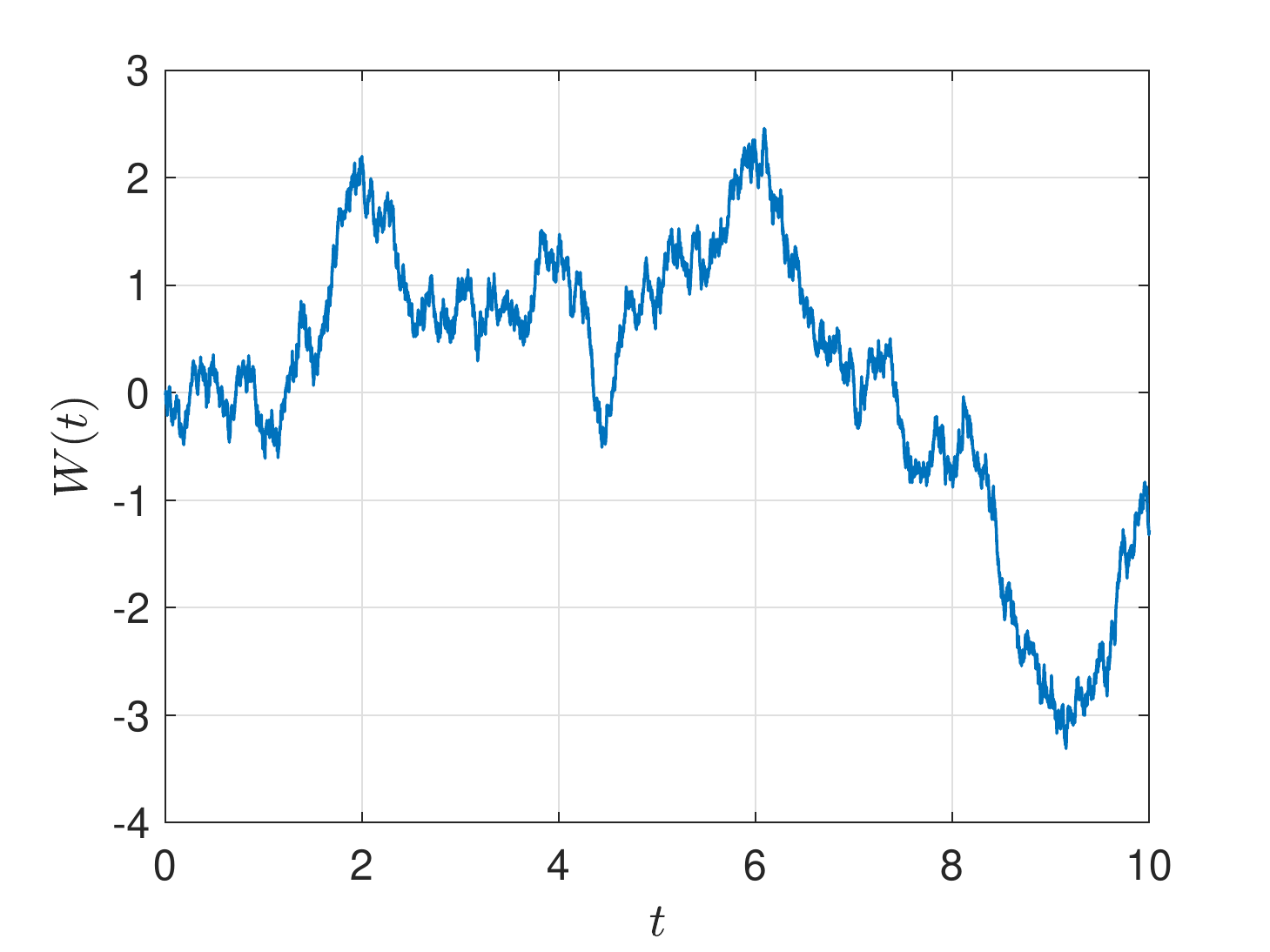} 
		\caption{}
	\end{subfigure}
	\begin{subfigure}[h]{.495\linewidth}
		\centering
		\includegraphics[width=1.0\textwidth]{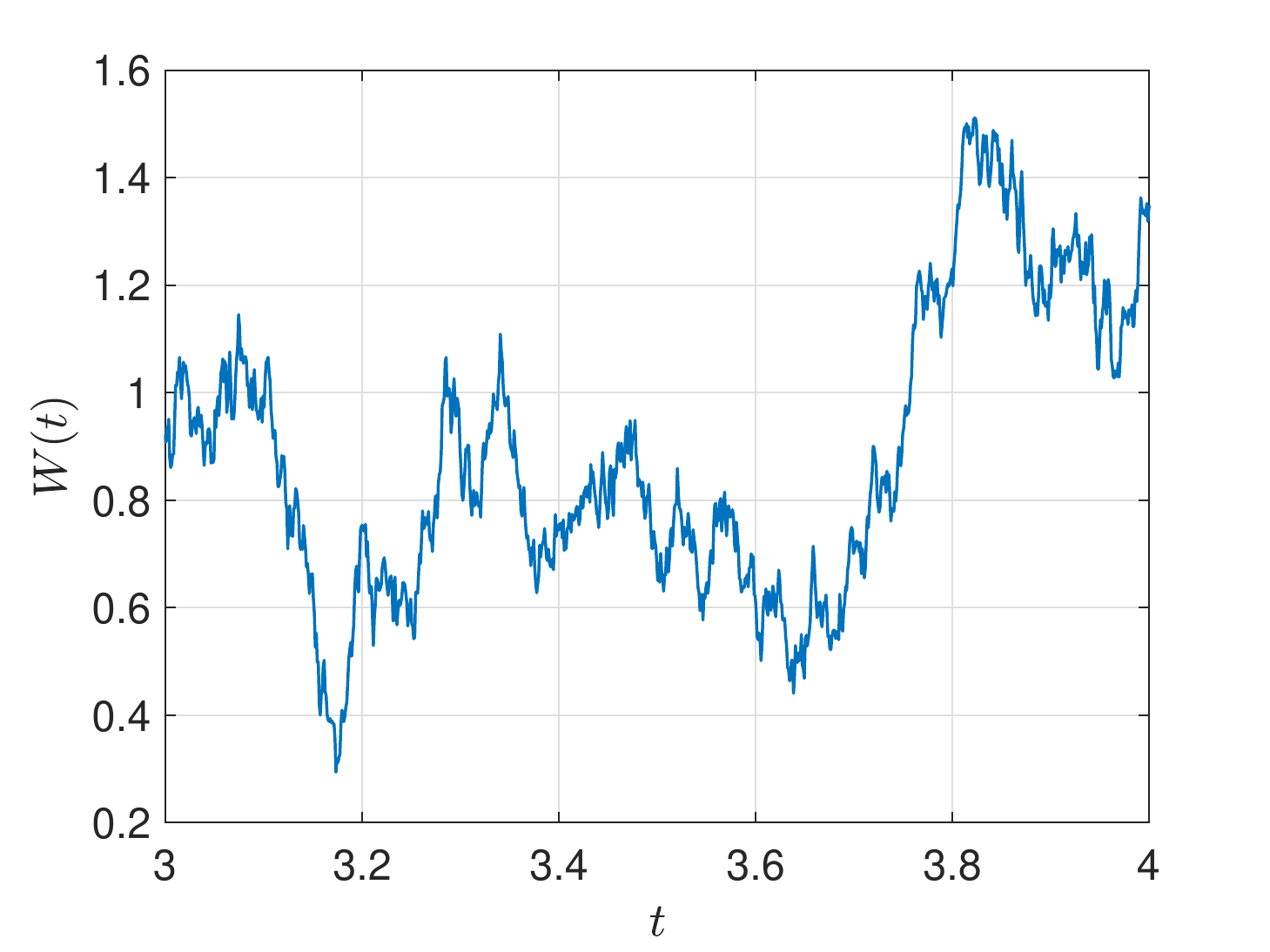} 
		\caption{}
	\end{subfigure}
	\caption{Пример траектории винеровского процесса на а)\;интервале \mbox{$t\in[0,10]$}~и~б)\;интервале $t\in[3,4]$}\label{fig:TypicalWiener}
\end{figure}

\begin{figure}[p]
	\centering
	\includegraphics[scale=0.7]{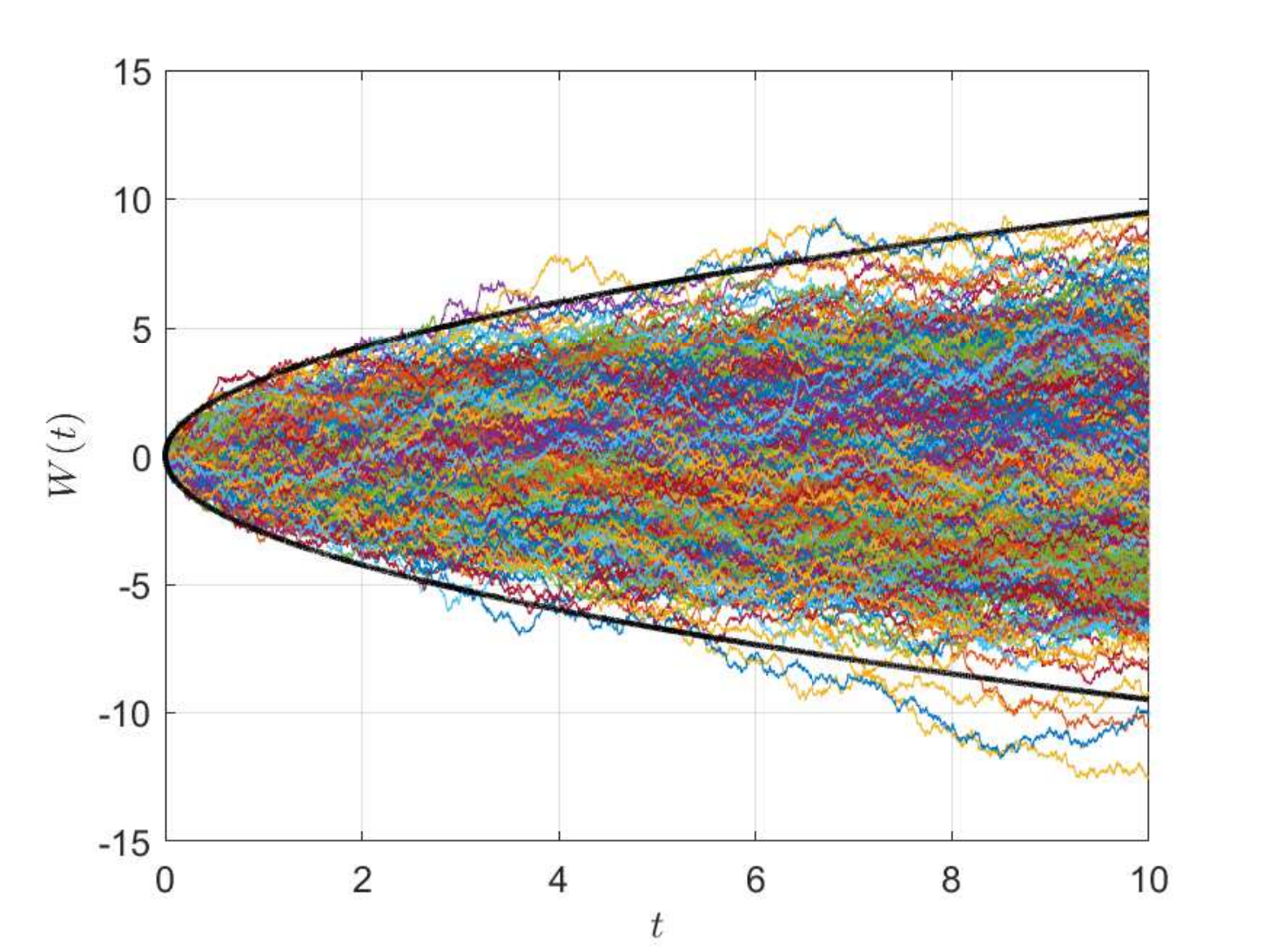}
	\caption{Реализации винеровского процесса с вероятностью приблизительно $99.7\%$ попадают в область между графиками функций $y\!=\!\pm 3\sqrt{t}$}
	\label{fig:ManyTrajsWiener}
\end{figure}

Известны и более сильные результаты о поведении винеровского процесса (см., например,~\cite[гл. III]{BulinskyShiryaev2005}).
\begin{theorem}[ (закон повторного логарифма, Хинчин)]
\textit{С~вероятностью единица}
$$
\liminf\limits_{t\to\infty}\frac{W(t)}{\sqrt{2t \ln\ln t}} = -1, \ \limsup\limits_{t\to\infty}\frac{W(t)}{\sqrt{2t \ln\ln t}} = 1.
$$
\end{theorem}

С помощью преобразования  ${\widetilde{W} (t) = t W(1/t)}$ получается снова винеровский процесс, поведение которого в окрестности нуля описывается поведением $W(t)$ на бесконечности. Откуда вытекает следующая теорема.
\begin{theorem}[ (локальный закон повторного логарифма)]\textit{ С~вероятностью единица}
$$
\liminf\limits_{t\to 0+} \frac{W(t)}{\sqrt{2t \ln\ln (1/t)}} = -1, \ 
\limsup\limits_{t\to 0+} \frac{W(t)}{\sqrt{2t \ln\ln (1/t)}} = 1.
$$
\end{theorem}

\begin{theorem}[ (Башелье)]
\textit{При всех }$T$, $x \ge 0$
$$
\mathbb{P} \left( \max_{t\in [0,T]} W(t) \ge x \right) = 2 \mathbb{P} \left(  W(T) \ge x \right) = \mathbb{P} \left(  |W(T)| \ge x \right).
$$
\end{theorem}

\subsection{Пуассоновский процесс}\label{sec:poisson_proc}

\begin{definition}
\textit{Пуассоновским процессом с интенсивностью} $\lambda>0$ называется случайная функция $\{K(t),t \ge 0\}$, удовлетворяющая свойствам: 
\begin{enumerate}
	\item $K(0)=0$ почти всюду.
	\item $K(t)$ -- процесс с независимыми приращениями.
	\item $\forall t > s \ge 0$ выполнено $K(t)-K(s)\in\mathrm{Po}(\lambda (t-s))$.
\end{enumerate}
\end{definition}

Легко показать, что математическое ожидание, дисперсия и корреляционная функция пуассоновского процесса выражаются формулами ${\mathbb{E}K(t)=\lambda t}$, ${\mathbb{D}K(t)=\lambda t}$, ${R_K(t,s)=\lambda\min (t,s)}$. Так как для любых ${t>s\ge 0}$ приращение ${K(t)-K(s)\in\mathrm{Po}(\lambda (t-s))}$ неотрицательно и может принимать лишь значения $0,1,2,\dots$, то траектории пуассоновского процесса -- неубывающие кусочно-постоянные функции. Эти функции начинаются в нуле, так как ${K(0)=0}$ п.н., и испытывают прыжки (скачки) в случайные моменты времени.

\begin{theorem} \shmaxg{\textit{Пусть $\{\xi_n\}$ -- последовательность независимых случайных величин с распределением $\mathrm{Exp}(\lambda)$. Обозначим $S_n = \xi_1 + \dots + \xi_n$ ($S_0=0$). Введем процесс
\begin{equation}\label{PP}
    X(t) = \sup\{n:S_n \le t\}.
\end{equation}
Тогда $X(t)$ -- пуассоновский процесс интенсивности $\lambda$.}}
\end{theorem}

\begin{proof} \shmaxg{Рассмотрим случайный вектор $(S_1,\dots,S_n)$. Этот вектор связан со случайным вектором $(\xi_1,\dots,\xi_n)$ линейным преобразованием:
$$\left[ {\begin{array}{*{20}{c}}
  {{S_1}} \\ 
  {{S_2}} \\ 
  {{S_3}} \\ 
   \vdots  \\ 
  {{S_n}} 
\end{array}} \right] = \left[ {\begin{array}{*{20}{c}}
  1&0&0& \ldots &0 \\ 
  1&1&0& \ldots &0 \\ 
  1&1&1& \ldots &0 \\ 
   \vdots & \vdots & \vdots & \ddots & \vdots  \\ 
  1&1&1&1&1 
\end{array}} \right]\left[ {\begin{array}{*{20}{c}}
  {{\xi _1}} \\ 
  {{\xi _2}} \\ 
  {{\xi _3}} \\ 
   \vdots  \\ 
  {{\xi _n}} 
\end{array}} \right]$$
Это наблюдение позволяет вычислить плотность вектора $(S_1,\dots,S_n)$. Действительно определитель этой матрицы равен единице, поэтому плотность вектора $(S_1,\dots,S_n)$ равна
$$p_{S_1,\dots,S_n}(x_1,\dots,x_n)=p_{\xi_1,\dots,\xi_n}(x_1,x_2-x_1,\dots,x_n-x_{n-1}).$$ Остается вспомнить, что случайные величины $\xi_1,\dots,\xi_n$ независимы, значит
$$p_{S_1,\dots,S_n}(x_1,\dots,x_n) = \prod\limits_{j=1}^n \lambda e^{-\lambda(x_j - x_{j-1})}I(x_j - x_{j-1} \ge 0)=$$
$$=\lambda^n e^{-\lambda x_n}I(x_n > x_{n-1} > \dots > x_1).$$ Зафиксируем моменты времени $0=t_0<t_1<\dots<t_n$ и введем неотрицательные целые числа $k_1,\dots,k_n\in\mathbb{Z}_+$ и $0=k_0 \le k_1 \le k_2 \le \dots \le k_n$. Чтобы доказать, что процесс $X(t)$ пуассоновский нужно доказать, что приращения его независимые и имеют пуассоновское распределение. Для этого мы вычислим вероятность
$$\mathbb{P}(X(t_n)-X(t_{n-1})=k_n-k_{n-1},\dots,X(t_2)-X(t_1)=k_2-k_1, X(t_1)=k_1),$$ покажем, что она распадается на произведение вероятностей и что эти вероятности равны тому, чему нужно, чтобы эти случайные величины имели пуассоновское распределение. Эта вероятность равна
$$\mathbb{P}(\{S_1,\dots,S_{k_1}\}\in(0,t_1),\{S_{k_1+1},\dots,S_{k_2}\}\in(t_1,t_2),\dots,$$$$\dots\{S_{k_{n-1}+1},\dots,S_{k_n}\}\in(t_{n-1},t_n),S_{k_n+1} > t_n)=$$
$$=\int\dots\int \lambda^{k_n+1}e^{-\lambda x_{k_n+1}}I(0<x_1<\dots<x_{k_n+1})\,dx_1\dots dx_n,$$ где интеграл берется по области
$$A = \{(x_1,\dots,x_{k_n+1})\in\mathbb{R}^{k_n+1}: \{x_1,\dots,x_{k_1}\}\in(0,t_1),\dots,x_{k_n+1}>t_n\}.$$
Далее, интеграл выше равен
$$\int\limits_{t_n}^{\infty}\lambda e^{-\lambda x_{k_n+1}}\,dx_{k_n+1} \lambda^{k_n}\prod\limits_{j=1}^n\int\dots\int I(x_{k_{j-1}+1} < \dots < x_{k_j})\,dx_1,\dots dx_{k_{n}}=$$
$$=e^{-\lambda t_n}\lambda^{k_n}\prod\limits_{j=1}^n\frac{(t_j-t_{j-1})^{k_j-k_{j-1}}}{(k_j-k_{j-1})!}=\prod\limits_{j=1}^n \frac{(\lambda(t_j-t_{j-1}))^{k_j-k_{j-1}}}{(k_j-k_{j-1})!}e^{-\lambda(t_j-t_{j-1})}.$$ Интеграл в выражении выше без индикатора равен объему многомерного прямоугольника, а с индикатором -- объему симплекса, который равен объему прямоугольника поделить на факториал размерности пространства. Из полученного выражения следует, что все случайные величины $X(t_j)-X(t_{j-1})$ независимы, причем вероятность того, что $$\mathbb{P}(X(t_j)-X(t_{j-1}))=\frac{(\lambda(t_j-t_{j-1}))^{k_j-k_{j-1}}}{(k_j-k_{j-1})!}e^{-\lambda(t_j-t_{j-1})},$$
т.е. ${X(t_j)-X(t_{j-1})\in\mathrm{Po}(\lambda(t_j-t_{j-1}))}$. А это и есть то, что мы хотели доказать: приращения процесса независимы и распределены по Пуассону с нужными параметрами. \EndProof}

\end{proof}

\shmaxg{Теперь разберемся с тем, какие следствия мы получаем из явной конструкции пуассоновского процесса.}

\shmaxg{1) Скачки происходят в моменты времени $\tau_1=S_1$, $\tau_2=S_2$ и так далее. Так как $S_n$ представляет собой сумму $n$ независимых случайных величин с распределением $\mathrm{Exp}(\lambda)$, то $\tau_n=S_n$ имеет распределение Эрл\'{а}нга $\mathrm{Erl}(n,\lambda)$.}

\shmaxg{2) Между скачками проходит случайное время $$\tau_n-\tau_{n-1}=S_n-S_{n-1}=\xi_n\in\mathrm{Exp}(\lambda).$$ Времена между соседними последовательными скачками -- независимые случайные величины.}

\shmaxg{3) С вероятностью 1 все скачки пуассоновского процесса являются единичными. Действительно, скачки происходят только в моменты времени $\tau_1=S_1$, $\tau_2=S_2$ и так далее, поэтому
$$\mathbb{P}(\exists\,\text{скачок размера}\ge2)=\mathbb{P}(\exists n:S_n = S_{n+1})=\mathbb{P}(\exists n: \xi_{n+1}=0)=0,$$ так как все $\xi_j$ имеют непрерывное распределение.}

\shmaxg{Доказывая вышеприведенную теорему, мы получили заодно и $n$-мерную функцию вероятности пуассоновского процесса:
$$\mathbb{P}(K(t_1)=k_1,\dots,K(t_n)=k_n)=\prod\limits_{j=1}^n \frac{(\lambda \Delta t_j)^{\Delta k_j}}{\Delta k_j!}\exp{(-\lambda \Delta t_j)},$$ где $\Delta t_j = t_j - t_{j-1}$, $\Delta k_j=j_j-k_{j-1}$, $t_0=0$, $k_0=0$.}

\shmaxg{Процессы вида
$$X(t) = \sup\left\{n: \sum\limits_{k=1}^n \xi_k \le t\right\},$$ где $\xi_k$ -- независимые случайные величины (не обязательно одинаково или показательно распределенные) называются еще \textit{процессами восстановления}. Теорема, сформулированная и доказанная выше, говорит о том, что пуассоновский процесс -- это процесс восстановления, построенный по случайным величинам с показательным распределением. Последовательность $\{S_n\}$, где $S_n=\sum_{k=1}^n\xi_k$, называется \textit{случайным блужданием}.}

На рисунке~\ref{fig:PoissonProcTraj} изображена типичная траектория пуассоновского процесса.

\begin{figure}[h]
	\centering
	\includegraphics[scale=0.5]{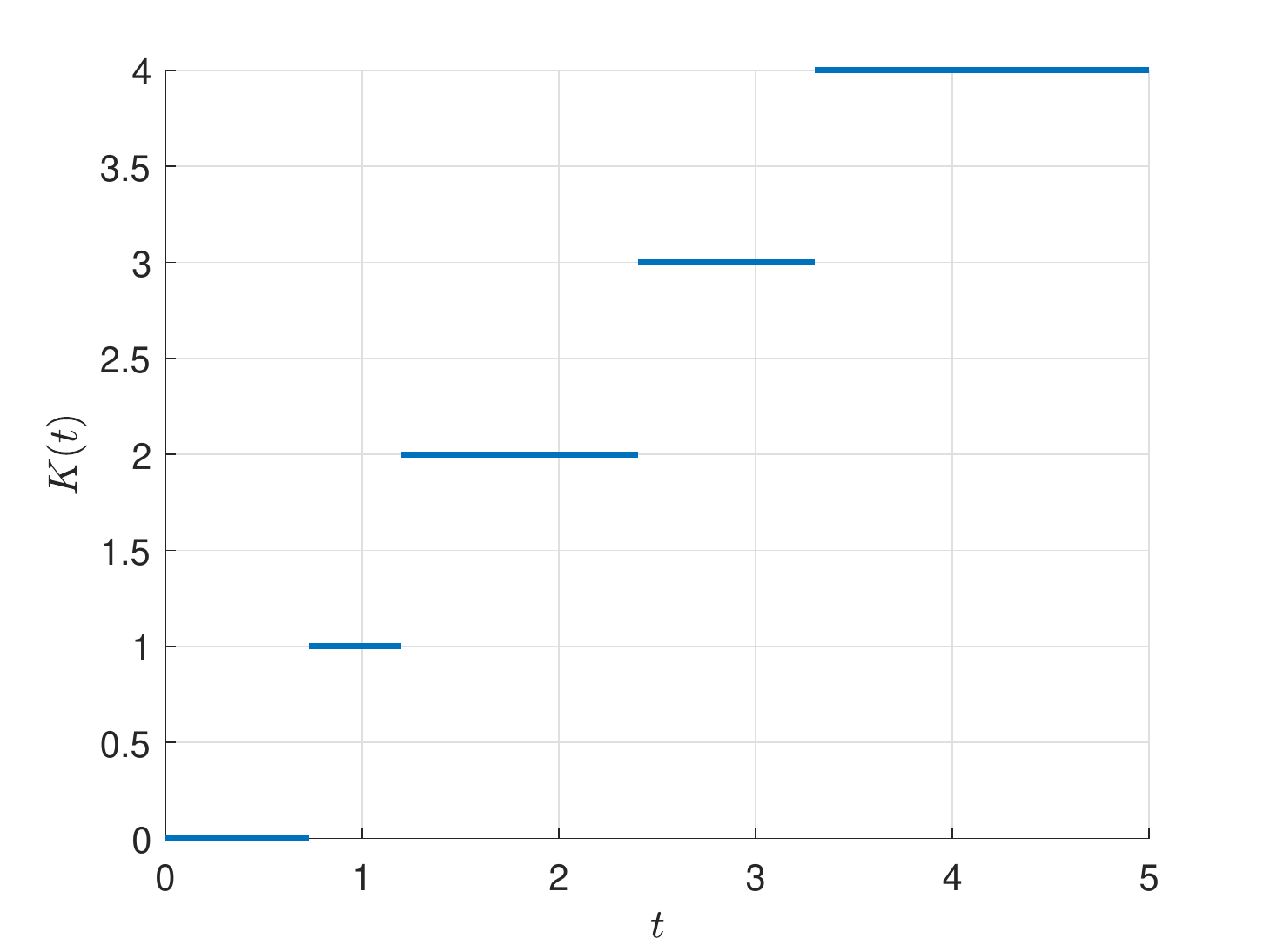}
	\caption{Пример типичной траектории пуассоновского процесса}
	\label{fig:PoissonProcTraj}
\end{figure}

Также отметим, что формула~\eqref{PP} дает конструктивное определение пуассоновского процесса на вероятностном пространстве $([0,1],$ $\mathcal{B}([0,1]), \mathcal{L})$, $\mathcal{L}$  -- мера Лебега. Для этого по $\omega\in [0,1]$ нужно построить последовательность независимых равномерно на $[0,1]$ распределенных случайных величин $\omega_k(\omega)$, $k\ge 1$, далее методом обратной функции построить показательные случайные величины $\xi_k(\omega) = -\ln{\omega_k (\omega)}$, $k\ge 1$ и воспользоваться формулой~\eqref{PP}. Для построения последовательности $\omega_k$, $k\ge 1$, можно провести следующую процедуру: записать $\omega$ в двоичной системе исчисления, получив тем самым последовательность независимых случайных величин из распределения $\mathrm{Be}(1/2)$, разбить ее на счетное число непересекающихся подпоследовательностей (например, записав исходную последовательность <<змейкой>> в таблицу, выделить в ней строки) и снова свернуть каждую из них в десятичное представление.

Покажем теперь, как пуассоновский процесс может быть получен из последовательности процессов путем предельного перехода. Пусть ${X_k \in \mathrm{R}(0, N)}$, ${k = 1,\dots,N}$ и нас интересует число точек, которые попадают на отрезок $[0,\lambda t] \subseteq [0, N]$, где $\lambda > 0$, $t \ge 0$. Для этого введем случайную величину ${\nu_k = \mathsf{I}(X_k \in [0,\lambda t]) \in \Be(\lambda t / N)}$. Тогда из теоремы Пуассона следует, что
$$\sum_{k=1}^{N}\nu_k \xrightarrow[N\to\infty]{d} X(t) \in \Po(\lambda t).$$
Можно показать, что определённый для ${t \ge 0}$ предельный процесс $X(t)$ является пуассоновским процессом $K(t)$ с интенсивностью $\lambda$. Данный пример демонстрирует важное понятие, играющее ключевую роль в статистической механике -- \textit{термодинамический предельный переход}. Пусть есть сосуд длины $N$ и $N$ невзаимодействующих частиц, равномерно распределенных в нем. Тогда если выбрать в этом сосуде какую-нибудь произвольную область длины $L$, то при предельном переходе $N\to\infty$ и сохранении равномерности распределения частиц (постоянстве концентрации) окажется, что вероятность нахождения $k$ частиц в области $L$ равна $e^{-L}L^k/k!$. Данный пример демонстрирует естественность возникновения распределения Пуассона при таком предельном переходе. Много интересных приложений пуассоновские процессы (поля) находят в~\cite{KendallMoran, Kingman, KoralovSinai, Malyshev}. 

Приведем и другой важный пример естественного возникновения пуассоновского процесса. Предположим имеется остановка для автобусов, и время между прибытиями автобусов на остановку имеет показательное распределение ${\xi_i \in \mathrm{Exp}(\lambda)}$. Отметим, что в классе распределений, имеющих плотность, только показательное распределение удовлетворяет следующему условию (отсутствие последействия): для всех $t, \tau > 0$:
$$\mathbb{P}(\xi \ge t + \tau \, | \, \xi \ge t) = \mathbb{P}(\xi \ge \tau).$$
Это значит, что если человек приходит на остановку в некоторый фиксированный момент времени $t_0$, не связанный с пуассоновским процессом, то в силу отсутствия памяти у показательного распределения время ожидания автобуса будет по-прежнему иметь  показательный закон с тем же параметром $\lambda$. Отметим еще, что в классе дискретных распределений таким свойством обладает только геометрическое распределение.

\begin{example}[ (нерадивый пешеход)]
Пешеход хочет перейти дорогу в неположенном месте. Для того, чтобы перебежать дорогу, ему требуется $a$ секунд. Кроме того, ему приходится ждать некоторое случайное время, когда зазор между соседними машинами, образующими пуассоновский процесс интенсивности $\lambda$, будет больше, чем $a$ секунд. Вычислить математическое ожидание времени перехода дороги с учетом ожидания.
\end{example}

\begin{solution}
Сначала введем случайную величину, равную времени перехода дороги с учетом ожидания <<зазора>> между машинами:
$$T_a = \min \left( t>0: \, K(t-a) = K(t) \right) .$$
Для вычисления ее математического ожидания воспользуемся формулой полной вероятности, <<обуславливая>> по моменту времени $\xi_1$ проезда мимо пешехода первой машины:
$$\mathbb{E}T_a = \int\limits_{0}^{\infty} \mathbb{E}(T_a \, | \, \xi_1=x)f_{\xi_1}(x)\,dx = \int\limits_0^\infty \lambda e^{-\lambda x} \mathbb{E}(T_a \, | \, \xi_1 = x)\, dx,$$ где $f_{\xi_1}(x)$ -- плотность распределения $\xi_1$. Заметим, что если $\xi_1 > a$, то пешеходу не нужно ждать, т.е. $T_a = a$, а если $\xi_1 = x < a$, то так как до следующей машины пройдет показательно распределенное время, мы получаем ту же задачу, но со сдвигом на время $x$, т.е. $$\mathbb{E}(T_a \, | \, \xi=x) = x + \mathbb{E}T_a, \ x < a.$$
Поэтому, разбивая интеграл на две части, по отрезку $[0,a]$ и $[a,+\infty)$, получаем
$$
\mathbb{E} T_a = a e^{-\lambda a}+ \int\limits_0^a \lambda e^{-\lambda x} (x + \mathbb{E} T_a)\,dx.
$$
Разрешая это уравнение относительно $\mathbb{E}T_a$, получаем окончательно $$\mathbb{E}T_a = \frac{e^{\lambda a} - 1}{\lambda}.$$ Отсюда следует, что при интенсивном потоке машин в среднем переход через дорогу придется ждать экспоненциально долго. \EndEx
\end{solution}

Приведем еще одно определение пуассоновского про\-цесса. Предположим, что во времени наблюдается некоторый поток случайных событий. Эти события происходят через случайные моменты времени, но не произвольным образом, а по следующим правилам:

\setlength{\parskip}{0pt}

1) \textit{Стационарность}: распределение числа событий, происшедших на фиксированном (неслучайном) интервале времени $[t,s]$ есть функция только длины интервала $t-s$.

2) \textit{Отсутствие последействия}: числа событий, происшедших на непересекающихся временных интервалах независимы в совокупности.

3) \textit{Ординарность}: вероятность того, что на интервале $[t,t+\Delta t]$ произойдет хотя бы одно событие, равна $\lambda(t)\Delta t + o(1)$, $\Delta t\to 0$, где для стационарного процесса $\lambda(t)=\lambda$ не зависит от $t$.

Поток событий с вышеперечисленными тремя свойствами называется \textit{простейшим потоком событий}. Обозначим за $\nu(t,s)$ число событий такого потока, произошедшее на интервале времени $[t,s]$. Оказывается, что $K(t)=\nu(0,t)$. Другими словами, пуассоновский процесс можно определить как число событий простейшего потока событий на интервале времени $[0,t]$.

Это определение удобно в том смысле, что допускает естественное обобщение понятия пуассоновского процесса до \textit{пуассоновского поля} \elena{ на $\mathbb R^d$, $d\ge 1$  и даже на произвольном измеримом пространстве $S$, не обладающем евклидовой структурой (см. \cite{Kingman}). При этом в евклидовом случае параметр времени теперь можно интерпретировать как пространственную переменную (см. пример~\ref{ex:Forest} далее). Случайная конфигурация точек (иначе, событий) на $S$, обладающих свойствами:}

1) распределение числа событий в произвольной измеримой области $A$ зависит лишь от величины некоторой (неатомической) меры $\mu$ области $A$.

2) число событий, произошедших на непересекающихся измеримых областях независимы в совокупности.

3) число событий $\nu(A)$ в измеримой области $A$ имеет распределение Пуассона с параметром $\mu(A)$.

\elena{Функцию, определенную на измеримых подмножествах $A\subset S$, $\nu(A)$, называют \textit{считающей функцией } пуассоновского поля. Если $S = \mathbb{R}^+$,  то принято называть пуассоновским процессом не случайную конфигурацию точек на действительной полуоси, а $K(t) = \nu \left( (0,\, t]  \right)$ и интерпретировать параметр $t\ge 0$ как время. }В случае \elena{$S=\mathbb R$ } мера $\mu$ может, в частности, иметь постоянную плотность относительно меры Лебега ${\mu([0,t]) = \lambda t}$, тогда мы получаем пуассоновский процесс с постоянной интенсивностью $\lambda$, а может быть представлена к 
$$\mu([0,t]) = \int\limits_0^t \lambda (u)\,du,$$
где $\lambda(u)$ -- неотрицательная функция (\textit{плотность интенсивности}). В~этом случае говорят о пуассоновском процессе с \textit{переменной интенсивностью}.

\begin{example}[ (модель леса как пуассоновского поля, \elena{\cite{Kingman}})]\label{ex:Forest}
Рассматривается пуассоновское поле в $\mathbb{R}^d$ с постоянной плотностью интенсивности $\lambda$. Точки-события этого поля будем интерпретировать как деревья в лесу. Выберем в качестве начала координат местоположение грибника. Нужно найти плотность распределения случайной величины $\xi$ -- расстояние до ближайшего дерева. 

Важное свойство пуассоновских моделей состоит в том, что их можно отображать в другие пространства, при этом образы случайных точек вновь образуют пуассоновский процесс~\cite{Kingman}. 

Рассмотрим в качестве первого отображения -- переход от декартовых координат к полярным: 
$$
f_1(x,y) = \left(\sqrt{x^2 + y^2}, \ \arctg (y/x) \right).
$$
Мера интенсивности при этом преобразуется по формуле
$$
\mu_1 (B) = \iint_{f_1^{-1}(B)} \lambda \,d x\,d y = \iint_{B}\lambda r \,dr \,d\theta,
$$
где $B$ -- борелевское множество в полосе $\left \{ (r, \theta): \, r>0, \, \theta\in [0, 2\pi)  \right\}$. 

В качестве второго отображения рассмотрим проекцию на радиальную компоненту:
$f_2 (r,\theta) = r$, тогда 
$$\mu_2 (C) = \iint_{C\times[0,2\pi)} \lambda r \,dr \,d\theta = \int\limits_{C} 2 \pi \lambda r \,dr,$$
$C$ -- борелевское множество в $\mathbb{R}^+$.

Таким образом, композиция отображений $f_1$ и $f_2$ преобразует пуассоновское поле на  $\mathbb{R}^d$ с постоянной интенсивностью $\lambda$ в пуассоновский процесс на  $\mathbb{R}^+$ с переменной интенсивностью $2 \pi \lambda r$, $r>0$. Значит,
$$
\mathbb{P} (\xi > x ) = \exp\left(-\int\limits_0^x 2 \pi \lambda u \,du \right),
$$
а плотность распределения $\xi$ есть $2\pi\lambda x \exp{( - \pi \lambda x^2)}$, $x > 0$. \EndEx
\end{example}

\begin{definition}
Случайный процесс
$$Q(t)=\sum\limits_{i=1}^{K(t)}V_i,$$ где $K(t)$~--- пуассоновский процесс, $\{V_i \}_{i=1}^\infty$~-- независимые одинаково распределенные случайные величины, не зависящие от $K(t)$, называется \textit{сложным пуассоновским процессом}. \elena{Предполагается, что если $K(t)=0$, то и $Q(t) = 0$.}
\end{definition}

Найдем математическое ожидание, дисперсию, корреляционную и характеристическую функцию сложного пуассоновского процесса в предположении ${\mathbb{E}V_1^2 < \infty}$. Для этого удобно воспользоваться формулами полной вероятности 

$$\mathbb{E}X = \mathbb{E}(\mathbb{E}(X \, | \, Y)), \ \ \mathbb{D}X = \mathbb{D}\left(\mathbb{E}(X \, | \, Y)\right) + \mathbb{E} (\mathbb{D}(X \, | \, Y)).$$ При вычислении условного математического ожидания или условной дисперсии значение вспомогательной случайной величины $Y$ фиксируется и воспринимается как неслучайная величина. Более подробно с понятием условными математическими ожиданиями и, в частности, с доказательством формул выше можно ознакомиться в книге~\cite{ShiryaevT1}. Легко видеть, что $$\mathbb{E}Q(t) = \mathbb{E}(\mathbb{E}(Q(t) \, | \, K(t)))=\mathbb{E}(K(t)\mathbb{E}V_1) = \lambda t \mathbb{E}V_1.$$ Что касается дисперсии, то удобно сначала вычислить
$$\mathbb{E}(Q(t) \, | \, K(t))=K(t)\mathbb{E}V_1, \ \mathbb{D}(Q(t) \,|\, K(t))=K(t)\mathbb{D}V_1,$$ и затем подставить в формулу для дисперсии:
$$\mathbb{D}Q(t)=\mathbb{D}(K(t)\mathbb{E}V_1) + \mathbb{E}(K(t)\mathbb{D}V_1) = \lambda t (\mathbb{E}V_1)^2 + \lambda t \mathbb{D}V_1=\lambda t \mathbb{E}V_1^2.$$ 

Перейдем теперь к расчету корреляционной функции $Q(t)$: $$R_Q(t,s)=\mathbb{E}\overset{\circ}{Q}(t)\overset{\circ}{Q}(s).$$ Для расчета этого выражения установим сначала вспомогательный факт: сложный пуассоновский процесс -- это процесс с независимыми приращениями. 
\elena{ Для доказательства независимости приращений воспользуемся  аппаратом характеристических функций. 
Пусть $\phi_V(s) = \mathbb E e^{is V}$ -- характеристическая функция случайной величины $V$. Вычислим сначала характеристическую функцию приращения $Q(t_2) - Q(t_1)$, $0\le t_1 < t_2$, воспользовавшись формулой полной вероятности для математического ожидания.
$$\phi_{Q(t_2) - Q(t_1) }(s) = \mathbb{E} e^{is \left (Q(t_2) - Q(t_1) \right)} = \mathbb{E} \exp \left \{is \sum\limits_{j =K(t_1) + 1}^{K(t_2)} V_j \right \} =$$
$$=\sum_{m_1, m_2 = 0}^\infty  \mathbb{E} \left [ \exp \left \{is \sum_{j =K(t_1) + 1}^{K(t_2)} V_j \right \} \Bigm | K(t_1) = m_1, K(t_2) = m_1 + m_2\right] \times$$
$$\times \mathbb{P } \left ( K(t_1) = m_1, K(t_2) = m_1 + m_2 \right)=$$
$$=\sum_{m_1, m_2 = 0}^\infty   \mathbb{E} \exp \left \{is \sum_{j =m_1 + 1}^{m_1 + m_2} V_j \right \} \mathbb{P } \left ( K(t_1) = m_1, K(t_2) = m_1 + m_2 \right)=$$
$$=\sum_{m_1, m_2 = 0}^\infty  \mathbb{E} \exp \left \{is \sum_{j =m_1 + 1}^{m_1 + m_2} V_j \right \}  \mathbb{P } \left ( K(t_1) = m_1\right)\mathbb P \left(  K(t_2) - K(t_1) =  m_2 \right) =$$
$$=\sum_{m_1, m_2 = 0}^\infty   \left [ \phi_V(s) \right]^{m_2}  \mathbb{P } \left ( K(t_1) = m_1\right) \mathbb P \left(  K(t_2) - K(t_1) =  m_2 \right) =$$
$$= \sum_{m_2 = 0}^\infty   \left [ \phi_V(s) \right]^{m_2}   \mathbb P \left(  K(t_2) - K(t_1) =  m_2 \right) =$$
$$= \sum_{m_2 = 0}^\infty   \left [ \phi_V(s) \right]^{m_2}   \frac{(\lambda (t_2 - t_1))^{m_2}}{m_2! }e^{-\lambda (t_2 - t_1)} =$$$$=\exp \left \{ -\lambda (t_2 - t_1) (1 - \phi_V(s)) \right\}.$$}

\elena{Таким образом, характеристическая функция приращения сложного пуассоновского процесса на интервале $[t_1,\,t_2]$ равна
\begin{equation}
\label{char funct Q(t1) - Q(t2)}
     \phi_{Q(t_2) - Q(t_1) }(s) = \exp \left\{ -\lambda (t_2 - t_1) (1 - \phi_V(s) \right\}.
\end{equation}
В частности, если положить $t_1 = 0$, $t_2 = t$, получаем характеристическую функцию сечения
\begin{equation}
\label{char funct Q}
    \phi_{Q(t) }(s) = \mathbb{E} e^{is Q(t)} = \exp \left\{ -\lambda (t) (1 - \phi_V(s) )\right\}.
\end{equation}
}
\elena{
Из формул \eqref{char funct Q(t1) - Q(t2)} и \eqref{char funct Q} можно сделать вывод, что сложный пуассоновский процесс является процессом со стационарными приращениями (иначе говоря, является однородным по времени): $$Q(t_2) - Q(t_1) \eg{\overset{d}{=}} Q(t_2 - t_1)$$ для любых $0\le t_1<t_2$.
}

\elena{
Вычислим теперь совместную характеристическую функцию вектора приращений  $\left ( Q(t_2)-Q(t_1), Q(t_3) - Q(t_2),\ldots, Q(t_n) - Q(t_{n-1}) \right)$, где $t_1 < t_2 < \dots < t_n$, и увидим, что она распадается в произведение характеристических функций компонент вектора (что равносильно независимости компонент рассматриваемого вектора приращений).
$$\phi_{Q(t_2) - Q(t_1), Q(t_3) - Q(t_2),\ldots, Q(t_n) - Q(t_{n-1}) }(s_1,\ldots,s_n) = $$
$$=\mathbb{E} e^{i \sum_{k=1}^n s_k \left (Q(t_{k}) - Q(t_{k-1})\right)}= \mathbb{E} \exp\left \{i \sum_{k=1}^n s_k \sum\limits_{j =K(t_{k-1}) + 1}^{K(t_{k})} V_j \right \} =$$
$$=\sum_{\substack{ m_i = 0\\i=1,\ldots n}}^\infty  \mathbb{E} \left [ \exp \left \{i\sum_{k=1}^n s_k \sum\limits_{j =K(t_{k-1}) + 1}^{K(t_{k})} V_j\right \} \Bigm | \bigcap_{k=1}^n \left \{ K(t_k) = \sum_{j=1}^k m_j \right \} \right] \times$$
$$\times \mathbb{P } \left (  \bigcap_{k=1}^n \left \{ K(t_k) = \sum_{j=1}^k m_j \right \} \right)=$$
$$=\sum_{\substack{ m_i = 0\\
    i=1,\ldots n}}^\infty  \mathbb{E}  \exp \left \{i\sum_{k=1}^n s_k \sum\limits_{j =m_{k-1} + 1}^{m_k} V_j \right \} \mathbb{P } \left (   \bigcap_{k=1}^n \left \{ K(t_k) - K(t_{k-1}) = m_k \right \}  \right)=$$
$$=\sum_{\substack{ m_i = 0\\
    i=1,\ldots n}}^\infty   \prod_{k=1}^n \left[\phi_V(s_k) \right]^{m_k}\prod_{k=1}^n \mathbb{P } \left (    K(t_k) - K(t_{k-1}) = m_k  \right)=$$
$$=\prod_{k=1}^n\sum_{m_k = 0}^\infty \left[\phi_V(s_k) \right]^{m_k}  \mathbb{P } \left( K(t_k) - K(t_{k-1}) =  m_k\right) =$$
$$=\prod_{k=1}^n \sum_{m_k = 0}^\infty   \left [ \phi_V(s) \right]^{m_k}     \frac{(\lambda (t_k - t_{k-1}))^{m_k}}{m_k! }e^{-\lambda (t_k - t_{k-1})} =$$
$$=\prod_{k=1}^n\exp \left\{ -\lambda (t_k - t_{k-1}) (1 - \phi_V(s_k) )\right\}.$$
С учётом формулы \eqref{char funct Q(t1) - Q(t2)} получаем 
$$
\mathbb{E} e^{i \sum_{k=1}^n s_k \left (Q(t_{k}) - Q(t_{k-1})\right)}= \prod_{k=1}^n \mathbb{E} e^{i s_k \left (Q(t_{k}) - Q(t_{k-1})\right)}.$$}

Независимость приращений сложного пуассоновского процесса установлена. Теперь легко вычислить $$\mathbb{E}\overset{\circ}{Q}(t)\overset{\circ}{Q}(s)=\mathbb{E}\overset{\circ}{Q}(\min(t,s))(\overset{\circ}{Q}(\max(t,s))-\overset{\circ}{Q}(\min(t,s))+\overset{\circ}{Q}{^2}(\min(t,s))),$$ что с учетом независимости приращений дает
\begin{equation}\label{eq:CorrLevy}
\mathbb{E}\overset{\circ}{Q}(t)\overset{\circ}{Q}(s)=\mathbb{E}\overset{\circ}{Q}{^2}(\min(t,s))=\mathbb{D}Q(\min(t,s)),
\end{equation}
откуда сразу получаем $$R_Q(t,s)=\lambda \min(t,s)\mathbb{E}V_1^2.$$ Отметим, что формула~\eqref{eq:CorrLevy} для корреляционной функции является общей для всех процессов с независимыми приращениями.

\elena{Характеристическая функция сечения сложного пуассоновского процесса была нами получена выше, как частный случай (см. формулу~\eqref{char funct Q}). Выведем эту формулу еще раз независимо от предыдущих рассуждений.} 
Для этого воспользуемся формулой полной вероятности для математического ожидания 
$$\mathbb{E}X = \mathbb{E}(\mathbb{E}(X\,|\,Y))$$
и тем фактом, что характеристическая функция суммы независимых случайных величин равна произведению их характеристических функций:
\begin{eqnarray*}
\varphi_{Q(t)}(s) &=& \mathbb{E}e^{isQ(t)}=\mathbb{E}(\mathbb{E}(e^{isQ(t)} \,|\, K(t))) = \mathbb{E}(\varphi_{V_1}(s))^{K(t)}= \\ &=& \sum\limits_{k=1}^{+\infty} (\varphi_{V_1}(s))^k \frac{(\lambda t)^k}{k!}e^{-\lambda t}=\exp\left (\lambda t (\varphi_{V_1}(s) - 1)\right).
\end{eqnarray*}

Теперь поговорим о распределениях сечений сложного пуассоновского процесса.
\begin{definition}\label{def:complex_poisson_variable}
Распределение случайной величины $Q(1)$, т.е. распределение случайной величины
$$Q(1)=\sum\limits_{k=1}^{\xi}V_k,$$ где $\xi\in\mathrm{Po}(\lambda)$, а $V_k$ -- независимые одинаково распределенные случайные величины, называется \textit{сложным пуассоновским распределением}.
\end{definition}

Из формулы для характеристической функции $\varphi_{Q(t)}(s)$ следует, что $$\varphi_{Q(1)}(s)=\exp{(\lambda(\varphi_{V_1}(s)-1))},$$ причем $$\varphi_{Q(t)}=\varphi_{Q(1)}^t.$$

Сложное пуассоновское распределение зависит от двух вещей: параметра ${\lambda>0}$ и распределения $V_1$. В последующем нам пригодится следующее представление характеристической функции сложного пуассоновской случайной величины:
$$\varphi_{Q(1)}(s)=\exp{\int\limits_{-\infty}^{+\infty}\lambda(e^{ixs}-1)\,dF_{V_1}(x)}=\exp{\int\limits_{-\infty}^{+\infty}(e^{ixs}-1)\,\mu(dx)},$$ где ${\varphi_{V_1}(s)=\int_{-\infty}^{+\infty}e^{ixs}\,dF_{V_1}(x)}$ и ${\mu(dx)=\lambda dF_{V_1}(x)}$. Теперь можно сказать, что сложное пуассоновское распределение определяется лишь только одной мерой $\mu(x)$. Обозначать принадлежность этому распределению поэтому будем так: ${Q(1)\in\mathrm{CPo}(\mu)}$, обозначение происходит от слов \textit{Compound Poisson Process}.

Обычное пуассоновское распределение является частным случаем сложного пуассоновского распределения: ${\mathrm{Po}(\lambda) = \mathrm{CPo}(\lambda\delta(x-1))}$, где $\delta(x-1)$ -- дельта-функция. Если $\xi\in\mathrm{Po}(\lambda)$, то $a\xi\in\mathrm{CPo}(\lambda\delta(x-a))$ для любой неслучайной величины $a$. Пусть даны неслучайные числа $\{a_k\}$, $k=1,\dots,n$ и независимые случайные величины $X_k\in\mathrm{CPo}(\mu_k)$. Тогда
\begin{eqnarray*}
\varphi_{\sum a_k X_k}(s) &=& \prod\limits_{k=1}^n \varphi_{a_k X_k}(s)\!=\!\exp\left( \sum\limits_{k=1}^n \int\limits_{-\infty}^{+\infty}\lambda_k(e^{ixs}\!-\!1)\,\delta(x\!-\!a_k)\,dx \right)\!= \\ &=& \exp\left(\int\limits_{-\infty}^{+\infty}(e^{ixs}-1)\sum\limits_{k=1}^{n}\lambda_k\delta(x-a_k)\,dx\right),
\end{eqnarray*}

Рассмотрим теперь несколько частных случаев сложного пуассоновского процесса, они отличаются только выбором случайных величин $V_i$. Если $V_i = 1$ п.н., то ${Q(t)=K(t)}$, т.е. обычный пуассоновский процесс. Если ${V_i\in\mathrm{Be}(p)}$, то из характеристической функции получаем, что $Q(t)$ имеет распределение Пуассона с параметром $\lambda p t$. \elena{На самом деле, этого достаточно, чтобы утверждать, что такой процесс является пуассоновским, поскольку, как будет написано ниже в разделе \ref{subsec:LevyProcess} процесс Пуассона является процессом Леви, для описания которого достаточно знать распределение только одного сечения (см. Теорему \ref{th:PP-LP} и Замечание  1 в конце раздела \ref{subsec:LevyProcess}. 
Однако, рассуждения ниже приводятся здесь для полноты картины.}

С учетом независимости приращений процесса $Q(t)$ получаем, что это снова пуассоновский процесс, но с интенсивностью $\lambda p$. Действительно, из $K(0)=0$ п.н. следует $Q(0)=0$ п.н. Независимость приращений $Q(t)$ была показана выше. Осталось показать, что $Q(t+u)-Q(t))\in\mathrm{Po}(\lambda p u)$, чтобы установить, что $Q(t)$ является пуассоновским процессом с интенсивностью $\lambda p$. Напомним, что характеристическая функция случайной величины $\xi \in\mathrm{Po}(\lambda)$ равняется $\varphi_\xi(s) = \exp\left(\lambda (e^{is}-1)\right)$, а для $\eta \in \mathrm{Be}(p)$ характеристическая функция имеет вид $\varphi_\eta(s) = 1 + p(e^{is}-1)$. Поэтому
$$
    \varphi_{Q(t)}(s) = \exp\left(\lambda p t(e^{is}-1)\right).
$$
Рассмотрим произвольные ${t,u \ge 0}$ и вычислим характеристическую функцию случайной величины $Q(t+u) - Q(t)$. Из независимости приращений сложного пуассоновского процесса получаем
\begin{eqnarray*}
    \varphi_{Q(t+u)}(s) &=& \EE e^{isQ(t+u)} = \EE\left(e^{isQ(t)}e^{is(Q(t+u)-Q(t))}\right)=\\
    &=& \EE e^{isQ(t)} \EE e^{is(Q(t+u)-Q(t))} = \varphi_{Q(t)}(s)\varphi_{Q(t+u)-Q(t)}(s),
\end{eqnarray*}
откуда следует, что
$$
    \varphi_{Q(t+u)-Q(t)}(s) = \exp\left(\lambda p u(e^{is}-1)\right),
$$
что соответствует характеристической функции пуассоновской случайной величины с параметром $\lambda p u$. Распределение случайной величины\footnote{Аналогичное утверждение выполнено и для случайных векторов.} однозначно задается ее характеристической функцией (см., например, в~\cite{Gnedenko}), а значит, $Q(t+u) - Q(t) \in \mathrm{Po}(\lambda p u)$.

Итак, мы показали, что при ${V_i\in\mathrm{Be}(p)}$ процесс $Q(t)$ является пуассоновским с интенсивностью $\lambda p$. В этом случае говорят, что процесс Пуассона \textit{устойчив по отношению к случайному прореживанию}. Мо\-жно дать следующую интерпретацию: каждое событие исходного процесса $K(t)$ независимо от других с вероятностью $p$ оставляем, а с вероятностью $1-p$ удаляем. Такая процедура называется \textit{случайным прореживанием}.

Если $V_i$ равновероятно принимают значения $+1$ и $-1$, то процесс $Q(t)$ называют \textit{рандомизированным случайным блужданием}. 
Распределение координаты частицы при этом выражается через функции Бесселя (см.~\cite[Том~2, гл.~II, \S 7]{Feller}).

\subsection{Безгранично делимые случайные величины}

Теперь перейдем к более общему классу процессов, который включает в себя и винеровский, и пуассоновский процессы и сложный пуассоновский процесс -- процессам Леви. Предварительно введем несколько вспомогательных определений.

\begin{definition}
Случайную величину $X$ будем называть \textit{безгранично делимой}, если для любого натурального ${n\ge1}$ существует набор $\{X_{kn}\}_{k=1}^{n}$ независимых одинаково распределенных случайных величин, такой, что 
$$X \overset{d}{=} \sum_{k=1}^{n}X_{kn}.$$
\end{definition}

Для безгранично делимых случайных величин справедливы следующие две теоремы, которые мы приводим без доказательства. С~доказательствами можно познакомиться в~\cite[гл.~I, \S~3]{GikhmanSkorohod1}. Эти теоремы дают эквивалентные представления характеристических функций безгранично делимых случайных величин. В зависимости от задачи удобно пользоваться разными представлениями, которые суть преобразование компонент триплета (см. далее). Также читателю могут быть полезны книги~\cite{GikhmanSkorohod2, SatoBook1999}.

\begin{theorem}[ (Леви--Хинчина)]
\label{Levy-Hinchin}
\textit{Случайная величина $Y$ является безгранично делимой тогда и только тогда, когда логарифм ее характеристической функции $\varphi_Y(s)$ имеет вид}
\begin{equation}\label{eq:LeviKhinchinFormula}
    g(s) = \ln\varphi_Y(s) = ibs - \frac{\sigma^2s^2}{2} + \int\limits_{-\infty}^{+\infty}\left(e^{isx}-1-isx\mathsf{I}(|x|<1)\right)\nu(dx),
\end{equation}
\textit{где ${b\in\mathbb{R}}$ и ${\sigma^2 \ge 0}$ -- некоторые числа, а $\nu(x)$ -- некоторая мера на вещественной оси $\mathbb{R}$ со свойствами $\nu(\{0\})=0$ и $\int_{\mathbb{R}}\min(1,x^2)\nu(dx)<\infty$. Тройка $(b,\sigma^2,\nu)$ для каждой безгранично делимой случайной величины определяется единственным образом.}
\end{theorem}

\begin{theorem}
\label{char funct 2 for i.d.r.v.}
\textit{Случайная величина $Y$ является безгранично делимой тогда и только тогда, когда логарифм ее характеристической функции $\varphi_Y(s)$ имеет вид}
\begin{equation}
\label{eq:char funct 2 for i.d.r.v.}
    g(s) = \ln\varphi_Y(s) = i\bar b s - \frac{\sigma^2s^2}{2} + \int\limits_{-\infty}^{+\infty}\left(e^{isx}-1-\frac{isx}{x^2+1}\right)\frac{1+x^2}{x^2}\,d \bar\nu(x), 
\end{equation}
\textit{где ${\bar b\in\mathbb{R}}$, ${\sigma^2 \ge 0}$ -- некоторые числа, а ${\bar\nu(x)}$ -- некоторая конечная (то есть ${\bar\nu(\mathbb{R}) < \infty}$) мера на вещественной оси $\mathbb{R}$, такая, что ${\bar\nu(\{0\})=0}$.
Тройка $(\bar b,\sigma^2,\bar\nu)$ для каждой безгранично делимой случайной величины определяется единственным образом.}

\textit{Если случайная величина $Y$ имеет конечный второй момент, то формула \eqref{eq:char funct 2 for i.d.r.v.} может быть уточнена следующим образом}:
\begin{equation}
\label{eq:char funct 3 for i.d.r.v.}
    g(s) = \ln\varphi_Y(s) = i\Tilde{b}s - \frac{\Tilde\sigma^2s^2}{2} + \int\limits_{-\infty}^{+\infty}\left(e^{isx}-1-isx\right)\frac{1}{x^2}\Tilde{\nu}(dx), 
\end{equation}
\textit{где $\Tilde{\nu}$ -- конечная (то есть ${\Tilde\nu(\mathbb{R}) < \infty}$) мера  и $\Tilde{\nu}(\{0\}) = 0$. Тройка $(\Tilde b,\Tilde\sigma^2,\Tilde\nu)$ для каждой безгранично делимой случайной величины с конечным вторым моментом определяется  единственным образом.}
\end{theorem}

\begin{theorem}\label{th:SuffCondIDRV}
\textit{Пусть для каждого $n\ge1$ набор $\{X_{kn}\}_{k=1}^n$ состоит из независимых одинаково распределенных случайных величин и, кроме того, $$\sum_{k=1}^{n}X_{kn} \xrightarrow[n \to \infty]{d} X.$$ Тогда $X$~-- безгранично делимая случайная величина.}
\end{theorem}




Рассмотрим важнейшие примеры безгранично делимых случайных величин.

\begin{example}
Пусть случайная величина $X$ принимает одно и то же значение $m$ при любом исходе, т.е. ${X \equiv m}$. По определению, это безгранично делимая случайная величина, т.к. для любого ${n\ge1}$ ее можно представить как сумму $n$ чисел $m/n$, которые являются независимыми и одинаково распределенными случайными величинами.
	
Случайная величина $X$ является безгранично делимой и согласно теореме Леви--Хинчина. Действительно, ее характеристическая функция $$\varphi_X(s)=e^{ims}$$ имеет вид~\eqref{eq:LeviKhinchinFormula} при $b=m$, $\sigma^2=0$ и $\nu(x)=0$. \EndEx
	
\end{example}
    
	
	
\begin{example}
Пусть случайная величина $X$ имеет распределение Пуассона с параметром $\lambda>0$, т.е. $X\in \Po(\lambda)$ и $$\mathbb{P}(X=k)=\frac{\lambda^k}{k!}e^{-\lambda}, \ k=0,1,2,\dots$$

Покажем, что эта случайная величина является безгранично делимой по определению. Для каждого $n\ge1$ рассмотрим набор $\{X_{kn}\}_{k=1}^n$ независимых одинаково распределенных случайных величин, имеющих распределение $\Po(\lambda/n)$. Тогда, как известно, $$\sum\limits_{k=1}^n X_{kn}\in\Po\left(\sum\limits_{k=1}^n \lambda/n\right)=\Po(\lambda),$$ значит, $X\overset{d}{=}\sum_{k=1}^n X_{kn}$, т.е. $X\in\Po(\lambda)$ -- безгранично делимая случайная величина.
	
	
Легко видеть, что характеристическая функция случайной величины $X$ имеет вид~\eqref{eq:LeviKhinchinFormula}, где ${b=0}$, ${\sigma^2=0}$, ${\nu(x)=\lambda\delta(x-1)}$, и ${\delta(x)}$ -- дельта-функция Дирака. \EndEx
\end{example}

\begin{example}\label{ex:normal_inf_div}
Покажем теперь, что случайная величина, имеющая стандартное нормальное распределение ${X \in \mathrm{N}(0,1)}$, т.е. непрерывная случайная величина с плотностью распределения 
$$f(x)=\frac{1}{\sqrt{2\pi}}e^{-x^2/2}, \ x\in\mathbb{R},$$
тоже является безгранично делимой. Для этого рассмотрим набор независимых одинаково распределенных по закону
$\mathrm{N}(0,1/n)$ случайных величин $\{X_{kn}\}_{k=1}^n$. 
Как известно, 
$$
\sum\limits_{k=1}^n X_{kn} \in \mathrm{N} \left(0,\sum\limits_{k=1}^n 1/n \right)=\mathrm{N}(0,1),
$$ 
значит, 
$X\overset{d}{=}\sum_{k=1}^n X_{kn}$, т.е. ${X \in \mathrm{N}(0,1)}$ -- безгранично делимая случайная величина. 

	
Свойство безграничной делимости можно было бы установить и следуя теореме Леви--Хинчина, заметив, что характеристическая функция $X$ имеет вид~\eqref{eq:LeviKhinchinFormula} при $b=0$, $\sigma^2=1$, $\nu(x)=0$. \EndEx
\end{example}

Безгранично делимыми случайными величинами являются также случайные величины с гамма-распределением, отрицательным биномиальным распределением, геометрическим распределением и многие другие.  Случайные величины с  биномиальным~$\Bi(n,p)$, равномерным~$\mathrm{R}(a,b)$ и всяким другим невырожденным распределением с ограниченным носителем не являются безгранично делимыми. 

Теперь дадим некоторые пояснения вышеприведенным теоремам. Теорема~\ref{th:SuffCondIDRV} утверждает, что для сумм независимых одинаково распределенных случайных величин предельным может быть только безгранично делимое распределение. С другой стороны, теоремы~\ref{Levy-Hinchin}  и~\ref{char funct 2 for i.d.r.v.} предоставляют явный вид характеристической функции произвольного безгранично делимого распределения. Значит, предельное распределение может иметь характеристическую функцию только вида~\eqref{eq:LeviKhinchinFormula} и характеризуется тремя параметрами: $b$, $\sigma^2$ и $\nu(x)$ либо эквивалентного вида \eqref{eq:char funct 2 for i.d.r.v.}  с преобразованным триплетом: $\bar b$, $\sigma^2$ и $\bar\nu(x)$ (отметим еще раз, что еще один вид характеристической функции, данный в формуле \eqref{eq:char funct 3 for i.d.r.v.}, является уточнением \eqref{eq:char funct 2 for i.d.r.v.} при условии конечности второго момента). Теперь рассмотрим, например, представление Леви--Хинчина. Тогда при $b\ne0$, ${\sigma^2=0}$, ${\nu(x)=0}$ предельное распределение представляет собой распределение постоянной случайной величины, константы. При ${b\in\mathbb{R}}$, ${\sigma^2 > 0}$ и ${\nu(x)=0}$ получается нормальное распределение с математическим ожиданием $b$ и дисперсией $\sigma^2$. Наконец, при ${b=0}$, ${\sigma^2=0}$ и ${\nu(x)=\lambda\delta(x-1)}$ получается распределение Пуассона с параметром $\lambda$, а при ${b=\int_{-1}^{1}x\,\nu(dx)}$, ${\sigma^2=0}$ и произвольной мере $\nu(x)$ (равной какой-либо вероятностной мере с точностью до множителя), получается сложное пуассоновское распределение. \eduard{В некотором смысле, любое другое предельное распределение, отвечающее какому-то набору $(b,\sigma^2,\nu(x))$, <<складывается>> из двух <<базисных>> распределений: нормального распределения $\mathrm{N}(b,\sigma^2)$ и сложного распределения Пуассона $\mathrm{CPo}(\nu)$ или предела этих распределений.}

Чрезвычайно интересно отметить, что вообще \textit{класс безгранично делимых распределений совпадает с классом пределов последовательностей сложных пуассоновских распределений}~\cite[Том 2, Глава IX, §5, Теорема 2]{Feller}. Даже нормальное распределение является пределом последовательности сложных пуассоновских распределений. За подробной информацией о безгранично делимых случайных величинах и их свойствах мы также отправляем читателя к монографии~\cite{SatoBook1999}.

Нам важно отметить, что при весьма общих предположениях о случайных последовательностях их суммы сходятся к одному из трех вышеописанных распределений.
Для того, чтобы это показать, удобно использовать так называемую \textit{теорему непрерывности}. Эта теорема связывает сходимость по распределению и поточечную сходимость характеристических функций. Так как эта теорема является мощным инструментом для исследования свойств сходимости случайных последовательностей, мы приводим ее формулировку здесь.

\begin{theorem}[ (непрерывности~\cite{ShiryaevT1})]\label{th:ContTheorem}
\textit{Последовательность случайных величин $\{X_k\}_{k=1}^{\infty}$ сходится по распределению к случайной величине $X$ тогда и только тогда, когда в каждой точке $s\in\mathbb{R}$ $$\varphi_{X_{k}}(s) \to \varphi_X(s), \ k\to\infty,$$ где $\varphi_{X_{k}}(s)$ -- характеристическая функция $X_k$, а $\varphi_X(s)$ -- характеристическая функция $X$.}
\end{theorem}

Рассмотрим произвольную последовательность независимых одинаково распределенных случайных величин $\{X_k\}_{k=1}^\infty$ с математическим ожиданием $m$. Характеристическую функцию случайной величины $X_k$ обозначим за ${\varphi_{X_k}(t)=\mathbb{E}\exp(i t X_k)}$. Так как случайные величины независимы, то характеристическая функция суммы ${S=\sum_{k=1}^n X_k}$ будет равна произведению характеристических функций слагаемых, а так как распределения равные, то -- ${\varphi_S(t)=(\varphi_{X_1}(t))^n}$. Напомним, что для произвольной случайной величины $\xi$ есть равенство $\varphi_{a\xi}(t)=$\linebreak $=\varphi_{\xi}(at)$, $\forall a\in\mathbb{R}$. Далее получаем, что для каждого $t\in\mathbb{R}$
$$\varphi_{S/n}(t)=\varphi_S(t/n)=(\varphi_{X_1}(t/n))^n=\left( 1 + itm/n + o(1/n) \right)^n, \ n\to\infty,$$ откуда получаем
$$\varphi_{S/n}(t) \to \exp(itm), \ n\to\infty.$$
Так как $\exp(itm)$ -- это характеристическая функция случайной величины $X\equiv m$, то по теореме непрерывности получаем, что $1/n\sum_{k=1}^n X_k$ сходится по распределению к константе $m$. Этот факт известен в литературе как \textit{закон больших чисел по Хинчину}.

Рассмотрим теперь последовательность $\{X_{kn}\}_{k=1}^n$ независимых и одинаково распределенных случайных величин ${X_{kn}\in\mathrm{Be}(p_n)}$, причем $p_n\to 0$, $np_n \to \lambda > 0$.
$$
\varphi_{X_{kn}}(s) = 1 + p_n\left(e^{is}-1\right).
$$
Пусть ${S_n = \sum_{k=1}^{n}X_{kn}}$. Тогда
$$
\varphi_{S_n}(s) = \left(\varphi_{X_{1n}}(s)\right)^n = \left(1+p_n(e^{is}-1)\right)^n \to e^{\lambda(e^{is}-1)}, \ n\to \infty.
$$
Это значит, что $S_n$ сходится по распределению к распределению Пуассона с параметром $\lambda$. Этот факт известен в литературе как \textit{теорема Пуассона}.

Рассмотрим, наконец, последовательность независимых и одинаково распределенных случайных величин $\{X_k\}_{k=1}^n$ с конечным вторым моментом $\mathbb{E}X_k^2 < \infty$. Обозначим математическое ожидание ${m=\mathbb{E}X_1}$, а дисперсию  ${\sigma^2=\mathbb{D}X_1}$. Рассмотрим случайную последовательность $$\eta_n = \frac{1}{\sqrt{\sigma^2/n}}\left(\sum_{k=1}^n X_k/n - m\right) = \sum\limits_{k=1}^n \left( X_k \cdot \frac{1}{\sigma \sqrt{n}} - \frac{m}{\sigma\sqrt{n}} \right).$$ Все коэффициенты здесь подобраны так, чтобы $\mathbb{E}\eta_n=0$, а $\mathbb{D}\eta_n = 1.$ Отсюда можно получить, что $$\varphi_{\eta_n}(t) = \left( 1 - \frac{t^2}{2n} + o\left(\frac{1}{n}\right)\right)^n, \ n\to\infty,$$ т.е. $\varphi_{\eta_n}(t)\to\exp(-t^2/2)$, $n\to\infty$, что означает сходимость к стандартному нормальному распределению. Этот факт в литературе известен как \textit{центральная предельная теорема}.




\subsection{Процессы Леви}
\label{subsec:LevyProcess}

Рассмотрим класс случайных процессов $\{X(t), t \ge 0\}$ с независимыми стационарными приращениями (у которых распределение приращения $X(t+s) - X(t)$ не зависит от $t$) и выходящими из нуля, то есть ${X(0) = 0}$ п.н.  Пусть $F_t$ -- распределение $X(t)$. Тогда для любых $t,s > 0$
$$
F_{t+s} = F_t \star F_s,
$$
где $\star$ означает свертку распределений. В частности, для любого $t>0$ и $n$ получается
$F_t = (F_{t/n})^{n \star}$, т.е. $F_t$ -- распределение безгранично делимой случайной величины (безгранично делимое распределение). Таким образом, для описания случайного процесса, выходящего из нуля, с независимыми стационарными приращениями требуется безгранично делимое распределение $F_t$. Кстати говоря, если отказаться от условия $X(0) = 0$, то для задания процесса дополнительно потребуется распределение $X(0)$.

Справедливо и обратное утверждение: всякое безгранично делимое распределение с характеристической функцией вида $e^{t \psi}$, где $\psi$ не зависит от $t$, определяет случайный процесс, выходящий из нуля, с независимыми стационарными приращениями. Этот результат можно обобщить на случай нестационарных приращений потребовав взамен стохастическую непрерывность: для всех $t \ge 0$ $X(t+h) - X(t)$ стремится по вероятности к нулю при $h\to 0$~\cite{Feller}.

В связи с описанным выше введем важный класс процессов -- процессов Леви.

\begin{definition}
Случайный процесс $\{X(t), t \ge 0\}$ будем называть \textit{процессом Леви}, если
\begin{enumerate}[topsep=3pt,itemsep=0pt]
	\item $X(0)=0$ почти всюду,
	\item $X(t)$ -- процесс с независимыми приращениями,
	\item для любых $t,s \ge 0$ случайная величина $X(t+s) - X(t)$ имеет распределение, не зависящее от $t$ (стационарность приращений),
	\item $X(t)$~-- стохастически непрерывный процесс, т.е. $$\forall t\ge0 \  \ X(t+\varepsilon) \xrightarrow[\varepsilon\to 0]{\PR} X(t).$$
	\end{enumerate}
\end{definition}

Иногда в определении процесса Леви дополнительно требуют, чтобы с вероятностью единица траектории процесса Леви были непрерывны справа при ${t\ge 0}$ и имели конечный предел слева при ${t>0}$ (c\`adl\`ag функции от французского термина continue \`a  droite, limite \`a gauche, или согласно английской терминологии rcll -- right continuous, left-hand limits). На самом деле это дополнение несущественно, так как можно показать, что всегда существует модификация процесса Леви с данным условием.

Отметим, что винеровский, пуассоновский, сложный пуассоновский процессы являются процессами Леви. Можно даже ввести следующие эквивалентные определения.
\begin{theorem}
\label{th:PP-LP}
\textit{Случайный процесс $\{K(t), t \ge 0\}$ является пуассоновским процессом с параметром  $\lambda >0$, если}
\begin{itemize}
    \item $K(t)$ -- \textit{процесс Леви},
    \item \textit{для любого $t>0$ сечение $K(t)$ имеет распределение $\mathrm{Po}(\lambda t)$}.
\end{itemize}
\end{theorem}

\begin{theorem}
\textit{Случайный процесс ${\{Q(t), t \ge 0\}}$ является сложным пуассоновским процессом $\sum_{i=1}^{K(t)}V_i$, где $K(t)$ -- пуассоновский процесс интенсивности $\lambda>0$ и $V_i$ определяются вероятностной мерой $\mathbb{P}_{V_1}$, если}
\begin{itemize}
    \item $Q(t)$ -- \textit{процесс Леви},
    \item \textit{для любого $t>0$ сечение $Q(t)$ имеет распределение $\mathrm{CPoiss}(\lambda t \mathbb{P}_{V_1})$}.
\end{itemize}
\end{theorem}

\begin{theorem}
Случайный процесс ${\{W(t), t \ge 0\}}$ является \ag{\textit{винеровским процессом}}, если
\begin{itemize}
    \item $W(t)$ -- \textit{процесс Леви},
    \item \textit{для любого $t>0$ сечение $W(t)$ имеет распределение $\mathrm{N} (0, t)$}.
\end{itemize}
\end{theorem}

Несложно заметить, что для процессов Леви справедливо функциональное уравнение относительно математического ожидания:
$$\mathbb{E} X(t+s) = \mathbb{E} X(t) + \mathbb{E} X(s)$$
для любых $t,s > 0$. \shmaxg{Это -- частный вид известного в литературе функционального \textit{уравнения Коши} или \textit{уравнения Гамеля}, в общем случае записывающегося как уравнение $$f(x+y)=f(x)+f(y)$$ на функцию $f$. Читателю предлагается доказать, что без каких-либо дополнительных предположений на функцию $f$ в рациональных точках $x$ решения этого уравнения имеют вид ${f(x) = c x}$, где ${c\in\mathbb{Q}}$ -- произвольная постоянная. Если же дополнительно предположить, что функция $f$ непрерывная, то отсюда следует, что решения этого уравнения имеют вид ${f(x)=c x}$, где ${c\in\mathbb{R}}$ -- произвольная постоянная. Поэтому если процесс Леви обладает непрерывным математическим ожиданием, то $\mathbb{E}X(t)= c t$ для некоторой постоянной $c$.}

Если процесс Леви имеет конечные вторые моменты, то аналогичное уравнение можно записать для дисперсии, откуда ${\mathbb{D}X(t) = \sigma^2 t}$. При этом корреляционная функция равна $R_X(t,s) = \sigma^2 \min(t,s)$.


\begin{theorem}
\textit{Справедливо следующее}.

1) \textit{Если $X(t)$, $t\ge 0$ -- процесс Леви, то для любого $t > 0$ сечение $X(t)$ -- безгранично делимая случайная величина}.

2) \textit{Если $F$ -- безгранично делимое распределение, то существует процесс Леви $X(t)$, $t \ge 0$, такой, что $X(1)$ имеет распределение $F$}.

3) \textit{Если $X(t)$, $Y(t)$, $t \ge 0$ -- два процесса Леви таких, что распределения $X(1)$ и $Y(1)$ совпадают, то конечномерные распределения процессов $X(t)$ и $Y(t)$ совпадают}.
\end{theorem}
\begin{proof}
Для доказательства первого пункта заметим, что для произвольного натурального $N$ выполнено
$$X(t) - X(0) = \sum\limits_{j=1}^N \left(X\left(\frac{jt}{N}\right) - X\left(\frac{(j-1)t}{N}\right) \right).$$
По определению процесса Леви все слагаемые в выписанной сумме независимы и одинаково распределены. Так как, кроме того, $X(0) = 0$ почти наверное, то по определению получаем, что $X(t)$~--- безгранично делимая случайная величина.

\setlength{\parskip}{0pt}

Далее мы приведём только схему доказательства пункта 2 (подробности см. в~\cite[Chapter 2, \S~7, Theorem 7.10]{SatoBook1999}), пункт~3 же следует из пункта~2. Во-первых, в качестве $X(1)$ достаточно взять случайную величину, имеющую распределение $F$, а в качестве $X(0)$~--- константу $0$. Оказывается, что таким образом мы задали распределения сечений во всех положительных рациональных точках (действительно, достаточно воспользоваться представлением из пункта 1, чтобы задать распределение сечений в рациональных точках из $[0,1]$, а затем воспользоваться тем, что распределение приращения $X(t+s) - X(t)$ зависит только от $s$, чтобы задать распределение сечений в рациональных точках из $[1,2]$, затем из $[2,3]$ и так далее). Пользуясь стохастической непрерывностью, можно доопределить процесс и во всех иррациональных точках. \EndProof
\end{proof}
\setlength{\parskip}{0pt}

Как мы убедились класс процессов Леви тесно связан с классом безгранично делимых случайных величин. Вид характеристических фунций сечений процесса Леви обобщает теоремы~\ref{Levy-Hinchin} и~\ref{char funct 2 for i.d.r.v.} и дается в следующей теореме (с доказательством можно познакомиться, например, в~\cite[гл.~I, \S~3, теорема~4]{GikhmanSkorohod1}).

\begin{theorem}\label{Levy}
\textit{Пусть $X(t)$ -- процесс Леви. Тогда характеристическая функция }
 $\varphi_{X(t + \tau) - X(\ag{t})}(s)=\exp\left(\tau g(s)\right)$,
\textit{где $g(s)$ дается формулой~\eqref{eq:LeviKhinchinFormula} или~\eqref{eq:char funct 2 for i.d.r.v.}. 
Если процесс $X(t)$ является процессом второго порядка, то для $g(s)$ имеет место  представление} \eqref{eq:char funct 3 for i.d.r.v.}.
\end{theorem}




Приведем несколько частных случаев процессов Леви.

а) Если ${\sigma = 0}$, ${\nu \equiv 0}$, то ${\varphi_{X(t)}(s)=\exp\left(i t b s\right)}$ -- это характеристическая функция вырожденного распределения, то есть $X(t) = b t$.

б) Если ${\nu \equiv 0}$, то ${\varphi_{X(t)}(s)=\exp\left(i t b s - t\sigma^2s^2/2 \right)}$. В этом случае приращения $X(t + \tau) - X(\ag{t})$ имеют нормальное распределение с математическим ожиданием $b\ag{\tau}$ и дисперсией $\sigma^2 \ag{\tau}$. Таким образом, $X(t)$, $t\ge 0$ -- гауссовский процесс и его можно представить как $X(t) = b t +\sigma W(t)$, $t\ge 0$, где $W(t)$ -- винеровский процесс.

в) Если ${b = 0}$, ${\sigma = 0}$, мера $\nu$ сосредоточена в точке ${x_0 \in (-1, 1)}$ и ${\nu(\{ x_0\}) = \lambda > 0}$, то $\varphi_{X(t)}(s)=\exp\left(t \lambda (e^{i s x_0} - 1 - i s x_0)\right)$. Процесс $X(t)$ можно представить как $X(t) = x_0 ( K(t) - \lambda t)$, $t \ge 0$ где $K(t)$ -- пуассоновский процесс.

г) Если $\sigma = 0$ и $\int_{-1}^1 d \nu(x) < \infty$, то  $$\varphi_{X(t)}(s)=\exp\left(t\cdot \left[ib_1s + \lambda \int\limits_{-\infty}^{+\infty}\left(e^{isx}-1 \right)\,d\nu_1(x)\right]\right),$$
где $b_1 = b+\int_{-1}^1 d \nu(x)$, $\lambda = \int_{-\infty}^\infty\,d \nu(x)$, $\nu_1 = \lambda^{-1} \nu$ -- вероятностная мера. Иначе говоря,
$$\varphi_{X(t)}(s)=e^{i t b s} \sum_{n=0}^\infty e^{-\lambda t} \frac{(\lambda t)^n}{n!} \left [  \int\limits_{-\infty}^\infty e^{i s x}\,d \nu_1(x) \right]^n,$$
то есть $X(t) = b_1 t + \sum_{n=1}^{K(t)} \xi_n$, $t\ge 0$, где $K(t)$ -- пуассоновский процесс интенсивности $\lambda$, $\{ \xi_n \}_{n=1}^\infty$ -- последовательность независимых случайных величин с одинаковым распределением $\nu_1$ (этот случай соответствует сложному пуассоновскому процессу).

Ранее мы выяснили, что безгранично делимая случайная величина представляется в виде суммы константы, гауссовской случайной величины и сложной пуассоновской случайной величины (при  ${\nu(\R)<\infty}$ в \eqref{eq:LeviKhinchinFormula}) или, в общем случае, предела сложных пуассоновских случайных величин (при $\nu(\R)\le\infty$). Если воспользоваться аналогичными рассуждениями, то из примера выше мы получаем, что процесс Леви с конечной мерой $\nu$ (см. теоремы~\ref{Levy-Hinchin},~\ref{Levy}) можно представить в виде суммы процесса вида $X(t) = bt$, промасштабированного винеровского процесса и сложного пуассоновского процесса. 

В качестве примера процесса Леви с ${\nu(\R) = \infty}$ рассмотрим \textit{субординатор Морана}. По определению это процесс Леви, отвечающий гамма-распределению: $\nu(dx) = x^{-1}e^{-x}dx$ (см.~\cite[гл. 5, задача 42]{StochAn2016}). При помощи субординатора Морана можно генерировать последовательности случайных величин, имеющие распределение Пуассона–Дирихле. Это распределение упорядоченных по убыванию нормированных скачков субординатора Морана (детали см. в~\cite{Kingman}). Распределение Пуассона–Дирихле часто применяется в популяционной генетике и экономике (оно является равновесным распределением
для ряда эволюционных моделей).

\elena{\textbf{Замечание 1.} Стоит еще раз подчеркнуть, что процессы Леви в некотором смысле просты в своем описании. А именно, для их задания не нужно знать семейства конечномерных распределений. Они полностью описываются распределением одного сечения, например, в момент времени $t=1$. Действительно, если характеристическая функция сечения $X(1)$ равна $\mathbb{E} e^{is X(1)} = \phi(s)$, то характеристическая функция 
\begin{itemize}
    \item произвольного сечения равна 
$$\mathbb{E} e^{is X(t)} = \left( \phi(s) \right)^t,$$
\item приращения на интервале $[t_1, t_2]$
$$\mathbb{E} e^{is \left (X(t_2) - X(t_1)\right)} =\left( \phi(s) \right)^{t_2 - t_1},$$
\item векторы из приращений $X(t_1) - X(t_0)$, $ X(t_2) - X(t_1)$,  $X(t_3) - X(t_2)$,\ldots, $X(t_n) - X(t_{n-1})  $, где $0=t_0 = t_1 < t_2< \ldots <t_n$:
$$\mathbb{E} e^{i \sum_{k=1}^n s_k \left (X(t_{k}) - X(t_{k-1})\right)} =\prod_{k=1}^n \left( \phi(s_k) \right)^{t_{k} - t_{k-1}},$$
\item векторы из сечений
 $\left( X(t_1),X(t_2),\ldots, X(t_n)  \right)$, где $t_1 < t_2< \ldots <t_n$:
$$\mathbb{E} e^{i \sum_{k=1}^n s_k  X(t_k)} =\prod_{k=1}^n \left( \phi(\tilde s_k) \right)^{t_{k} - t_{k-1}},$$
где $\tilde s_k = \sum_{i=k}^n s_i$.
\end{itemize}
 Последняя формула следует из связи характеристических функций векторов, получающихся друг из друга с помощью линейного преобразования.
}

\section{Введение в стохастический анализ} 
\label{correlation}

Данный раздел представляет собой введение в стохастический анализ случайных процессов. Здесь рассматриваются случайные процессы, любое сечение которых является случайной величиной второго порядка, т.е. ${\forall t\in T}$ выполнено $\mathbb{E}X^2(t)<\infty$. Такие процессы мы называем \textit{случайными процессами второго порядка}. Кроме того, будем считать для определенности, что ${T = [0,+\infty)}$. Здесь вводятся понятия непрерывности, дифференцируемости и интегрируемости случайных процессов.

\subsection{Пространство $L_2$ случайных величин}
\label{gilbert}

В данном разделе мы приведем несколько свойств \textit{случайных величин второго порядка}, т.е. случайных величин с конечным вторым моментом: ${\Exp X^2 < \infty}$. Эти свойства понадобятся нам при построении корреляционной теории случайных процессов в следующем разделе.

Обозначим за $L_2$ множество случайных величин второго порядка, заданных на одном вероятностном пространстве. Билинейная функция $\langle\cdot,\cdot\rangle: L_2\times L_2 \to \Rbb$, такая, что ${\forall X,\,Y \in L_2 \hookrightarrow \langle X, Y\rangle = \Exp (XY)}$ определяет скалярное произведение на $L_2$. Чтобы это показать, проверим аксиомы скалярного произведения:
\begin{enumerate}
	\item[1.] ${\forall X \in L_2\hookrightarrow \langle X, X\rangle = \Exp X^2 \ge 0}$, т.к. $X^2 \ge 0$. Кроме того, $$\langle X, X\rangle = 0 \; \Leftrightarrow \; \Exp X^2 = 0 \; \Leftrightarrow \; X \overset{\text{п.н.}}{=} 0.$$
	\item[2.] $\forall X,\, Y \in L_2 \hookrightarrow \langle X, Y\rangle = \Exp (XY) = \Exp (YX) = \langle Y,X\rangle$.
	\item[3.] $\forall X,\, Y,\, Z \in L_2,\, \alpha,\, \beta \in \Rbb \hookrightarrow \langle\alpha X + \beta Y, Z\rangle = \alpha\langle X , Z\rangle + \beta\langle Y, Z\rangle$.
\end{enumerate}

Аксиомы скалярного произведения проверены, линейное пространство $L_2$ с введенным скалярным произведением является евклидовым. Отсюда вытекает несколько простых и полезных следствий, которые являются общими для всех евклидовых пространств. 

\begin{theorem}[ (неравенство Коши--Буняковского--Шварца)]
$$\forall X, \, Y \in L_2 \hookrightarrow \langle X, Y \rangle^2 \leqslant \langle X, X \rangle \langle Y, Y \rangle .$$
\end{theorem}

\begin{proof}
Выражение $\langle X+\alpha Y , X+\alpha Y \rangle$ является квадратичной функцией $\alpha$; при этом оно неотрицательно для каждого ${\alpha\in\mathbb{R}}$. Неравенство из утверждения теоремы -- лишь условие неотрицательности дискриминанта этой квадратичной функции. \EndProof
\end{proof}

Теперь перейдём к вопросу сходимости случайных величин из $L_2$. Введем евклидову норму, связанную со скалярным произведением: $\left\|X\right\|_2 = \sqrt{\langle X,X\rangle}$. Договоримся далее о следующих обозначениях:
$$
X_n \xrightarrow[n\to\infty]{L_2} X \; \Leftrightarrow \; \left\|X_n-X\right\|_2 \xrightarrow[n\to\infty]{} 0 \; \Leftrightarrow \; X = \underset{n\to\infty}{\operatorname{l.i.m.}}\,X_n,
$$
означающих, что последовательность случайных величин $\{X_n\}_{n=1}^\infty$ \\ сходится \textit{в среднем квадратичном} к случайной величине $X$.

\begin{theorem}[ (непрерывность скалярного произведения)]\label{th:ContScalarProd}
\textit{Если} $X_n \xrightarrow[n\to\infty]{L_2} X$, $Y_m \xrightarrow[m\to\infty]{L_2} Y$, \textit{то} $$ \langle X_n, Y_m \rangle \xrightarrow[n,m\to\infty]{} \langle X, Y \rangle .$$
\end{theorem}

\begin{proof}
Заметим, что согласно неравенству Коши--Буня\-ковского--Шварца
$$|\langle X_n+X,Y_m-Y\rangle|^2 \le \underset{\text{ограничено}}{\underbrace{\langle X_n+X,X_n+X \rangle}}\cdot \underset{\xrightarrow[m\to\infty]{}0}{\underbrace{\langle Y_m-Y,Y_m-Y \rangle}} \xrightarrow[n,m\to\infty]{} 0,$$
что означает $$\langle X_n+X,Y_m-Y\rangle \xrightarrow[n,m\to\infty]{} 0.$$
Аналогично доказывается, что 
$$\langle X_n-X,Y_m+Y\rangle \xrightarrow[n,m\to\infty]{} 0.$$
Пользуясь линейностью скалярного произведения, легко получаем
$$\underset{\xrightarrow[n,m\to\infty]{} 0}{\underbrace{\langle X_n+X,Y_m-Y\rangle + \langle X_n-X,Y_m+Y\rangle}} = 2\left(\langle X_n,Y_m\rangle - \langle X,Y\rangle  \right),$$
откуда следует $\langle X_n,Y_m \rangle \to \shmaxg{\langle X,Y\rangle}$ при $n,m\to\infty$, что и требовалось доказать. \EndProof
\end{proof}

Следующую теорему приводим без доказательства (доказательство можно найти, например, в~\cite{Borovkov1999}).

\begin{theorem}[ (полнота $L_2$)]\label{th:L2Complete}
\textit{Любая фундаментальная по Коши последовательность из $L_2$ сходится в среднем квадратичном к случайной величине из} $L_2$.
\end{theorem}

Мы же воспользуемся этой теоремой для доказательства следующей теоремы.

\begin{theorem}\label{th:ExistProccess}
\textit{Пусть для последовательности ${\{X_n\}_{n=1}^{\infty} \subset L_2}$ найдется такое $c\in\Rbb$, что для любых подпоследовательностей $\{X_{n_k}\}_{k=1}^{\infty}$ и  $\{X_{n_m}\}_{m=1}^{\infty}$ выполнено $ \langle X_{n_k} , X_{n_m} \rangle \xrightarrow[k,m\to\infty]{} c$. Тогда} $$\exists X\in L_2: \; X = \underset{n\to\infty}{\operatorname{l.i.m.}} X_n.$$
\end{theorem}

\begin{proof}
Рассмотрим произвольные подпоследовательности \\ $\{X_{n_k}\}_{k=1}^{\infty}, \{X_{n_m}\}_{m=1}^{\infty}\subset \{X_n\}_{n=1}^{\infty}$ и квадрат их разности
$$(X_{n_k} - X_{n_m})^2 = X_{n_k}^2 - 2X_{n_k}X_{n_m} + X_{n_m}^2.$$
Из условия теоремы следует, что существует $c\in\mathbb{R}$, такое что $$ \langle X_{n_k} , X_{n_m} \rangle \xrightarrow[k,m\to\infty]{} c, \ \langle X_{n_k} , X_{n_k} \rangle \xrightarrow[k,m\to\infty]{} c, \ \langle X_{n_m} , X_{n_m} \rangle \xrightarrow[k,m\to\infty]{} c.$$
Отсюда получаем, что
$$\underset{k,m\to\infty}{\lim} \langle X_{n_k} - X_{n_m}, X_{n_k} - X_{n_m} \rangle = c - 2c + c = 0,$$
т.е.
$$X_{n_k} - X_{n_m} \xrightarrow[k,m\to\infty]{L_2} 0.$$
Если рассмотреть в качестве подпоследовательностей $\{X_{n_k}\}$ и $\{X_{n_m}\}$ исходную последовательность $\{X_n\}_{n=1}^\infty$, то получим, что
$$X_k - X_m \xrightarrow[k,m\to\infty]{L_2} 0,$$
т.е. что $\{X_n\}_{n=1}^\infty$~--- фундаментальная по Коши последовательность. Но по теореме \ref{th:L2Complete} это значит, что она сходится. \EndProof
\end{proof}

Теперь, когда введено пространство случайных величин со скалярным произведением и пределом, мы приступаем к определению таких важнейших понятий любого анализа как \textit{непрерывность}, \textit{дифференцируемость} и \textit{интегрируемость}.

Всюду далее мы будем называть процессами второго порядка те процессы, у которых каждое сечение имеет конечный второй момент, т.е. для любого ${t\ge0}$ второй момент ${\mathbb{E}X^2(t)<\infty}$. Принадлежность классу процессов второго порядка мы будем обозначать так: $X(t)\in L_2$, аналогично тому, как мы поступали для случайных величин второго порядка.
\subsection{Непрерывность в среднем квадратичном}

\begin{definition}
Случайный процесс второго порядка $X(t)$ называется \textit{непрерывным в среднем квадратичном} (\textit{с.к.-непрерывным, непрерывным по математическому ожиданию}), если $$	\forall t \ge 0 \hookrightarrow X(t+\varepsilon) \xrightarrow[\varepsilon\to 0]{L_2} X(t). $$
\end{definition}

\begin{theorem}[ (критерий с.к.-непрерывности)]\label{th:SKCont1}
\textit{Процесс $X(t)\in$\linebreak $\in L_2$ является с.к.-непрерывным тогда и только тогда, когда его ковариационная функция $K_X(t_1,t_2)$ непрерывна на множестве }$[0,+\infty)^2$.
\end{theorem}

\begin{proof}
Пусть процесс $X(t)$ является с.к.-непрерывным. Тогда для любых $t_1,t_2 \ge0$ можно записать $$X(t_1+\varepsilon_1) \xrightarrow[\varepsilon_1\to 0]{L_2} X(t_1), \ X(t_2+\varepsilon_2) \xrightarrow[\varepsilon_2\to 0]{L_2} X(t_2).$$
Из теоремы \ref{th:ContScalarProd} получаем, что $$\langle X(t_1+\varepsilon_1),X(t_2+\varepsilon_2)\rangle \xrightarrow[\varepsilon_1,\varepsilon_2\to 0]{} \langle X(t_1),X(t_2) \rangle,$$ что и означает $K_X(t_1+\varepsilon_1,t_2+\varepsilon_2) \to K_X(t_1,t_2)$ при $\varepsilon_1,\varepsilon_2\to 0$.

Предположим теперь, что функция $K_X(t_1,t_2)$ непрерывна всюду на $[0,+\infty)^2$. Тогда она непрерывна во всех точках вида $(t,t)$, $t\ge0$. Остается заметить, что 
$$\Exp \left(X(t+\varepsilon) - X(t) \right)^2 = K_X(t+\varepsilon,t+\varepsilon) - 2K_X(t+\varepsilon,t) + K_X(t,t)$$ и, взяв предел при $\varepsilon\to0$, получить $$X(t+\varepsilon) \xrightarrow[\varepsilon\to 0]{L_2} X(t).\text{ \EndProof}$$
\end{proof}

Заметим, что при доказательстве с.к.-непрерывности процесса при непрерывной ковариационной функции используется ее непрерывность (как функции двух аргументов) лишь в точках вида $(t,t)$. Получается, что непрерывность ковариационной функции в точках вида $(t,t)$ влечет с.к.-непрерывность процесса, а с.к.-непрерывность процесса влечет непрерывность ковариационной функции всюду (не только в точках вида $(t,t)$). Это значит, что непрерывность ковариационной функции в точках вида $(t,t)$ влечет ее непрерывность во всех точках множества $[0,+\infty)^2$.

Кроме того, легко показать, что непрерывность ковариационной функции равносильна непрерывности функции математического ожидания процесса и непрерывности корреляционной функции. Действительно, если математическое ожидание и корреляционная функция всюду непрерывны, то непрерывна и ковариационная функция, т.к. они связаны соотношением ${K_X(t_1,t_2)=R_X(t_1,t_2)+m_X(t_1)m_X(t_2)}$. Наоборот, если ковариационная функция всюду непрерывна (и следовательно процесс -- с.к.-непрерывный), то $$\Exp \left(X(t+\varepsilon) - X(t) \right)^2=\mathbb{E}\left( \overset{\circ}{X}(t+\varepsilon) - \overset{\circ}{X}(t)+m_X(t+\varepsilon) - m_X(t)  \right)^2=$$
$$=\mathbb{E} \left( \overset{\circ}{X}(t+\varepsilon) - \overset{\circ}{X}(t) \right)^2 + \left( m_X(t+\varepsilon) - m_X(t) \right)^2 \xrightarrow[\varepsilon\to 0]{}0.$$ Так как оба слагаемых неотрицательные, то каждое из них стремится к нулю при ${\varepsilon\to 0}$. Сходимость к нулю второго слагаемого означает непрерывность математического ожидания. Сходимость к нулю первого слагаемого означает с.к.-непрерывность центрированного процесса, а следовательно, -- непрерывность его ковариационной функции. Но~ковариационная функция центрированного процесса -- это корреляционная функция исходного процесса.

Итак, теорему \ref{th:SKCont1} можно сформулировать и следующим образом.

\addtocounter{theorem}{-1}
\begin{theorem}[$'$ (критерий с.к.-непрерывности)] 
\textit{Процесс $X(t)\in$\linebreak $\in L_2$ является с.к.-непрерывным тогда и только тогда, когда его математическое ожидание $m_X(t)$ непрерывно при $t\ge0$ и корреляционная функция $R_X(t_1,t_2)$ непрерывна на множестве} $[0,+\infty)^2$.
\end{theorem}

\begin{example}
Винеровский процесс является с.к.-непрерывным, так как его математическое ожидание всюду непрерывно (оно всюду равно нулю) и корреляционная функция $R_X(t_1,t_2)=\min(t_1,t_2)$ непрерывна всюду. То же справедливо для пуассоновского процесса интенсивности ${\lambda > 0}$: его математическое ожидание ${m_K(t)=\lambda t}$ и корреляционная функция $R_K(t_1,t_2)=$ $\lambda\min(t_1,t_2)$ непрерывны всюду на своих областях определения. Процесс из задачи \ref{ex:Ex6.1} является с.к.-непрерывным, так как его математическое ожидание непрерывно (всюду равно нулю) и корреляционная функция ${R_{\xi}(t_1,t_2)=1/2\cos{(t_1-t_2)}}$ непрерывна всюду. Процесс из примера \ref{ex:Ex6.3} является с.к.-непрерывным, так как его математическое ожидание ${m_{\xi}(t)=1-t}$ и корреляционная функция ${R_{\xi}(t_1,t_2)=\min(t_1,t_2)-t_1t_2}$ непрерывны всюду на своих областях определения. \EndEx
\end{example}
\subsection{Дифференцируемость в среднем квадратичном}

\begin{definition}
Случайный процесс $X(t)\in L_2$ называется \textit{дифференцируемым в среднем квадратичном} (\textit{с.к.-дифференцируемым или дифференцируемым по математическому ожиданию}), если существует процесс $Y(t)\in L_2$, $t\ge0$, такой, что  $$\forall t\ge0 \hookrightarrow \frac{X(t+\varepsilon) - X(t)}{\varepsilon} \xrightarrow[\varepsilon\to0]{L_2} Y(t).$$ Случайный процесс $Y(t)$ в этом случае называется \textit{с.к.-производной} процесса $X(t)$ и будет обозначаться штрихом, т.е. $Y(t)=X'(t)$.
\end{definition}

\begin{theorem}[ (критерий с.к.-дифференцируемости)]\label{th:SKDiff1}
\textit{Случайный процесс $X(t)\!\!\!\!\in\!\!\!\!L_2$ является с.к.-диф\-фе\-рен\-цируемым тогда\linebreak и~только тогда, когда для любой пары ${t_1,t_2 \ge 0}$ существует конечный предел }
\begin{multline*}
\lim\limits_{\varepsilon_1, \varepsilon_2 \to 0} \frac{1}{\varepsilon_1 \varepsilon_2} (K_X(t_1+ \varepsilon_1, t_2 + \varepsilon_2) - K_X(t_1 + \varepsilon_1, t_2)\,- \\ 
- K_X(t_1, t_2 + \varepsilon_2) + K_X(t_1, t_2)) < \infty.
\end{multline*}
\end{theorem}

\begin{proof}
Пусть $X(t)$ является с.к.-дифференцируемым. Введем обозначения
$$Y_{\varepsilon_1}(t_1) = \frac{X(t_1+\varepsilon_1) - X(t_1)}{\varepsilon_1},\quad Y_{\varepsilon_2}(t_2) = \frac{X(t_2+\varepsilon_2) - X(t_2)}{\varepsilon_2}.$$ Тогда $Y_{\varepsilon_1}(t_1) \xrightarrow[\varepsilon_1\to0]{L_2} X'(t_1)$ и $Y_{\varepsilon_2}(t_2) \xrightarrow[\varepsilon_2\to0]{L_2} X'(t_2)$. Из теоремы \ref{th:ContScalarProd} следует, что
$$
\Exp(Y_{\varepsilon_1}(t_1)Y_{\varepsilon_2}(t_2)) \xrightarrow[\varepsilon_1, \varepsilon_2 \to 0]{} \Exp(X'(t_1)X'(t_2)) \le \Exp X'(t_1)^2 \cdot \Exp X'(t_2)^2 < \infty.
$$ 
Остается заметить, что $\mathbb{E} (Y_{\varepsilon_1}(t_1) Y_{\varepsilon_2}(t_2))$ равняется подпредельному выражению из утверждения теоремы и
\begin{multline*}
\infty > \lim\limits_{\varepsilon_1, \varepsilon_2 \to 0} \mathbb{E}(Y_{\varepsilon_1}(t_1) Y_{\varepsilon_2}(t_2)) = \lim\limits_{\varepsilon_1, \varepsilon_2 \to 0} \frac{1}{\varepsilon_1 \varepsilon_2} (K_X(t_1 + \varepsilon_1, t_2 + \varepsilon_2) - \\ 
- K_X(t_1 + \varepsilon_1, t_2) - K_X(t_1, t_2 + \varepsilon_2) + K_X(t_1, t_2)).
\end{multline*}

Пусть теперь существует предел из утверждения теоремы для любой пары $(t_1,t_2)$. Тогда этот предел существует для пары $(t,t)$, т.е. при ${t_1=t_2=t}$. Из существования предела выражения $\mathbb{E} (Y_{\varepsilon_1}(t) Y_{\varepsilon_2}(t))$ при ${\varepsilon_1,\varepsilon_2 \to 0}$ и теоремы \ref{th:ExistProccess} следует существование процесса ${Y(t)\in L_2}$ такого, что ${Y(t)=\operatorname{l.i.m.}_{\varepsilon\to0}Y_{\varepsilon}(t)}$. По определению, это и означает существование с.к.-производной процесса $X(t)$. \EndProof
\end{proof}

Теперь, когда теорема доказана, необходимо сделать несколько важных замечаний.

Первое замечание состоит в том, что предел из утверждения теоремы не является смешанной производной второго порядка ковариационной функции $K_X(t_1,t_2)$. Дело в том, что по определению смешанная производная второго порядка функции двух переменных -- это производная по одной переменной от производной по другой переменной, т.е. повторный предел по каждой переменной. Предел из теоремы -- это предел по двум переменным, он отличается от второй смешанной производной так же, как предел по двум переменным отличается от повторного предела. Предел из теоремы обычно называется \textit{обобщенной смешанной производной} (она не имеет никакого отношения к обобщенным функциям!). Из существования обобщенной смешанной производной вытекает существование обычной смешанной производной второго порядка; обратное, вообще говоря, не верно. Достаточным условием существования обобщенной смешанной производной функции $K_X(t_1,t_2)$ является непрерывность хотя бы одной из частных производных $$\frac{\partial }{\partial t_1} \left(\frac{\partial K_X(t_1,t_2)}{\partial t_2}\right), \ \frac{\partial }{\partial t_2} \left(\frac{\partial K_X(t_1,t_2)}{\partial t_1}\right).$$

Второе замечание состоит в том, что при доказательстве существования с.к.-производной мы воспользовались лишь конечностью обобщенной смешанной производной в точках вида $(t,t)$. Получается, что конечность обобщенной смешанной производной в точках вида $(t,t)$ влечет существование с.к.-производной исходного процесса, что влечет конечность обобщенной смешанной производной \textit{всюду}. Поэтому, теорему можно сформулировать и следующим образом.

\addtocounter{theorem}{-1}
\begin{theorem}[$'$ (критерий с.к.-дифференцируемости)]
\textit{Случайный процесс ${X(t)\in L_2}$ является с.к.-диф\-фе\-рен\-цируемым тогда и только тогда, когда для любого ${t \ge 0}$ существует конечный предел }
\begin{multline*}
\lim\limits_{\varepsilon_1, \varepsilon_2 \to 0} \frac{1}{\varepsilon_1 \varepsilon_2}( K_X(t + \varepsilon_1, t + \varepsilon_2) - K_X(t + \varepsilon_1, \shmaxg{t})\,- \\ 
- K_X(t, t + \varepsilon_2) + K_X(t, t)) < \infty.
\end{multline*}
\end{theorem}

Наконец, третье замечание заключается в том, что существование обобщенной смешанной производной ковариационной функции равносильно существованию производной математического ожидания процесса и существованию обобщенной смешанной производной корреляционной функции. Мы опускаем доказательство этого факта, но сформулируем еще один вариант теоремы \ref{th:SKDiff1}.

\addtocounter{theorem}{-1}
\begin{theorem}[$''$ (критерий с.к.-дифференцируемости)] 
\textit{Случайный процесс ${X(t)\in L_2}$ является с.к.-диф\-фе\-рен\-цируемым тогда и только тогда, когда для любого ${t \ge 0}$ существует производная математического ожидания процесса $m_X(t)$ и существует конечный предел }
\begin{multline*}
\lim\limits_{\varepsilon_1, \varepsilon_2 \to 0} \frac{1}{\varepsilon_1 \varepsilon_2} (R_X(t + \varepsilon_1, t + \varepsilon_2) - R_X(t + \varepsilon_1, \shmaxg{t})\,- \\ 
- R_X(t, t + \varepsilon_2) + R_X(t, t)) < \infty.
\end{multline*}
\end{theorem}

\begin{example}
Винеровский процесс не является с.к.-диф\-фе\-рен\-ци\-руемым ни в какой точке, так как не существует обобщенной производной корреляционной функции в точках вида $(t,t)$. Действительно, возьмем произвольную точку $t$ и рассмотрим предел 
\begin{multline*}
\lim\limits_{\eps_1, \eps_2 \to 0} \frac{1}{\eps_1 \eps_2} (R_W(t+\eps_1,t+\eps_2)-R_W(t+\eps_1,t)\,- \\ 
- R_W(t,t+\eps_2) + R_W(t,t)).
\end{multline*}
Так как ${R_W(t,s)=\min(t,s)=(t+s-|t-s|)/2}$, то данный предел равен пределу $$\lim\limits_{\eps_1,\eps_2\to0}\frac{|\eps_1|+|\eps_2|-|\eps_1-\eps_2|}{2\eps_1\eps_2}.
$$
Конечный предел не существует, так как при $\eps_1=\eps_2$ он равен $+\infty$. \EndEx
\end{example}

\begin{example}
Процесс из задачи \ref{ex:Ex6.1} является с.к.-дифференцируе\-мым. Действительно, его математическое ожидание тождественно равно нулю, а значит, дифференцируемо всюду. Корреляционная функция $R_\xi(t_1,t_2)=$ $1/2\cos(t_1-t_2)$ имеет непрерывную смешанную производную второго порядка в любой точке вида $(t,t)$; следовательно, обобщенная производная существует и конечна всюду.
\end{example}
\subsection{Интегрируемость в среднем квадратичном}

\begin{definition}
Пусть процесс $X(t)\in L_2$ определен на отрезке ${[a,b]\subset [0,+\infty)}$. На отрезке $[a,b]$ построим некоторое разбиение $a=$\linebreak $=t_0 < t_1 < \dots < t_{n-1} < t_n = b $, и на каждом из промежутков этого разбиения выберем произвольную точку ${\tau_i\in[t_{i-1},t_i)}$, $i=1,\dots,n$. Если при ${n\to\infty}$ и ${\max\limits_{i=1,\dots,n}(t_i-t_{i-1})\to0}$ существует предел в среднеквадратическом $$\sum\limits_{i=1}^n X(\tau_i)(t_i-t_{i-1})\xrightarrow[\varepsilon\to0]{L_2}Y,$$ не зависящий от способа разбиения $\{t_i\}$ и выбора точек $\{\tau_i\}$, то процесс $X(t)$ называется \textit{с.к.-интегрируемым на} $[a,b]$, а случайная величина $Y$ называется ее \textit{с.к.-интегралом} или \textit{стохастическим интегралом Римана} на отрезке $[a,b]$ и обозначается $$Y=\int\limits_a^b X(t)\,dt.$$ 
\end{definition}

Определение с.к.-интеграла допускает полезное обобщение.

\begin{definition}
Пусть процесс $X(t)\in L_2$ определен на отрезке ${[a,b]\subset [0,+\infty)}$, а $g(t)$ -- неслучайная непрерывная функция на отрезке $[a,b]$. Тогда с.к.-интеграл случайного процесса $g(t)X(t)$, если он существует, называется \textit{с.к.-интегралом процесса $X(t)$ с непрерывной функцией $g(t)$} или \textit{стохастическим интегралом Римана по математическому ожиданию процесса $X(t)$ с непрерывной функцией $g(t)$}. Процесс $X(t)$ в этом случае называется \textit{с.к.-интегрируемым с непрерывной функцией $g(t)$}.
\end{definition}

Следующие два критерия с.к.-интегрируемости мы приводим без доказательства.

\begin{theorem}[ (критерий с.к.-интегрируемости)]
\textit{Случайный\linebreak процесс ${X(t)\in L_2}$ с.к.-интегрируем с непрерывной функцией $g(t)$ на отрезке ${[a,b]\subset [0,+\infty)}$ тогда и только тогда, когда существует \shmaxg{конечный} интеграл Римана}: 
$$\int\limits_{a}^{b} \int\limits_{a}^{b} g(t_1)g(t_2)K_X(t_1,t_2)\,dt_1dt_2 < \infty.$$
\end{theorem}

\addtocounter{theorem}{-1}
\begin{theorem}[$'$ (критерий с.к.-интегрируемости)]
\textit{Случайный\linebreak процесс ${X(t)\in L_2}$ с.к.-интегрируем с непрерывной функцией $g(t)$ на отрезке $[a,b]\subset[0,+\infty)$ тогда и только тогда, когда существуют \shmaxg{конечные} интегралы Римана}: $$\int\limits_{a}^{b}g(t)m_X(t)\,dt < \infty, \
\int\limits_{a}^{b} \int\limits_{a}^{b} g(t_1)g(t_2)R_X(t_1,t_2)\,dt_1dt_2 < \infty.$$
\end{theorem}

Отметим, что если процесс является с.к.-непрерывным, то он является с.к.-интегрируемым на любом отрезке и с любой непрерывной функцией. Если процесс является с.к.-дифференцируемым, то он является с.к.-непрерывным. Например, т.к. винеровский и пуассоновский процессы являются с.к.-непрерывными, то они являются с.к.-интегрируемыми на любом отрезке с любой непрерывной функцией.
\subsection{Полезные формулы и примеры}

Итак, мы рассмотрели понятие с.к.-предела в пространстве $L_2$ случайных величин второго порядка и ввели на его основе понятия\linebreak с.к.-~непрерывности, с.к.-производной и с.к.-интеграла случайных процессов. Совершенно аналогично можно было бы ввести понятия непрерывности, дифференцируемости и интегрируемости, рассматривая их в смысле сходимости почти наверное, по вероятности и по распределению. Для этого во всех определениях выше вместо с.к.-предела нужно поставить предел в соответствующем смысле. Отметим, что критерии непрерывности, дифференцируемости и интегрируемости в этом случае будут другие.

\begin{example}
Выше мы видели, что пуассоновский процесс не является с.к.-дифференцируемым. Покажем, что он является дифференцируемым в смысле сходимости по вероятности, т.е. что существует такая случайная величина $Y(t)$, что $$\forall t \ge 0 \hookrightarrow \frac{K(t+\eps)-K(t)}{\eps} \xrightarrow[\eps\to0]{\mathbb{P}}Y(t),$$ или, по определению сходимости по вероятности,
$$\forall t \ge 0, \ \forall \delta > 0 \hookrightarrow \mathbb{P}\left( \left|\frac{K(t+\eps)-K(t)}{\eps} - Y\right| > \delta \right) \xrightarrow[\eps\to0]{}0.$$ Пусть $\lambda$ -- интенсивность пуассоновского процесса. Возьмем ${Y(t)=0}$. Заметим, что $K(t+\eps)-K(t)\in\mathrm{Po}(\lambda\eps)$ и следовательно, $$\mathbb{P}(K(t+\eps)-K(t)=k)=\frac{(\lambda\eps)^k}{k!}e^{-\lambda\eps}, \ k=0,1,2,\dots.$$ При ${\eps<1/\delta}$ число $\delta\eps < 1$, событие $\{K(t+\eps)-K(t)>\delta\eps\}$ совпадает с событием $\{K(t+\eps)-K(t)>0\}$, и значит, $$\mathbb{P}(K(t+\eps)-K(t)>\delta\eps)= \mathbb{P}(K(t+\eps)-K(t)>0)=1-e^{-\lambda\eps}\to0$$ при ${\eps\to0}$. Итак, мы показали, что пуассоновский процесс в любой точке обладает производной в смысле сходимости по вероятности и эта производная всюду равна 0. Отсюда следует и то, что пуассоновский процесс в любой точке обладает производной в смысле сходимости по распределению, и эта производная тоже всюду равна 0. \EndEx
\end{example}

\begin{example}
Покажем, что винеровский процесс не является дифференцируемым ни в какой точке даже по распределению. Предположим противное. Зафиксируем ${t\ge0}$. Пусть существует случайная величина $Y$ такая, что $$\forall t \ge 0 \hookrightarrow \frac{W(t+\eps)-W(t)}{\eps} \xrightarrow[\eps\to0]{d}Y.$$ Введем обозначение $$X_{\eps}=\frac{W(t+\eps)-W(t)}{\eps}.$$ По теореме непрерывности, из сходимости по распределению следует сходимость характеристических функций, т.е. $$\forall s\in\mathbb{R} \hookrightarrow \varphi_{X_{\eps}}(s)\to\varphi_{Y}(s), \ \eps\to0.$$ Так как ${X_{\eps}\in\mathrm{N}(0,1/|\eps|)}$, то ${\varphi_{X_{\eps}}(s)=e^{-s^2/(2|\eps|)}}$. Следовательно, при ${s=0}$ имеем ${\varphi_{X_{\eps}}(0)=1}$, а при ${s\ne0}$ получаем ${\varphi_{X_{\eps}}(s)\to0}$, ${\eps\to0}$. Таким образом, предельная функция $\varphi_Y(s)$ не является непрерывной в точке ${s=0}$, но так как характеристическая функция всякой случайной величины всюду непрерывна, получаем противоречие. \EndEx
\end{example}

Приведем также ряд соотношений, вывод которых предлагается выполнить читателю в качестве упражнения. Все процессы здесь предполагаются достаточно с.к.-гладкими.

$$\mathbb{E}X'(t) = \frac{d}{dt}\mathbb{E}X(t),$$


$$\mathrm{cov} (X'(t),X'(s)) = R_{X'}(t,s) = \frac{\partial^2R_X(t,s)}{\partial t \, \partial s},$$

$$\mathrm{cov} (X'(t),X'(t)) = \mathbb{D}X'(t) = \frac{\partial^2R_X(t,s)}{\partial t \, \partial s}\bigg\rvert_{s=t},$$

$$\mathrm{cov} \left(X(t),X'(s)\right) = \frac{\partial R_X(t,s)}{\partial s},$$

$$\mathbb{E}\int\limits_a^b X(t)\,dt=\int\limits_a^b\mathbb{E}X(t)\,dt,$$

$$\mathrm{cov} \left(X(t),\int\limits_{a}^b X(t)\,dt \right) = \int\limits_{a}^b R_X(t,s) ds,$$

$$\mathrm{cov} \left(\int\limits_{a}^b X(t)\,dt,\int\limits_{c}^d X(t)\,dt \right) = \int\limits_{a}^b\int\limits_{c}^d R_X(t,s)\,dt ds,$$

$$\mathrm{cov} \left(\int\limits_{a}^b X(t)\,dt,\int\limits_{a}^b X(t)\,dt \right) = \mathbb{D}\int\limits_{a}^b X(t)\,dt= \int\limits_{a}^b \int\limits_{a}^b R_X(t,s)\,dt ds.$$

\begin{example}
Пусть ${X(t)\in L_2, X'(t) \in L_2, J(t) = \int_{0}^{t}X(\tau)d\tau \in L_2}$ и $m_X(t)\equiv0$. Вычислить взаимную корреляционную функцию случайных процессов $X'(t)$ и $J(t)$.
\end{example}

\begin{solution}
Из соотношений выше сразу следует, что $m_{X'}(t)=0$ и $m_{J}(t)=0$ для всех $t$. Следовательно,
$$R_{X',J}(t_1,t_2) = \Exp (X'(t_1)J(t_2)) =$$
$$	= \Exp \left( \underset{\varepsilon\to0}{\operatorname{l.i.m.}} \frac{X(t_1+\varepsilon) - X(t_1)}{\varepsilon}\cdot \underset{\substack{0=\tau_0 < \tau_1<\ldots<\tau_n=t_2 \\
\underset{i=\overline{1,n}}{\max}\Delta \tau_i \xrightarrow[n\to\infty]{}0\\
\tau_i'\in[\tau_i,\tau_{i+1})}}{\operatorname{l.i.m.}} \sum_{i=1}^{n}X(\tau_i')\Delta\tau_i \right) = $$
$$	\overset{\text{Теорема \ref{th:ContScalarProd}}}{=} \underset{\substack{\varepsilon\to0\\
n\to\infty}}{\lim}\Exp\left(\frac{X(t_1+\varepsilon_1)-X(t_1)}{\varepsilon}\cdot\sum_{i=1}^{n}X(\tau_i')\Delta\tau_i\right) =$$
$$
= \underset{\substack{\varepsilon\to0\\	n\to\infty}}{\lim}\Exp\left(\sum_{i=1}^{n}\frac{X(t_1+\varepsilon)X(\tau_i') - X(t_1)X(\tau_i')}{\varepsilon}\Delta\tau_i\right) =
$$
$$
= \underset{\substack{\varepsilon\to0\\
n\to\infty}}{\lim}\sum_{i=1}^{n}\frac{R_X(t_1+\varepsilon,\tau_i') - R_X(t_1,\tau_i')}{\varepsilon}\Delta\tau_i = \int\limits_{0}^{t_2}\frac{\partial R_X(t_1,\tau)}{\partial t_1}\,d\tau,$$ где ${\Delta \tau_i=\tau_{i}-\tau_{i-1}}$, ${i \ge 1}$, и начиная с третьей строчки условия о стремлении мелкости разбиения отрезка $[0,t_2]$ к нулю, заменены для краткости одним условием $n\to\infty$. \EndEx
\end{solution}
\subsection{Интеграл \gav{Римана--Стилтьеса}}

Теперь введем понятие \textit{интеграла \gav{Римана--Стилтьеса}}, аналогичное понятию интеграла Стилтьеса для неслучайных функций.

\begin{definition}\label{def:ItoIntegral}
Пусть процесс $X(t)\in L_2$ определен на отрезке ${[a,b]\subset [0,+\infty)}$ и $g(t)$ -- непрерывная на отрезке $[a,b]$ функция. На~отрезке $[a,b]$ построим некоторое разбиение $a=t_0 < t_1 < \dots <$\linebreak $< t_{n-1} < t_n = b $, и на каждом из промежутков этого разбиения выберем произвольную точку ${\tau_i\in[t_{i-1},t_i)}$, $i=1,\dots,n$. Если при ${n\to\infty}$ и $\shmaxg{\varepsilon}={\max\limits_{i=1,\dots,n}(t_i-t_{i-1})\to0}$ существует предел в среднеквадратическом $$\sum\limits_{i=1}^n g(\tau_i)   (X(\tau_i) - X(\tau_{i-1}))\xrightarrow[\varepsilon\to0]{L_2}Y,$$ не зависящий от способа разбиения $\{t_i\}$ и выбора точек $\{\tau_i\}$, то случайная величина $Y$ называется \textit{интегралом \gav{Римана--Стилтьеса} с непрерывной функцией $g(t)$} на отрезке $[a,b]$ и обозначается символом $$Y=\int\limits_a^b g(t)\,dX(t).$$ 
\end{definition}

Приведем без доказательства критерий существования \shmaxg{такого интеграла}. 


\begin{theorem}\textbf{(см.~\cite{NGG3})}
Пусть $X(t)\in L_2$. Интеграл Римана--Стил\-тьеса от случайного процесса $X(t)$ с непрерывной функцией $g(t)$ существует тогда и только тогда, когда 
$$\int\limits_{a}^{b} \int\limits_{a}^{b} g(t_1)g(t_2)\,dK_X(t_1,t_2) < \infty,$$
где интеграл понимается в смысле интеграла Римана--Стилтьеса\footnote{Здесь интеграл понимается в следующем смысле. Пусть мера клетки $[a_1,a_2]\times[b_1,b_2]$ равна $K(a_2,b_2) - K(a_1,b_2) - K(a_2,b_1) + K(a_1,b_1)$. При помощи клеточных разбиений и сумм Римана вводится интеграл Римана--Стилтьеса от непрерывной функции подобно тому, как определялся интеграл Римана--Стилтьеса в одномерном случае (см.~\cite[гл. 3, \S 6]{Borovkov1999}) и кратный интеграл Римана.}.
\end{theorem}

Интеграл \gav{Римана--Стилтьеса} по конечным отрезкам естественным образом обобщается на случай бесконечных интервалов. Например, по определению будем считать, что
$$\int\limits_{0}^{+\infty}g(t)\,dX(t)=\  \underset{b\to+\infty}{\operatorname{l.i.m.}}\int\limits_0^b g(t)\,dX(t),$$
если этот предел существует.

\subsection{Стохастический интеграл и формула Ито}

Введем теперь понятие стохастического интеграла  от случайного процесса по винеровскому процессу ${\{W(t), t\ge0\}}$. Это понятие пригождается для определения стохастических дифференциальных уравнений.

Далее нам пригодится обозначение $\sigma\{X(t),t\in T\}$ сигма-алгебры, порожденной случайным процессом $X(t)$ на множестве $t\in T$. Это минимальная сигма-алгебра, относительно которой измеримы все сечения $X(t)$ на множестве $t\in T$.

\begin{definition}
Случайная функция ${\{X(t),t\in [a,b]\}}$ называется \textit{неупреждающей} относительно процесса ${\{W(t), t\in [a,b]\}}$, если для любого ${t\in [a,b]}$ и для любого борелевского множества ${B\in\mathcal{B}}$ выполнено ${\{X(t) \in B\}\in\sigma\{W(s),s \le t\}}$.
\end{definition}
Если говорить очень грубо, то это значит, что события, связанные с процессом $X(t)$ в момент $t$, связаны с событиями процесса $W(t)$ на интервале времени до $t$ включительно и не связаны с <<будущим>> процесса $W(t)$, т.е. с событиями на интервалах времени после момента времени $t$.

Как обычно в теории меры, интеграл от случайной функции мы построим как предел интеграла от простых (ступенчатых, ку\-соч\-но-постоянных) функций. Предел этот будем мы будем понимать\linebreak в~с.к.-смысле. Под интегралом же простой функции $X_n(t)$, принимающей значения ${X_n(t)=X_n(t_k)}$ на интервалах ${t_k\le t < t_{k+1}}$ разбиения ${\{t_k\}_{k=1}^n}$ интервала $[a,b]$ будет понимать просто сумму
$$I(X_n)=\sum\limits_{k=1}^n X_n(t_k)(W(t_{k+1})-W(t_k)).$$
\begin{theorem} 
\textbf{(см.~\cite[С.~214]{MillerPankov})}
\textit{Пусть ${\{X(t),t\in [a,b]\}}$ --\linebreak с.к.-непрерывная на $[a,b]$ неупреждающая функция. Тогда найдется последовательность $X_n(t)$ простых неупреждающих функций $\{X_n(t)\}$, такая, что $$\int\limits_{a}^b \mathbb{E}|X_n(t)-X(t)|^2\,dt \to 0, \ n\to\infty,$$ и для которой существует с.к.-предел} $I(X_n) \xrightarrow{\text{с.к.}} I $, $n\to\infty$.
\end{theorem}

\begin{definition}
Предел $I$ в теореме выше называется \textit{стохастическим интегралом} функции $X(t)$ по винеровскому процессу $W(t)$ на интервале $[a,b]$ и обозначается $$I(X)=\int\limits_{a}^bX(t)\,dW(t).$$
\end{definition}

Если сравнить это определение с интегралом от неслучайной функции по случайному процессу, то можно видеть, что все отличие состоит в паре формальностей: вместо непрерывной подынтегральной функции мы имеем дело с с.к.-непрерывной функцией и дополнительно требуем от нее свойство неупреждаемости. Можно доказать, что значение $I$ не зависит от выбора последовательности простых функций. Поэтому на практике интервал $[a,b]$ можно разбивать равномерно по времени. Разберем хрестоматийный пример.

\begin{example}
Вычислить стохастический интеграл $$I = \int\limits_{0}^T W(t)\,dW(t).$$
\end{example}

\textbf{Решение}. Подынтегральная функция $W(t)$ является всюду\linebreak с.к.-непрерывной и, очевидно, неупреждающей функцией. Вычислять интеграл будем по определению, для этого введем равномерное разбиение отрезка $[0,T]$, $t_1 < \dots t_n$, $t_{k+1}-t_k=h$, запишем интегральную сумму 
\begin{equation}\label{eq:ItoIntForm}
I_n = \sum\limits_{k=1}^n W(t_k)(W(t_{k+1}) - W(t_k)).
\end{equation}
Для краткости обозначений введем $\Delta t_k = t_{k+1}-t_k$, $\Delta W_k = W(t_{k+1}) -$\linebreak $- W(t_k)\in\mathrm{N}(0,\Delta t_k)$. Теперь заметим, что $$W(t_k)\Delta W_k = \frac{1}{2}\left( W^2(t_{k+1}) - W^2(t_k) \right) - \frac{1}{2}(\Delta W_k)^2.$$ Отсюда следует, что
$$I_n = \frac{1}{2}\left(W^2(T)-W^2(0)\right) - \frac{1}{2}\sum\limits_{k=1}^n (\Delta W_k)^2.$$ Первое слагаемое этого выражения не зависит от $n$ и с измельчением отрезка $[0,T]$ не меняется. Найдем с.к.-предел второго слагаемого.\footnote{\ag{Ниже приводится только схема рассуждений. Более подробно соответствующие выкладки расписаны в примере \ref{exDW}.}} Попробуем сначала найти его математическое ожидание
$$\mathbb{E}\sum\limits_{k=1}^n (\Delta W_k)^2=\sum\limits_{k=1}^n\mathbb{E}(\Delta W_k)^2=nh=T$$ и дисперсию
\begin{eqnarray*}
\mathbb{D}\sum\limits_{k=1}^n (\Delta W_k)^2 &= \sum\limits_{k=1}^n\mathbb{D}(\Delta W_k)^2=\sum\limits_{k=1}^n\mathbb{E}(\Delta W_k)^4 - (\mathbb{E}(\Delta W_k)^2)^2\\
&= \sum\limits_{k=1}^n 3(\mathbb{D}\Delta W_k)^2 - nh^2=3h^2n-h^2n=2h^2n,
\end{eqnarray*}

где для расчета момента четвертого порядка можно воспользоваться теоремой Вика (теорема~\ref{th:WickTheorem} настоящего пособия). Получается, что $\mathbb{D}\sum_{k=1}^n (\Delta W_k)^2=2T^2/n\to0$, $n\to\infty$, но по определению с.к.-предела это значит, что $\sum_{k=1}^n (\Delta W_k)^2 \xrightarrow{\text{с.к.}}T$. Итак, мы показали, что $$I_n \xrightarrow{\text{с.к.}}\frac{1}{2}\left(W^2(T) - T\right), \ n\to\infty,$$ поэтому окончательно заключаем, что 
$$ \int\limits_{0}^T W(t)\,dW(t) = \frac{1}{2}\left(W^2(T) - T\right). \ \ \text{\EndEx}$$

Обратим внимание на то, что если бы в формуле~\eqref{eq:ItoIntForm} значение подынтегральной функции вычислялось не в крайней левой точке отрезка $[t_k,t_{k+1}]$, а в какой-либо другой, то значение интеграла изменилось бы. Если выбрать произвольное ${\theta\in[0,1]}$ и в качестве промежуточной точки взять $\tau_k^\theta=(1-\theta)t_k + \theta t_{k+1}$ и рассмотреть интегральную сумму 
$$I(X_n)=\sum\limits_{k=1}^n X(\tau_k^\theta)(W(t_{k+1}) - W(t_k)),$$ то при тех же условиях, что и ранее, справедливо утверждение: $$I(X_n) \xrightarrow{\text{с.к.}} I^{\theta}, \ n\to\infty,$$ где предельная величина $I^{\theta}$ в общем случае зависит от параметра $\theta$ и называется \textit{стохастическим} $\theta$-\textit{интегралом}. При ${\theta=0}$ мы получаем случай, рассмотренный выше. Такой интеграл называется еще \textit{стохастическим интегралом Ито}. Если ${\theta=1/2}$, то интеграл называется \textit{стохастическим интегралом Стратоновича}. Можно показать, что в случае Стратоновича интеграл из примера выше будет равен $$I^{1/2}(W)=\frac{1}{2} W^2(T).$$ Кстати говоря, в этом же случае справедлива привычная формула интегрирования по частям:
$$I^{1/2}(W)=\int\limits_0^T W(t)\,dW(t) = \frac{1}{2}W^2(t)\biggr\rvert_0^T=\frac{1}{2}W^2(T).$$

В общем случае вычислять стохастические интегралы помогает \textit{формула Ито}.
\begin{definition}
Пусть даны два неупреждающих случайных процесса $\{f(t),t\in[a,b]\}$ и $\{g(t),t\in[a,b]\}$ такие, что
$$\mathbb{P}\left(\omega:\int\limits_{a}^b |f(\omega,t)|\,dt < \infty\right) = 1, \ \mathbb{P}\left(\omega:\int\limits_{a}^b |g(\omega,t)|^2\,dt < \infty\right) = 1.$$ Случайный процесс $\{X(t),t\in[a,b]\}$ называется \textit{процессом Ито}, если
\begin{equation}\label{eq:ItoProcessIntegralForm}
X(t)=X(0)+\int\limits_{a}^t f(\omega,s)\,ds + \int\limits_{a}^t g(\omega,s)\,dW(s),
\end{equation}
где первый интеграл понимается в <<потраекторном>> смысле, а второй -- в смысле стохастического интеграла Ито.
\end{definition}
Обычно вместо~\eqref{eq:ItoProcessIntegralForm} используется запись в <<дифференциалах>>
\begin{equation}\label{eq:ItoProcessDiffForm}
dX(t)= f(\omega,t)dt + g(\omega,t)dW(t),
\end{equation}
говоря при этом, что процесс $X(t)$ имеет стохастический дифференциал. Запись~\eqref{eq:ItoProcessDiffForm} следует всегда понимать как сокращенную запись формулы~\eqref{eq:ItoProcessIntegralForm}.

Теперь приведем один из основных результатов стохастического анализа.

\begin{theorem}[ (формула Ито~\cite{Bulinsky2010})]
\textit{Пусть неслучайная функция $F(t,x)$ не\-прерывно дифференцируема по ${t\ge0}$ и дважды непрерывно дифференцируема по $x\in\mathbb{R}$. Пусть процесс $\{X(t),t\ge0\}$ имеет стохастический дифференциал~\eqref{eq:ItoProcessDiffForm}. Тогда процесс $\{F(t,X(t)),t\ge0\}$ также имеет стохастический дифференциал}:
\begin{equation}
    dF(t,X(t)) = \left[ \frac{\partial F}{\partial t} + f(\omega,t)\frac{\partial F}{\partial x} + \frac{1}{2}g^2(\omega,t)\frac{\partial^2 F}{\partial x^2} \right]\,dt + \frac{\partial F}{\partial x} g(\omega,t)\,dW(t).
\end{equation}
\end{theorem}

Эту формулу иногда записывают в более удобном для запоминания виде:
$$dF = \frac{\partial F}{\partial t} dt + \frac{\partial F}{\partial x}dX(t) + \frac{1}{2}\frac{\partial ^2 F}{\partial x^2} (dX(t))^2,$$
если всюду принять $dt\,dt=0$, $dW(t)\,dt=0$ и $(dW(t))^2=dt$.

Эта теорема позволяет очень просто вычислить стохастический интеграл Ито $\int_0^T W(t)\,dW(t)$. По сути, мы ищем процесс $X(t)$, дифференциал которого равен $W(t)\,dW(t)$:
$$dX = W(t)\,dW(t).$$


Задача поиска $X(t)$ сводится к задаче поиска подходящей функции $F(t,x)$. Возьмем $F(t,x)=x^2$ и вычислим
$$dF(t,W(t))=0 + 2W(t)\,dW(t) + (dW(t))^2=2W(t)\,dW(t) + dt.$$
Как мы уже говорили выше, эта запись является сокращенной формой записи
$$W^2(t)=W^2(0)+2\int\limits_0^t W(s)\,dW(s) + \int\limits_0^t\,ds,$$
откуда сразу получаем значение интересующего интеграла.
\subsection{Стохастические дифференциальные уравнения*}

Приведем, следуя~\cite{Oksendal2003}, определение стохастического дифференциального уравнения и простой пример такого уравнения.
\begin{definition}
\textit{Стохастическое дифференциальное уравнение Ито}
\begin{equation}\label{eq:StochDiffEq}
dX(t)=f(t,X(t))\,dt+g(t,X(t))\,dW(t)
\end{equation}
с коэффициентами $f$ и $g$ и начальным условием $X_0$ -- это задача поиска случайной функции $X(t)$ с дифференциалом~\eqref{eq:StochDiffEq} и ${X(0)=X_0}$. Решение с п.н.-непрерывными траекториями называется \textit{сильным решением} такого уравнения.
\end{definition}
Если функции $f(t,x)$ и $g(t,x)$ достаточно <<хорошие>>, а именно если они по переменной $x$ удовлетворяют локальному условию Липшица:
$$\forall n\in\mathbb{N} \ \forall x,y \in [-n,n] \ \exists C(n) \ |f(t,x)-f(t,y)| \le C(n) |x-y|,$$
$$\forall n\in\mathbb{N} \ \forall x,y \in [-n,n] \ \exists C(n) \ |g(t,x)-g(t,y)| \le C(n) |x-y|,$$
а также условию линейного роста:
$$\exists C>0 \ \forall t\ge 0 \ \forall x\in\mathbb{R} \ |f(t,x)| \le C |x|, \ |g(t,x)| \le C|x|,$$
то сильное решение уравнения~\eqref{eq:StochDiffEq} существует и единственно для любой случайной величины $X_0$, измеримой относительно $W(0)$. 

\gav{Приведем} пример одного стохастического дифференциального уравнения:
$$dS(t)=aS(t)dt + \sigma S(t)\,dW(t).$$
Решение $S(t)$ этого уравнения называется \textit{процессом Башелье--Саму\-эль\-сона}. Этот процесс является одной из первых моделей стоимости акций в финансовой математике. С помощью формулы Ито можно \gav{проверить}, что
$$S(t)=S_0 e^{at}e^{\sigma W(t)-\sigma^2 t / 2}.$$
\section{Гауссовские случайные процессы}
\label{gauss}

\subsection{Моментные характеристики}

В разделе~\ref{random_processes} нами было введено важно понятие -- \textit{гауссовский процесс}. Гауссовский процесс -- это процесс, все конечномерные распределения которого нормальные (гауссовские). Это значит, что любой случайный вектор, составленный из сечений такого процесса, имеет нормальное распределение. Такие случайные векторы обладают целом рядом полезных для практических расчетов свойств. Цель данного раздела -- привести эти свойства и продемонстрировать их использование в задачах.


Напомним, что случайный вектор $X\in\mathbb{R}^n$ является \textit{гауссовским} с математическим ожиданием $m$ и \gav{корреляционной} матрицей $R \in \mathbb{R}^{n \times n}$ (обозначается ${X\in\mathrm{N}(m, R)}$), если его характеристическая функция задается формулой
$$
\varphi_{X}(\shmaxg{s}) \overset{\text{def}}{=} \Exp e^{i\langle \shmaxg{s},X \rangle} = \exp\left(i\langle \shmaxg{s},m \rangle - \frac{1}{2}\langle \shmaxg{s}, R\shmaxg{s} \rangle\right), \ \shmaxg{s}\in\Rbb^n.
$$ Если $\det R \neq 0$, то $X$ обладает плотностью распределения
$$
p_{X}(u) = \frac{1}{\sqrt{(2\pi)^n \det R}} \exp\left(-\frac{1}{2}\langle u, R^{-1}u \rangle\right), \ u\in\Rbb^n.
$$ Если же ${\det R = 0}$, то плотность распределения в $\Rbb^n$ не существует. В этом случае найдется вектор $c$ такой, что ${Rc=0}$. Но так как\linebreak $R=\mathbb{E}(XX^{\top})$, то из цепочки равенств
$$
0 = c^{\top} Rc = c^{\top}\Exp(XX^{\top}) c = \Exp(c^{\top}XX^{\top}c) = \Exp(c^{\top}X)^2 = \Exp\langle c,X\rangle^2$$ вытекает равенство $\langle X,c \rangle = 0$ почти всюду, то есть некоторые компоненты вектора $X$ являются линейными комбинациями других. В~этом случае вероятностная мера сосредоточена в пространстве меньшей размерности.

Класс нормальных случайных векторов $X$ замкнут относительно линейных преобразований. А именно, если $X\in\mathrm{N}(m, R)$ -- нормальный случайный вектор с $n$ компонентами, то для любой (даже прямоугольной или невырожденной) матрицы $A \in\mathbb{R}^{m\times n}$ и любого вектора $b\in \mathbb{R}^m$ случайный вектор (или случайная величина для $m=1$) $Y = AX + b \in \mathrm{N}(A m + b, A R A^{\top})$. Этот факт легко доказать с помощью аппарата характеристических функций. Действительно,
\begin{eqnarray*}
    \varphi_{Y}(\shmaxg{s}) &=&  \Exp e^{i\langle \shmaxg{s}, Y \rangle} = \Exp e^{i\langle \shmaxg{s}, AX + b \rangle} = e^{i\langle \shmaxg{s}, b \rangle} \Exp e^{i\langle A^{\top} \shmaxg{s},X  \rangle} = e^{i\langle \shmaxg{s}, b \rangle} \varphi_{X}(A^{\top} \shmaxg{s})=\\
    &=&\exp\left(i\langle \shmaxg{s},b \rangle+i\langle A^{\top} \shmaxg{s},m \rangle - \frac{1}{2}\langle A^{\top} \shmaxg{s}, R A^{\top} \shmaxg{s} \rangle\right)= \\
    &=&\exp\left(i\langle \shmaxg{s},(A m + b) \rangle - \frac{1}{2}\langle \shmaxg{s},A R A^{\top} \shmaxg{s} \rangle\right), \ \shmaxg{s}\in\Rbb^m.
\end{eqnarray*}

Для нормальных случайных векторов известны явные выражения для моментов вида $\Exp(X_1^{\alpha_1}\ldots X_n^{\alpha_n})$, где $X=(X_1,\dots,X_n)^{\top}$. А именно, справедлива следующая теорема.

\begin{theorem}[ (Вика)]\label{th:WickTheorem}
\textit{Пусть дан нормальный случайный вектор\linebreak ${X=(X_1,\dots,X_n)\in\mathrm{N}(0, R)}$ с нулевыми средними и произвольной \gav{корреляционной} матрицей ${R=\left\| R_{ij} \right\|_{i,j=1}^n}$. Тогда если $n$ нечетно, то}
$$\mathbb{E}(X_1\dots X_n)=0.$$
\textit{Если же $n$ четно, то $$\mathbb{E}(X_1 X_2 \dots X_n)=\sum \underbrace{R_{i_1 j_1}\dots R_{i_{n/2} j_{n/2}}}_{n/2\text{ множителей}},$$ где сумма берется по всем неупорядоченным разбиениям множества $\{1,\dots,n\}$ на $n/2$ неупорядоченных пар.}
\end{theorem}

\begin{proof}
Для доказательства случая с нечетным $n$ заметим, что вектор $Y=-X$ имеет то же распределение, что и вектор $X$, а значит и тот же момент $\mathbb{E}(Y_1\dots Y_n)=\mathbb{E}(X_1\dots X_n)$. Но с другой стороны $\mathbb{E}(Y_1\dots Y_n)=(-1)^n\mathbb{E}(X_1\dots X_n)$. Для нечетного $n$ отсюда следует, что $\mathbb{E}(X_1\dots X_n)=0$.

Для доказательства второго случая воспользуемся связью между смешанными моментами и производящей функцией моментов:
$$\mathbb{E}(X_1\dots X_n)=\frac{\partial^n}{\partial s_1 \dots \partial s_n} \mathbb{E}e^{\langle s^\top X \rangle}\biggr\rvert_{s=0}=\frac{\partial^n}{\partial s_1 \dots \partial s_n} \exp\left( \frac{1}{2} s^\top R s \right)\biggr\rvert_{s=0},$$
где $s=(s_1,\dots,s_n)$. Для расчета смешанной производной в правой части выражения выше выполним разложение функции $\exp\left( \frac{1}{2} s^\top R s \right)$ по степеням компонент $s$ и выделим только слагаемые со степенью $s_1\dots s_n$, так как остальные слагаемые после взятия производной в точке $s=0$ обращаются в нуль. Для этого будем в разложении отбрасывать члены, содержащие $s_i$, $i=1,\dots,n$, во второй степени и выше. Сначала заметим, что $$\exp\left(\frac{1}{2} s^\top R s\right) = \exp\left(\sum_{i<j} R_{ij}s_is_j + A(s)\right),$$ где $A(s)$ содержит не нужные нам слагаемые. Далее,
$$\exp\left(\sum_{i<j} R_{ij}s_is_j\right)=\prod_{i<j} \exp(R_{ij}s_is_j)=\prod_{i<j} (1+R_{ij}s_is_j+B_{ij}(s)),$$ где функции $B_{ij}(s)$ содержат не нужные нам слагаемые. В итоге получается, что
$$\frac{\partial^n}{\partial s_1 \dots \partial s_n} \exp\left( \frac{1}{2} s^\top R s \right)\biggr\rvert_{s=0} = \frac{\partial^n}{\partial s_1 \dots \partial s_n} \prod_{i<j} (1+R_{ij}s_is_j)\biggr\rvert_{s=0}.$$ Для расчета полученной смешанной производной нужно распределить $n$ частных производных, по множителям, а точнее -- распределить $n/2$ пар $(s_i,s_j)$, взятых из множества $\{s_1,\dots,s_n\}$ по множителям. Получается, что  
$$\frac{\partial^n}{\partial s_1 \dots \partial s_n} \prod_{i<j} (1+R_{ij}s_is_j)\biggr\rvert_{s=0}=\sum \underbrace{R_{i_1 j_1}\dots R_{i_{n/2} j_{n/2}}}_{n/2\text{ множителей}},$$ где сумма берется по всем неупорядоченным разбиениям $$\{1,\dots,n\} = \{i_1,j_1\} \cup \dots \cup \{i_{n/2},j_{n/2}\}$$ множества $\{1,\dots,n\}$ на $n/2$ неупорядоченных пар. \EndProof
\end{proof}

\textbf{Замечание}. Упомянутая теорема позволяет вычислять произвольные моменты вида $\mathbb{E}(X_1^{\alpha_1}\dots X_n^{\alpha_n})$. Для этого можно просто рассмотреть нормальный вектор $Y$ с $N=\alpha_1+\ldots+\alpha_n$ компонентами, воспользоваться формулой для расчета момента $\mathbb{E}(Y_1\dots Y_N)$ и затем заменить все $Y_i$, $i=1,\dots,\alpha_1$ на $X_1$, все $Y_i$, $i=1+\alpha_1,\dots,\alpha_1+\alpha_2$ на $X_2$ и т.д. Поясним сказанное на примерах.

\begin{example}
Пусть имеется нормальный вектор $(X_1,X_2,X_3,X_4)$ с нулевым средним и \gav{корреляционной} матрицей $R=\left\|R_{ij}\right\|_{i,j=1}^4$. Тогда
$$\mathbb{E}(X_1 X_2 X_3 X_4) = R_{12}R_{34} + R_{13}R_{24} + R_{14} R_{23}.$$ Каждое слагаемое содержит $n/2=4/2=2$ сомножителя. Количество слагаемых определяется количеством разбиения множества $\{1,2,3,4\}$ на $n/2=2$ пары. Множество $\{1,2,3,4\}$ на 2 пары можно разбить тремя способами: $\{1,2\}\cup\{3,4\}$, $\{1,3\}\cup\{2,4\}$, $\{1,4\}\cup\{2,3\}$.

Так как любой подвектор нормального вектора является нормальным, то и $(X_1,X_2,X_3)$ является нормальным случайным вектором. Для него справедливо равенство $$\mathbb{E}(X_1X_2X_3)=0,$$ так как множителей -- нечетное количество. \EndEx
\end{example}

А теперь рассмотрим задачу про предел в среднем квадратичном сумм квадратов приращений винеровского процесса. Оказывается, что этот предел равен длине отрезка времени. Если бы мы вычисляли такой предел для гладкой функции, то он равнялся бы нулю. Отсюда на качественном уровне можно сделать вывод, что \textit{траектории винеровского процесса почти нигде не дифференцируемы, хотя почти наверное всюду непрерывны}. Подробнее см.~\cite{BulinskyShiryaev2005, VentselAD}. 
\begin{example}\label{exDW}
Пусть $\{W(t), t \ge 0\}$ -- винеровский процесс, 
$[a, b] \subset$\linebreak $\subset T = [0, \infty)$ и $a = t_0 < t_1 < \dots, < t_N = b$. Покажем теперь, что в смысле с.к.-сходимости существует предел
$$
\mathop{\mathrm{l.i.m.}}\limits_{\max(t_k - t_{k-1}) \to 0} \sum_{k=1}^N |W(\omega, t_k) - W(\omega, t_{k-1})|^2 = b - a.
$$ 
Согласно определению винеровского процесса, случайные величины $W( \omega, t_k) - W(\omega, t_{k-1})$, $k = \overline{1, N}$, являются независимыми, имеют нулевые математические ожидания и дисперсии
$$
\mathbb{D}(W(\omega, t_k) - W(\omega, t_{k-1})) = \mathbb{E}|W(\omega, t_k) - W(\omega, t_{k-1})|^2 = t_k - t_{k-1}.
$$
А так как
$$\mathbb{E}  \sum\limits_{k=1}^N |W(\omega, t_k) - W(\omega, t_{k-1})|^2 = \sum\limits_{k=1}^N \mathbb{E} |W(\omega, t_k) - W(\omega, t_{k-1})|^2=$$
$$=\sum\limits_{k=1}^N (t_k - t_{k-1})  = b - a,$$
то очевидно, что
\begin{equation*}
\begin{split}
    &\lim\limits_{\max(t_k - t_{k-1}) \to 0} \mathbb{E}  \left( \sum\limits_{k=1}^N |W(\omega, t_k) - W(\omega, t_{k-1})|^2 - (b - a)^2 \right) = \\
    &=\lim\limits_{\max(t_k - t_{k-1}) \to 0} \mathbb{D} \sum\limits_{k=1}^N |W(\omega, t_k) - W(\omega, t_{k-1})|^2 = \\
    &=\lim\limits_{\max(t_k - t_{k-1}) \to 0} \sum\limits_{k=1}^N \mathbb{D} |W(\omega, t_k) - W(\omega, t_{k-1})|^2.
\end{split}
\end{equation*}
Кроме того, случайная величина
$$
\eta_k(\omega) = \dfrac{W(\omega, t_k) - W(\omega, t_{k-1})}{\sqrt{t_k - t_{k-1}}}
$$
распределена по нормальному закону с нулевым математическим ожиданием и единичной дисперсией, а случайная величина $\eta^2_k(\omega)$ распределена по закону $\chi^2_1$. Поэтому $\mathbb{D}\eta^2_k(\omega) \equiv 2 $ и
$$
\mathbb{D}  |W(\omega, t_k) - W(\omega, t_{k-1})|^2 \equiv 2(t_k - t_{k-1})^2.
$$
Таким образом,
\begin{equation*}
\begin{split}
    &\lim\limits_{\max(t_k - t_{k-1}) \to 0} \mathbb{E} \left( \sum\limits_{k=1}^N |W(\omega, t_k) - W(\omega, t_{k-1})|^2 - (b - a)^2 \right) = \\
    &=2 \lim\limits_{\max(t_k - t_{k-1}) \to 0} \sum\limits_{k=1}^N (t_k - t_{k-1})^2 \leq \\
    &\egor{\leq}2 \lim\limits_{\max(t_k - t_{k-1}) \to 0} \max\limits_k (t_k - t_{k-1}) \sum\limits_{k=1}^N (t_k - t_{k-1}) = \\
    &=2 (b - a) \lim\limits_{\max(t_k - t_{k-1}) \to 0} \max\limits_k (t_k - t_{k-1}) = 0,
\end{split}
\end{equation*}
что и требовалось доказать.
\EndEx
\end{example}

\subsection{Условные распределения сечений}

Пусть дан нормальный случайный вектор $X$ с $n$ компонентами. Пусть $X=(Y,Z)$, где $Y$ и $Z$ -- два подвектора вектора $X$ с $n_1$ и $n_2$ компонентами, соответственно; ${n=n_1+n_2}$. \gav{Корреляционную} матрицу вектора $X$ можно представить в виде 
$${R} = \left[ {\begin{array}{*{20}{c}}
  {{{R}_{11}}}&{{{R}_{12}}} \\ 
  {{{R}_{21}}}&{{{R}_{22}}} 
\end{array}} \right],$$ где $R_{12}$ -- \gav{корреляционная} матрица вектора $Y$, $R_{22}$ -- \gav{корреляционная} матрица вектора $Z$, а матрицы $R_{12}$ и $R_{21}=R_{12}^{\top}$ состоят из \gav{корреляций} компонент вектора $Y$ и $Z$ (взаимные корреляционные матрицы). Вектор математического ожидания $\Exp X = m$ также разбивается на два подвектора $\Exp Y = m_1$ и $\Exp Z = m_2$.

Найдем условное распределение подвектора  $Y$ при фиксированном подвекторе $Z$. Будем предполагать, что матрица $R_{22}$ невырождена (если это не так, то некоторые компоненты $Z$ являются линейными комбинациями других, а значит, их можно исключить, понизив размерность $Z$).

Введем \shmaxg{случайную величину}
$$
\hat Y  = m_1 + R_{12}R_{22}^{-1} (Z - m_2).
$$
Заметим, что \shmaxg{для любых матриц $A$, $B$, $C$ верны равенства
$$\mathrm{tr}(AB)=\mathrm{tr}(BA), \ \mathrm{tr}(ABC)=\mathrm{tr}(BCA)=\mathrm{tr}(CAB).$$ Отсюда можно получить, что
\begin{equation}
\label{orthogonal}
    \Exp ( Y - \hat{Y}  )^{\top} \left (Z - m_2 \right) = 0.
\end{equation}
}
Действительно,
$$\shmaxg{\Exp ( Y - \hat{Y} )^{\top} \left (Z - m_2 \right) =}$$
$$\shmaxg{=\Exp \left ( Y - m_1 -  R_{12}R_{22}^{-1} (Z - m_2) \right)^{\top} \left (Z - m_2 \right) =}$$
$$\shmaxg{=\mathrm{tr}\left(R_{12}\right) - \mathrm{tr}\left(R_{12}R_{22}^{-1} R_{22}\right) = 0.}$$
Условие~\eqref{orthogonal} означает, что компоненты вектора ${Y - \hat{Y}}$ и компоненты вектора $Z$ некоррелированны, а значит, и независимы (в силу гауссовости эти понятия совпадают).

Таким образом, вектор
$Y$ раскладывается в сумму нормальных векторов $Y =  ( Y - \hat{Y} ) + \hat{Y} = ( Y - \hat{Y} ) +  m_1 + R_{12}R_{22}^{-1} (Z - m_2)$,
где первое слагаемое $ Y - \hat{Y} $ не зависит от $Z$, а \shmaxg{$\hat{Y}$} при \shmaxg{фиксированном} $Z$ является константой. Значит, условное распределение $Y$ при \shmaxg{фиксированном} $Z$ определяется нормальным законом с условным математическим ожиданием
\begin{eqnarray}\label{eq:X1condX2}
\Exp (Y \,|\, Z ) &=& \Exp \left ( Y - \hat{Y} \right) +  m_1 + R_{12}R_{22}^{-1} (Z - m_2)= \notag
\\ &=& m_1 + R_{12}R_{22}^{-1} (Z - m_2)
\end{eqnarray}

\noindent и условной корреляционной матрицей
\begin{equation}\label{eq:X1X1condX2}
     \begin{split}
         &\Exp \left( \left( Y - \Exp(Y \,|\, Z ) \right) \left(Y - \Exp(Y \,|\, Z) \right)^{\top} \,\big|\, Z \right) = \Exp( Y - \hat{Y} )  ( Y - \hat{Y} )^{\top} = \\
         &=\Exp \left ( Y - m_1 \right)\left ( Y - m_1 \right)^{\top} - R_{12}R_{22}^{-1} \Exp \left (Z - m_2 \right)\left ( Y - m_1 \right)^{\top} - \\
         &-\Exp \left ( Y - m_1 \right)\left (Z - m_2 \right)^{\top}R_{22}^{-1}R_{21} + \\
         &+R_{12}R_{22}^{-1} \Exp \left (Z - m_2 \right)\left ( Z - m_2 \right)^{\top} R_{22}^{-1}R_{21}= \\
         &=R_{11} - R_{12}R_{22}^{-1}R_{21}. 
     \end{split}
\end{equation}

Отметим, что условная \gav{корреляционная} матрица не зависит от $Z$, а условное математическое ожидание является линейной функцией от~$Z$.

Приведем и докажем теперь важную и полезную теорему о том, что условное математическое ожидание является на самом деле проекцией на подпространство функций от случайных величин, стоящих в условии условного математического ожидания. Эта теорема, с одной стороны, предоставляет геометрическую интуицию этого понятия, а с другой стороны, помогает решать задачи. Ради простоты мы докажем эту теорему для случая, когда условное математическое ожидание вычисляется от случайной величины, хотя аналогичные выкладки можно провести и для случайного вектора.

\begin{theorem}\label{th:cond_expectation}
\textit{Пусть случайная величина $X$ имеет ограниченный второй момент, т.е. ${X \in L_2}$. Пусть также ${Y_1,\dots,Y_n\in L_2}$. Доказать, что}
$$
\|X - \EE\left(X\mid Y_1,\dots,Y_n\right)\|_2 = \min\limits_{\varphi\in H}\|X - \varphi(Y_1,\dots,Y_n)\|_2,
$$
\textit{где $H$~--- подпространство пространства $L_2$ всевозможных борелевских функций $\varphi(Y_1,\dots,Y_n)\in L_2$, $\EE\left(X\mid Y_1,\dots, Y_n\right)$~--- условное математическое ожидание случайной величины $X$ относительно $\sigma$-алгеб\-ры, порожденной случайными величинами $Y_1,\dots,Y_n$. Напомним, что норма и скалярное произведение определяются следующим образом}:
$$
\|X\|_2 = \shmaxg{\sqrt{\langle X,X\rangle}},\quad \langle X, Y \rangle = \EE\left(X Y\right).
$$
\end{theorem}

\textbf{Доказательство}. Заметим, что указанное равенство равносильно следующему утверждению:
$$
\forall \varphi\in H \; \langle X - \EE\left(X\mid Y_1,\dots,Y_n\right), \varphi(Y_1,\dots,Y_n) \rangle = 0,
$$
т.е. разность $X - \EE\left(X\mid Y_1,\dots,Y_n\right)$ ортогональна пространству $H$. Действительно, если это доказать, то мы докажем и исходное равенство, т.к. $$\EE\left(X\mid Y_1,\dots,Y_n\right)\in H$$ по определению условного математического ожидания относительно набора случайных величин (см. в~\cite[гл. 2, \S~7]{ShiryaevT1}). Во-первых, если случайная величина $Z$ измерима относительно $\sigma$-алгебры, порожденной случайными величинами $Y_1,\dots,Y_n$, то
$$
\EE\left(XZ\mid Y_1,\dots,Y_n\right) = Z\EE\left(X\mid Y_1,\dots, Y_n\right).
$$
Это свойство является одним из основных свойств условного математического ожидания и доказывается в два этапа: сначала для простых случайных величин, а затем при помощи предельного перехода под знаком условного математического ожидания доказывается для произвольных случайных величин из $H$ (детали см. в \cite[гл. 2, \S 7, свойство K*]{ShiryaevT1}). Кроме того, напомним еще одно важнейшее свойство условного математического ожидания (см. \cite[гл. 2, \S~7, свойства H* и I*]{ShiryaevT1}):
$$
\EE X = \EE\left(\EE\left(X\mid Y_1,\dots,Y_n\right)\right).
$$
Используя эти факты, получаем для произвольной $\varphi \in H$
\begin{equation*}
\begin{split}
    &\EE\left((X - \EE\left(X\mid Y_1,\dots,Y_n\right))\cdot\varphi(Y_1,\dots,Y_n)\right) = \\ 
    &=\EE\left(\EE\left((X - \EE\left(X\mid Y_1,\dots,Y_n\right))\cdot\varphi(Y_1,\dots,Y_n)\mid Y_1,\dots,Y_n\right)\right) = \\ 
    &=\EE\left(\varphi(Y_1,\dots,Y_n)\EE\left(X - \EE\left(X\mid Y_1,\dots,Y_n\right)\mid Y_1,\dots,Y_n\right)\right) = \\ 
    &=\EE\left(\varphi(Y_1,\dots,Y_n)\left(\EE\left(X\mid Y_1,\dots,Y_n\right) - \EE\left(X\mid Y_1,\dots,Y_n\right)\right)\right) = 0,
\end{split}
\end{equation*}
а значит,
$$
\langle X - \EE\left(X\mid Y_1,\dots,Y_n\right), \varphi(Y_1,\dots,Y_n) \rangle = 0.
$$
В случае, когда вектор $(X,Y_1,\dots,Y_n)$ является нормальным случайным вектором, можно показать \shmaxg{(см.~\cite[Глава 2, §13]{ShiryaevT1})}, что \shmaxg{условное математическое ожидание} ${\EE\left(X\mid Y_1,\dots,Y_n\right)}$ является линейной функцией от $Y_1,\dots, Y_n$: \shmaxg{$$\EE\left(X\mid Y_1,\dots,Y_n\right) = \sum_{k=1}^n c_k Y_k \ag{+ b},$$} а значит, можно искать проекцию именно в таком виде. Это следует из того, что для гауссовских случайных величин независимость равносильна некоррелированности. Отсюда, в частности следует, что часто прогнозы для гауссовских процессов, оптимальные в смысле квадрата невязки, получаются по явным и простым формулам. Подробности можно найти в заключении данной книги.
\EndProof

\begin{example}
Для винеровского процесса $\{W(t), t \ge 0\}$, вычислить условные математическое ожидание и дисперсию $$\mathbb{E}\left( W(t) \,|\, W(s) = x \right), \ \mathbb{D}\left( W(t) \,|\, W(s) = x \right) $$ для произвольных $t \ge 0$, $s \ge 0$, $x\in\mathbb{R}$.
\end{example}

\textbf{Решение}. Так как все конечномерные распределения винеровского процесса являются нормальными, то случайные векторы, составленные из его сечений, являются нормальными случайными векторами. В частности, вектор $(W(t),W(s))$ является нормальным с нулевым вектором математических ожиданий и корреляционной матрицей $$R=\left[ {\begin{array}{*{20}{c}}
	{t}&{\min(t,s)} \\ 
	{\min(t,s)}&{s} 
	\end{array}} \right].$$
По формуле~\eqref{eq:X1condX2} условное математическое ожидание равно $$\mathbb{E}(\xi\,|\,Z=x)=\mathbb{E}\xi + R_{12}R_{22}^{-1}(x-\mathbb{E}\eta ),$$
а условная дисперсия, по формуле~\eqref{eq:X1X1condX2}, равна
$$\mathbb{D}(\xi\,|\,\eta=x) = \mathbb{D}\xi - R_{12}R_{22}^{-1}R_{21}.$$
где $R_{ij}$ -- это компонента матрицы $R$ с индексами $i,j$. В нашем случае $$\mathbb{E}\left( W(t) \,|\, W(s) = x \right) = \mathbb{E}W(t) + \min(t,s)\cdot\frac{1}{s}\cdot(x-\mathbb{E}W(s))=\frac{\min(t,s)}{s}x.$$ В частности, при $t>s$ имеем $\mathbb{E}\left( W(t) \,|\, W(s) = x \right)=x$. Условная дисперсия равна $$\mathbb{D}\left( W(t) \,|\, W(s) = x \right) = t - \min(t,s)^2\cdot\frac{1}{s}=\frac{st-\min(t,s)^2}{s}.$$ В частности, она равна $t-s$ при $t>s$. \EndEx
 
 \textbf{Замечание}. Совершенно аналогично решается задача о поиске условной плотности распределения винеровского процесса $W(t)$ при условии, что $W(t_1)=A$ и $W(t_2)=B$, $t_1 < t < t_2$. В этом случае рассмотрим нормальный случайный вектор $(W(t),W(t_1),W(t_2))$. Тогда $W(t)$ при условии ${W(t_1)=A}$ и $W(t_2)=B$ будет также иметь некоторое нормальное распределение $\mathrm{N}(\mu,\sigma^2)$, остается только найти среднее $$\mu=\mathbb{E}\left(W(t) \,|\, W(t_1)=A,W(t_2)=B\right)$$ и дисперсию $$\sigma^2=\mathbb{D}\left(W(t) \,|\, W(t_1)=A,W(t_2)=B\right).$$ Корреляционная матрица $R$ вектора $(W(t),W(t_1),W(t_2))$ имеет размеры $3\times 3$. Из формул~\eqref{eq:X1condX2} и~\eqref{eq:X1X1condX2} следует
\begin{equation*}
\begin{split}
    &\mathbb{E}\left(W(t) \,|\, W(t_1) = A, W(t_2) = B\right) = \\
    &=\left[ {\begin{array}{*{20}{c}}
     {R_{12}}&{R_{13}} \\ 
     \end{array}} \right] \left[ {\begin{array}{*{20}{c}}
     {R_{22}}&{R_{23}} \\ 
     {R_{32}}&{R_{33}} 
     \end{array}} \right]^{-1} \left[ {\begin{array}{*{20}{c}}
     {A}\\{B}\\ 
     \end{array}} \right], \\ \\
     &\mathbb{D}\left(W(t) \,|\, W(t_1) = A, W(t_2) = B\right) = \\
     &=R_{11} - \left[ {\begin{array}{*{20}{c}}
     {R_{12}}&{R_{13}} \\ 
     \end{array}} \right] \left[ {\begin{array}{*{20}{c}}
     {R_{22}}&{R_{23}} \\ 
     {R_{32}}&{R_{33}} 
     \end{array}} \right]^{-1} \left[ {\begin{array}{*{20}{c}}
     {R_{21}}\\{R_{31}}\\ 
     \end{array}} \right].
\end{split}
\end{equation*}
 
 	
 
 
Остается лишь выразить $R$ в терминах $t$, $t_1$ и $t_2$: $$R = \left[ {\begin{array}{*{20}{c}}
	{t}&{t_1}&{t}\\
	{t_1}&{t_1}&{t_1}\\
	{t}&{t_1}&{t_2}
	\end{array}} \right]$$
и получить ответ.

\section{Стационарные процессы}
\label{spectral}

В этом разделе мы приступаем к изучению одного из важнейших понятий теории случайных процессов -- \textit{стационарности}. Стационарные случайные процессы активным образом используются в инженерных приложениях и в особенности -- в теории управления динамическими системами. В данном разделе дается определение стационарного процесса, разбираются базовые примеры стационарных процессов и базовый инструментарий для работы с такими процессами -- разложение стационарного процесса и его корреляционной функции на гармоники.

\subsection{Комплекснозначные случайные процессы}

Теория стационарных случайных процессов помимо веществен\-но\-значных процессов (принимающих реализации в $\mathbb{R}$) опирается на понятие\textit{ комплекснозначных процессов} (принимающих значения в $\mathbb{C}$). Введем аккуратные определения и распространим некоторые ранее введенные понятия на случай комплекснозначных процессов.

\begin{definition}
\textit{Комплекснозначным процессом} ${\{Z(t), t\in T\}}$, определенном на вероятностном пространстве $(\Omega,\mathcal{F},\mathbb{P})$, называется\linebreak функция ${Z(\omega,t)=X(\omega,t)+i Y(\omega,t)}$, где ${i^2=-1}$, $\omega\in\Omega$, а $X(t)$ и $Y(t)$ -- два вещественнозначных процесса, определенных при $t\in T$ и принадлежащих тому же вероятностному пространству $(\Omega,\mathcal{F},\mathbb{P})$.
\end{definition}

Определим математическое ожидание такого процесса $Z(t)$ по формуле ${\mathbb{E}Z(t)=\mathbb{E}X(t)+i\mathbb{E}Y(t)}$, где ${X(t)=\Re Z(t)}$ -- вещественная часть процесса $Z(t)$, а ${Y(t)= \Im Z(t)}$ -- мнимая часть процесса $Z(t)$. Дисперсией комплекснозначного процесса $Z(t)$ будем называть $$\mathbb{D}Z(t)=\mathbb{E}\accentset{\circ}Z(t)\overline{\accentset{\circ}Z(t)},$$ а корреляционной функцией 
$$R_Z(t,s)=\mathbb{E}\accentset{\circ}Z(t)\overline{\accentset{\circ}Z(s)},$$ где черта означает комплексное сопряжение.

Математическое ожидание и корреляционная функция комплекснозначного процесса могут принимать комплексные значения. Дисперсия, как и прежде, может принимать только вещественные значения.

\begin{definition}
\textit{Комплекснозначным процессом $\{Z(t)$, ${t\in T}\}$ второго порядка} будем называть комплекснозначный случайный процесс со всюду конечным вторым моментом $\mathbb{E}|Z(t)|^2<\infty$, $t\in T$.
\end{definition}

Множество комплекснозначных случайных процессов второго порядка будем обозначать символом $CL_2$, т.е. будем писать $X(t)\in CL_2$. Скалярное произведение на $CL_2$ задается как $\langle X, Y\rangle = \Exp X\mean{Y} $.

Кроме того, введем полезное для дальнейшего понятие комплекснозначного гауссовского процесса.

\begin{definition}
${\{Z(t), t\in T\}}$ -- \textit{комплекснозначный гауссовский процесс}, если для любого ${n\in \mathbb{N}}$ и любых $t_1,\ldots,t_n \in T$ вектор $$\left ( X(t_1),  Y(t_1), \ldots , X(t_n), Y(t_n) \right),$$ где ${X(t)=\Re Z(t)}$, ${Y(t)=\Im Z(t)}$, имеет нормальное распределение.
\end{definition}

Вещественнозначные процессы мы будем считать частным случаем комплекснозначных процессов (с нулевой мнимой частью). При этом, если речь идет о комплекснозначном процессе, то его мнимая часть может быть как ненулевой, так и нулевой. Если же рассматривается вещественнозначный процесс, то считается, что мнимая часть нулевая.

\subsection{Стационарность в широком и узком смыслах}

\begin{definition}
Процесс ${X(t)\in CL_2}$ называется \textit{стационар\-ным в узком смысле} (еще говорят, \textit{сильно стационарным}), если ${\forall n\in \Nbb}$, для любых моментов времени $t_1$, $\dots$, $t_n$ и любого ${h > 0}$ распределение вектора ${(X(t_1),\dots,X(t_n))}$ совпадает с распределением вектора $(X(t_1+h),\dots,X(t_n+h))$.
\end{definition}

Иными словами, процесс стационарен в узком смысле, если все его конечномерные распределения не зависят от сдвига моментов времени на одну и ту же величину. Отсюда следует, в частности для $n=1$, что распределение $X(t)$ совпадает с распределением $X(t+h)$ для любого $h$, т.е. одномерное распределение не зависит от времени. Но это значит, что и никакие численные характеристики одномерного распределения такого процесса не зависят от времени. Например, математическое ожидание и дисперсия стационарного в узком смысле процесса не зависят от времени: ${m_X(t)=\mathrm{const}}$, ${D_X(t)=\mathrm{const}}$.

Что касается двумерного распределения стационарного в узком смысле процесса, то оно зависит лишь от разности $t_2-t_1$. Следовательно, и все численные характеристики двумерного распределения (например, корреляционная и ковариационная функции) тоже зависят лишь от разности между $t_1$ и $t_2$. Это значит, что существует функция $R(\tau)$ одного аргумента, такая, что корреляционная функция стационарного в узком смысле процесса $X(t)$ равна $R_X(t_1,t_2)=R(t_2-t_1)$.

\begin{definition}
Процесс ${X(t)\in CL_2}$ называется \textit{стационарным в широком смысле} (еще говорят, \textit{слабо стационарным}), если его математическое ожидание не зависит от времени, а корреляционная функция зависит лишь от разности аргументов $t_1$ и $t_2$.
\end{definition}

Из того, что у стационарных в широком смысле процессов корреляционная функция является функцией лишь разности аргументов, следует, что дисперсия не зависит от времени. Таким образом, дисперсия постоянна как для стационарного в узком смысле процесса (если это процесс второго порядка), так и для стационарного в широком смысле процесса.

Как видно из определений, для процессов второго порядка из стационарности в узком смысле следует стационарность в широком смысле. Вне класса процессов второго порядка это следствие уже не верно, так как определение стационарности в широком смысле предполагает существование вторых моментов сечений процесса. Например, можно рассмотреть процесс $X(t)$, для любого $t\ge0$, равного
$$X(t)=\xi\in\mathrm{C}(0,1),$$ где $\mathrm{C}(0,1)$ -- стандартное распределение Коши с плотностью $$f(x)=\frac{1}{\pi(1+x^2)}.$$ Такой процесс не зависит от времени, следовательно, является стационарным в узком смысле процессом. Однако, хорошо известно, что распределение Коши не имеет даже первый момент, поэтому $X(t)$ пространству процессов второго порядка не принадлежит, и понятие стационарности в широком смысле к нему неприменимо.

Итак, простейшим примером стационарного в узком смысле процесса является не зависящий от времени процесс ${X(t)=\xi}$, где $\xi$ -- какая-нибудь случайная величина. Если $\xi$ обладает конечным вторым моментом ${\mathbb{E}|\xi|^2<\infty}$, то процесс $X(t)$ будет стационарным и в широком смысле, потому что является стационарным в узком смысле. Математическое ожидание этого процесса $\mathbb{E}X(t)=\mathbb{E}\xi$ не зависит от времени, а корреляционная функция $R_X(t,s)=\mathbb{E}\xi^2 - (\mathbb{E}\xi)^2$ вообще не зависит ни от $t$, ни от $s$, поэтому и подавно ${R_X(t,s)=R_X(t+h,s+h)}$ для любого $h$. В частности, если ${\xi=C=\const}$, т.е. это вырожденная случайная величина, принимающая одно значение независимо от исхода, то $\mathbb{E}X(t)=C$ и $R_X(t,s)=0$.

А теперь возьмем какую-нибудь случайную величину $\xi$ c конечным вторым моментом $\mathbb{E}|\xi|^2<\infty$ и рассмотрим случайный процесс $$X(t)=\xi f(t)$$ с какой-нибудь (быть может комплекснозначной) неслучайной и непостоянной функцией $f(t)$. Попробуем выяснить, в каких случаях этот процесс будет стационарным в широком смысле. Математическое ожидание ${\mathbb{E}X(t)=f(t)\mathbb{E}\xi}$ не будет зависеть от времени тогда и только тогда, когда $\mathbb{E}\xi=0$, в этом случае $\mathbb{E}X(t)=0$. Корреляционная функция этого процесса равна $$R_X(t,s)=\mathbb{E}X(t)\overline{X(s)}-\mathbb{E}X(t)\mathbb{E}\overline{X(s)}=f(t)\overline{f(s)}\cdot\mathbb{E}|\xi|^2.$$ Выясним, в каких случаях ${R_X(t,s)=R_X(t+h,s+h)}$ для любых $t$, $s$ и $h$. Пусть сначала ${t=s}$, тогда ${R_X(t,t)=R_X(t+h,t+h)}$ приводит к $|f(t)|^2=\const$, откуда сразу следует ${f(t)=r\exp{(i\varphi(t))}}$ для произвольных ненулевых $r,\varphi(t)\in\mathbb{R}$. Предположим дополнительно, что $f(t)$ -- всюду непрерывная функция. Функция $R_X(t+h,s+h)$ не будет зависеть от $h$ тогда и только тогда, когда $\varphi(t+h)-\varphi(s+h)$ не будет зависеть от $h$, т.е. тогда и только тогда, когда
$$\varphi(t+h)-\varphi(s+h)=\varphi(t)-\varphi(s).$$
Без потери общности будем считать, что $s=0$, и перепишем это выражение в виде
$$\varphi(t+h)=\varphi(t)+\varphi(h)-\varphi(0).$$ Теперь если ввести обозначение $g(t)=\varphi(t)-\varphi(0)$, то мы имеем непрерывную функцию $g(t)$, во всех точках удовлетворяющую уравнению Гамеля \ag{\cite{Feller}}
$$g(t+h)=g(t)+g(h), \ \forall t,h\in\mathbb{R}.$$ Можно доказать, что решением этого уравнения являются функции вида $g(t)=\omega t$, где $\omega\in\mathbb{R}$ -- произвольная постоянная, не зависящая от $t$. Обозначив $\theta=-\varphi(0)$, мы приходим к выражению для $\varphi(t)$: $$\varphi(t)=\omega t + \theta, \ \text{где }\omega,\theta\in\mathbb{R}.$$ Итак, мы выяснили, что процесс вида $X(t)=\xi f(t)$ будет стационарным в широком смысле тогда и только тогда, когда \begin{equation}\label{eq:StatHarmonicProcess}
    X(t)=\xi \cdot r e^{i(\omega t+\theta)}.
\end{equation}
Случайная величина $\xi$ в~\eqref{eq:StatHarmonicProcess} может быть и комплекснозначная. Теперь если объединить этот случай с независящими от времени процессами, а число $r$ включить в состав случайной величины $\xi$, то мы получаем следующее утверждение.
\begin{theorem}
\textit{Случайный процесс вида ${X(t)=\xi f(t)}$ для $\mathbb{E}|\xi|^2\!<\!\infty$ и непрерывной неслучайной функции $f(t)$ будет стационарным в широком смысле тогда и только тогда, когда $X(t)=\xi \exp{(i(\omega t+\theta))}$ для  неслучайных} $\omega,\theta\in\mathbb{R}$.
\end{theorem}

Теперь рассмотрим процесс вида $X(t)=\xi_1 e^{i\omega_1 t} + \xi_2 e^{i\omega_2 t}$ с $\mathbb{E}\xi_1=0$ и $\mathbb{E}\xi_2=0$ и ненулевыми частотами $\omega_1\ne\omega_2$ и выясним, при каких условиях этот процесс будет стационарным в широком смысле. Математическое ожидание $\mathbb{E}X(t)=0$ не зависит от времени. Корреляционная функция
\begin{eqnarray*}
R_X(t,s) &=& \mathbb{E}|\xi_1|^2e^{i\omega_1\tau} + \mathbb{E}(\xi_1\overline{\xi_2})e^{i(\omega_1-\omega_2)t+i\omega_1 \tau}+ \\
&& + \ \mathbb{E}(\xi_2\overline{\xi_1})e^{-i(\omega_1-\omega_2)t+i\omega_2 \tau} + \mathbb{E}|\xi_2|^2e^{i\omega_2\tau},
\end{eqnarray*}
где $\tau=t-s$, является функцией лишь $\tau$ тогда и только тогда, когда $\mathbb{E}(\xi_1\overline{\xi_2})=\mathbb{E}(\xi_2\overline{\xi_1})=0$, что следует из линейной независимости функцией перед этими коэффициентами. Получается, что процесс $X(t)=$\linebreak $=\xi_1 e^{i\omega_1 t} + \xi_2 e^{i\omega_2 t}$ является стационарным в широком смысле тогда и только тогда, когда случайные величины $\xi_1$ и $\xi_2$ некоррелированы. Корреляционная функция в этом случае равна $$R_X(t,s)=\mathbb{E}|\xi_1|^2e^{i\omega_1(t-s)}+\mathbb{E}|\xi_2|^2e^{i\omega_2(t-s)}.$$
Рассуждая совершенно аналогичным образом, нетрудно доказать следующую теорему.
\begin{theorem}\label{th:ComplexHarmonisStationary}
\textit{Случайный процесс }
\begin{equation}\label{eq:SumHarmonicsStatProcess}
    X(t)=\sum\limits_{k=1}^n \xi_k e^{i\omega_k t}, \ \omega_k\ne\omega_m, \ \mathbb{E}\xi_k=0,
\end{equation}
\textit{является стационарным в широком смысле тогда и только тогда, когда случайные величины $\xi_k$ попарно некоррелированы. В этом случае корреляционная функция равна }
$$R_X(t,s)=R(\tau)=\sum\limits_{k=1}^n \mathbb{E}|\xi_k|^2 e^{i\omega_k\tau}, \ \tau=t-s.$$
\end{theorem}
Читателю следует здесь обратить внимание на то, что для стационарности в широком смысле необходима и достаточна некоррелированность комплексных гармоник (экспонент). Кроме того, в то время как процесс представляется в виде суммы некоррелированных гармоник с некоторыми случайными коэффициентами, корреляционная функция также представляет собой сумму гармоник с амплитудами, равными вторым моментам этих коэффициентов. Например, случайный процесс ${X(t)=\xi\cos{t}}$ не является стационарным в широком смысле не просто потому, что его корреляционная функция $$R_X(t,s)=\mathbb{E}|\xi|^2\cos{t}\cos{s}$$ не обладает свойством ${R_X(t,s)=R_X(t+h,s+h)}$ для любого $h$, но и потому что расписав ${X(t)=1/2\xi e^{it} + 1/2\xi e^{-it}}$, мы видим, что коэффициенты при экспонентах не являются некоррелированными (они равны). Эти замечания будут важны для нас потом, когда мы установим представление стационарных процессов в суммы из гармоник.

Ошибочно думать, впрочем, что некоррелированность гармоник достаточна для стационарности процесса в узком смысле. Действительно, пусть дан процесс ${Z(t)=X\cos{t} + Y\sin{t}}$, где случайные величины $X$ и $Y$ независимы и принимают значения $\pm 1$ с вероятностью $1/2$. Легко получить, что ${\mathbb{E}X(t)=0=\const}$, а корреляционная функция \begin{eqnarray*}
R_Z(t,s) &=& \mathbb{E}Z(t)Z(s)-\underbrace{\mathbb{E}Z(t)\mathbb{E}Z(s)}_{0}= \\
&=& \mathbb{E}(X\cos{t} + Y\sin{t})(X\cos{s} + Y \sin{s}) \\
&=& \underbrace{\mathbb{E}X^2}_{1}\cos{t}\cos{s} + \underbrace{\mathbb{E}Y^2}_{1}\sin{t}\sin{s}= \\ 
&& + \underbrace{\mathbb{E}XY}_{\mathbb{E}X\mathbb{E}Y=0}(\cos{t}\sin{s} + \sin{t}\cos{s})= \\
&=& \cos{t}\cos{s} + \sin{t}\sin{s}=\\ 
&=& \cos{(t-s)}
\end{eqnarray*}
зависит только от разницы $t-s$. Это говорит о стационарности процесса в широком смысле. но если бы процесс был стационарен в узком смысле, то его одномерное распределение не зависело бы от $t$. Однако для ${t=0}$ сечение ${Z(0)=X}$ принимает значения $\pm1$, а при $t=\pi/4$ сечение $${Z(\pi/4)=(X+Y)/\sqrt{2}}$$ принимает значения $\pm\sqrt{2}$ и $0$. Значит, уже одномерное распределение меняется при сдвиге времени, поэтому процесс не является стационарным в узком смысле.

Выше были рассмотрены процессы-гармоники с неслучайными частотами и фазами. Разберем пример со случайной частой и фазой.

\begin{example}
\label{14}
Дан случайный процесс ${Z(t)=A\cos{(Bt+\phi)}}$, ${t\ge0}$, в котором $A$, $B$ и $\phi$ являются случайными величинами, причем $\phi$ не зависит от $A$ и $B$ и распределено равномерно на отрезке $[0,2\pi]$. Про $A$ и $B$ известно, что они имеют совместную плотность распределения $f(a,b)$ и $A\ge0$, $B\ge0$ п.н. Исследовать процесс $Z(t)$ на стационарность в обоих смыслах.
\end{example}

\textbf{Решение}. Сначала исследуем процесс на стационарность в широком смысле. Для этого вычислим математическое ожидание $$\mathbb{E}Z(t)=\mathbb{E}A\cos{(Bt+\phi)}=\mathbb{E}A\cos{Bt}\cos{\phi}-\mathbb{E}A\sin{Bt}\sin{\phi}.$$ Так как $\phi$ не зависит от $A$ и $B$, то $\mathbb{E}A\cos{Bt}\cos{\phi}=\mathbb{E}A\cos{Bt}\cdot\mathbb{E}\cos{\phi}\!=\!0$ \elena{(в предположении, что $\mathbb{E}|A| < \infty$)}, что следует из того, что $$\mathbb{E}\cos{\phi}=\int\limits_{0}^{2\pi}\frac{1}{2\pi}\cos{x}\,dx=0.$$ Аналогично получаем, что $\mathbb{E}A\sin{Bt}\sin{\phi}=\mathbb{E}A\sin{Bt}\cdot\mathbb{E}\sin{\phi}=0$. Следовательно, ${\mathbb{E}Z(t)=0}$ для любых ${t\ge0}$, т.е. от времени не зависит. Вычислим теперь корреляционную функцию \elena{(в предположении, что $\mathbb{E} A^2 < \infty$)}
\begin{gather*}
\begin{split}
R_Z(t,s)&=\mathbb{E}Z(t)Z(s)-\underbrace{\mathbb{E}Z(t)\mathbb{E}Z(s)}_{0}=\\
&=\mathbb{E}A^2\cos{(Bt+\phi)}\cos{(Bs+\phi)}=\\
&=\mathbb{E}A^2\cdot\frac{1}{2}\left(\cos{\left(\frac{B(t+s)}{2}+\phi\right)}+\cos{\frac{B(t-s)}{2}}\right)=\\
&=\underbrace{\mathbb{E}A^2\cdot\frac{1}{2}\cos{\frac{B(t-s)}{2}}}_{\text{функция }t-s} + \underbrace{\mathbb{E}A^2\cdot\frac{1}{2}\cos{\left(\frac{B(t+s)}{2}+\phi\right)}}_{0}.
\end{split}
\end{gather*}
Получается, что корреляционная функция $R_Z(t,s)$ зависит от $t$ и $s$ только через их разность. Принимая во внимание постоянность математического ожидания, заключаем, что процесс $Z(t)$ является стационарным в широком смысле \elena{(в предположении, что $\mathbb{E} A^2 < \infty$). При этом, как будет показано далее, без каких-либо дополнительных предположений процесс будет стационарен в узком смысле.}

Теперь исследуем процесс на стационарность в узком смысле. Возьмем произвольные ${n\ge1}$ сечений процесса $Z(t)$ в моменты времени $t_1,\dots,t_n\ge0$ и произвольный сдвиг по времени $h>0$. Пусть $$F(x_1,\dots,x_n;t_1,\dots,t_n)$$ есть функция распределения вектора $(Z(t_1),\dots,Z(t_n))$, тогда стационарность в узком смысле означает, что $$F(x_1,\dots,x_n;t_1,\dots,t_n)=F(x_1,\dots,x_n;t_1+h,\dots,t_n+h)$$ или, в терминах вероятности,
$$\mathbb{P}(Z(t_1)<x_1,\dots,Z(t_n)<x_n)=\mathbb{P}(Z(t_1+h)<x_1,\dots,Z(t_n+h)<x_n).$$
Подставим выражение для $Z(t)$ и получим
\begin{gather}
\begin{split}
    &\mathbb{P}(A\cos{(Bt_1+\phi)}<x_1,\dots,A\cos{(Bt_n+\phi)}<x_n) = \\
    &=\mathbb{P}(A\cos{(Bt_1+Bh+\phi)}<x_1,\dots,A\cos{(Bt_n+Bh+\phi)}<x_n).
\end{split}
\end{gather}
Обусловим полученные выражения по $A$ и $B$
\begin{gather}
\begin{split}
    &\iint\limits_{\mathbb{R}^2}\mathbb{P}(a\cos{(bt_1+\phi)}<x_1,\dots,a\cos{(bt_n+\phi)}<x_n)f(a,b)\,da\,db = \\
    &=\iint\limits_{\mathbb{R}^2}\mathbb{P}(a\cos{(bt_1+bh+\phi)}<x_1,\dots,a\cos{(bt_n+bh+\phi)}<x_n)f(a,b)\,da\,db,
\end{split}
\end{gather}
где интегралы берутся во всем ${a>0}$ и ${b>0}$. Чтобы интегралы были равны, \textit{достаточно}, чтобы подынтегральные функции были равны при любых $a>0$, $b>0$:
\begin{gather}\label{eq:Ex41}
\begin{split}
    &\mathbb{P}(a\cos{(bt_1+\phi)}<x_1,\dots,a\cos{(bt_n+\phi)}<x_n)=\\
    &=\mathbb{P}(a\cos{(bt_1+bh+\phi)}<x_1,\dots,a\cos{(bt_n+bh+\phi)}<x_n).
\end{split}
\end{gather}
В этом выражении $a$ и $b$ -- неслучайные величины, случайной же величиной является лишь ${\phi\in U(0,2\pi)}$. Теперь заметим, что левая часть равна $$\int\limits_{I}f_{\phi}(y)\,dy=\frac{\lambda(I)}{2\pi},$$ где $f_{\phi}(y)$ -- плотность вероятности $\phi$, множество $$I=\{y\in[0,2\pi]:a\cos{(bt_1+y)}<x_1,\dots,a\cos{(bt_n+y)}<x_n\},$$
а $\lambda(\cdot)$ -- мера Лебега\footnote{Про меру Лебега здесь достаточно знать, что $\lambda([p,q])=q-p$ для любого интервала $[p,q]$ и что мера конечного или счетного объединения непересекающихся интервалов равна сумме мер этих интервалов.}. Аналогично, правую часть равенства~\eqref{eq:Ex41} можно выразить через интеграл $$\int\limits_{I'}f_{\phi}(z)\,dz=\frac{\lambda(I')}{2\pi},$$ где множество $$I'=\{z\in[0,2\pi]:a\cos{(bt_1+bh+z)}<x_1,\dots,a\cos{(bt_n+bh+z)}<x_n\}.$$ Теперь заметим, что множества $I$ и $I'$ представляют собой объединение конечного числа непересекающихся интервалов на $[0,2\pi]$. Далее, если $z\in I'$, то $$(z+bh\mod{2\pi})\in I.$$ Это преобразование сдвига сохраняет длины интервалов. Следовательно, меры Лебега множеств $I$ и $I'$ совпадают, т.е. ${\lambda(I)=\lambda(I')}$. Это равносильно выполнению равенства~\eqref{eq:Ex41}, откуда следует равенство интегральных выражений и равенство функций распределения, что означает стационарность процесса в узком смысле. \EndEx

Приведенные примеры показывают, что понятие стационарности в узком смысле сложнее понятия стационарности в широком смысле. В то время как понятие стационарности в широком смысле опирается лишь на две численные характеристики случайных процессов (математическое ожидание и корреляционная функция), которые в общем случае не определяют процесс однозначно, стационарность в узком смысле опирается на все конечномерное распределение процесса. Впрочем, в тех случаях, когда математическое ожидание и корреляционная функция однозначно определяют случайный процесс, может оказаться, что стационарность в широком смысле влечет стационарность в узком смысле. Таким свойством обладают, в частности, гауссовские процессы.

\begin{theorem}
\textit{Гауссовский процесс является стационарным в широком смысле тогда и только тогда, когда он является стационарным в узком смысле.}
\end{theorem}

\begin{proof}
Рассмотрим произвольный стационарный в широком смысле гауссовский процесс $\{X(t), t \ge 0\}$. Составим вектор из сечений процесса $X=(X(t_1), \dots, X(t_n))$. Тогда, по определению гауссовского процесса, характеристическая функция этого вектора имеет вид 
$$\varphi_{X}(\shmaxg{s}) = \exp \left( i \shmaxg{s}^{\top} m - \frac{1}{2} \shmaxg{s}^{\top} R \shmaxg{s} \right),$$ 
где \shmaxg{$s=(s_1,\dots,s_n)$}, $m=(m_1,\dots,m_n)$, $m_i=\Exp X(t_i)$, $R=\left\| R_{ij} \right\|_{i,j=1}^n$, $R_{ij}=\Exp X(t_i) X(t_j) - \Exp X(t_i) \Exp X(t_j)$.
Так как процесс стационарный в широком смысле, то при эквидистантном изменении моментов времени вектор $m$ и матрица $R$ не изменятся. Значит, не изменится характеристическая функция \shmaxg{$\varphi_X(s)$}, т.е. не изменится распределение\linebreak вектора~$X$, который был выбран произвольно. По определению,\linebreak $X(t)$ -- стационарный в узком смысле процесс. \EndProof
\end{proof}

\begin{example}
Исследовать на стационарность в широком и узком смыслах процесс ${X(t)=W(t+a)-W(t)}$, где $W(t)$ -- винеровский процесс.
\end{example}

\textbf{Решение}. В первую очередь замечаем, что математическое ожидание процесса $\mathbb{E}X(t)=0$ для любого $t$ и, следовательно, не зависит от времени. Корреляционная функция
\begin{gather*}
\begin{split}
    &R_X(t,s)=\mathbb{E}(W(t+a)-W(t))(W(s+a)-W(s))= \\ &=\min(t+a,s+a) - \min(t+a,s) - \min(t,s+a) + \min(t,s).
\end{split}
\end{gather*}
Хорошо видно, что если добавить к $t$ и $s$ любую величину $h$, то каждое слагаемое увеличится на $h$, а значение корреляционной функции не изменится. Поэтому процесс $X(t)$ является стационарным в широком смысле.

Теперь остается заметить, что процесс $X(t)$ является гауссовским. Действительно, составив из его произвольных сечений случайный вектор, мы получим с точностью до линейного преобразования вектор из сечений винеровского процесса, т.е. гауссовский вектор. А так как процесс $X(t)$ гауссовский и стационарный в широком смысле, то он является стационарным и в узком смысле. \EndEx

Конечно, существуют процессы, которые не являются стационарными ни в каком смысле. Например, винеровский процесс не является стационарным в широком смысле (и, следовательно, стационарным в узком смысле), так как его дисперсия ${\mathbb{D}_W(t)=t}$ зависит от времени. Пуассоновский процесс также не является стационарным ни в каком смысле, так как его математическое ожидание ${m_K(t)=\lambda t}$ зависит от времени.

\subsection{Корреляционная функция}

В данном разделе мы приведем некоторые простейшие сведения о корреляционных функциях стационарных процессов. Они послужат вспомогательным материалом для понимания основных результатов теории стационарных процессов, изложенных в следующих двух разделах. Кроме того, в этом разделе мы отвечаем на вопрос, в каких случаях произвольная функция является корреляционной функцией некоторой стационарного процесса. Это первый шаг к пониманию природы стационарных процессов.

Для стационарного в широком смысле процесса $X(t)$ удобно ввести функцию $R_X(t)$ одной переменной: $${R_X(t)=R_X(t,0)}=\mathbb{E}X(t)\overline{X(0)} - \mathbb{E}X(t)\mathbb{E}\overline{X(0)}=\mathbb{E}\accentset{\circ}X(t)\overline{\accentset{\circ}X(0)}.$$ Как и $R_X(t,0)$, будем называть ее \textit{корреляционной функцией стационарного в широком смысле процесса}. Данная функция обладает рядом легко проверяемых свойств.

а) Во-первых, заметим, что корреляционная функция связана с дисперсией процесса по формуле ${R_X(0)=\mathbb{D}X(0)=\mathbb{D}X(t)}$ для любого $t$, так как дисперсия стационарного процесса не зависит от времени. Отсюда следует, в частности, что в нуле значение этой функции вещественное и неотрицательное: ${R_X(0)\ge0}$. Если ${R_X(0)=0}$, то тогда ${\mathbb{D}X(t)=0}$ для каждого $t$, что значит ${X(t)=f(t)}$, где $f(t)$ -- произвольная неслучайная функция.

б) Далее, так как в силу стационарности ${R_X(t,s)=R_X(t+h,s+h)}$ для любого $h\in\mathbb{R}$, то $R_X(t,s)=R_X(0,s-t)$. Отсюда следует, что
$$R_X(-t)=R_X(-t,0)=R_X(0,t)=\mathbb{E}\accentset{\circ}X(0)\overline{\accentset{\circ}X(t)}=\overline{R_X(t,0)}=\overline{R_X(t)},$$
т.е. ${R_X(-t)=\overline{R_X(t)}}$ для любого ${t\in\mathbb{R}}$. Функции с таким свойством называются \textit{эрмитовыми}.

в) В силу неравенства Коши--Буняковского
$$|R_X(t)|=|\mathbb{E}\accentset{\circ}X(t)\overline{\accentset{\circ}X(0)}|\le\sqrt{\mathbb{D}X(t)\mathbb{D}X(0)}=\mathbb{D}X(0)=R_X(0),$$ т.е. абсолютное значение функции $R_X(t)$ ограничено сверху значением этой функции в нуле: $|R_X(t)|\le R_X(0)$.

г) Пусть дан вектор $(X(t_1),\dots,X(t_n))$ из сечений некоторого стационарного процесса $X(t)$. Ковариационная матрица этого вектора состоит из элементов $$R_{ij}=\mathbb{E}\accentset{\circ}X(t_i)\overline{\accentset{\circ}X(t_j)}=R_X(t_i,t_j)=R_X(t_j-t_i),$$ где $i,j$ пробегают значения от 1 до $n$. Так как любая ковариационная матрица является неотрицательно определенной в том смысле, что $$\sum\limits_{i=1}^n\sum\limits_{j=1}^n R_{ij} z_i \overline z_j \ge 0 \ \forall z_i \in\mathbb{C},$$ то мы получаем, что и матрица $\left\|R_X(t_j-t_i)\right\|$ тоже является неотрицательно определенной, причем для любого набора $t_1,\dots,t_n$ (необязательно упорядоченного и необязательно состоящего из различных элементов). В связи с этим удобно даже ввести определение неотрицательно определенной функции одной или двух переменных.

\begin{definition}
Функция $f(t)\in\mathbb{C}$ одной переменной называется \textit{неотрицательно определенной}, если для любых комплексных чисел $z_1$, $\dots$, $z_n\in\mathbb{C}$ и моментов времени $t_1$, $\dots$, $t_n$, $n\ge1$: $$\sum\limits_{i=1}^n\sum\limits_{j=1}^n z_i \bar{z}_j f(t_i-t_j)\ge0.$$
\end{definition}

\begin{definition}
Функция $f(t_1,t_2)\in\mathbb{C}$ двух переменных называется \textit{неотрицательно определенной}, если для любых комплексных чисел $z_1$, $\dots$, $z_n\in\mathbb{C}$ и моментов времени $t_1$, $\dots$, $t_n$, $n\ge1$: $$\sum\limits_{i=1}^n\sum\limits_{j=1}^n z_i \bar{z}_j f(t_i,t_j)\ge0.$$
\end{definition}

Получается, что корреляционная функция $R_X(t,s)$ любого процесса (необязательно стационарного) является неотрицательно определенной функцией двух переменных. А корреляционная функция $R_X(t)$ стационарного процесса является неотрицательно определенной функцией одной переменной. Кстати говоря, это можно показать и прямо:
$$\sum\limits_{i=1}^n\sum\limits_{j=1}^n z_i \bar{z}_j R_X(t_i-t_j)=\mathbb{E}\left| \sum\limits_{i=1}^n (X(t_i)-m_X)z_i \right|^2\ge0, \ \forall{z}_i\in\mathbb{C}.$$


Оказывается, верно и обратное (см., например, ~\cite[гл.~2, §~6, теорема~4, с.~53]{BulinskyShiryaev2005}).

\begin{theorem}
\textit{Класс неотрицательно определенных комп\-лек\-с\-но\-знач\-ных функций $\{R(t,s), {s,\,t\in T}\}$ совпадает с классом корреляционных функций процессов второго порядка $\{X(t), {t\in T}\}$ и, более того, совпадает с классом корреляционных функций комплекснозначных гауссовских процессов} $\{X(t)$, $t\in T\}$.
\end{theorem}

\textbf{Следствие}. Отсюда следует, что класс неотрицательно определенных комплекснозначных функций $\{R(t)$, ${t\in T}\}$ совпадает с классом корреляционных функций стационарных в широком смысле процессов второго порядка $\{X(t)$, ${t\in T}\}$ и, более того, совпадает с классом корреляционных функций комплекснозначных стационарных в широком смысле гауссовских процессов $\{X(t)$, $t\in T\}$.

Другими словами, функция является корреляционной для некоторого процесса тогда и только тогда, когда она является неотрицательно определенной. Неотрицательно определенная функция неоднозначно определяет процесс, для которого она является корреляционной функцией. Но, по крайней мере, всегда можно ей сопоставить подходящий гауссовский процесс.


д) Ранее в разделе~\ref{correlation} мы установили, что непрерывность ковариационной функции в точках вида $(t,t)$ влечет ее непрерывность во всех точках множества $[0,+\infty)^2$ (или множества $\mathbb{R}^2$, если предполагать ${T=\mathbb{R}}$). Стационарный в широком смысле процесс имеет постоянное (а значит, непрерывное) математическое ожидание. Следовательно, непрерывность \textit{корреляционной} функции в точках вида $(t,t)$ влечет ее непрерывность в любой точке вида $(t_1,t_2)$. Это значит, что из непрерывности функции $R_X(t)$ в точке $t=0$ следует ее непрерывность в любой другой точке $t\in\mathbb{R}$.

\begin{example}
Существует ли стационарный в широком смысле процесс $X(t)$ с корреляционной функцией
\begin{equation}\label{eq:RXT1}
R_X(t)=\left\{ {\begin{array}{*{20}{c}}
	{\sigma^2},&{|t|\le t_0}, \\
	{0,}&{|t|>t_0}; 
	\end{array}} \right.
\end{equation}
с такой корреляционной функцией:
\begin{equation}\label{eq:RXT2}
R_X(t)=\left\{ {\begin{array}{*{20}{c}}
	{\sigma^2},&{|t|\le t_0,\,t\ne0}, \\
	{0,}&{|t|>t_0,\,t=0}; 
	\end{array}} \right.
\end{equation}
с такой корреляционной функцией:
\begin{equation}\label{eq:RXT3}
R_X(t)=\left\{ {\begin{array}{*{20}{l}}
	{\sigma^2}&{|t|\le t_0}, \\
	{f(t)}&{t_0<|t| \le t_0+\Delta}, \\
	{0}&{|t|>t_0+\Delta},
	\end{array}} \right.
\end{equation}
где $f(t)$ -- функция с $f(t_0)=\sigma^2$ и $f(t)=0$ для $|t|\ge t_0+\Delta$?
\end{example}

\begin{solution}
Функция $R_X(t)$ из~\eqref{eq:RXT1} непрерывна в точке ${t=0}$, но не является непрерывной в точках ${t=\pm t_0}$. Поэтому стационарного процесса с такой корреляционной функцией не существует.

Функция $R_X(t)$ из~\eqref{eq:RXT2} имеет разрыв в нуле, поэтому, в принципе, может иметь разрывы и в других точках. Впрочем, эта функция все равно не является корреляционной ни для какого стационарного процесса, так как не выполнено неравенство ${|R_X(t)|\le R_X(0)}$ для ${t\in(0,t_0)}$. Противоречие можно также обнаружить, показав, что предложенная функция $R_X(t)$ не является неотрицательно определенной. Действительно, пусть $t_1=t_0/2$, $t_2=-t_0/2$. Тогда определитель
\setlength{\arraycolsep}{2pt}$$\det\left[ {\begin{array}{*{20}{c}}
	{R_X(t_1-t_1)}&{R_X(t_1-t_2)}\\
	{R_X(t_2-t_1)}&{R_X(t_2-t_2)}
	\end{array}} \right] = \det\left[ {\begin{array}{*{20}{c}}
	{0}&{R_X(t_0)}\\
	{R_X(-t_0)}&{0}
	\end{array}} \right]=-\sigma^2<0,$$ и, следовательно, функция $R_X(t)$ не является неотрицательно определенной.
\setlength{\arraycolsep}{5pt}

Перейдем теперь к функции $R_X(t)$ из~\eqref{eq:RXT3}. Предположим, что существует стационарный процесс $X(t)$, для которого эта функция  корреляционная. В этом случае функция $R_X(t)$ является неотрицательно определенной. Пусть $t_1=0$, $t_2=t_0$, $t_3=2t_0$. Тогда ковариационная матрица $\left\|R_X(t_i-t_j)\right\|_{i,j=1}^{3}$, равная $$\left[ {\begin{array}{*{20}{c}}
	{{\sigma ^2}}&{{\sigma ^2}}&{R_X\left( {2{t_0}} \right)}\\
	{{\sigma ^2}}&{{\sigma ^2}}&{{\sigma ^2}}\\
	{R_X\left( {2{t_0}} \right)}&{{\sigma ^2}}&{{\sigma ^2}}
	\end{array}} \right],$$ будет неотрицательно определенной. Здесь мы учли, что $R(t)$ является вещественной функцией, поэтому всюду $R_X(-t)=R_X(t)$. Определитель этой матрицы равен $$-\sigma^2(\sigma^4-\sigma^2R_X(2t_0))+R_X(2t_0)(\sigma^4-\sigma^2R_X(2t_0))\ge0.$$ Разделим это выражение на $-\sigma^2$, тогда $$R_X^2(2t_0)-2R_X(2t_0)\sigma^2+\sigma^4 \le 0 \Leftrightarrow (R_X(2t_0)-\sigma^2)^2 \le 0,$$ 
что возможно, только лишь если ${R_X(2t_0)=\sigma^2}$. Теперь пусть ${t_1=0}$, ${t_2=t_0}$, ${t_3=3t_0}$. Тогда ковариационная матрица $\left\|R_X(t_i-t_j)\right\|_{i,j=1}^{3}$, равная $$\left[ {\begin{array}{*{20}{c}}
	{{\sigma ^2}}&{{\sigma ^2}}&{R_X\left( {3{t_0}} \right)}\\
	{{\sigma ^2}}&{{\sigma ^2}}&{{\sigma ^2}}\\
	{R_X\left( {3{t_0}} \right)}&{{\sigma ^2}}&{{\sigma ^2}}
	\end{array}} \right],$$ является неотрицательно определенной. Отсюда следует, что $R_X(3t_0)\!\!=$\linebreak $=\sigma^2$. И вообще, необходимо, чтобы ${R_X(nt_0)=\sigma^2}$ для любого $n\in\mathbb{Z}$. Поэтому найдется такое $n_0$, что ${n_0 t_0 > t_0+\Delta}$, но тогда ${R_X(n_0t_0)=\sigma^2}$, а~нам дана функция с ${R_X(n_0t_0)=0}$. Мы получили противоречие. Это значит, что какую бы область сглаживания $(-t_0-\Delta,-t_0)\cup(t_0,t_0+\Delta)$ мы ни выбрали и какую бы функцию сглаживания $f(t)$ на ней мы ни выбрали, функция $R_X(t)$ не может быть корреляционной функцией ни для какого стационарного процесса $X(t)$. \EndEx
\end{solution}

Перейдем теперь к более содержательным результатам корреляционной теории стационарных в широком смысле процессов.

\subsection{Спектральное представление процессов}

Теперь, когда даны начальные представления о стационарных процессах, самое время поговорить подробно о природе таких процессов. Выше мы видели, что суммы гармоник (комплексных экспонент) образуют стационарный процесс тогда и только тогда, когда коэффициенты при этих гармониках являются попарно некоррелированными. В~этом разделе мы увидим сильное обобщение этого наблюдения.

Начнем с содержательного примера. Рассмотрим случайный процесс $$X(t)=\xi e^{i\Omega t},$$
где $\xi$ и $\Omega$ -- независимые случайные величины, ${\mathbb{E}\xi=0}$, ${\mathbb|\xi|^2<\infty}$, $\xi$ -- в общем случае комплекснозначная случайная величина, а $\Omega$ -- вещественнозначная случайная величина. Математическое ожидание этого процесса $$\mathbb{E}X(t)=\mathbb{E}\xi e^{i\Omega t}=\mathbb{E}\xi\cdot\mathbb{E}e^{i\Omega t} = 0$$ в силу независимости случайных величин $\xi$ и $\Omega$ (а стало быть, и произвольных функций от них) и условия $\mathbb{E}\xi=0$. Корреляционная функция
$$R_X(t,s)=\mathbb{E}X(t)\mathbb{E}\overline{X(s)}=\mathbb{E}|\xi|^2 e^{i\Omega(t-s)}$$ зависит лишь от разности ${t-s}$. Отсюда следует, что процесс $X(t)$ является стационарным в широком смысле.

Перепишем корреляционную функцию в другом виде. Введем функцию 
$R\ag{_X}(\tau)=R_X(t,s)$, $\tau=t-s$, и запишем $$R(\tau)=\mathbb{E}|\xi|^2\cdot\mathbb{E}e^{i\Omega \tau} = \mathbb{E}|\xi|^2\int\limits_{-\infty}^{+\infty}e^{i\nu \tau}\,dF_{\Omega}(\nu), \ F_{\Omega}(\nu)=\mathbb{P}(\Omega < \nu),$$ что равносильно
$$R\ag{_X}(\tau) = \int\limits_{-\infty}^{+\infty}e^{i\nu \tau}\,d(\mathbb{E}|\xi|^2 F_{\Omega}(\nu)).$$ Если ввести еще обозначение ${S(\nu)=\mathbb{E}|\xi|^2 F_{\Omega}(\nu)}$, то мы приходим к выражению
\begin{equation}\label{eq:CorrFuncRep1}
   R\ag{_X}(\tau)=\int\limits_{-\infty}^{+\infty}e^{i\nu \tau}\,dS(\nu), 
\end{equation}
где $S(\nu)$ -- это функция, которая с точностью до произвольного неотрицательного множителя совпадает с функцией распределения случайной величины $\Omega$.

Кстати говоря, мы получили, что если функция одной переменной имеет вид~\eqref{eq:CorrFuncRep1} с функцией $S(\nu)$, которая с точностью до неотрицательного множителя совпадает с функцией распределения некоторой случайной величины, то эта функция является корреляционной функцией \shmaxg{некоторого} стационарного с.к.-непрерывного процесса $X(t)=\xi e^{i\Omega t}$, для которого $S(\nu)=\mathbb{E}|\xi|^2 F_{\Omega}(\nu)$.

Оказывается, справедливо и обратное (см. \cite[гл. 7, теорема 8]{BulinskyShiryaev2005}).

\begin{theorem}[ (Бохнер--Хинчин)]
\label{BochnerHinchin}
Для того чтобы функция $R(\tau)$ являлась корреляционной функцией некоторого стационарного в широком смысле непрерывного в среднеквадратическом смысле случайного процесса, необходимо и достаточно, чтобы она была представима в следующем виде:
\begin{equation}\label{eq:BohnerKhinchinFormula}
    R\ag{_X}(\tau) = \int\limits_{-\infty}^{+\infty}e^{i\nu \tau}dS(\nu),
\end{equation}
где $S(\nu)=\mathbb{E}|\xi|^2\cdot\mathbb{P}(\Omega < \nu)$ для некоторых случайных величин $\xi$ и $\Omega$.
\end{theorem}
Функция $S(\nu)$ из теоремы Бохнера--Хинчина называется \textit{спектра\-льной функцией}. Это название оправдано тем, что она представляет собой с точностью до множителя функцию распределения \textit{частоты} гармоники $e^{i\Omega t}$.

Обращаем внимание читателя на то, что в теореме Бохнера--Хинчи\-на речь идет именно об с.к.-непрерывных процессах. Функция вида \eqref{eq:BohnerKhinchinFormula} нигде не может иметь разрывы. Естественно, существуют стационарные процессы с разрывной корреляционной функцией, например фун\-кция $R\ag{_X}(t)$, для которой $R\ag{_X}(0)=1$, $R\ag{_X}(t)=0$ для остальных $t$. Эта функция является неотрицательно определенной, поэтому как минимум существует гауссовский стационарный процесс с такой корреляционной функцией\ag{.} 

Бывает, что случайная величина $\Omega$ из теоремы Бохнера--Хинчина имеет плотность вероятности распределения. В таких случаях $S(\nu)$ тоже имеет плотность.
\begin{definition}
Если функция $S(\nu)$ абсолютно непрерывна, т.е. если существует неотрицательная функция $\rho(\nu)$ такая, что для всех $\nu\in\mathbb{R}$: $$S(\nu)=\int\limits_{-\infty}^{\nu}\rho(\tau)\,d\tau,$$ то функция $\rho(\nu)$ называется \textit{спектральной плотностью}. 
\end{definition}

В этом случае корреляционная функция процесса представима фо\-рмулой 
\begin{equation}\label{eq:Rtrho}
R_X(t)=\int\limits_{-\infty}^{+\infty}e^{i\nu t}\rho(\nu)\,d\nu.
\end{equation}
Отметим здесь же, что дисперсия процесса в этом случае
\begin{equation*}
\mathbb{D}X(t)=R_X(0)=\int\limits_{-\infty}^{+\infty}\rho(\nu)\,d\nu.
\end{equation*}

Напомним также из теории интегралов Фурье достаточное условие существования спектральной плотности.
\begin{theorem}
\textit{Если корреляционная функция $R_X(t)$ с.к.-непрерыв\-ного процесса абсолютно интегрируема, т.е.} $$\int\limits_{-\infty}^{+\infty}|R_X(t)|\,dt<\infty,$$ \textit{то спектральная функция этого процесса обладает непрерывной и ог\-раниченной спектральной плотностью $\rho(\nu)$, причем $$\rho(\nu)=\frac{1}{2\pi}\int\limits_{-\infty}^{+\infty}e^{-i\nu t}R_X(t)\,dt.$$ 
Для вещественной функции $R_X(t)$ выражение выше упрощается}: $$\rho(\nu)=\frac{1}{\pi}\int\limits_{0}^{+\infty}\cos{\nu t}\,R_X(t)\,dt$$ и $\rho(\nu)$ \textit{является четной функцией} $\nu$.
\end{theorem}

\begin{example}
Дана спектральная плотность процесса $X(t)$: $$\rho(\nu)=\left\{ {\begin{array}{*{20}{c}}
	{\dfrac{\sigma^2}{2\pi},}&{|\nu|\le\nu_0}, \\ \\
	{0,}&{|\nu|>\nu_0}.
	\end{array}} \right.$$
Найти корреляционную функцию $R_X(t)$ и дисперсию $\mathbb{D}X(t)$ процесса.
\end{example}

\begin{solution}
Из формулы~\eqref{eq:Rtrho} следует $$R_X(t)=\int\limits_{-\infty}^{+\infty}e^{i\nu t}\rho(\nu)\,d\nu=\int\limits_{-\nu_0}^{\nu_0}e^{i\nu t}\,\rho(\nu)\,d\nu=\frac{\sigma^2}{2\pi}\int\limits_{-\nu_0}^{\nu_0}\cos{\nu t}\,d\nu.$$ Отсюда получаем $$R_X(t)=\frac{\sigma^2}{\pi t}\sin{\nu_0 t} \text{ при} \ t \ne 0, \ R_X(0)=\frac{\sigma^2\nu_0}{\pi}.$$ Дисперсия $\mathbb{D}X(t)=R_X(0)=\sigma^2\nu_0/\pi$. \EndEx
\end{solution}

\begin{example}
Корреляционная функция некоторого стационарного процесса $X(t)$ имеет вид
$$R_X(t)=De^{-a|t|}.$$
Найти спектральную плотность процесса.
\end{example}

\textbf{Решение}. Существование спектральной функции следует из того, что корреляционная функция этого процесса является абсолютно интегрируемой на $\mathbb{R}$. Кроме того, спектральную плотность можно вычислить по формуле
\begin{eqnarray*}
\rho(\nu) &=& \frac{1}{2\pi}\int\limits_{-\infty}^{+\infty}e^{-i\nu t}R_X(t)\,dt = \frac{1}{2\pi} \int\limits_{-\infty}^{+\infty}De^{-i\nu t - a|t|}\,dt= \\
&=& \frac{D}{2\pi}\int\limits_{0}^{+\infty}(e^{-i\nu t-at} + e^{i\nu t -at})\,dt = \frac{D}{2\pi} \left(\frac{1}{i\nu + a} + \frac{1}{-i\nu + a}\right)= \\ 
&=& \frac{Da}{\pi}\frac{1}{a^2+\nu^2}. \text{\EndEx}
\end{eqnarray*}

\begin{example}
\label{white noise}
Найти спектральную функцию случайного процесса с корреляционной функцией, заданной \textit{формально} как $$R_X(t)=\sigma^2 \delta(t),$$ где $\delta(t)$ -- дельта-функция.
\end{example}

\textbf{Решение}. Абстрагируясь от математической строгости, относящейся 
к теории выше, мы можем записать:
$$\rho_X(\nu)=\frac{1}{2\pi}\int\limits_{-\infty}^{+\infty}e^{-i\nu t}R_X(t)\,dt=\frac{1}{2\pi}\int\limits_{-\infty}^{+\infty}e^{-i\nu t}\sigma^2\delta(t)\,dt=\frac{\sigma^2}{2\pi}. \ \text{\EndEx}$$

Хотя, строго говоря, процессов с такой корреляционной функцией не существует, тем не менее эта абстракция оказывается чрезвычайно полезна в приложениях. Процесс с такой формальной корреляционной функцией называется \textit{белым шумом}. Его особенность в том, что любые два сколь угодно близких сечения процесса являются некоррелированными. Конечно, существует математически строгое описание процессов, у которых корреляционная функция представляет собой какие-то комбинации обобщенных функций. Такие процессы называются \textit{обобщенными процессами}, к ним принадлежит и белый шум. За подробной справкой об определении и свойствах таких процессов мы отправляем читателя к монографии~\cite[гл. 17]{KoralovSinai}.

\elena{Чтобы избежать описанных выше трудностей, можно рассматривать корреляционную функцию из примера \ref{white noise} в предположении, что время дискретно. В этом случае не нужно требовать непрерывности корреляционной функции, так как теперь $R\ag{_X}(n)$, $n\in\mathbb Z$ -- последовательность. Аналогом теоремы Бохнера--Хинчина \ref{BochnerHinchin} для случайных процессов в дискретном времени (т.е. случайных последовательностей) является теорема Герглотца (см. гл. 7, теорема 4 \cite{BulinskyShiryaev2005}).
\begin{theorem}
\label{Gerglotz}
\textit{Для того, чтобы функция $R\ag{_X}(n)$, $\in\mathbb Z$, являлась корреляционной функции некоторой стационарной в широком смысле последовательности, необходимо и достаточно, чтобы было справедливо представление}
\begin{equation}
    R\ag{_X}(n) = \int_{-\pi}^\pi e^{in \nu} dS(\nu),
\end{equation}
\textit{где $S$ -- конечная мера на борелевской сигма-алебре отрезка $[-\pi, \pi]$.}
\end{theorem}
Теперь рассмотрим пример, аналогичный примеру \ref{white noise}, только в дискретном времени.
Пусть корреляционная функция имеет вид:
$$
R\ag{_X}(n) = \left\{ \begin{array}{*{20}{c}}
	{\sigma^2,}&{n=0}, \\ \\
	{0,}&{n\neq 0}.
	\end{array} \right.
$$
Как несложно проверить, последовательность некоррелированных случайных величин $X_1, X_2,\ldots$ с общей дисперсией $\mathbb D X = \sigma^2$ имеет такую корреляционную функцию. Это и есть модель \textit{белого шума} в дискретном времени. Спектральная плотность для дискретного белого шума равна $\rho(\nu) = 1/(2\pi)$, $\nu\in[-\pi, \pi]$. \EndEx
}

Мы видели выше, что комплексные гармоники с нулевым математическим ожиданием являются стационарными в широком смысле процессами. Линейные комбинации таких некоррелированных гармоник также являются стационарными процессами. Кроме того, мы видели, что корреляционная функция стационарного процесса представляет собой интеграл (<<непрерывную сумму>>) гармоник. Возникает вопрос: можно ли и сам стационарный процесс представить в виде конечной или бесконечной суммы гармоник с некоррелированными коэффициентами? Оказывается, это всегда возможно.

Чтобы это продемонстрировать, вернемся к примеру со стационарным процессом ${X(t)=\xi e^{i\Omega t}}$, где $\xi$ -- комплекснозначная случайная величина второго порядка с нулевым математическим ожиданием и $\Omega$ -- вещественнозначная случайная величина с произвольной функцией распределения $F_{\Omega}(\nu)$ и не зависящая от $\xi$. Представим этот процесс в виде интеграла от комплексной экспоненты.

Для простоты рассмотрим случай, когда ${\xi=1}$ п.н. и ${\Omega\in\mathrm{R}(0,1)}$. Введем процесс ${V(\nu)=\mathsf{I}(\Omega < \nu)}$, где $\nu$ выполняет роль времени для случайного процесса $V(\nu)$. В таком случае $$X(t)= e^{i\Omega t}=\int\limits_{-\infty}^{+\infty} e^{i\nu t}\,dV(\nu),$$ где интеграл понимается в смысле определения~\ref{def:ItoIntegral}. Чтобы это показать, возьмем произвольное разбиение ${0=\nu_1<\dots<\nu_n=1}$ произвольного отрезка $[0,1]$, выберем произвольные точки $\nu'_k\in[\nu_k,\nu_{k+1}]$ и запишем интегральную сумму: 
$$S_n = \sum\limits_{k=1}^{n-1} e^{i\nu'_k t}(V(\nu_{k+1})-V(\nu_{k}))=\sum\limits_{k=1}^{n-1} e^{i\nu'_k t}\cdot \mathsf{I}(\Omega\in[\nu_{k},\nu_{k+1}]).$$ Очевидно, что для каждого исхода $\omega$ в сумме выше будет только одно ненулевое слагаемое и оно будет зависеть от исхода: $e^{i\nu'_k(\omega)t}$. Кроме того, при измельчении разбиения, ${\nu'_k(\omega)\to\Omega(\omega)}$. Получается, что имеется сходимость п.н. ${S_n\to e^{i\Omega t}}$, но так как последовательность ${e^{i\nu'_k(\omega)t}}$ ограничена по модулю единицей, то это будет также сходимость и в среднем квадратичном. С другой стороны, по определению есть сходимость в среднем квадратичном ${S_n\to\int_0^1 e^{i\nu t}\,dV(\nu)}$. Значит, п.н. $$\int\limits_0^1 e^{i\nu t}\,dV(\nu) = e^{i\Omega t}.$$ Теперь в качестве отрезка интегрирования взять ${[a,b]}$, ${a\le0}$, ${b\ge1}$, и устремить ${a\to-\infty}$, ${b\to+\infty}$ в среднем квадратичном. Значение интеграла при этом меняться не будет, так как при таких $\nu$ не меняется $V(\nu)$. Следовательно, мы доказали, что п.н.
$$\int\limits_{-\infty}^{+\infty} e^{i\nu t}\,dV(\nu) = e^{i\Omega t}.$$

Введем теперь одно вспомогательное понятие и сформулируем главную теорему о спектральном представлении стационарного процесса.
\begin{definition}
\shmaxg{Центрированный} случайный процесс ${V(t) \in CL_2}$ будем называть \textit{процессом с ортогональными приращениями}, если \shmaxg{для любых $t_1 < t_2 \le t_3 < t_4$} $$\Exp (V(t_4) - V(t_3))\overline{(V(t_2) - V(t_1))} = 0.$$
\end{definition}

\begin{theorem}[ (Крамер)]\label{th:WeakStatRepresent}
\textit{Любому стационарному в широком смысле с.к.-непрерывному случайному процессу ${X(t)\in CL_2}$ соответствует случайный процесс $V(\nu)$, заданный на том же вероятностном пространстве, что и процесс $X(t)$, с ортогональными приращениями, такой, что с вероятностью единица}
\begin{equation*}
    X(t) = m_X + \int\limits_{-\infty}^{+\infty}e^{i\nu t}\,dV(\nu), \ t\in\mathbb{R},
\end{equation*}
\textit{где интеграл понимается как несобственный интеграл в смысле Рима\-на--Стилтьеса, а $m_X$ -- математическое ожидание процесса $X(t)$.}
\end{theorem}
С доказательством этой теоремы можно познакомиться, например, в \cite[гл. 7, теорема 9]{BulinskyShiryaev2005}.

Отметим, что ортогональность приращений процесса $V(\nu)$ -- это аналог ортогональности коэффициентов при комплексных экспонентах в теореме~\ref{th:ComplexHarmonisStationary}.

\elena{Аналогичная теорема имеет место и для стационарных последовательностей (см.  \cite[гл. 7, теорема 5]{BulinskyShiryaev2005}). 
\begin{theorem}
\textit{Любой стационарной в широком смысле  случайной последовательности $X_n\in CL_2$ соответствует случайный процесс $V(\nu)$, $\nu\in[-\pi,\pi]$, заданный на том же вероятностном пространстве, что и последовательность $X_n$, с ортогональными приращениями, такой, что с вероятностью единица}
\begin{equation*}
    X_n = m_X + \int\limits_{-\pi}^{+\pi}e^{i\nu n}\,dV(\nu), \ n\in\mathbb{Z}.
\end{equation*}
\end{theorem}
}

Ранее нами была выведена формула корреляционной функции $$R_X(t)=\sum\limits_{k=1}^n \mathbb{E}|\xi_k|^2 e^{i\omega_k t}$$ стационарного случайного процесса $$X(t)=\sum\limits_{k=1}^n \xi_k e^{i\omega_k t}.$$ 
Выведем теперь связь между корреляционной функцией $R_X(t)$ и процессом $V_X(\nu)$ для произвольного стационарного процесса $X(t)$. По оп\-ределению
\begin{equation*}
    \accentset{\circ}{X}(t) = \int\limits_{-\infty}^{+\infty}e^{i\nu t}dV(\nu) = \underset{\substack{N\to\infty \\ h \to +0}}{\liminmean}\sum\limits_{k=1}^{2N/h}e^{i\nu_kt}(V(\nu_k) - V(\nu_{k-1})),
\end{equation*}
где рассматривается разбиение отрезка $[-N,N]$ на отрезки длины $h$ (не умаляя общности, считаем, что $2N/h$~-- целое число), $\nu_k$~-- точка из $k$-го отрезка разбиения. Тогда корреляционную функцию случайного процесса $X(t)$ можно записать в виде 
\elena{
\begin{equation*}
\begin{split}
    &R_X(\tau) = \Exp \accentset{\circ}{X}(t+\tau) \overline{\accentset{\circ}{X}(t)} = \\
    &=\!\Exp\left( \underset{\substack{N\to\infty \\ h \to +0}}{\liminmean}\sum\limits_{k=1}^{2N/h}e^{i\nu_k (t+\tau)}(V(\nu_k) - V(\nu_{k-1})) \sum\limits_{j=1}^{2N/h} e^{-i\nu_jt} \overline{(V(\nu_j) - V(\nu_{j-1}))}\right)\!\!\!=\\
    &=\underset{\substack{N\to\infty \\ h \to +0}}{\lim}\sum\limits_{k=1}^{2N/h}\sum\limits_{j=1}^{2N/h}e^{i(\nu_k (t+\tau) - \nu_jt)}\underbrace{\Exp (V(\nu_k) - V(\nu_{k-1}))\overline{(V(\nu_{j}) - V(\nu_{j-1}))}}_{=0,\; \text{если}\; k\neq j} =\\
    &=\underset{\substack{N\to\infty \\ h \to +0}}{\lim}\sum\limits_{k=1}^{2N/h}e^{i\nu_k \tau }\Exp |V(\nu_k) - V(\nu_{k-1})|^2 .
\end{split}
\end{equation*}}
Полученное выражение представляет собой интеграл по мере $$dS(\nu)=\mathbb{E}|dV(\nu)|^2$$ и может быть использовано в качестве определения такого интеграла.

\elena{
\subsection{Закон больших чисел для вещественнозначных стационарных в широком смысле процессов в  дискретном и непрерывном времени.}
Закон больших чисел для стационарных в широком смысле последовательностей имеет вид
\begin{equation}
\label{LBN}
    \frac{1}{N}\sum_{k=1}^N X_k \xrightarrow[N\to\infty]{L_2}m_X + V(0).
\end{equation}
Отметим, что в отличие от обычного закона больших чисел для независимых с.в. в общем случае может быть еще дополнительная случайная компонента $V(0)$. \ag{Формула}
\eqref{LBN} следует из спектрального разложения:
$$
\frac{1}{N}\sum_{k=1}^N X_k = \frac{1}{N}\sum_{k=1}^N \int_{-\pi}^\pi e^{i\nu k} dV(\nu) = \int_{-\pi}^\pi \Psi_N(\nu) dV(\nu),
$$
где  функция $\Psi\ag{_N}(\nu) = \frac{1}{N}\frac{e^{i\nu } - e^{i\nu (N+1)}}{1-e^{i\nu }} $ такая что $|\Psi\ag{_N}(\nu) | \le 1$, $\ag{\Psi}_N(0) = 1$, но  если $\nu\neq 0$, то $|\Psi\ag{_N}(\nu) | \to 0$ \ag{при $N \to \infty$}.
Несложно показать, используя спектральное представление для корреляционной функции, что аналогично \eqref{LBN} имеет место
\begin{equation}
    \frac{1}{N}\sum_{k=1}^N R\ag{_X}(k) \xrightarrow[N\to\infty]{} S(\{0\}).
\end{equation}
Можно показать, что для стационарной в широком смысле последовательности $X_1, X_2, \ldots $ следующие условия эквивалентны:
\begin{itemize}
    \item  последовательность $X_1, X_2, \ldots $ является эргодичной по математическому ожиданию (см. далее определение \ref{def:ErgodicDefin}), то есть $$ \frac{1}{N}\sum_{k=1}^N X_k \xrightarrow[N\to\infty]{L_2}m_X;$$
    \item соответствующая спектральная мера непрерывна в нуле, то есть $S(\{0\}) =0$, а значит $$  \frac{1}{N}\sum_{k=1}^N R\ag{_X}(k) \xrightarrow[N\to\infty]{} 0$$ (см. далее теорему Слуцкого \ref{Slutsky});
    \item $V(0) = 0$ п.н.
\end{itemize}
Детали доказательства см. \cite[гл. 7 ]{BulinskyShiryaev2005}.
Аналогичные результаты можно выписать и для стационарных в широком смысле с.к.-непрерывных случайных процессов, если заменить суммы на интегралы. Отметим также, что понятие эргодичности по математическому ожиданию  можно вводить и для нестационарных процессов, как будет показано в  разделе~\ref{ergodic}. В этом  разделе будет также приведен закон больших чисел для стационарных в узком смысле процессов.
\begin{example}
Для процесса из примера \ref{14} написать закон больших чисел.
\end{example}
Заметим, что представление $$Z(t)=A\cos{(Bt + \phi)} = \frac{A}{2} e^{i\phi}e^{Bt} + \frac{A}{2} e^{-i\phi}e^{-Bt}$$ даёт спектральное разложение процесса. А именно, если $B\neq 0$, то  
\begin{equation*}
    V(\nu) = \left \{ 
    \begin{array}{*{20}{c}}
	 {\frac{A}{2} e^{i\phi},}&{\nu = B}, \\ \\
	 {\frac{A}{2} e^{-i\phi},}&{\nu = -B},
	 \\ \\
	 {0,}&{\text{иначе}}.
	\end{array}\right.
\end{equation*}
В этом случае
\begin{equation*}
    \frac{1}{T}\int_0^T Z(t)\,dt \xrightarrow[T\to\infty]{L_2}0.
\end{equation*}
Если $B=0$, то 
\begin{equation*}
    V(\nu) = \left \{ 
    \begin{array}{*{20}{c}}
	 {A\cos{\pi},}&{\nu =0}, \\ \\
	 {0,}&{\text{иначе}}.
	\end{array}\right.
\end{equation*}
В этом случае
\begin{equation*}
    \frac{1}{T}\int_0^T Z(t)\,dt \xrightarrow[T\to\infty]{L_2}A\cos{\phi}. \text{ \EndEx}
\end{equation*}
}
\subsection{Линейные преобразования процессов}

Спектральное представление случайных процессов встречается в основном в инженерных приложениях, когда задана некоторая линейная система, на ее вход поступает стационарный случайный процесс, а на выходе получается преобразованный стационарный процесс. Эта линейная система осуществляет линейное преобразование над входным процессом, обычно это какая-то комбинация производных и интегралов от входного процесса. Оказывается, свойства преобразованного процесса удается сравнительно легко описать в терминах спектральных функций входного процесса. Для того чтобы можно было корректно описывать преобразованный процесс в терминах входного процесса, кроме теоремы \ref{th:WeakStatRepresent} нам потребуются следующие два факта (см., например, \cite{MillerPankov}).

\begin{theorem}\label{th:SpectralCharact}
\textit{Пусть даны стационарный процесс ${X(t)\in CL_2}$ \ag{с} спектральной функцией $S_X(\nu)$ и представлением $$X(t)=\int\limits_{-\infty}^{+\infty}e^{i\nu t}\,dV_X(\nu)$$ и комплекснозначная функция $\Phi(\nu)$, удовлетворяющая условию $$\int\limits_{-\infty}^{+\infty}|\Phi(\nu)|^2\,dS_X(\nu)<\infty.$$ Пусть некоторый центрированный стационарный случайный процесс ${Y(t)\in CL_2}$ допускает представление $$Y(t)=\int\limits_{-\infty}^{+\infty}e^{i\nu t}\Phi(\nu)\,dV_X(\nu).$$ Тогда спектральная функция $S_X(\nu)$ процесса $X(t)$ связана со спектральной функцией $S_Y(t)$ процесса $Y(t)$ соотношением} $$dS_Y(\nu)=|\Phi(\nu)|^2\,dS_X(\nu).$$
\end{theorem}

\begin{proof}
Так как процесс $Y(t)\in L_2$ является стационарным, то в силу теоремы \ref{th:WeakStatRepresent} найдется процесс $V_Y(\nu)$ с ортогональными приращениями, такой, что почти всюду будет верно равенство
$$\int\limits_{-\infty}^{+\infty}e^{i\nu t}\Phi(\nu)\,dV_X(\nu) = \int\limits_{-\infty}^{+\infty}e^{i\nu t}\,dV_Y(\nu).$$ Но так как процесс $V_Y(\nu)$ определен с точностью до аддитивной случайной величины (см. теорему \ref{th:WeakStatRepresent}), то $$\Phi(\nu)\,dV_X(\nu)=dV_Y(\nu).$$ Остается возмести обе части в квадрат и вычислить математическое ожидание: $$|\Phi(\nu)|^2\,\mathbb{E}|dV_X(\nu)|^2=\mathbb{E}|dV_Y(\nu)|^2,$$ откуда сразу получаем требуемое утверждение. \EndProof
\end{proof}

\begin{theorem}\label{th:SpectralMoments}
\textit{Если для некоторого ${k\ge0}$ существует конечный интеграл $$\int\limits_{-\infty}^{+\infty}\nu^{2k}\,dS_{X}(\nu) < \infty,$$ а процесс $X(t)$ является $k$ раз с.к.-дифференцируемым, то с.к.-произ\-водная $$X^{(k)}(t)=\int\limits_{-\infty}^{+\infty}e^{i\nu t}(i\nu)^{k}\,dV_{X}(\nu),$$ т.е. производную можно вносить под знак стохастического интеграла.}
\end{theorem}

\begin{example}
Дана линейная система $$a_1\dot{X}(t)+a_0 X(t) = Y(t)$$ с входным сигналом $Y(t)$ и выходным сигналом $X(t)$. Пусть ${\mathbb{E}Y(t)=0}$ для любого $t\in\mathbb{R}$. Считая известной спектральную плотность $\rho_{Y}(\nu)$ процесса $Y(t)$, вычислить спектральную плотность процесса $X(t)$.
\end{example}

\begin{solution}
Будем искать стационарное с.к.-дифференцируемое решение $X(t)$ указанного уравнения. Для этого представим процесс $X(t)$ в виде
$$X(t)=\int\limits_{-\infty}^{+\infty}e^{i\nu t}dV_X(\nu)$$ и подставим это выражение в уравнение, воспользовавшись теоремой~\ref{th:SpectralMoments}:
$$\int\limits_{-\infty}^{+\infty}(a_1 i\nu+a_0)e^{i\nu t}dV_X(\nu)=Y(t),$$ 
предположив, что для искомого решения выполняются условия
\begin{equation}\label{eq:Cond1}
\int\limits_{-\infty}^{+\infty}|a_1 i \nu + a_0|^2\, dS_X(\nu) < \infty,
\end{equation}
\begin{equation}\label{eq:Cond2}
\int\limits_{-\infty}^{+\infty}\nu^2\,dS_X(\nu) < \infty.
\end{equation}
По теореме \ref{th:SpectralCharact} получаем, что спектральные функции $S_Y(\nu)$ и $S_X(\nu)$ связаны друг с другом соотношением $$dS_Y(\nu)=|a_1 i \nu + a_0|^2 dS_X(\nu),$$ откуда следует
\begin{equation}\label{eq:rhoX}
\rho_X(\nu) = \frac{\rho_Y(\nu)}{|a_1 i \nu + a_0|^2}=\frac{\rho_Y(\nu)}{a_1^2\nu^2+a_0^2}.
\end{equation}
В результате мы получаем, что если условия~\eqref{eq:Cond1} и~\eqref{eq:Cond2} выполнены, то спектральная плотность процесса $X(t)$ существует и удовлетворяет~\eqref{eq:rhoX}. Заметим, что для функции~\eqref{eq:rhoX} оба условия оказываются выполнены. Действительно, условие~\eqref{eq:Cond1} равносильно условию абсолютной интегрируемости спектральной плотности $\rho_Y(\nu)$, а условие~\eqref{eq:Cond2} выполнено, как следует из оценки $$\frac{\nu^2\rho_Y(\nu)}{a_1^2\nu^2+a_0^2}\le\frac{1}{a_1^2}\,\rho_Y(\nu)$$ и, опять же, интегрируемости функции $\rho_Y(\nu)$.

Заметим, наконец, что в принципе могут существовать решения, которые не удовлетворяют условию~\eqref{eq:Cond2}. Более того, спектральная плотность не однозначно определяет стационарный процесс. Другими словами, спектральной плотности, которую мы нашли, может соответствовать много разных процессов $X(t)$. \EndEx
\end{solution}

Закрепим предыдущий пример, рассмотрев содержательную задачу из курса физики.
\begin{example}
Рассмотрим колебательный контур, состоящий из последовательно соединенных катушки индуктивности, конденсатора, сопротивления и источника сторонних эдс. Эдс $\mathcal{E}(t)$, заряд $q(t)$ на обкладках конденсатора и производные $\dot q(t)$, $\ddot q(t)$ считаются достаточно с.к.-гладкими стационарными в широком смысле случайными процессами с нулевым математическим ожиданием. Считая известной спектральную плотность $\rho_{\mathcal{E}}(\nu)$ процесса $\mathcal{E}(t)$, вычислить спектральную плотность $\rho_{q}(\nu)$ процесса $q(t)$.
\end{example}

\begin{solution}
Уравнение тока в цепи имеет вид (см. \cite{Kingsep2001}): $$L\ddot q + R\dot q + \frac{q}{C}=\mathcal{E}.$$ Так как $\mathcal{E}(t)$ и $q(t)$ считаются стационарными процессами с нулевым математическим ожиданием, то они могут быть представлены в виде $$\mathcal{E}(t)=\int\limits_{-\infty}^{+\infty}e^{i\nu t}dV_{\mathcal{E}}(\nu), \ \ q(t)=\int\limits_{-\infty}^{+\infty}e^{i\nu t}dV_{q}(\nu).$$ Теперь вычислим поочередно первую и вторую производную $q(t)$: $$\dot q(t)=\int\limits_{-\infty}^{+\infty}(i\nu)e^{i\nu t}dV_{q}(\nu), \ \ \ddot q(t)=\int\limits_{-\infty}^{+\infty}(i\nu)^2e^{i\nu t}dV_{q}(\nu)$$
в предположении условий
\begin{equation}\label{eq:Cond3}
\int\limits_{-\infty}^{+\infty}|i\nu|^2dS_q(\nu)<\infty, \ \ \int\limits_{-\infty}^{+\infty}|i\nu|^4dS_q(\nu)<\infty.
\end{equation}
Подставляя выражения для $q(t)$, $\dot q(t)$ и $\ddot q(t)$ в исходное уравнение, мы получаем
$$\int\limits_{-\infty}^{+\infty}\left[ L(i\nu)^2 + R(i\nu) + C^{-1} \right]e^{i\nu t}dV_q(\nu)=\mathcal{E}(t).$$ Теперь предположим, что
\begin{equation}\label{eq:Cond4}
\int\limits_{-\infty}^{+\infty}\left| L(i\nu)^2 + R(i\nu) + C^{-1} \right|^2\,dS_q(\nu)<\infty.
\end{equation}
Тогда по теореме \ref{th:SpectralCharact} получаем, что спектральные функции $S_q(\nu)$ и $S_{\mathcal{E}}(\nu)$ связаны соотношением
$$dS_{\mathcal{E}}(\nu) = \left| L(i\nu)^2 + R(i\nu) + C^{-1} \right|^2 dS_q(\nu),$$
что означает существование спектральной плотности $$\rho_q(\nu)=\frac{\rho_{\mathcal{E}}(\nu)}{\left|L(i\nu)^2+R(i\nu)+C^{-1}\right|^2}.$$ Легко видеть, что условия~\eqref{eq:Cond3} и~\eqref{eq:Cond4} для этой функции выполнены.

Итак, среди всех достаточно с.к.-гладких стационарных в широком смысле процессов, удовлетворяющих условиям~\eqref{eq:Cond3} и~\eqref{eq:Cond4}, решения исходного уравнения (если существуют) обладают спектральной плотностью, записанной выше. Вопрос о существовании таких решений здесь не рассматривается. \EndEx
\end{solution}

За более полным изложением спектральной теории случайных процессов мы отправляем читателя к классическим источникам~\cite{Kramer1969,Yaglom1987}. В частности, из важного, но не вошедшего в данное пособие, отметим \textit{разложение Вольда}~\cite{BulinskyShiryaev2005, Shiryaev2} 
и его использование при предсказании поведения случайного процесса по имеющейся истории.

\section{Эргодические процессы}
\label{ergodic}

В этом разделе мы 
введем одно из важнейших понятий теории случайных процессов -- понятие эргодичности. Свойство эргодичности процессов нам будет встречаться до конца этой книги.
\elena{С результатами эргодического типа читатели уже знакомы из курса теории вероятностей: (усиленный) закон больших чисел (ЗБЧ, УЗБЧ) для последовательности независимых одинаково распределенных случайных величин $X_1, X_2,\ldots$, которые утверждают, что среднее по времени $\frac{1}{N}\sum_{n=1}^N X_n$ в соответствующем смысле (по вероятности или п.н.) близко к среднему по вероятностному пространству $\mathbb E X_1$. В курсе случайных процессов результаты типа ЗБЧ тоже уже возникали для стационарных в широком смысле процессов, см. раздел 4.5. В этом разделе 5 будут получены результаты эргодического типа, но 
\begin{itemize}
    \item для более широкого класса процессов (см. \ag{раздел} 5.1). При этом теоремы, касающиеся стационарных в широком смысле процессов будут получены здесь без использования спектрального представления (см. теорему Слуцкого \ref{Slutsky}).
    \item для стационарных в узком смысле последовательностей (теория которых тесно связана с понятием динамических систем, а именно с соответствующим преобразованием, сохраняющим вероятностную меру, и его свойствами (эргодичности, перемешивания и т.д.), см. раздел 5.2.
\end{itemize}
}

\subsection{Эргодические случайные процессы}\label{sec:ergodic_processes_subsection}

\elena{Рассмотрим следующую прикладную задачу курса случайных процессов. Пусть имеется единственная, но достаточно длинная реализация некоторого случайного процесса $X(t)$, где $t\in [0,T]$, если время предполагается непрерывным, или $t\in\{1,2,\ldots,N\}$, если время дискретно. Требуется оценить какие-то характеристики этого процесса.}

\begin{definition}\label{def:ErgodicDefin}
Вещественный случайный процесс ${X(t)\in L_2}$ называется \textit{\shmaxg{эргодическим} по математическому ожиданию}, если его математическое ожидание постоянно ${m_X(t) \equiv m_X\in \Rbb}$ и $$\langle X\rangle_T \overset{\text{def}}{=} \frac{1}{T}\int\limits_{0}^{T}X(t)\,dt\xrightarrow[T\to\infty]{L_2} m_X.$$
\end{definition}

В выражении выше подразумевается, что процесс $X(t)$ с.к.-интег\-ри\-руем на любом отрезке вида $[0,T]$, ${T>0}$, а интеграл понимается в смысле с.к.-интеграла.

\elena{Таким образом для \shmaxg{эргодического} по математическому ожиданию процесса <<хорошей>> оценкой параметра $m_X$ служит среднее по времени единственной наблюдаемой реализации $\langle X\rangle_T$.}

\begin{theorem}[ (критерий эргодичности)]\label{th:ErgodicCriteria}
\textit{Процесс ${X(t)\in L_2}$ с постоянным математическим ожиданием ${m_X(t)\equiv m_X}$ является эргодическим по математическому ожиданию тогда и только тогда, когда }
\begin{equation}\label{eq:ErgodCriteria}
    \underset{T\to\infty}{\lim} \,\frac{1}{T^2}\int\limits_{0}^{T}\int\limits_{0}^{T}R_X(t_1,t_2)\,dt_1dt_2 = 0.
\end{equation}
\end{theorem}

\begin{proof}
Покажем, что
$$
\frac{1}{T^2}\int\limits_{0}^{T}\int\limits_{0}^{T}R_X(t_1,t_2)dt_1dt_2 = \Exp\Big|\langle X \rangle_T - m_X\Big|^2.
$$
Для этого разложим выражение $\Exp\left|\langle X \rangle_T - m_X\right|^2$ на два множителя, представим в обоих множителях $\langle X \rangle_T$ в виде с.к.-предела, при этом разбиения отрезка $[0,T]$ выберем одинаковыми, а точки между ними различными, и воспользуемся теоремой~\ref{th:ContScalarProd}:
	$$
	\Exp\Big|\langle X \rangle_T - m_X\Big|^2 = \Exp\Big(\underset{\substack{0=t_0 < t_1<\ldots<t_m=T \\
	\underset{i=\overline{1,m}}{\max}\Delta t_i \xrightarrow[n\to\infty]{}0\\
	t_i'\in[t_i,t_{i+1})}}{\liminmean}\frac{1}{T} \sum_{i=1}^{m}(X(t_i')-m_X)\Delta t_i\Big)\times
	$$
	$$
	\times\Big(\underset{\substack{0=t_0 < t_1<\ldots<t_m=T \\
	\underset{j=\overline{1,m}}{\max}\Delta t_j \xrightarrow[n\to\infty]{}0\\
	t_j'\in[t_j,t_{j+1})}}{\liminmean}\frac{1}{T} \sum_{j=1}^{m}(X(t_j')-m_X)\Delta t_j\Big) \overset{\text{Теорема \ref{th:ContScalarProd}}}{=} 
	$$
	$$
	= \underset{\substack{0=t_0 < t_1<\ldots<t_m=T \\
	\underset{i=\overline{1,m}}{\max}\Delta t_i \xrightarrow[n\to\infty]{}0\\
	t_i'\in[t_i,t_{i+1})\\
	t_j'\in[t_j,t_{j+1})}}{\lim}\Exp\frac{1}{T^2}\sum_{i=1}^{m}\sum_{j=1}^{m}(X(t_i')-m_X)(X(t_j')-m_X)\Delta t_i\Delta t_j =
	$$
	$$
	= \underset{\substack{0=t_0 < t_1<\ldots<t_m=T \\
	\underset{i=\overline{1,m}}{\max}\Delta t_i \xrightarrow[n\to\infty]{}0\\
	t_i'\in[t_i,t_{i+1})\\
	t_j'\in[t_j,t_{j+1})}}{\lim}\frac{1}{T^2}\sum_{i=1}^{m}\sum_{j=1}^{m}R_X(t_i',t_j')\Delta t_i\Delta t_j = 
	$$
	$$
	= \frac{1}{T^2}\int\limits_{0}^{T}\int\limits_{0}^{T}R_X(t_1,t_2)\,dt_1dt_2,
	$$
	откуда и получаем требуемое, т.е.
	$$
	\langle X\rangle_T \xrightarrow[T\to\infty]{L_2} m_X\; \Leftrightarrow \; \underset{T\to\infty}{\lim} \frac{1}{T^2}\int\limits_{0}^{T}\int\limits_{0}^{T}R_X(t_1,t_2)\,dt_1dt_2 = 0. \text{ \EndProof}
	$$
\end{proof}
	
Обратим еще внимание на то, что из условия~\eqref{eq:ErgodCriteria} следует существование с.к.-интеграла  $\int_0^T X(t)\,dt$ для любого $T>0$.

\begin{example}\label{ex:Ex7.2}
Пусть $\{X(t) \in L_2, t \ge 0\}$ -- случайный процесс с нулевым математическим ожиданием и корреляционной функцией $$R_X(t_1,t_2) = e^{-|t_1-t_2|}.$$ Исследовать процесс $X(t)$ на эргодичность по математическому ожиданию.
\end{example}

\begin{solution}
Математическое ожидание данного процесса есть константа, а так как функция ${R_X(t_1,t_2)}$ интегрируема по Риману на любом квадрате ${[0,T]^2}$, ${T>0}$, то $X(t)$ является с.к.-интегрируемым на любом отрезке ${[0,T]}$, ${T>0}$. Воспользуемся теперь критерием эргодичности по математическому ожиданию (теорема \ref{th:ErgodicCriteria}). Для этого рассмотрим интеграл 
\begin{eqnarray*}
\int\limits_{0}^T \int\limits_0^T R_X(t_1, t_2) \,dt_1 dt_2 &=& \int\limits_{0}^T\int\limits_0^T e^{-|t_1-t_2|}\,dt_1dt_2= \\
&=& \int\limits_0^T \left( \int\limits_0^{t_2} e^{t_1-t_2}\,dt_1 + \int\limits_{t_2}^T e^{t_2-t_1}\,dt_1 \right)\,dt_2= \\
&=& \int\limits_0^T \left( 1-e^{-t_2}+1-e^{t_2-T} \right)\,dt_2 =\\
&=& 2T+e^{-T}-1+e^{-T}-1 =\\
&=& 2T+2e^{-T}-2.
\end{eqnarray*}
Очевидно, $T^{-2}(2T+2e^{-T}-2)\to0$ при $T\to\infty$, поэтому процесс $X(t)$ является эргодическим по математическому ожиданию. \EndEx
\end{solution}

\begin{example}\label{ex:Ex7.3}
Пусть $\{W(t)$, $t\ge0\}$ -- винеровский процесс. Исследовать на эргодичность по математическому ожиданию процесс $X(t)=W(t)/t$, $t\ge1$.
\end{example}

\begin{solution}
Так как винеровский процесс в каждом сечении имеет конечные моменты произвольного порядка, то и процесс $X(t)$ имеет моменты произвольного порядка. Следовательно, ${X(t)\in L_2}$. Далее, очевидно, $m_X(t)=0$ всюду и $$R_X(t_1,t_2)=\frac{\min(t_1,t_2)}{t_1t_2}=\frac{1}{\max(t_1,t_2)}, \ t_1\ge1, \ t_2\ge1.$$ Теперь остается заметить, что $\max(t_1,t_2) \ge t_1$ всюду, следовательно, $$\int\limits_1^T\int\limits_1^T \frac{1}{\max(t_1,t_2)}\,dt_1dt_2 \le \int\limits_1^T \int\limits_1^T \frac{1}{t_1}\,dt_1dt_2=(T-1)\ln{T}.$$ Очевидно, что ${(T-1)^{-2}(T-1)\ln{T}\to0}$ при $T\to\infty$, поэтому процесс ${X(t)=W(t)/t}$, ${t\ge1}$, является эргодическим по математическому ожиданию. \EndEx
\end{solution}

\begin{example}
На вход некоторой вычислительной машины поступают зашумленные значения функции вида $e^{at}$ с неизвестным значением параметра $a$; $t$ интерпретируется как время. Модель фактически получаемого сигнала принята следующей: $$S(t)=e^{at+W(t)},$$ где $W(t)$ -- винеровский процесс. Ставится задача оценки параметра $a$.
\end{example}

\begin{solution}
Для оценивания параметра $a$ воспользуемся теорией эргодических процессов. \elena{Для этого найдем такой процесс $X(t)$, являющийся преобразованием исходного процесса, чтобы $\mathbb E X(t) = a$ и $X(t)$ был эргодичен по математическому ожиданию. Тогда хорошей оценкой параметра $a$ будет служить его среднее по времени. В качестве такого процесса рассмотрим}  случайный процесс: $$X(t)=\frac{\ln{S(t)}}{t}=a+\frac{W(t)}{t}, \ t\ge1.$$ Каждое сечение этого процесса имеет моменты произвольного порядка, поэтому $X(t)\in L_2$. Математическое ожидание этого процесса постоянно и равно неизвестному параметру $a$. Что касается корреляционной функцией, то она совпадает с корреляционной функции процесса из прошлого примера, т.е. ${R_X(t_1,t_2)=1/\max(t_1,t_2)}$, т.к. аддитивные постоянные (в нашем случае это $a$) не влияют на значения корреляционной функции или дисперсии. Из выкладок предыдущего примера следует, что $X(t)$ является эргодическим по математическому ожиданию процессом. Следовательно, наблюдая процесс $S(t)$ на достаточно продолжительном интервале времени $[1,T]$, мы можем оценить параметр $a$ по формуле $$a \approx \frac{1}{T-1}\int\limits_{1}^T X(t)\,dt=\frac{1}{T-1}\int\limits_1^T\frac{\ln{S(t)}}{t}\,dt. \text{ \EndEx}$$
\end{solution}

\shmaxg{\begin{theorem}[ (достаточное условие эргодичности)]\label{th:ErgodicSuffCond}
\textit{Для того, чтобы процесс второго порядка $X(t)$, с.к.-интегрируемый на любом отрезке и с постоянным математическим ожиданием, был эргодичен по математическому ожиданию, достаточно, чтобы} $$\lim\limits_{|t-s|\to\infty} R_X(t,s)=0, \ \lim\limits_{T\to\infty}\frac{1}{T}\max\limits_{t\in[0,T]}R_X(t,t) =0.$$
\end{theorem}
\begin{proof}. Пусть дан процесс с нужными свойствами. Тогда для любого ${\eps>0}$ найдется $T_0$ такой, что для любого $T=|t_2-t_1|>T_0$ выполнено $|R_X(t_1,t_2)|<\eps$. Зафиксируем $\eps>0$, возьмем подходящий $T_0$ и разобьем квадрат $[0,T]^2$ на два множества: $$G_1=\{(t_1,t_2)\in T\times T:|t_2-t_1|>T_0\},$$ $$G_2=\{(t_1,t_2)\in T\times T:|t_2-t_1|\le T_0\}.$$ Обозначим за $S_1$ и $S_2$ площади множеств $G_1$ и $G_2$ соответственно. Тогда
$$\left| \frac{1}{T^2}\int\limits_0^T\int\limits_0^T R_X(t_1,t_2)\,dt_1dt_2 \right|=$$
$$=\left| \frac{1}{T^2}\iint\limits_{G_1} R_X(t_1,t_2)\,dt_1dt_2  +  \frac{1}{T^2}\iint\limits_{G_2} R_X(t_1,t_2)\,dt_1dt_2 \right|\le$$
$$\le\frac{1}{T^2}\left( \iint\limits_{G_1} |R_X(t_1,t_2)|\,dt_1dt_2  +  \iint\limits_{G_2} |R_X(t_1,t_2)|\,dt_1dt_2 \right)\le$$
$$\le \frac{\eps S_1 + \max_{G_2}{|R_X(t_1,t_2)|}S_2}{T^2}\le \eps + 2\max_{G_2}{|R_X(t_1,t_2)|}\frac{T_0}{T},$$ так как $S_1 \le T^2$ и $S_2 = T^2 - 2(1/2)(T-T_0)^2 = 2TT_0-T_0^2 \le 2T_0T$. Из неравенства Коши--Буняковского следует, что $$|R_X(t_1,t_2)|^2\le R_X(t_1,t_1)R_X(t_2,t_2),$$ поэтому $\max_{G_2}{|R_X(t_1,t_2)|}\le \max_{t\in[0,T]}|R_X(t,t)|$. Теперь достаточно взять $T \ge T_0$ и такое, чтобы $\max_{t\in[0,T]}|R_X(t,t)|/T \le \eps$. \EndProof
\end{proof}}

\shmaxg{Из этой теоремы следует, что для того, чтобы стационарный в широком смысле процесс второго порядка $X(t)$, с.к.-интегрируемый на любом отрезке и с постоянным математическим ожиданием, был эргодичен по математическому ожиданию достаточно, чтобы $$\lim\limits_{t\to\infty} R_X(t)=0.$$
Заметим, что у стационарного в широком смысле процесса дисперсия постоянна и поэтому условие $\lim_{T\to\infty}\max_{t\in[0,T]}R_X(t,t)/T=0$ выполнено автоматически.}

Эргодические процессы не обязательно являются стационарными в каком-либо смысле. Например, процесс из примера \ref{ex:Ex7.3} не является стационарным в широком смысле, так как его корреляционная функция не является функцией разности $t_1$ и $t_2$. Однако в случаях, когда процесс является стационарным в широком смысле, достаточное условие становится использовать особенно легко. Например, в примере \ref{ex:Ex7.2} дан процесс с корреляционной функцией $R_X(t_1,t_2)=\exp(-|t_1-t_2|)$. Очевидно, что ${R_X(t_1,t_2)\to0}$ при ${|t_1-t_2|\to\infty}$, поэтому с учетом теоремы \ref{th:ErgodicSuffCond} легко заключаем, что этот процесс является эргодическим по математическому ожиданию.


\begin{theorem}[ (Слуцкий, см. \cite{Yaglom1987})] 
\label{Slutsky}
\textit{Пусть случайный процесс $X(t)$ стационарен в широком смысле и при этом ${R_X(t_1, t_2) = R\ag{_X}(t_2-t_1)}$. Тогда}
\begin{equation}
\label{Slutsky 1}
    J = \frac{1}{T^2}\int\limits_{0}^{T}\int\limits_{0}^{T}R_X(t_1,t_2)\,dt_1dt_2 = \frac{2}{T}\int\limits_{0}^{T}\left(1-\frac{\tau}{T}\right)R\ag{_X}(\tau)\,d\tau.
\end{equation}
\elena{При этом  процесс  $X(t)$ эргодичен по математическому ожиданию тогда и только тогда, когда
\begin{equation}
    \label{Slutsky 2}
     \frac{1}{T}\int\limits_{0}^{T}R\ag{_X}(\tau)\,d\tau \to 0,\quad T\to\infty.
\end{equation}
Иначе говоря, \eqref{Slutsky 2} равносильно тому, что \ag{$J$} из формулы \eqref{Slutsky 1} стремится к нулю при $T\to\infty$.}
\end{theorem}

\begin{proof}
Выполним замену переменных: $\tau=t_2\!-\!t_1,$ $s = t_2$; модуль определителя матрицы Якоби такого преобразования равен единице. Область интегрирования в новых координатах $(\tau,s)$~--- это два прямоугольных треугольника:
\begin{center}
    $I_1$ с вершинами в точках $(0,T)$, $(0,0)$, $(-T,0)$,
\end{center}
\begin{center}
    $I_2$ с вершинами в точках $(0,0)$,  $(0,T)$, $(T,T)$.
\end{center}
\begin{center}
\begin{tikzpicture}[scale=1.5]
\definecolor{tempcolor}{RGB}{255,204,0}
    \draw [fill=tempcolor] (0, 0) -- (0, 1) -- (-1, 0) -- (0, 0);
    \draw [fill=tempcolor] (0, 0) -- (0, 1) -- (1, 1) -- (0, 0);
    
    \draw[thick, ->] (-1.5,0)--(1.5,0) node[above] {$\tau$}; 
    \draw[thick, ->] (0,0)--(0,1.5) node[left] {$s$}; 
    
    \draw[thick] (0, 1) -- (-1, 0);
    \draw[thick] (0, 1) -- (1, 1);
    \draw[thick] (0, 0) -- (1, 1);
    \draw[dashed] (1, 0) -- (1, 1);
    
    \draw (-0.15, 1.1) node {$T$};
    \draw (1, -0.15) node {$T$};
    \draw (-1, -0.15) node {$-T$};
    
    \draw (-0.3, 0.3) node {$I_1$}; 
    \draw (0.3, 0.6) node {$I_2$}; 
\end{tikzpicture}
\end{center}

Следовательно,
\begin{eqnarray*}
J &=& \frac{1}{T^2}\iint_{I_1}R\ag{_X}(\tau)\,d\tau ds + \frac{1}{T^2}\iint_{I_2}R\ag{_X}(\tau)\,d\tau ds= \\
&=& \frac{1}{T^2}\int\limits_{-T}^{0}R\ag{_X}(\tau)\,d\tau\int\limits_{0}^{\tau+T}\,ds + \frac{1}{T^2} \int\limits_{0}^{T}R\ag{_X}(\tau)\,d\tau\int\limits_{\tau}^{T}\,ds= \\
&=& \frac{1}{T^2}\int\limits_{-T}^{0}(T+\tau)R\ag{_X}(\tau)\,d\tau + \frac{1}{T^2}\int\limits_{0}^{T}(T-\tau)R\ag{_X}(\tau)\,d\tau= \\
&=& \frac{1}{T^2}\int\limits_{0}^{T}(T-\tau)R\ag{_X}(-\tau)\,d\tau + \frac{1}{T^2}\int\limits_{0}^{T}(T-\tau)R\ag{_X}(\tau)\,d\tau= \\
&=& \frac{2}{T}\int\limits_{0}^{T}\left(1-\frac{\tau}{T}\right)R\ag{_X}(\tau)\,d\tau,
\end{eqnarray*}
где было учтено равенство $R\ag{_X}(-\tau)=R\ag{_X}(\tau)$ для вещественных процессов.

\elena{Для доказательства второй части теоремы, заметим, что $$\frac{1}{T}\int\limits_{0}^{T}R\ag{_X}(\tau)\,d\tau = \mathbb E \left [ ( X(0) -m_X) (\langle X\rangle_T - m_X)  \right]. $$ Отсюда и из неравенства Коши--Буняковского следует, что если \linebreak ${\mathbb E (\langle X\rangle_T - m_X)^2  \to 0}$, то ${ \frac{1}{T}\int_{0}^{T}R\ag{_X}(\tau)\,d\tau \to 0 }$. Это доказывает необходимость.}

\elena{Для доказательства достаточности воспользуемся критерием  \ref{th:ErgodicCriteria} и сделанной выше выкладкой. Удобнее поменять порядок интегрирования: }
\elena{
\begin{equation*}
J = \frac{2}{T^2}\iint_{I_2}R\ag{_X}(\tau)\,d\tau ds= \frac{2}{T^2}\int\limits_{0}^{T}\int\limits_{0}^{s} R\ag{_X}(\tau)\,d\tau\,ds.
\end{equation*}
Далее из условия $\frac{1}{T}\int_{0}^{T}R\ag{_X}(\tau)\,d\tau \to 0$ следует, что для любого $\epsilon > 0$ существует $T_\circ(\epsilon)$ такое, что для любого $s > T_\circ$ $\int_0^s R\ag{_X}(\tau)\,d\tau < \epsilon s$. При $s< T_\circ$ воспользуемся неравенством (которое на самом деле справедливо для любых $s$) $\int_0^s R\ag{_X}(\tau)\,d\tau \le R\ag{_X}(0) s$. 
\begin{eqnarray*}
J &=& \frac{2}{T^2}\int\limits_{0}^{T}\int\limits_{0}^{s} R\ag{_X}(\tau)\,d\tau\,ds =\\
&=& \frac{2}{T^2}\int\limits_{0}^{T_\circ}\int\limits_{0}^{s} R\ag{_X}(\tau)\,d\tau\,ds  + \frac{2}{T^2}\int\limits_{T_\circ}^{T}\int\limits_{0}^{s} R\ag{_X}(\tau)\,d\tau\,ds \le\\
& \le & \frac{2}{T^2}\int\limits_{0}^{T_\circ}R\ag{_X}(0)s\,ds + \frac{2}{T^2}\int\limits_{T_\circ}^{T}\int\epsilon s\,ds =\\
&=&  \frac{ T_\circ^2}{T^2}R\ag{_X}(0) + \frac{T^2 - T_\circ^2}{T^2}t\epsilon < \frac{ T_\circ^2}{T^2}R\ag{_X}(0) + \epsilon.
\end{eqnarray*}
Откуда следует, что $J\to 0$ при $T\to\infty$. Что и доказывает достаточность.}~\EndProof 
\end{proof}


Аналогично тому, как мы ввели ранее определение эргодичности по математическому ожиданию, можно ввести понятия \textit{эргодичности по дисперсии} и \textit{эргодичности по корреляционной функции}.

\begin{definition}
Случайный процесс $X(t)\in L_2$ называют \textit{\shmaxg{эргодическим} по дисперсии}, если случайный процесс $${Y(t) = \elena{\accentset{\circ}{X}^2(t)}= (X(t)-m_X(t))^2}$$ эргодичен по математическому ожиданию.
\end{definition}

\elena{Таким образом, для процесса $X(t)$ с постоянной дисперсией $\sigma^2_X = \mathbb E \accentset{\circ}{X}^2(t)$ хорошей оценкой дисперсии будет
$$\langle Y\rangle_T \overset{\text{def}}{=} \frac{1}{T}\int\limits_{0}^{T}\accentset{\circ}{X}^2(t)\,dt\xrightarrow[T\to\infty]{L_2} \sigma^2_X.$$}
\begin{definition}
Случайный процесс ${X(t)\in L_2}$ называют \textit{\shmaxg{эргодическим} по корреляционной функции}, если для каждого ${\tau\ge0}$ случайный процесс ${Y_\tau(t) = \accentset{\circ}{X}(t) \accentset{\circ}{X}(t+\tau)}$ эргодичен по математическому ожиданию.
\end{definition}

\elena{Таким образом, для процесса $X(t)$ с корреляционной функцией, зависящей от разности аргументов $R_X(\tau) = \mathbb E \accentset{\circ}{X}(t)\accentset{\circ}{X}(t+\tau)$ хорошей оценкой корреляционной функции в каждой фиксированной точке $\tau$ будет
$$\langle Y_\tau\rangle_T \overset{\text{def}}{=} \frac{1}{T}\int\limits_{0}^{T}\accentset{\circ}{X}(t) \accentset{\circ}{X}(t+\tau) \,dt\xrightarrow[T\to\infty]{L_2} R_X(\tau).$$}

	
	

\subsection[Эргодические динамические системы\vspace{2mm}]{Эргодические динамические системы \elena{и закон больших чисел для стационарных в узком смысле случайных последовательностей}*}\label{sec:ErgodicDS}

В данном разделе вводится новое понятие -- \textit{динамическая система}, а вместе с ней и понятие \textit{эргодической динамической системы}. Здесь мы постараемся продемонстрировать, как теория эргодических процессов может помочь в понимании жемчужины теории динамических систем \textit{эргодической теоремы Биркгофа--Хинчина--Фон~Неймана} \cite{KoralovSinai,SinaiLections1996,Shiryaev2}. 

\begin{definition}
Совокупность $M = (\Omega, \F, \mathbb P = \mu, T)$, 
где $(\Omega, \F, \mathbb P=\mu)$~--- вероятностное пространство, и \elena{измеримое} отображение $T:\Omega\to \Omega$ удовлетворяет свойству
	$$
	\forall B\in\F \hookrightarrow \mu\left(T^{-1}(B)\right) = \mu(B)
	$$
	(вероятностная мера $\mu$ инвариантна относительно преобразования $T$\elena{, либо иначе говоря, $T$ -- сохраняющее меру $\mu$ преобразование}), будем называть \textit{динамической системой}.
\end{definition}

\elena{На самом деле для теории динамических систем мера $\mu$ не обязана быть вероятностной, однако для связи результатов со случайными процессами, мы будем все же предполагать,что $\mu$ вероятностная мера.}

\elena{Пусть $\xi$ некоторая случайная величина, заданная на вероятностном пространстве  $(\Omega, \F, \mu)$. Рассмотрим случайный процесс, порожденный динамической системой (преобразованием $T$) $X_0(\omega) = \xi(\omega)$, $X_1(\omega) = \xi(T \omega)$, $X_2(\omega) = \xi(T^2 \omega)$, \ldots Несложно убедиться, что такая последовательность является стационарной в узком смысле. На самом деле верен и обратный результат, что для каждой стационарной в узком смысле последовательности найдется динамическая система, порождающая её (см. \cite[гл. V, \S 1]{Shiryaev2}). 
}

\begin{definition}
Динамическая система ${M = (\Omega, \F, \mu, T)}$ называется \textit{эргодической}, если из равенства ${T^{-1}(B)=B}$ ($\mu$-п.н.) следует, что либо $B=\varnothing$ ($\mu$-п.н.), либо $B=\Omega$ ($\mu$-п.н.).
\end{definition}

Здесь и далее равенство множеств ${A=B}$ ($\mu$-п.н.) означает равенство ${\mu((A\cup B) \setminus (A\cap B))=0}$, т.е. что мера симметрической разности множеств равна нулю.



\begin{example}\label{ex:Ex8.1}
Пусть $\Omega = S^1$~--- окружность единичной длины на плоскости (будем пользоваться ее гомеоморфностью отрезку $[0,1]$, концы которого отождествлены), $\F = \B(S^1)$~--- борелевская $\sigma$-алгебра на $S^1$, ${\mu=\lambda}$~--- классическая мера Лебега на $S^1$, ${T(\omega) = \{\log_{10}{2}+\omega\}}$, где $\{y \}$ обозначает дробную часть числа $y$, ${\omega\in\Omega}$. Показать, что $M\!\!=$\linebreak $=(\Omega,\F,\mu,T)$ является динамической системой.
\end{example}

\begin{solution}
Из постановки задачи сразу следует, что тройка $(\Omega,\F,\mu)$ является вероятностным пространством. Остается доказать, что 
$$
\forall B\in\F \hookrightarrow \mu\left(T^{-1}(B)\right) = \mu(B).
$$
Заметим, что преобразование $T:\Omega\to\Omega$ осуществляет сдвиг на величину $\log_{10}2$ \gav{по часовой стрелке} вдоль окружности. Обратное же преобразование осуществляет сдвиг на ту же величину\gav{, но} в обратном направлении \gav{(против часовой стрелке). Так} как мера Лебега не зависит от преобразования сдвига, то соотношение выше действительно имеет место. \EndEx
\end{solution}

\begin{theorem}\label{th:CircErgodic}
\textit{Динамическая система ${M=(S^1,\B(S^1),\lambda,T)}$ с преобразованием ${T(\omega)=\{\alpha+\omega\}}$, $\omega\in S^1$, является эргодической тогда и только тогда, когда число $\alpha$ иррационально.}
\end{theorem}

Доказательство этой теоремы можно найти, например, в книгах \cite{Katok1999, SinaiLections1996, Shiryaev2}; мы же оставляем здесь только формулировку.




\elena{
\begin{definition}
Динамическая система $(\Omega,\F,\mu,T)$  обладает свойством \textit{перемешивания}, если для любых $A_1,A_2\in\mathcal F$ при $n\to\infty$
$$
\mu\left( A_1 \bigcap T^{-n}A_2\right) \to \mu(A_1) \mu(A_2). 
$$
\end{definition}
Рассмотрим теперь порожденную динамической системой стационарную в узком смысле последовательность $X_k(\omega) $, $k=0,1,2,\ldots$.
\begin{definition}
Стационарная в узком смысле последовательность $X_k(\omega) $, $k=0,1,2,\ldots$ обладает свойством \textit{слабой зависимости}, если для любого $k$ $X_k$ и $X_{k+n}$ асимптотически независимы при $n\to\infty$, то есть для любых борелевских множеств $B_1,B_2$ \ag{при $n\to\infty$}:
$$
\mu \left(X_k\in B_1, X_{k+n}\in B_2 \right) \to \mu(X_0\in B_1) \mu (X_0\in B_2).
$$
\end{definition}
Стоит отметить, что в случае конечного второго момента $\mathbb E \xi^2 < \infty $ из свойства  слабой зависимости следует свойство асимптотической некоррелированности: $R_X(n)\to 0$ \ag{при $n\to\infty$}. Для  гауссовской меры эти понятия эквивалентны.
}

\elena{
На самом деле, любая последовательность $X_k(\omega) $, $k=0,1,2,\ldots$, порожденная динамической системой обладает свойством слабой зависимости  тогда и только тогда, когда динамическая система обладает свойством перемешивания.
\begin{example}
\label{ex:5.5}
Забегая немного вперед, заметим, что можно рассмотреть  
в качестве иллюстрации предыдущих определений конечную, неразложимую, апериодическую однородную марковскую цепь, на множестве состояний $E$. Это означает, что  цепь <<забывает>>  любое свое начальное распределение и сходится к стационарному распределению с положительными компонентами $\{ \pi_i \} _{i\in E}$  : \ag{при $n \to \infty$}
$$
\mathbb P \left( X_{n}\in B | X_0 = i\right) \to \pi(B)=\sum_{j\in B} \pi_j
$$
для любых $i\in E$ и $B\subset E$.
Положим начальное распределение равное инвариантному, тогда $X_n $, $n=0,1,2,\ldots$ образует стационарный в узком смысле процесс. Несложно проверить, что $X_n $, $n=0,1,2,\ldots$ обладает свойством слабой зависимости (а значит, порождающая цепь динамическая система -- свойством перемешивания). Действительно, для любых подмножеств состояний $B_1, B_2 \subset E$ выполнено \ag{при $n\to\infty$}
\begin{equation*}
    \begin{split}
        &\mathbb P \left(X_0\in B_1, X_{n}\in B_2 \right) = \sum_{i\in B_1}\mathbb P  \left(X_{n}\in B_2 | X_0 = i \right) \pi_i\to\\
        &\to\pi(B_2) \sum_{i\in B_1}\pi_i = \pi(B_1)\pi(B_2).\quad \text{\EndEx}
    \end{split}
\end{equation*}
\end{example}
\begin{definition}
Динамическая система $(\Omega,\F,\mu,T)$  обладает свойством \textit{перемешивания в среднем}, если для любых $A_1,A_2\in\mathcal F$ при $n\to\infty$
$$
\frac{1}{n}\sum_{m=1}^n\mu\left( A_1 \bigcap T^{-m}A_2\right) \to \mu(A_1) \mu(A_2). 
$$
\end{definition}
\begin{definition}
Стационарная в узком смысле последовательность $X_k(\omega) $, $k=0,1,2,\ldots$ обладает свойством \textit{слабой зависимости в среднем}, если  для любых борелевских множеств $B_1,B_2$ \ag{при $n \to \infty$}:
$$
\frac{1}{n}\sum_{m=1}^n\mu \left(X_0\in B_1, X_{m}\in B_2 \right) \to \mu(X_0\in B_1) \mu (X_0\in B_2).
$$
\end{definition}
Аналогично предыдущему, любая последовательность $X_k(\omega)$, $k=0,1,2,\ldots$, порожденная динамической системой обладает свойством слабой зависимости в среднем тогда и только тогда, когда динамическая система обладает свойством слабой зависимости в среднем.}

\elena{
\begin{example}
Пусть цепь из предыдущего примера \ref{ex:5.5} обладает периодом $d>1$. Тогда цепь уже не будет обладать свойством слабой зависимости, но будет обладать свойством слабой зависимости в среднем.~\EndEx
\end{example}
Доказательство следующих  теорем можно посмотреть в книге \cite[гл.~16, \S~2]{Borovkov1999}.
\begin{theorem}
Динамическая система эргодична тогда и только тогда, когда она обладает свойством перемешивания в среднем.
\end{theorem}
\begin{theorem}
Динамическая система $(\Omega, \F, \mu, T)$ эргодична тогда и только тогда, когда для любого $A\in\F$ такого, что $\mu(A) > 0$, выполняется
$$
\mu \left( \bigcup_{n=0}^\infty T^{-n}A \right)
=1,
$$
что означает, что множества $T^{-n}A$, $n=0,1,2,\ldots$ <<заметают>> все пространство $\Omega$.
\end{theorem}}

\elena{
Пусть $X_k(\omega) = \xi(T^k\omega)$, $k=0,1,2,\ldots$ -- стационарная в узком смысле последовательность, порожденная динамической системой  $(\Omega, \F, \mu, T)$. Доказательство следующей теоремы можно посмотреть в \cite[Теорема 1.3]{BillingsleyErgod}, или \cite[гл.~16, \S~3]{Borovkov1999},  \cite[гл.~V, \S~3]{Shiryaev2}. 
\begin{theorem}[~(Биркгоф, Хинчин)]
\label{th: Bir-Hin}
Пусть $\mathbb E |\xi| < \infty$, тогда $\mu-$п.н. существует предел $\eta$, $\mathbb E |\eta| < \infty$
$$
\lim\limits_{n\to\infty} \frac{1}{n}\sum_{k=0}^{n-1}X_k(\omega)=\lim\limits_{n\to\infty} \frac{1}{n}\sum_{k=0}^{n-1}\xi(T^k\omega)=\eta(\omega).
$$
При этом $ \mathbb E \eta = \mathbb E \xi $ и $\eta$ -- инвариантная случайная величина, то есть $\mu-$п.н. выполнено 
$$
\eta(T \omega) = \eta(\omega).
$$
Если к тому же динамическая система эргодична, то $\eta = \mathbb E \xi$.
\end{theorem}
Другими словами, случайная величина $\eta$ есть условное математическое ожидание случайной величины $\xi$ относительно сигма-алгебры, порожденной инвариантными множествами. Если $T$ эргодично, то такая сигма-алгебра тривиальна и тогда $\eta$ равно константе, равной безусловному математическому ожиданию $\xi$.}

\elena{Вместо доказательства теоремы \ref{th: Bir-Hin} читатель может легко убедиться в более слабом утверждении.
\begin{example}
Если динамическая система обладает свойством перемешивания, то для $\xi(\omega) = I\{\omega\in A\}$, где $A\in\F$
$$
\frac{1}{n}\sum_{k=0}^{n-1} I \{ T^k\omega\in A\}
$$
сходится по вероятности при $n\to\infty$ к $\mathbb P (A)$.
Схема доказательства:  свойство перемешивания динамической системы эквивалентно свойству слабой зависимости для рассматриваемого процесса, откуда следует асимптотическая некоррелированность сечений этого процесса. Тогда дисперсия $\frac{1}{n}\sum_{k=0}^{n-1} I \{ T^k\omega\in A\}$ стремится к нулю при $n\to\infty$. Применяя неравенство Чебышева получаем нужное утверждение. \EndEx
\end{example}
\begin{example}[ Следствие теоремы \ref{th: Bir-Hin}]
Если динамическая система эргодична, то для $\xi(\omega) = I\{\omega\in A\}$, где $A\in\F$
$$
\lim\limits_{n\to\infty}\frac{1}{n}\sum_{k=0}^{n-1} I \{ T^k\omega\in A\} = \mathbb P (A) \text{$\mu$-п.н.~\EndEx}
$$
\end{example}
}

\elena{
Доказательство следующей теоремы можно посмотреть в \cite[Теорема 2.1]{BillingsleyErgod}
\begin{theorem}[ (фон Нейман)] 
\label{Newmann}
\textit{Пусть $\xi\in L_2(\Omega,\F,\mu)$, тогда
существует  $\eta\in L_2(\Omega,\F,\mu)$, 
$$
 \frac{1}{n}\sum_{k=0}^{n-1}\xi(T^k\omega) \xrightarrow[N\to\infty]{L_2} \eta(\omega).
$$
При этом $ \mathbb E \eta = \mathbb E \xi $ и $\eta$ -- инвариантная случайная величина, то есть $\mu-$п.н. выполнено 
$$
\eta(T \omega) = \eta(\omega).
$$
Если к тому же динамическая система эргодична, то $\eta = \mathbb E \xi$.
}
\end{theorem}
}

\elena{
\textbf{Замечание 1.}  Интересно еще раз отметить, что в сравнении с классическим усиленным законом больших чисел для независимых случайных величин, в случае стационарной последовательности предел   в общем случае является случайной величиной (ср. с законом больших чисел для стационарных в широком смысле последовательностей). Случай, когда предел является константой является эргодическим. }

\elena{Рассмотрим такой
\begin{example}
Пусть $\Omega = \{\omega_1,\omega_2,\ldots,\omega_d\}$, $d=2l$, мера $\mu$ -- равномерная, 
$T\omega_i = \omega_{(i+2)\mod d}$. Несложно проверить, что такая динамическая система не является эргодичной, так как множество $A=\{1,3,5,2l-1 \}$ будет инвариантным относительно $T$: $A = T^{-1}A$, при этом $\mu(A) = 1/2$. Для любой $\xi$ пределом суммы $ \frac{1}{n}\sum_{k=0}^{n-1}\xi(T^k\omega) $ будет случайная величина $\eta$ равновероятно принимающая два значения:
$ \frac{1}{l}\sum_{j=0}^{l-1}\xi(2j+1) $, если $\omega$ нечетное, и $ \frac{1}{l}\sum_{j=1}^{l}\xi(2j) $, если $\omega$ четное.~\EndEx
\end{example}
}

\elena{Еще раз резюмируем написанное выше. Следующие условия являются эквивалентными:
\begin{itemize}
    \item Динамическая система является эргодической.
    \item Динамическая система обладает свойством перемешивания в среднем.
    \item Для любого $A\in\F$ такого, что $\mu(A) > 0$, выполняется \linebreak
$
\mu \left( \bigcup_{n=0}^\infty T^{-n}A \right)
=1
$.
\item Любая последовательность, порожденная рассматриваемой динамической системой, обладает свойством слабой зависимости в среднем.
\item Любая последовательность, порожденная рассматриваемой динамической системой, с дополнительным условием $\mathbb E \xi^2 < \infty$ имеет непрерывную в нуле спектральную меру, а корреляционная функция имеет нулевой предел (в смысле Чезаро): $\frac{1}{n }\sum_{k=0}^{n-1} R\ag{_X}(k) \to 0$, $n\to\infty$ (см. раздел 4.5).
\end{itemize}
}

\elena{Далее мы будем вместо $\xi$ использовать  $f$, как правило предполагая, что функция $f \in L_2(\Omega,\F, \mu)$.}

Теперь рассмотрим один содержательный пример, который нам потребуется далее в этой главе. Пусть ${M = (\Omega, \F, \mu, T)}$~--- эргодическая динамическая система. Рассмотрим функцию $f \in L_2(\Omega,\F, \mu)$ и случайную последовательность $\{f(T^k\omega)\}_{k=0}^{\infty}$.

Заметим, что математическое ожидание величины $f(T^k \omega)$, $\omega\in\Omega$ не зависит от $k$ и совпадает с математическим ожиданием величины $f(T^0\omega)=f(\omega)$, т.е.
$$
\int\limits_{\Omega}f(T^k\omega)\,d\mu(\omega)=\int\limits_{\Omega}f(\omega)\,d\mu(\omega)=m.
$$
Кроме того, по теореме фон Неймана
$$
\frac{1}{N}\sum_{k=1}^{N} f(T^k\omega) \xrightarrow[N\to\infty]{L_2} m=\int\limits_{\Omega}f(\omega')\,d\mu(\omega').
$$
Итак, мы выяснили, что случайная последовательность $\{f(T^k\omega)\}$ имеет постоянное математическое ожидание, а ее среднее по времени сходится к среднему по пространству (т.е. по мере). Напомним, что аналогичным образом было дано определение эргодичности по математическому ожиданию случайного процесса с непрерывным временем (см. определение \ref{def:ErgodicDefin}). 

Приведем теперь ряд задач, которые могут быть решены с помощью изложенных выше фактов теории эргодических динамических систем.

\begin{example}[ (псевдослучайная последовательность)]\label{ex:MC}
Вернемся к динамической системе из примера \ref{ex:Ex8.1}. Введем случайную последовательность ${X_k(\omega)=T^k\omega}$, ${\omega\in\Omega=S^1}$, ${k\ge1}$\gav{, $\mu$ -- равномерная мера (Лебега) на $S^1$, что можно понимать, как $\omega \in \mathrm{R}[0,1]$}. Покажем, что для каждого $k$ случайная величина $X_k$ имеет равномерное на отрезке $[0,1]$ распределение. Рассмотрим множество ${B=\{\omega:\omega<x\}}$, ${x\in(0,1)}$. Тогда ${T^{-k}B=\{T^{-k}\omega:\omega<x\}=\{\omega':T^k\omega'<x\}}$. Но так как $\mu(T^{-k}B)=$\linebreak $=\mu(B)$, то ${\mu(\{\omega:X_k(\omega)<x\})=\mu(\{\omega:\omega<x\})=x}$, т.е. функция распределения случайной величины $X_k$ равна ${F_{X_k}(x)=x}$ на интервале $(0,1)$. Очевидно, что ${F_{X_k}(x)=0}$ при $x\le0$ и ${F_{X_k}(x)=1}$ при $x\ge1$. Следовательно, $X_k$ имеет равномерное распределение на отрезке $[0,1]$. \gav{Более того, аналогичным образом можно показать, что случайный процесс $\{X_k\}_{k\ge 1}$ -- стационарный в узком смысле. При этом случайные величины $\{X_k\}_{k\ge 1}$ не будут независимыми!} \EndEx
\end{example}

\begin{example}[ (метод Монте-Карло)]
Пусть дана интегрируемая на отрезке $[0,1]$ функция ${f=f(x)}$ и требуется вычислить интеграл $\int_{0}^{1}f(y)dy$. Рассмотрим последовательность независимых одинаково распределенных случайных величин \gav{$\{X_k\}_{k\ge1}$} из распределения $\mathrm{R}(0,1)$. Тогда последовательность $\{f(X_k)\}$ также будет состоять из независимых одинаково распределенных случайных величин и, в силу усиленного закона больших чисел,
\begin{equation}\label{Pseudo}
\frac{1}{N}\sum_{k=1}^{N}f(X_k) \xrightarrow[n\to\infty]{\text{п.н.}} \Exp f(X_1)  = \int\limits_{0}^{1}f(y)\,dy.
\end{equation}
Это значит, что среднее арифметическое чисел $f(X_k)$ может служить оценкой для искомого интеграла и чем выше $N$, тем ближе (на множестве исходов единичной меры) будет значение оценки к истинному значению интеграла.

\gav{Получить на практике по настоящему независимые одинаково распределенные на отрезке $[0,1]$ случайные величины, в свою очередь, представляет собой сложную задачу. Поэтому, как правило, используют <<псевдослучайные последовательности>>. В качестве такой последовательности можно взять стационарную последовательность $\{X_k\}$ из примера \ref{ex:MC}. Согласно эргодической теореме Биркгофа--Хинчина для почти всех точек старта $\omega \in [0,1]$ будет иметь место формула \eqref{Pseudo}, в которой сходимость почти наверное по равномерной мере $\mu$ как раз и понимается с точки зрения выбора $\omega \in [0,1]$. Впрочем, для сдвига окружности эти оговорки несущественны, см. следующий пример.}
\EndEx
\end{example}


\begin{example}[ (задача Вейля)]
Рассмотрим последовательность степеней двойки $2^k, k\in \Nbb$. Зададимся вопросом: как часто встречается цифра $m\in\{1,\dots,9\}$ в качестве первой цифры степени двойки? 

Пусть $a_k$ -- первая цифра числа $2^k$, тогда найдется целое $r$, такое, что $2^k=a_k\cdot 10^r + b_k$, где $b_k<10^r$. Отсюда следует, что $$\log_{10} 2^k=\log_{10}a_k+r+\log_{10}\left(1+\frac{b_k}{a_k 10^r}\right),$$ и, следовательно, $$\{\log_{10}2^k\}=\left\{\log_{10}a_k+\log_{10}\left(1+\frac{b_k}{a_k10^r}\right)\right\},$$ где фигурные скобки обозначают дробную часть числа. Заметим, что выражение, которое стоит справа под фигурными скобками, не превышает единицы, поэтому фигурные скобки можно убрать:
$$\{\log_{10}2^k\}=\log_{10}a_k+\log_{10}\left(1+\frac{b_k}{a_k10^r}\right).$$

Пусть $a_k=m$, тогда $\log_{10}a_k=\log_{10}m$, а $$\log_{10}\left(1+\frac{b_k}{a_k10^r}\right)\in\left[ 0, \log_{10}\left(1+\frac{1}{m}\right) \right),$$ т.к. $b_k$ лежит в интервале $[0,10^r)$. Это значит, что если ${a_k=m}$, то ${\{\log_{10}2^k\}\in[\log_{10}m,\log_{10}(m+1))}$. Нетрудно установить и обратное: из ${\{\log_{10}2^k\}\in[\log_{10}m,\log_{10}(m+1))}$ следует ${a_k=m}$.

Пусть $\Omega = S^1$~--- окружность единичной длины на плоскости; как и прежде, будем пользоваться ее гомеоморфностью отрезку $[0,1]$, концы которого отождествлены. Пусть ${\F = \B}$~--- борелевская $\sigma$-алгебра на $S^1$, $\mu$~--- классическая мера Лебега на $S^1$, ${T(x) = \{\log_{10}2+x\}}$. Как следует из теоремы \ref{th:CircErgodic}, $M=(\Omega,\F,\mu,T)$ является эргодической динамической системой. Введем квадратично интегрируемую функцию 
$$f(x)=\left\{ {\begin{array}{*{20}{c}}
  {1,}&{x\in[\log_{10}m,\log_{10}(m+1)),} \\ 
  {0,}&{x\notin [\log_{10}m,\log_{10}(m+1)).} 
\end{array}} \right.$$
Тогда согласно теореме Биркгофа\gav{--}Хинчина: $$
\frac{1}{N}\sum_{k=1}^{N}f(T^kx) \xrightarrow[N\to\infty]{\mu\text{-п.н.}} \int\limits_{\Omega}f(\omega)\,d\mu(\omega)=\log_{10}\left(1+\frac{1}{m}\right).$$ Заметим, что теорема устанавливает лишь сходимость почти наверное. Из нее не следует сходимость, например, при $x=0$. Однако для сдвигов по окружности верен и более сильный результат:
$$
\frac{1}{N}\sum_{k=1}^{N}f(T^kx) \xrightarrow[N\to\infty]{\text{всюду}} \int\limits_{\Omega}f(\omega)\,d\mu(\omega)=\log_{10}\left(1+\frac{1}{m}\right).$$
Этот факт, как и его доказательство, можно найти в книге \cite{SinaiLections1996} (см.~лек\-цию~3). Значит, сходимость имеет место в том числе и для ${x=0}$. Но~$T^k0=\{\log_{10}2^k\}$, поэтому сумма $\sum_{k=1}^{N}f(T^kx)$ представляет собой не что иное, как количество степеней двойки, начинающихся с цифры $m$. Получается, что в пределе доля единиц среди первых цифр будет равна $\log_{10}2$, доля двоек будет равна $\log_{10}(3/2)$, доля троек -- $\log_{10}(4/3)$, и т.д. \EndEx
\end{example}


\begin{example}[ (Гаусс--Гильден--Виман--Кузьмин)]
Рассмотрим цепную дробь для числа $\omega\in[0,1)$:
$$
\omega = \cfrac{1}{a_1(\omega)+\cfrac{1}{a_2(\omega)+\cfrac{1}{a_3(\omega)+\ldots}}},
$$
где $(a_1(\omega),a_2(\omega),\ldots)$ будем называть \textit{цепным разложением} числа $\omega$, а $a_k(\omega)$~-- \textit{$k$-м основанием числа $\omega$}. Известно, что если $\omega$~-- рациональное число, то его цепное разложение конечно. Если $\omega$~-- квадратичная иррациональность, т.е. ${\omega = \sqrt{\alpha}}$ для рационального $\alpha$, то цепное разложение числа $\omega$ периодично. Оказывается, что для любого заданного натурального числа $m$ для почти всех (по мере Лебега на $[0,1]$) чисел на $[0,1)$ частота встречаемости $m$ в цепном разложении будет одной и той же.

Для того чтобы это показать, рассмотрим динамическую систему ${M=(\Omega,\B,\mu,T)}$, где ${\Omega=[0,1)}$, $\B$ -- борелевская $\sigma$-алгебра на $[0,1)$, мера $\mu$ определяется как $$d\mu(x)=\frac{1}{\ln 2}\frac{dx}{1+x},$$ а преобразование $T(\omega)=\{1/\omega\}$, $\omega\in\Omega$.

Мы опускаем доказательство того, что эта система является динамической, как и то, что она является эргодической.

Теперь заметим, что отображение $T$ удаляет первый элемент цепного разложения, т.е. если
$$
\omega = \cfrac{1}{a_1(\omega)+\cfrac{1}{a_2(\omega)+\cfrac{1}{a_3(\omega)+\ldots}}},
$$
то
$$
T(\omega) = \cfrac{1}{a_2(\omega)+\cfrac{1}{a_3(\omega)+\cfrac{1}{a_4(\omega)+\ldots}}}.
$$
Кроме того, легко видеть, что $$a_1(\omega) = m \Leftrightarrow \omega\in\left(\frac{1}{m+1},\frac{1}{m}\right].$$ Это значит, что $$a_k(\omega)=m \Leftrightarrow T^k\omega \in \left(\frac{1}{m+1},\frac{1}{m}\right].$$ Остается рассмотреть квадратично интегрируемую функцию: $$f(x)=\left\{ {\begin{array}{*{20}{c}}
  {1,}&{x\in(1/(m+1),1/m],} \\ 
  {0,}&{x\notin(1/(m+1),1/m],} 
\end{array}} \right.$$ и тогда по теореме Биркгофа--Хинчина
$$
	\frac{1}{N}\sum_{k=1}^{N}f(T^k\omega) \xrightarrow[n\to\infty]{\mu-\text{п.н.}} \int\limits_{0}^{1}f(x)\,d\mu(x) =
	$$
	$$
	= \int\limits_{\frac{1}{m+1}}^{\frac{1}{m}}\frac{1}{\ln2}\cdot\frac{dx}{1+x} = \log_2\left(1+ \frac{1}{m(m+2)}\right).
	$$
	Это и значит, что для почти всех $\omega\in[0,1)$ частота встречаемости числа $m\in\Nbb$ в цепном разложении $\omega$ не зависит от $\omega$ и равна
	$$
	\log_2\left(1+ \frac{1}{m(m+2)}\right). \text{ \EndEx}
	$$
\end{example}

\section{Дискретные цепи Маркова}
\label{sec:DiscreteChains}
\setlength{\parskip}{0pt}
\subsection{Базовые определения и примеры}

Рассмотрим\,случайную\,последовательность\,$\{\xi_k\}\!=\!\{\xi_0,\xi_1,\dots\}$,\,ком\-поненты которой могут принимать лишь конечное или счетное множество значений, например: $$E=\{0,1,2,\dots\}, \ |E|\le\infty.$$ 
Такие случайные последовательности будем называть \textit{дискретными цепями} (\textit{ДЦ}), множество $E$ -- \textit{множеством состояний}, а его элементы -- \textit{состояниями} (дискретными, так как время дискретно; цепями, так как множество состояний тоже дискретно). Если множество состояний конечно, т.е. ${|E|<\infty}$, то цепь называют \textit{конечной}. Термин \textit{цепь} обосновывает дискретность множества состояний, а прилагательное \textit{дискретная} цепь относится ко времени (время тоже принимает дискретные значения).

\begin{definition}\label{def:DiscrMarkChain}
Дискретная цепь $\{\xi_k\}$ называется \textit{дискретной цепью Маркова} или \textit{дискретная марковская цепь} (\textit{ДМЦ}), если равенство
\begin{equation}\label{eq:MarkovPropertyChain}
    \mathbb{P}(\xi_n=x_n \,|\, \xi_{n-1}=x_{n-1},\dots,\xi_0=x_0) = \mathbb{P}(\xi_n=x_{n}\,|\,\xi_{n-1}=x_{n-1})
\end{equation}
выполнено для любых ${n\ge1}$ и всех состояний $x_0,\dots,x_n\in E$, для которых указанные условные вероятности существуют.

Свойство~\eqref{eq:MarkovPropertyChain} также называют \textit{марковским свойством}. Случайная величина $\xi_n$ характеризует состояние цепи в момент времени ${n\ge0}$, поэтому событие $\{\xi_n=j\}$ читается как <<на шаге $n$ цепь находится в состоянии $j\in E$>>.
\end{definition}

\begin{definition}
    Число ${p_{k,j}(n)=\mathbb{P}(\xi_n=j \,|\, \xi_{n-1}=k)}$ называется \textit{вероятностью перехода} из состояния $k\in E$ в состояние $j\in E$ за один шаг в момент $n\ge1$. Числа $p_{k,j}(n)$ образуют \textit{переходную матрицу} $$P(n)=\left[ {\begin{array}{*{20}{c}}
	{{p_{0,0}(n)}}&{{p_{0,1}(n)}}& \cdots &{{p_{0,N}(n)}} \\ 
	{{p_{1,0}(n)}}&{}&{}&{} \\ 
	\vdots &{}& \ddots &{} \\ 
	{{p_{N,0}(n)}}&{}&{}&{{p_{N,N}(n)}} 
	\end{array}} \right],$$ где ${N=|E|-1\le\infty}$. Сумма компонент произвольной переходной матрицы в любой строке равна единице, т.е. $$\forall k\in E \ \ \sum\limits_{j=0}^N p_{k,j}=1.$$
\end{definition}

\begin{definition}
    Если $P(n)$ не зависит от $n$, то соответствующую цепь называют \textit{однородной}. Для краткости будем писать ОДМЦ.
\end{definition}


\begin{definition}
    Вероятность ${p_k(n)=\mathbb{P}(\xi_n=k)}$, ${k\in E}$, называется \textit{вероятностью состояния} $k$ в момент ${n\ge 0}$. Вектор $$p(n)=[p_0(n), p_1(n),\dots]^\top$$ называют \textit{распределением вероятностей состояний} в момент $n\ge0$.
\end{definition}

Следующие две теоремы являются прямыми следствиями формулы полной вероятности и марковского свойства цепи.

\begin{theorem}\label{thm:kom_capm_1}
\textit{Распределение вероятностей состояний $p(n)$ связано с распределением вероятностей состояний $p(n-1)$ соотношением $$p(n)=P^\top(n) p(n-1),$$ где $P(n)$ -- матрица перехода на шаге $n\ge1$. Если цепь однородная, то }$$p(n)=\left(P^\top\right)^n p(0).$$
\end{theorem}

\begin{proof}
    Выразим вероятность $p_i(n)$ оказаться в состоянии $i$ на $n$-м шаге через $p(n-1)$ и переходную матрицу в момент времени $n$. По формуле полной вероятности
    \[
        p_i(n) = \sum\limits_{j\in E} \PP(\xi_n = i\,\mid\, \xi_{n-1}=j)\PP(\xi_{n-1}=j) = \sum\limits_{j\in E} p_{j,i}(n) p_j(n-1).
    \]
    Если записать равенство выше в матричной форме, то получим
    \[
        p(n) = P^\top(n) p(n-1),
    \]
    что и требовалось доказать. В случае, когда марковская цепь однородна, имеем $P(n) = P$ и $p(n) = P^\top p(n-1) = \left(P^\top\right)^2 p(n-2) = \ldots =$\linebreak $= \left(P^\top\right)^n p(0)$. \EndProof
\end{proof}

Введем следующее обозначение.
Вероятность перехода за $n$ шагов из состояния ${k\in E}$ в состояние ${j\in E}$ для однородной марковской цепи обозначим символом $$p_{k,j}(n)=\mathbb{P}(\xi_n=j \,|\, \xi_0=k), \ n\ge \gav{1},$$ где $E$ -- множество состояний.

\begin{theorem}[(уравнение\,Колмогорова--Ч\gav{э}пмена)]\textit{Для\,ОДМЦ\linebreak} (\textit{однородной дискретной марковской цепи}) \textit{для любых $n,m\ge \gav{1}$ справедливо равенство}
    $$p_{i,j}(n+m) = \sum\limits_{k\in E} p_{i,k}(n) p_{k,j}(m).$$
\end{theorem}
\begin{proof}
    Из теоремы~\ref{thm:kom_capm_1} следует, что $p_{k,j}(n)$~--- это элемент матрицы $P^n$, стоящий в $k$-й строке и $j$-м столбце. Если записать равенство $P^{n+m} = P^{n}P^{m}$ покомпонентно, то получим
    \[
        p_{i,j}(n) = \sum\limits_{k\in E} p_{i,k}(n) p_{k,j}(m).
    \]
    Это же равенство можно доказать через формулу полной вероятности. Действительно, так как любая траектория из $n+m$ шагов из $i$-го состояния в $j$-е через $n$ шагов окажется в некотором состоянии $k$, а потом из этого состояния через $m$ шагов попадет в $j$, то достаточно рассмотреть все возможные состояния $k$ (воспользоваться формулой полной вероятности):
\begin{equation*} 
\begin{split}
    p_{i,j}(n+m) &= \PP(\xi_{n+m} = j \,\mid\, \xi_0 = i)= \\ 
    &= \sum\limits_{k\in E}\PP(\xi_{n+m} = j \,\mid\, \xi_n = k, \xi_0 = i) \PP(\xi_{n} = k \,\mid\, \xi_0 = i) =\\ 
    &= \sum\limits_{k\in E} \underbrace{\PP(\xi_{n+m} = j \,\mid\, \xi_n =k)}_{p_{k,j}(m)} \underbrace{\PP(\xi_{n} = k \,\mid\, \xi_0 = i)}_{p_{i,k}(n)}.
\end{split}
\end{equation*}
    Последнее равенство следует из марковского свойства цепи. \EndProof
\end{proof}

Дискретные цепи удобно представлять в виде \textit{стохастического графа}.
Это ориентированный граф, в вершинах которого расположены состояния цепи, а веса ребер между состояниями равны вероятностям перехода за один шаг между этими состояниями. Рассмотрим, к примеру, стохастический граф

\begin{figure}[!h]
	\centering
	\includegraphics[scale=0.50]{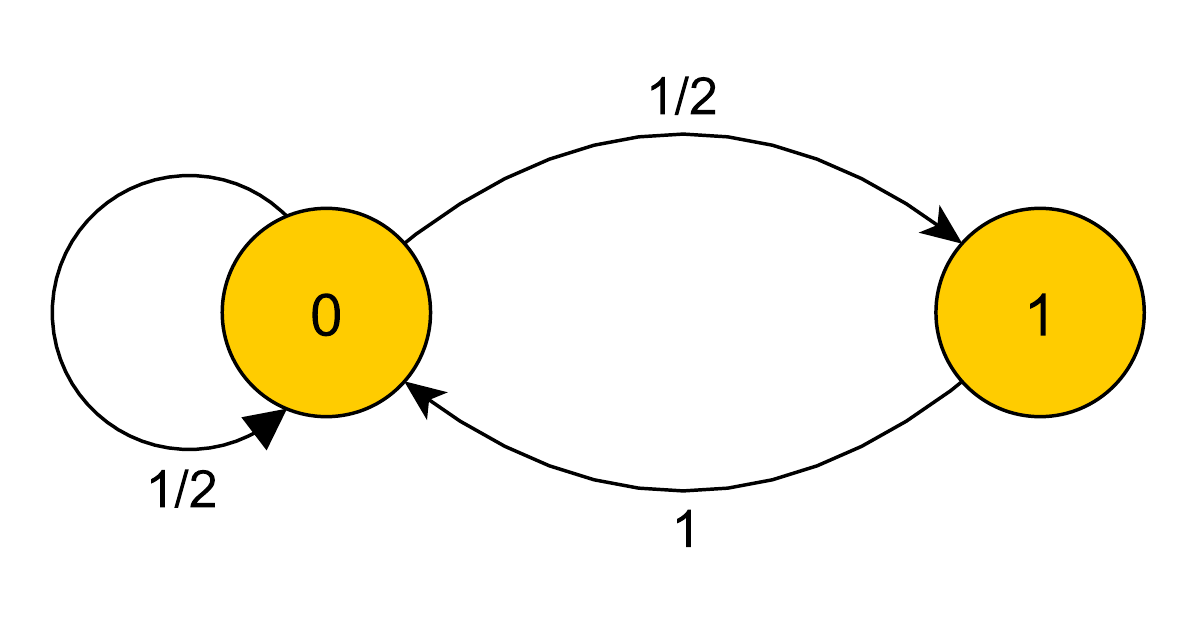}  
	\label{fig:DiscreteChainGraphExample}
\end{figure}

\noindent Ему соответствует матрица переходов $$P=\left[ {\begin{array}{*{20}{c}}
	{1/2}&{1/2} \\
	{1}&{0} \\ 
	\end{array}} \right].$$
Пусть начальное распределение ${p(0)=[0,1]^\top}$. Это значит, что в момент времени ${n=0}$ цепь находится в состоянии $0$ с вероятностью $0$ и в состоянии $1$ с вероятностью $1$. Найдем распределение на следующем шаге: $$p(1)=P^\top p(0)=\left[ {\begin{array}{*{20}{c}}
	{1/2}&{1} \\
	{1/2}&{0} \\ 
	\end{array}} \right]\left[ {\begin{array}{*{20}{c}}
	{0} \\
	{1} \\ 
	\end{array}} \right]=\left[ {\begin{array}{*{20}{c}}
	{1} \\
	{0} \\ 
	\end{array}} \right].$$
Еще через шаг распределение станет таким:
$$p(2)=P^\top p(1)=\left[ {\begin{array}{*{20}{c}}
	{1/2}&{1} \\
	{1/2}&{0} \\ 
	\end{array}} \right]\left[ {\begin{array}{*{20}{c}}
	{1} \\
	{0} \\ 
	\end{array}} \right]=\left[ {\begin{array}{*{20}{c}}
	{1/2} \\
	{1/2} \\ 
	\end{array}} \right].$$
А еще через шаг таким:
$$p(3)=P^\top p(2)=\left[ {\begin{array}{*{20}{c}}
	{1/2}&{1} \\
	{1/2}&{0} \\ 
	\end{array}} \right]\left[ {\begin{array}{*{20}{c}}
	{1/2} \\
	{1/2} \\ 
	\end{array}} \right]=\left[ {\begin{array}{*{20}{c}}
	{3/4} \\
	{1/4} \\ 
	\end{array}} \right].$$
Попробуем теперь вычислить распределение \gav{в} произвольный момент времени, т.е. $\pi(n)$, $n\ge0$. Для этого воспользуемся следующим соображением: если матрица $A$ может быть представлена в виде ${A=SDS^{-1}}$, где $D$ -- диагональная матрица, то ${A^n=SD^nS^{-1}}$. Иногда это удается сделать, решив задачу на собственные числа и векторы матрицы $A$. Так как в нашем случае $$p(n)=\left(P^\top\right)^n p(0),$$ то попробуем найти собственные числа и собственные векторы матрицы $P^\top$: $$\left|  {\begin{array}{*{20}{c}}
	{1/2-\lambda}&{1} \\
	{1/2}&{-\lambda} \\ 
	\end{array}}  \right|=\lambda^2 - \frac{1}{2}\lambda - \frac{1}{2}=0 \Rightarrow \lambda_1=1, \ \lambda_2=-\frac{1}{2}.$$
Стоит отметить, что у  стохастической матрицы единица всегда является собственным значением (см. далее). Еще одно собственное значение можно найти из следа матрицы.

При $\lambda=1$ получаем $$P^\top-\lambda I=\left[ {\begin{array}{*{20}{r}}
	{-1/2}&{1} \\
	{1/2}&{-1} \\ 
	\end{array}} \right],$$
откуда собственный вектор $v_1=[2,1]^\top$. При $\lambda=-1/2$ получаем $$P^\top-\lambda I=\left[ {\begin{array}{*{20}{c}}
	{1}&{1} \\
	{1/2}&{1/2} \\ 
	\end{array}} \right],$$
откуда собственный вектор $v_2=[1,-1]^\top$. Значит, в нашем случае $$P^\top=\underbrace{\left[ {\begin{array}{*{20}{r}}
		{2}&{1} \\
		{1}&{-1} \\ 
		\end{array}} \right]}_{S}\underbrace{\left[ {\begin{array}{*{20}{c}}
	{1}&{0} \\
	{0}&{-1/2} \\ 
	\end{array}} \right]}_{D}\underbrace{\left[ {\begin{array}{*{20}{r}}
	{2}&{1} \\
	{1}&{-1} \\ 
	\end{array}} \right]^{-1}}_{S^{-1}}.$$ Далее 
$$\left(P^\top\right)^n=\left[ {\begin{array}{*{20}{r}}
		{2}&{1} \\
		{1}&{-1} \\ 
		\end{array}} \right]\left[ {\begin{array}{*{20}{c}}
		{1}&{0} \\
		{0}&{(-1/2)^n} \\ 
		\end{array}} \right]\left[ {\begin{array}{*{20}{r}}
		{2}&{1} \\
		{1}&{-1} \\ 
		\end{array}} \right]^{-1}=$$$$=\frac{1}{3}\left[ {\begin{array}{*{20}{r}}
		{2+(-1/2)^n}&{2-2(-1/2)^n} \\
		{1-(-1/2)^n}&{1+2(-1/2)^n} \\ 
		\end{array}} \right].$$
Итак, мы получили, что для $p(0)=[0,1]^\top$ распределение на шаге $n\ge0$ $$p(n)=\frac{1}{3}\left[ {\begin{array}{*{20}{r}}
	{2-2(-1/2)^n} \\
	{1+2(-1/2)^n} \\ 
	\end{array}} \right].$$
Заметим, что при $n\to\infty$ $$\left(P^\top\right)^n \to \left[ {\begin{array}{*{20}{r}}
	{2/3}&{2/3} \\
	{1/3}&{1/3} \\ 
	\end{array}} \right].$$
Поэтому, если взять произвольное начальное распределение $p(0)=$\linebreak $=[p_0(0),p_1(0)]$, то в пределе $$p(n)\to \left[ {\begin{array}{*{20}{r}}
	{2/3}&{2/3} \\
	{1/3}&{1/3} \\ 
	\end{array}} \right]\left[ {\begin{array}{*{20}{r}}
	{\pi_0(0)} \\
	{\pi_1(0)} \\ 
	\end{array}} \right]=\left[ {\begin{array}{*{20}{r}}
	{2/3} \\
	{1/3} \\ 
	\end{array}} \right],$$ так как $p_0(0) + p_1(0)=1$. Это значит, что в каком бы состоянии не находилась цепь изначально и каким бы ни было начальное распределение вероятностей состояний, в пределе вероятность оказаться в состоянии 0 равна $2/3$, а в состоянии 1 она равна $1/3$. Это проявление \textit{эргодичности} цепи, строгое определение и свойства будут дан\gav{ы} позднее.

Теперь несколько слов о матрице перехода $P$. Как уже было сказано, сумма компонент в каждой ее строке равна единице, и каждая ее компонента неотрицательна и не превышает единицу. Матрицы с такими свойствами называются \textit{стохастическими}. Общих свойств таких матриц крайне мало, вот одно из них \cite{KelbertSukhov2010}, \cite{Nikaydo}.

\begin{theorem}
\label{eigen value 1}
\textit{Матрицы $P$ и $P^\top$ всегда имеют собственное значение, равное $1$. Остальные собственные значения по модулю не превосходят $1.$}
\end{theorem}




\begin{proof}
Так как сумма компонент матрицы $P$ в каждой ее строке равна 1, то достаточно взять вектор $1=[1,\dots,1]^\top$ и увидеть, что $$P1=1.$$ Значит, матрица $P$ имеет собственное значение $\lambda=1$, а соответствующий собственный вектор равен $1$. Отсюда следует, что и матрица $P^\top$ имеет собственное значение $\lambda=1$, так как детерминанты произвольной матрицы и ее транспонированной совпадают: $$\det(P-\lambda I)=\det(P - \lambda I)^\top=\det(P^\top - \lambda I) = 0.$$ То есть характеристические многочлены совпадают, а значит, и набор собственных значений матриц $P$ и $P^\top$ совпадает. Отметим, что вектор $1$ не обязан быть собственным вектором матрицы $P^\top$. Теперь предположим, что нашлось собственное значение $\lambda$ матрицы $P^\top$ с ${|\lambda|>1}$, а $v$ -- соответствующий собственный вектор. Но тогда ${\left(P^\top\right)^nv=\lambda^nv}$ и правая часть растет экспоненциально с ростом $n$. Следовательно, найдется элемент матрицы $\left(P^\top\right)^n$, который превышает 1, но это невозможно, так как матрица $\left(P^\top\right)^n=\left(P^n\right)^\top$ равна транспонированной стохастической матрице\footnote{Если $P$ -- стохастическая матрица, то и $P^n$ -- стохастическая матрица. Матрица $P^n$ есть \textit{матрица перехода за} $n$ \textit{шагов} и состоит из соответствующих условных вероятностей.}, элементы которой не могут превышать 1. Мы пришли к противоречию, значит, для всех собственных значений $\lambda$ выполнено $|\lambda|\le 1$. \EndProof
\end{proof}

Попытки простым способом описать общие свойства матриц $P^\top$ встречают серьезное сопротивление. Это подтверждают следующие примеры.
\begin{enumerate}[topsep=-5pt,itemsep=5pt]
	\item Матрица $P^\top$ может быть вырожденной: $$P^\top=\left[ {\begin{array}{*{20}{r}}
		{0}&{0} \\
		{1}&{1} \\ 
		\end{array}} \right].$$
	\item Матрица $P^\top$ может иметь отрицательные собственные значения и отрицательный детерминант: $$P^\top=\left[ {\begin{array}{*{20}{r}}
		{0}&{1} \\
		{1}&{0} \\ 
		\end{array}} \right].$$
	\item Матрица $P^\top$ может иметь комплексные собственные значения и быть ортогональной $(P^\top=P^{-1})$: $$P^\top=\left[ {\begin{array}{*{20}{r}}
		{0}&{1}&{0} \\
		{0}&{0}&{1} \\
		{1}&{0}&{0} 
		\end{array}} \right].$$
	\item Кратность собственного значения $\lambda=1$ может быть больше единицы:$$P^\top=\left[ {\begin{array}{*{20}{r}}
		{1}&{0} \\
		{0}&{1} \\ 
		\end{array}} \right].$$
	\item Не всякая матрица $P^\top$ диагонализуема, например, $$P^\top=\left[ {\begin{array}{*{20}{r}}
		{5/12}&{1/4}&{1/3} \\
		{5/12}&{1/4}&{1/3} \\
		{1/6}&{1/2}&{1/3} 
		\end{array}} \right].$$
\end{enumerate}

\subsection{Классификация состояний}

Перейдем теперь к классификации состояний однородной дискретной цепи Маркова. Напомним, что вероятность перехода за $n$ шагов из состояния ${k\in E}$ в состояние ${j\in E}$ мы обозначили символом $$p_{k,j}(n) = \mathbb{P}(\xi_n=j \,|\, \xi_0=k), \ n\ge 0,$$ где $E$ -- множество состояний. Введем также обозначение для вероятности первого возвращения за $n$ шагов в состояние $k\in E$: $$f_k(n)=\mathbb{P}(\xi_n=k,\xi_{n-1}\ne k,\dots,\xi_1\ne k \,|\, \xi_0=k), \ n\ge 1,$$ и вероятность возвращения в $k\in E$ за конечное число шагов:
\begin{equation}
    \label{F_k}
    F_k=\sum\limits_{n=1}^{\infty}f_k(n).
\end{equation}

\textbf{Замечание}. Вероятности $p_{k,k}(n)$ и $f_k(n)$ не обязаны совпадать. Вероятность $f_k(n)$ -- это вероятность \textit{первого} возвращения из $k$ в $k$ (при условии, что цепь вышла из состояния $k$). Вероятность же $p_{k,k}(n)$ подсчитывается с учетом того, что на промежуточных шагах цепь может находиться в состоянии $k$. В общем случае 
$$p_{k,k}(n) = \sum\limits_{m=1}^n f_k(m)p_{k,k}(n-m) = f_k(n)+\sum\limits_{m=1}^{n-1} f_k(m)p_{k,k}(n-m) \ge f_k(n).$$

\subsubsection{Классификация состояний на основе арифметических свойств $p_{i,j}(n)$}

\begin{definition}
    Состояния $k,j\in E$ называются \textit{сообщающимися}, если $$\left(\exists m\ge1: \ p_{k,j}(m) > 0\right) \ \wedge \ \left(\exists n\ge1: \ p_{j,k}(n)>0\right).$$ Иначе состояния $k,j\in E$ называются \textit{несообщающимися}: $$\left(\forall m\ge1: \ p_{k,j}(m)=0\right) \ \vee \ \left(\forall n\ge1: \ p_{j,k}(n)=0\right).$$
\end{definition}

\begin{definition}
\label{nesucshest}
    Состояние $k\in E$ называется \textit{несущественным}, если $$\exists j\in E: \ \left( \exists m\ge1 \ p_{k,j}(m)>0 \right) \wedge \left( \forall n\ge1 \ p_{j,k}(n)=0 \right).$$
Иначе состояние $k\in E$ называется \textit{существенным}: $$\forall j\in E \ \left( \left( \forall m\ge1 \ p_{k,j}(m)=0 \right) \vee \left( \exists n\ge1 \ p_{j,k}(n)>0 \right) \right).$$
\end{definition}

\begin{definition}\label{def:per1}
    Пусть ${\exists n_0\ge 1:\,p_{k,k}(n_0)>0}$ и $d_k$ -- наибольший общий делитель чисел $$\{n\ge1: f_k(n)>0\}.$$ Тогда если ${d_k>1}$, то состояние $k$ называется \textit{периодическим} с периодом $d_k$. Если же $d_k=1$, то состояние $k$ называется \textit{апериодическим}.
\end{definition}

\begin{definition}\label{def:per2}
    Пусть ${\exists n_0\ge 1:\,p_{k,k}(n_0)>0}$ и $d_k$ -- наибольший общий делитель чисел $$\{n\ge1: p_{k,k}(n)>0\}.$$ Тогда если ${d_k>1}$, то состояние $k$ называется \textit{периодическим} с периодом $d_k$. Если же $d_k=1$, то состояние $k$ называется \textit{апериодическим}.
\end{definition}

\elena{Со схемой доказательства следующей теоремы можно познакомиться в абзаце после определения 1.6 на с.~361 в~\cite{Borovkov1999}.}
\begin{theorem}
Определения~\ref{def:per1} и~\ref{def:per2} эквивалентны.
\end{theorem}

\textbf{Замечание}. Если состояние несущественное, то найдется состояние, с которым оно не сообщается. Если все состояния цепи сообщающиеся, то все они существенные. В конечной цепи всегда есть существенное состояние. Оно может быть одно. В бесконечной цепи может не быть существенных состояний (см. пример~\ref{1} ниже).

\textbf{Замечание}. Понятие периодичности вводится только для состояний, для которых  ${\exists n_0\ge 1:\,p_{k,k}(n_0)>0}$. 

Введем следующие обозначения. Если для какого-то ${m\ge 1}$ вероятность $p_{i,j}(m)>0$, то будем писать  $i\to j$ (состояние $j$ \textit{следует} за $i$). Если для любого ${m\ge 1}$ вероятность ${p_{i,j}(m)=0}$, то будем писать ${i \not\to j}$. Кроме того, если ${i\to j}$ и ${j\to i}$, то будем писать ${i \leftrightarrow j}$. По определению, состояния $i$ и $j$ называются сообщающимися, если ${i \leftrightarrow j}$. Состояние $i$ несущественное, если $\exists j\in E:(i\to j)\wedge (j \not\to i)$.

{
	\renewcommand{\baselinestretch}{0.95}
	\selectfont
	
Отношение ${i \leftrightarrow j}$ является отношением эквивалентности (действительно, 1) ${i \leftrightarrow i}$, так как $p_{i,i}(0) = 1$, 2) из ${i \leftrightarrow j}$ следует ${j \leftrightarrow i}$, 3) если ${i \leftrightarrow j}$ и ${j \leftrightarrow k}$, то ${i \leftrightarrow k}$). Таким образом, все множество состояний разбивается на классы эквивалентности. При этом может оказаться, что некоторые классы эквивалентности будут замкнутые (то есть из них нельзя с положительной вероятностью выйти), а некоторые -- открытые (см.~рис.~\ref{fig:classes}). Состояния, принадлежащие открытым классам эквивалентности, будем называть \textit{несущественными}, а из замкнутых -- \textit{существенными} (см. ~другое определение~\ref{nesucshest}). В конечной цепи всегда найдется хотя бы один замкнутый класс эквивалентности. В случае счетной цепи (то есть с бесконечным числом состояний) может не оказаться замкнутого класса эквивалентности. Так, в примере~\ref{1} каждое состояние образует открытый класс эквивалентности (из одного состояния) и число классов бесконечно. Если замкнутый класс состоит из одного состояния, то такое состояние называют \textit{поглощающим}. Если есть только один класс эквивалентности (он автоматически замкнут), то цепь называют \textit{неразложимой} (см. другое определение~\ref{reducable} ниже).

}

\begin{figure}[h]
    \centering
    \includegraphics[width=0.7\textwidth]{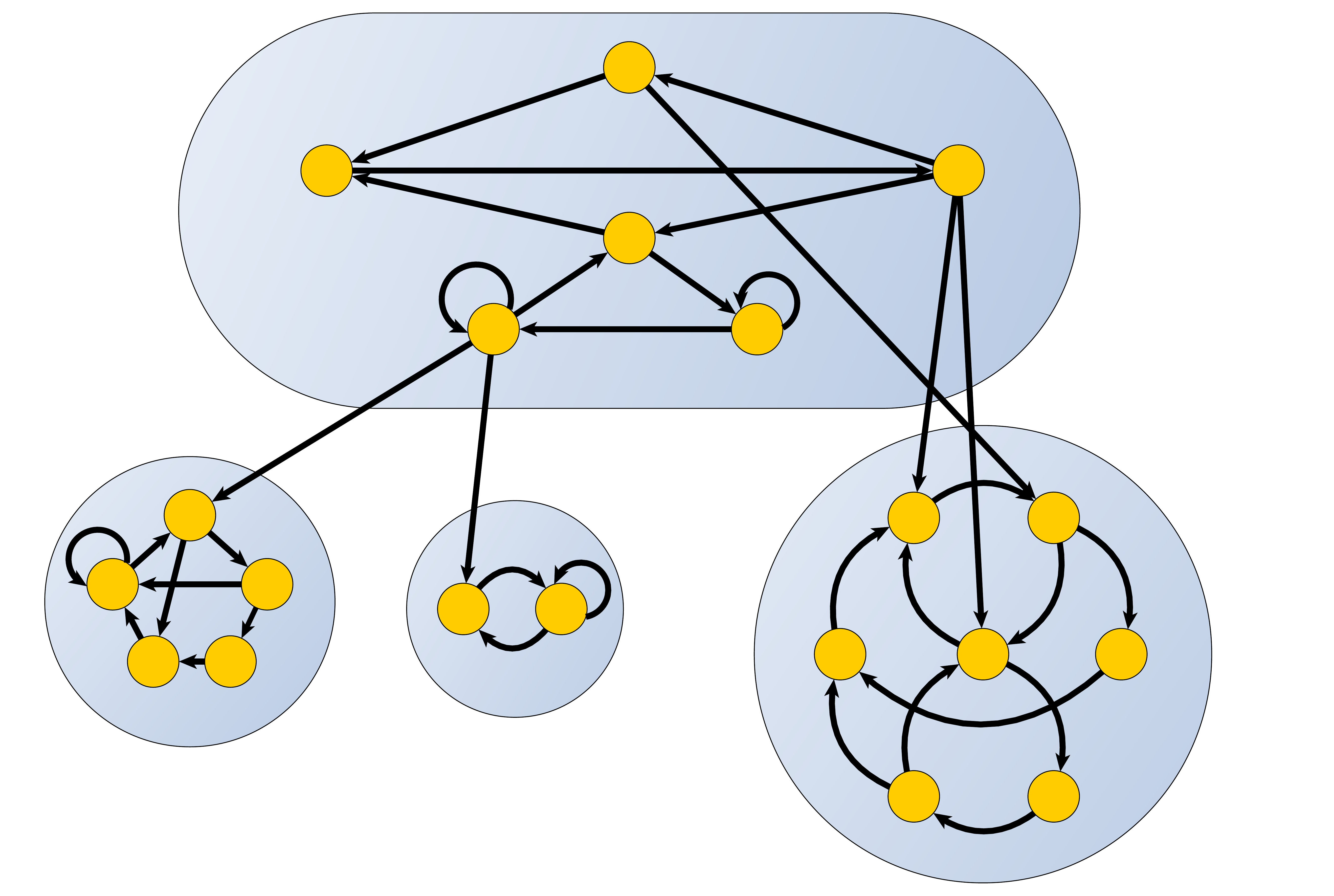}
    \caption{Граф конечной марковской цепи с одним открытым классом эквивалентности и тремя замкнутыми.}
    \label{fig:classes}
\end{figure}

{
\renewcommand{\baselinestretch}{0.95}
\selectfont

Классификация состояний внутри класса эквивалентности одинакова для всех состояний из этого класса (см.~далее \gav{теорему~\ref{solidar} о солидарности}).
	
Разбиение множества состояний на классы эквивалентности приводит к блочной структуре матрицы переходных вероятностей, если правильно занумеровать состояния (см.~рис.~\ref{fig:block_matrix}). На диагонали будут стоять квадратные блоки, отвечающие за вероятности перехода внутри замкнутых классов эквивалентности (если такие есть). Из замкнутости класса следует, что остальные вероятности перехода (вне диагонального блока) суть нули. Для состояний из открытых классов эквивалентности блочная структура немного сложнее: вне диагонали в~общем случае стоят не нули, а вероятности перехода в состояния из~других классов.

}

\begin{figure}[h]
    \centering
    \begin{blockmatrixtabular}
    \valignbox{\mblockmatrix{0.15in}{2.6in}{$P =$}}
    \valignbox{\leftparen{2.6in}}
    \end{blockmatrixtabular}
    \begin{blockmatrixtabular}
    \valignbox{
    \begin{blockmatrixtabular}
    \fblockmatrix     [1.0,1.0,0.7]{2.8in}{0.37in}{несущественные состояния}\\
    \fblockmatrix     {0.5in}{0.5in}{нули}
    \fblockmatrix     [0.8,0.8,1.0]{0.5in}{0.5in}{$P_1$}
    \fblockmatrix     {1.6in}{0.5in}{нули}\\
    \fblockmatrix     {1.1in}{0.3in}{нули}
    \fblockmatrix     [1.0,0.8,0.8]{0.3in}{0.3in}{$P_2$}
    \fblockmatrix     {1.2in}{0.3in}{нули}\\
    \fblockmatrix     {1.5in}{1.2in}{нули}
    \fblockmatrix     [0.8,1.0,0.8]{1.2in}{1.2in}{$P_3$}
    \end{blockmatrixtabular}}
    \end{blockmatrixtabular}
    \begin{blockmatrixtabular}
    \valignbox{\rightparen{2.6in}}
    \end{blockmatrixtabular}
    \caption{Блочная структура матрицы переходных вероятностей $P$ c тремя замкнутыми классами эквивалентности. Здесь $P_1,P_2,P_3$~--- \gav{неразложимые} \elena{квадратные матрицы, стоящие на диагонали общей матрицы $P$,} соответствующие вероятностям перехода внутри трех замкнутых классов, блок <<несущественные состояния>> соответствует либо вероятностям перехода из несущественных состояний в состояния из замкнутых классов эквивалентности, либо вероятностям остаться в этом открытом классе, блоки <<нули>>~--- матрицы правильных размеров, состоящие из одних нулей}
    \label{fig:block_matrix}
\end{figure}


Пусть имеется замкнутый класс эквивалентности с периодом $d>1$ и подмножеством состояний $E_\ast$.  Тогда существует $d$ групп состояний $C_0$, $C_1$, \ldots, $C_{d-1}$ (так называемых \textit{циклических подклассов}): $E_\ast = \cup_{j=0}^{d-1} C_j$, если $X_n \in C_m$, то для любого целого $l$ $X_{\elena{n+l}} \in C_{(m+l) \mod d}$, см. рис.~\ref{fig:cyclic_diagram}.

Если перенумеровать состояния, перечислив последовательно состояния из $C_0$, $C_1$, \ldots, $C_{d-1}$, то подматрица переходных вероятностей, соответствующая данному замкнутому классу состояний, будет иметь вид, как на рис.~\ref{fig:block_period_matrix}.

\textbf{Замечание.} Если такую матрицу возвести в степень $d$, то новая матрица $\tilde{P} = P^d$ будет блочно-диагональной. Если рассмотреть \mbox{ОДМЦ} с матрицей переходных вероятностей $\tilde{P}$, то она  является разложимой и имеет $d$ замкнутых апериодических классов эквивалентности $C_0$, $C_1$, \ldots, $C_{d-1}$. Таким образом, при исследовании асимптотического поведения $p_{i,j}(n)$ можно сразу предполагать, что цепь неразложима и апериодична.

\begin{figure}[h]
	\centering
	\includegraphics[width=0.4\textwidth]{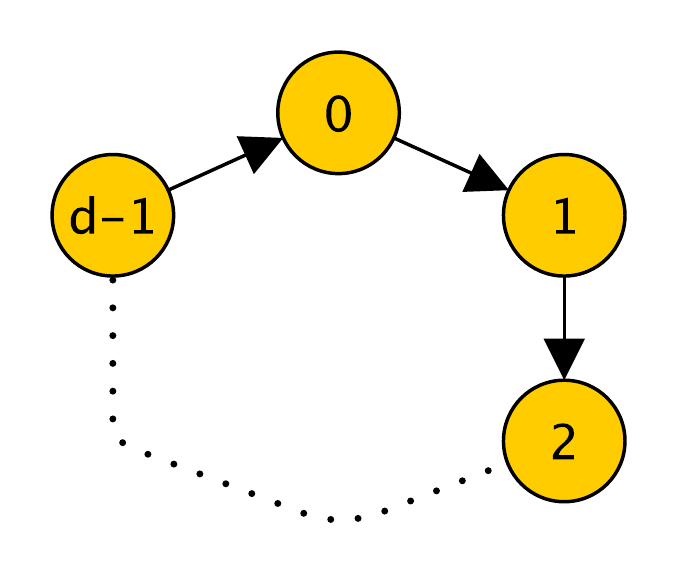} 
	\caption{Диаграмма движения по циклическим подклассам. На диаграмме $i$-я вершина соответствует подклассу $C_i$. Из вершин подкласса $C_0$ можно попасть только в вершины подкласса $C_1$, \dots, из вершин подкласса $C_{d-1}$ можно попасть только в вершины подкласса $C_0$}
	\label{fig:cyclic_diagram}
\end{figure}

\begin{figure}[h!]
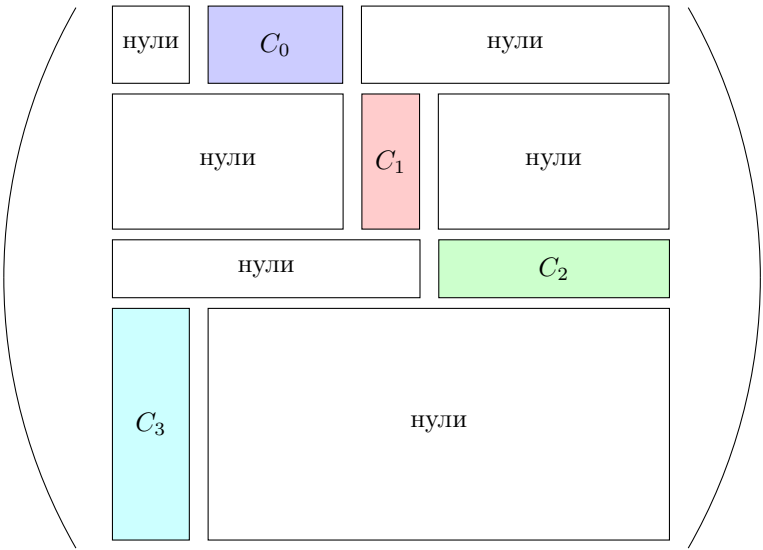

	\centering
	\begin{blockmatrixtabular}
		\valignbox{\leftparen{2.8in}}
	\end{blockmatrixtabular}
	\begin{blockmatrixtabular}
		\valignbox{
			\begin{blockmatrixtabular}
				\fblockmatrix     {0.4in}{0.4in}{нули}
				\fblockmatrix     [0.8,0.8,1.0]{0.7in}{0.4in}{$C_0$}
				\fblockmatrix     {1.6in}{0.4in}{нули}\\
				\fblockmatrix     {1.2in}{0.7in}{нули}
				\fblockmatrix     [1.0,0.8,0.8]{0.3in}{0.7in}{$C_1$}
				\fblockmatrix     {1.2in}{0.7in}{нули}\\
				\fblockmatrix     {1.6in}{0.3in}{нули}
				\fblockmatrix     [0.8,1.0,0.8]{1.2in}{0.3in}{$C_2$}\\
				\fblockmatrix     [0.8,1.0,1.0]{0.4in}{1.2in}{$C_3$}
				\fblockmatrix     {2.4in}{1.2in}{нули}
		\end{blockmatrixtabular}}
	\end{blockmatrixtabular}
	\begin{blockmatrixtabular}
		\valignbox{\rightparen{2.8in}}
	\end{blockmatrixtabular}
	\caption{Блочная структура подматрицы переходных вероятностей, соответствующей замкнутому классу эквивалентности с четырьмя циклическими подклассами. Здесь блок $C_i$ содержит вероятности переходов из вершин подкласса $C_i$ в вершины подкласса $C_{i+1 \mod 4}$, блоки <<нули>>~--- матрицы правильных размеров, состоящие из одних нулей}
	\label{fig:block_period_matrix}
\end{figure}

\clearpage

\subsubsection{Классификация состояний на основе асимптотических свойств $p_{i,i}(n)$}

\begin{definition}
    Состояние ${k\in E}$ называется \textit{нулевым}, если существует $$\lim\limits_{n\to\infty}p_{k,k}(n)=0.$$ Если указанный предел не существует или не равен нулю, то состояние $k$ называется \textit{ненулевым}.
\end{definition}

\begin{definition}\label{def:defrec1}
    Состояние $k\in E$ называется \textit{возвратным}, если вероятность возвращения за конечное число шагов $F_k$ (см.~\eqref{F_k}) равна единице. Иначе состояние $k$ называется \textit{невозвратным}.
\end{definition}

\begin{definition}\label{def:defrec2}
    Состояние $k\in E$ называется \textit{возвратным}, если 
$$\mathbb{P}\left( \bigcap\limits_{m=1}^{\infty}\bigcup\limits_{n\ge m}\{\xi_n = k\} \,\Big\rvert\, \xi_0 = k \right)= 1$$ Если же эта вероятность равна нулю, то состояние $k$ называется \textit{невозвратным}.
\end{definition}
Заметим, что $\bigcap_{m=1}^{\infty} \bigcup_{n\ge m}\{\xi_n = k\}$ означает, что цепь бесконечно часто возвращается в состояние $k$. Это событие из остаточной сигма-алгебры, а значит, по теореме Колмогорова <<закон нуля или единицы>> рассматриваемая вероятность не может принимать промежуточных значений.

\elena{Доказательство следующей теоремы можно найти в \cite{KaiLai1964} (следствие теоремы 4.3, с.~39).}
\begin{theorem}
\textit{Определения}~\ref{def:defrec1} \textit{и}~\ref{def:defrec2} \textit{эквивалентны.}
\end{theorem}

\begin{theorem}[ (критерий возвратности)]
\label{criteria transience}
\textit{Состояние $j\in E$ является возвратным тогда и только тогда, когда расходится ряд} $$P_j=\sum\limits_{n=1}^{\infty} p_{j,j}(n) = \infty.$$ \textit{При этом для невозвратного состояния} $$F_j=\frac{P_j}{1+P_j}.$$
\end{theorem}

\textbf{Замечание.} Данная теорема является обобщением леммы Бореля--Кантелли. Напомним, что если события $A_n$, $n\ge 1$ являются независимыми, то расходимость ряда $\sum_{n=1}^\infty P(A_n)$ эквивалентна тому, что события $A_n$ происходят бесконечно часто с вероятностью 1. Рассматриваемая теорема обобщает ситуацию, когда события $A_n$ <<слабо>> зависимы: если $j$ -- фиксированное состояние, то пусть $A_n$ означает, что на $n$-м шаге цепь возвращается в $j$-е состояние. Тогда расходимость ряда $\sum_{n=1}^\infty P(A_n) = \sum_{n=1}^{\infty}p_{j,j}(n)$ эквивалентна тому, что с вероятностью 1 цепь бесконечно часто возвращается в состояние $j$, то есть состояние $j$ -- возвратное.

\begin{proof}
    Из формулы полной вероятности и марковского свойства цепи получаем, что
    \[
        p_{j,j}(n) = \sum\limits_{k=0}^n f_j(n-k)p_{j,j}(k).
    \]
    Заметим, что данное выражение напоминает формулу для коэффициента многочлена, равного произведению двух многочленов. Развивая эту идею, рассмотрим два степенных ряда:
    \[
        V(z) = \sum\limits_{k=0}^\infty p_{j,j}(k) z^k,\quad U(z) = \sum\limits_{k=0}^\infty f_j(k)z^k.
    \]
    Тогда в силу того, что $f_j(0) = 0$ и $p_{j,j}(0)=1$, получим, что
    \[
        V(z) - 1 = U(z)V(z) \Longrightarrow V(z) = \frac{1}{1-U(z)}.
    \]
    Заметим, что $U(1) = F_j$~--- вероятность возвращения за конечное число шагов, а ряд $V(z)$ сходится при всех $|z|<1$ (при этом формально $V(1) = \sum\limits_{n=0}^\infty p_{j,j}(n) = 1 + P_j$). Переходя к пределу при $z\to 1 -0$ в полученном равенстве, приходим к следующему условию: 
    \[
        F_j = 1\text{ (состояние возвратно) } \Longleftrightarrow P_j = \infty.
    \]
    Если $F_j < 1$ (состояние невозвратно), то мы дополнительно доказали соотношение: $F_j = \frac{V(1)-1}{V(1)} = \frac{P_j}{1+P_j}$. \EndProof
\end{proof}

Из необходимого условия сходимости ряда получаем следующее следствие.

\textbf{Следствие}.  \textit{Если состояние является невозвратным, то оно нулевое. Если состояние ненулевое, оно является возвратным.}

Ниже мы вернемся к связям между нулевыми/ненулевыми и возвратными/невозвратными состояниями и увидим, что ситуации сильно отличаются в зависимости от конечности или счетности числа состояний.

\begin{theorem}
\textit{Несущественное состояние невозвратно.}
\end{theorem}
 С доказательством можно познакомиться в~\cite[теорема~4.4, с.~39]{KaiLai1964}.
\begin{definition}
\label{reducable}
    Цепь Маркова называется \textit{неразложимой}, если все ее состояния являются сообщающимися. В противном случае цепь называется \textit{разложимой}.
\end{definition}

\begin{theorem}[ (свойство солидарности)]
\label{solidar}
\textit{Для неразложимой\linebreak ОДМЦ справедливо, что}
\begin{enumerate}
	\item \textit{если хотя бы одно состояние возвратное, то все состояния возвратные,}
	\item \textit{если хотя бы одно состояние нулевое, то все состояния нулевые,}
	\item \textit{если хотя бы одно состояние имеет период ${d>1}$, то все остальные состояния периодичные с периодом $d$; если хотя бы одно состояние апериодично, то все состояния апериодичны.}
	\end{enumerate}
\end{theorem}
\begin{proof}
    \begin{enumerate}
        \item Пусть нашлось состояние $i$, являющееся возвратным. Из критерия возвратности следует, что это равносильно тому, что $\sum_{n=1}^\infty p_{i,i}(n)\!\!=$\linebreak $=\infty$. Рассмотрим произвольное состояние $j \neq i$ и покажем, что оно тоже является возвратным. Из неразложимости следует, что существуют такие числа $L$ и $M$, что $p_{j,i}(L) > 0$ и $p_{i,j}(M) > 0$. Поэтому 
        $$
        \sum\limits_{n=1}^\infty p_{j,j}(n) \ge \sum\limits_{n=1}^\infty p_{j,j}(L+n + M) \ge
        $$
        $$
        \ge\sum\limits_{n=1}^\infty p_{j,i}(L) p_{i,i}(n) p_{i,j}(M) = p_{j,i}(L) p_{i,j}(M) \sum\limits_{n=1}^\infty p_{i,i}(n) = \infty,
        $$
        а значит, $j$ тоже является возвратным.
        
        \item Рассмотрим два произвольных состояния $i,j\in S$. Из неразложимости следует, что существуют такие числа $L$ и $M$, что $p_{j,i}(L) > 0$ и $p_{i,j}(M) > 0$. Если состояние $j$~--- нулевое, то из неравенства
        \[
            p_{j,j}(L+n+ M) \ge p_{j,i}(L) p_{i,i}(n) p_{i,j}(M)
        \]
        следует, что $p_{i,i}(n) \xrightarrow[]{n\to\infty}0$, т.е. $i$ тоже является нулевым. Если же $j$ не является нулевым, то из неравенства
        \[
            p_{i,i}(M+n+L) \ge p_{i,j}(M) p_{j,j}(n) p_{j,i}(L)
        \]
        следует, что $p_{i,i}(M+n+L)$ не стремится к нулю при $n\to\infty$, т.е. $i$ тоже не является нулевым.
        
        \item Рассмотрим два произвольных состояния $i,j\in S$. Из неразложимости следует, что существуют такие числа $L$ и $M$, что $p_{j,i}(L) > 0$ и $p_{i,j}(M) > 0$. Пусть $\{n: p_{i,i}(n)>0\} = \{n_1^i,n_2^i,n_3^i,\ldots\}$ и $\{n: p_{j,j}(n)>$\linebreak $>0\} = \{n_1^j,n_2^j,n_3^j,\ldots\}$, причем НОД первого множества равен $d_i$, а~НОД второго~-- $d_j$. Заметим, что из неравенств $p_{i,i}(M+L) \ge$\linebreak $\ge p_{i,j}(M) p_{j,i}(L)> 0$ и $p_{j,j}(M+L) \ge p_{j,i}(L) p_{i,j}(M) > 0$ следует, что $M+L$ лежит в обоих множествах, а~значит, $M+L$ делится на $d_i$ и~на~$d_j$. Кроме того, для любого числа $n_k^i$ выполняется $p_{j,j}(L+n_k^i + M) \ge$\linebreak $\ge p_{j,i}(L) p_{i,i}(n_k^i) p_{i,j}(M) > 0$, а~значит, $n_k^i + M + L$ делится на $d_j$. Но,~поскольку $M+L$ тоже делится на $d_j$, заключаем, что и $n_k^i$ делится на~$d_j$, причем это верно для всех $k$. Отсюда следует, что для всех $k$ числа $n_k^i$ делятся на $d_j$, а~значит, $d_j \le d_i$, ведь $d_i$~--- наибольший общий делитель чисел $n_k^i$. Аналогичными рассуждениями можно показать, что для всех $k$ числа $n_k^j$ делятся на $d_i$, а значит, $d_i \le d_j$. В результате получили, что $d_j \le d_i$ и $d_i \le d_j$, т.е. $d_i = d_j$. Это означает, что периоды у всех состояний цепи одинаковы (в силу произвольности выбора $i$ и~$j$).
        \EndProof
    \end{enumerate}
\end{proof}

Отметим, что предыдущая теорема доказана без использования предположения о конечности цепи.

\begin{theorem}
\textit{В неразложимой конечной ОДМЦ все состояния ненулевые} (\textit{а значит, и возвратные}).
\end{theorem}
\begin{proof}
    Докажем теорему от противного. Пусть состояние $i$ является нулевым: ${\lim_{n\to\infty} p_{i,i}(n)= 0}$. Для произвольного состояния $j\neq i$ из условия неразложимости цепи найдется такое ${N = N(i,j)}$, что $p_{j,i}(N) > 0$. Тогда в следующей сумме из неотрицательных слагаемых оставим только одно, получив оценку снизу:
    \[
    p_{i,i}(n + N) = \sum_{k\in E} p_{i,k}(n) p_{k,i}(N) \ge p_{i,j}(n) p_{j,i}(N).
    \]
    С учетом нашего предположения $\lim_{n\to\infty} p_{i,i}(n)= 0$ и $p_{j,i}(N) > 0$ получаем, что $\lim_{n\to\infty} p_{i,j}(n)= 0$, или, иначе говоря, для любого $\epsilon > 0$ найдется $M_\epsilon = M_\epsilon(i,j)$, что для любого $n > M_\epsilon$ $p_{i,j} (n) < \epsilon$. Но тогда из условия конечности цепи существует $\bar M_\epsilon = \max_{j\in E} M_\epsilon(i,j)$, что для любого $n > \bar M_\epsilon$ $\sum_{j\in E} p_{i,j} (n) < \epsilon |E |$. Но последнее в силу произвольности выбора $\epsilon$ и $|E| < \infty$ противоречит нормировке $\sum_{j\in E} p_{i,j}(n) = 1$. Значит, наше предположение о нулевости некоторого состояния $i$ не верно. \EndProof
\end{proof}

\textbf{Следствие.} \textit{В замкнутом классе эквивалентности с  конечным числом состояний все состояния ненулевые и возвратные.}

Таким образом, мы получаем, что классификация состояний из \textbf{конечных} классов ОДМЦ делится только на \textbf{два} типа, а именно справедлива
\begin{theorem}
\label{finite_class}
\textit{Если $S_\ast$ -- множество состояний ОДМЦ, образующий конечный класс эквивалентности, тогда все состояния из }$S_\ast$
\begin{enumerate}
    \item \textit{либо нулевые и невозвратные, при этом $S_\ast$ является открытым классом} (\textit{состояния несущественные}).
    \item \textit{либо ненулевые и возвратные, при этом $S_\ast$ является замкнутым классом} (\textit{состояния существенные}).
\end{enumerate}
\end{theorem}

Иначе говоря, для \textbf{конечных} неразложимых подцепей
\begin{equation*}
    \text{нулевое состояние} \Longleftrightarrow \text{невозвратное} \Longleftrightarrow \text{несущественное}
\end{equation*}
\begin{equation*}
    \text{ненулевое состояние} \Longleftrightarrow \text{возвратное} \Longleftrightarrow \text{существенное}
\end{equation*}

\begin{example}\label{1}
Классифицировать состояния цепи, считая $p\in(0,1)$, $q\in(0,1)$.
\end{example}
\begin{figure}[!h]
	\centering
	\includegraphics[scale=0.40]{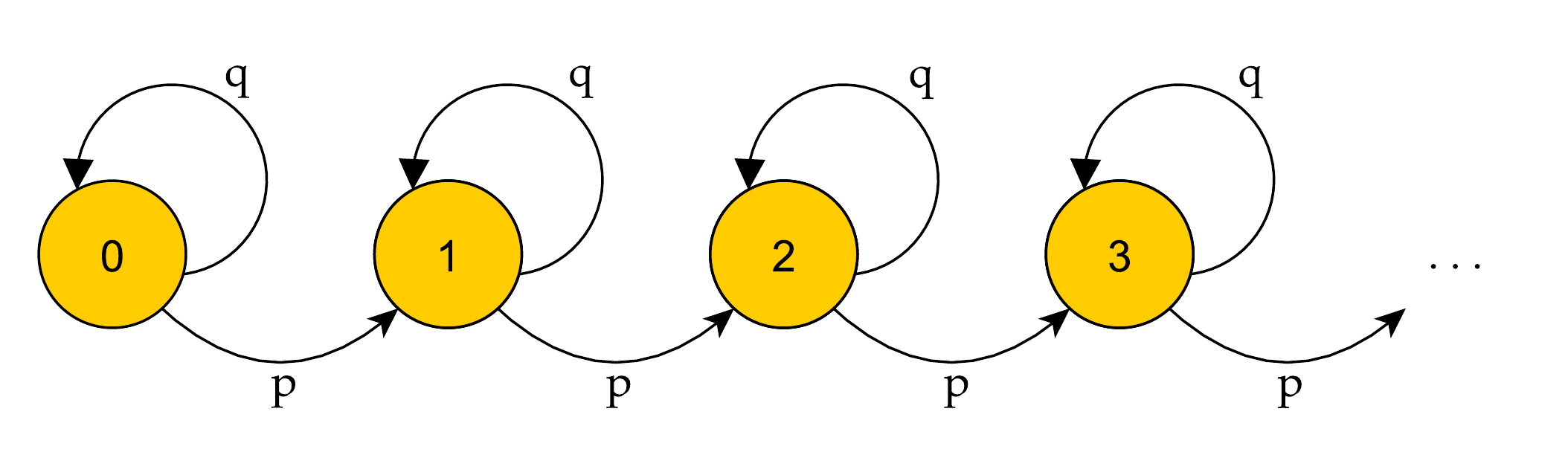}
	\caption{К примеру~\ref{1}}
	\label{fig:Problem9.1}
\end{figure}

\textbf{Решение}. Здесь ${E=\{0,1,2,\dots\}}$ -- счетное множество. Однако каждое состояние является открытым классом эквивалентности (при этом число классов счетно),  а значит, является нулевым и невозвратным. 

\elena{Данный пример демонстрирует, что в цепи может вовсе не быть существенных состояний (такая ситуация возможна, только если число состояний счетно).} \EndEx

В общей ситуации, когда класс имеет счетное число состояний, ситуация усложняется. Рассмотрим следующий
\begin{example}\label{randomWalk}
Классифицировать состояния цепи, считая $p\in(0,1)$, $q\in(0,1)$.
\begin{figure}[!h]
	\centering
	\includegraphics[scale=0.4]{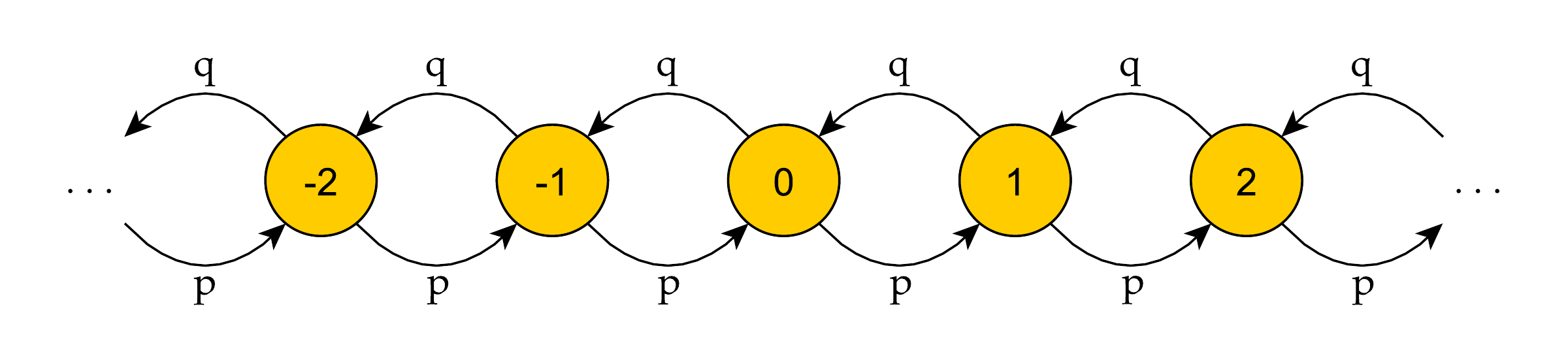}
	\caption{К примеру~\ref{randomWalk}}
	\label{fig:Problem92}
\end{figure}
\end{example}

\textbf{Решение}. Здесь ${E=\{\dots,-1,0,1,\dots\}}$. Все состояния цепи являются сообщающимися. Поэтому все они являются существенными. Возвращение в любое состояние возможно лишь за четное число шагов, следовательно ${\forall j\in E}$ период ${d_j=2}$, поэтому все состояния являются периодическими с периодом 2.

Теперь выясним, являются ли состояния нулевыми. Для любого $j\in E$: $$p_{j,j}(2n)=C_{2n}^n\,p^n\,q^n=\frac{(2n)!}{(n!)^2}\,p^n\,q^n.$$ По формуле Стирлинга, $n! \sim n^ne^{-n}\sqrt{2\pi n}$, $n\to\infty$. Отсюда $$p_{j,j}(2n) \sim \frac{(2n)^{2n}e^{-2n}\sqrt{2\pi\cdot 2n}}{\left(n^n e^{-n} \sqrt{2\pi n}\right)^2} p^nq^n = \frac{4^n\sqrt{2\pi\cdot2n}}{2\pi n}p^nq^n = \frac{(4p(1-p))^n}{\sqrt{\pi n}}$$ при $n\to\infty$. Но так как ${p(1-p)\le1/4}$ для всех ${p\in(0,1)}$, то ${p_{j,j}(2n)\to0}$, ${n\to\infty}$. Учитывая, что $p_{j,j}(2n+1)=0$ для всех $n$, получаем, что все состояния являются нулевыми.

Теперь исследуем состояния на свойство возвратности. Для этого рассмотрим ряд $$P_j = \sum\limits_{n=1}^{\infty} p_{j,j}(n) = \sum\limits_{n=1}^{\infty} p_{j,j}(2n).$$ Если ${p=q=1/2}$, то ${4pq=1}$ и ряд расходится. В этом случае все состояния являются возвратными. Если же ${p\ne q}$, то ряд сходится и ${P_j<\infty}$. В этом случае все состояния являются невозвратными. Вероятность вернуться в состояние $j$ при условии, что система вышла из него в начальный момент времени, равна $F_j=P_j/(1+P_j)<1$. \EndEx

\textbf{Замечание.} Симметричное случайное блуждание по целочисленной решетке возвратно в пространствах одного и двух измерений и невозвратно в пространстве трех и более измерений. Это утверждение есть теорема Пойа~\cite{Polya}.
В одномерном случае утверждение можно обобщить на случай произвольных симметричных случайных блужданий на целочисленной прямой. Пусть теперь шаг блуждания $\xi$ принимает произвольные значения, не только $+1$ и $-1$. Если $\xi$ -- симметричная целочисленная случайная величина с конечным математическим ожиданием ${\mathbb{E} \xi = 0}$, то случайное блуждание ${X(n) = \sum_{j=1}^n \xi_j}$, где $\xi_j$ -- независимые одинаково распределенные случайные величины, образует возвратную ОДМЦ с нулевыми состояниями. Если при этом предположить, что наибольший общий делитель неотрицательных значений $\xi$ равен $1$, то ОДМЦ бесконечное число раз побывает в любом состоянии $i$. С доказательствами этих фактов можно познакомиться, например, в книге~\cite{Borovkov1999}.

Из рассмотренного примера~\ref{randomWalk} видно, что в общем случае (счетного числа состояний) нулевое состояние может быть как возвратным 
(при этом существенным), так и невозвратным (при этом тоже существенным).

\begin{definition}
    Возвратное ненулевое состояние называется \textit{положительно возвратным}.
    
        Возвратное нулевое состояние называется \textit{нуль возвратным}.
\end{definition}

В случае \textbf{счетного} класса эквивалентности существенное состояние может принадлежать к одному из \textbf{трех} типов:
\begin{enumerate}
    \item либо положительно возвратное состояние,
    \item либо нуль возвратное,
    \item либо невозвратное (как следствие, нулевое).
\end{enumerate}

Таким образом, в общей ситуации справедливы только импликации
\begin{equation*}
    \text{ненулевое состояние} \Longrightarrow \text{возвратное}\Longrightarrow \text{существенно}
\end{equation*}
\begin{equation*}
    \text{несущественно}\Longrightarrow \text{невозвратное состояние} \Longrightarrow \text{нулевое}
\end{equation*}
как уже и утверждалось в следствии к теореме~\ref{criteria transience}.

\elena{\textbf{Замечание 1.}}
Отметим, что в случае возвратности состояния $i$, последовательность $f_i(n)$, $n \ge 1$, образует распределение вероятности. А значит, можно говорить, например, о математическом ожидании времени \textit{первого} возвращения в $i$-е состояние: 
\elena{ 
$$\mu_i:=\mathbb{E}_i\tau_i = \mathbb{E} [\tau_i|X_0 = i]= \sum_{n=1}^\infty n f_i (n),$$
где $\tau_i = \min\{ n\ge 1:\,\, X_n = i\}$}. Оказывается, что если состояние  нуль-возвратное, то $\mathbb{E}\elena{_i} \tau_i = \infty$, а если состояние  положительно возвратное, то $\mathbb{E}\elena{_i} \tau_i =  \frac{1}{\pi_i}$, где $\pi$ -- стационарное (инвариантное) распределение
\elena{(про которое будет подробно рассказано в следующем подразделе), то есть распределение вероятностей, являющееся решением уравнения $P^T\pi=\pi$. Этот результат несложно показать из формулы полной вероятности. Введём следующие обозначения
\begin{equation}
\label{mu ij}
    \mu_{ij}:=\mathbb{E}_i\tau_j = \mathbb{E} [\tau_j|X_0 = i],
\end{equation}
в частности $\mu_i = \mu_{ii}$. Тогда из формулы полной вероятности (усредняя по всем возможным состояниям на первом шаге $X_1=k$, $k\in E$) имеем систему линейных уравнений для любых состояний $i$ и $j$
\begin{equation*}
    \mu_{ij} = \sum_{k\neq j} p_{i,k}\mu_{kj}  + 1.
\end{equation*}
Умножим на $\pi_i$ и просуммируем по всем состояниям $i\in E$, воспользовавшись равенствами $\sum_{i\in E} \pi_i p_{i,j} = \pi_j$ и $\sum_{i\in E} \pi_i =1$:
\begin{equation*}
    \mu_j\pi_j = 1.
\end{equation*}
}

\elena{Другое доказательство равенства $\mu_i \pi_i = 1$ на основе предельных законов см. в примере~\ref{ex:mu j pi j =1}.}


\begin{landscape}
\begin{figure}[!h]
\begin{subfigure}[b]{0.99\textwidth}

\tikzset{every picture/.style={line width=0.75pt}} 

\begin{tikzpicture}[x=0.75pt,y=0.75pt,yscale=-1,xscale=1]

\draw    (160,10) -- (160,28) ;
\draw [shift={(160,30)}, rotate = 270] [color={rgb, 255:red, 0; green, 0; blue, 0 }  ][line width=0.75]    (10.93,-3.29) .. controls (6.95,-1.4) and (3.31,-0.3) .. (0,0) .. controls (3.31,0.3) and (6.95,1.4) .. (10.93,3.29)   ;

\draw   (155,52.5) -- (160,48) -- (165,52.5) -- (162.5,52.5) -- (162.5,61.5) -- (165,61.5) -- (160,66) -- (155,61.5) -- (157.5,61.5) -- (157.5,52.5) -- cycle ;
\draw    (160,10) -- (280,10) ;

\draw    (500,10) -- (500,28) ;
\draw [shift={(500,30)}, rotate = 270] [color={rgb, 255:red, 0; green, 0; blue, 0 }  ][line width=0.75]    (10.93,-3.29) .. controls (6.95,-1.4) and (3.31,-0.3) .. (0,0) .. controls (3.31,0.3) and (6.95,1.4) .. (10.93,3.29)   ;

\draw    (380,10) -- (500,10) ;

\draw   (495,53.5) -- (500,49) -- (505,53.5) -- (502.5,53.5) -- (502.5,62.5) -- (505,62.5) -- (500,67) -- (495,62.5) -- (497.5,62.5) -- (497.5,53.5) -- cycle ;
\draw   (155,89.5) -- (160,85) -- (165,89.5) -- (162.5,89.5) -- (162.5,98.5) -- (165,98.5) -- (160,103) -- (155,98.5) -- (157.5,98.5) -- (157.5,89.5) -- cycle ;
\draw   (495,89.5) -- (500,85) -- (505,89.5) -- (502.5,89.5) -- (502.5,98.5) -- (505,98.5) -- (500,103) -- (495,98.5) -- (497.5,98.5) -- (497.5,89.5) -- cycle ;
\draw   (155,149.5) -- (160,145) -- (165,149.5) -- (162.5,149.5) -- (162.5,158.5) -- (165,158.5) -- (160,163) -- (155,158.5) -- (157.5,158.5) -- (157.5,149.5) -- cycle ;
\draw   (495,149.5) -- (500,145) -- (505,149.5) -- (502.5,149.5) -- (502.5,158.5) -- (505,158.5) -- (500,163) -- (495,158.5) -- (497.5,158.5) -- (497.5,149.5) -- cycle ;
\draw    (80,226) -- (80,244) ;
\draw [shift={(80,246)}, rotate = 270] [color={rgb, 255:red, 0; green, 0; blue, 0 }  ][line width=0.75]    (10.93,-3.29) .. controls (6.95,-1.4) and (3.31,-0.3) .. (0,0) .. controls (3.31,0.3) and (6.95,1.4) .. (10.93,3.29)   ;

\draw    (80,226) -- (130,226) ;

\draw    (240,226) -- (240,244) ;
\draw [shift={(240,246)}, rotate = 270] [color={rgb, 255:red, 0; green, 0; blue, 0 }  ][line width=0.75]    (10.93,-3.29) .. controls (6.95,-1.4) and (3.31,-0.3) .. (0,0) .. controls (3.31,0.3) and (6.95,1.4) .. (10.93,3.29)   ;

\draw    (190,226) -- (240,226) ;

\draw (160.5,38) node   {возвратное};
\draw (160,75.5) node   {$\mathbb{P}( X_{n} = i$ \text{бесконечно часто} $|\, X_{0}  = i)  = 1$};
\draw (330.5,7) node   {$i$-е состояние};
\draw (500.5,39) node   {невозвратное};
\draw (159.5,127) node   [align=left] {\ \ Вероятность возвращения за  конечное \\ \ \ \ \ \ \ \ число шагов: $F_{i} =\sum\limits^{\infty}_{n=1} f_{i} (n)= 1$};
\draw (500,75.5) node   {$\mathbb{P}( X_{n} = i$ \text{бесконечно часто} $|\, X_{0}  = i)  = 0$};
\draw (160.5,208.5) node   {$\mathbb{E}\sum\limits ^{\infty }_{n=1}\mathsf{I} (X_{n} =i)=\sum\limits ^{\infty }_{n=1} p_{ii} (n)=\infty $};
\draw (80.5,274) node   {$\lim\limits _{n\rightarrow \infty } p_{ii} (n)=0$};
\draw (80.5,302) node   {$\mu _{i} =\sum\limits ^{\infty }_{n=1} nf_{i} (n)=\infty $};
\draw (240.5,274) node   {$\lim\limits _{n\rightarrow \infty } p_{ii} (n)\neq 0$};
\draw (240,303.5) node   {$\mu _{i} =\sum\limits ^{\infty }_{n=1} nf_{i} (n)=\frac{1}{\pi _{i}}$};
\draw (240.5,350) node   {$\lim\limits _{n\rightarrow \infty } p_{ii} (n)=\pi _{i}  >0$};
\draw (240,374.69) node   {Иначе $\lim\limits _{n\rightarrow \infty }\frac{1}{n}\sum\limits ^{n}_{k=1} p_{ii} (k)=\pi _{i}  >0$};
\draw (159.5,179) node   [align=left] {Математическое ожидание числа \\ \ \ \ \ \ посещений $i$-го состояния:};
\draw (499.5,179) node   [align=left] {Математическое ожидание числа \\ \ \ \ \ \ посещений $i$-го состояния:};
\draw (500.5,208.5) node   {$\mathbb{E}\sum\limits^{\infty }_{n=1}\mathsf{I} (X_{n} =i)=\sum\limits ^{\infty }_{n=1} p_{ii} (n)< \infty $};
\draw (240.5,253.5) node   {положительно возвратное};
\draw (80.5,253.5) node   {нуль возвратное};
\draw (240.5,328.5) node   {Если $i$-е состояние апериодическое, то};
\draw (500.5,246.5) node   [align=left] {\qquad \qquad \ \ \ (как следствие: \\ $\lim\limits_{n \to \infty } p_{ii} (n)=0$ -- нулевое состояние)};
\draw (500.5,127) node  [align=left] {\ \ Вероятность возвращения за  конечное \\ \ \ \ \ \ \ \ число шагов: $F_{i} =\sum\limits^{\infty}_{n=1} f_{i} (n)< 1$};

\end{tikzpicture}
\end{subfigure}
    \caption{Классификация состояний ОДМЦ по асимптотическим свойствам $p_{i,i}(n)$} 
    \label{fig:StatesClassification}
\end{figure}
\end{landscape}

\textbf{Замечание \elena{2.}}
\begin{enumerate}
    \item Нуль-возвратные состояния могут быть только в случае счетных классов (классическим примером является симметричное случайное блуждание на одномерной решетке -- пример~\ref{randomWalk}).
    \item В случае конечных замкнутых классов состояния положительно возвратные.
    \item Примеры невозвратных состояний (а значит, нулевых) в случае
    \begin{enumerate}
        \item $|E| < \infty$: такими являются только открытые классы (то есть несущественные состояния).
    \item $|E| = \infty$: такими могут быть, как несущественные состояния, так и существенные (несимметричное блуждание из примера~\ref{randomWalk}).
    \end{enumerate}
\end{enumerate}


\begin{figure}
    \centering
    \includegraphics[width=0.6\textwidth]{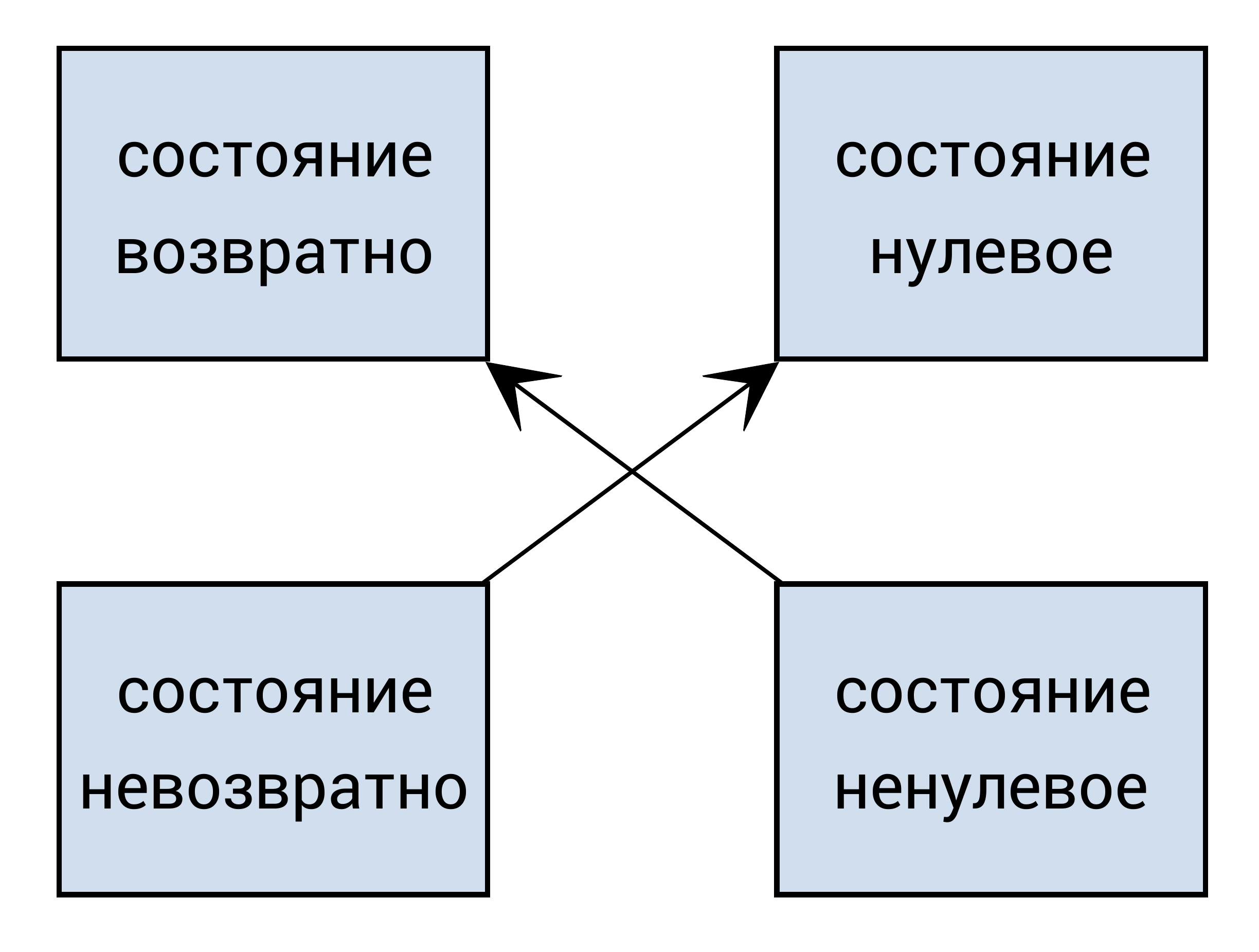}
    \caption{Связь возвратности и нулевости состояния ОДМЦ. Стрелки означают логические импликации. Обратные стрелки в общем случае провести нельзя, если число состояний счетно. Если цепь конечна, то можно провести и обратные стрелки}
    \label{fig:zero_return}
\end{figure}

\subsection{Эргодические дискретные цепи Маркова}

Перейдем теперь к важнейшим понятиям стационарности распределения и эргодичности цепи.

\begin{definition}
    Распределение вероятностей ${\pi = (\pi_0, \pi_1,\dots)}$ называется \textit{стационарным, или инвариантным}, если $$P^\top \pi = \pi,$$ где $P$ -- переходная матрица \gav{однородной дискретной марковской цепи (ОДМЦ)}.
\end{definition}

\textbf{Замечание.} ОДМЦ является стационарным в узком смысле случайным процессом тогда и только тогда, когда распределение вероятностей на множестве состояний является стационарным. 

Как мы увидим позже, стационарного распределения может и не быть. Такая ситуация возникает в случае счетных цепей.

\begin{example}
\label{Ex1}
Рассмотрим несколько примеров.

\end{example}


\begin{figure}[h]
\begin{minipage}[h]{0.41\linewidth}
\center{\includegraphics[width=\textwidth]{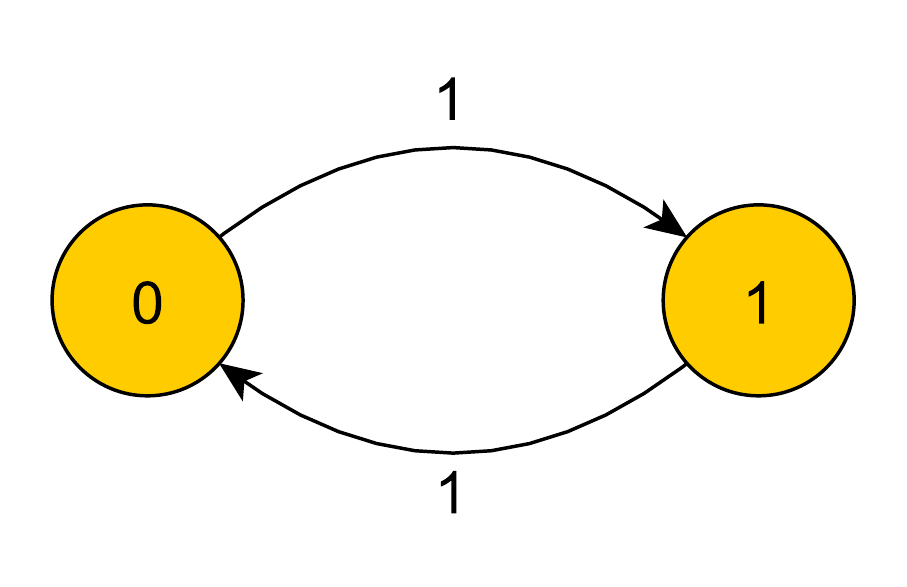}}
\end{minipage}
\hfill
\begin{minipage}[h]{0.58\linewidth}
$P=\left[ {\begin{array}{*{20}{r}}
	{0}&{1} \\
	{1}&{0} \\ 
	\end{array}} \right], \ \ \pi=\left[ {\begin{array}{*{20}{r}}
	{1/2} \\
	{1/2} \\ 
	\end{array}} \right]$
\end{minipage}
\vfill
\begin{minipage}[h]{0.41\linewidth}
\center{\includegraphics[width=\textwidth]{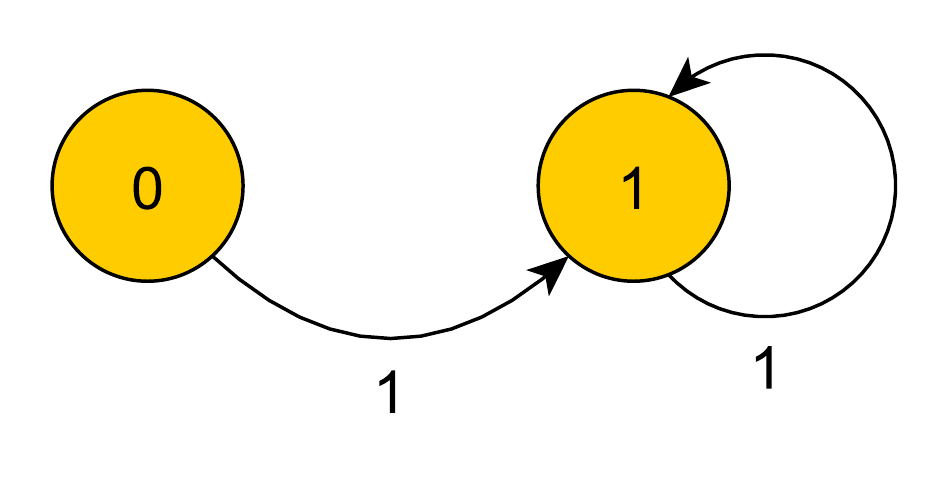}}
\end{minipage}
\hfill
\begin{minipage}[h]{0.58\textwidth}
$P=\left[ {\begin{array}{*{20}{r}}
	{0}&{1} \\
	{0}&{1} \\ 
	\end{array}} \right], \ \ 
	\pi=\left[ {\begin{array}{*{20}{c}}
	{0} \\
	{1} \\ 
	\end{array}} \right]$
\end{minipage}
\vfill
\begin{minipage}[h]{0.41\linewidth}
\center{\includegraphics[width=\textwidth]{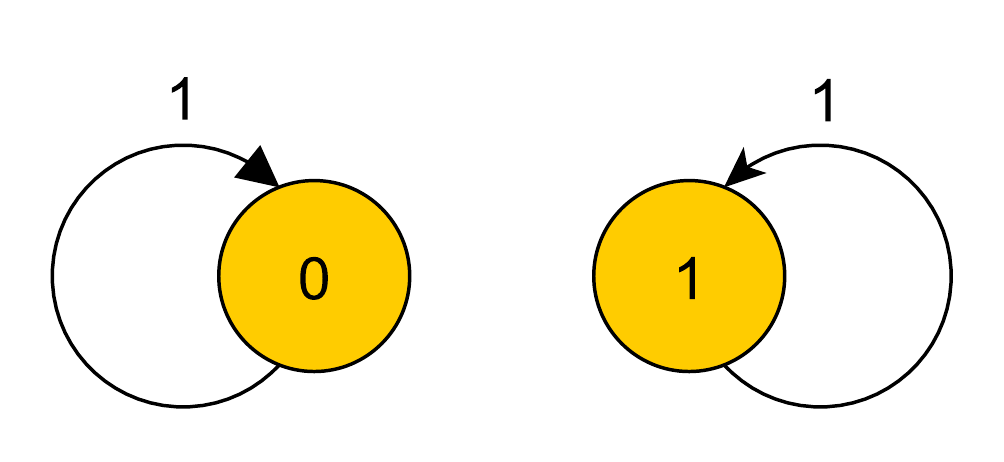}}
\end{minipage}
\hfill
\begin{minipage}[h]{0.58\textwidth}
$
P=\left[ {\begin{array}{*{20}{r}}
{1}&{0} \\
{0}&{1} \\ 
\end{array}} \right], \ \ 
\pi=\left[ {\begin{array}{*{20}{c}}
{p} \\
{1-p} \\ 
\end{array}} \right] 
\forall p \in[0,1]
$
\end{minipage}
\end{figure}
$$$$
\begin{theorem}
\label{existence invar}
\textit{В любой конечной цепи Маркова найдется хотя бы одно стационарное распределение.}
\end{theorem}

\begin{proof} 
Утверждение теоремы следует из одного варианта теоремы Брауэра о неподвижной точке, и для нашего случая доказательство проводится следующим образом. Рассмотрим множество всех стохастических векторов: 
$$\mathcal{P}=\{p=(p_1, \dots, p_N):  \ p_1 \ge 0, \, \dots, \, p_N \ge 0, \, p_1 + \dots + p_N = 1\}.$$ 
Легко видеть, что это множество выпуклое, замкнутое и ограниченное. Возьмем произвольный вектор ${p \in \mathcal{P}}$, введем обозначение ${A=P^\top}$ и рассмотрим последовательность $$w_n=\frac{1}{n+1}\sum\limits_{k=0}^n A^k p.$$ Так как преобразование $A$ переводит стохастические векторы в стохастические, а множество $\mathcal{P}$ выпуклое, то $\{w_n\}\subset\mathcal{P}$. Далее, так как последовательность $\{w_n\}$ ограничена (она лежит в ограниченном множестве), то по теореме Больцано--Вейерштрасса из нее можно выделить сходящуюся подпоследовательность $\{w_{n_m}\}$, причем предел $w$ этой подпоследовательности лежит в $\mathcal{P}$, так как множество $\mathcal{P}$ замкнутое. Из~оценок $$|Aw_{n_m}-w_{n_m}|=\frac{1}{n_m+1}\left|\sum\limits_{k=0}^{n_m} \left(A^{k+1}p - A^k p\right)\right|=$$ $$=\frac{1}{n_m+1}\left|A^{n_m+1}p - p\right|\le\frac{\mathrm{diam}(\mathcal{P})}{n_m+1}$$ следует, что ${|Aw_{n_m}-w_{n_m}|\to0}$, $m\to\infty$; здесь $\mathrm{diam}(\mathcal{P})$ -- диаметр множества $\mathcal{P}$, он конечен в силу ограниченности $\mathcal{P}$. Так как $$A(w_{n_m}-w)-(w_{n_m}-w)\to0, \ {Aw_{n_m}-w_{n_m}\to0},$$ то необходимо ${Aw=w}$. Значит, мы нашли стационарное распределение $p^0=w$. \EndProof

\end{proof} 

\textbf{Замечание.}  Теорема~\ref{existence invar} является следствием теоремы~\ref{eigen value 1}, так как собственное пространство конечной  матрицы $P^T$, отвечающее собственному значению 1, порождается стационарным распределением. Более того, если цепь содержит единственный замкнутый конечный класс эквивалентности, то $1$ является собственным значением матриц $P$ и $P^T$, причем алгебраическая и геометрическая кратность равны единице, и соответствующее собственное пространство матрицы $P$ порождается вектором из одних единиц, а соответствующее собственное пространство матрицы $P$ порождается стационарным распределением, см. теорему о спектре стохастической матрицы 1.12.3~\cite{KelbertSukhov2010}. Если матрица бесконечная, то собственный вектор матрицы $P^T$ с неотрицательными компонентами для собственного значения 1 тоже существует, но из-за бесконечности размера не всегда существует нормировочная константа, переводящая этот вектор в распределение вероятности. 

\begin{theorem}
\textit{\textit{Пусть $\pi$ -- стационарное распределение конечной цепи Маркова, а сос\-то\-яние $i$ является несущественным. Тогда}} $\pi_i=0$.
\end{theorem}

\begin{proof}
Все состояния цепи разобьем на два класса:\linebreak класс несущественных состояний $S_0$ и класс всех остальных состояний $S_1$. Перенумеровав подходящим образом состояния цепи, мы можем записать переходную матрицу: $$P=\left[ {\begin{array}{*{20}{c}}
{Q} & {R} \\
{O} & {S} \\ \end{array}} \right],$$ где матрица $Q$ отвечает переходам между несущественными состояниями, матрица $R$ отвечает переходам от несущественных состояний к существенным, матрица $O$ нулевая, а матрица $S$ отвечает переходам между существенными состояниями. 
$$P = \left[ {\begin{array}{*{20}{c}}
{Q} & {R} \\
{O} & {S} \\ 
\end{array}} \right], \ \ P^2 = \left[ {\begin{array}{*{20}{c}}
{Q^2} & {QR+RS} \\
{O} & {S^2} \\ 
\end{array}} \right], \ \dots, \ P^n = \left[ {\begin{array}{*{20}{c}}
	{Q^n} & {A} \\
	{O} & {S^n} \\ 
	\end{array}} \right],$$
 где $A$ -- некоторая матрица. Заметим, что в силу конечности цепи найдется $n_0$ такое, что для всех несущественных состояний $i\in{1,\dots,k}$ выполнено неравенство
\begin{equation}\label{eq:ineq}
\sum\limits_{j=1}^k Q_{ij}^{n_0} < 1,
\end{equation}
где ${k<|E|}$ -- количество несущественных состояний. Теперь рассмотрим матричную норму для произвольных матриц $M$ размера $$\left\| M \right\|_{\infty} = \max\limits_{1\le i \le n}\sum\limits_{j=1}^n |M_{ij}|.$$ Тогда из~\eqref{eq:ineq} следует неравенство $\left\| Q^{n_0} \right\|_{\infty} < 1$. 

Докажем, что ${\left\| Q^{n} \right\|_{\infty} \to 0}$, ${n\to\infty}$. Из неравенства ${\left\| Q \right\|_{\infty}\le 1}$ следует, что $${\left\| Q^n \right\|_{\infty}\le \left\| Q \right\|_{\infty}^n \le 1}$$ для всех ${n\le n_0}$. Далее, для всех $n\in[n_0,2n_0)$ имеем $${\left\| Q^n \right\|_{\infty}\le \left\| Q^{n_0} \right\|_{\infty} \left\| Q \right\|_{\infty}^{n-n_0} \le \left\| Q^{n_0} \right\|_{\infty}}.$$ Для $n\in[2n_0,3n_0)$ получаем аналогично
$${\left\| Q^n \right\|_{\infty}\le \left\| Q^{n_0} \right\|_{\infty}^{2} \left\| Q \right\|_{\infty}^{n-2n_0} \le \left\| Q^{n_0} \right\|_{\infty}^{2}}.$$ В общем случае получаем, что 
$$\left\| Q^n \right\|_{\infty}\le\left\| Q^{n_0} \right\|_{\infty}^{[n/n_0]}\to0, \ n \to\infty,$$ где квадратные скобки означают взятие целой части. Это означает, что с ростом $n$ каждый элемент матрицы $Q$ стремится к нулю.

Теперь транспонируем матрицу $P$: $$\left(P^\top\right)^n = \left[ {\begin{array}{*{20}{c}}
{\left(Q^\top\right)^n} & {O} \\
{A^\top} & {\left(S^\top\right)^n} \\ 
\end{array}} \right]. $$
Произвольное стационарное распределение $\pi$ удовлетворяет равенствам $$\pi = P^\top \pi = \cdot = (P^\top)^n \pi = \dots = \lim\limits_{n\to\infty}\left(P^\top\right)^n\pi.$$ Так как первые $k$ компонент вектора $\lim_{n\to\infty}\left(P^\top\right)^n\pi$ равны нулю, то нулю равны и первые $k$ компонент вектора $\pi$, как следует из равенства выше. Значит, для всех несущественных состояний $i=1,\dots,k$ вероятности $\pi_i = 0$. \EndProof
\end{proof}

\elena{На самом деле справедливо более общее утверждение.} 

\elena{
 \begin{theorem}
 \textit{Пусть $\pi$ -- стационарное распределение ОДМЦ, а сос\-то\-яние $i$ является нулевым. Тогда $\pi_i=0$. }
 \end{theorem}
 }
 
 \elena{
 \begin{proof}
     Стационарное распределение удовлетворяет системе уравнений для любого $n\ge 1$
$$
\left(P^T \right)^n \pi = \pi.
$$
В частности, 
\begin{equation}
\label{pi=0}
    \pi_i  = \sum_{k\in E}\pi_k p_{ki}(n)\le  \sup_{k\in E} p_{ki}(n).
\end{equation}
Заметим, что если $i$ состояние нулевое, то для любого состояния $k$ $p_{ki}(n)\to 0$ при $n\to\infty$. Это  следует, если перейти к пределу $n\to\infty$ в формуле полной вероятности
 $$
 p_{ki}(n) = \sum_{m=1}^n f_{ki}(m)p_{ii}(n-m),
 $$
 где $f_{ki}(m)$ --- вероятность первого попадания из $k$ состояния в $i$ на шаге  $m$. 
 Возвращаясь к доказательству теоремы, возьмем предел $n\to\infty$ в формуле~\eqref{pi=0}, получим требуемое утверждение.
 \EndProof
 \end{proof}
 }

\textbf{Замечание 1.} В случае конечной ОДМЦ с единственным замкнутым классом эквивалентности существует единственное стационарное распределение, причем его значения на несущественных состояниях (т.е. тех состояниях, что принадлежат открытым классам эквивалентности) нулевые (см. пример~\ref{Ex1}, первую и вторую ситуации). 

\textbf{Замечание 2.} Рассмотрим теперь ОДМЦ с $m$ конечными замкнутыми классами эквивалентности: $E_1$, $E_2$, \ldots, $E_m$ (число открытых классов и их размер не важны, так стационарное распределение на несущественных состояниях все равно нулевое -- обозначим класс несущественных состояний через $E_0$). Каждый замкнутый класс эквивалентности, если его рассматривать как самостоятельную цепь, имеет единственное стационарное распределение: $\pi^{(k)}$ -- вектор размерности $|E_k|$, $k=1,\ldots,m$. Введем также вектор $\pi^{(0)}$ размерности $|E_0|$, состоящий из нулей: $\pi^{(0)} = \left(0,\dotsc,0\right)^T$. Этот вектор отвечает несущественным состояниям. Тогда стационарным распределением всей цепи $\pi$ является любая 
выпуклая комбинация стационарных распределение классов (см. пример~\ref{Ex1}, третью ситуацию).
Формально можно записать это так:
\begin{equation}\label{erg_conv}
\pi = \left(\pi^{(0)},\alpha_1 \pi^{(1)},\dotsc,\alpha_m \pi^{(m)}\right)^T,
\end{equation}
 где $\alpha_k \ge 0$, $\sum_{k=1}^m \alpha_k = 1$.
 
 \begin{myquote}
     Важно заметить\footnote{Собственно, вся последующая часть данного раздела и весь следующий раздел можно понимать как обоснование и обсуждение того, что написано в этом абзаце. Этот абзац является, пожалуй, самым главным материалом, который рекоменудется вынести из теории марковских процессов.}, что формула~\eqref{erg_conv} описывает всевозможные (в~зависимости от начального распределения $p(0)$) предельные (финальные) распределения $\lim_{t\to\infty} p(t)$ для однородных марковских цепей в дискретном и непрерывном времени (в дискретном времени, в случае периодической цепи, сходимость следует понимать в смысле~\eqref{Chesaro}). Набор чисел $\{\alpha_k\}$ однозначно определяется начальным распределением, см. пример \ref{ex:casino}.  Причем все это верно и для $m = \infty$ с счетными классами эквивалентности. Только в~случае счетного числа состояний в общем случае вместо $\alpha_k \ge 0$, $\sum_{k=1}^m \alpha_k = 1$ можно лишь утверждать, что $\alpha_k \ge 0$, $\sum_{k=1}^m \alpha_k \le 1$. Пример \ref{ex:casino} при $m=1$ как раз описывает такую ситуацию, в~которой неравенство может быть строгим. Наиболее интересный случай (эргодический), когда предельное распределение не зависит от начального, т.е. когда $m=1$, $\alpha_1 = 1$. Далее мы в основном будем рассматривать только этот случай.
 \end{myquote}

\textbf{Замечание 3.}  В случае ОДМЦ со счетным множеством состояний может оказаться, что вовсе нет стационарного распределения (см. пример~\ref{birth-death discr} ниже). В этой ситуации необходимо дополнительное условие: существование положительно возвратного класса, где состояния ненулевые и возвратные. Если такой класс единственен, то и стационарное распределение единственно. Иначе возникает смесь стационарных распределений, как и в предыдущем замечании.
\raggedbottom
Есть несколько частных случаев, когда поиск стационарного распределения для замкнутого класса эквивалентности достаточно легок. 

а) Пусть матрица переходов  конечной ОДМЦ дважды  стохастическая, то есть сумма ее элементов и по каждой строке, и по каждому столбцу равна единице. Тогда, как несложно проверить, равномерное распределение на множестве состояний является стационарным, то есть $\pi_i = 1/|E|$ для любого $i\in  E$.

б) Пусть ОДМЦ описывает случайное блуждание на некотором неориентированном связном графе из $n$ вершин \elena{и $N$ ребер}: находясь в произвольной вершине, на следующем шаге выбирается равновероятно одна из смежных вершин. Легко убедиться, что стационарное распределение на множестве вершин следующее:
$$ \pi_i = \frac{\text{deg}_i}{2\elena{N}},$$
где $\text{deg}_i$ -- степень вершины $i$, т.е. число смежных с ней вершин.

в) \elena{ Важным классом марковских цепей образуют цепи, обратимые во времени, а именно введем следующее определение.
\begin{definition}
\label{reversible chain}
    ОДМЦ (возможно со счетным множеством состояний) обладает свойством \textit{обратимости}, если существует нетривиальное решение следующего уравнения \textit{детального баланса} (см., например, главу~6, \cite{StochAn2016})\footnote{Как бы выполнен закон сохранения массы: сколько вышло из $i$-го состояний в $j-$е, столько и вошло из $j$-го в $i$-е состояние.}:
для любых 
\begin{center}
$i,j \in E$ $\pi_i P_{i,j} = \pi_j P_{j,i}$,
\end{center}
где $P_{i,j}$ -- вероятность перейти из состояния $i$ в $j$.
\end{definition} 
}
Просуммируем уравнение детального баланса по $j\in E$, получим уравнение на поиск стационарного распределения. Таким образом, \elena{распределение вероятностей, являющееся} решением уравнения детального баланса, является стационарным распределением (обратной импликации нет). 

\elena{Стоит также заметить (см.\cite[\S~1.10]{KelbertSukhov2010}), что свойство обратимости можно понимать следующим образом. Из уравнения детального баланса следует, что для любого $n$ и для любых $i_0, i_1, \ldots, i_n \in E$
\begin{equation*}
    \pi_{i_0} P_{i_0, i_1}  P_{i_1, i_2} \cdots  P_{i_{n-1}, i_n} = \pi_{i_n} P_{i_n, i_{n-1}}  P_{i_{n-1}, i_{n-2}} \cdots  P_{i_1, i_0}
\end{equation*}
Иначе говоря, для марковской цепи, обладающей свойством обратимости и имеющей в качестве начального распределения стационарное выполнено
\begin{equation*}
    \left(X_0 X_1 \ldots X_n \right) \overset{d}{=} \left(X_n X_{n-1} \ldots X_0 \right),
\end{equation*}
то есть у цепи с обратимым временем $\left( X_{n-k}\right)_{k=0}^n$ такое же распределение как и у исходной цепи  $\left( X_{k}\right)_{k=0}^n$.
}

\begin{example}
\label{random walk B-cube}
В качестве примера рассмотрим случайное блуждание на булевом кубе $\{0,1\}^N$: находясь в вершине $v = (v_1,\ldots,v_N)$, где $v_m\in\{0,1\}$ для всех $m=1,\ldots,N$, равновероятно выбирается одна их $N$ компонент и меняется на противоположную. Ясно, что матрица переходных вероятностей будет дважды стохастической, а значит, и стационарное распределение равномерное: \gav{$\pi_v \equiv 2^{-N}$}. \elena{Тот же ответ можно получить, воспользовавшись формулой выше из пункта б) для поиска стационарного распределения для симметричного случайного блуждания на неориентированном графе.}
\end{example}

\begin{example}
\label{Erenfest invar}
В качестве следующего примера рассмотрим дискретную модель диффузии Эренфестов, см. также раздел~\ref{EhrenfestModel}. Пусть имеется сосуд с $N$ молекулами, разделенный мембраной на две части. В каждый момент времени $n$ выбирается равновероятно одна из молекул и переводится в противоположную часть сосуда. Пусть номер состояния $i$ -- число молекул в фиксированной части сосуда, ${i \in \{0, 1, \ldots, N\}}$. \elena{Тогда
\begin{equation*}
    P_{i,j} = \begin{cases}
    1-\frac{i}{N}, & \text{если $j=i+1$;} \\ 
    \frac{i}{N}, & \text{если $j=i-1$;} \\ 0, & \text{иначе.} \\ 
    \end{cases}
\end{equation*} } Несложно показать, что ОДМЦ является обратимой, и найти $\pi_i = C_N^i2^{-N}$. Заметим, что этот результат можно получить из следующих рассуждений. Предположим, что молекулы пронумерованы. Тогда динамика диффузии эквивалентна симметричному случайному блужданию на булевом кубе $\{0,1\}^N$ \elena{ (см. пример~\ref{random walk B-cube})} с равномерным стационарным распределением. Далее, осталось перейти от пронумерованных молекул к неразличимым, объединив подмножества состояний в одно (по числу ненулевых компонент). \elena{Иначе говоря, если $X_n$ -- марковская цепь, описывающая симметричное случайное блуждание на булевом кубе  $\{0,1\}^N$, то $W(X_n)$ -- марковская цепь, описывающая диффузию в модели Эренфестов, где $W (V) = \sum_{i=1}^N v_i$, $v_i\in\{0,1\}$, (компоненты вектора $V$) -- хэммингов вес (число ненулевых компонент в булевом слове). }
\end{example}

Отметим также, что условие детального баланса активно используется в подходе Markov Chain Monte Carlo (см., например,~\cite[глава 7]{StochAn2016} и раздел \ref{sec:PageRank} ниже).

Наличие стационарного распределения указывает на то, каким может быть предельное (финальное) распределение, если последнее существует.
\begin{definition}\label{def:ergodicity}
    Конечная ОДМЦ называется \textit{эргодической}, если стационарное распределение единственно.  
\end{definition}
Как мы увидим в дальнейшем, единственность стационарного распределения (и для конечной, и для счетной цепи) обеспечивается наличием ровно одного
\gav{замкнутого} класса сообщающихся состояний при возможном наличии несущественных состояний. 


Оказывается единственность стационарного распределения $\pi$ обеспечивает эргодичность конечной ОДМЦ $\{X_n\}$ в смысле, который обсуждался ранее в пунктах \ref{sec:ergodic_processes_subsection}, \ref{sec:ErgodicDS}: для любой собственной (конечной) функции $f(x)$, заданной на состояниях ОДМЦ, \gav{почти наверное (с вероятностью 1)} имеет место 
\begin{equation}\label{GeneralChesaro}
    \lim_{N\to\infty}\frac{1}{N}\sum\limits_{n=1}^{N}f(X_n) = \EE_{\pi} f(X) = \sum_{k} \pi_k f(k),
\end{equation}
в частности, при $f(x) = \mathsf{I}(x = k)$, имеем
\begin{equation}\label{Chesaro}
    \lim_{N\to\infty}\frac{1}{N}\sum\limits_{n=1}^{N} p_k(n) =  \pi_k.
\end{equation}


Приведем, следуя~\cite[глава 13]{Borovkov1999} и~\cite[глава 11]{Vereschyagin}, пример приложения \gav{небольшого обобщения} формулы~\eqref{GeneralChesaro} (эргодической теоремы) для конечных ОДМЦ в теории информации. 
\begin{example}
В теории информации одной из основных задач является эффективное сжатие информации. В частности, если по каналу передавать сигнал, в котором равновероятно (с одинаковыми частотами) может встретиться одна из $m$ букв, то для передачи $n$ букв потребуется $n\log_2 m$ битов, т.е. $\log_2 m$ битов на букву. Если же вероятности появления букв в тексте (частоты) не равновероятны $p =$\linebreak $= \{p_k\}_{k=1}^m$, то передачу букв можно организовать эффективнее: так, чтобы на одну букву приходилось лишь $H(p) = -\sum_{k=1}^m p_k\log_2 p_k \le$\linebreak $\le\log_2 m$ битов. Покажите, что если буквы в тексте появляются согласно конечной однородной эргодической марковской цепи с матрицей переходных вероятностей $P = \|p_{ij}\|_{i,j=1}^{m,m}$ (альтернативная интерпретация -- известна матрица частот появления всевозможных пар букв $(i,j)$ друг за другом), то передачу букв можно организовать еще эффективнее: так, чтобы на одну букву приходилось лишь $H(\{p_{ij}\}_{i,j=1}^{m,m}) =$\linebreak $= -\sum_{i=1}^m \pi_i\sum_{j=1}^m p_{ij}\log_2 p_{ij}$ битов, где $\pi$ -- стационарное распределение марковской цепи с матрицей переходных вероятностей $P$. 
\end{example}

\begin{solution}
 Рассмотрим вероятностных подход (частотный подход см. в~\cite[глава 11]{Vereschyagin}). Рассмотрим $n$-буквенные случайные последовательности  $X = \left(X_1,\dotso,X_n\right)$, сгенерированные согласно введенной марковской цепи. Будем называть последовательность \ag{$x = \left(x_1,\dotso,x_n\right)$} $\delta$-типичной, если вероятность ее появления лежит в специальном диапазоне:  
 \begin{equation}\label{Shanon}
 \PP(X\ag{ = x}) \in \left[2^{-n\left(H(\{p_{ij}\}_{i,j=1}^{m,m}) + \delta\right)},2^{-n\left(H(\{p_{ij}\}_{i,j=1}^{m,m}\}) - \delta\right)}\right].
  \end{equation}
 Число таких последовательностей будет $O\left(2^{n\left(H(\{p_{ij}\}_{i,j=1}^{m,m}) + \delta\right)}\right)$, что может быть много меньше общего числа возможных последовательностей $2^n$. Если будет показано, что для любых (сколь угодно малых) $\delta > 0$ и $\sigma > 0$ можно подобрать такое (достаточно большое) $n(\delta,\sigma)$, начиная с которого событие~\eqref{Shanon} имеет место с вероятностью $\ge 1 - \sigma$, то  эффективное кодирование можно организовать, кодируя только\linebreak $\delta$-типичные последовательности.
 
 Итак, рассмотрим на траектории марковского процесса функцию от двух соседних сечений \ag{$f(x_k,x_{k+1}) = - \log_2 p_{x_k x_{k+1}}$}. \ag{Поскольку} \ag{$$\PP(X\ag{=x}) = p_{\ag{x}_1}\prod_{k=1}^{n-1} p_{\ag{x}_k \ag{x}_{k+1}},$$ то ф}ормула~\eqref{Shanon} будет выполняться для достаточно больших $n$, если \ag{почти наверное}
 $$\frac{1}{n}\sum_{k=1}^n f(X_k,X_{k+1})\xrightarrow[n\to\infty]{}H \left(\{p_{ij}\}_{i,j=1}^{m,m}\right).$$
 Но поскольку рассматриваемый марковский процесс по предположению эргодический, то предел левой части существует и равняется математическому ожиданию $f(X_k,X_{k+1})$ в предположении, что $X_k$ имеет стационарное (инвариантное) распределение $\pi$, а $X_{k+1}$ случайно порождается по $X_k$ согласно матрице переходных вероятностей $P$, т.е.
\begin{eqnarray*}
&\EE_{X_k\sim\pi} f(X_k,X_{k+1}) = -\EE_{X_k\sim\pi} \log_2 p_{X_k X_{k+1}} = \\
&=- \sum_{i=1}^m \pi_i\sum_{j=1}^m p_{ij} \log_2 p_{ij} = H(\{p_{ij}\}_{i,j=1}^{m,m}).
\end{eqnarray*}
 
 Приведенные здесь рассуждения можно провести более точно. В~частности, можно оценить $n(\delta,\sigma)$~\cite{Cover}. \EndEx
\end{solution}

Для изучаемого класса случайных процессов (ОДМЦ) при некоторых дополнительных предположениях сходимость в чезаровском смысле~\eqref{Chesaro} можно превратить в обычную сходимость. Оказывается, для этого нужно дополнительно предположить, что ОДМЦ апериодическая. 

В связи с некоторыми особенностями перенесения результатов об эргодичности конечных ОДМЦ на ОДМЦ со счетным числом состояний, оказывается, удобно также ограничиться классом ОДМЦ с отсутствием нулевых состояний. Далее об этом будет написано подробнее. 

В связи с написанным выше также используют более сильное определение эргодичности ОДМЦ \cite{NGG2}.
\begin{definition}
    Конечная или счетная ОДМЦ называется \textit{сильно эргодической}, если для любого ${j\in E}$ существует не зависящий от $i\in E$ положительный предел: $$\lim\limits_{n\to\infty}p_{i,j}(n)=p_j^*>0.$$ 
\end{definition}

Несложно заметить, что если ОДМЦ сильно эргодическая, то стационарное распределение $\pi$ единственное и $p^* = \pi$. 

\begin{example}
Рассмотрим несколько примеров.
\end{example}

\begin{minipage}{0.4\textwidth}
\includegraphics[width=\textwidth]{Section06/Problem94_1-eps-converted-to.pdf}
\end{minipage}
\begin{minipage}{0.54\textwidth}
Здесь ${p_{0,0}(2n)=1,\,p_{0,0}(2n+1)=0}$, $n\ge1$. Поэтому пре\-де\-л $\lim_{n\to\infty}p_{0,0}(n)$ не существует, цепь не будет сильно эргодической, но будет эргодической.
\end{minipage}


\begin{minipage}{0.4\textwidth}
\includegraphics[width=\textwidth]{Section06/Problem94_2-eps-converted-to.pdf}
\end{minipage}
\begin{minipage}{0.54\textwidth}
${P=P^2=P^3=\dots}$, значит, для всех $i,j$ предел $\lim_{n\to\infty}p_{i,j}(n)$ существует, но $\lim_{n\to\infty}p_{0,0}(n)=0$, поэтому цепь не будет сильно эргодической, но будет эргодической.
\end{minipage}


$$$$

\setstretch{1.0}
Рассмотрим теперь такую цепь:
\begin{figure}[!h]
	\centering
	\includegraphics[width=0.4\textwidth]{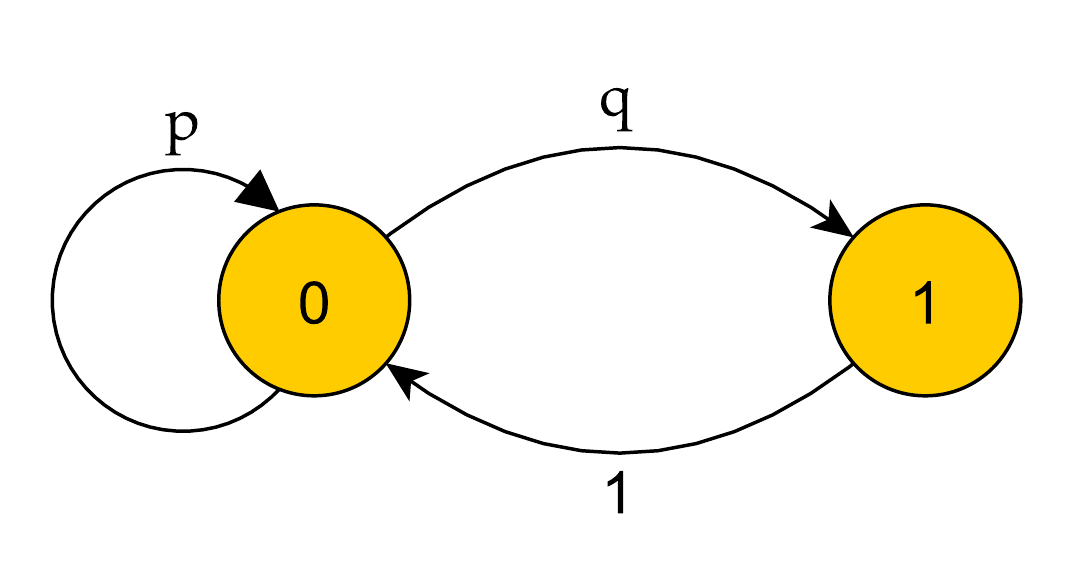}
\end{figure}

\setlength{\parskip}{-0.5em}
\noindent Ранее мы получили, что $$P^n \to \frac{1}{1+q}\left[ {\begin{array}{*{20}{r}}
	{1}&{q} \\
	{1}&{q} \\ 
	\end{array}} \right], \ n\to\infty,$$ значит, существуют пределы $\lim\limits_{n\to\infty}p_{i,j}(n)=p_j^*>0$, причем $$p_0=\frac{1}{1+q}, \ p_1=\frac{q}{1+q}, \ p_0+p_1=1.$$ По определению, данная цепь сильно эргодическая.  \EndEx

\setlength{\parskip}{1.0em}


\setlength{\parskip}{0pt}
Следующая теорема является ключевой для доказательства критерия сильной эргодичности конечной ОДМЦ (однородной дискретной марковской цепи). Отметим, что доказательство теоремы ниже остается верным, если число состояний счетно, а время непрерывно. Марковские цепи с непрерывным временем подробнее изложены в разделе~\ref{sec:ContChainMarkov}.
\begin{theorem}\label{th:technical_ergodic}
    \textit{Пусть для ОДМЦ с множеством состояний $S$ и~переходными вероятностями $p_{i,j}$ выполняется условие}:
    \[
        \exists \delta > 0\; \exists j_0\in S\; \exists n_0 :\; \forall i\in S\; p_{i,j_0}(n_0) \ge \delta.
    \]
   \textit{ Тогда $\exists \lim\limits_{n\to\infty}p_{i,j}(n) = p_j^* = \pi_j$ для всех $i,j\in S$, причем $p_{j_0}\elena{^*} > 0$, и}
    $$
        |p_{i,j}(n) - p_j^*| \le (1-\delta)^{[n/n_0]} \quad \forall i,j\in S.
    $$
\end{theorem}
\begin{proof}
    Пусть 
    \begin{center}
    ${M_j(n) = \max_{i\in S}p_{i,j}(n)}$ и $m_j(n) = \min_{i\in S}p_{i,j}(n)$.
    \end{center}
    Из определения чисел $M_j(n)$ и $m_j(n)$ следует, что $m_j(n) \le p_{i,j}(n) \le$\linebreak $\le M_j(n)$ для всех $i\in S$. Кроме того, из уравнения Колмогорова--Чепмена имеем
    \[
        p_{i,j}(n+1) = \sum\limits_{k\in S} p_{i,k} \underbrace{p_{k,j}(n)}_{\le M_j(n)} \le M_j(n) \sum\limits_{k\in S}p_{i,k} = M_j(n)\quad \forall i,j\in S,
    \]
    \[
        p_{i,j}(n+1) = \sum\limits_{k\in S} p_{i,k} \underbrace{p_{k,j}(n)}_{\ge m_j(n)} \ge m_j(n) \sum\limits_{k\in S}p_{i,k} = m_j(n)\quad \forall i,j\in S.
    \]
    Отсюда следует, что
    $$
        M_j(n+1) = \max\limits_{i\in S}p_{i,j}(n+1) \le M_j(n), 
    $$$$
        m_j(n+1) =\min\limits_{i\in S}p_{i,j}(n+1) \ge m_j(n). 
    $$
    Таким образом, мы показали, что для каждого $j\in S$ последовательность $\{M_j(n)\}_{n=1}^\infty$ возрастает, а последовательность $\{m_j(n)\}_{n=1}^\infty$ убывает. Следовательно, обе последовательности имеют предел. Покажем, что эти пределы совпадают, т.е. покажем, что
    \[
        M_j(n) - m_j(n) \xrightarrow[n\to\infty]{} 0.
    \]
    Отсюда будет следовать, что существуют числа 
    $$p_j^* = \lim\limits_{n\to\infty}p_{i,j}(n) = 
    \lim\limits_{n\to\infty}M_j(n) = \lim\limits_{n\to\infty}m_j(n)\ge 0,$$ причем в силу того, что $p_{i,j_0}(n_0) \ge \delta > 0$ для всех $i\in S$, получаем
    \[
        m_{j_0}(n_0) \ge \delta > 0 \Longrightarrow \forall n \ge n_0\; m_{j_0}(n) \ge \delta \Longrightarrow p_{j_0}\elena{^*} \ge \delta > 0.
    \]
    Начнем с того, что оценим разность $M_j(n) - m_j(n)$:
    \begin{eqnarray*}
         M_j(n) - m_j(n) &=& \max\limits_{i\in S}p_{i,j}(n) + \max\limits_{k\in S}\left(-p_{k,j}(n)\right)= \\ 
         &=&\max\limits_{i,k\in S}\left(p_{i,j}(n) - p_{k,j}(n) \right).
    \end{eqnarray*}
    Далее будем считать, что $n > n_0$. Тогда, продолжая предыдущую цепочку равенств и используя уравнение Колмогорова--Чепмена, получаем
    \begin{eqnarray*}
         M_j(n) - m_j(n) &=& \max\limits_{i,k\in S}\sum\limits_{l\in S}\left(p_{i,l}(n_0) - p_{k,l}(n_0) \right) p_{l,j}(n-n_0).
    \end{eqnarray*}
    Рассмотрим теперь два множества индексов: 
    $$I_{i,k}^+ = \{l\in S : p_{i,l}(n_0) - p_{k,l}(n_0) \ge 0\},$$  $$I_{i,k}^- = \{l\in S : p_{i,l}(n_0) - p_{k,l}(n_0) < 0\}.$$ Тогда
    \begin{eqnarray*}
         M_j(n) - m_j(n) &=& \max\limits_{i,k\in S}\Bigg(\sum\limits_{l\in I_{i,k}^+}\left(p_{i,l}(n_0) - p_{k,l}(n_0)\right)p_{l,j}(n-n_0)+ \\
         &&\quad\quad\quad + \sum\limits_{l\in I_{i,k}^-}\left(p_{i,l}(n_0) - p_{k,l}(n_0)\right) p_{l,j}(n-n_0)\Bigg)\le
    \end{eqnarray*}
    \vspace{-0.5cm}
    \begin{eqnarray*}
         {\color{white} M_j(n) - m_j(n)}&\le& \max\limits_{i,k\in S}\Bigg(\sum\limits_{l\in I_{i,k}^+}\left(p_{i,l}(n_0) - p_{k,l}(n_0)\right)M_j(n-n_0)\,+
    \end{eqnarray*}
    \vspace{-0.5cm}
    \begin{eqnarray*}
         {\color{white} M_j(n) - m_j(n)}&&\quad\quad\quad + \sum\limits_{l\in I_{i,k}^-}\left(p_{i,l}(n_0) - p_{k,l}(n_0)\right)m_j(n-n_0)\Bigg).
    \end{eqnarray*}
    Введем обозначения: 
    $$A_{i,k} = \sum\limits_{l\in I_{i,k}^+}\left(p_{i,l}(n_0) - p_{k,l}(n_0)\right) \text{ и }  B_{i,k} = \sum\limits_{l\in I_{i,k}^-}\left(p_{i,l}(n_0) - p_{k,l}(n_0)\right).$$ Заметим, что
    \begin{eqnarray*}
        A_{i,k} + B_{i,k} &=& \sum\limits_{l\in S}\left(p_{i,l}(n_0) - p_{k,l}(n_0)\right)=\\
        &=& \sum\limits_{l\in S}p_{i,l}(n_0) - \sum\limits_{l\in S}p_{k,l}(n_0)=\\
        &=& 1 - 1 = 0 \Longrightarrow A_{i,k} = - B_{i,k}.
    \end{eqnarray*}
    Поэтому
    \begin{eqnarray*}
         M_j(n) - m_j(n) &\le& \left(M_j(n-n_0) - m_j(n-n_0)\right)\max\limits_{i,k\in S} A_{i,k}.
    \end{eqnarray*}
    Наконец, оценим сверху $A_{i,k}$. Рассмотрим два случая: когда $j_0\in I_{i,k}^+$ и когда $j_0\in I_{i,k}^-$. В первом случае $A_{i,k}$ оценивается сверху следующим образом:
    \begin{eqnarray*}
        A_{i,k} &=& \underbrace{\sum\limits_{l\in I_{i,k}^+}p_{i,l}(n_0)}_{\le 1} - \underbrace{\sum\limits_{l\in I_{i,k}^+}p_{k,l}(n_0)}_{\ge p_{k,j_0}(n_0)} \le 1 - p_{k,j_0}(n_0) \le 1 - \delta. 
    \end{eqnarray*}
    Во втором случае немного другая выкладка приводит к той же самой оценке сверху:
    \begin{eqnarray*}
        A_{i,k} &=& \sum\limits_{l\in I_{i,k}^+}\left(p_{i,l}(n_0) - p_{k,l}(n_0)\right) \le \sum\limits_{l\in I_{i,k}^+}p_{i,l}(n_0)\le\\
        &\le& 1 - \underbrace{\sum\limits_{l\in I_{i,k}^-}p_{i,l}(n_0)}_{\ge p_{i,j_0}(n_0)} \le 1 - p_{i,j_0}(n_0) \le 1-\delta.
    \end{eqnarray*}
    Возвращаясь к полученной оценке на разность $M_j(n)$ и $m_j(n)$, получаем
    \begin{eqnarray*}
        M_j(n) - m_j(n) &\le& \left(M_j(n-n_0) - m_j(n-n_0)\right)(1-\delta).
    \end{eqnarray*}
    Отметим, что данное неравенство было доказано для произвольного $n > n_0$. Следовательно, верно и следующее неравенство:
    \begin{eqnarray*}
        M_j(n) - m_j(n) &\le& \left(M_j(n-n_0) - m_j(n-n_0)\right)(1-\delta)\le\\
        &\le& \left(M_j(n-2n_0) - m_j(n-2n_0)\right)(1-\delta)^2\le\\
        &\le& \ldots\\
        &\le& \underbrace{\left(M_j(n\mod n_0) - m_j(n\mod n_0)\right)}_{\le 1 - 0 = 1}(1-\delta)^{[n/n_0]}\le\\
        &\le& (1-\delta)^{[n/n_0]}.
    \end{eqnarray*}
    Отметим, что полученное неравенство при $n<n_0$ становится тривиальным. Отсюда следует, что $M_j(n) - m_j(n) \xrightarrow[n\to\infty]{} 0$, т.е. существуют числа $$p_j^* = \lim\limits_{n\to\infty}p_{i,j}(n) = \lim\limits_{n\to\infty}M_j(n) = \lim\limits_{n\to\infty}m_j(n) \ge 0,$$ причем, как было отмечено в начале доказательства, $p_{j_0}\elena{^*} > 0$. Осталось теперь заметить, что для всех $i,j\in S$ и всех $n\ge 1$ числа $p_{i,j}(n)$ и $p_j^*$ лежат на отрезке $[m_j(n), M_j(n)]$. Поэтому из полученной оценки на $M_j(n) - m_j(n)$ следует, что
    \[
        |p_{i,j}(n) - p_j^*| \le (1-\delta)^{[n/n_0]},
    \]
    что и требовалось доказать. \EndProof
\end{proof}

\elena{\textbf{Замечание 1.} Согласно теореме~\ref{th:technical_ergodic}  скорость сходимости к стационарному распределению экспоненциальная (иначе, геометрическая). Для конечных цепей условия этой теоремы всегда выполнены, если цепь содержит ровно один замкнутый класс эквивалентности (в частности, неразложима) и апериодична (см. далее теорему~\ref{Th:CritErg}). Однако, в этом случае оценка скорости сходимости, приведённая в теореме~\ref{th:technical_ergodic} как правило оказывается грубой. Более аккуратно: оценка скорости сходимости определяется вторым по модулю максимальным собственным значением стохастической матрицы переходных вероятностей $P$:
\begin{equation}
    \label{exp converges}
    |p_{i,j}(n) - p_j^*| = O\left( (1-\delta)^{n}\right),
\end{equation}
где
$$
\delta = \min\left\{ 1 - |\lambda|,\text{ где $\lambda$ -- собственное значение $P$, отличное от 1}  \right\}
$$
-- спектральная щель матрицы $P$ (см. рис.~\ref{fig: spectr P}). Действительно, предположим, что матрица переходных вероятностей неразложимой апериодической (то есть сильно эргодической) ОДМЦ $P$ размера $N+1$ является диагонализуемой, то есть имеет $N+1$ линейно независимых собственных векторов $u_0$,  $u_1$, \ldots $u_N$ соответствующих (необязательно попарно различным) собственным значениям  $\lambda_0$, $\lambda_1$, \ldots $\lambda_N$, (при чём в силу неразложимости и апериодичности $|\lambda_k|< 1$ для всех $k\ge 1$, поскольку кратность собственного значения $1$ равна числу замкнутых классов эквивалентности, а при наличие периода $d$ на единичной окружности имеются $d$ собственных значений $e^{i\frac{2\pi}{d}m}$, $m=0,1,\ldots,d-1$). Аналогично матрица $P^T$ имеет $N+1$ линейно независимых собственных векторов $v_0$,  $v_1$, \ldots $v_N$ соответствующих  тем же собственным значениям. Поскольку собственные векторы определяются с точностью до множителя, выберем их так, чтобы при этом для любых $i$, $j$ $u_i^Tv_j=\delta_{ij}$. Выберем $v_0 = \pi$ (в действительности, как уже было сказано $v_0 = c \pi$ для любого $c\neq 0$). Тогда $u_0 = \mathbf{1} $ -- вектор из одних единиц.
Тогда
\begin{equation*}
\begin{split}
    &P =\sum_{k=0}^N \lambda_k u_k v_k^T,\\
       & P^n=\sum_{k=0}^N \lambda_k^n u_k v_k^T = \Pi + \sum_{k=1}^N \lambda_k^n u_k v_k^T,
\end{split}
\end{equation*}
где $\Pi = \mathbf{1} \pi^T$ -- матрица, у которой по строкам стоит стационарное распределение вероятностей. Тогда
\begin{equation*}
    |p_{i,j}(n) - \pi_j|\le  \sum_{k=1}^N |\lambda_k|^n | u_{k,i} | | v_{k,j} | = O\left(  N (1-\delta)^n \right).
\end{equation*}
}

\elena{В общем случае недиагонализуемых матриц тоже можно выписать формулы для степеней матрицы переходных вероятностей (см., например, формулу Перрона в~\cite[5.I]{Romanovsky}). Но мы только отметим важный частный случай диагонализуемых матриц -- матрицы переходных вероятностей для обратимых цепей (см. определение~\ref{reversible chain}). В этом случае все  собственные значения действительны и попарно различны (то есть спектр лежит на отрезке $[-1, 1]$, а в предположении апериодичности -- на полуоткрытом интервале $(-1,1]$), а собственные векторы действительны и образуют ортонормированный базис. Это следует из теории симметричных матриц и  того факта, что матрица $A$ с элементами $A_{i,j} = \sqrt{\frac{\pi_i}{\pi_j}} P_{i,j}$ является симметричной (см.~\cite[Теорема~1.12.4]{KelbertSukhov2010}).}

\begin{figure}[!h]
	\centering
	\tikzset{every picture/.style={line width=0.75pt}} 

    \begin{tikzpicture}[x=0.75pt,y=0.75pt,yscale=-1,xscale=1]
    
    \draw   (110,147.25) .. controls (110,80.29) and (164.29,26) .. (231.25,26) .. controls (298.21,26) and (352.5,80.29) .. (352.5,147.25) .. controls (352.5,214.21) and (298.21,268.5) .. (231.25,268.5) .. controls (164.29,268.5) and (110,214.21) .. (110,147.25) -- cycle ;
    \draw  [dash pattern={on 4.5pt off 4.5pt}] (138.88,147.25) .. controls (138.88,96.23) and (180.23,54.88) .. (231.25,54.88) .. controls (282.27,54.88) and (323.63,96.23) .. (323.63,147.25) .. controls (323.63,198.27) and (282.27,239.63) .. (231.25,239.63) .. controls (180.23,239.63) and (138.88,198.27) .. (138.88,147.25) -- cycle ;
    \draw    (352.5,147.25) -- (231.25,147.25) (231.5,151.25) -- (231.5,143.25) ;
    \draw  [draw opacity=0] (324.08,147.16) .. controls (326.16,141.81) and (331.53,137.99) .. (337.83,137.99) .. controls (344.28,137.99) and (349.75,141.98) .. (351.72,147.53) -- (337.83,152.06) -- cycle ; \draw   (324.08,147.16) .. controls (326.16,141.81) and (331.53,137.99) .. (337.83,137.99) .. controls (344.28,137.99) and (349.75,141.98) .. (351.72,147.53) ;
    \draw  [fill={rgb, 255:red, 0; green, 0; blue, 0 }  ,fill opacity=1 ] (248.5,147) .. controls (248.5,145.34) and (249.84,144) .. (251.5,144) .. controls (253.16,144) and (254.5,145.34) .. (254.5,147) .. controls (254.5,148.66) and (253.16,150) .. (251.5,150) .. controls (249.84,150) and (248.5,148.66) .. (248.5,147) -- cycle ;
    \draw  [fill={rgb, 255:red, 0; green, 0; blue, 0 }  ,fill opacity=1 ] (349.5,147.25) .. controls (349.5,145.59) and (350.84,144.25) .. (352.5,144.25) .. controls (354.16,144.25) and (355.5,145.59) .. (355.5,147.25) .. controls (355.5,148.91) and (354.16,150.25) .. (352.5,150.25) .. controls (350.84,150.25) and (349.5,148.91) .. (349.5,147.25) -- cycle ;
    \draw  [fill={rgb, 255:red, 0; green, 0; blue, 0 }  ,fill opacity=1 ] (235.5,216) .. controls (235.5,214.34) and (236.84,213) .. (238.5,213) .. controls (240.16,213) and (241.5,214.34) .. (241.5,216) .. controls (241.5,217.66) and (240.16,219) .. (238.5,219) .. controls (236.84,219) and (235.5,217.66) .. (235.5,216) -- cycle ;
    \draw  [fill={rgb, 255:red, 0; green, 0; blue, 0 }  ,fill opacity=1 ] (138.5,169) .. controls (138.5,167.34) and (139.84,166) .. (141.5,166) .. controls (143.16,166) and (144.5,167.34) .. (144.5,169) .. controls (144.5,170.66) and (143.16,172) .. (141.5,172) .. controls (139.84,172) and (138.5,170.66) .. (138.5,169) -- cycle ;
    \draw  [fill={rgb, 255:red, 0; green, 0; blue, 0 }  ,fill opacity=1 ] (158.5,157) .. controls (158.5,155.34) and (159.84,154) .. (161.5,154) .. controls (163.16,154) and (164.5,155.34) .. (164.5,157) .. controls (164.5,158.66) and (163.16,160) .. (161.5,160) .. controls (159.84,160) and (158.5,158.66) .. (158.5,157) -- cycle ;
    \draw  [fill={rgb, 255:red, 0; green, 0; blue, 0 }  ,fill opacity=1 ] (218.5,67) .. controls (218.5,65.34) and (219.84,64) .. (221.5,64) .. controls (223.16,64) and (224.5,65.34) .. (224.5,67) .. controls (224.5,68.66) and (223.16,70) .. (221.5,70) .. controls (219.84,70) and (218.5,68.66) .. (218.5,67) -- cycle ;
    \draw  [fill={rgb, 255:red, 0; green, 0; blue, 0 }  ,fill opacity=1 ] (268.5,127) .. controls (268.5,125.34) and (269.84,124) .. (271.5,124) .. controls (273.16,124) and (274.5,125.34) .. (274.5,127) .. controls (274.5,128.66) and (273.16,130) .. (271.5,130) .. controls (269.84,130) and (268.5,128.66) .. (268.5,127) -- cycle ;
    \draw    (321.5,63) -- (321.5,71) ;
    \draw [shift={(321.5,67)}, rotate = 90] [color={rgb, 255:red, 0; green, 0; blue, 0 }  ][line width=0.75]    (-5.59,0) -- (5.59,0)(0,5.59) -- (0,-5.59)   ;
    \draw    (281.5,149.33) -- (355.5,222.33) ;
    \draw    (371.17,101.33) -- (340.17,136.33) ;
    \draw    (329.83,20.33) -- (302.83,46.33) ;
    
    \draw (168,95.4) node [anchor=north west][inner sep=0.75pt]    {$|\lambda |\leq \delta ( P)$};
    \draw (227,152.4) node [anchor=north west][inner sep=0.75pt]    {$0$};
    \draw (330,56.4) node [anchor=north west][inner sep=0.75pt]    {$\varkappa $};
    \draw (364.5,139.73) node [anchor=north west][inner sep=0.75pt]    {$\lambda _{0} =1$};
    \draw (357.5,214.73) node [anchor=north west][inner sep=0.75pt]    {$\rho ( P) =1$};
    \draw (373.5,90.73) node [anchor=north west][inner sep=0.75pt]    {$\delta ( P) ,$};
    \draw (330.83,9.33) node [anchor=north west][inner sep=0.75pt]   [align=left] {спектральный радиус};
    \draw (411.83,94.33) node [anchor=north west][inner sep=0.75pt]   [align=left] {спектральная щель};
    \end{tikzpicture}
    \caption{$P$ неприводима и апериодична}
    \label{fig: spectr P}
\end{figure}


\begin{theorem}[ (критерий сильной эргодичности)]\label{Th:CritErg}
\textit{Конечная ОДМЦ является сильно эргодической тогда и только тогда, когда она неразложима и апериодична. }
\end{theorem}
\begin{proof}
    Сначала покажем, что неразложимость и апериодичность являются необходимыми условиями сильной эргодичности конечной однородной цепи Маркова. Итак, рассмотрим произвольную конечную ОДМЦ с множеством состояний $S$ и переходными вероятностями $p_{i,j}$. По определению сильной эргодичности имеем
    \[
        \forall i,j \; \lim\limits_{n\to\infty} p_{i,j}(n) = p_j^*  > 0.
    \]
    В частности, отсюда следует, что
    \[
        \forall i,j \;\forall \varepsilon > 0\; \exists n_0 = n_0(i,j,\varepsilon):\;\forall n > n_0 p_{i,j}(n) > p_j^* - \varepsilon.
    \]
    Пусть
    \begin{center}
    $\overline{\varepsilon} = \frac{1}{2}\min\limits_{j\in S}p_j$ и $\overline{n}_0 = \max\limits_{i,j\in S}n_0(i,j,\overline{\varepsilon})$.
      \end{center}
      Тогда для всех $n > \overline{n}_0$ и для всех $i,j$ вероятность $p_{i,j}(n) > p_j^* - \overline{\varepsilon} > 0$. Отсюда следует, что все состояния сообщаются, то есть цепь неразложима. Кроме того, цепь апериодична, т.к. для всех $n > \overline{n}_0$ и для всех $i$ вероятности $p_{i,i}(n) > 0$ и $p_{i,i}(n+1) > 0$.
    
    Теперь с помощью теоремы~\ref{th:technical_ergodic} покажем, что неразложимость и апериодичность являются и достаточными условиями для сильной эргодичности конечной ОДМЦ. Идея состоит в том, чтобы показать, что условие теоремы~\ref{th:technical_ergodic} выполняется для всех вершин такой цепи. Если это показать, то получим, что для всех $i,j\in S$ существуют пределы $$\lim_{n\to\infty} p_{i,j}(n) = p_j^* > 0,$$ не зависящие от $i$, что и означает сильную эргодичность цепи.
    
    Итак, рассмотрим, не умаляя общности, вершину $s$ (считаем, что $S = \{1,2,\ldots,s\}$). Из неразложимости следует, что существуют такие числа $M_1, M_2, \ldots, M_{s-1}$, что $$p_{1,s}(M_1) > 0, p_{2,s}(M_2) > 0, \ldots, p_{s-1,s}(M_{s-1}) > 0$$ (для удобства будем считать, что ${M_s = 0}$). Кроме того, из апериодичности следует, что найдутся два числа $A$ и $B$, что $A$ и $B$ взаимно просты, и $p_{s,s}(A) > 0$ и $p_{s,s}(B) > 0$. Следовательно, чтобы выполнялось условие теоремы~\ref{th:technical_ergodic}, достаточно доказать, что существуют такие неотрицательные целые числа $n_0,k_1,m_1,k_2,m_2,\ldots, k_s,m_s$, что $n_0 = M_1 + Ak_1 + Bm_1 = M_2 + Ak_2 + Bm_2 = \ldots = M_{s-1} + $\linebreak $+Ak_{s-1} + Bm_{s-1} = Ak_s + Bm_s$ (ведь в таком случае $p_{i,s}(M_i+Ak_i + $\linebreak $+ Bm_i) \ge p_{i,s}(M_i) p_{s,s}(Ak_i) p_{s,s}(Bm_i) \ge p_{i,s}(M_i) \left(p_{s,s}(A)\right)^{k_i} \left( p_{s,s}(B) \right)^{m_i} > $\linebreak $> 0$ для всех $i\in S$, а в силу конечности множества $S$ можно заключить, что найдется такое $\delta > 0$, что $p_{i,s}(M_i+Ak_i + Bm_i) \ge \delta$ для всех $i\in S$). Это эквивалентно системе линейных уравнений в целых числах:
    \begin{eqnarray*}
        A(k_2-k_1) + B(m_2-m_1) &=& M_1 - M_2,\\
        A(k_3-k_2) + B(m_3-m_2) &=& M_2 - M_3,\\
        \ldots&&\\
        A(k_s - k_{s-1}) + B(m_s - m_{s-1}) &=& M_{s-1}.
    \end{eqnarray*}
    Так как $A$ и $B$ взаимно просты, то уравнение $Ax + By = 1$ разрешимо в целых числах. Пусть $(x,y)$~--- решение этого уравнения. Тогда
    \begin{eqnarray*}
        k_2-k_1 = (M_1 - M_2)x,& m_2-m_1 = (M_1 - M_2)y,\\
        k_3-k_2 = (M_2 - M_3)x,& m_3-m_2 = (M_2 - M_3)y,\\
        \ldots&\\
        k_s-k_{s-1} = M_{s-1}x,& m_s-m_{s-1} = M_{s-1}y.
    \end{eqnarray*}
    Из полученных уравнений можно выразить все $k_i,m_i$ для $i\ge 2$ через $k_1,m_1$. Если выбрать $k_1$ и $m_1$ достаточно большими, то все $k_i,m_i$ для $i\ge 2$ получатся неотрицательными. Но это и означает, как было отмечено выше, что ${\exists \delta > 0\; \exists n_0:\forall i\in S\; p_{i,s}(n_0) \ge \delta}$. Поэтому из теоремы~\ref{th:technical_ergodic} получаем, что существует такое число $p_s^* > 0$, что предел $\lim_{n\to\infty}p_{i,s}(n) = p_s^*$ для всех $i\in S$. Аналогично можно доказать, что существуют числа $p_j^* > 0$, что $\lim_{n\to\infty}p_{i,j}(n) = p_j^*$ для всех $i\in S$.
    \EndProof
\end{proof}

Приведем еще несколько очевидных фактов.

а) В конечной сильно эргодической цепи существует такой шаг $N$, что $\forall n\ge N$ вероятности $p_{i,j}(n)>0$ для всех $i,\,j$. Другими словами, начиная с некоторого момента все элементы матрицы перехода станут строго положительными.

б) Если цепь неразложимая и на диагонали переходной матрицы стоит хотя бы один положительный элемент, то цепь является апериодической. Действительно, в этом случае для соответствующего состояния ${f_i(1)=p_{i,i}>0}$, поэтому наибольший общий делитель множества ${\{n\ge1:f_i(n)>0\}}$ равен 1 и состояние $i$ будет апериодическим. По свойству солидарности все состояния цепи будут апериодическими. \elena{ Для конечной цепи это будет означать сильную эргодичность.}

в) Если в какой-то момент все элементы переходной матрицы строго положительны, то все состояния являются сообщающимися и цепь является поэтому неразложимой. А так как на диагонали переходной матрицы стоят положительные элементы, то цепь является апериодической. Для конечной цепи это будет означать сильную эргодичность.

Из этих фактов вытекает следующая теорема.

\begin{theorem}
\textit{Конечная цепь Маркова является сильно эргодической тогда и только тогда, когда }$$\exists n_0\ge1 \ \ \forall i,j\in E \ \ p_{i,j}(n_0)>0.$$
\end{theorem}

В случае ОДМЦ со счетным множеством состояний ситуация нем\-ного усложняется. Мы уже видели, что в счетных цепях может вовсе не быть стационарного распределения (например, в простом блуждании на целочисленной решетке). Поэтому для эргодичности счетных ОДМЦ требуется дополнительное условие (которое автоматически выполняется в случае конечных ОДМЦ). А именно, справедлива следующая

\begin{theorem} [ (см.~{\cite[теорема 3.5]{NGG2}})] 
\textit{В неразложимой  ОДМЦ либо все состояния нулевые, либо все состояния ненулевые. ОДМЦ со счетным числом состояний является сильно эргодической тогда и~только тогда, когда она неразложима, апериодична и положительно возвратна} (\textit{то есть у нее все состояния ненулевые}). 
\end{theorem}


\textbf{Замечание 2.} В случае конечных неразложимых ОДМЦ все состояния всегда являются положительно возвратными. Проверять положительную возвратность счетных цепей можно с помощью существования стационарного распределения. Справедлива следующая
\begin{theorem}
\textit{ОДМЦ является сильно эргодической тогда и только тогда, когда цепь апериодична, существует единственное стационарное распределение, все компоненты этого стационарного р\ag{ac}пре\-де\-ления  строго больше нуля.}
\end{theorem}

\textbf{Замечание.} Единственность стационарного распределения с ненулевыми компонентами равносильна неразложимости цепи (то есть все состояния сообщаются, иначе говоря,  образуют единственный класс эквивалентности).

\gav{Выше мы ввели определение эргодичности конечной ОДМЦ (см. определение \ref{def:ergodicity}), обеспечивающее выполнение условия~\eqref{GeneralChesaro}. Определение по сути свелось к требованию единственности стационарного распределения. Для счетных ОДМЦ ситуация сложнее. Если под эргодичностью по-прежнему понимать выполнение условия~\eqref{GeneralChesaro} с ограниченной функцией $f(k)$, то единственность стационарного распределения является только необходимым условием эргодичности. Более того, счетные ОДМЦ могут вообще не иметь стационарных распределений, в отличие от конечных ОДМЦ. Ранее мы уже это упоминали. Разберем теперь эту ситуацию более подробно.}

В следующем примере рассматриваемая \gav{счетная} ОДМЦ описывает так называемый процесс рождения/гибели в дискретном времени.

\begin{example}
\label{birth-death discr}
Рассмотрим простое случайное блуждание на неотрицательной части одномерной целочисленной решетки с частично отражающим экраном в состоянии $0$: $p_{i,i+1} = p$, $p_{i,i-1} =q= 1-p$, $i \ge 1$, $p\in (0,1)$, $p_{0,1} = p_\circ$, $p_{0,0} = q_\circ=1 - p_\circ $, $p_\circ\in (0,1).$


\begin{figure}[!h]
	\centering
	\includegraphics[scale=0.55]{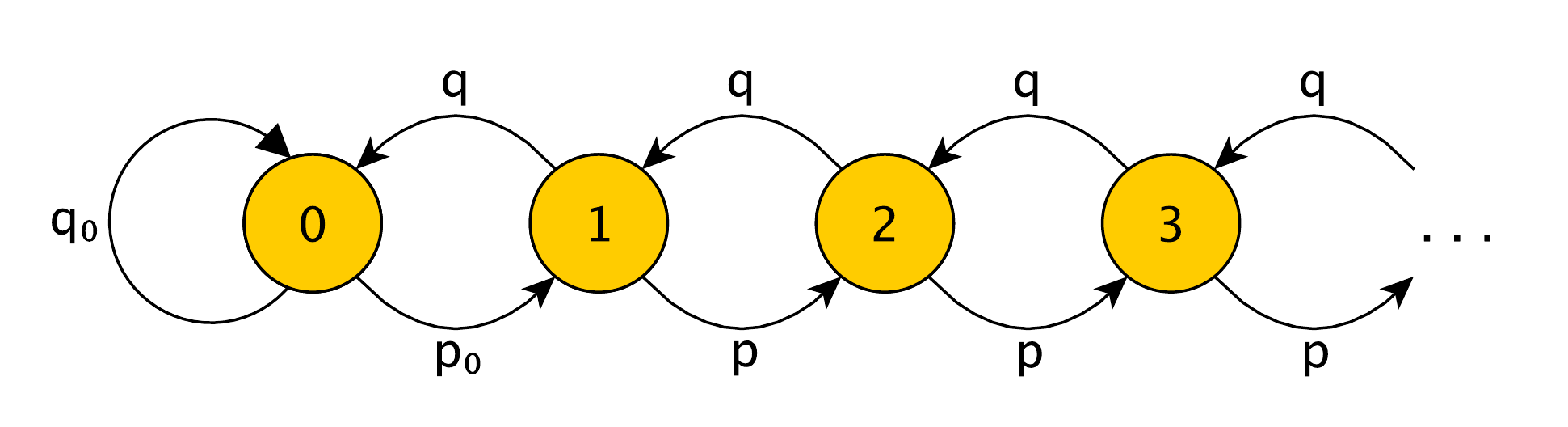}
	\label{fig:BirthDeathBorder}
	\caption{К примеру~\ref{birth-death discr}}
\end{figure}

Все состояния сообщаются, то есть все множество состояний образует единственный класс эквивалентности, значит цепь неразложима. Наличие <<петельки>> в состоянии $0$ говорит об апериодичности цепи.

Попробуем найти стационарное распределение.
\begin{equation*}
\left \{
\begin{aligned}
        &\pi_k = p \pi_{k-1} + q \pi_{k+1}, \quad k \ge 2,\\
    &\pi_1 = p_\circ \pi_0 + q \pi_2,\\
    &\pi_0 = q_\circ \pi_0 + q \pi_1.
\end{aligned}
    \right .
\end{equation*}
Несложно увидеть, что оно существует только при условии $p<1/2$: $\pi_k = A \left ( \frac{p}{q}\right)^k$, $k\ge 1$, $\pi_0 = A \frac{p}{p_\circ}$, где нормировочная константа $A =$\linebreak $= \frac{p_\circ (q-p)}{p(p_\circ + q-p)}$. То есть эта ситуация \gav{сильно} эргодическая.

В случае $p>1/2$ формальное решение уравнения инвариантности то же: $\pi_k = A \left ( \frac{p}{q}\right)^k$, $k\ge 1$, $\pi_0 = A \frac{p}{p_\circ}$, но теперь $\frac{p}{q}>1$, поэтому такое решение не может давать распределение вероятностей. В случае $p=1/2$ формальное решение уравнения стационарности $\pi_k = A$, $k\ge 1$, $\pi_0 = A / (2 p_\circ)$,
что тоже не может давать вектор распределения вероятностей никаким выбором нормировочной константы $A$.
Значит, при $p\ge 1/2$ цепь неэргодична.

Убедимся, что при $p> 1/2$ цепь невозвратна. Интуитивно это понятно: есть тенденция <<ухода>> вправо, на бесконечность. Вычислим вероятность вернуться в состояние $0$ за конечное число шагов $f_0$. Согласно формуле полной вероятности 
$$
f_0 = p_\circ f_{1,0} + q_\circ,
$$
где $f_{1,0}$ -- вероятность достигнуть состояния $0$, стартуя из состояния $1$. Но $f_{1,0}$ есть не что иное, как вероятность разорения игрока в игре в <<орлянку>> с казино, если начальный капитал равен $1$. Если обозначить $\rho_k$ -- вероятность разорения при начальном капитале $k$, то из формулы полной вероятности имеем
$$
\rho_k = p \rho_{k+1} + q \rho_{k-1}, \quad k\ge 1,
$$
с граничным условием $\rho_0 = 1$. Если $q < p$, то $\rho_k = \left (\frac{q}{p}  \right)^k$, если $q \ge p$, то $\rho_k = 1$. Поэтому в нашей задаче при $p > 1/2$ цепь невозвратна, а при $p \le 1/2$ цепь возвратна. Но, как мы уже убедились, цепь положительно возвратна (состояния возвратные, ненулевые) только при $p < 1/2$, так как в этой ситуации  имеется стационарное решение. При $p=1/2$ цепь нуль-возвратна (состояния возвратные, нулевые). \gav{\EndEx}
\end{example}

Как уже ранее отмечалось, к эргодическим теоремам также относят утверждения типа <<усреднение по времени на большом временном интервале совпадает с усреднением по пространству состояний относительно стационарного распределения>>. Приведем далее одну из таких теорем, которую можно также понимать, как переформулировку формулы~\eqref{Chesaro}.

Введем случайные величины обозначающие число посещений ОДМЦ состояния $i$ до момента времени $n$:
$$
V_i(n) = \sum_{k=0}^n \mathsf{I}(X_k = i).
$$


Если существует предел
$$
\lim\limits_{n\to\infty} \frac{V_i(n)}{n},
$$
то его называют \textit{предельной пропорцией} времени, проведенном в состоянии $i$.
\begin{theorem} \textit{Рассматривается эргодическая ОДМЦ с конечным или счетным числом состояний, т.е. ОДМЦ с единственным стационарным распределением $\pi$.
Для любого состояния $i$ такой цепи выполняется равенство}
$$
P \left( \lim_{n\to\infty} \frac{V_i(n)}{n} = r_i \right) = 1,
$$
\textit{где $r_i = \pi_i$, если состояние положительно возвратно, и $r_i = 0$, если состояние невозвратно или имеет нулевую возвратность.}
\end{theorem}
С доказательством можно ознакомиться в~\cite{KelbertSukhov2010}.

В качестве применения этой теоремы рассмотрим асимптотическое решение следующей задачи.
\begin{example}\label{bandit}
Рассмотрим следующую <<наивную>> (но тем не менее  осмысленную) стратегию задачи о двуруком бандите, см. также раздел~\ref{bandit_problem}. Имеется игральный автомат с двумя ручками. При нажатии первой приз появляется с вероятностью $p_1$ (будем называть это событие успехом). Аналогично при нажатии второй успех будет с вероятностью $p_2$. Стратегия заключается в повторении той же ручки, если на предыдущем шаге она давала успех, и смене ручки, если был неуспех. Ясно, что такая игра представима в виде ОДМЦ на множестве состояний $\{1,2\}$, где номер состояния -- номер ручки на данном шаге игры.


\begin{figure}[!h]
	\centering
	\includegraphics[scale=0.55]{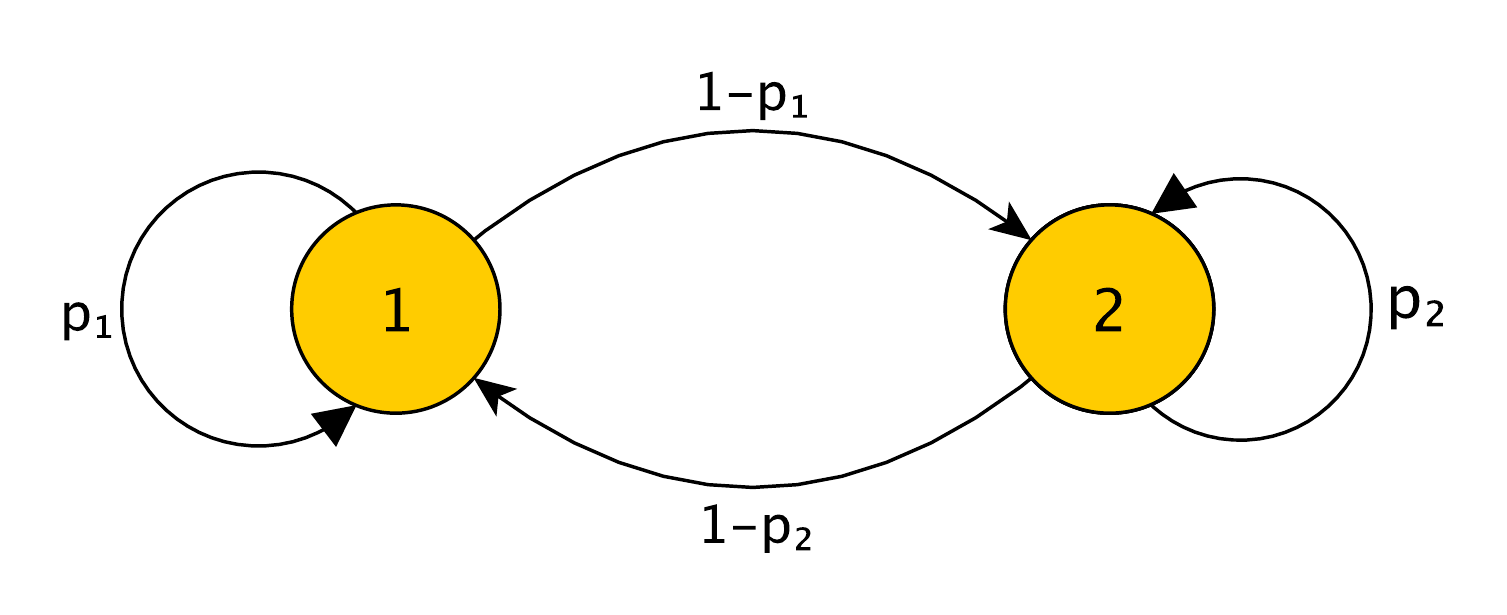}
	\label{fig:TwoArmedBandit}
	\caption{К примеру~\ref{bandit}}
\end{figure}

Вычислим приближенно (когда игра продолжается достаточно долго) среднее число успехов при данной стратегии. Обозначим  $a_1$ -- число успехов при выборе первой ручки, аналогично $a_2$ -- число успехов при выборе второй ручки. Тогда 
$$
a_1 = \sum_{k=1}^n \mathsf{I} \left( X_{k-1} = 1, X_k = 1\right),
$$
$$
a_2 = \sum_{k=1}^n \mathsf{I} \left( X_{k-1} = 2, X_k = 2 \right).
$$
После усреднения получаем
$$
\mathbb{E} a_1 = \sum_{k=1}^n  \mathbb{P} (X_{k-1} = 1, X_k = 1) = p_1  \sum_{k=1}^n  \mathbb{P} \left( X_{k-1} = 1 \right).
$$
Но согласно предыдущей теореме $ \sum_{k=1}^n  \mathbb{P} \left( X_{k-1} = 1 \right) = \mathbb{E} V_1(n) \approx n \pi_1$. Аналогично, для $a_2$. В итоге среднее число успехов приближенно равно
$$
n(p_1 \pi_1 + p_2 \pi_2) = \frac{p_1 + p_2 - 2p_1 p_2}{2-p_1 - p_2}.
$$
Рассмотренная стратегия тем больше отличается от стратегии с равновероятным выбором ручки на каждом шаге, чем больше $|p_1 - p_2|$.~ \EndEx
\end{example}


Сформулируем, наконец, аналоги закона больших чисел и центра\-льной предельной теоремы для ОДМЦ. Введем случайную величину $V_i(n)$, равную числу попаданий цепи $\xi_n$ в фиксированное состояние $i \in E$ за время $n$, т.е. $$V_i(n)=\sum\limits_{k=1}^n \mathsf{I}(\xi_k=i),$$ где $\mathsf{I}(A)$ -- индикаторная функция события $A$. Пусть $\tau_j$ есть время возвращения цепи в состояние $i$, т.е. $$\mathbb{P}(\tau_i = k) = f_i(k).$$ Следующая теорема справедлива как для конечных, так и для счетных цепей~\cite[с. 263]{Borovkov1999}.
\begin{theorem}\label{th:clt_lln}
\textit{Пусть ОДМЦ эргодическая, т.е. существует единственное стационарное распределение $\pi$. Тогда независимо от того, где находилась ОДМЦ в начальный момент, при }$n\to\infty$ $$\frac{\mathbb{E}V_i(n)}{n} \to \pi_i, \ \frac{V_i(n)}{n} \overset{\mathrm{\text{п.н.}}}{\longrightarrow} \pi_i.$$
\textit{Если дополнительно} $\mathbb{D}\tau_i = \sigma_i^2 < \infty$, $\pi_i > 0$, то
$$\mathbb{P}\left(\frac{V_i(n) - n\pi_i}{\sigma_i\sqrt{n\pi_i^3}} < x \,\Big\rvert\, \xi_0 = s \right) \to \Phi(x)$$
\textit{при $n\to\infty$, где $\Phi(x)$ -- функция распределения} $\mathrm{N}(0,1)$.
\end{theorem}

\begin{example}
\label{ex:mu j pi j =1}
Пусть ОДМЦ сильно эргодическая\footnote{В действительности, достаточно предположить единственность стационарного распределения. Результат $\mu_j^{-1}=\pi_j\ge 0$ останется верным.}, т.е. существует единственное стационарное распределение $\pi$ и все компоненты этого распределения положительны. Доказать, что для всех $j\in E$ выполнено $\mu_j^{-1}=\pi_j>0$, где $$\mu_j=\sum\limits_{n=1}^{\infty}nf_j(n)$$ -- среднее время до первого возвращения в состояние $j$.
\end{example}
\begin{solution}
    Рассмотрим случайное блуждание по ОДМЦ, которое начинается в вершине $j$. Пусть $T_j^1$~--- длина первой <<прогулки>>, то есть число шагов, сделанных до первого возвращения в состояние $j$; $T_j^2$~--- длина второй <<прогулки>> и т.д. Получили последовательность $\{T_j^k\}_{k=1}^\infty$ независимых одинаково распределенных случайных величин. Из усиленного закона больших чисел следует, что
    \[
    \frac{\sum\limits_{k=1}^n T_j^k}{n} \xrightarrow[n\to\infty]{\text{п.н.}} \Exp T_j^1 = \mu_j.
    \]
    Из свойств сходимости почти наверное следует, что
    \[
    \frac{n}{\sum\limits_{k=1}^nT_j^k} \xrightarrow[n\to\infty]{\text{п.н.}} \frac{1}{\mu_j}.
    \]
    Заметим, кроме того, что
    \[
    \frac{n}{\sum\limits_{k=1}^{n+1}T_j^k} \xrightarrow[n\to\infty]{\text{п.н.}} \frac{1}{\mu_j}.
    \]
    Рассмотрим следующие случайные величины:
    \[
        \xi_j^k = \begin{cases}1,\text{ если на $k$-м шаге <<прогулка>> вернулась в состояние $j$},\\ 0 \text{ иначе.} \end{cases}
    \]
    Если ввести обозначение $\tilde{N} = \sum_{k=1}^n T_{j}^k$, то нетрудно видеть, что $n =$\linebreak $= \sum_{k=1}^{\tilde{N}}\xi_j^k$. Рассмотрим теперь случайную величину $\tilde{n} = \sum_{k=1}^N\xi_j^k$~--- число возвращений за первые $N$ шагов. Отсюда следует, что
    \[
        \sum\limits_{k=1}^{\tilde{n}}T_j^k \le N \le \sum\limits_{k=1}^{\tilde{n}+1}T_j^k,
    \]
    а значит,
    \[
        \frac{\tilde{n}}{\sum\limits_{k=1}^{\tilde{n}+1}T_j^k} \le \frac{1}{N}\sum\limits_{k=1}^N\xi_j^k \le \frac{\tilde{n}}{\sum\limits_{k=1}^{\tilde{n}}T_j^k}.
    \]
    Теперь заметим, что с вероятностью $1$ $\tilde{n} \to \infty$ при $N\to\infty$. Следовательно,
    \[
        \frac{1}{N}\sum\limits_{k=1}^N\xi_j^k \xrightarrow[N\to\infty]{\text{п.н.}} \frac{1}{\mu_j}.
    \]
    Так как $\frac{1}{\mu_j}$~--- константа, то указанная последовательность сходится и в среднем к $\frac{1}{\mu_j}$:
    \[
        \Exp \left(\frac{1}{N}\sum\limits_{k=1}^N\xi_j^k \right) \xrightarrow[N\to\infty]{} \frac{1}{\mu_j}.
    \]
    Пользуясь линейностью математического ожидания, получим
    \begin{eqnarray*}
        \Exp \left(\frac{1}{N}\sum\limits_{k=1}^N\xi_j^k \right) &=& \frac{1}{N}\sum\limits_{k=1}^N \Exp \xi_j^k=\\
        &=& \frac{1}{N}\sum\limits_{k=1}^N p_{j,j}(k) \xrightarrow[N\to\infty]{} \frac{1}{\mu_j}.
    \end{eqnarray*}
    В силу эргодичности цепи имеем
    \[
    \forall j\in S\; p_{j,j}(k) \xrightarrow[k\to\infty]{} \pi_j,
    \]
    а значит,
    \[
        \frac{1}{N}\sum\limits_{k=1}^N p_{j,j}(k) \xrightarrow[N\to\infty]{} \pi_j,
    \]
    откуда следует, что $\frac{1}{\mu_j} = \pi_j$ для всех $j\in S$.
    \EndEx
    
    \elena{Другое решение см. в замечании 1 раздела 6.2.2.}
\end{solution}

Как уже отмечалось, в случае \textbf{конечных} ОДМЦ эргодичность эквивалентна единственности стационарного распределения (которое и будет предельным). Другими словами, в конечных цепях для эргодичности должен быть только один замкнутый класс эквивалентности, при этом несущественных состояний может \gav{быть} любое конечное число. Это оправдывает термин \textit{несущественных состояний}: они не влияют на предельное поведение ОДМЦ: $\lim_{n\to\infty} p_i(n) = 0$, если $i$-е состояние несущественное.

В случае \textbf{счетных} ОДМЦ ситуация усложняется. Во-первых, напомним, что в случае счетных цепей может и не быть стационарного распределения: существование единственного стационарного распределение эквивалентно существованию и единственности положительно возвратного замкнутого класса эквивалентности. Однако это по-прежнему недостаточно для того, чтобы предельное распределение не зависело от начального. 




\begin{example}[ (см.~{\cite[гл. VII, \S~8, пример 4]{Shiryaev2}})]\label{ex:casino}
\eduard{Рассмотрим следующую игру. Человек изначально имеет $k$ рублей, а на каждом шаге игры подбрасывается несимметричная монетка (вероятность выпадения орла равна $p$, вероятность выпадения решки равна $q = 1-p$). Если выпадает орёл, то человек получает рубль от казино, иначе~--- отдаёт свой рубль. Игра продолжается до тех пор, пока либо человек не разорится, либо не накопит $M$ рублей (можно интерпретировать $M$ как общую сумму денег человека и казино на момент начала игры, если считать, что общее количество денег у человека и казино не меняется во время игры).}
\eduard{Пусть $\xi(t)$~--- количество рублей у человека в момент времени $t$. Последовательность $\{\xi(t)\}_{t=0}^\infty$ является марковской цепью, которая задаётся графом на~рис.~\ref{fig:casino}.}

\begin{figure}[!h]
	\centering
	\includegraphics[scale=0.45]{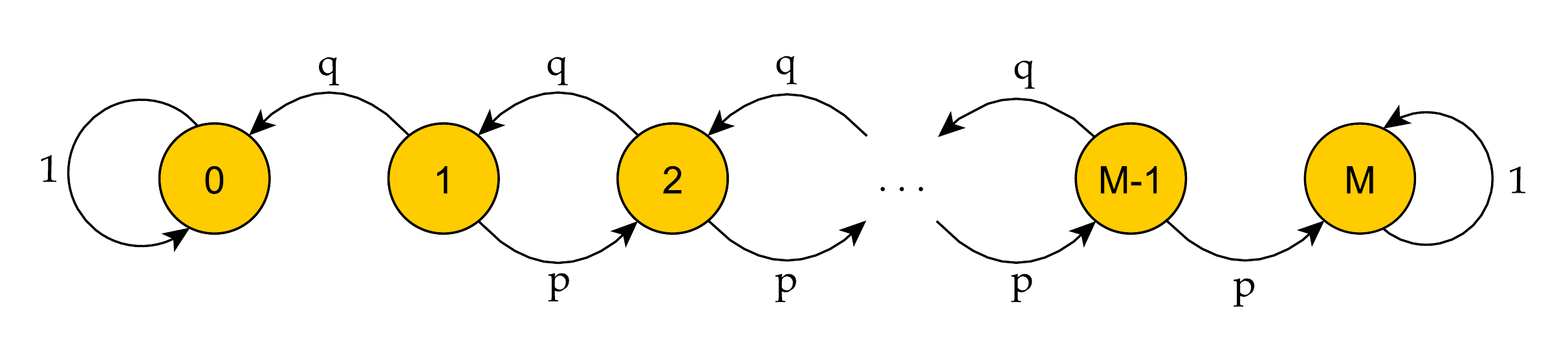}
	\caption{Игра в <<орлянку>> с казино}
	\label{fig:casino}
\end{figure}

\eduard{Во-первых, заметим, что все состояния, кроме $0$ и $M$, образуют несущественный класс с периодом $2$, состояния $0$ и $M$ положительно возвратны и образуют $2$ замкнутых класса. Поэтому в случае $M < \infty$ мы получаем, что цепь не может быть эргодической. Более интересен случай, когда $M = \infty$ (бесконечно богатое казино). Тем не менее случай конечного $M$ будет полезен, чтобы перейти к случаю $M = \infty$.}

\eduard{Рассмотрим теперь $M = \infty$. В этом случае состояния $1,2,\dots$ являются несущественными, а состояние $0$ является положительно возвратным и образует единственный замкнутый класс эквивалентности. Нетрудно заметить, что $\lim_{n\to\infty}p_{i,j}(n) = 0$ для $i,j = 1,2,\dots$. Пусть $\pi_k$~--- вероятность разорения игрока, если изначально у него было $k$ рублей. Так как состояние $0$ является поглощающим, то $p_{k,0}(n) =$\linebreak $= \sum_{m<n}f_{k,0}(m)$, где $f_{k,0}(m)$~--- вероятность, что через $m$ шагов цепь окажется в состоянии $0$ в первый раз, стартуя из состояния $k$. Но тогда получаем, что $$\lim\limits_{n\to\infty}p_{k,0}(n) = \sum\limits_{m=0}^\infty f_{k,0}(m) = \pi_k.$$
Из уравнения Колмогорова--Чэпмена получаем
$$
\pi_k = p\pi_{k+1} + q\pi_{k-1},\quad \pi_0 = 1.
$$
Записывая характеристический многочлен
$$
\lambda = \lambda^2 p + q
$$
и находя его корни (при $p=q=\frac{1}{2}$ корни $\lambda_1 = \lambda_2 = 1$, при $p\neq q$ корни $\lambda_1 = \frac{q}{p}$), получаем общий вид решения:
$$
\pi_k = C_1\left(\frac{q}{p}\right)^k + C_2 \ \text{ при } p\neq q,
$$
$$
\pi_k = C_1k + C_2\ \text{ при } p= q=\frac{1}{2}.
$$
Рассмотрим сначала случай $p=q=\frac{1}{2}$. В силу ограниченности $\pi_k$ для всех $k$ получаем, что $C_1 = 0$. Если учесть начальное условие $\pi_0 = 1$, то получим, что $C_2 = 1$, т.е.\ $\pi_k = 1$ для всех $k$, а значит, не зависит от $k$. В этом случае цепь эргодична. Если $q > p$, то, как и в первом случае, получаем, что из ограниченности $\pi_k$ необходимо, чтобы $C_1 = 0$, а из начального условия следует, что $C_2 = 1$, т.е.\ в случае $q > p$ цепь тоже получается эргодической. Это интуитивно понятно, ведь в случае $q > p$ есть <<тенденция>> движения <<влево>> по графу и в пределе человек разоряется, какова бы не была его стартовая сумма денег. В обоих случаях, разобранных выше, предельное распределение есть вектор $\pi = [1,0,0,\dots]^\top$. Если же $p > q$, то есть тендеция движения вправо, а значит, чем больше было изначально денег у человека, тем меньше вероятность, что он разорится, т.е.\ $\pi_k \to 0$ при $k \to \infty$, а значит, $C_2 = 0$, и из начального условия получаем $\pi_k = \left(\frac{q}{p}\right)^k$. Это означает, что предел $\lim_{n\to\infty}p_{k,0}(n) = \pi_k$ зависит от $k$, т.е.\ цепь не является эргодической.}

\eduard{Докажем теперь формулу $\pi_k = \left(\frac{q}{p}\right)^k$ для $q < p$ формально. Вот тут-то и оказывается полезна ситуация с $M < \infty$. Вернёмся теперь к случаю конечного $M$. Пусть $\pi_{k,M}$~--- это вероятность того, что стартуя из состояния $k$, цепь раньше окажется в состоянии $0$, чем в состоянии $M$ за конечное число шагов. Тогда
$$
\pi_{k,M} = p\pi_{k+1,M} + q\pi_{k-1,M},\ \pi_{0,M} = 1, \ \pi_{M,M} = 0. 
$$
Решая аналогичным способом данное рекуррентное уравнение, получим
$$
\pi_{k,M} = \frac{\left(\frac{q}{p}\right)^k - \left(\frac{q}{p}\right)^M}{1 - \left(\frac{q}{p}\right)^M},\ 0\le i\le M.
$$
Следовательно, $\lim_{M\to\infty}\pi_{k,M} = \left(\frac{q}{p}\right)^k$ (устремляем бюджет казино на бесконечность), т.е.\ нам достаточно показать, что $\pi_k = \lim_{M\to\infty}\pi_{k,M}$. Как и раньше, интуитивно это понятно, но мы покажем это строго. Пусть $A_k$~--- событие, состоящее в том, что найдётся такое $M$, что человек, имея в начале $k$ рублей, разорится раньше, чем накопит $M$ рублей. Понятно, что $\pi_k = \PP(A_k)$. Кроме того, рассмотрим событие $A_{k,M}$~--- событие, состоящее в том, что человек, имея в начале $k$ рублей, разорится раньше, чем накопит $M$ рублей (для конкретного $M$). Тогда $$A_k = \bigcup\limits_{M = k+1}^\infty A_{k,M} \text{ и } A_{k,M} \subseteq A_{k,M+1}.$$ По теореме непрерывности вероятности получаем
$$
\pi_k =\PP(A_k) = \PP\left(\bigcup\limits_{M = k+1}^\infty A_{k,M}\right) = \lim\limits_{M\to\infty}\PP(A_{k,M}) = \lim\limits_{M\to\infty}\pi_{k,M}.
$$
Итак, если $p > q$, то предел $\lim_{n\to\infty}p_{k,0}(n) = \left(\frac{q}{p}\right)^k$ зависит от $k$. \EndEx}
\end{example}

\eduard{В заключение зафиксируем основные моменты, связанные с понятием эргодичности для конечных и счетных марковских цепей.}

\eduard{Начнем с конечных цепей Маркова. Общая картина отражена на рис.~\ref{fig:finite_ergodic}. При этом, как уже было отмечено ранее,
\begin{enumerate}
    \item[1)] в случае непрерывных марковских цепей и апериодичных дискретных марковских цепей $\lim_{t\to\infty}p(t) = \pi$, где $t\in\R_+$ или $t\in \NN$, \item[2)] в случае же периодических ОДМЦ предел нужно понимать в более слабом смысле\footnote{\eduard{Как известно из курса математического анализа, из обычной сходимости числовой последовательности следует и сходимость по Чезаро. Обратное в общем случае неверно.}} (по Чезаро): $\lim_{n\to\infty}\frac{1}{n}\sum_{k=0}^{n-1}p(k) = \pi$.
\end{enumerate}
Если состояние $i$ не принадлежит единственному замкнутому классу эквивалентности $S_*$, то $\pi_i = 0$. Если же $i\in S_*$, то $\pi_i > 0$.}

\begin{figure}
    \centering
    \includegraphics[width=0.5\textwidth]{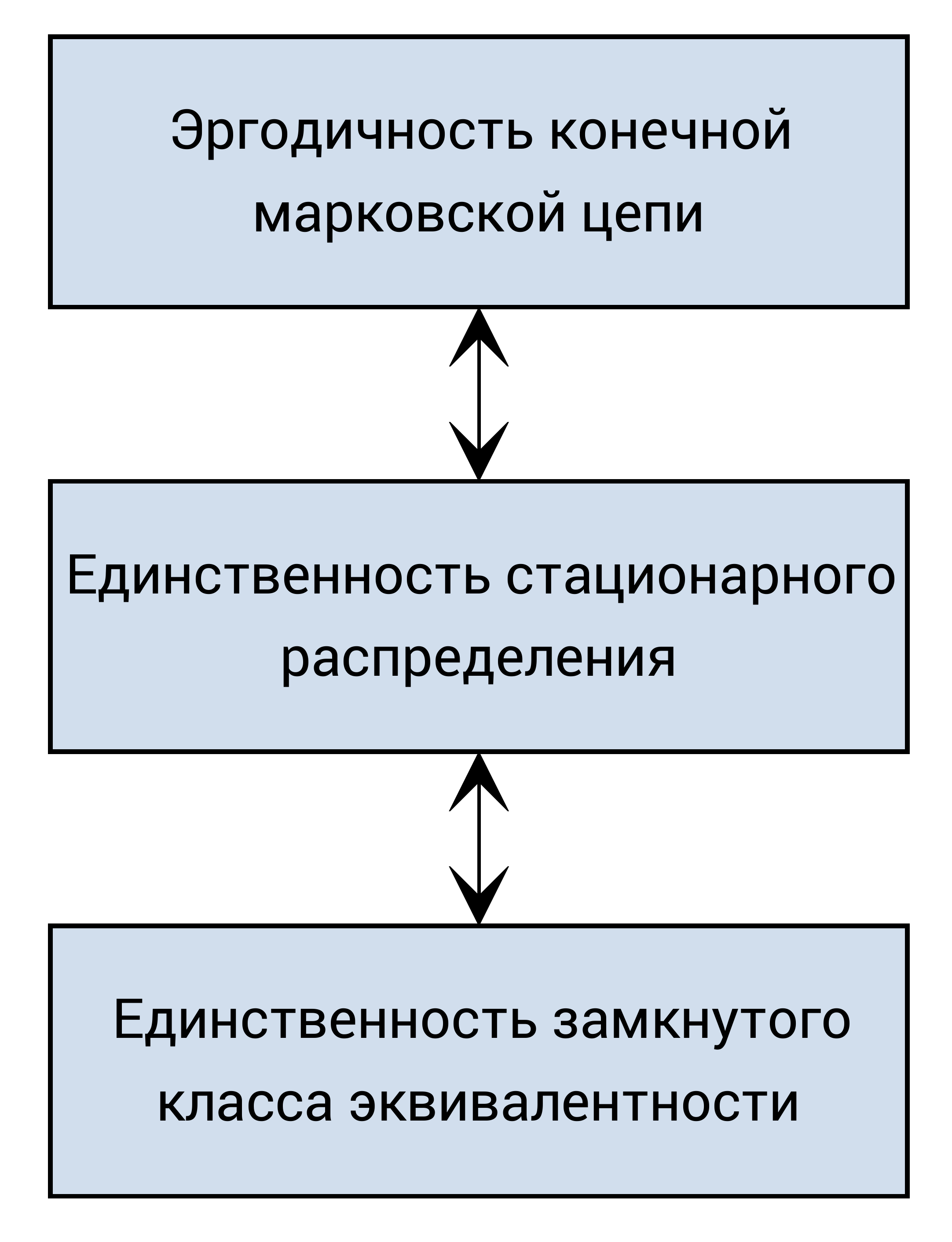}
    \caption{\eduard{Эргодичность для конечных марковских цепей. Стрелки означают логические импликации}}
    \label{fig:finite_ergodic}
\end{figure}

\eduard{\textbf{Алгоритм анализа конечных марковских цепей\\ на эргодичность}}
\eduard{\begin{enumerate}
    \item[1.] Выделить классы эквивалентности.
    \item[2.] Проверить единственность замкнутого класса $S_*$.
    \item[3.] Найти стационарное (инвариантное) распределение на подцепи $S_*$, т.е.\ найти вектор $\widetilde{\pi}\in\R^{|S_*|}$, что $\widetilde{P}^\top \widetilde{\pi} = \widetilde{\pi}$, $\sum_{i\in S_*}\widetilde{\pi}_i = 1$, $\widetilde{\pi}_i > 0$, \mbox{$i\in S_*$} (не умаляя общности, здесь мы считаем, что состояния $1,2,\ldots,|S_*|$ образуют класс $S_*$), где $\widetilde{P} \in \R^{|S_*|\times |S_*|}$, $\widetilde{P}_{ij} = P_{ij}$ $\forall i,j \in S_*$.
    \item[4.] Предельное распределение (в смысле 1) или 2), указанных выше): $$\lim\limits_{t\to\infty}p_i(t) = \begin{cases}\widetilde{\pi}_i, \text{ если } i\in S_*,\\ 0 \text{ иначе.} \end{cases}$$
\end{enumerate}
}

\eduard{Теперь перейдем к вопросу о счетных марковских цепях. Общая картина отражена на рис.~\ref{fig:countable_ergodic}. Сходимость понимается в том же смысле, как для конечных марковских цепей.}

\begin{figure}
    \centering
    \includegraphics[width=0.8\textwidth]{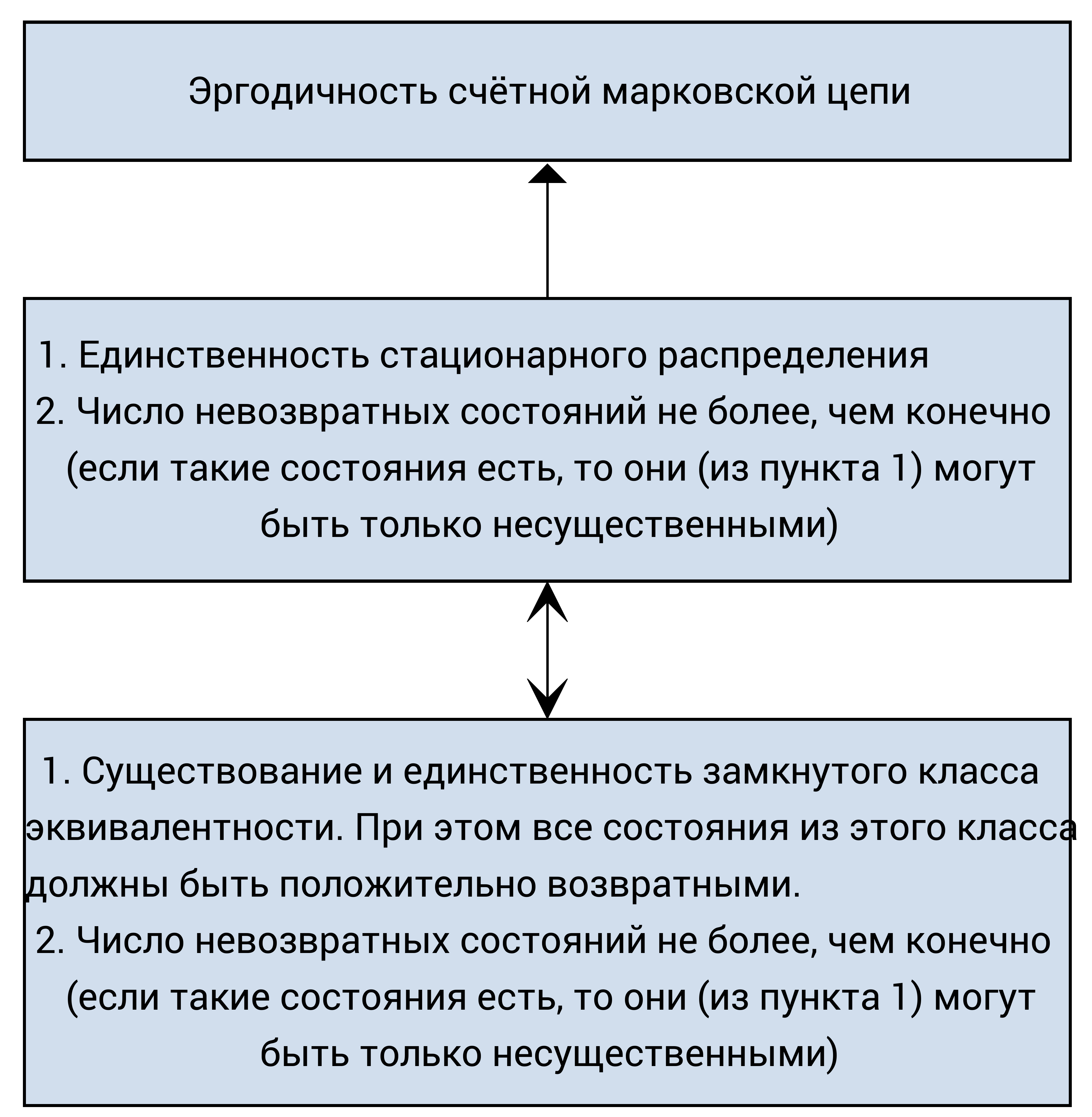}
    \caption{\eduard{Эргодичность для счетных марковских цепей. Стрелки означают логические импликации}}
    \label{fig:countable_ergodic}
\end{figure}

\eduard{\textbf{Алгоритм анализа счетных марковских цепей на эргодичность} полностью аналогичен соответствующему алгоритму для случая конечных марковских цепей, за исключением пункта~2, который в счетном случае будет следующим (курсивом выделены отличия от случая конечной марковской цепи).
\begin{enumerate}
    \item[2.1.] Проверить \textit{существование} и единственность  замкнутого \elena{положительно возвратного} класса $S_*$.
    \item[2.2.] \textit{Проверить конечность числа невозвратных состояний.}
\end{enumerate}}

\eduard{\textbf{Сравнение случаев конечных и счетных марковских цепей}
\begin{enumerate}
\item[1)] В конечных марковских цепях замкнутый класс эквивалентности всегда существует. Более того, любой конечный замкнутый класс является положительно возвратным. В счетных цепях может быть не так (см. пример~\ref{randomWalk}).
\item[2)] Если существует счетное число несущественных невозвратных состояний, цепь может остаться <<жить>> в этом <<плохом>> подмножестве состояний, не попав в замкнутый класс. Примеры с нарушением требования конечности числа невозвратных состояний: пример~\ref{1}, пример~\ref{ex:casino}.
\item[3)] Как для конечных, так и для счетных марковских цепей в случае, когда цепь разложима, финальное распределение всегда есть, однако в таком случае оно зависит от начального, поэтому цепь не является эргодической.
\end{enumerate}}


\section{Непрерывные цепи Маркова}\label{sec:ContChainMarkov}
\setlength{\parskip}{0pt}

\subsection{Базовые понятия и предположения}
\gav{Начнем с обобщения определения \ref{def:DiscrMarkChain} на случай непрерывного времени.}
\begin{definition}
Случайная функция $\{\xi(t),t\ge0\}$, принимающая при каждом $t$ значения из множества ${E=\{0,1,\dots\}}$, конечного или бесконечного, называется \textit{марковским процессом с непрерывным временем и дискретным множеством значений} (или непрерывной цепью Маркова), если для любых ${n\ge 2}$, любых моментов времени $0\le t_1\le\dots\le t_{n-1}\le t_n$ и значений ${x_1,\dots,x_n\in E}$ выполняется равенство $$\mathbb{P}(\xi(t_n)=x_n\,|\,\xi(t_{n-1})=x_{n-1},\dots,\xi(t_1)=x_1)=$$$$=\mathbb{P}(\xi(t_n)=x_n\,|\,\xi(t_{n-1})=x_{n-1}).$$
Как и прежде, предполагается, что это равенство выполнено лишь в тех случаях, когда указанные условные вероятности существуют. Это равенство также называют \textit{марковским свойством} непрерывной цепи Маркова. Множество $E$ называется \textit{множеством состояний}, а его элементы -- \textit{состояниями} цепи.
\end{definition}

\begin{definition}
\textit{Вероятностью перехода из состояния ${i\in E}$ в состояние ${j\in E}$} непрерывной цепи Маркова называется функция двух аргументов: $$p_{i,j}(s,t)=\mathbb{P}(\xi(t)=j\,|\,\xi(s)=i), \ \ 0\le s \le t.$$
Набор чисел $P(s,t)=\{p_{i,j}(s,t)\}$ (конечный или бесконечный) образует \textit{матрицу перехода}. Матрица перехода непрерывной цепи, как и матрица перехода дискретной цепи, обладает следующими очевидными свойствами:
\begin{enumerate}[topsep=0pt,itemsep=5pt]
	\item $\sum_{j\in E}p_{i,j}(s,t)=1$ для всех $i\in E$ и $0\le s \le t,$
	\item $p_{i,i}(t,t)=1$ для всех $i\in E$ и $t\ge0$,
	\item $p_{i,j}(t,t)=0$ для всех $i\ne j$ и $t\ge0$,
	\item справедливо \textit{уравнение Колмогорова--Ч\gav{э}пмена}: $$P(s,t)=P(s,u)P(u,t), \ \ 0\le s\le u \le t.$$
\end{enumerate}
Отметим, что уравнение Колмогорова--Ч\gav{э}пмена прямо вытекает из формулы полной вероятности.
\end{definition}

\begin{definition}
\textit{Вероятностью $i$-го состояния цепи $\xi(t)$ в момент $t\ge0$} называется величина $$p_i(t)=\mathbb{P}(\xi(t)=i).$$ Очевидно, что $${p_i(t)\ge0}, \ {\sum\limits_{i\in E}p_i(t)=1}$$ для любого момента ${t\ge0}$. Набор вероятностей ${p(t)=\{p_i(t)\}}$ называется \textit{распределением вероятностей состояний}. Распределение $p(t)$ удовлетворяет уравнению $$p(t)=P^T(s,t)p(s), \ \ 0\le s \le t,$$ что, как и в случае дискретных цепей Маркова, является прямым следствием применения формулы полной вероятности и марковского свойства.
\end{definition}
\begin{definition}
Непрерывная цепь Маркова с переходной матрицей $P$ называется \textit{однородной}, если для любых $s,t\ge0$ $$P(s,s+t)=P(0,t).$$ Для однородных цепей мы введем также обозначение $P(t)=P(0,t)$, а элементы этой матрицы будем обозначать просто $p_{i,j}(t)$. Всюду далее, если не оговорено противное, непрерывные цепи Маркова будут предполагаться однородными.
\end{definition}

\elena{Стоит сразу отметить, что так же, как и в случае общих случайных процессов в непрерывном времени, может возникнуть проблема с неизмеримостью некоторых множеств, например, $\sup_{t\in[0,1]} \xi(t)$.  В случае дискретного времени таких проблем, конечно, не было, так как соответствующие множества выражались через не более, чем счетное число объединений и пересечений измеримых множеств, а значит были тоже измеримы. Напомним также, что в случае непрерывного времени возникают сложности в <<неоднозначности>> задания случайных процессов по семейству конечномерных распределений, так как появляется возможность <<испортить>> процесс в конечном и даже в счётном множестве случайных моментов времени. Чтобы избежать трудностей такого типа, мы (как и всегда) предполагаем сепарабельность процессов. В частности, для непрерывных цепей Маркова, везде далее предполагается, что реализации суть кусочно-постоянные функции непрерывные справа.}

\elena{Формальное понимание непрерывных цепей связано полугрупповым свойством семейства матриц $P(t)$: $P(t+s) = P(t) P(s)$ для любых $t,s\ge 0$, и требует языка теории групп (инфинитезимального оператора, порождающего группу и т.д.) Мы будем придерживаться более прикладного способа понимания непрерывных марковских цепей. Заметим прежде всего, что время между скачками марковской цепи в непрерывном времени имеет показательный закон распределения. Действительно, пусть в некоторый момент времени $t_0$ цепь находится в состоянии $i$ и случайная величина $\tau_i(t_0)$ есть время ожидания до выхода из этого состояния:
$$
\tau_i(t_0) = \inf\{ u> 0:\, X(u+ t_0)\neq i\}
$$
Пусть в течение интервала наблюдения $[t_0,t_1]$, где $t_1 = t_0+s$, цепь не покидала $i$-го состояния, то есть $\tau_i(t_0) > s $. Тогда в силу однородности и марковости рассматриваемого процесса заключаем, что случайная величина $\tau_i(t_0)$ не зависит от $t_0$ и обладает свойством отсутствия памяти, то есть остаточное время пребывания цепи в $i$-ом состоянии в момент времени $t_1$ (то есть $\tau_i(t_1)$) имеет такой же закон распределения, как и в момент времени $t_0$ (то есть $\tau_i(t_1)$), иначе говоря $\tau_i(t_0) = \tau_i(t_1)=\tau_i$, где равенство понимается как равенство по распределению, и:
$$
\mathbb{P} \left ( \tau_i(t_0)> s+t | \tau_i(t_0) > s \right)= \mathbb{P} \left ( \tau_i(t_1)> t \right) = \mathbb{P} \left ( \tau_i(t_0)> t \right).
$$
В классе невырожденных непрерывных законов распределения этим свойством обладает только показательное распределение (с некоторым параметром, зависящем только от номера состояния):
$$
 \mathbb{P} \left ( \tau_i > t \right) = e^{-\lambda_i t}.
$$
Если $\lambda_i\in(0,\infty)$, то цепь за малый промежуток времени $\Delta t$ изменит свое $i$-е состояние с вероятностью
$$
\lambda_i \Delta t + o(\Delta t),\quad \Delta t\to 0+.
$$
Если $\lambda_i = 0$, то цепь остается в $i$-ом состоянии навсегда (этот вырожденный случай соответствует $\tau_i = \infty$ с вероятностью единица, то есть несобственная случайная величина). В этом случае состояние $i$ называют \textit{поглощающим}. Формально, возможна также ситуация, когда $\lambda_i = \infty$ (этот вырожденный случай соответствует $\tau_i = 0$ с вероятностью единица). В этом случае состояние $i$ называют \textit{мгновенным}, но мы не будем рассматривать такие ситуации (см. ниже). 
}

\elena{ По прошествии случайного времени $\tau_i$ с $\lambda_i$-показательным законом ($\lambda_i\in(0,\infty)$) цепь совершает скачок в новое состояние $j$ с  вероятностью, зависящей только от $i$ и $j$, обозначим эту вероятность (при условии, что цепь совершает прыжок в новое состояние) $\rho_{i,j}$:
$$
\rho_{i,j} = \lim_{\Delta t \to 0+} \mathbf{P} \left (  X(t+\Delta t) = j | X(t) = i,\, X(t+\Delta t) \neq i\right ).
$$
Безусловная вероятность того, что цепь за малый промежуток времени $\Delta t$ окажется в состоянии $j\neq i$ равна $(1 - e^{-\lambda_i \Delta t} ) \rho_{i,j}$. В связи с этим будем далее предполагать, что переходные вероятность однородной марковской цепи в непрерывном времени удовлетворяют следующим условиям:
\begin{equation}
    \label{condition 1 cont MC}
    \lim_{\Delta t\to 0}\frac{1-p_{i,i}(\Delta t)}{\Delta t} = \lambda_i
\end{equation}
\begin{equation}
    \label{condition 2 cont MC}
    \lim_{\Delta t\to 0}\frac{p_{i,j}(\Delta t)}{\Delta t} = \lambda_i \rho_{i,j},\quad i\neq j,\text{ при этом } \sum_{j\neq i} \rho_{i,j} = 1
\end{equation}
}

\elena{Отметим, что
вместо~\eqref{condition 1 cont MC},~\eqref{condition 2 cont MC} можно предполагать 
\begin{equation}
    \label{strong standart}
    \lim_{\Delta t\to 0} p_{i,j}(\Delta t) = \delta_{i,j} \quad \text{равномерно по $i,j$}.
\end{equation}
Тогда~\eqref{condition 1 cont MC},~\eqref{condition 2 cont MC} будут выполнены автоматически (см.~\cite{Afanas'evaBulinskaya}).
}

\elena{Другими словами, эволюция непрерывной марковской цепи  при сделанных предположениях задаётся локальными характеристиками 1) $\lambda_i$, которые суть \textit{интенсивности выхода} из состояния $i$, и 2) $\rho_{i,j}$. $i\neq j$, которые суть условные вероятности того, что если цепь меняет свое состояние $i$, то скачок переводит цепь в состояние $j\neq i$, либо 2') $\Lambda_{i,j}=\lambda_i \rho_{i,j}$, которые суть \textit{интенсивности перехода} из состояния $i$ в состояние $j\neq i$.
\begin{definition}
Матрица ${\Lambda=\left\| \Lambda_{i,j} \right\|_{i,j=1}^{|E|}}$ с компонентами  $\Lambda_{i,j}\ge 0$ для ${i\ne j}$, $\Lambda_{i,i}=-\lambda_i\le 0$ называется \textit{инфинитезимальной матрицей} непрерывной марковской цепи или \textit{генератором}.
\end{definition}
\begin{definition}
\label{def: jump chain}
Для марковской цепи в непрерывном времени определим \textit{вложенную цепь} или \textit{цепь скачков}, как ОДМЦ на том же множестве состояний $E$ с вероятностями перехода за один шаг $p_{i,j} = \Lambda_{i,j} / \lambda_i$, $p_{i,i} = 0$, если $\lambda_i > 0$ и $p_{i,j} = 0$, $p_{i,i}=1$,  если $\lambda_i =0$.
\end{definition} 
Вложенная цепь или цепь скачков попросту фиксирует последовательность состояний, через которые проходит $\xi(t)$ безотносительно к длительностям пребывания в различных состояниях.}

\elena{Таким образом, задание непрерывной марковской цепи сводится к заданию двух независимых объектов: 1) дискретной цепи скачков $X_k$, $k=0,1,\ldots$; 2)  последовательности независимых стандартных показательных случайных величин $\tau_0,\tau_1,\tau_2,\ldots$
Тогда формально
\begin{equation}
    \label{jump chain}
    \xi(t) = X_{\eta(t)}, \quad \eta(t) = \max\left \{ n: \, \sum_{k=0}^n \frac{\tau_k}{\lambda_{X_k}} \le t \right\}.
\end{equation}
}
 
\elena{\textbf{Ограничения на инфинитезимальную матрицу.} Отметим ещё раз, что в данном пособии рассматриваются \textit{консервативные} цепи, то есть все элементы инфинитезимальной матрицы конечны и сумма элементов каждой строки равна нулю: $\lambda_i = \sum_{j\neq i} \Lambda_{i,j} < \infty$ для любого $i\in E$. Теория марковских цепей, не удовлетворяющих условиям консервативности, гораздо сложнее. Например, если разрешить для некоторого состояния $i$ $\lambda_i = \infty$, то такое состояние называют \textit{мгновенным}, так как, попав в такое состояние, цепь мгновенно его покидает: вероятность того, что время нахождения в мгновенном состоянии положительно, равна нулю. Если разрешить инфенитезимальной матрице сумму по строке иметь не нулевую, а отрицательную: $\sum_{j\in E} \Lambda_{i,j} \le 0$, то соответствующая матрица переходных вероятностей $P(t)$ будет субстохастической: $\sum_{j\in E} P_{i,j}(t) \le 1$. В этом случае цепь с положительной вероятностью <<уходит>> на бесконечность. Также в данном пособии не будут рассматриваться так называемые \textit{взрывные} цепи: когда на конечном интервале времени цепь может совершить бесконечно много прыжков и <<уйти>> из любого конечного подмножества  $E$ (такая ситуация невозможна, если цепь конечная, т.е. $|E|<\infty$). Другими словами, марковская цепь в непрырывном времени будет невзрывной, если с вероятностью один выполнено
$$
\lim_{n\to\infty} T_n = \infty,
$$
где $T_n$ --  момент времени $n-$го прыжка $\xi(t)$.
Для невзрывных цепей  в формуле \eqref{jump chain} для любого конечного $t$ $\eta(t) < \infty $. Достаточным условием отсутствия взрыва является равномерная ограниченность параметров $\lambda_i$, то есть $\sup_{i\in E} \lambda_i \le c <\infty$ (см. далее теорему~\ref{th:gener}). Действительно, если $\lambda = \sup_{i\in E} \lambda_i \le c < \infty$, то 
$$
\mathbb{P}\left( T_{n+1} - T_n \ge t  \right) \ge e^{-\lambda t}.
$$
 Тогда по лемме Бореля--Кантелли с единичной вероятностью произойдет бесконечно много событий $ \left\{ T_{n+1} - T_n \ge t  \right\}$. Откуда будет следовать, что 
$$
\lim_{n\to\infty} T_n = \infty.
$$
Примером взрывной марковской цепи является процесс чистого размножения, где интенсивности перехода зависят от номера состояния так, что $\sum_{i\in E} \Lambda_{i,i+1}^{-1} < \infty$ (см.~\cite[теорема 2.5.3]{KelbertSukhov2010} или \cite[Т.~1, гл.~XVII, \S~4]{Feller}).}

\elena{Отметим также, что вместо требуемых ограничений на инфинитезимальную матрицу можно накладывать условия на переходные вероятности\eqref{strong standart}}.

\elena{Разобьём промежуток времени $[0, t + \Delta t)$ на два интервала $[0, \Delta t)$ и $[\Delta t, t + \Delta t)$. В зависимости от того, был или не был скачок в малом промежутке времени $[0, \Delta t)$ мы можем записать по формуле полной вероятности:
$$p_{i,j}(t+\Delta t) = \sum\limits_{k\neq i}p_{k,j}(t)\Lambda_{i,k} \Delta t + p_{i,j}(t) (1-\lambda_i\Delta t) + o(\Delta t)$$
или
$$\frac{p_{i,j}(t+\Delta t) -p_{i,j}(t) }{\Delta t } = \sum\limits_{k\neq i}p_{k,j}(t)\Lambda_{i,k} - p_{i,j}(t)\lambda_i  + o(1).$$
Переходя к пределу $\Delta t\to 0$, получаем матричное дифференциальное уравнение первого порядка
\begin{equation}
    \label{eq:forwardKolmogFeller}
    \dot P(t) =\Lambda  P(t),
    \end{equation}
(где точка означает дифференцирование по времени $t$) c начальным условием $P(0) = I$. 
Полученное уравнение называют \textit{обратными уравнением Колмогорова--Феллера}. Стоит отметить, что переход к пределу $\Delta t\to 0$ при выводе обратных уравнений законен в рамках принятых предположений~\eqref{condition 1 cont MC} и \eqref{condition 2 cont MC}.}

\elena{Разобьём теперь промежуток времени $[0, t + \Delta t)$ на два интервала $[0, t)$ и $[t, t + \Delta t)$. В зависимости от того, был или не был скачок в малом промежутке времени $[t, t + \Delta t)$ мы можем записать по формуле полной вероятности:
$$p_{i,j}(t+\Delta t) = \sum\limits_{k\neq j}p_{i,k}(t)\Lambda_{k,j} \Delta t + p_{i,j}(t) (1-\lambda_j\Delta t) + o(\Delta t)$$
или
$$\frac{p_{i,j}(t+\Delta t) -p_{i,j}(t) }{\Delta t } = \sum\limits_{k\neq j}p_{i,k}(t)\Lambda_{k,j} - p_{i,j}(t)\lambda_j  + o(1).$$
Стоит отметить, что теперь для для существования предела при $\Delta t\to0$ в правой части последнего выражения помимо принятых условий~\eqref{condition 1 cont MC} и \eqref{condition 2 cont MC} требуется дополнительное предположение о том, что при фиксированном $j$ переход к пределу в~\eqref{condition 2 cont MC} равномерен по $i$. В итоге получаем матричное дифференциальное уравнение первого порядка
\begin{equation}
    \label{eq:forwardKolmogFeller}
    \dot P(t) = P(t) \Lambda, \,P(0) = I.
    \end{equation}
Полученное уравнение называют \textit{прямым уравнением Колмогорова--Феллера}. }

\elena{
    Ясно, что если $|E|<\infty$, то единственным решением этих уравнений является матричная экспонента $${P(t)=\exp(t\Lambda)}.$$
    Этот факт будет верен и для случая бесконечных матриц с учётом принятых выше ограничений на инфинитезимальную матрицу, а именно, что $\sup_{i\in E} \lambda_i < \infty$ (см. теорему Колмогорова~\ref{th:Kolmogorov}). То есть на самом деле, в предположении, что цепь невзрывная (на конечном интервале времени происходит лишь конечное число скачков) обе системы прямых и обратных  уравнений эквивалентны. При этом наше дополнительное предположение о равномерности предела в условии~\eqref{condition 2 cont MC} при выводе прямых уравнений носит технический характер, его можно избежать, если уравнения записывать в терминах преобразования Лапласа (см.~\cite[Т.~2, гл.~XIV, \S~7]{Feller}). }
    
    \elena{
Под уравнениями Колмогорова--Феллера также понимают уравнения не на матрицу $P(t)$ , а на вектор распределения вероятностей $p(t)$:
$$
\dot{p}(t) = \Lambda^T p(t).
$$}

\begin{theorem}
\label{th:gener}
\textit{Пусть матрица $\Lambda$ такова, что} $$\Lambda_{i,i} \elena{\le} 0,\quad \Lambda_{i,j} \ge 0, \quad i\ne j,$$
$$\sum_{j}\Lambda_{i,j} = 0, \quad \elena{\sup_{i\in E}}\sum_{j}|\Lambda_{i,j}| \le C < \infty. $$
\textit{Тогда для любого }$t\ge 0$
\begin{enumerate}
    \item \textit{ряд} $\sum_{k=0}^\infty \frac{t^k \Lambda^k}{k!}$ \textit{сходится};
    \item \textit{матричная экспонента $P(t) = e^{t \Lambda} = \sum_{k=0}^\infty \frac{t^k \Lambda^k}{k!}$ является стохастической матрицей};
    \item $P(t)$ \textit{дифференцируемо по $t$, причем } $\frac{d P(t)}{d t} = \Lambda P(t) = P(t) \Lambda$, $P(0) = I$. 
\end{enumerate}
\end{theorem}
В пособии рассматриваются только инфинитезимальные мат\-ри\-цы~$\Lambda$, удовлетворяющие этой теореме.

\elena{\textbf{Доказательство.}
Пункт а) следует из верхней оценки для нормы (например, $l_\infty-$нормы)
$$
\| e^{t\Lambda} \|_\infty = \left\| \sum_{k=0}^\infty \frac{t^k \Lambda^k}{k!} \right\|_\infty\le \sum_{k=0}^\infty \frac{t^k \| \Lambda \|_\infty^k}{k!} = e^{t\|\Lambda\|_\infty} \le e^{tC}.
$$
Для доказательства пункта б) заметим, что неотрицательность элементов матричной экспоненты следует из
$$
P(t) = e^{t\Lambda} = e^{-tC} e^{t(\Lambda + C I)},
$$
где матрица $\Lambda + C I$ состоит из неотрицательных элементов. Сумма в каждой строке матрицы $P(t)$ равна единице, поскольку
$$
P(t) = I + \sum_{k\ge 1} \frac{t^k \Lambda^k}{k!},
$$
где сумма в каждой строке матрицы $\Lambda^k$, $k\ge 1$, равна нулю.
Пункт в) следует  из дифференцирования представления матричной экспоненты в виде сходящегося ряда (пункт а) и коммутативности степеней матрицы $\Lambda$ (в условиях теоремы). \EndProof
}

Сгруппируем все полученные результаты в рамках одной теоремы, \elena{доказательство которой можно посмотреть в \cite [гл.~21, \S~2, теорема~2.2]{Borovkov1999}, либо \cite[Т.~2, гл.~XIV, \S~7, следствие~1]{Feller}}

\begin{theorem}
\label{th:Kolmogorov}
\textit{В однородных непрерывных цепях Маркова в 
предположениях \elena{теоремы~\ref{th:gener}} справедливы следующие уравнения}:
$$\dot{P}(t)=P(t) \Lambda, \ P(0) = I,$$
$$\dot{P}(t)=\Lambda P(t), \ P(0) = I,$$
$$\dot{p}(t) = \Lambda^T p(t),$$
\textit{где $\Lambda$ -- инфинитезимальная матрица, $P(t)$ -- матрица перехода, $p(t)$ -- распределение вероятностей состояний, $I$ -- единичная матрица. \elena{Единственным р}ешением первых двух уравнений является матричная экспонента }$${P(t)=\exp(t\Lambda)}.$$
\end{theorem}

\elena{Отметим, что при отсутствии ограничений на инфинитезимальную матрицу, либо при отсутствии условия~\eqref{strong standart} ситуация сильно усложняется, см., например,~\cite{Afanas'evaBulinskaya}.}

Рассмотрим теперь несколько конкретных примеров. Непрерывные цепи Маркова удобно обозначать в виде ориентированных графов, в вершинах которых расположены состояния цепи, а веса ребер между состояниями равны интенсивностям перехода между состояниями (в дискретных цепях там стоит вероятность перехода за один шаг).


\begin{example}
Записать \elena{и решить} уравнения Колмогорова--Феллера для процесса Пуассона с интенсивностью $\lambda$.
\end{example}
\begin{solution}
\elena{Инфинитезимальная матрица $\Lambda = \left\|\Lambda_{i,j}\right\|_{i,j=1}^\infty$ имеет следующий вид:
$$
\Lambda = \begin{bmatrix}-\lambda & \lambda & 0 & 0 & 0 & \dots \\
0 & -\lambda & \lambda & 0 & 0 & \ddots \\
0 & 0 & -\lambda & \lambda & 0 & \ddots \\
\vdots & \ddots & \ddots & \ddots & \ddots & \ddots\end{bmatrix},
$$
т.е.\ для элементов матрицы $\Lambda$ справедлива следующая формула:
\begin{center}
$\Lambda_{i,j} = \begin{cases}-\lambda, \text{ если } j=i,\\
\lambda, \text{ если } j = i+1,\\ 
0 \text{ иначе. }\end{cases}$
\end{center}
Обратные уравнения Колмогорова--Феллера записываются в виде
$$
    \dot{p}_{i,j}(t) = -\lambda p_{i,j}(t) + \lambda p_{i+1,j}(t),
$$
с начальным условием $p_{i,j}(t) =\delta_{ij}$. Ясно, что для всех $j < i$ $p_{i,j}(t) = 0$, а при $i\le j$ $p_{i,j}(t) = p_{0,j-i}(t)$. Таким образом, достаточно найти $p_{0,i}(t)$. 
}

\elena{Сделаем замену, обозначив $f_i(t) = p_{0,i}(t) e^{\lambda t}$. Тогда 
$$
\left\{ 
\dot f_0(t) = 0,
f_i(0) = 1,
\right .
$$
а для $i\ge1 $
$$
\left\{ 
\dot f_i(t) = \lambda f_{i-1}(t),
f_i(0) = 0.
\right .
$$
Откуда 
\begin{equation*}
\begin{split}
        &f_0(t) = 1,\\
        &f_1(t) = \lambda t,\\
        &f_2(t) = \frac{(\lambda t)^2}{2!},\\
        &\ldots\\
        &f_i(t) = \frac{(\lambda t)^i}{i!}.
\end{split}
\end{equation*}
А значит
\begin{equation*}
    p_{i,j}(t) = 
    \begin{cases}
    \frac{(\lambda t)^{j-i}}{(j-i)!}, &\text{ если $j\ge i$,}\\
    0, &\text{иначе.}
    \end{cases}
\end{equation*}
}

\elena{Прямые уравнения имеют вид
$$
    \dot{p}_{i,j}(t) = -\lambda p_{i,j-1}(t) + \lambda p_{i+1,j}(t),
$$
с начальным условием $p_{i,j}(t) =\delta_{ij}$. Решение будет таким же. \EndEx
}

\end{solution}

\elena{В следующих примерах приведены  методы получения распределения в момент времени $t$ и переходной матрицы в момент времени $t$. Важно заметить, что при решении задачи о предельном поведении марковской цепи, не разумно честно вычислять $p(t)$ и $P(t)$. Правильно воспользоваться результатами ниже из раздела~\ref{sec: Erg Cont MC}.}

\begin{example}
Найти распределение вероятностей состояний $p(t)$ и переходную матрицу $P(t)$ для следующей непрерывной цепи Маркова ($\lambda>0$, $\mu>0$).
\end{example}

\textbf{Решение}. В первую очередь выпишем матрицы $\Lambda$ и $\Lambda^\top$ для данной цепи. По условию задачи, $\Lambda_{0,1}=\lambda$, $\Lambda_{1,0}=\mu$. Так как $\Lambda_{0,0}+\Lambda_{0,1}=0$, то~$\Lambda_{0,0}=-\lambda$. Аналогично, так как $\Lambda_{1,0}+\Lambda_{1,1}=0$, то~$\Lambda_{1,1}=-\mu$. Поэтому
$$\Lambda = \left[ {\begin{array}{*{20}{r}}
{-\lambda}&{\lambda} \\ 
{\mu}&{-\mu} 
\end{array}} \right], \ \ \Lambda^\top = \left[ {\begin{array}{*{20}{r}}
{-\lambda}&{\mu} \\ 
{\lambda}&{-\mu} 
\end{array}} \right].$$

\begin{figure}[!h]
	\centering
	\includegraphics[scale=0.6]{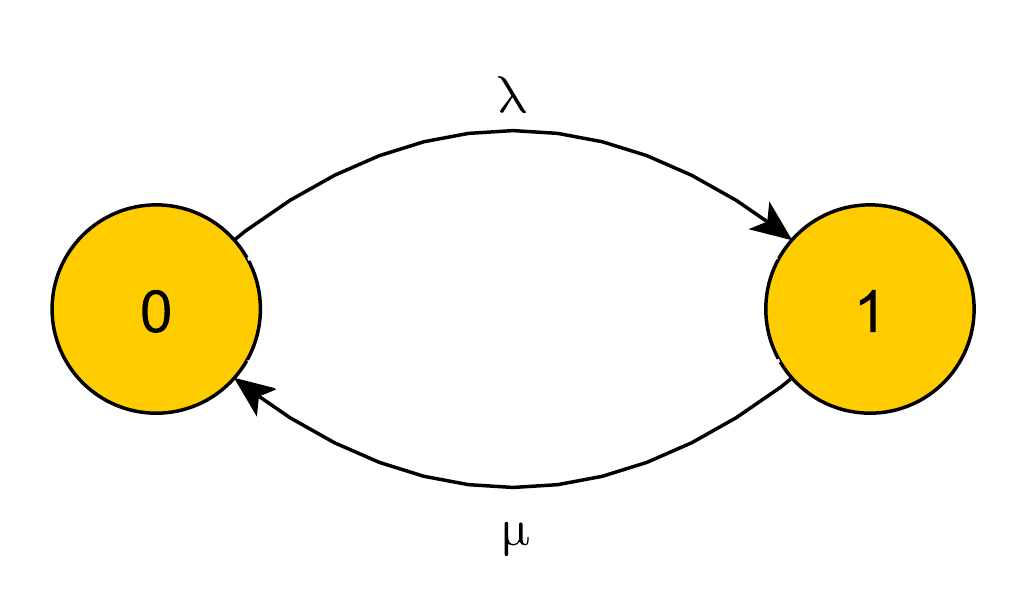}
	\label{fig:ContChainMarkov_1}
\end{figure}

Теперь запишем дифференциальное уравнение Колмогорова: $$\dot{p}(t)=\Lambda^\top p(t),$$ где $p(t)=[p_0(t),p_1(t)]^\top$. Это однородная линейная система дифференциальных уравнений, которую можно решить двумя способами.

\textbf{Способ 1 (общий)}. Как известно, решение линейных систем уравнений представляет собой линейную комбинацию линейно независимых решений. Линейно независимые решения можно искать через собственные (и, быть может, присоединенные) векторы и собственные числа матрицы системы. Найдем собственные числа матрицы $\Lambda^\top$: $$\left| {\begin{array}{*{20}{c}}
{-\lambda-r}&{\mu} \\ 
{\lambda}&{-\mu-r} 
\end{array}} \right|=r^2 + (\lambda+\mu)r = 0,$$ откуда ${r_1=0}$ и ${r_2=-(\lambda+\mu)<0}$. В действительности, для матрицы размера $2$ даже не нужно решать квадратное уравнение для поиска собственных значений, так как одно собственное значение $0$, а второе находится, зная след матрицы. Так как собственные числа различные, то собственных векторов будет достаточно, чтобы сформировать линейно не\-за\-ви\-си\-мые решения. При $r=0$ получаем $$\Lambda^\top-r I=\left[ {\begin{array}{*{20}{r}}
{-\lambda}&{\mu} \\
{\lambda}&{-\mu} \\ 
\end{array}} \right],$$
откуда собственный вектор $v_1=[\mu,\lambda]^\top$. При $r=-(\lambda+\mu)$ получаем $$\Lambda^\top- r I=\left[ {\begin{array}{*{20}{r}}
{\mu}&{\mu} \\
{\lambda}&{\lambda} \\ 
\end{array}} \right],$$
откуда собственный вектор $v_2=[1,-1]^\top$. Отсюда получаем общий вид решения: $$p(t)=C_1 v_1 + C_2 v_2 e^{-t(\lambda+\mu)}=\left[ {\begin{array}{*{20}{c}}
{C_1 \mu + C_2 e^{-t(\lambda+\mu)}} \\ 
{C_1 \lambda - C_2 e^{-t(\lambda+\mu)}} 
\end{array}} \right].$$ 

\noindent Пусть в начальный момент времени известно распределение $p(0)=$\linebreak $=[p_0(0),p_1(0)]^\top$. Получаем систему линейных уравнений относительно $C_1$ и $C_2$: $$\left\{ \begin{gathered}
C_1\mu+C_2=p_0(0)\hfill \\
C_1\lambda-C_2=p_1(0)\hfill \\ 
\end{gathered}  \right. \Rightarrow C_1=\frac{1}{\lambda+\mu}, \ C_2=\frac{\lambda p_0(0) - \mu p_1(0)}{\lambda+\mu}.$$

\noindentЗначит, искомые вероятности $$p_0(t)=\frac{\mu}{\lambda+\mu}+\frac{\lambda p_0(0) - \mu p_1(0)}{\lambda+\mu} e^{-t(\lambda+\mu)},$$  $$p_1(t)=\frac{\lambda}{\lambda+\mu}-\frac{\lambda p_0(0) - \mu p_1(0)}{\lambda+\mu} e^{-t(\lambda+\mu)}.$$

\noindent\textbf{Способ 2 (простой)}. Запишем сначала уравнения по отдельности: $$\dot p_0 = \mu p_1 -\lambda p_0,$$ $$\dot p_1 = \lambda p_0 - \mu p_1.$$ Затем воспользуемся равенством $p_0(t)+ p_1(t)=1$, которое выполнено в любой момент времени $t$, и сведем задачу к одному уравнению: $$\dot p_0=-(\lambda+\mu) p_0 + \mu,$$ решение которого имеет вид $$p_0(t)=C e^{-t(\lambda+\mu)}+\frac{\mu}{\lambda+\mu}.$$ 


\noindentКонстанта $C$ находится из начального условия $p_0(0)$. Далее определяется $$p_1(t)=1- p_0(t).$$
Теперь перейдем к вычислению переходной матрицы $P(t)$. Для этого возьмем уравнения $$\dot P(t)=P(t)Q$$
и заметим, что каждая строка $P_i(t)$ матрицы $P(t)$ удовлетворяет уравнениям $$\dot P_i^\top(t)=\Lambda^\top P_i^\top(t).$$ 

\noindentОбщий вид этих уравнений нам уже известен, и мы можем сразу написать:
$$P_{0,0}(t)=\frac{{\mu}}{\lambda+\mu}+\frac{{\lambda}P_{0,0}(0) - {\mu}P_{0,1}(0)}{\lambda+\mu} e^{-t(\lambda+\mu)},$$
$$P_{0,1}(t)=\frac{{\lambda}}{\lambda+\mu}-\frac{{\lambda}P_{0,0}(0)- {\mu}P_{0,1}(0)}{\lambda+\mu} e^{-t(\lambda+\mu)},$$
$$P_{1,0}(t)=\frac{{\mu}}{\lambda+\mu}+\frac{{\lambda}P_{1,0}(0) - {\mu}P_{1,1}(0)}{\lambda+\mu} e^{-t(\lambda+\mu)},$$
$$P_{1,1}(t)=\frac{{\lambda}}{\lambda+\mu}-\frac{{\lambda}P_{1,0}(0)- {\mu}P_{1,1}(0)}{\lambda+\mu} e^{-t(\lambda+\mu)}.$$

\noindentТеперь остается вспомнить, что $P_{0,0}(0)=P_{1,1}(0)=1$, $P_{0,1}(0)=P_{1,0}(0)=$\linebreak $=0$, тогда
$$P_{0,0}(t)=\frac{{\mu}}{\lambda+\mu}+\frac{{\lambda}}{\lambda+\mu} e^{-t(\lambda+\mu)},$$
$$P_{0,01}(t)=\frac{{\lambda}}{\lambda+\mu}-\frac{ {\lambda}}{\lambda+\mu} e^{-t(\lambda+\mu)},$$
$$P_{1,0}(t)=\frac{{\mu}}{\lambda+\mu}-\frac{  {\mu}}{\lambda+\mu} e^{-t(\lambda+\mu)},$$
$$P_{1,1}(t)=\frac{{\lambda}}{\lambda+\mu}+\frac{ {\mu}}{\lambda+\mu} e^{-t(\lambda+\mu)}. \text{ \EndEx}$$

\pagebreak

\begin{example}
Найти распределения вероятностей состояний $p(t)$ и~переходную матрицу $P(t)$ для следующей непрерывной цепи Маркова.
\end{example}
\begin{figure}[!h]
	\centering
	\includegraphics[scale=0.6]{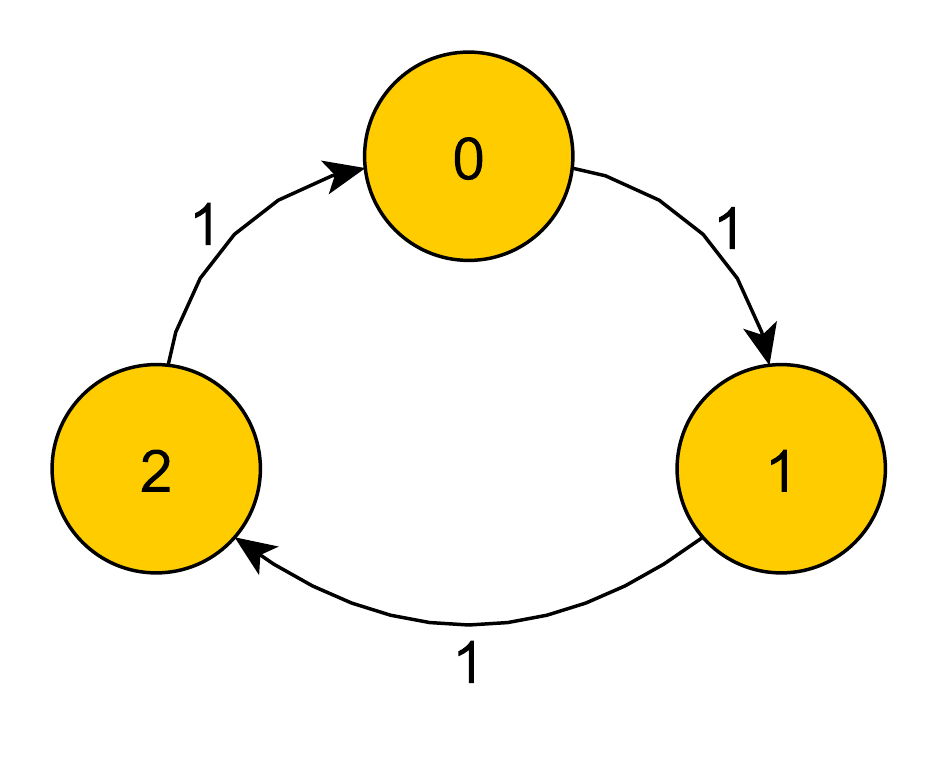}
	\label{fig:Graph2}
\end{figure}

\textbf{Решение}. Выпишем матрицы $\Lambda$ и $\Lambda^\top$ данной цепи:
$$\Lambda = \left[ {\begin{array}{*{20}{r}}
	{-1}&{1}&{0} \\ 
	{0} &{-1}&{1} \\
	{1}&{0}&{-1}
	\end{array}} \right], \ \ \Lambda^\top = \left[ {\begin{array}{*{20}{r}}
	{-1}&{0}&{1} \\ 
	{1} &{-1}&{0} \\
	{0}&{1}&{-1}
	\end{array}} \right],$$ составим характеристическое уравнение:
$$\left[ {\begin{array}{*{20}{r}}
	{-1-\lambda}&{0}&{1} \\ 
	{1} &{-1-\lambda}&{0} \\
	{0}&{1}&{-1-\lambda}
	\end{array}} \right]=-(1+\lambda)^3+1=0$$ и найдем собственные числа $$\lambda_1=0, \ \ \lambda_2 = -\frac{3}{2}+i\frac{\sqrt{3}}{2}, \ \ \lambda_3 = -\frac{3}{2}-i\frac{\sqrt{3}}{2}.$$
Так как все собственные числа различные, то общий вид уравнения $$\dot{p}(t)=\Lambda^\top p(t)$$ будет выглядеть следующим образом: $$p(t)=C_1 v_1 + C_2\mathrm{Re}\left(v_2e^{t\lambda_2}\right)+C_3\mathrm{Im}\left( v_2 e^{t\lambda_2} \right),$$ где $v_1$ и $v_2$ -- собственные векторы, отвечающие собственным значениям $\lambda_1$ и $\lambda_2$ соответственно. Явное выражение для $p(t)$ можно получить двумя способами.

\textbf{Способ 1 (общий)}. Найдем все собственные векторы. При $\lambda=0$ $$\Lambda^\top-\lambda I=\left[ {\begin{array}{*{20}{r}}
	{-1}&{0}&{1} \\ 
	{1} &{-1}&{0} \\
	{0}&{1}&{-1}
	\end{array}} \right],$$
откуда собственный вектор $v_1=[1,1,1]^\top$. Для $\lambda=-\frac{3}{2}+i\frac{\sqrt{3}}{2}$ введем обозначение $\rho=1+\lambda=e^{2\pi i/3}$, тогда
$$\Lambda^\top-\lambda I=\left[ {\begin{array}{*{20}{r}}
	{-\rho}&{0}&{1} \\ 
	{1} &{-\rho}&{0} \\
	{0}&{1}&{-\rho}
	\end{array}} \right] \sim \left[ {\begin{array}{*{20}{r}}
	{1}&{0}&{-\rho^{-1}} \\ 
	{1} &{-\rho}&{0} \\
	{0}&{1}&{-\rho}
	\end{array}} \right]  \sim \left[ {\begin{array}{*{20}{r}}
	{1}&{0}&{-\rho^{-1}} \\ 
	{0} &{-\rho}&{\rho^{-1}} \\
	{0}&{1}&{-\rho}
	\end{array}} \right]\sim $$ $$\sim \left[ {\begin{array}{*{20}{r}}
	{1}&{0}&{-\rho^{-1}} \\ 
	{0} &{1}&{-\rho^{-2}} \\
	{0}&{1}&{-\rho}
	\end{array}} \right] \sim [-\rho+\rho^{-2}=0] \sim \left[ {\begin{array}{*{20}{r}}
	{1}&{0}&{-\rho^{-1}} \\ 
	{0} &{1}&{-\rho^{-2}} \\
	{0}&{0}&{0}
	\end{array}} \right],$$ откуда $v_2=[\rho^{-1},\rho^{-2},1]^\top$, или можно взять $v_2=[\rho,1,\rho^2]^\top$. Далее, $$v_2e^{t\lambda_2}=\left[ {\begin{array}{*{20}{c}}
	{-\frac{1}{2}+i\frac{\sqrt{3}}{2}} \\ 
	{1} \\
	{-\frac{1}{2}-i\frac{\sqrt{3}}{2}}
	\end{array}} \right]e^{-\frac{3}{2}t}\left(\cos{\frac{\sqrt{3}}{2}t}+i\sin{\frac{\sqrt{3}}{2}t}\right),$$ откуда
$$\mathrm{Re}\left(v_2e^{t\lambda_2}\right)=e^{-\frac{3}{2}t} \left( \left[ {\begin{array}{*{20}{c}}
	{-1/2} \\ 
	{1} \\
	{-1/2}
	\end{array}} \right]\cos{\frac{\sqrt{3}}{2}t}+\left[ {\begin{array}{*{20}{c}}
	{-\sqrt{3}/2} \\ 
	{0} \\
	{\sqrt{3}/2}
	\end{array}} \right]\sin{\frac{\sqrt{3}}{2}t} \right), $$ 
$$\mathrm{Im}\left(v_2e^{t\lambda_2}\right)=e^{-\frac{3}{2}t} \left( \left[ {\begin{array}{*{20}{c}}
	{\sqrt{3}/2} \\ 
	{0} \\
	{-\sqrt{3}/2}
	\end{array}} \right]\cos{\frac{\sqrt{3}}{2}t}+\left[ {\begin{array}{*{20}{c}}
	{-1/2} \\ 
	{1} \\
	{-1/2}
	\end{array}} \right]\sin{\frac{\sqrt{3}}{2}t} \right). $$ Отсюда можно получить выражение для распределения вероятностей $p(t)$. Остается лишь воспользоваться начальным условием $p(0)$ для поиска $C_1$, $C_2$ и $C_3$.

\textbf{Способ 2 (простой)}. Найдем собственный вектор $v_1=[1,1,1]^\top$, это было сделано в предыдущем пункте. Далее, так как $\mathrm{Re}(\lambda_{2,3})<0$, то в пределе $$p(t)\to C_1 \cdot \left[ {\begin{array}{*{20}{c}}
	{1} \\ 
	{1} \\
	{1}
	\end{array}} \right].$$ Но так как сумма компонент вектора $p(t)$ равна 1 для любого $t$, то же самое справедливо и для предела. Значит, $C_1=1/3$, причем независимо от начального распределения $p(0)$. Далее заметим, что решение может быть представлено в виде $$p(t)=C_1 + C_2\cdot e^{-\frac{3}{2}t}\cos{\frac{\sqrt{3}}{2}t}+ C_3\cdot e^{-\frac{3}{2}t}\sin{\frac{\sqrt{3}}{2}t}$$ с некоторыми постоянными векторами $C_1$, $C_2$ и $C_3$, причем мы уже нашли $C_1=[1/3,1/3,1/3]^\top$. Но тогда в начальный момент времени $$p(0)=C_1+C_2 \Rightarrow C_2=\left[ {\begin{array}{*{20}{c}}
	{p_0(0)-1/3} \\ 
	{p_1(0)-1/3} \\
	{p_2(0)-1/3}
	\end{array}} \right].$$ Константу $C_3$ находим из соотношения $$\dot{p}(0)=\Lambda^\top p(0)=\left[ {\begin{array}{*{20}{c}}
	{p_2(0)-p_0(0)} \\ 
	{p_0(0)-p_1(0)} \\
	{p_1(0)-p_2(0)}
	\end{array}} \right]=-\frac{3}{2}C_2+\frac{\sqrt{3}}{2}C_3,$$ откуда получаем $$C_3=\frac{2}{\sqrt{3}}\dot{p}(0)+\sqrt{3}\cdot C_2.$$
Что касается компонент переходной матрицы, то, как мы уже поняли ранее, ее строки удовлетворяют тем же самым дифференциальным уравнениям, что и $p(t)$, а значит, вид решения для них тот же самый. Отличие будет лишь в том, что надо воспользоваться начальным условием $P(0)=I$ и соотношением $\dot P(0)= P(0)\Lambda=\Lambda$ для определения констант интегрирования. \EndEx

\subsection{Классификация состояний}
\label{sec:Classification ContMCh}

Перейдем теперь к классификации состояний непрерывной цепи Маркова. Сделать это можно аналогично тому, как это делалось для дискретных цепей Маркова~(см., например,~\cite{KaiLai1964}).

\begin{definition}
Состояния ${k,j\in E}$ называются \textit{сообщающимися}, если $$\left(\exists t>0: \ p_{k,j}(t)>0\right) \ \wedge \ \left(\exists \tau>0: \ p_{j,k}(\tau)>0\right).$$ Иначе состояния $k,j\in E$ называются \textit{несообщающимися}: $$\left(\forall t>0 \ p_{k,j}(t)=0\right) \ \vee \ \left(\forall \tau>0 \ p_{j,k}(\tau)=0\right).$$
\end{definition}

\begin{definition}
Состояние $k\in E$ называется \textit{несущественным}, если $$\exists j\in E: \ \left( \exists t>0 \ p_{k,j}(t)>0 \wedge \forall \tau\ge1 \ p_{j,k}(\tau)=0 \right).$$
Иначе состояние $k\in E$ называется \textit{существенным}: $$\forall j\in E \ \left( \forall t>0 \ p_{k,j}(t)=0 \right) \vee \left( \exists \tau>0 \ p_{j,k}(\tau)>0 \right).$$
\end{definition}

\begin{definition}
Состояние $k\in E$ называется \textit{возвратным}, если $$\mathbb{P}(\omega:\{t:\xi(\omega,t)=k\}\text{ не ограничено})=1.$$ Состояние $k\in E$ называется \textit{невозвратным}, если $$\mathbb{P}(\omega:\{t:\xi(\omega,t)=k\}\text{ не ограничено})=0.$$
\end{definition}

\begin{definition}
Состояние $k\in E$ называется \textit{нулевым}, если $$\lim\limits_{t\to\infty}p_{k,k}(t)=0.$$ Иначе состояние $k$ называется \textit{ненулевым}.
\end{definition}

Заметим, что для непрерывных цепей Маркова понятия периодического состояния нет.

Для непрерывных цепей Маркова известен критерий возвратности состояния, аналогичный критерию для дискретных цепей Маркова. 
\begin{theorem}[ (критерий возвратности)]
\textit{Состояние ${k\in E}$ является возвратным тогда и только тогда, когда} $$\int\limits_{0}^{+\infty}p_{k,k}(t)\,dt=+\infty.$$
\end{theorem}
С доказательством этой теоремы ознакомиться здесь~\cite[теорема 2.7.11, \S~2.7, глава~2]{KelbertSukhov2010} или в~\cite{KaiLai1964}.

\begin{definition}
Непрерывная цепь Маркова называется \textit{неразложимой}, если все ее состояния сообщаются.
\end{definition}

\begin{definition}
Распределение $\pi(t)$ называется \textit{ста\-ци\-о\-на\-р-\\ным}, если $$\forall t\ge 0 \quad \pi=P^\top(t)\pi,$$ или ${\dot{\pi}(t)=0 \Leftrightarrow \Lambda^\top \pi = 0}$. Уравнения ${\Lambda^\top \pi = 0}$ называются \textit{алгебраическими уравнениями Колмогорова}. 
\end{definition}
\eduard{Иными словами,} стационарное распределение -- собственный вектор матрицы $\Lambda^\top$ для собственного значения $0$, удовлетворяющий условиям нормировки $\sum_{k\ge 1} \pi_k = 1$ и неотрицательности компонент. Как и было уже сказано,  $0$ всегда является собственным значением для матриц $\Lambda$ и $\Lambda^\top$. Но так же, как и в случае дискретных марковских цепей, соответствующий собственный вектор может не удовлетворять условию нормировки (такая ситуация может возникнуть только для цепей со счетным числом состояний).

\elena{\textbf{Замечание 1.} Стоит отметить, что классификация состояний непрерывных цепей  эквивалентна классификации состояний соответствующей дискретной цепи скачков (или вложенной цепи) (см. определение~\ref{def: jump chain}).  Стоит отметить также, как связаны стационарные распределения непрерывной марковской цепи и дискретной цепи её скачков (в предположении, что стационарное распределение единственно). Если $\tilde \pi$ -- стационарное распределение цепи скачков, т.е. $\tilde P^T \tilde\pi = \tilde\pi$ или $\sum_{j\neq i} \tilde\pi_j \frac{\Lambda_{j,i}}{\Lambda_j} = \pi_i$ для всех $i\in E$, то стационарное распределение непрерывной цепи  $\pi_i = \tilde\pi_i \Lambda_i^{-1}/ \sum_{j\in E} \tilde\pi_j \Lambda_j^{-1}$. Например, если непрерывная цепь задается стохастическим графом, как на рис.~\ref{fig:BirthDeathProcess}, то соответствующая дискретная цепь -- простое блуждание на полуоси с отражающим \mbox{$0$-сос}\-тоя\-нием и вероятностями шага вправо $\lambda / (\lambda + \mu)$ и влево $\mu / (\lambda + \mu)$. При $\lambda / (\lambda + \mu) > 1/2$, т.е. $\lambda > \mu$, цепь невозвратна, а значит, все состояния нулевые, при  $\lambda / (\lambda + \mu) = 1/2$, т.е. $\lambda = \mu$, цепь возвратна, но все состояния нулевые (нуль-возвратность, то есть время возвращение имеет бесконечное математическое ожидание), и при  $\lambda / (\lambda + \mu) < 1/2$, т.е. $\lambda < \mu$, существует стационарное распределение $\tilde\pi_k = \frac{\mu^2 - \lambda^2}{2\lambda\mu} \left( \frac{\lambda}{\mu} \right)^k$, $k\ge 1$, $\tilde\pi_0 = \frac{\mu - \lambda}{2\mu} $ и цепь возвратна с ненулевыми состояниями (положительная возвратность) (см. пример \ref{birth-death discr}). В последнем случае стационарное распределение непрерывной цепи $\pi_k = \frac{\mu^2 - \lambda^2}{\mu} \left( \frac{\lambda}{\mu} \right)^k $, $k\ge 1$ и $\pi_0 = \frac{\mu-\lambda}{\mu}$.
\begin{figure}[!h]
	\centering
	\includegraphics[scale=0.50]{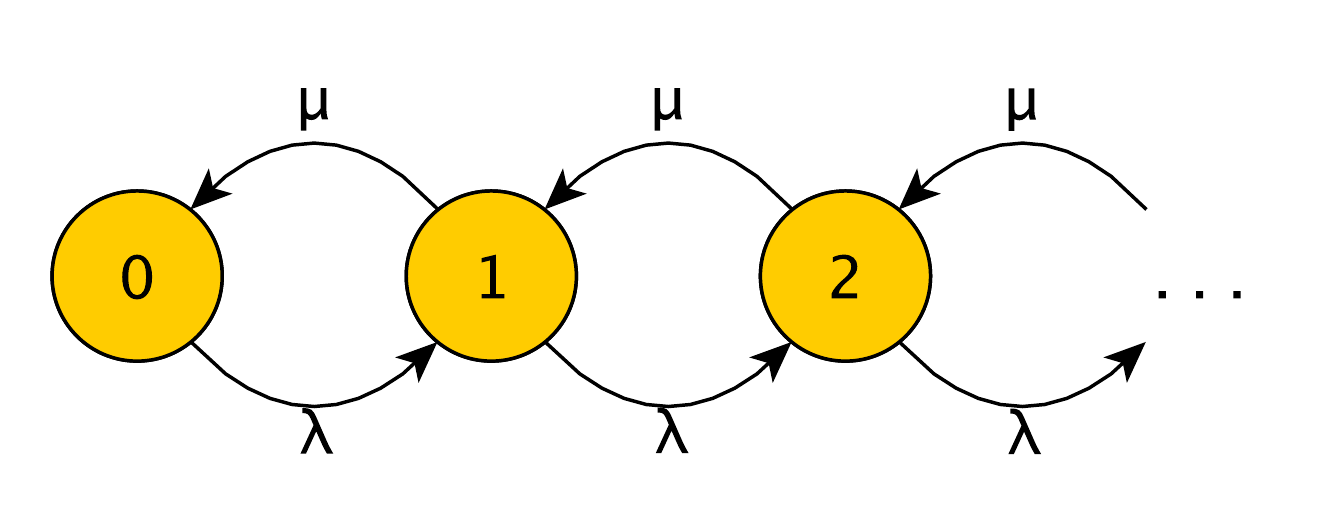}
	\caption{Процесс рождения/гибели с параметрами $\lambda/\mu$}
	\label{fig:BirthDeathProcess}
\end{figure}
\newpage}

\subsection{Эргодические непрерывные цепи Маркова}
\label{sec: Erg Cont MC}

\begin{theorem}\textbf{(см.~\cite[теорема 2.8.1]{KelbertSukhov2010})} 
\textit{В непрерывных цепях Маркова всегда существуют пределы} $$\lim\limits_{t\to\infty}p_{i,j}(t)=p_{i, j}^*.$$ 
\end{theorem}

\begin{definition}
\label{strong ergod continious}
Непрерывная однородная цепь Маркова  называется  \textit{сильно эргодической}, если для любого состояния $j\in E$ существует не зависящий от состояния $i\in E$ положительный предел: $$\lim\limits_{t\to\infty}p_{i,j}(t)=p_j^*>0.$$
\end{definition}
Несложно заметить, что если непрерывная однородная цепь Марко\-ва сильно эргодическая, то стационарное распределение $\pi$ единственное и $p^* = \pi$.  

\elena{Поскольку асимптотическое поведение непрерывных марковских цепей в случае отсутствия взрыва (то есть, удовлетворяющих условиям теоремы~\ref{th:gener}) равносильно асимптотическому поведению дискретной цепи скачков, то с}ледствием эргодических теорем из предыдущего раздела является

\begin{theorem}[ (эргодическая теорема)]
\label{ErgodicTh}
\textit{Для непрерывной однородной цепи Маркова с конечным \elena{или счётным} числом состояний следующие утверждения эквивалентны}:
\begin{enumerate}
    \item \textit{цепь сильно эргодична, т.е. для любых состояний} $i,j$ $\lim\limits_{t\to\infty}p_{i,j}(t)=\pi_j>0$;
    \item \textit{существует такое $t_0$, что все элементы матрицы $P(t_0)$ положительны};
    \item \textit{существует единственное стационарное распределение, все компоненты которого положительны;}
    \elena{\item \textit{существует единственное стационарное распределение у цепи скачков, все компоненты которого положительны.}}
\end{enumerate}
\textit{Кроме того, для  сильно эргодической цепи выполнено следующее}:
\begin{enumerate}
	\item \textit{для любого начального распределения и для любого $j\in E$ имеет место сходимость $p_j(t)\to\pi_j$ при} $t\to\infty$;
	\item \textit{для любого начального распределения и для любого} $j\in E$ $$\frac{1}{T}\int\limits_{0}^T \mathrm{I}(\xi(s)=j)\,ds \overset{\mathrm{\text{п.н.}}}{\longrightarrow} \pi_j, \ T\to\infty.$$
\end{enumerate}
\end{theorem}

\textbf{Замечание}.  В случае конечных непрерывных цепей для сильной эргодичности необходимо и достаточно неразложимости, т.е. не нужна апериодичность \elena{(как было в дискретном случае)}, которую  к тому же и непонятно как в непрерывном времени и определять.  \elena{По этим же причинам наличие периода у цепи скачков не влияет на эргодичность непрерывной цепи.}
Подробности можно найти в конце раздела~\ref{sec:DiscreteChains}. В частности, отметим, что теорема~\ref{th:technical_ergodic} верна и в случае марковских цепей с непрерывным временем. Более того, все результаты раздела~\ref{sec:DiscreteChains} переносятся на непрерывные однородные марковские цепи без изменения, за исключением отсутствия необходимости требовать аппериодичность. Сходимость в непрерывном времени всегда имеет место в~обычном смысле (не только по Чезаро). \elena{Таким образом, более слабое условие эргодичности, не требующее положительности предела из определения~\ref{strong ergod continious}, равносильно эргодичности (не в сильном смысле) дискретной цепи скачков, то есть что при любом начальном распределении цепь с единичной вероятностью попадает в единственный положительно-возвратный класс эквивалентности, так что предельное распределение совпадает со стационарным, которое равно нулю вне замкнутого положительно-возвратного класса эквивалентности. }

В качестве примера использования марковских цепей в прикладных областях рассмотрим задачи из теории систем массового обслуживания (СМО). 
\begin{example}\label{ex:queuing system}
Пусть имеется $N$ серверов и входной поток заявок, который предполагается пуассоновским с параметром $\lambda$. Каждая заявка вне зависимости от остальных, поступив в систему, обслуживается любым из свободных серверов в течение случайного времени, распределенного по показательному закону с параметром $\mu$, или безвозвратно покидает систему, если свободных серверов нет. Пусть $X(t)$ -- марковская цепь на множестве состояний $S$, равных числу занятых серверов, или, иначе говоря, числу заявок в системе, т.е. $S = \{0,1,\ldots,N\}$. Требуется найти предельное распределение числа заявок в системе. Стохастический граф для этой задачи представлен на рис.~\ref{fig:ServerProblem}.

\begin{figure}[!h]
	\centering
	\includegraphics[scale=0.46]{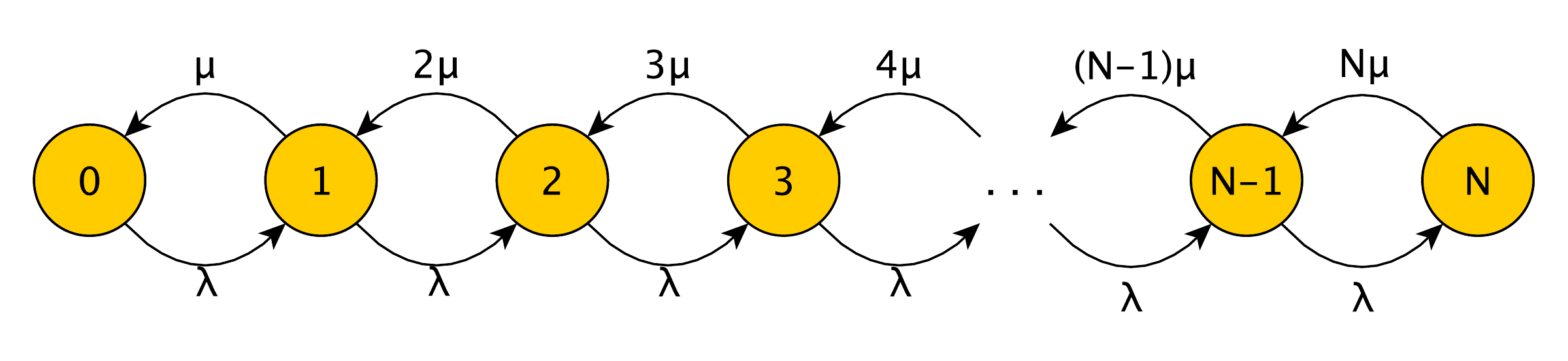}
	\caption{К примеру~\ref{ex:queuing system}}
	\label{fig:ServerProblem}
\end{figure}

Действительно, если в текущий момент времени в системе $k\le N$ заявок, каждая из них имеет остаточное время обслуживания -- показательное с тем же параметром $\mu$ (в силу свойства отсутствия памяти у показательного закона). Минимум из $k$ независимых показательных случайных величин имеет также показательный закон с суммарной интенсивностью, то есть в нашем случае $k\mu$.

Так как цепь конечна и неразложима, то согласно теореме~\ref{ErgodicTh} и~замечания к ней существует предельное распределение, которое совпадает с единственным стационарным распределением, которое находится из решения
$\Lambda^\top \pi = 0$, $\sum_{j=0}^N \pi_j = 1$, $\pi_j \ge 0$. Легко проверить, что решением является неполное пуассоновское распределение: 
$$
\pi_j = \left ( \sum_{k=0}^N \frac{(\lambda/ \mu )^k}{k!}  \right)^{-1} \frac{(\lambda/ \mu )^j}{j!}.
$$

Можно этой задаче дать немного другую интерпретацию. Пусть серверы -- это служащие некоторой компании, заявки -- клиенты. Каждый обслуженный клиент (есть еще клиенты, которые если не могут сразу переговорить с представителем компании, то больше не приходят) приносит доход компании в $A$ \gav{рублей}. Заработная плата служащего составляет $B$ \gav{рублей} в час. Считается, что параметры $\lambda^{-1}$ и $\mu^{-1}$ имеют единицы измерения часы. Требуется найти оптимальное число служащих.

Будем считать, что директор компании рассчитывает доход своей фирмы в предположении длительной ее работы. Тогда согласно эргодической теоремы предельное распределение занятых служащих совпадает со стационарным, и число клиентов, обслуженных в течение часа, можно оценить, как $\sum_{j=1}^N j \pi_j$. Доход компании в течение часа составляет $J(N) = A \sum_{j=1}^N j \pi_j - B N$. Остается решить задачу оптимизации $J(N) \to \max$\gav{, учитывая, что $\pi_j := \pi_j(N)$.} \EndEx
\end{example}

\begin{example}\label{ex:buses}
Пусть на остановку приходят автобусы согласно пуас\-сновскому процессу с параметром $\mu$. Поток пассажиров, приходящих на остановку, тоже считается пуассоновским с параметром $\lambda$ и независящим от прихода автобусов. Предполагается, что автобус, подошедший к остановке, забирает всех пассажиров, ожидающих транспорт. Требуется найти предельное распределение числа пассажиров, ожидающих  на остановке автобус. Для решения задачи рассмотрим марковскую цепь $X(t)$, где $X(t)\in \{0,1,2,\ldots\}$ -- число людей на остановке. Стохастический граф, описывающий динамику цепи, представлен на рис.~\ref{fig:BusProblem}. 

\begin{figure}[!h]
	\centering
	\includegraphics[scale=0.58]{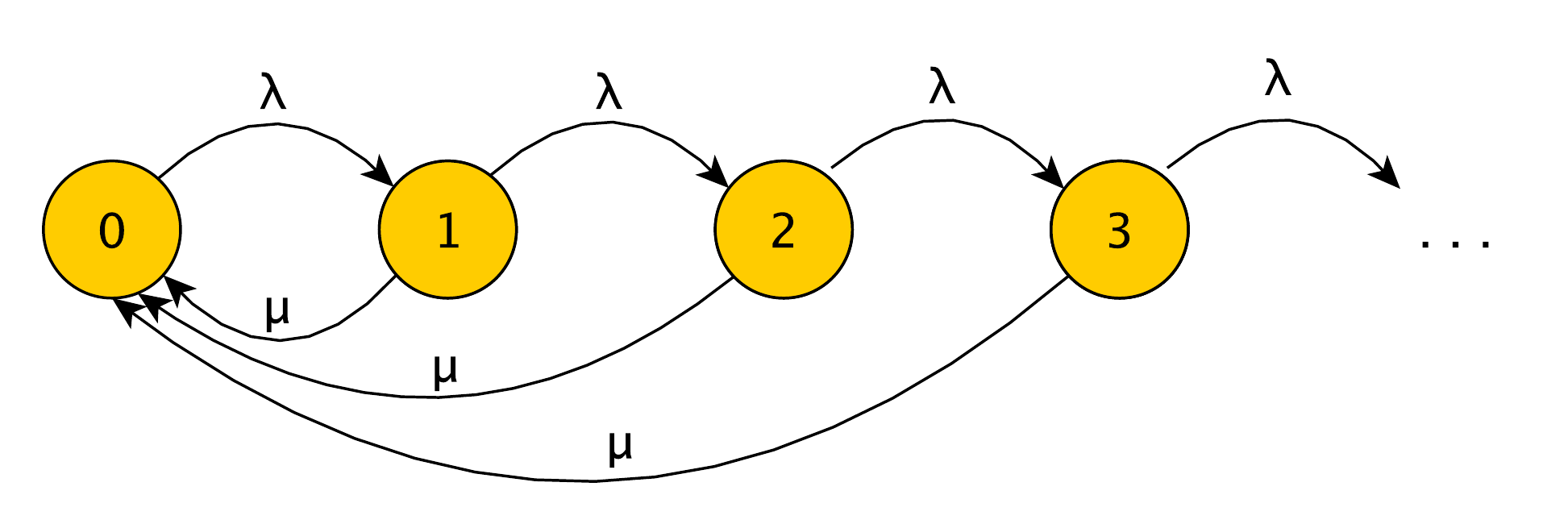}
	\caption{К примеру \ref{ex:buses}}
	\label{fig:BusProblem}
\end{figure}

Это пример счетной цепи. Легко проверить, что геометрическое распределение $\pi_j = \frac{\lambda + \mu }{ \mu } \left (  \frac{\lambda}{ \lambda + \mu }\right)^j$ является стационарным распределением. С учетом неразложимости оно будет единственным, и согласно теореме~\ref{ErgodicTh} это и будет предельным распределением. \EndEx
\end{example}


Очень важный класс цепей образуют \textit{процессы гибели и рождения}. По определению, это цепи со стохастическим графом, как на рис.~\ref{fig:ContChainMarkov_3}.

\begin{figure}[!h]
	\centering
	\includegraphics[scale=0.6]{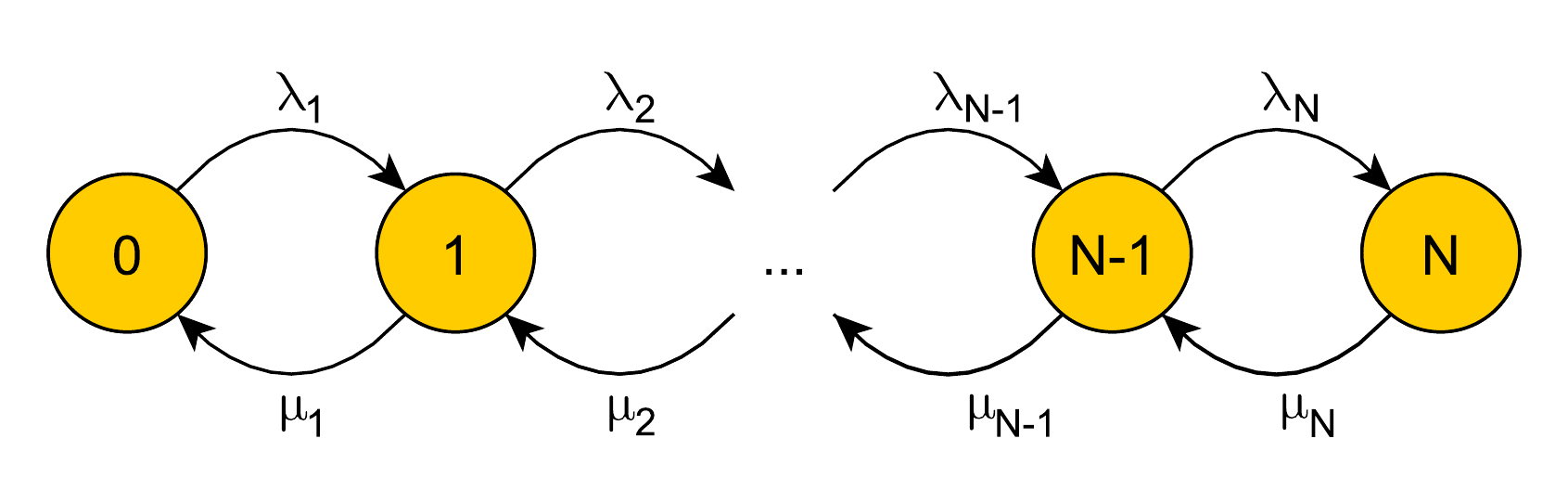}
	\caption{Стохастический граф процесса гибели и рождения}
	\label{fig:ContChainMarkov_3}
\end{figure}

Известны формулы для вычисления стационарного распределения состояний таких процессов: 
$$\pi_0 = \left(1 + \frac{\lambda_1}{\mu_1} + \frac{\lambda_1 \lambda_2}{\mu_1 \mu_2} + \dots \frac{\lambda_1\dots\lambda_N}{\mu_1\dots\mu_N} \right)^{-1},$$
\begin{center}
$\pi_k=\frac{\lambda_k}{\mu_k}\pi_{k-1}, \ k=1,\dots,N.$ \gav{\EndEx}
\end{center}

\subsection{Поведение цепи и время возвращения}

В заключение данного раздела приведем несколько более формальную сторону факта о том, как ведет себя непрерывная цепь Маркова при базовых предположениях.

Введем для состояния ${i\in E}$ на множестве ${\Omega_i=\{\omega:\xi(\omega,0) = i\}}$ случайную величину: $$\tau_i(\omega) = \inf\{t:t>0, \ \xi(\omega,t)\ne i\}.$$ Она называется \textit{моментом первого выхода} из состояния $i$. Тогда если ${\lambda_i=-\Lambda_{i,i}<\infty}$ (в конечных цепях в базовых предположениях это выполнено), то $$\mathbb{P}(\tau_i \ge t \,|\, \xi(0)=i)=e^{-\lambda_i t}, \ t \ge 0,$$ т.е. время, проведенное в каждом состоянии цепи, имеет показательное распределение с параметром $\lambda_i$. Поэтому среднее время пребывания в состоянии $i$ есть $\lambda_i^{-1}$.

Введем на $\Omega_i=\{\omega:\xi(\omega,0) = i\}$ случайную величину $$\eta_i(\omega)=\xi(\omega,\tau_i(\omega)).$$ Это состояние, в котором окажется цепь после прыжка из состояния $i$. Тогда в наших обозначениях $$\mathbb{P}(\eta_i = j\,|\,\xi(0)=i) = \frac{(1-\delta_{ij})\Lambda_{i,j}}{\Lambda_{i}} = p_{i,j},$$ где $\delta_{ij}$ -- символ Кронекера. Это  вероятность того, что после состояния $i$ цепь примет состояние $j$.

Введем теперь на множестве ${\Omega_i=\{\omega:\xi(\omega,0) = i\}}$ случайные величины: $$\alpha_{ij}(\omega)=\inf\{t: t>\tau_i(\omega), \ \xi(\omega,t)=j\}, \ j\in E.$$ Оказывается, что если состояние $i$ возвратное и для него ${0<\lambda_i<\infty}$ (все это выполнено для произвольного состояния в неразложимой конечной цепи), то $$\lim\limits_{t\to\infty}p_{i,i}(t)=\frac{1}{\Lambda_i \mathbb{E}(\alpha_{ii}\,|\,\Omega_i)}.$$ Этой формулой можно пользоваться для того, чтобы в \gav{сильно} эргодической цепи вычислить среднее время возвращения в состояние $i$, т.е. $$\mathbb{E}(\alpha_{ii}\,|\,\Omega_i)=\frac{1}{\Lambda_i \pi_i}.$$

\section{Непрерывные процессы Маркова}\label{sec:ContProcMarkov}

\subsection{Уравнения Колмогорова}

\begin{definition}
    Случайный процесс $X(t)$, ${t\in \R},$ называется \textit{марковским процессом}, если $\forall n\in \NN\; \forall t_1 < t_2 < \ldots < t_n < t_{n+1}\; \forall x_n\in\R$ и для любых борелевских множеств $B_1, B_2,\ldots, B_{n-1}, B_{n+1}$ выполняется
    $$
        \PP\left(X(t_{n+1})\in B_{n+1}\mid X(t_n)=x_n, X(t_{n-1})\in B_{n-1},\ldots, X(t_1)\in B_1\right)=$$$$
        = \PP\left(X(t_{n+1})\in B_{n+1}\mid X(t_n) = x_n\right).
    $$
\end{definition}

Говоря неформальным языком, данное определение можно интерпретировать следующим образом: марковский процесс~--- это случайный процесс, у которого <<будущее>> определяется только <<настоящим>> и не зависит от <<прошлого>>.

\begin{example}
    Любой процесс Леви является марковским процессом, что гарантирует свойство независимости приращений. В частности, винеровский и пуассоновский процессы являются марковскими. 
\end{example}

Оставшуюся часть данного раздела мы посвятим изучению мар\-ковских процессов с непрерывным временем и \textit{континуальным} числом состояний, а именно, мы приведем так называемые \shmaxg{\textit{первое} и \textit{второе уравнения Колмогорова}}, на которые можно смо\-треть, как на обобщения и следствия уравнений Колмогорова--Ф\egor{е}\-ллера и Колмогорова--Ч\gav{э}пмена.

Рассмотрим условную функцию распределения
\[
    F_X(x,t \,|\, x_0,t_0) = \PP\left(X(t) < x \,|\, X(t_0) = x_0\right)
\]
и условную плотность распределения (в предположении, что она существует)
\[
    p_X(x,t\,|\,x_0,t_0) = \frac{d}{dx}F_X(x,t\,|\,x_0,t_0).
\]
При выводе уравнения Колмогорова--Ф\egor{е}ллера мы пользовались формулой полной вероятности
\begin{equation*}
\begin{split}
    &\PP\left(X(t) = j\,|\, X(0) = i\right) \underset{t_1 <t}{=} \\
    &=\sum\limits_{k}\PP\left(X(t_1)=k\,|\, X(0) = i\right)\PP\left(X(t) = j\,|\, X(t_1) = k\right),
\end{split}
\end{equation*}
которая в случае однородных марковских цепей эквивалентна очевидному равенству
\[
    P^t = P^{t_1}P^{t-t_1},
\]
где $P$~--- матрица переходных вероятностей. Отметим, что момент времени $t_1$ можно выбрать произвольным из интервала $(0,t)$, а суммирование ведётся по всем возможным состояниям. Данный подход легко переносится и на случай континуального числа состояний: обуславливая функцию распределения по всем возможным состояниям $\shmaxg{v}$ в промежуточный момент времени $t_1$, получаем уравнение
\begin{eqnarray*}
    F_X(x,t\,|\,x_0,t_0) &=& \int\limits_\R F_X(x,t\,|\,\shmaxg{v},t_1)d_{\shmaxg{v}}F_X(\shmaxg{v},t_1\,|\,x_0,t_0)=\\
    &=& \int_\R F_X(x,t\,|\,\shmaxg{v},t_1)p_X(\shmaxg{v},t_1\,|\,x_0,t_0)\,d\shmaxg{v}.
\end{eqnarray*}
Дифференцируя левую и правую части предыдущего равенства по $x$, получаем непрерывный аналог уравнений Колмогорова--Ф\egor{е}ллера и Колмогорова--Ч\gav{э}пмена
\begin{equation}\label{eq:cont_kolmogorov}
    p_X(x,t\,|\,x_0,t_0) = \int\limits_\R p_X(x,t\,|\,\shmaxg{v},t_1)p_X(\shmaxg{v},t_1\,|\,x_0,t_0)\,d\shmaxg{v}.
\end{equation}
Эти уравнения на функции $F_X$ и $f_X$ называются в литературе \textit{обобщенными уравнениями Маркова}.

\begin{theorem}[ (первое уравнение Колмогорова)] \shmaxg{\textit{Пусть выполнены следующие условия:}}

\shmaxg{\textit{1) марковский случайный процесс $X(t)$ непрерывен в том смысле, что каково бы ни было постоянное $\delta>0$, имеет место соотношение}
\begin{equation}\label{eq13:Cond1}
    \lim\limits_{\Delta t \to 0}\frac{1}{\Delta t}\int_{|x-x_0|\ge\delta}d_xF_X(x,t\,|\,x_0,t_0-\Delta t)=0.
\end{equation}}

\shmaxg{\textit{2) частные производные}
$$\frac{\partial F_X(x,t\,|\,x_0,t_0)}{\partial x_0}, \ \frac{\partial^2 F_X(x,t\,|\,x_0,t_0)}{\partial x_0^2}$$ \textit{существуют и непрерывны при любых значениях $t_0$, $x_0$ и $t>t_0$.}}

\shmaxg{\textit{3) каково бы ни было $\delta>0$, существуют пределы}
\begin{equation}\label{eq13:Cond2}
    \lim\limits_{\Delta t \to 0}\frac{1}{\Delta t}\int_{|x-x_0|<\delta}(x-x_0)\,d_xF_X(x,t\,|\,x_0,t_0-\Delta t)=a(t,x),
\end{equation}
\begin{equation}\label{eq13:Cond3}
    \lim\limits_{\Delta t \to 0}\frac{1}{\Delta t}\int_{|x-x_0|<\delta}(x-x_0)^2\,d_xF_X(x,t\,|\,x_0,t_0-\Delta t)=b(t,x),
\end{equation}
\textit{и эта сходимость равномерна относительно $x_0$.}}

\shmaxg{\textit{Тогда функция $F_X(x,t\,|\,x_0,t_0)$ удовлетворяет уравнению}
\begin{equation}\label{eq13:FirstEqKolmF}
    \frac{\partial F_X(x,t\,|\,x_0,t_0)}{\partial t_0}=-a(t_0,x_0)\frac{\partial F_X(x,t\,|\,x_0,t_0)}{\partial x_0}-\frac{b(t_0,x_0)}{2}\frac{\partial^2 F_X(x,t\,|\,x_0,t_0)}{\partial x_0^2}.
\end{equation}}
\end{theorem}
\begin{proof}
\shmaxg{Согласно обобщенному уравнению Маркова
$$F_X(x,t\,|\,x_0,t_0-\Delta t)=\int F_X(x,t\,|\,v,t_0)\,d_vF_X(v,t_0\,|\,x_0,t_0-\Delta t).$$ С другой стороны, если ввести <<умную единицу>>
$$F_X(x,t\,|\,x_0,t_0)=\int F_X(x,t\,|\,x_0,t_0)\,d_vF_X(v,t_0\,|\,x_0,t_0-\Delta t).$$ Из этих равенств заключаем, что
$$\frac{F_X(x,t\,|\,x_0,t_0-\Delta t)-F_X(x,t\,|\,x_0,t_0)}{\Delta t}=$$
$$=\frac{1}{\Delta t}\int (F_X(x,t\,|\,v,t_0)-F_X(x,t\,|\,x_0,t_0))\,d_vF_X(v,t_0\,|\,x_0,t_0-\Delta t).$$
Воспользуемся формулой Тейлора:
$$F_X(x,t\,|\,v,t_0)=F_X(x,t\,|\,x_0,t_0) + (v-x_0)\frac{\partial F_X(x,t\,|\,x_0,t_0)}{\partial x_0} + $$
$$+\frac{1}{2}(v-x_0)^2\frac{\partial^2 F_X(x,t\,|\,x_0,t_0)}{\partial x_0^2} + o((v-x_0)^2).$$ Далее все просто:
$$\frac{F_X(x,t\,|\,x_0,t_0-\Delta t)-F_X(x,t\,|\,x_0,t_0)}{\Delta t}=$$
$$=\frac{1}{\Delta t}\int_{|v-x_0|\ge\delta} (F_X(x,t\,|\,v,t_0)-F_X(x,t\,|\,x_0,t_0))\,d_vF_X(v,t_0\,|\,x_0,t_0-\Delta t)+$$
$$+\frac{1}{\Delta t}\int_{|v-x_0|<\delta} (F_X(x,t\,|\,v,t_0)-F_X(x,t\,|\,x_0,t_0))\,d_vF_X(v,t_0\,|\,x_0,t_0-\Delta t)=$$
$$=\frac{1}{\Delta t}\int_{|v-x_0|\ge\delta} (F_X(x,t\,|\,v,t_0)-F_X(x,t\,|\,x_0,t_0))\,d_vF_X(v,t_0\,|\,x_0,t_0-\Delta t)+$$
$$+\frac{\partial F_X(x,t\,|\,x_0,t_0)}{\partial x_0}\frac{1}{\Delta t}\int_{|v-x_0|<\delta}(v-x_0)\,d_vF_X(v,t_0\,|\,x_0,t_0-\Delta t)+$$
$$+\frac{1}{2}\frac{\partial^2 F_X(x,t\,|\,x_0,t_0)}{\partial x_0^2}\frac{1}{\Delta t}\int_{|v-x_0|<\delta}[(v-x_0)^2+o((v-x_0)^2)]\,d_vF_X(v,t_0\,|\,x_0,t_0-\Delta t).$$
Перейдем теперь к пределу при ${\Delta t\to0}$. Первое слагаемое правой части в силу~\eqref{eq13:Cond1} имеет своим пределом 0. Второе слагаемое, согласно~\eqref{eq13:Cond2}, в пределе равно $a(t_0,x_0)\partial F_X/\partial x_0$. Наконец, третье слагаемое может отличаться от $1/2 b(t_0,x_0)\partial^2 F_X/\partial x_0^2$ только на слагаемое, стремящееся к нулю при $\delta\to0$. Но так как левая часть последнего равенства от $\delta$ не зависит и указанные предельные значения от $\delta$ не зависят, то предел правой части существует и равен
$$a(t_0,x_0)\frac{\partial F_X(x,t\,|\,x_0,t_0)}{\partial x_0}+\frac{1}{2}b(t_0,x_0)\frac{\partial^2 F_X(x,t\,|\,x_0,t_0)}{\partial x_0^2}.$$ Отсюда мы заключаем о существовании предела:
$$\lim\limits_{\Delta t\to0}\frac{F_X(x,t\,|\,x_0,t_0-\Delta t)-F_X(x,t\,|\,x_0,t_0)}{\Delta t}=-\frac{\partial F_X(x,t\,|\,x_0,t_0)}{\partial t_0}.$$ Отсюда следует утверждение теоремы. \EndProof}
\end{proof}

\shmaxg{Если существует плотность распределения $p_X(x,t\,|\,x_0,t_0)$, то дифференцирование уравнения~\eqref{eq13:FirstEqKolmF} показывает, что она удовлетворяет уравнению
\begin{equation}\label{eq13:FirstEqKolmf}
    \frac{\partial p_X(x,t\,|\,x_0,t_0)}{\partial t_0}=-a(t_0,x_0)\frac{\partial p_X(x,t\,|\,x_0,t_0)}{\partial x_0}-\frac{b(t_0,x_0)}{2}\frac{\partial^2 p_X(x,t\,|\,x_0,t_0)}{\partial x_0^2}.
\end{equation}
Уравнение \eqref{eq13:FirstEqKolmF} или \eqref{eq13:FirstEqKolmf} и называется в литературе \textit{первым уравнением Колмогорова} или \textit{обратным уравнением Колмогорова}. Это дифференциальные уравнения в частных производных, независимыми переменными здесь являются переменные $x_0$ и $t_0$, а переменные $x$, $t$ играют роль фиксированных параметров. Эти уравнения были впервые строго доказаны А.Н. Колмогоровым.}

\begin{theorem}[ (второе уравнение Колмогорова)]
\shmaxg{\textit{Добавим к условиям 1)--3) предыдущей теоремы еще два условия:}}

\shmaxg{\textit{4) существует плотность распределения вероятностей}
$$p_X(x,t\,|\,x_0,t_0)=\frac{\partial F_X(x,t\,|\,x_0,t_0)}{\partial x}.$$}

\shmaxg{\textit{5) существуют непрерывные производные}
$$\frac{\partial p_X(x,t\,|\,x_0,t_0)}{\partial t}, \ \frac{\partial}{\partial x}[a(t,x)p_X(x,t\,|\,x_0,t_0)], \ \frac{\partial^2}{\partial x^2}[b(t,x)p_X(x,t\,|\,x_0,t_0)].$$}

\shmaxg{\textit{Тогда плотность $p_X(x,t\,|\,x_0,t_0)$ удовлетворяет уравнению}
\begin{equation}\label{eq13:SecondEqKolm}
   \frac{\partial p_X(x,t\,|\,x_0,t_0)}{\partial t}=-\frac{\partial}{\partial x}[a(t,x)p_X(x,t\,|\,x_0,t_0)] + \frac{1}{2}\frac{\partial^2}{\partial x^2}[b(t,x)p_X(x,t\,|\,x_0,t_0)].
\end{equation}}
\end{theorem}
\shmaxg{Доказательство этого утверждения можно найти в книге \cite[С. 301]{Gnedenko}. Уравнение~\eqref{eq13:SecondEqKolm} называется \textit{вторым уравнением Колмогорова} или \textit{прямым уравнением Колмогорова}. Это уравнение было также независимо получено физиками Фоккером и Планком, поэтому часто их называют также уравнениями Фоккера--Планка или уравнениями Колмо\-горова--Фоккера--Планка. В этих уравнениях независимыми переменными являются $x$ и $t$, а переменные $x_0$ и $t_0$ играют роль параметров.}

\shmaxg{Выясним теперь физический смысл коэффициентов $a(t,x)$ и $b(t,x)$. Для этого мы предположим вместо~\eqref{eq13:Cond1}, что при любом $\delta > 0$ имеет место соотношение
$$\lim\limits_{\Delta t \to 0}\frac{1}{\Delta t}\int_{|x-x_0|\ge \delta}(x-x_0)^2\,d_x F_X(x,t\,|\,x_0,t_0-\Delta t)=0.$$ Это более сильное предположение, из него следует~\eqref{eq13:Cond1}. Ограничения~\eqref{eq13:Cond2} и~\eqref{eq13:Cond3} в силу этого требования упрощаются до
$$\lim\limits_{\Delta t \to 0}\frac{1}{\Delta t}\int(x-x_0)\,d_xF_X(x,t\,|\,x_0,t_0-\Delta t)=a(t,x),$$
$$\lim\limits_{\Delta t \to 0}\frac{1}{\Delta t}\int(x-x_0)^2\,d_xF_X(x,t\,|\,x_0,t_0-\Delta t)=b(t,x),$$ т.е. теперь интегрирование осуществляется по всей числовой оси. Остальные ограничения не изменятся. Теперь видно, что
$$\int(x-x_0)\,d_xF_X(x,t\,|\,x_0,t_0-\Delta t)=$$$$=\mathbb{E}(X(t)-X(t-\Delta t)\,|\,X(t_0-\Delta t)=x_0)=\mathbb{E}(X(t)-X(t-\Delta t))$$ является математическим ожиданием изменения процесса за время $\Delta t$, а
$$\lim\limits_{\Delta t \to 0}\frac{1}{\Delta t}\int(x-x_0)^2\,d_xF_X(x,t\,|\,x_0,t_0-\Delta t)=\mathbb{E}(X(t)-X(t-\Delta t))$$ есть математическое ожидание квадрата изменения $X(t)$. Теперь предположим, что $X(t)$ -- это координата точки, движущейся под влиянием случайных воздействий. Тогда $a(t,x)$ есть средняя скорость изменения этой координаты, а $b(t,x)$ пропорциональна средней кинетической энергии.}

\shmaxg{\begin{example}
Например, у винеровского процесса
$$a(t,x)=\lim\limits_{\Delta t\to0}\frac{\mathbb{E}(W(t)-W(t-\Delta t))}{\Delta t}=0,$$
$$b(t,x)=\lim\limits_{\Delta t\to0}\frac{\mathbb{E}(W(t)-W(t-\Delta t))^2}{\Delta t}=\lim\limits_{\Delta t\to0}\frac{\Delta t}{\Delta t}=1.$$ Кроме того, у винеровского процесса существует условная плотность распределения $p_W(x,t\,|\,x_0,t_0)$ как плотность распределения невырожденного подвектора (из 1 компоненты) нормального вектора (из 2 компонент) и для этой функции, как можно проверить, все условия регулярности выполнены. Поэтому первым уравнением Колмогорова для винеровского процесса будет
$$\frac{\partial p_W(x,t\,|\,x_0,t_0)}{\partial t_0}=-\frac{1}{2}\frac{\partial^2 p_W(x,t\,|\,x_0,t_0)}{\partial x_0^2},$$ а вторым уравнением Колмогорова будет
$$\frac{\partial p_W(x,t\,|\,x_0,t_0)}{\partial t}=\frac{1}{2}\frac{\partial^2 p_W(x,t\,|\,x_0,t_0)}{\partial x^2}.$$ Легко убедиться, что функция
$$p_W(x,t\,|\,x_0,t_0) = \frac{1}{\sqrt{2\pi(t-t_0)}}\exp{\left(-\frac{(x-x_0)^2}{2(t-t_0)}\right)}$$ удовлетворяет обоим уравнениям. Эту функцию условной плотности распределения можно было вывести и независимым образом как функцию плотности подвектора нормального вектора. \EndEx
\end{example}}

Также заметим, что для выписанных уравнений в частных производных для возможности их решения необходимо еще задать начальные / граничные условия. Для прямого  уравнения задают следующее начальное условие (Коши):
$$p_X(x,t_0\,|\,x_0,t_0) = \delta(x - x_0),$$
где $\delta(x)$ -- дельта функция Дирака. Такое начальное условие вызывает сразу много вопросов, главный из которых: как понимать решение, если начальное условие -- обобщенная функция? Ключевую роль в ответе на этот вопрос играет линейность рассматриваемых уравнений и возможность их переписать в терминах $F_X(x,t_0\,|\,x_0,t_0)$~\cite{Gnedenko}. Другой важный вопрос: можно ли задавать, какие-то другие начальные условия, задаваемые произвольными функциями распределения? Ответ на этот вопрос для прямого уравнения положительный. Только тогда требуется изменить обозначения и вместо $p_X(x,t_0\,|\,x_0,t_0)$ или (в общем случае) $F_X(x,t_0\,|\,x_0,t_0)$ писать просто $p_X(x)$ или $F_X(x)$, не подчеркивая в виде функции начальное условие. 

\subsection[Процесс Ито и формула Дынкина]{Процесс Ито и формула Дынкина*}

В конце раздела~\ref{correlation} был введен стохастический диффузионный процесс Ито, который является важным представителем рассматриваемого в этой главе класса марковских процессов в непрерывном времени с несчетным (континуальным) множеством состояний. Далее конспективно (без доказательств) будет продемонстрирована связь этих процессов (Ито) с изложенными выше результатами.

    \begin{definition}
        \textit{Однородный по времени диффузионный процесс Ито}~--- это случайный процесс $X_t(\omega) = X(\omega, t):\Omega \times [0,+\infty) \to \R^n$, удовлетворяющий стохастическому дифференциальному уравнению вида
        \begin{equation}\label{eq:diff_ito_def2}
        dX_t = a(X_t)dt + \sigma(X_t)dW_t,\quad t\ge 0,\quad X_0 = x,
    \end{equation}
    где $W_t$~--- $n$-мерный винеровский процесс (в каждой компоненте которого независимая реализация винеровского процесса), а векторная и матричная функции $a:\R^n \to \R^n$ и $\sigma:\R^n \to \R^{n\times n}$ -- равномерно липшицевы.
    \end{definition}
     \begin{definition}\label{infinitesimal}
        \textit{Производящий} (\textit{инфинитезимальный}) оператор $A$ однородного по времени диффузионного процесса Ито $X_t$ определяется, как
        \begin{equation*}
            Af(x) = \lim\limits_{t\to 0+}\frac{\EE^{x}f(X_t) - f(x)}{t}, \ \EE^{x}f(X_t) = \EE\left(f(X_t)\,|\,X_0 = x\right), x\in \R^n,
        \end{equation*}
        если этот предел существует.
    \end{definition}
    
    \begin{definition}
    Случайная величина $\tau$ называется \textit{моментом остановки} (относительно случайного процесса ${X_t = X(t)}$), если для любого  $t \ge 0$ событие $\left\{\tau\le
    t \right\}$ измеримо относительно сигма-алгебры, порожденной
    $\left\{X_s\right\}_{0\le s \le t}$.
    \end{definition}
    
    В простейшем случае $\tau = t$ -- не случайная величина. Нетривиальным примером является первый момент выхода случайного процесса из множества $D$:
    \begin{equation}
    \label{tau_D}
         \tau_D = \inf\left\{t\ge 0: X_t \notin D \right\}.
    \end{equation}

    Следующая теорема является аналогом основной теоремы из математического анализа о связи интеграла и производной.
    \begin{theorem}[ (формула Дынкина)]
        \textit{Пусть $X_t$~-- однородный по времени диффузионный процесс Ито, $f$~-- дважды непрерывно дифференцируемая функция с компактным носителем, $\tau$~-- момент остановки, причём $\EE^{x}\tau < \infty$. Тогда справедлива формула}
        \begin{equation}\label{eq:dynkin_formula}
            \EE^{x}\left(f(X_\tau)\right) = f(x) + \EE^x\left(\int\limits_{0}^\tau Af(X_s) ds\right).
        \end{equation}
    \end{theorem}

Пусть $X_t$ удовлетворяет стохастическому дифференциальному ура\-внению Ито \eqref{eq:diff_ito_def2}.
Тогда по теореме 7.3.3~\cite{Oksendal2003}
$$Af = \sum_{i}a_i\frac{\partial f}{\partial x_i} + \frac{1}{2}\sum_{i,j}\left(\sigma\sigma^T\right)_{ij}\frac{\partial^2 f}{\partial x_i \partial x_j}.$$

Если $f$ -- дважды гладкая функция с компактным носителем (отлична от нуля на ограниченном множестве), то согласно формуле Дынкина можно ввести функцию
\begin{equation}
\label{u_def}
    u(t,x) = \EE^x\left(f(X_t)\right),
\end{equation}
$$u(0,x) = f(x),$$
которая будет удовлетворять  следующему уравнению:
$$\frac{\partial u}{\partial t} = \EE^x\left(Af(X_t)\right),$$
что в свою очередь можно переписать следующим образом (аналог обратного уравнения Колмогорова(--Фоккера--Планка)):
$$\frac{\partial u}{\partial t} = \EE^x\left(Af(X_t)\right) = A\EE^x\left(f(X_t)\right) = Au,$$
см. теорему 8.1.1~\cite{Oksendal2003}.

Заметим также, что в введенных ранее обозначениях $$p_X(x,t\,|\,x_0,t_0=0)=p_X(x,t\,|\,x_0)$$ имеем
$$\EE^{x_0}\left(f(X_t)\right) = \int\limits_{\mathbb{R}^n} f(x)p_X(x,t\,|\,x_0)dx.$$
Отсюда по формуле Дынкина получим
$$\int\limits_{\mathbb{R}^n}f(x)p_X(x,t\,|\,x_0)dx = f(x_0) + \int\limits_{0}^t\int\limits_{\mathbb{R}^n}Af(x)p_X(x,s\,|\,x_0)dxds.$$
Если продифференцировать это тождество по $t$ и воспользоваться равенством (в $L_2$) 
$$\langle A\phi, \psi\rangle = \langle \phi, A^*\psi\rangle,$$
где функции $\phi,\psi$ дважды гладкие и имеют компактный носитель, а 
$$A^*\phi = -\sum_{i}\frac{\partial (a_i \phi)}{\partial x_i} + \frac{1}{2}\sum_{i,j}\frac{\partial^2 (\sigma_{ij} \phi)}{\partial x_i \partial x_j},$$
то получим прямое уравнение Колмогорова(--Фоккера--Планка)
$$\frac{\partial p_X(x,t\,|\,x_0)}{\partial t} = A^{*}p_X(x,t\,|\,x_0).$$

В заключении этого раздела, следуя~\cite[глава 9]{Oksendal2003}, покажем, как с помощью диффузионных процессов Ито решать краевые задачи для эллиптических уравнений: 
\begin{center}
$Au = 0$ в $D$,  \\
$u = g$ на $\partial D$
\end{center}
с оператором $A$, определенным ранее. Рассмотрим соответствующий $ A$, диффузионный процесс $X_t$ \eqref{eq:diff_ito_def2}. Положим подобно \eqref{u_def} 
\begin{equation}\label{ux}
    u(x) = \EE^x\left(g(X_{\tau_D})\right),
\end{equation}
где $\tau_D$ определяется согласно \eqref{tau_D}. Тогда по формуле Дынкина с $f(x) =$\linebreak $= u(x)$
при $x\in D$ и $0\le t\le\tau_D$: 
$$\EE^{x}\left(u(X_t)\right) = u(x) + \EE^x\left(\int\limits_{0}^t Au(X_s) ds\right).$$

Из определения~\eqref{ux} следует, что 
$$\EE^{x}\left(u(X_t)\right) = u(x).$$
Таким образом, для $0\le t\le\tau_D$
$$\EE^x\left( Au(X_t)\right) = 0.$$
Отсюда при некоторых дополнительных предположениях (матрица $\sigma$ -- положительно определенная) уже можно получить, что
$$Au=0.$$ Осталось только заметить, что $u(x) = g(x)$ при $x \in \partial D$.

Следовательно, на формулу \eqref{ux} можно смотреть, как на способ представления решения рассматриваемой краевой задачи для эллиптического уравнения.  Представление \eqref{ux} можно использовать для численного вычисления решения. Для этого траектории диффузионного процесса заменяются соответствующими случайными блужданиям (с соотношением размера скачков с шагом по времени соответствующим диффузионному скейлингу из раздела \ref{Winer}), выпущенными из точки, в которой хочется оценить решение. Траектории блужданий отслеживаются до первого момента  попадания на границу $\partial D$. Если взять достаточно много таких траекторий и считать, что масштаб скачков в блужданиях достаточно мал, то среднее арифметическое полученных значений функции $g$ в точках $\partial D$, в которые пришли траектории, дает хорошую оценку оцениваемого значения (метод Монте-Карло).

В заключение  заметим, что здесь мы в основном исследовали марковские процессы в непрерывном времени с континуальным числом состояний на всем пространстве. Последний пример демонстрирует важность рассмотрения марковских процессов, <<живущих>> внутри множеств с поглощающей границей. В целом следует отметить, что в приложениях большой интерес вызывают марковские процессы (в~непрерывном и дискретном времени), которые живут на компактных множествах. В этом случае существование условной  плотности вероятностей $p_X(x,t\,|\,x_0,t_0)$ обеспечивает эргодичность процесса. Это утверждение является аналогом эргодических теорем из предыдущих разделов -- наличие плотности обеспечивает единственность стационарной меры, а компактность носителя -- условие положительной возвратности, что достаточно для эргодичности в общем случае~\cite{Malyshev}. Приведем пару примеров.

В приложениях методов Монте-Карло важную роль играет возможность генерировать точки, равномерно распределенные в заданном конечномерном \ag{выпуклом (это условие можно ослабить)} связном компакте в предположении, что задан граничный оракул для рассматриваемого компакта. Последнее означает, что для любой прямой такой оракул выдает отрезок из этой прямой, по которому прямая пересекается с компактом. С помощью такого оракула строится марковский процесс в дискретном времени с непрерывным множеством состояний (точки компакта): в текущем положении процесса равновероятно и независимо выбирается направление, это направление и текущее положение задают уравнение прямой; граничный оракул по этому уравнению выдает отрезок; на отрезке равновероятно и независимо выбирается новая точка, которая принимается за новое состояние процесса. Несложно понять, что инвариантной (стационарной) мерой здесь будет равномерная мера на компакте. Таким образом, после достаточно большого числа шагов (полиномиально зависящего от размерности пространства) можно гарантировать, что положение описываемого процесса будет с хорошей точностью равновероятно распределено на компакте, независимо от точки старта. В~действительности, по ходу итераций гененируется много точек, с похожими свойствами. Проблема только в том, что они зависимы между собой. Однако корреляционная функция этого процесса (и любого другого марковского процесса в условиях эргодичности) экспоненциально убывает с увеличением числа шагов между сечениями\footnote{Показатель убывания пропорционален spectral gap -- расстоянию между максимальным собственным значением оператора, отвечающего за переходные вероятности (это число всегда равно 1), и следующим (по величине модуля) собственным значением.}, поэтому при должном прореживании точек, генерируемых описанным процессом, можно получить не одну (конечную), а целый набор почти независимых точек, распределенных почти равномерно в заданном компакте. Описанный здесь алгоритм генерирования точек называется \textit{Hit and Run}~\cite{Lovasz}.

Второй пример связан с популярной в последние десятилетия теорией \textit{случайных матриц}, см., например,~\cite{Vershynin,Vershynin2}. Пусть матрицы $$G_1,\dotsc,G_n,\dotsc$$ независимы и одинаково распределены (с распределением, имеющим плотность $\pi_G$ \gav{с финитным носителем}). Требуется доказать, что существует (максимальный показатель Ляпунова): \begin{equation*}\label{lmax}
\lambda_{\max} = \lim_{n\to\infty} \frac{1}{n} \ln \| G_n\cdot\dotsc\cdot G_1\|_2.
\end{equation*}
\gav{Введем  $Y_n = G_n\cdot\dotsc\cdot G_1 X_0$. Положим  $X_n = Y_n/\|Y_n\|_2$. Заметим, что $\{X_k\}$ -- дискретный марковский процесс, множества состояний которого -- единичная сфера (компактное множество).} Далее следует заметить (см. п. 3.2~\cite{Malyshev}), что $$\lambda_{\max} = \lim_{n\to\infty} \frac{1}{n} \sum\limits_{k=1}^n \ln \| G_k X_{k-1}\|_2.$$
\gav{При сделанных предположениях} $\{X_k\}$ -- эргодический марковский процесс, \gav{поэтому}
$$\lambda_{\max} = \EE_{\pi_G,\pi}\left(\ln\|GX\|_2\right),$$
где случайная матрица $G$ распределена согласно плотности $\pi_G$, а независимый от нее случайный вектор $X$ имеет плотность распределения $\pi$\gav{, где $\pi$ -- стационарное распределение для $\{X_k\}$.}





\begin{appendix}
    \renewcommand\thesection{\Asbuk{section}}
    
    \newpage

\section*{\Large\centering ПРИЛОЖЕНИЯ}\addcontentsline{toc}{section}{ПРИЛОЖЕНИЯ\vspace{-1.5mm}} 	
\section[Модель Эренфестов\vspace{-1mm}]{Модель Эренфестов}
\label{EhrenfestModel}

\gav{В этом разделе на простейшем примере (модель Эренфестов) мы постараемся продемонстрировать основы теории макросистем. Более подробно о теории макросистем будет написано в следующем разделе.}

Различают две модели Эренфестов: дискретную~\gav{\cite{KelbertSukhov2010}} и непрерывную \elena{\cite[раздел~6,~задача~1]{StochAn2016}}. 

\begin{definition}
\textit{Непрерывной цепью Эренфестов} будем называть непрерывную цепь Маркова со следующим стохастическим графом\elena{, задающим инфинитезимальные вероятности переходов}:
\end{definition}

\begin{figure}[!h]
	\centering
	\includegraphics[scale=0.6]{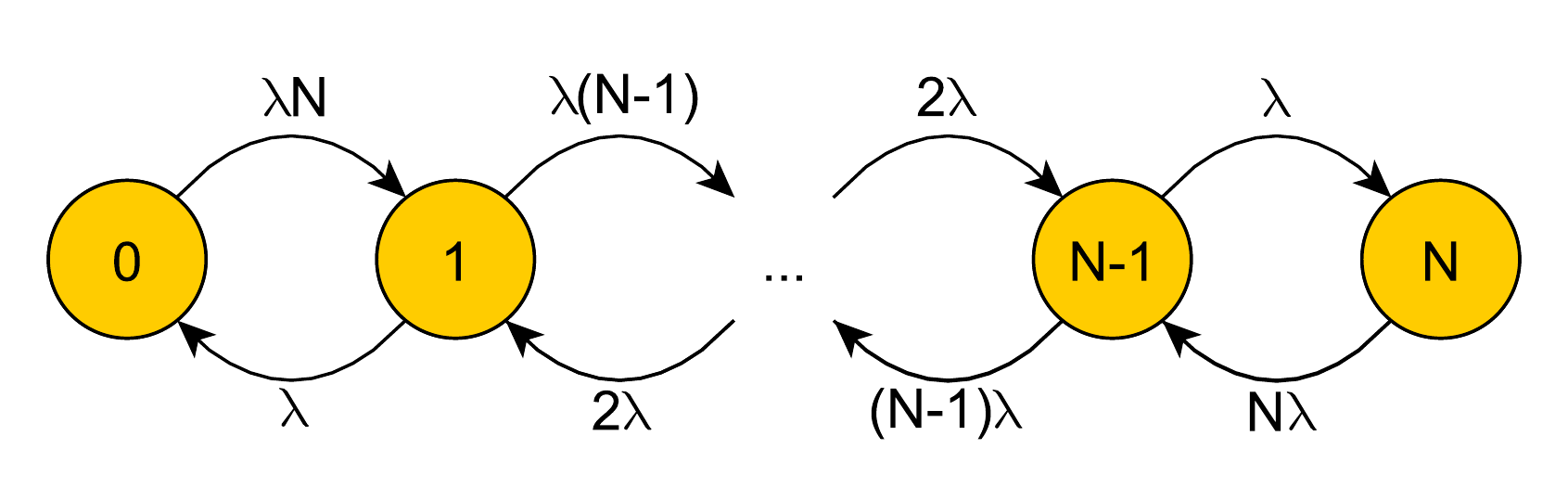}
\end{figure}

\elena{
\begin{definition}
\textit{Дискретной цепью Эренфестов} будем называть дискретную цепь Маркова со следующим стохастическим графом, задающим вероятности переходов:
\end{definition}
}

\begin{figure}[!h]
    \centering
    \includegraphics[scale=0.6]{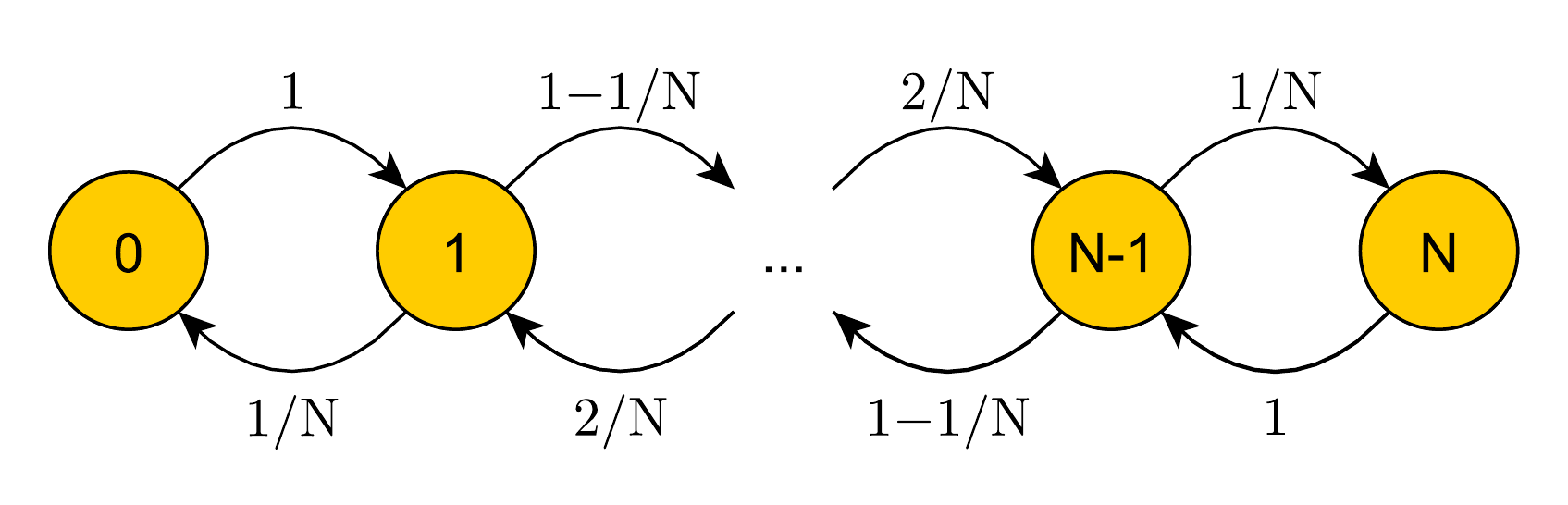}
\end{figure}


\elena{В этом разделе нашей целью будет  анализ непрерывной модели Эренфестов, но поскольку (как несложно убедиться) дискретная цепь Эренфестов является цепью скачков (см. определение~\ref{jump chain}) для непрерывной цепи Эренфестов, то нам потребуется и дискретная модель.}

Непрерывная цепь Эренфестов $\ag{\xi(t)}$\elena{, $t\ge 0$,} имеет конечное число состояний $N+1$, причем интенсивности перехода между состояниями определяются формулами:
$$\lambda_{i,i+1}=\lambda (N-i), \ 0 \le i \le N-1,$$
$$\lambda_{i,i-1} = \lambda i, \ 1 \le i \le N,$$
где $\lambda > 0$, и $\lambda_{i,j}=0$ во всех остальных случаях.

\elena{Дискретная цепь Эренфестов $X_n$, $n=0,1,2,\ldots$, также имеет конечное число состояний $N+1$, причем вероятности перехода между состояниями определяются формулами:
$$p_{i,i+1}=\frac{N-i}{N}=1-\frac{i}{N}, \ 0 \le i \le N-1,$$
$$p_{i,i-1} =\frac{i}{N}, \ 1 \le i \le N,$$
и $p_{i,j}=0$ во всех остальных случаях.}

\elena{П. и Т. Эренфесты в 1907 году предложили данную модель для описания диффузии через мембранную перегородку в сосуде или теплообмена между двумя изолированными телами. Мы будем придерживаться немного другой (шутливой) интерпретации  данной цепи  (см. \cite[раздел~6,~задача~1]{StochAn2016}). Две собаки сидят рядом друг с другом и страдают от $N$ блох. Каждая блоха независимо от остальных в течение малого промежутка времени длины $h$ перескакивает на соседнюю собаку с вероятностью $\lambda h + o(h)$. Тогда $\xi(t)$ -- число блох на первой собаке в момент времени $t$, $N-\xi(t)$ -- число блох на второй собаке в момент времени $t$.    }

\begin{example}\label{eq:EhrenfestProb1}
Найти для непрерывной цепи Эренфест\elena{ов} стационарное распределение, среднее время возвращения в состояние 0. Оценить предельную (по времени и количеству состояний $N$) вероятность отклонения состояния цепи от $N/2$. \elena{Оцените время выхода на стационарный режим.}
\end{example}
\textbf{Решение}. Данная цепь является процессом гибели и рождения, в которой интенсивности ${\lambda_k = \lambda (N-k+1)}$ и $\mu_{k} = \lambda k$, $k=1,\dots,N$. Для вычисления стационарного распределения сначала заметим, что $$\frac{\lambda_1\dots\lambda_k}{\mu_1\dots\mu_k}=\frac{N(N-1)\dots(N-k+1)}{k!}=C_N^k.$$ Отсюда получаем $$\pi_0=\left( 1 + C_N^1 + \dots + C_N^N \right)^{-1}=2^{-N}.$$ Далее получаем $\pi_k = 2^{-N}C_N^k$, $k=1,\dots,N$. Итак, для произвольного состояния $$\pi_k = 2^{-N} C_N^k = C_N^k \left(\frac{1}{2}\right)^{k} \left(\frac{1}{2}\right)^{N-k}, \ k=0,\dots,N.$$ 

\elena{ Отметим, что данный результат можно получить, рассмотрев дискретную модель Эренфестов, найти для неё стационарное распределение (см. пример~\ref{Erenfest invar}) $\tilde\pi_k = 2^{-N} C_N^k$ и воспользоваться замечанием~1 в конце раздела~\ref{sec:Classification ContMCh}.}

\elena{Заметим, что стационарное распределения является биномиальным распределением}  с параметрами $N$ и $p=1/2$, т.е. $\textrm{Bi}(N,1/2)$. \elena{Иначе говоря, стационарное распределение соответствует ситуации, когда все $N$ блох независимо и  равновероятно распределились по обеим собакам (классическая схема Бернулли).} При больших значениях $N$ \elena{согласно теореме Муавра--Лапласа} это распределение приближенно совпадает с распределением $\mathrm{N}(N/2,N/4)$. Отсюда следует, что $$\lim\limits_{N\to\infty}\lim\limits_{t\to\infty}\mathbb{P}\left( \elena{2} \frac{|\xi(t)-N/2|}{\sqrt{N}} \le \mathrm{N}_{\frac{1+\gamma}{2}}(0,1) \right) = \elena{\int\limits_{-\mathrm{N}_{\frac{1+\gamma}{2}}(0,1)}^{\mathrm{N}_{\frac{1+\gamma}{2}}(0,1)}\frac{e^{-\frac{u^2}{2}}}{\sqrt{2\pi}}\,du=}\gamma, $$ где $\mathrm{N}_{\frac{1+\gamma}{2}}(0,1)$ -- это $(1+\gamma)/2$-квантиль распределения $\mathrm{N}(0,1)$. Например, если $\gamma=0.99$, то $\mathrm{N}_{\frac{1+\gamma}{2}}(0,1)=2.5758$. \elena{Итого, если время $t$ и число блох $N$ достаточно велики, то блохи почти поровну распределятся на собаках (поскольку $2|\xi(t) - N/2 |$ -- модуль  разности числа блох на собаках), при этом флуктуации не велики, а именно $O(\sqrt{N})$. На самом деле, этот результат является только предельным и не даёт скорости сходимости.}

\elena{Для получения более аккуратных оценок рассмотрим сначала дискретную цепь скачков $X_n$, $n=0,1,2,\ldots$ Обозначим математическое ожидание разности числа блох на первой и второй собаке в момент времени $n$ как
$$
a_n = \Exp (2X_n - N).
$$
Применив формулу полной вероятности, получаем, что
\begin{equation*}
    \begin{split}
        &\Exp X_n = \sum_{k=0}^N \Exp [X_n|X_{n-1}=k ]p_{k}(n-1) = \\
        &=\sum_{k=0}^N\left((k-1) \frac{k}{N} +(k+1) (1-\frac{k}{N})  \right) p_{k}(n-1) = \\
        &=\left( 1-\frac{2}{N}\right)\sum_{k=0}^N k p_{k}(n-1) + 1 = \left( 1-\frac{2}{N}\right)\Exp X_{n-1} + 1,
    \end{split}
\end{equation*}
откуда
\begin{equation*}
    \begin{split}
        &a_n = 2\Exp X_n - N = 2 \left( 1-\frac{2}{N}\right)\Exp X_{n-1} + 2-N=\\
        &=\left( 1-\frac{2}{N}\right)\left(2\Exp X_{n-1}-N \right) = \left( 1-\frac{2}{N}\right) a_{n-1},
    \end{split}
\end{equation*}
а значит
$$
a_n = \left( 1-\frac{2}{N}\right)^n a_{0}.
$$
}

\elena{ Таким образом, первый момент для случайной величины, равной разности числа блох между собаками, с экспоненциальной скоростью стремится к нулю, однако, чем больше $N$, тем медленнее. Это утверждение имеет место для случая, когда $a_0\neq 0$. Иначе, когда $a_0 = 0$, то есть в начальный момент времени распределение блох между собаками симметрично, то очевидно, что в последующие моменты времени распределение блох между собаками будет также симметричным, что даст $a_n = 0$ для любого $n$.}

\elena{
Проведем аналогичные вычисления для второго момента.
Обозначим математическое ожидание квадрата разности числа блох на первой и второй собаке в момент времени $n$ как
$$
b_n = \Exp (2X_n - N)^2.
$$
Применив формулу полной вероятности, получаем, что
\begin{equation*}
    \begin{split}
        &\Exp X_n^2 = \sum_{k=0}^N \Exp [X_n^2|X_{n-1}=k ]p_{k}(n-1) = \\
        &=\sum_{k=0}^N\left((k-1)^2 \frac{k}{N} +(k+1)^2 (1-\frac{k}{N})  \right) p_{k}(n-1) = \\
        &=\left( 1-\frac{4}{N}\right)\Exp X_{n-1}^2 + 2\Exp X_{n-1} + 1,
    \end{split}
\end{equation*}
откуда
\begin{equation*}
    \begin{split}
        &b_n = \Exp \left(2 X_n - N\right)^2 = \left( 1-\frac{4}{N}\right) b_{n-1} + 4,
    \end{split}
\end{equation*}
а значит
$$
b_n = \left( 1-\frac{4}{N}\right)^n b_{0}+ 4\sum_{k=0}^{n} \left( 1-\frac{4}{N}\right)^k = \left( 1-\frac{4}{N}\right)^n b_{0}+ N\left [1-\left( 1-\frac{4}{N}\right)^n \right].
$$
}

Теперь найдем среднее время возвращения в состояние 0. 
\elena{Рассмотрим сначала дискретную модель. Согласно принятым в замечании 1 раздела 6.2.2  обозначениям средние времена первого достижения состояния $j$ из состояния $i$ $\mu_{ij}$ удовлетворяют системе линейных уравнений~\eqref{mu ij}. При чем $\mu_i = \pi_i^{-1}$. В частности, 
$$
\mu_0 = 2^N. $$
 Для качественного понимания полученных результатов заметим, что 
$$
\mu_{N/2} = \frac{2^N}{C_N^{\frac{N}{2}}}\sim \frac{1}{\lambda}\sqrt{\frac{\pi N}{2}},\quad N\to\infty.
$$
Таким образом, для больших значений $N$, если цепь находится в состоянии $0$, то за разумное время наблюдения цепь не возвращается туда. И напротив, если цепь находится в состоянии макроравновесия $N/2 = \arg\max_{k} \pi_k$, то циклы возвращения очень короткие.
}

\elena{
Для перенесения результатов на непрерывную модель нужно заметить, что время между скачками в непрерывном случае имеет показательное распределение со средним $1/(\lambda N)$. Иначе говоря, ответы для непрерывной модели получаются из дискретной делением на  $\lambda N$ (см. также теорему 12.3 на с. 289 в~\cite{KaiLai1964} или раздел~\ref{sec:ContChainMarkov} данного пособия).
}

Перейдем теперь к \ag{немного другому описанию}
модели Эренфестов. Нас \ag{по-прежнему} будет интересовать поведение макросистемы ($N\to\infty$) на продолжительных временах ($t\to\infty$) и ее равновесное состояние.

Пусть дана непрерывная цепь Эренфестов $\xi(t)$ с ${N+1}$ состоянием \elena{и начальным распределением $\mathbb{P}(\xi(0)=0)$ (все блохи в начальный момент времени находятся на второй собаке)}. Рассмотрим $N$ независимых непрерывных цепей Маркова $\eta_i$, $i=1,\dots,N$ (см.~рис.~\ref{fig:ForEhrenfestModel}) 
с одинаковым начальным распределением $\mathbb{P}(\eta_i(0)=0)$. \elena{Иначе говоря, каждая цепь $\eta_i$ соответствует поведению $i-$ой блохи, $i=1,\dots,N$ (поведения всех блох независимы друг от друга), так что $\eta_i(t) = 1$, если в момент времени $t$ $i-$ая блоха находится на первой собаке, и $\eta_i(t) = 0$, если в момент времени $t$ $i-$ая блоха находится на второй собаке.} Тогда  $$\xi(t)\overset{d}{=}\sum\limits_{i=1}^N \eta_i(t).$$ Отсюда следует, в частности, что $$\mathbb{P}(\xi(t)=0)=\prod\limits_{i=1}^N\mathbb{P}(\eta_i(t)=0)=\mathbb{P}(\eta_1(t)=0)^N.$$ В примере~\ref{eq:EhrenfestProb1} мы выяснили, что при $t\to\infty$ вероятность $\mathbb{P}(\xi(t)=0)$ сходится к $(1/2)^N$, следовательно, $\mathbb{P}(\eta_i(t)=0)\to1/2$ при $t\to\infty$ для любого $i=1,\dots,N$.

\begin{figure}[!h]
	\centering
	\includegraphics[scale=0.6]{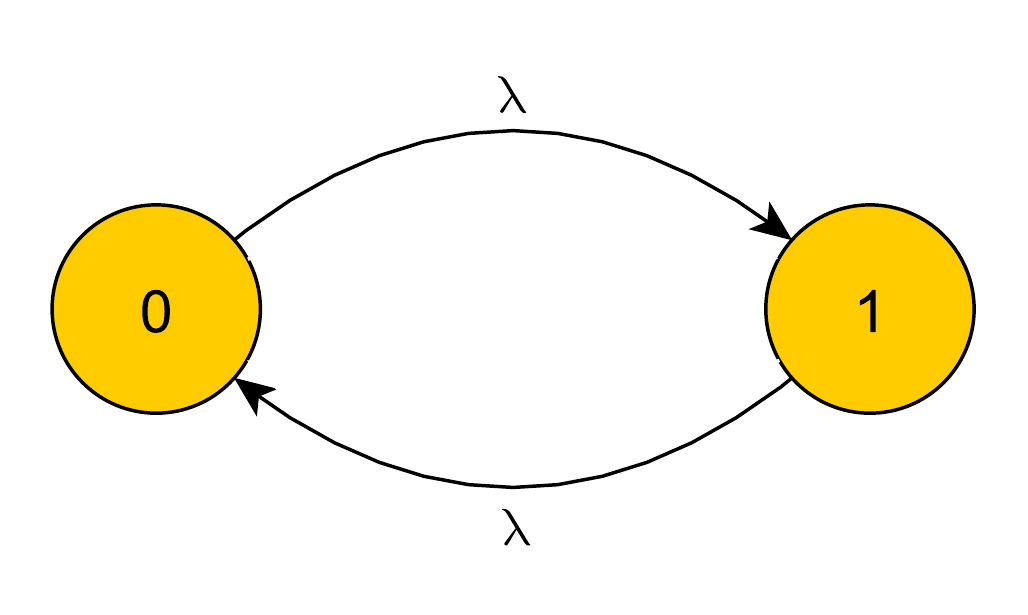}
	\caption{Стохастический граф цепи Маркова $\eta_i$.}
	\label{fig:ForEhrenfestModel}
\end{figure}

Теперь пусть ${\nu_0(t)=1-\xi(t)/N}$, ${\nu_1(t)=\xi(t)/N}$ -- доли цепей $\eta_i$, которые в момент $t$ находятся в состоянии 0 и в состоянии 1 соответственно. Тогда
$$\mathbb{P}(\nu_0(t)=c_1,\nu_1(t)=c_2)=\frac{N!}{(c_1N)!(c_2N)!}\mathbb{P}(\eta_1(t)=0)^{c_1N}\mathbb{P}(\eta_1(t)=1)^{c_2N}.$$ Перейдем к пределу при $t\to\infty$ и получим
$$\lim\limits_{t\to\infty}\mathbb{P}(\nu_0(t)=c_1,\nu_1(t)=c_2)=\frac{N!}{(c_1N)!(c_2N)!}\frac{1}{2^N}.$$ Далее перейдем к пределу по $N\to\infty$ и, воспользовавшись формулой Стирлинга $n! = n^n e^{-n} \sqrt{2\pi n} (1+O(1/n))$, получим
$$\lim\limits_{t\to\infty}\mathbb{P}(\nu_0(t)=c_1,\nu_1(t)=c_2) \approx \frac{2^{-N}}{\sqrt{2\pi c_1 c_2}\sqrt{N}}\frac{1}{c_1^{c_1N}c_2^{c_2N}},$$ откуда
$$\lim\limits_{t\to\infty}\mathbb{P}(\nu_0(t)=c_1,\nu_1(t)=c_2) \approx \frac{2^{-N}}{\sqrt{2\pi c_1 c_2}\sqrt{N}} \exp{(-N H(c_1,c_2))},$$ где $H(c_1,c_2)=-c_1\ln{c_1}-c_2\ln{c_2}$. 

Зададимся теперь вопросом, при каких $c_1$ и $c_2$ система оказывается равновесной. Под равновесием будем понимать такое состояние, т.е. такой вектор $(c_1,c_2)$, в малой окрестности которого концентрируется стационарная мера, т.е. вероятность принятия частотами $\nu_0(t)$ и $\nu_1(t)$ значений $c_1$ и $c_2$ при $t\to\infty$. Оценка этой вероятности при больших $N$ выписана выше. Максимизация этого выражения при $c_1\ge0$, $c_2\ge0$, $c_1+c_2=1$ равносильна минимизации функции $H(c_1,c_2)$ при тех же $c_1$, $c_2$. Решением являются значения ${c_1=c_2=1/2}$.

В более общих случаях макросистем оценку стационарной меры при большом числе степеней свободы $N$ удается произвести с использованием теорем И.\,Н.~Санова~\cite{Sanov1957} о сходимости эмпирических мер на произвольных алфавитах (в нашем случае алфавитом являлось множество значений $\{0,1\}$ цепей $\eta_i(t)$).

Интересно отметить, что функция $-H(c_1,c_2)$ (с минусом) оказывается функцией энтропии рассматриваемой системы, а равновесное состояние отвечает максимуму энтропии. Это проявление принципа максимума энтропии Больцмана--Джейнса охватывает и более общие случаи макросистем\gav{, см., например,~\cite{Baimurzina2015} и следующ}\eduard{ий раздел}.

Тот же результат можно получить, рассматривая пределы в обратном порядке: сначала по ${N\to\infty}$, затем по ${t\to\infty}$. А именно, предположим, что при $t=0$ существует предел 
\begin{equation}\label{eq:cilimit}
    \lim\limits_{N\to\infty} \nu_i(t) \overset{\text{п.н.}}{=} c_i(t).
\end{equation}
Оказывается (теорема Т. Куртца~\cite{EthierKurtz}), что в этом случае предел~\eqref{eq:cilimit} существует и при любом другом $t>0$, причем $c_1(t)$ и $c_2(t)$ -- неслучайные функции, удовлетворяющие системе обыкновенных дифференциальных уравнений
$$\frac{dc_1}{dt}=\lambda(c_2-c_1),$$
$$\frac{dc_2}{dt}=\lambda(c_1-c_2).$$
Положение равновесия этой системы существует, единственно и соответствует ${c_1=c_2=1/2}$. Интересно отметить, что функция $H(c_1,c_2)$ является функцией Ляпунова этой системы (убывает на траекториях этой системы и имеет минимум в точке $c_1=c_2=1/2$).

Итак, мы \gav{продемонстрировали} три точки зрения на изучение эволюции и равновесия макросистемы на больших временах: 1) стохастическую (концентрация меры), \gav{2) механическую (функция Ляпунова) и 3) термодинамическую (максимум энтропии). Причем к 3) можно прийти как из 1), так и из 2).}

\gav{В следующем} \eduard{разделе} \gav{описанная выше схема для графа специального вида, изображенного на рисунке~\ref{fig:ForEhrenfestModel}, будет перенесена на произвольные графы.}

Оценки времени выхода марковских цепей на стационарное распределение~\cite{KelbertSukhov2010} (mixing time) особенно  важны в изучении эффективности алгоритмов, в основе которых лежит Markov Chain Monte Carlo, см. следующее приложение.

    \section{Вектор PageRank и Google Problem}\label{sec:PageRank}

В данном разделе 
на примере ранжирования web-страниц приводится наглядный способ интерпретации основного уравнения (Кол\-мо\-го\-ро\-ва--Чэпмена) и основной теоремы (эргодической) теории однородных дискретных марковских цепей, изложенной в предыдущих разделах пособия. Также в данном разделе на примере задачи ранжирования web-страниц демонстрируется важный в различных современных приложениях алгоритм Markov Chain Monte Carlo~\cite{Diaconis} и закрепляется материал из предыдущего раздела про то, как можно понимать равновесия макросистем. 
Для более глубокого погружения в описываемые далее вопросы рекомендуем~\cite{pagerank1,pagerank2} и литературу, на которую в этих статьях ссылаются. Для закрепления материала мы умышленно повторяем некоторые результаты, изложенные ранее.

\subsection{Google problem и эргодическая теорема}




В 1998 г. Ларри Пейджем и Сергеем Брином был предложен специальный способ 
ранжирования web-страниц, который и лёг в основу поисковой системы Google.

Рассмотрим ориентированный взвешенный граф (см.~рис.~\ref{fig:classes}). Граф имеет $n$ 
вершин. Каждой паре вершин соответствует некоторый вес $p_{i,j} \ge 
0$. Ребро, выходящее из вершины $i$ в вершину $j$, имеет вес $p_{i,j} > 
0$. Если из вершины $i$ в вершину $j$ ребра нет, то полагаем $p_{i,j} =0$. 
Число $p_{i,j} $ интерпретируется как вероятность перейти из вершины $i$ в 
вершину $j$.  Набор чисел $\left\{ {p_{i,j} } \right\}_{i,j=1,1}^{n,n} $ 
удобно будет записать в виде матрицы $P=\left\| {p_{i,j} } 
\right\|_{i,j=1,1}^{n,n}$. 
Поскольку распределение вероятностей должно быть нормировано на 
единицу, то для любой вершины $i$ имеет место равенство $\sum_{j=1}^n 
{p_{i,j} } =1$, и матрица $P$ является стохастической по строкам. 



На граф можно посмотреть как на город, вершинами которого являются различные 
районы города, а ребрами -- дороги (вообще говоря, с односторонним движением). 
Предположим, что в городе имеется <<Красная площадь>> -- такой район, в 
который можно попасть по дорогам из любого другого. Оказывается, в этом 
предположении верен следующий результат (\textbf{\textit{эргодическая 
теорема для марковских процессов}} \gav{в варианте теоремы~\ref{th:technical_ergodic}}): \textit{если пустить блуждать человека по городу в течение длительного времени так, что человек будет случайно перемещаться из района в район согласно весам ребер графа, то доли }$\left\{ \nu_k \right\}_{k=1}^n $\textit{ времени, которые человек провел в разных районах, будут удовлетворять следующему уравнению}: $\sum_{i=1}^n {\nu_i p_{i,j} } =\nu_j $\textit{ или в~векторном виде }$\nu^\top P=\nu^\top$\textit{, имеющему единственное решение, удовлетворяющее} $\sum_{i=1}^n 
{\nu _i } =1$. В связи с последним представлением говорят, что
$\nu = [\nu_1, \dots, \nu_n]^\top$ является левым собственным вектором 
матрицы $P$. Вектор $\nu$ также называют инвариантным или стационарным распределением вероятности в марковском процессе. В теории неотрицательных матриц~\cite{Nikaydo}, используемой в математической экономике, вектор $\nu^T$ также называют \textit{вектором Фробениуса--Перрона}. 
Предположение о <<Красной площади>> обеспечивает единственность $\nu $. 
 Если это предположение не 
верно, то вектор $\nu $ существует, 
но он не единственен и зависит от района, из которого стартовал человек. В общем случае город 
распадается на отдельные несвязанные между собой существенные 
кластеры\footnote{ То есть из каждого такого кластера в любой другой кластер 
дорог нет.} из районов, связанных между собой внутри кластера, и отдельного 
кластера несущественных районов. Последний определяется тем, что из 
любого его района можно попасть в один из существенных кластеров, но 
попасть обратно в несущественный кластер из существенных кластеров 
невозможно. Из такого несущественного кластера человек в 
конечном итоге попадет в один из существенных кластеров, откуда уже не 
выйдет. Рисунок~\ref{fig:classes} из раздела~\ref{sec:DiscreteChains} поясняет общий случай.


Чтобы понять, откуда получается приведенная выше в 
вольной трактовке эргодическая теорема, получим с помощью формулы полной вероятности уравнение 
Колмогорова--Чэмпена. Это уравнение является основным при описании \textit{дискретных однородных марковских цепей}. Обозначим через 
$p_i \left( t \right)$  вероятность того, что человек находится в момент 
времени $t$ в районе $i$. Тогда по формуле полной вероятности для любого 
$j=1,...,n$
$$
p_j(t+1) = \sum\limits_{i=1}^n {\underbrace {\PP\left( 
{\begin{array}{c}
 \mbox{\tiny человек в момент времени $t$} \\ 
 \mbox{\tiny находился в районе $i$} \\ 
 \end{array}} \right)}_{p_i(t)}}
 \underbrace {\PP\left( 
{\begin{array}{c}
 \mbox{\tiny человек перешел в район $j$ при} \\ 
 \mbox{\tiny условии, что был в районе $i$} \\ 
 \end{array}} \right)}_{p_{i,j}}.
$$

Или в матричном виде
\begin{equation}
\label{eq1}
p^\top(t+1)=p^\top(t)P.
\end{equation}
Последнее равенство и называют \textit{уравнением 
Колмогорова--Чэп\-мена}. 

Предположим теперь, что существует предел $\mathop {\lim }_{t\to 
\infty } p(t)=\nu $. Какому уравнению должен удовлетворять 
вектор $\nu $? Переходя к пределу в обеих частях равенства~(\ref{eq1}) и учитывая, 
что $\mathop {\lim }_{t\to \infty } \left( {p^\top( t )P} 
\right)=$\linebreak $=\mathop {\lim }_{t\to \infty } \left( {p^\top\left( t \right)} 
\right)P$, получим уже известное нам соотношение
\begin{equation}
\label{eq2}
\nu^TP=\nu ^T.
\end{equation}
Именно из соотношения~ \eqref{eq2} и было предложено искать вектор ранжирования 
web-страниц, также назваемый \textit{вектор PageRank}, в модели Брина--Пейджа. 
Отличие этой модели от описанной выше только в интерпретации. Теперь вершины графа -- это  
web-страницы, ребра графа это гиперссылки. Под \textit{Google problem} будем понимать задачу поиска вектора $\nu$ в этой модели.


Отметим, что приведенные выше рассуждения справедливы в 
предположении существования предела $\mathop {\lim }_{t\to \infty } 
p\left( t \right)=\nu $. Казалось бы, что предположения о <<Красной 
площади>> будет достаточно и тут.\footnote{В непрерывном времени так оно и есть, см.~раздел~9.} Однако, как показывает простейший пример 
(рис.~\ref{fig3}), в котором $n=2$, $p_{1,1} =p_{2,2} =0$, $p_{1,2} =p_{2,1} =1$, хотя 
вектор $\nu =\left( {1 \mathord{\left/ {\vphantom {1 2}} \right. 
\kern-\nulldelimiterspace} 2,1 \mathord{\left/ {\vphantom {1 2}} \right. 
\kern-\nulldelimiterspace} 2} \right)^\top$ существует и единственен, 
предел $\mathop {\lim }_{t\to \infty } p\left( t \right)$ не существует, 
поскольку с ростом $t$ будет происходить периодическое чередование нулей и 
единиц в каждой компоненте вектора $p\left( t \right)$.

\begin{figure}[htbp]
\centerline{\includegraphics[width=0.5\textwidth]{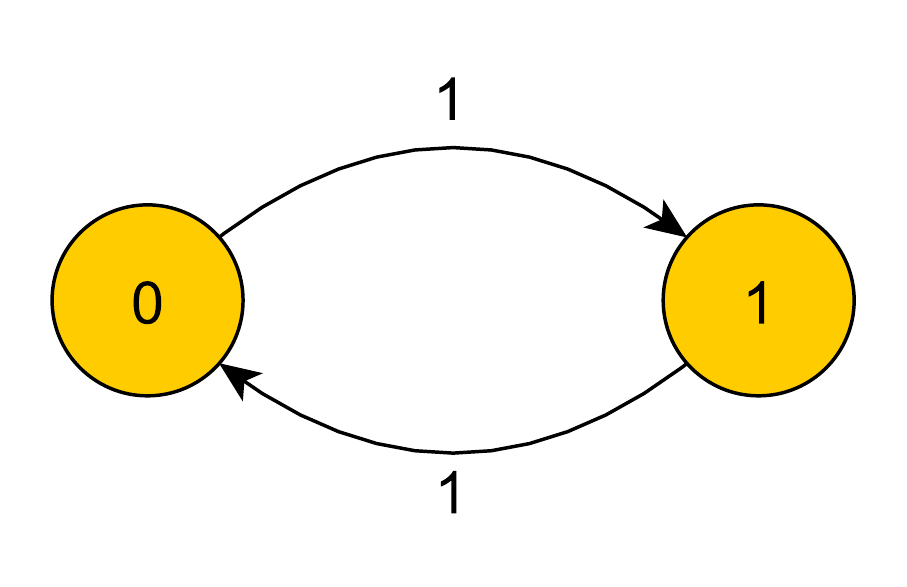}}
\caption{Периодическая марковская цепь c периодом 2}
\label{fig3}
\end{figure}

Оказывается, что если выполняется условие <<непериодичности>>, то предел 
$\mathop {\lim }_{t\to \infty } p\left( t \right)=\nu $ 
действительно существует. Более того, эти два условия (существования 
<<Красной площади>> и <<непериодичности>>) являются не только достаточными, 
но и необходимыми для существования предела. Опишем, в чем заключается 
условие <<непериодичности>>. Из <<Красной площади>> выходит много различных 
маршрутов, которые в конце снова приводят на <<Красную площадь>>. Условие 
<<непериодичности>> означает, что наибольший общий делитель 
последовательности длин всевозможных маршрутов (начинающихся и 
заканчивающихся на <<Красной площади>>) равен 1. Уточним, что длина маршрута 
равна числу ребер, вошедших в маршрут. В типичных web-графах оба отмеченных 
условия выполняются, поэтому в дальнейшем мы считаем эти два условия выполненными.

\subsection{Стандартные численные подходы к решению задачи}


С одной стороны, вектор $\nu$ является решением уравнения~\eqref{eq2} $\nu^\top\! P\!\!=$\linebreak $= \nu^\top$. С другой стороны, он является пределом $\nu =\mathop {\lim }_{t\to 
\infty } p{\left( t \right)}$, где $p{\left( t \right)}$ рекуррентно 
рассчитывается согласно уравнению Колмого\-рова--Чепмена $p{\left( t+1 \right)} = P^\top p{\left( t \right)}$. Значит, это соотношение можно использовать для численного расчета $p{(t)}$ как приближения к вектору PageRank (иными словами, можно применить метод простой итерации). Перейдем к анализу скорости сходимости такого метода.

Назовем \textit{спектральной щелью} матрицы $P$ наименьшее такое число $\alpha =1-\left| \beta 
\right|>0$, где (вообще говоря, комплексное) число $\beta $ удовлетворяет 
условию: $\left| \beta \right|<1$ и существует такой вектор $\eta $, что 
$\eta ^T P=\beta \eta ^T$. Хотя бы одно такое $\beta $ существует. Другими 
словами,  $\alpha $ -- расстояние между максимальным собственным значением 
матрицы $P$ (для стохастических матриц оно всегда равно 1) и следующим по 
величине модуля. Отсюда и название -- спектральная щель (спектральный зазор 
-- от англ. spectral gap). Оказывается, что имеет место следующий результат:
\begin{equation}
\label{eq3}
\left\| {p\left( t \right)-\nu } \right\|_1 \mathop =\limits^{def} 
\sum\limits_{k=1}^n {\left| {p_k \left( t \right)-\nu _k } \right|} \le 
C\exp \left( {-{\alpha t} / \tilde{C}} \right).
\end{equation}
Здесь и далее константы $C$, $\tilde {C}$ будут обозначать некоторые (каждый 
раз свои) универсальные константы, которые зависят от некоторых 
дополнительных деталей постановки и которые, как правило, ограничены числом 
10. 

Заметим, что часто факт сходимости процесса $p{\left( t+1 \right)} = P^\top p{\left( t \right)}$ к решению $\nu^\top P = \nu^\top$ также называют 
эргодической теоремой. В~предыдущем разделе мы видели, что для этого есть основания. Однако сходимость процесса $p{\left( t+1 \right)} = P^\top p{\left( t \right)}$ больше связана с принципом сжимающих 
отображений. Если под пространством понимать всевозможные лучи неотрицательного 
ортанта, а под метрикой на этом пространстве, в которой матрица $P$ 
<<сжимает>>, понимать метрику Биркгофа--Гильберта, то можно показать, что сходимость процесса $p{\left( t+1 \right)} = P^\top p{\left( t \right)}$ к решению $\nu^\top P = \nu^\top$ следует из принципа сжимающих отображений. К сожалению, детали здесь не совсем 
элементарны, поэтому мы ограничимся только следующей аналогией. Формула~(\ref{eq3}) 
отражает геометрическую скорость сходимости, характерную для принципа 
сжимающих отображений, а коэффициент $\alpha $ как раз и характеризует 
степень сжатия, осуществляемого матрицей $P$. Геометрически себе это можно 
представлять (правда, не очень строго) как сжатие с коэффициентом, не меньшим 
$1-\alpha $ к инвариантному направлению, задаваемому вектором $\nu $.

С одной стороны, оценка~(\ref{eq3}) всего лишь сводит задачу оценки скорости 
сходимости метода  $p{\left( t+1 \right)} = P^\top p{\left( t \right)}$ к задаче оценки спектральной щели $\alpha $. С другой стороны,
последняя задача в ряде случаев может быть эффективно решена. В частности, к 
эффективным инструментам оценки $\alpha $ относятся изопериметрическое 
неравенство Чигера и неравенство Пуанкаре, которое также 
можно понимать как неравенство концентрации меры. Имеются и другие способы оценки 
$\alpha $, однако все эти 
способы далеко выходят за рамки нашего курса. Поэтому здесь мы 
ограничимся простыми, но важными в контексте рассматриваемых приложений 
случаями.

Предположим сначала, что для рассматриваемого web-графа существует 
такая web-страница, на которую есть ссылка из любой web-страницы, в том числе 
из самой себя (усиленный аналог предположения о наличии <<Красной площади>> 
-- поскольку теперь на <<Красную площадь>> из любого района есть прямые дороги), 
более того, предположим, что на каждой такой ссылке стоит вероятность, не 
меньшая, чем $\gamma $. Для такого web-графа имеет место неравенство $\alpha 
\ge \gamma $.

Предположим далее, что в модели блуждания по web-графу имеется 
<<телепортация>>: с вероятностью $1-\delta $ человек действует как 
в исходной модели, а с вероятностью $\delta $ <<забывает про все 
правила>> и случайно равновероятно выбирает среди $n$ вершин одну, в которую и 
переходит. 
Тогда, если ввести квадратную матрицу $E$ размера $n$ на $n$, состоящую из 
одинаковых элементов $1 \mathord{\left/ {\vphantom {1 n}} \right. 
\kern-\nulldelimiterspace} n$, уравнение $p{\left( t+1 \right)} = P^\top p{\left( t \right)}$ примет вид
\begin{equation}
\label{eq4}
p{\left( {t+1} \right)}=\left( {\left( {1-\delta }\right)P^\top+\delta E} \right)p{\left( t \right)}.
\end{equation}
В таком случае вектор PageRank необходимо будет искать из уравнения
\begin{equation}
\label{eq5}
\nu =\left( {\left( {1-\delta } \right)P^\top+\delta E} \right)\nu.
\end{equation}
При $0<\delta <1$ уравнение~(\ref{eq5}) гарантированно 
имеет единственное в классе распределений 
вероятностей решение. Более того, для спектральной щели матрицы $\left( {1-\delta } \right)P+\delta E$ имеет место оценка $\alpha 
\ge \delta $. На практике для вычисления вектора 
PageRank обычно используют уравнение~(\ref{eq5}) c $\delta =0.15$.

Поскольку матрица $P$ для реальных web-графов обычно сильно разреженная -- большая часть элементов равна нулю, то 
использовать формулу~(\ref{eq4}) в таком виде не рационально: одна 
итерация будет стоить $3n^2$ арифметических операций типа сложения и умножения чисел. Однако формулу~(\ref{eq4}) можно переписать следующим образом: 
\begin{equation}
\label{eq6}
p{\left( {t+1} \right)}=\left( {1-\delta } \right)P^\top p{\left( {t} \right)}+\delta \cdot \left( {1 \mathord{\left/ {\vphantom {1 n}} \right. 
\kern-\nulldelimiterspace} n,...,1 \mathord{\left/ {\vphantom {1 n}} \right. 
\kern-\nulldelimiterspace} n} \right)^\top.
\end{equation}
Расчет по этой формуле требует по порядку лишь $2sn$ арифметических операций, где 
$s$ -- среднее число ненулевых элементов в строке матрицы $P$. 
Для Интернета $n\approx 10^{10}$, а $s\ll 10^4$. 

Выше был описан исторически самый первый алгоритм, использовавшийся 
для расчета вектора PageRank. Он получил название метода простой итерации 
(МПИ). Как было отмечено, МПИ эффективен, если $\alpha $ не очень близко к нулю. 
В частности, как в модели с телепортацией. Действительно, современный 
ноутбук в состоянии выполнять до $10^{10}$ арифметических операций в секунду. 
С учетом программной реализации и ряда других ограничений этот порядок на 
практике обычно уменьшается до $10^8$. Из оценки~(\ref{eq3}) видно, что 
сходимость, например, при $\alpha \ge 0.15$ очень быстрая. Хорошая точность 
получается уже после нескольких десятков итераций. С учетом того, что одна 
итерация требует по порядку $2sn\le 10^{13}$ арифметических операций, описанный метод позволяет за день найти вектор PageRank для всего Интернета. К 
сожалению, на самом деле все не так просто. Проблема в памяти, в которую 
необходимо загружать матрицу $P$. Разумеется, $P$ необходимо загружать не 
как матрицу из $n^2$ элементов, а в виде, так называемых, списков смежностей 
по строкам.
Однако это все 
равно не решает проблемы. Быстрая память компьютера -- кэш-память разных 
уровней. Кэш-память процессора совсем маленькая, ее совершенно недостаточно. 
Более медленная -- оперативная память. Тем не менее 
если удалось бы выгрузить матрицу $P$ в такую память, то производительность 
программы соответствовала бы оценке, приведенной выше. Обычно оперативной памяти в современном персональном компьютере не более нескольких десятков Гигабайт ($\sim 10^{10}$ байт), 
чего, очевидно, недостаточно. Следующая память -- жесткий диск. Обращение 
программы к этой памяти, по сути, останавливает нормальную работу программы. 
Как только кончаются ресурсы оперативной памяти, мы видим, что программа 
либо не работает совсем, либо начинает работать очень медленно. Отмеченную 
проблему можно решать, увеличивая оперативную память, либо используя 
распределенную память.

Но что делать, если $\alpha$ оказалось достаточно малым или мы не можем 
должным образом оценить снизу $\alpha$, чтобы гарантировать быструю 
сходимость МПИ? В таком случае полезным оказывается следующий  результат~\cite{PolyakTremba}, 
не предполагающий, кстати, выполнения условия <<непериодичности>>,
\begin{equation}
\label{eq7}
\left\| {P^\top\bar {p}_T -\bar {p}_T } \right\|_1 \le \frac{C}{T},
\quad
\bar {p}_T \stackrel{def}{=} \frac{1}{T}\sum\limits_{t=1}^T {p{\left( t \right)}} .
\end{equation}
Эта оценка уже никак не зависит от $\alpha $. Параллельное процедуре $$p{\left( t+1 \right)} = P^\top p{\left( t \right)}$$ суммирование получающихся векторов позволяет вычислять вектор $\bar {p}_T $ за время, по порядку величины равное времени расчета $p{\left( T \right)}$. К 
сожалению, для ряда задач, например, исследования степенного закона убывания 
компонент вектора PageRank, метод Поляка--Трембы~\gav{\cite{PolyakTremba}} не подходит. Причина связана 
с видом оценки~(\ref{eq7}). Имеется принципиальная разница с оценкой~(\ref{eq3}) в том, что 
в оценке~(\ref{eq3}) мы можем гарантировать близость найденного вектора $p{\left( T 
\right)}$ к вектору PageRank $\nu $, обеспечив малость $\left\| {p{\left( T 
\right)}-\nu } \right\|_1 $. Что касается соотношения~(\ref{eq7}), то в типичных 
случаях из $\left\| {P^\top\bar {p}_T -\bar {p}_T } \right\|\approx \varepsilon 
$ можно при больших $n$ лишь получить, что $\left\| {\bar {p}_T -\nu } \right\| \approx \varepsilon / \alpha$, т.е. опять возникает <<нехорошая>> зависимость в оценке от $\alpha$. Для симметричной матрицы $P$ и 2-нормы этот результат был получен Красносельским--Крейном~\cite{KrasnoselskyCrane} в 1952 г.

К сожалению, и другие способы поиска вектора PageRank, которые на первый 
взгляд не используют в своих оценках $\alpha $, на деле оказываются 
методами, выдающими такой вектор $\tilde {p}_T $, что $\left\| {P^\top\tilde 
{p}_T -\tilde {p}_T } \right\|\approx \varepsilon $ в некоторой норме 
(обычно это 1-норма, 2-норма и бесконечная норма -- в следующем пункте мы 
поясним, что имеется в виду под этими нормами), что приводит к той же сложности, что и в методе Поляка--Трембы. Тем не менее хочется отметить большой прогресс, 
достигнутый за последнее время (во многом благодаря разработкам Б.\,Т.~Поляка, 
А.\,С.~Немировского, Ю.\,Е.~Нестерова), в создании 
эффективных численных методов решения задач выпуклой оптимизации вида 
($b=1,2$; $l=1,2,\infty )$
\[
\left\| {P^\top p-p} \right\|_l^b \to \mathop {\min }\limits_{p\ge 
0:\;\sum\limits_{k=1}^n {p_k } =1} .
\]
Эти наработки оказываются полезными, поскольку вектор PageRank можно также 
понимать, как решение такой задачи. Действительно, всегда 
имеет место неравенство  $\left\|{P^\top p-p} 
\right\|~\ge~0$, и только на $\nu $ имеет место равенство $\left\|{P^\top \nu 
-\nu } \right\|=0$.

Отметим, в частности, метод условного градиента, который позволяет получить вектор $\tilde {p}_T $, удовлетворяющий $\left\| 
{P^\top\tilde {p}_T -\tilde {p}_T } \right\|_2 \le \varepsilon $,\linebreak  за~$C\cdot \left( {n+s^2\varepsilon ^{-2}\ln n} \right)$ 
арифметических операций~\gav{\cite{GasnikovOpt}}. Удивительно, что в эту оценку не 
входит $sn$ -- число ненулевых элементов матрицы~$P$. В~частности, при 
$s\approx \sqrt n $ получается сложность пропорциональная $n$, а~не~$ n^{3 
\mathord{\left/ {\vphantom {3 2}} \right. \kern-\nulldelimiterspace} 2}$, 
как можно было ожидать (к сожалению, за это есть и плата в виде множителя $ 
\varepsilon ^{-2})$. Тем не менее еще раз повторим, что несмотря на все 
возможные ускорения вычислений использование таких методов в наших целях, к 
сожалению, не представляется возможным.

Отмеченные в этом разделе методы (МПИ и Поляка--Трембы) не исчерпывают 
множество методов, гарантирующих малость $\left\| {\tilde {p}_T -\nu } 
\right\|$, в оценку скорости сходимости которых входит $\alpha $. К таким 
методам можно отнести, например, метод Д. Спилмана, являющийся вариацией описанного выше метода 
Поляка--Трембы. 

\subsection{Markov Chain Monte Carlo и Google problem}\label{MCMC}

Далее мы сосредоточимся на так называемых методах Монте-Карло. Наряду с тем, что эти методы являются 
численными методами поиска вектора PageRank, они также позволяют по-новому 
проинтерпретировать вектор PageRank.

С одной стороны, в предположении достаточной величины спектральной щели МПИ обеспечивает очень быструю сходимость по норме приближенного решения к вектору PageRank. С другой стороны, одна итерация этого метода требует порядка $sn$ арифметических операций, что может привести к невозможности его применения, если не выполнено неравенство $s \ll n$. К тому же каждому поисковому запросу обычно соответствует свой web-граф релевантных страниц. Поэтому, как правило, для каждого запроса  нужно вычислять свой вектор PageRank. 
Следовательно, для эффективной работы реальной поисковой системы скорость вычисления вектора PageRank оказывается важнее точности его вычисления.
МПИ не позволяет в полной мере пожертвовать точностью в угоду скорости вычисления. В самом деле, согласно оценке~(\ref{eq3}), переход, скажем, от точности $\varepsilon \approx 10^{-6}$ к точности $\varepsilon \approx 10^{-3}$ уменьшит время счета всего лишь вдвое. 
Предложим другой алгоритм, который будет более чувствителен к точности по сравнению с МПИ, но при этом будет иметь меньшую стоимость итерации.

В основе этого алгоритма лежит метод \textit{Markov Chain Monte Carlo} (MCMC). Основная идея заключается в практическом использовании 
эргодической теоремы. То есть, грубо говоря, нужно запустить 
человека и достаточно долго подождать, собирая статистику посещенных им 
вершин графа, т.е. web-страниц. При оценке эффективности работы такого метода 
важно ответить на два вопроса: насколько эффективным можно сделать шаг и сколько 
шагов надо сделать человеку.

Наиболее вычислительно затратная операция на каждом шаге метода -- это случайный выбор ребра для перехода из текущей вершины. 
Случай, в котором из каждой вершины (web-страницы) выходит не более $s\ll n$ ребер, является простым. На каждом шаге за порядка $s$ арифметических операций разыгрывается случайная величина в соответствии с вероятностями переходов. При этом память после каждой операции 
освобождается, т.е. не требуется хранить в памяти ничего, кроме матрицы $P$ и 
текущего вектора частот. Отметим, тем не менее, как можно организовать вычисления в сложном случае, когда $s$ является 
достаточно большим. Можно до старта итераций алгоритма подготовить 
специальным образом память за порядка $sn$ арифметических операций, например, следующим образом. 
Каждой вершине ставится в соответствие свое разбиение отрезка $\left[ {0,1} 
\right]$ так, чтобы число подотрезков равнялось числу выходящих ребер, а длины 
подотрезков были пропорциональны вероятностям перехода по этим ребрам. Тогда на каждом 
шаге достаточно один раз генерировать равномерно распределенную на отрезке 
$\left[ {0,1} \right]$ случайную величину (стандартные генераторы, как 
правило, именно это в первую очередь и предлагают), и нанести ее на 
соответствующий текущей вершине подотрезок. В зависимости от того, в какой 
из подотрезков она попала, выбрать исходящее ребро и сдвинуть человека по нему. 
Недостаток этого подхода 
-- использование дополнительно памяти для хранения порядка $sn$ чисел и высокая стоимость подготовки этой памяти. 
К счастью, такая подготовка, требующая порядка $sn$ арифметических операций, нужна только один раз, в отличие от МПИ, в котором таких затрат требует каждая итерация. При дополнительных 
предположениях о компактном описании формул расчета $\left\{ {p_{i,j} } \right\}_{j=1}^n $ существуют быстрые способы случайного выбора ребра для перехода из текущей вершины, которые приводят 
к тому, что шаг выполняется за $\sim \log _2 s\le \log _2 n$ арифметических 
операций. 
В частности, если 
имеет место равенство между собой отличных от нуля внедиагональных элементов 
матрицы $P$ в каждой строке, то сложность шага $\sim \log _2 s$ достигается, 
например, методом Кнута--Яо.

Перейдем теперь к обсуждению оценки числа шагов, которое нужно сделать человеку для достижения заданной точности приближения вектора PageRank. Постараемся грубо сформулировать основной результат. \textit{Пусть} $\nu \left( T \right)$\textit{ -- вектор частот} (\textit{случайный вектор}) \textit{пребывания в различных вершинах блуждающего человечка после} $T$\textit{ шагов и $T\gg T_0 \mathop =\limits^{def} C\alpha 
^{-1}\ln \left( {n \mathord{\left/ {\vphantom {n \varepsilon }} \right. 
\kern-\nulldelimiterspace} \varepsilon } \right)$.  Тогда с вероятностью не меньше }$1-\sigma 
$\textit{ (здесь и везде в дальнейшем }$\sigma \in \left( {0,1} \right)$\textit{ -- произвольный доверительный уровень) имеет место неравенство}
\begin{equation}
\label{eq8}
\left\| {\nu \left( T \right)-\nu } \right\|_2  \le C\sqrt {\frac{\ln n+\ln 
\left( {\sigma ^{-1}} \right)}{\alpha T}}\gav{.} 
\end{equation}

Приведем сравнительные характеристики упомянутых к настоящему 
моменту методов. Под сложностью понимается количество арифметических 
операций типа умножения двух чисел, которые достаточно осуществить, чтобы (в 
случае MCMC: с вероятностью не меньше $1-\sigma )$ достичь точности решения 
$\varepsilon $ по целевой функции\footnote{ Заметим, что для метода MCMC 
имеет смысл рассматривать только $\varepsilon \ll n^{{-1} \mathord{\left/ 
{\vphantom {{-1} 2}} \right. \kern-\nulldelimiterspace} 2}$, поскольку 
$\left\| {\left( {n^{-1},...,\;n^{-1}} \right)^\top} \right\|_2 =n^{{-1} 
\mathord{\left/ {\vphantom {{-1} 2}} \right. \kern-\nulldelimiterspace} 
2}$.}, которая в табл.~1 обозначается целью. 

\begin{flushright}
	{Т\,а\,б\,л\,и\,ц\,а\; 1}
\end{flushright}

\begin{center}
	{\bf Сравнение методов решения задачи поиска вектора~PageRank}*
\end{center}
\vskip-5mm

{
	\begin{table}[h!]
		\begin{center}
			\begin{tabular}{|c|c|c|}
				\hline
				\textbf{Метод}& 
				\textbf{Сложность}& 
				\textbf{Цель} \\
				\hline
				1& 
				$\frac{sn}{\alpha }\ln \left( {\frac{2}{\varepsilon }} \right)$& 
				$\left\| {\tilde {p}_T -\nu } \right\|_1 $ \\
				\hline
				2& 
				$\frac{2sn}{\varepsilon }$& 
				$\left\| {P^T\tilde {p}_T -\tilde {p}_T } \right\|_1 $ \\
				\hline
				3& 
				$C \cdot \left( {n+\frac{s^2}{\alpha \varepsilon }\ln \left( {\frac{1}{\varepsilon }} \right)} \right)$& 
				$\left\| {\tilde {p}_T -\nu } \right\|_\infty  $ \\
				\hline
				4 & 
				$C\cdot \left( {n+\frac{\log _2 n\cdot \ln \left( {n \mathord{\left/ {\vphantom {n \sigma }} \right. \kern-\nulldelimiterspace} \sigma } \right)}{\alpha \varepsilon ^2}} \right)$& 
				$\left\| {\tilde {p}_T -\nu } \right\|_2 $ \\
				\hline
				5& 
				$C\cdot \left( {n+\frac{s^2\ln n}{\varepsilon ^2}} \right)$& 
				$\left\| {P^T\tilde {p}_T -\tilde {p}_T } \right\|_2 $ \\
				\hline
			\end{tabular}
			\label{tab1}
		\end{center}
	\end{table}
}
\vspace*{-5mm}\noindent{\small*Методы: 1 -- МПИ, 2 -- метод Поляка--Трембы, 3 -- метод Д. Спилмана,\\\hspace*{1.55cm}4 -- MCMC, 5 -- вариация метода  условного градиента.}
\vskip5mm

Напомним, что вектор $\tilde {p}_T $ -- это результат работы метода, а $\nu $ --\linebreak вектор PageRank. Стоит 
также пояснить обозначение $\left\| x \right\|_\infty\!\!\!\!\!\!\!=$\linebreak $=\mathop {\max 
}_{k=1,...,n} \left| {x_k } \right|$.
Сравнение затрудняется тем, что у всех методов разные целевые функции, а значит и разные критерии точности. 
Кроме того, разные методы могут по-разному быть ускорены с использованием параллельных вычислений.
Для задач огромных размеров,\linebreak к~которым, безусловно, относится и Google 
problem, этот вопрос выходит на передний план. Если умножение матрицы на 
вектор в МПИ хорошо параллелится, то, например, возможность применения параллельных вычислений к~MCMC в описанном здесь варианте не ясна.


\subsection{Параллелизуемый вариант MCMC}
\label{Fisher}

Идею распараллеливания, собственно так же, как и идеи остальных описанных 
методов (МПИ, MCMC), подсказывает <<жизнь>>. А именно, в реальном Интернете 
<<блуждает>> не один человек (пользователь), а много пользователей. 
Обозначим число независимо блуждающих пользователей через $N$. Рассмотрим 
сначала для простоты ситуацию, когда в графе всего две вершины, т.е. $n=2$. 
Поскольку пользователи блуждают независимо, то для каждого из них можно 
ожидать, что на шаге $T_0 =C\alpha ^{-1}\ln \left( {2 \mathord{\left/ 
{\vphantom {2 \varepsilon }} \right. \kern-\nulldelimiterspace} \varepsilon 
} \right)$ вероятность обнаружить пользователя в вершине 1 с 
хорошей точностью ($\sim \varepsilon )$ равна $\nu _1 $. Соответственно, вероятность обнаружить его в 
вершине 2 равна $\nu _2 =1-\nu _1 $. 
Таким образом, посмотрев на то, в какой вершине 
находился каждый человек на шаге $T_0 $, мы с хорошим приближением получим 
так называемую \textit{простую выборку} $X$ (независимые одинаково распределенные случайные 
величины) из \textit{распределения Бернулли} с вероятностью успеха (выпадения орла), равной $\nu _1 $. 
Последнее означает, что каждый человек (подобно монетке) на шаге $T_0 $ 
независимо от всех остальных с вероятностью $\nu _1 $ будет обнаружен в 
состоянии 1 (<<выпал орлом>>), а с вероятностью $1-\nu _1 $ будет обнаружен 
в состоянии 2 (<<выпал решкой>>). Чтобы оценить вектор PageRank в данном 
простом случае, достаточно оценить по этой выборке $\nu _1 $ -- частоту 
выпадения орла. Интуиция подсказывает, что в качестве оценки неизвестного 
параметра $\nu _1 $ следует использовать $r \mathord{\left/ {\vphantom {r 
N}} \right. \kern-\nulldelimiterspace} N$ -- долю людей, которые оказались в 
состоянии 1. Интуиция подсказывает правильно! Далее будет показано, что 
такой способ действительно оптимальный.

Вероятность, что число людей $r$, которые наблюдались в состоянии 1, равно 
$k$, может быть посчитана по формуле (\textit{биномиальное распределение})
\[
\PP\left( {r=k} \right)=C_N^k \nu _1^k \left( {1-\nu _1 } \right)^{N-k}.
\]
Используя грубый вариант \textit{формулы Стирлинга} $k!\simeq \left( {k \mathord{\left/ {\vphantom 
{k e}} \right. \kern-\nulldelimiterspace} e} \right)^k$, получаем \textit{оценку типа Санова}:
\[
C_N^k \nu _1^k \left( {1-\nu _1 } \right)^{N-k}\simeq \exp \left( {-N\cdot 
KL\left( {k \mathord{\left/ {\vphantom {k N}} \right. 
\kern-\nulldelimiterspace} N,\nu _1 } \right)} \right),
\]
где
\[
KL\left( {p,q} \right)=-p\ln \left( {\frac{p}{q}} \right)-\left( {1-p} 
\right)\ln \left( {\frac{1-p}{1-q}} \right).
\]
Отсюда, по \textit{неравенству Пинскера}:
\[
KL\left( {p,q} \right)\ge 2\left( {p-q} \right)^2,
\]
следует, что \textit{с вероятностью не меньше }$1-\sigma $\textit{ имеет место неравенство}
\begin{equation}
\label{eq9}
\left| {\frac{r}{N}-\nu _1 } \right|\le C\sqrt {\frac{\ln \left( {\sigma 
^{-1}} \right)}{N}} ,
\end{equation}
которое иллюстрируется на рисунке~\ref{fig:fig4}. 

\begin{center}
\begin{figure}[htbp]
\centering
\hspace{1.6cm} 
\begin{tikzpicture}
\definecolor{tempcolor}{RGB}{255,204,0}

    \def\normaltwo{\x,{4*1/exp(((\x-3)^2)/1.5)}}
 
    \def\y{5}
    \def\x{1}
 
    \def\fy{4*1/exp(((\y-3)^2)/1.5)}
    \def\fx{4*1/exp(((\x-3)^2)/1.5)}
 
    \fill[fill=tempcolor] ({\x},-0.5) -- plot[domain={\x}:{\y}, samples=200] (\normaltwo) -- ({\y},-0.5) -- cycle;
 
    \draw[color=black, thick, domain=0:6, samples=200] plot (\normaltwo) node[right] {};
 
    \draw[thick,dashed] (3,4) -- (3,-0.5) node[below] {};
    \draw[thick] (3, -0.55-0.5) node {$\nu_1$};
    
    \draw[thick,dashed] ({\y},{\fy}) -- ({\y},-0.5) node[below] {$\nu_1 - C \sqrt{\frac{\ln(\sigma^{-1})}{N}}$};
    \draw[thick,dashed] ({\x},{\fx}) -- ({\x},-0.5) node[below] {$\nu_1 + C \sqrt{\frac{\ln(\sigma^{-1})}{N}}$};
 
     \draw[thick] (2,-.5) node[right] {$ $};
 
    \draw[thick,->] (-0.3,-0.5) -- (6.5,-0.5) node[right] {$r/N$};
    
    \draw[black,thick,<-] (3.8, 2.65) -- (4.5, 3) node[right,align=left]
    {\ \ \ площадь под \\ графиком $\geq 1 - \sigma$};
\end{tikzpicture}
\caption{График зависимости $P\left( r \right)=C_N^r \nu _1^r \left( {1-\nu_1} \right)^{N-r}$ при большом значении $N$}
\label{fig:fig4}
\end{figure}
\end{center}

По сути, этот рисунок отражает тот факт, что биномиальное распределение 
(биномиальная мера) с ростом числа исходов (людей) $N$ \textit{концентрируется} вокруг $\nu _1 $.

Вектор $\nu = [{\nu_1 ,\nu_2}]^\top$, в малой окрестности 
которого с большой вероятностью на больших временах находится реальный 
вектор распределения людей по web-страницам (в данном случае двум), 
естественно называть \textit{равновесием макросистемы}, описываемой 
блуждающими по web-графу людьми. \textit{Таким образом, мы пришли к еще одному пониманию вектора PageRank }$\nu $\textit{, как равновесию описанной выше макросистемы}. 

На соотношение~(\ref{eq9}) можно посмотреть и с другой стороны -- с точки зрения 
математической статистики. А именно, вспомним, 
что $\nu _1 $ нам не известно. В то время как реализацию случайной величины 
(статистики) $r \mathord{\left/ {\vphantom {r N}} \right. 
\kern-\nulldelimiterspace} N$ мы можем измерить. Соотношение~(\ref{eq9}) от такого 
взгляда на себя не перестанет быть верным. 

В общем случае для $n$ вершин можно провести аналогичные рассуждения: с 
заменой биномиального распределения \textit{мультиномиальным}. Чтобы получить аналог оценки~(\ref{eq9}), для 
\textit{концентрации мультиномиальной меры}, можно использовать 
векторное неравенство Хефдинга. Оказывается, \textit{с вероятностью не меньше }$1-\sigma $\textit{ имеет место следующее неравенство}:
\begin{equation}
\label{eq10}
\left\| {\frac{r}{N}-\nu } \right\|_2 \le C\sqrt {\frac{\ln \left( {\sigma 
^{-1}} \right)}{N}} ,
\end{equation}
\textit{где вектор }$r=\left( {r_1 ,...,r_n } \right)^\top$\textit{ описывает распределение людей по web-страницам в момент наблюдения }$T_0 $. Заметим, что константу $C$ здесь 
можно оценить сверху числом 4. Неравенства~(\ref{eq9}),~(\ref{eq10}) являются представителям 
класса \textit{неравенств концентрации меры}, играющего важную роль 
в современной теории вероятностей и ее приложениях.

Если запустить $N\sim \varepsilon ^{-2}\ln \left( {\sigma ^{-1}} \right)$ 
человек, дав каждому сделать $T_0 \sim \alpha ^{-1}\ln \left( {n 
\mathord{\left/ {\vphantom {n \varepsilon }} \right. 
\kern-\nulldelimiterspace} \varepsilon } \right)$ шагов (стоимость шага 
$\sim \log _2 n)$, то (вообще говоря, случайный) вектор $r \mathord{\left/ 
{\vphantom {r N}} \right. \kern-\nulldelimiterspace} N$, характеризующий 
распределение людей по web-страницам на шаге $T_0 $, с вероятностью не 
меньше $1-\sigma $ обладает таким свойством: $\left\| {r \mathord{\left/ 
{\vphantom {r N}} \right. \kern-\nulldelimiterspace} N-\nu } \right\|_2 \le 
\varepsilon $. Это следует из оценки~(\ref{eq10}). Таким образом, мы приходим к 
оценке сложности (параллельного варианта) метода MCMC:
\[
C\cdot \left( {n+\frac{\log _2 n\cdot \ln \left( {n \mathord{\left/ 
{\vphantom {n \varepsilon }} \right. \kern-\nulldelimiterspace} \varepsilon 
} \right)\ln \left( {\sigma ^{-1}} \right)}{\alpha \varepsilon ^2}} 
\right),
\]
которая с точностью до логарифмического множителя совпадает с оценкой, 
приведенной в табл.~1. Однако отличие данного подхода в том, 
что можно пустить блуждать людей параллельно. То есть можно распараллелить 
работу программы на $N$ процессорах. Разумеется, можно распараллелить и на 
любом меньшем числе процессоров, давая процессору следить сразу за 
несколькими людьми. Можно еще немного <<выиграть>>, если сначала запустить 
меньшее число людей, но допускать, что со временем люди случайно производят 
клонов, которые с момента рождения начинают жить независимой жизнью так, 
чтобы к моменту $T_0 $ наблюдалось, как и раньше, $N$ человек.

Пока мы не пояснили, в 
каком смысле и почему выбранный способ рассуждения оптимален. 
Для большей наглядности снова вернемся к 
случаю $n=2$ и заметим, что оценка $r \mathord{\left/ {\vphantom {r N}} 
\right. \kern-\nulldelimiterspace} N$ неизвестного параметра $\nu _1 $ может 
быть получена следующим образом:
\begin{equation}
\label{eq11}
r \mathord{\left/ {\vphantom {r N}} \right. \kern-\nulldelimiterspace} 
N=\arg \mathop {\max }\limits_{\nu \in \left[ {0,1} \right]} \;\nu ^r\cdot 
\left( {1-\nu } \right)^{N-r}.
\end{equation}

Действительно, максимум у функций $f_1 \left( \nu \right)=\nu ^r\cdot \left( 
{1-\nu } \right)^{N-r_1 }$ и
\[
f_2 \left( \nu \right)=\ln \left( {\nu ^r\cdot \left( {1-\nu } 
\right)^{N-r}} \right)=r\ln \nu +\left( {N-r} \right)\ln \left( {1-\nu } 
\right)
\]
достигается в одной и той же точке, поэтому рассмотрим задачу
$$f_2 \left( 
\nu \right)\to \mathop {\max }_{\nu \in \left[ {0,1} \right]}. $$
Из \textbf{\textit{принципа Ферма}} (максимум 
лежит либо на границе, либо среди точек, в 
которых производная обращается в ноль) получаем условие ${df_2 \left( \nu 
\right)} \mathord{\left/ {\vphantom {{df_2 \left( \nu \right)} {d\nu }}} 
\right. \kern-\nulldelimiterspace} {d\nu }=0$, которое в данном случае 
примет вид
\[
\frac{r}{\nu }-\frac{N-r}{1-\nu }=0,
\]
т.е.
\[
\frac{r \mathord{\left/ {\vphantom {r N}} \right. \kern-\nulldelimiterspace} 
N}{\nu }=\frac{1-r \mathord{\left/ {\vphantom {r N}} \right. 
\kern-\nulldelimiterspace} N}{1-\nu }.
\]
Отсюда и следует формула~(\ref{eq11}).

\textit{Функция правдоподобия}, стоящая в правой части~(\ref{eq11}), отражает вероятность конкретного\footnote{ 
Поскольку <<конкретного>>, то не нужно писать биномиальный коэффициент.} 
распределения людей по вершинам (web-страницам), для которого число людей в 
первой вершине равно $\nu $. То~есть оценка $r \mathord{\left/ {\vphantom 
{r N}} \right. \kern-\nulldelimiterspace} N$ может быть проинтерпретирована 
как \textit{оценка максимального правдоподобия}. Указанный выше 
способ построения оценок является общим способом получения хороших оценок 
неизвестных параметров. 
А именно, пусть 
есть схема эксперимента (параметрическая модель), согласно которой можно 
посчитать вероятность того, что мы будем наблюдать то, что наблюдаем. Эта 
вероятность (плотность вероятности) зависит от неизвестного вектора 
параметров $\theta $. Будем максимизировать эту вероятность по параметрам 
этого вектора при заданных значениях наблюдений, т.е. при заданной выборке. Тогда получим 
зависимость неизвестных параметров от выборки $X$. Именно эту зависимость 
$\hat {\theta }^N\left( X \right)$ в общем случае и называют \textit{оценкой максимального правдоподобия} вектора 
неизвестных параметров. Оказывается (\textbf{\textit{теорема Фишера}}), \textit{что в случае выполнения довольно общих условий регулярности используемой параметрической модели такая оценка} 
\textit{является асимптотически оптимальной} (\textbf{\textit{теория ле Кама}}). Последнее 
можно понимать так, что для каждой отдельной компоненты $k$ вектора $\theta 
$ можно написать неравенство, аналогичное~(\ref{eq9}), справедливое с вероятностью не 
меньшей $1-\sigma $:
\[
\left| {\hat {\theta }_k^N \left( X \right)-\theta _k } \right|\le C_{k,N} 
\left( \theta \right)\sqrt {\frac{\ln \left( {\sigma ^{-1}} \right)}{N}} 
\]
c $C_{k,N} \left( \theta \right)\to C_k \left( \theta \right)$ при $N\to 
\infty $, где $N$ -- объем выборки, т.е. число наблюдений. При этом $C_k 
\left( \theta \right)>0$ являются равномерно по $\theta $ и $k$ наименее 
возможными. Имеется в виду, что если как-то по-другому оценивать $\theta $ 
(обозначим другой способ оценивания $\tilde {\theta }^N\left( X \right))$, 
то для любого $\theta $ и $k$ с вероятностью не меньшей $1-\sigma $
\[
\left| {\tilde {\theta }_k^N \left( X \right)-\theta _k } \right|>\tilde 
{C}_{k,N} \left( \theta \right)\sqrt {\frac{\ln \left( {\sigma ^{-1}} 
\right)}{N}} ,
\]
где $\liminf_{N\to \infty } \,\tilde {C}_{k,N} 
\left( \theta \right)\ge C_k \left( \theta \right)$.

Строго говоря, именно такая зависимость правой части от $\sigma $ имеет 
место не всегда. В общем случае при зафиксированном $N$ с уменьшением 
$\sigma $ правая часть может вести себя хуже, однако при не очень маленьких 
значениях\footnote{Пороговое значение $\sigma _0 $ удовлетворяет следующей 
оценке: $\ln N\ll \ln \sigma _0^{-1} \ll N$.} $\sigma $ написанное верно 
всегда. 

К сожалению, приведенная выше весьма вольная формулировка не отражает в 
полной мере всю мощь теоремы Фишера. Ведь в таком виде открытым остается 
вопрос об оценках совместного распределения отклонений сразу нескольких 
компонент оценки максимального правдоподобия от истинных значений. На самом 
деле можно сформулировать результат (делается это чуть более громоздко) об 
асимптотической оптимальности оценки максимального правдоподобия и в таких 
общих категориях. Мы не будем здесь этого делать.

Основным недостатком теоремы Фишера в приведенном нами 
варианте ле Кама заключается в том, что оценка оптимальна только в пределе 
$N\to \infty $. В реальности же нам дана только конечная выборка, пусть и большого объема. 
Сформулированный выше результат ничего не 
говорит, насколько хорошей будет такая оценка при конечном $N$. Также 
теорема не говорит о том, как получать точные оценки на $C_{k,N} \left( 
\theta \right)$. Теорема лишь предлагает эффективный способ расчета $C_k 
\left( \theta \right)$.\footnote{$C_k \left( \theta \right)$ -- не 
универсальны и зависят от использующейся параметрической модели.} Скажем, 
приведенная нами ранее оценка~(\ref{eq9}), хотя по форме и выглядит так, как нужно, 
но все же она далеко не оптимальна. В~частности, в оптимальный вариант 
оценки~(\ref{eq9}) нужно прописывать под корнем в числителе еще множитель $\nu _1 
\cdot \left( {1-\nu _1 } \right)$, оцененный нами сверху в~(\ref{eq9}) константой $1 
\mathord{\left/ {\vphantom {1 4}} \right. \kern-\nulldelimiterspace} 4. $ Мы 
сделали такую замену в~(\ref{eq9}), чтобы в правую часть неравенства не входил неизвестный параметр~$\nu_1 $. При значениях $\nu _1 $ близких к нулю или 
единице, такая <<замена>> оказывается слишком грубой. Мы привели простой 
пример, в котором мы смогли проконтролировать грубость своих рассуждений. В~общем случае, к сожалению, это не так просто сделать (если говорить о 
потерях в мультипликативных константах типа $C)$. Поэтому если мы хотим 
использовать оценки типа~(\ref{eq9}),~(\ref{eq10}), то борьба за <<оптимальные константы>> 
сводится не только к выбору оптимального способа оценивания неизвестных 
параметров, но и к способу оценивания концентрации вероятностного 
распределения выбранной оценки вокруг истинного значения. За последние 
несколько лет теория ле Кама была существенно пересмотрена как раз в 
контексте отмеченных выше проблем. 

Другая проблема -- это зависимость $C_k $ от неизвестного параметра~$\theta $. Грубо, но практически эффективно, проблема решается подстановкой 
$C_k \left( {\hat {\theta }^N\left( X \right)} \right)$. Более точно надо 
писать неравенство концентрации для $C_k \left( {\hat {\theta }^N\left( X 
\right)} \right)$, исходя из неравенства на $\hat {\theta }^N\left( X 
\right)$. Казалось бы, что возникает <<порочный круг>>, и это приводит к 
иерархии  <<зацепляющихся неравенств>>. Однако если выполнить описанное выше 
огрубление\footnote{ Речь идет об оценке $\nu _1 \cdot \left( {1-\nu _1 } 
\right)\le 1 \mathord{\left/ {\vphantom {1 4}} \right. 
\kern-\nulldelimiterspace} 4$.} для того, чтобы обрезать эту цепочку не 
сразу, а на каком-то далеком вложенном неравенстве, то чем оно дальше, тем к 
меньшей грубости это в итоге приведет. Детали мы также вынуждены здесь 
опустить.

\begin{example}
    Оценить, сколько надо опросить человек на выходе с избирательных 
участков большого города, чтобы с точностью 5{\%} (0.05) и с вероятностью 
0.99 оценить победителя во втором туре выборов (два кандидата, графы против 
всех нет). Для решения этой задачи рекомендуется воспользоваться 
неравенством~(\ref{eq9}) с явно выписанными константами:
\[
\left| {\frac{r}{N}-\nu _1 } \right|\le \frac{1}{2}\sqrt {\frac{\ln \left( 
{2 \mathord{\left/ {\vphantom {2 \sigma }} \right. 
\kern-\nulldelimiterspace} \sigma } \right)}{N}}.\text{ \EndEx}
\]
\vskip5mm
\end{example}

\newpage
\subsection{Модель Бакли–-Остгуса и степенные законы}

С точки зрения практических нужд (например, для разработки алгоритмов борьбы со спамом), модели web-графа оказываются зачастую более востребованными, чем сами вебграфы ввиду того, что вебграфы имеют слишком большие для анализа размеры. Кроме того, с помощью моделей можно исследовать различные закономерности, присущие web-графу.
Так,
Интернету и многим социальным сетям присущи определенные хорошо изученные
закономерности: наличие гигантской компоненты, правило пяти рукопожатий,
степенной закон для распределения степеней вершин, специальные свойства
кластерных коэффициентов и т.д. Хотелось бы
предложить такую модель формирования этих сетей, которая
бы объясняла все эти закономерности. Построив такую модель, с помощью
фундаментальной науки мы можем открывать
новые свойства, присущие изучаемой сети, исследуя лишь свойства выбранной модели. Такие
исследования позволяют в дальнейшем использовать полученные результаты при
разработке алгоритмов для реальных сетей, в том числе Интернета.

Одной из лучших на текущий момент моделей web-графа считается модель
Бакли--Остгуса~\cite{Raigor}. Именно ее мы и будем
рассматривать. Новым свойством, которое мы хотим \gav{проверить}, будет степенной
закон распределения компонент вектора PageRank, посчитанного для графа,
сгенерированного по этой модели. Далее мы описываем модель Бакли--Остгуса и
выводим\footnote{ Не строго, а в так называемом \textit{термодинамическом пределе}.} степенной
закон распределения степеней вершин графа, построенного по этой модели.
Именно этот закон приводит к степенному закону распределения компонент
вектора PageRank.

Итак, рассмотрим следующую модель роста сети.

\textit{База индукции.} Сначала имеется всего одна вершина, которая ссылается сама на себя (вершина
с петлей).

\textit{Шаг индукции.} Предположим, что уже имеется некоторый ориентированный граф. Добавим в граф новую вершину. С вероятностью $\beta >0$ из этой вершины
проведем ребро равновероятно в одну из существующих вершин, а с
вероятностью $1-\beta $ из этой вершины проводится ребро в одну из
существующих вершин не равновероятно, а с вероятностями, пропорциональными
входящим степеням вершин\footnote{ Выходящая степень всех вершин одинакова и
равна 1.} -- \textit{правило предпочтительного присоединения} (от англ. preferential attachment).

Другими словами, если уже построен граф из $n-1$ вершины, то~новая $n$-я вершина
сошлется на вершину $i=1,...,n-1$ с вероятностью
\[
\frac{\text{Indeg} _{n-1} \left( i \right)+a}{\left( {n-1} \right)\left(
{a+1} \right)},
\]
где $\text{Indeg}_{n-1} \left( i \right)$ -- входящая степень вершины $i$ в
графе, построенном на шаге $n-1$. Параметры $\beta $ и $a$ связаны следующим
образом:
\[
a=\frac{\beta }{1-\beta }.
\]
При $a=1$ получается известная модель Боллобаша--Риордана. Интернету наилучшим образом
соответствует значение $a=0.277$~\cite{Raigor}.

Далее вводится число $m$ -- среднее число web-страниц на одном сайте, и
каждая группа web-страниц с номерами $km+1,...,\left( {k+1} \right)m$
объединяется в один сайт. При этом все ссылки, имеющиеся между
web-страницами, наследуются содержащими их сайтами -- получается, что с
одного сайта на другой\footnote{ Впрочем, сайты могут совпадать -- внутри
одного сайта web-страницы также могут друг на друга сослаться.} может быть
несколько ссылок. Пусть, скажем, получилось, что для заданной пары сайтов
таких (одинаковых) ссылок оказалось $l\le m$, тогда мы превращаем их в одну
ссылку, но с весом (вероятностью перехода) $l \mathord{\left/ {\vphantom {l
m}} \right. \kern-\nulldelimiterspace} m$. Именно для так построенного
взвешенного ориентированного графа можно изучать закон
распределения компонент вектора PageRank.

К сожалению, строго доказать, что имеет место степенной закон распределения
компонент вектора PageRank в этой модели, насколько нам известно, пока
никому не удалось. Имеется только один специальный результат на эту тему,
касающийся модели, близкой к модели Боллобаша--Риордана. Однако, чтобы у читателей появилась некоторая
интуиция, почему такой закон может иметь место в данном случае, мы приведем
далее некоторые аргументы.

Сначала установим степенной
закон распределения входящих вершин в модели Бакли--Остгуса. При этом
ограничимся случаем $m=1$. Обозначим через $X_k \left( t \right)$ число
вершин с входящей степенью $k$ в момент времени $t$, т.е. когда в графе
имеется всего $t$ вершин. Заметим, что по определению
\[
t=\sum\limits_{k\ge 0} {X_k \left( t \right)} =\sum\limits_{k\ge 1} {kX_k
\left( t \right)} =\sum\limits_{k\ge 0} {kX_k \left( t \right)} .
\]
Поэтому для $k\ge 1$ вероятность того, что $X_k \left( t \right)$
увеличится на единицу при переходе на следующий шаг $t\to t+1$ по формуле
полной вероятности равна
\[
\beta \frac{X_{k-1} \left( t \right)}{t}+\left( {1-\beta }
\right)\frac{\left( {k-1} \right)X_{k-1} \left( t \right)}{t}.
\]
Аналогично, для $k\ge 1$ вероятность того, что $X_k \left( t \right)$
уменьшится на единицу при переходе на следующий шаг $t\to t+1$
\[
\beta \frac{X_k \left( t \right)}{t}+\left( {1-\beta } \right)\frac{kX_k
\left( t \right)}{t}.
\]
Таким образом, <<ожидаемое>> приращение $\Delta X_k( t) = X_k \left( {t+1} \right)-X_k \left( t \right)$ за
$\Delta t=\left( {t+1} \right)-t=1$ будет\footnote{ Корректная запись: 
\begin{equation*}
\begin{split}
&\EE_{X_{k+1}(t)} \left({\left. {\frac{\Delta X_k \left( t \right)}{\Delta t}} \,\right|\,X_0 \left(t\right),...,X_k \left( t \right)} \right) = \\
&\beta \frac{{X_{k-1} \left(t \right) - X_k \left( t \right)}}{t} + \left( {1-\beta}\right) \frac{\left( {k-1} \right)X_{k-1} \left( t \right)-kX_k \left( t
\right)}{t}.
\end{split}
\end{equation*}
Беря от обеих частей математическое ожидание 
$\mathbb{E}_{X_0 \left( t \right), \dots, X_k \left( t \right)} \left( \cdot \right)$, получим
$$\EE \left({\frac{\Delta X_k(t)}{\Delta t}}\right) = 
\beta \frac{\EE\left( {X_{k-1} \left( t \right)} \right) - \EE\left( {X_k \left(t \right)} \right)}{t} + 
(1-\beta) \frac{\left( {k-1}\right)\EE\left( {X_{k-1} \left( t \right)} \right)-k\EE\left({X_k(t)} \right)}{t}.
$$}
\begin{equation}
\label{eq12}
\frac{\Delta X_k \left( t \right)}{\Delta t} = \beta \frac{{X_{k-1}(t) - X_k(t)}}{t} + ({1-\beta}) \frac{\left( {k-1} \right)X_{k-1} \left( t \right)-kX_k \left( t
\right)}{t}.
\end{equation}
Для $X_0 \left( t \right)$ уравнение, аналогичное~(\ref{eq12}), будет иметь вид
\begin{equation}
\label{eq21}
\frac{\Delta X_0 \left( t \right)}{\Delta t}=1-\beta \frac{X_0 \left( t
\right)}{t}.
\end{equation}
К сожалению, соотношения~(\ref{eq12}),~(\ref{eq21}) -- не есть точные уравнения, описывающие
то, как меняется $X_k \left( t \right)$, хотя бы потому, что изменение $X_k
\left( t \right)$ происходит случайно. Динамика же~(\ref{eq12}),~(\ref{eq21}) полностью
детерминированная. Однако для больших значений $t$, когда наблюдается
концентрация случайных величин $X_k \left( t \right)$ вокруг своих
математических ожиданий, реальная динамика поведения $X_k
\left( t \right)$ и динамика поведения средних значений $X_k \left( t
\right)$ становятся близкими\footnote{ Последняя динамика уже является
детерминированной динамикой.} -- вариация на тему \textit{теоремы Куртца}. Таким образом, на систему~(\ref{eq12}),~(\ref{eq21}) можно
смотреть, как на динамику средних значений, вокруг которых плотно
сконцентрированы реальные значения. Под плотной концентрацией имеется в
виду, что разброс значений величины контролируется квадратным корнем из ее
среднего значения.

Будем искать решение системы~(\ref{eq12}),~(\ref{eq21}) на больших временах ($t\to \infty )$
в виде $X_k \left( t \right)\sim c_k \cdot t$ (иногда такого вида режимы
называют \textit{промежуточными асимптотиками}). Подставляя это выражение в формулы
(\ref{eq12}),~(\ref{eq21}), получим
\[
c_0 =\frac{1}{1+\beta },
\quad
\frac{c_k }{c_{k-1} }=1-\frac{2-\beta }{1+\beta +k\cdot \left( {1-\beta }
\right)}\simeq 1-\left( {\frac{2-\beta }{1-\beta }} \right)\frac{1}{k}.
\]
Откуда получаем следующий \textbf{\textit{степенной закон}}:
\begin{equation}
\label{eq31}
c_k \sim k^{-\frac{2-\beta }{1-\beta }}=k^{-2-a}.
\end{equation}
Заметим, что если построить на основе~(\ref{eq31}) \textit{ранговый закон} распределения вершин по их
входящим степеням, т.е. отранжировать вершины по входящей степени, начиная с
вершины с самой высокой входящей степенью, то также получим степенной закон
\begin{equation}
\label{eq41}
\mbox{in}\deg (r) \sim r^{-1-\beta }.
\end{equation}
Действительно, обозначив для краткости $\mbox{in}\deg (r) $ через $x$,
получим, что нам нужно найти зависимость $x\left( r \right)$, если из
формулы~(\ref{eq31}) известно, что
\[
\frac{dr\left( x \right)}{dx}\sim -x^{-\frac{2-\beta }{1-\beta
}}=-x^{-1-\frac{1}{1-\beta }},
\]
где зависимость $r\left( x \right)$ получается из зависимости $x\left( r
\right)$ как решение уравнения $x\left( r \right)=x$. Остается только
подставить сюда и проверить приведенное соотношение~(\ref{eq41}).

Перейдем теперь к пояснению того, почему может иметь место степенной закон
распределения компонент вектора PageRank. Для этого предположим, что матрица
$P$ имеет вид
\[
P\sim \left[ {\begin{array}{l}
 1^{-\lambda }\;2^{-\lambda }\;3^{-\lambda }\;4^{-\lambda }\;5^{-\lambda
}\;... \\
 1^{-\lambda }\;2^{-\lambda }\;3^{-\lambda }\;4^{-\lambda }\;5^{-\lambda
}\;... \\
 1^{-\lambda }\;2^{-\lambda }\;3^{-\lambda }\;4^{-\lambda }\;5^{-\lambda
}\;... \\
 ......................................\\
 \end{array}} \right].
\]
Такой вид матрицы означает, что для каждого сайта имеет место точный (т.е. не в
вероятностном смысле) степенной закон распределения выходящих степеней вершин,
имеющий одинаковый вид для всех сайтов. Конечно, это намного более сильное
предположение, чем то, что мы выше получили для модели Бакли--Остгуса. Тогда
для выписанной матрицы $P$ выполняются условия единственности вектора
PageRank $\nu $, определяющегося формулой (2). Более того, этот вектор, неизбежно должен
совпадать со строчкой (не важно какой именно -- они одинаковые) матрицы $P$,
т.е.
\[
\nu _k \sim k^{-\lambda }.
\]
Но это и означает, что имеет место степенной закон распределения компонент
вектора PageRank. Разумеется, проведенные рассуждения ни в какой степени
нельзя считать доказательством. Тем не менее мы надеемся, что некоторую
интуицию эти рассуждения читателям все-таки смогли дать.

\gav{В заключении этого} \eduard{раздела} заметим, что можно получить закон~(\ref{eq31}) в более точных
вероятностных категориях. Хотя это можно сделать вполне элементарными
комбинаторными средствами,
тем не менее соответствующие выкладки оказываются достаточно громоздкие,
поэтому мы не приводим их здесь.

\subsection[Элементы теории макросистем\vspace{-1mm}]{Элементы теории макросистем}
В контексте написанного выше хотелось бы отметить, что подобно системе~(\ref{eq12}),
(\ref{eq21}) можно записать \textit{динамику средних} (говорят также
\textit{квазисредних}) и для
макросистемы блуждающих по web-графу людей. А именно,
предположим, что людей достаточно много и что каждый человек в любом
промежутке времени $\left[ {t,t+\Delta t} \right)$ независимо от остальных с
вероятностью, \gav{равной $\Delta t + o(\Delta t)$}, совершает переход по одной из
случайно выбранных согласно матрице $P$ ссылок. Обозначив через $c_k
\left( t \right)$ долю людей, находящихся в момент времени $t$ на
web-странице с номером $k$, получим следующую систему:
\begin{equation}
\label{eq51}
\frac{\Delta c^T\left( t \right)}{\Delta t}=c^T\left( t \right)P-c^T\left( t
\right).
\end{equation}
Формула~(\ref{eq51}) подтверждает вывод о том, что вектор PageRank $\nu $,
удовлетворяющий системе \eqref{eq2}, действительно можно понимать как равновесие
макросистемы. В самом деле, если существует предел $\nu =$\linebreak $=\mathop {\lim
}_{t\to \infty } c\left( t \right)$, то из~(\ref{eq51}) следует, что этот
предел должен удовлетворять \eqref{eq2}. Здесь предел всегда существует, но может зависеть от начального условия. Для
того чтобы предел не зависел от начального условия и был единственным, нужно
сделать предположение о наличии в графе <<Красной площади>>.

Имеется глубокая связь между приведенной выше схемой рассуждений и общими
моделями макросистем, которые, с точки зрения математики, можно понимать как
разнообразные модели стохастической химической кинетики. В частности, система \eqref{eq51} соответствует
\textit{закону действующих масс Гульдберга--Вааге}. При этом важно подчеркнуть, что возможность
осуществлять описанный выше канонический скейлинг, по сути заключающийся в замене
концентраций их средними значениями, обоснована теоремой
Куртца (в том
числе и для нелинейных систем, появляющихся, когда имеются не только
унарные реакции, как в примере с PageRank'ом) только на конечных отрезках
времени. Для бесконечного отрезка\footnote{ А именно эта ситуация нам
наиболее интересна, поскольку, чтобы выйти на равновесие, как правило,
необходимо перейти к пределу $t\to \infty $.} требуются дополнительные
оговорки, например, выполнение условия детального баланса и его обобщений. 
\gav{В}~общем
случае использованные нами предельные переходы (по времени и числу агентов)
не перестановочны, см., например,~\cite[глава 6]{StochAn2016}!

Рассуждения предыдущих пунктов
соответствуют следующему порядку предельных переходов: сначала $t\to \infty
$ (выходим на инвариантную меру/стационарное распределение), потом $N\to
\infty$ (концентрируемся вокруг наиболее вероятного макросостояния
инвариантной меры), а рассуждения этого пункта: сначала $N\to \infty $
(переходим на описание макросистемы на языке концентраций, устраняя
случайность с помощью законов больших чисел), потом $t\to \infty $
(исследуем аттрактор полученной при скейлинге детерминированной, т.е. не
стохастической, системы). Для примеров макросистем, рассмотренных в этом и предыдущем приложениях,   получается один и тот же результат. Более того,
если посмотреть на то, как именно концентрируется инвариантная мера для, то получим, что \textit{концентрация экспоненциальная}
$$
P\left( {r=k} \right)=\frac{N!}{k_1 !\cdot ...\cdot k_n !}\nu _1^{k_1 }
\cdot ...\cdot \nu _n^{k_n } \simeq \exp \left( {-N\cdot KL\left( {k
\mathord{\left/ {\vphantom {k N}} \right. \kern-\nulldelimiterspace} N,\nu }
\right)} \right),
$$
где \textit{функция действия} (\textit{функция Санова}) $KL\!\left( {x,y} \right)=\!-\!\!\sum\limits_{k=1}^n {x_k \ln \left( {{x_k }
\mathord{\left/ {\vphantom {{x_k } {y_k }}} \right.
\kern-\nulldelimiterspace} {y_k }} \right)} $. В других контекстах функцию
$KL$ чаще называют \textit{дивергенцией Куль\-бака--Лейблера }или просто\textit{ энтропией}\footnote{ Отсюда, по-видимому, и пошло, что
<<равновесие следует искать из \textit{принципа максимума энтропии}>>~\cite{Baimurzina2015}.}. При этом функция
$KL\!\left({c\!\left(t\right)\!,\!\nu}\right)$, как функция $t$, монотонно
убывает с ростом $t$ на траекториях системы \eqref{eq51}, т.е. является \textit{функцией Ляпунова}.
Оказывается, этот факт\footnote{ Известный из курса
термодинамики/статистической физики, как \textit{H-теорема Больцмана}.} имеет место и при намного более
общих условиях~\cite{StochAn2016}. Точнее говоря,
сам факт о том, что \textit{функция, характеризующая экспоненциальную концентрацию инвариантной меры, будет функцией Ляпунова динамической системы, полученной в результате скейлинга из марковского процесса, породившего исследуемую инвариантную меру}, имеет место всегда, а вот то, что именно такая
$KL$-функция будет возникать, соответствует макросистемам, удовлетворяющим
обобщенному условию детального баланса (условию
Штюкельберга--Батищевой--Пирогова),\linebreak и~только таким макросистемам
\cite{StochAn2016}.

    \section{Задача о разборчивой невесте}\label{secretary_problem}

Рассматриваемая далее задача о разборчивой невесте также изложена с точки зрения оптимальной остановки марковских процессов в~\cite{Shiryaev2}, где можно посмотреть детали. Популярное изложение имеется в брошюре~\cite{Gusein-Zade}.

На примере решения этой задачи  хотелось бы продемонстрировать элементы теории управляемых марковских процессов и основное уравнение данной теории -- \textit{\textbf{уравнение Вальда--Белл\-мана}}.

\subsection{Формулировка задачи}


Примерно $55$ лет тому назад Мартин Гарднер придумал такую задачу:
<<В некотором царстве, в некотором государстве пришло время принцессе выбирать себе жениха. В назначенный день явились $1000$ царевичей. Их построили в очередь в случайном порядке и стали по одному приглашать к принцессе. Про любых двух претендентов принцесса, познакомившись с ними, может сказать, какой из них лучше. Познакомившись с претендентом, принцесса может либо принять предложение (и тогда выбор сделан навсегда), либо отвергнуть его (и тогда претендент потерян: царевичи гордые и не возвращаются). Какой стратегии должна придерживаться принцесса, чтобы с наибольшей вероятностью выбрать лучшего?>>

\subsection{Оптимальная стратегия}

Везде в дальнейшем будем считать, что все $1000!$ способов расстановки претендентов равновероятны (это означает отсутствие изначально у невесты какой-либо информации о том, где может быть наилучший жених). Под оптимальной стратегией невесты понимается такая стратегия, которая максимизирует вероятность выбора наилучшего жениха. Другими словами, максимизирует число расстановок претендентов, на которых невеста выбирает наилучшего жениха.

\textbf{Оптимальная стратегия невесты}: пропустить первых $1/e$ претендентов (приблизительно треть) и затем выбрать первого наилучшего претендента (если такой появится: среди пропущенных претендентов мог быть самый лучший – в таком случае, никого лучше невеста уже не встретит, т.е. такая стратегия не позволит ей сделать выбор). Такая стратегия позволяет невесте выбрать наилучшего жениха с вероятностью $1/e$.


\subsection{Управляемые марковские процессы}\label{MDP}
Рассмотрим некоторую управляемую марковскую систему, которая со временем $(t = 0, 1, 2, \dots)$ претерпевает случайные изменения. Будем считать, что все время система может находиться лишь в конечном числе возможных состояний $S$. На каждом шаге система находится в одном из этих состояний, и на каждом шаге мы должны управлять системой, выбирая свою стратегию из конечного множества стратегий $A$.
Исходя из того, в каком состоянии $s \in S$ находится система в текущий момент $t$ и какое \textit{действие} было выбрано $a \in A$, можно определить, с какой вероятностью система окажется в следующий момент $t+1$ в состоянии
$s' \in S$. Обозначим эту вероятность $p(s,a;s')$. По определению $$ \sum_{s' \in S} p(s,a;s') = 1.$$
В каждый момент времени мы получаем \textit{вознаграждение} $r(s,a)$ ($r$ -- reward), зависящее от состояния системы $s$ и выбранного действия $a$  $(\mathbb{E} \, r(s,a) = R(s,a))$.

Цель (управления) – получить максимальное итоговое вознаграждение:
$$ V^*(s) = \max\limits_{a(\cdot)} \, \mathbb{E} \sum\limits_{t=0}^{\infty} \gamma^t r(s_t, a(s_t)), \ s_0=s.$$
В данной записи $\mathbb{E}(\cdot)$  обозначает математическое ожидание (среднее значение), максимум по $a(\cdot)$ означает, что мы должны оптимизировать по конечному множеству (мощности $|A|^{|S|}$ ) всевозможных функций $a : S \to A$, параметр $\gamma \in (0,1]$ обычно называют \textit{коэффициентом дисконтирования} (этот параметр отвечает, грубо говоря, за то, как обесцениваются деньги – считаем, что вознаграждение денежное), $s_t$ – состояние системы в момент времени $t$. Функция $V^*(s)$ – \textit{функция цены} (ожидаемый выигрыш при использовании оптимальной стратегии, если система в начальный момент находилась в состоянии $s$).

\subsection{Принцип динамического программирования} 

Функция $V^*(s)$ удовлетворяет \textbf{уравнению Вальда–-Беллмана}: $$ V^*(s) = \max\limits_{a \in A} \left( R(s,a) + \gamma \sum\limits_{s' \in S} p(s,a;s') V^*(s') \right),$$
а оптимальная стратегия может быть найдена из условия
\begin{equation}\label{eq:opt_startegy}
 a(s) = \argmax\limits_{a \in A} \left( R(s,a) + \gamma \sum\limits_{s' \in S} p(s,a;s') V^*(s') \right).
\end{equation}
Идея доказательства (здесь $s=s_0, s_1=s'$):
\begin{gather*}
\begin{split}
V^*(s) &= \max\limits_{a(\cdot)} \, \mathbb{E} \sum\limits_{t=0}^{\infty} \gamma^t r(s_t, a(s_t))=
\\&=
\max\limits_{a(\cdot)} \, \mathbb{E} \left( r(s_0, a(s_0)) + \gamma \sum\limits_{t=0}^{\infty} \gamma^t r(s_{t+1}, a(s_{t+1})) \right)= \\&=
\max\limits_{a \in A} \left( R(s_0, a) + \gamma \mathbb{E}_{s_1} V^*(s_1) \right)= 
\\&=
\max\limits_{a \in A} \left( R(s, a) + \gamma \sum\limits_{s_1} p(s,a;s_1) V^*(s_1) \right).
\end{split}
\end{gather*}
В общем случае искать решение \eqref{eq:opt_startegy} сложно! Но для задачи о разборчивой невесте, все можно сделать явно!

\subsection[Поиск оптимальной стратегии невесты\vspace{-1mm}]{Поиск оптимальной стратегии невесты}

Введем управляемую марковскую систему, соответствующую задаче. Положим $\gamma = 1, \ S = \{1,2, \dotsc, N, End \}, \ N = 1000$. Если система оказалась в состоянии $s$, то это соответствует тому, что претендент, вошедший $s$-м по порядку, оказался лучше всех предыдущих. Определим множество стратегий (действий невесты). Довольно очевидно, что возможны всего два действия $A = \{\text{не выбрать, выбрать}\}$. <<Фиктивное>> состояние $End$ наступает либо на следующем шаге, после того как невеста пропустила (не выбрав) лучшего жениха, либо на следующем шаге, после шага, когда невеста сделала свой выбор. Исходя из такого описания, можно посчитать соответствующие \gav{функции вознаграждения и вероятности переходов}: 
\begin{gather*}
R(s=End, \, a = \text{выбрать}) = 0, \ R(s, \, a = \text{не выбрать}) = 0; \\ R(s, \, a = \text{выбрать}) = s / N, \ s=1, \dots, N,
\end{gather*}

\noindent поскольку
$$ 
r(s, a = \text{выбрать}) = \begin{cases} 1, & \text{с вероятностью} \ s/N, \\ 
0, & \text{с вероятностью} \ 1-s/N. \end{cases}
$$
С переходными вероятностями немного сложнее:
\begin{gather*}
p(s, \, a = \text{выбрать}; \, s') = 0, \ s' \neq End, \\
p(s, \, a = \text{выбрать}; \, s' = End) = 1; \\
p(s=End, \, a, \, s' = End) = 1; \\
p(s, \, a = \text{не выбрать}; \, s'=End) = \\
{\displaystyle =\mathbb{P}}\left( 
\begin{split} s\text{-й претендент \ag{лучше} всех, если \ \ } \\ 
\text{известно только, что он лучше предыдущих} \end{split}
\right) = \frac{s}{N}; \\
p(s, \, a = \text{не выбрать}; \, s') = \\
{\displaystyle =\mathbb{P}} \left( 
\begin{split} s'\text{-й претендент -- первый кто, лучше $s$-го,} \\ 
\text{если известно, что $s$-й лучше предыдущих} \end{split}
\right) = \\
=\frac{{\displaystyle \mathbb{P}} \left(
\begin{split} s\text{-й претендент лучше предыдущих и \ \ } \\
s'\text{-й претендент -- первый кто, лучше $s$-го} \end{split}
\right)}{{\displaystyle \mathbb{P}}(s\text{-й претендент лучше предыдущих})} = \frac{s}{s' (s'-1)},
\end{gather*}

\noindent поскольку
\begin{gather*}
{\displaystyle \mathbb{P}} (s\text{-й претендент лучше предыдущих}) = \frac{(s-1)!}{s!} = \frac{1}{s}, \\
{\displaystyle \mathbb{P}} \left(
\begin{split} s\text{-й претендент лучше предыдущих и} \\
s'\text{-й -- первый кто, лучше $s$-го \ \ \ \  } \end{split}
\right) = \frac{(s'-2)!}{s'!} = \frac{1}{s'(s'-1)}.
\end{gather*}
\newpage
\noindent Граф, отвечающий марковской цепи в задаче о разборчивой невесте.
\begin{figure}[!h] 
\centering
    \begin{tikzpicture}[scale=0.2]
    \tikzstyle{every node}+=[inner sep=0pt]
    \draw [black] (15,-26.4) circle (3);
    \draw (15,-26.4) node {$1$};
    \draw [black] (25.9,-26.4) circle (3);
    \draw (25.9,-26.4) node {$2$};
    \draw [black] (38.9,-26.4) circle (3);
    \draw (38.9,-26.4) node {$S$};
    \draw [black] (51.1,-26.4) circle (3);
    \draw (51.1,-26.4) node {$S'$};
    \draw [black] (63.1,-26.4) circle (3);
    \draw (63.1,-26.4) node {$N$};
    \draw [black] (38.9,-42.7) circle (3);
    \draw (38.9,-42.7) node {$End$};
    \draw [black] (18,-26.4) -- (22.9,-26.4);
    \fill [black] (22.9,-26.4) -- (22.1,-25.9) -- (22.1,-26.9);
    \draw [black] (17.424,-24.639) arc (121.309:58.691:18.331);
    \fill [black] (36.48,-24.64) -- (36.05,-23.8) -- (35.53,-24.65);
    \draw [black] (17.24,-24.407) arc (128.29474:51.70526:25.512);
    \fill [black] (48.86,-24.41) -- (48.54,-23.52) -- (47.92,-24.3);
    \draw [black] (16.917,-24.094) arc (137.41614:42.58386:30.061);
    \fill [black] (61.18,-24.09) -- (61.01,-23.17) -- (60.27,-23.84);
    \draw [black] (40.223,-45.38) arc (54:-234:2.25);
    \draw (38.9,-49.95) node [below] {$1$};
    \fill [black] (37.58,-45.38) -- (36.7,-45.73) -- (37.51,-46.32);
    \draw [black] (36.16,-41.478) arc (-115.20876:-133.37987:72.908);
    \fill [black] (36.16,-41.48) -- (35.65,-40.69) -- (35.22,-41.59);
    \draw [black] (36.604,-40.77) arc (-131.9319:-150.92015:45.374);
    \fill [black] (36.6,-40.77) -- (36.34,-39.86) -- (35.67,-40.61);
    \draw [black] (38.9,-29.4) -- (38.9,-39.7);
    \fill [black] (38.9,-39.7) -- (39.4,-38.9) -- (38.4,-38.9);
    \draw (39.4,-34.55) node [right] {$p=\frac{S}{N}$};
    \draw [black] (49.851,-29.127) arc (-26.69552:-46.93171:41.26);
    \fill [black] (41.16,-40.73) -- (42.09,-40.55) -- (41.41,-39.82);
    \draw [black] (60.87,-28.406) arc (-48.9409:-63.13434:94.04);
    \fill [black] (41.6,-41.39) -- (42.54,-41.47) -- (42.09,-40.58);
    \draw (32.4,-25.9) node [above] {$...$};
    \draw (45,-25.9) node [above] {$...$};
    \draw (57.1,-25.9) node [above] {$...$};
    \draw [black] (48.485,-27.867) arc (-64.42214:-115.57786:23.129);
    \fill [black] (48.49,-27.87) -- (47.55,-27.76) -- (47.98,-28.66);
    \draw [black] (60.577,-28.021) arc (-59.95797:-120.04203:32.113);
    \fill [black] (60.58,-28.02) -- (59.63,-27.99) -- (60.13,-28.85);
    \draw [black] (41.268,-24.564) arc (122.98673:57.01327:17.875);
    \fill [black] (60.73,-24.56) -- (60.33,-23.71) -- (59.79,-24.55);
    \draw [black] (35.996,-27.141) arc (-79.88175:-100.11825:20.468);
    \fill [black] (36,-27.14) -- (35.12,-26.79) -- (35.3,-27.77);
    \draw [black] (37.277,-23.896) arc (202.68317:-22.68317:8.37);
    \fill [black] (52.72,-23.9) -- (53.49,-23.35) -- (52.57,-22.97);
    \draw (45,-11.8) node [above] {$p=\frac{S}{S'(S'-1)}$};
    \draw [black] (60.151,-26.939) arc (-83.0576:-96.9424:25.238);
    \fill [black] (60.15,-26.94) -- (59.3,-26.54) -- (59.42,-27.53);
    \end{tikzpicture}
\text{a = не выбрать}
\end{figure}
\begin{figure}[!h]
\centering
    \begin{tikzpicture}[scale=0.2]
    \tikzstyle{every node}+=[inner sep=0pt]
    \draw [black] (16.6,-25.2) circle (3);
    \draw (16.6,-25.2) node {$1$};
    \draw [black] (26.4,-25.2) circle (3);
    \draw (26.4,-25.2) node {$2$};
    \draw [black] (36.8,-25.2) circle (3);
    \draw (36.8,-25.2) node {$S$};
    \draw [black] (47.4,-25.2) circle (3);
    \draw (47.4,-25.2) node {$S'$};
    \draw [black] (57.7,-25.2) circle (3);
    \draw (57.7,-25.2) node {$N$};
    \draw [black] (36.8,-41) circle (3);
    \draw (36.8,-41) node {$End$};
    \draw [black] (18.96,-27.05) -- (34.44,-39.15);
    \fill [black] (34.44,-39.15) -- (34.11,-38.26) -- (33.5,-39.05);
    \draw (25.69,-33.6) node [below] {$1$};
    \draw [black] (28.05,-27.71) -- (35.15,-38.49);
    \fill [black] (35.15,-38.49) -- (35.13,-37.55) -- (34.29,-38.1);
    \draw (32.21,-31.77) node [right] {$1$};
    \draw [black] (36.8,-28.2) -- (36.8,-38);
    \fill [black] (36.8,-38) -- (37.3,-37.2) -- (36.3,-37.2);
    \draw (37.3,-33.1) node [right] {$1$};
    \draw [black] (45.73,-27.69) -- (38.47,-38.51);
    \fill [black] (38.47,-38.51) -- (39.33,-38.12) -- (38.5,-37.57);
    \draw (42.71,-34.44) node [right] {$1$};
    \draw [black] (55.31,-27.01) -- (39.19,-39.19);
    \fill [black] (39.19,-39.19) -- (40.13,-39.11) -- (39.53,-38.31);
    \draw (48.25,-33.6) node [below] {$1$};
    \draw (31.6,-24.7) node [above] {$...$};
    \draw (42.1,-24.7) node [above] {$...$};
    \draw (52.55,-24.7) node [above] {$...$};
    \draw [black] (38.123,-43.68) arc (54:-234:2.25);
    \draw (36.8,-48.25) node [below] {$1$};
    \fill [black] (35.48,-43.68) -- (34.6,-44.03) -- (35.41,-44.62);
    \end{tikzpicture}
    \text{a = выбрать}
\end{figure}


\newpage

Теперь можно выписать уравнение Вальда–Беллмана
$$ V^*(s) = \max\limits_{a \in A} \left( R(s,a) + \gamma \sum\limits_{s' \in S} p(s,a;s') V^*(s') \right), $$
$$ V^*(s) = \max \left( \frac{s}{N}, \sum\limits_{s'=s+1}^{N} \frac{s}{s'(s'-1)} V^*(s') \right), \ s = 1, \dotsc, N-1; \ V^*(N) = 1. $$

Если максимум достигается на первом аргументе, то $a(s) = $ \ag{выбрать}, если на втором, то $a(s)=$ не выбрать. Оказывается, что данное уравнение можно явно разрешить.

Для этого определим $s^*(N)$ из уравнения
$$ \frac{1}{s^*} + \frac{1}{s^*+1} + \dots + \frac{1}{N-1} \leq 1 \leq \frac{1}{s^*-1} + \frac{1}{s^*} + \dots + \frac{1}{N-1}.$$

Можно показать, что $$s^*(N) \simeq \left[ \frac{N}{e} \right].$$

Введем
$$ V^* = \frac{s^*(N)-1}{N} \left( \frac{1}{s^*(N)-1} + \frac{1}{s^*(N)} + \dots + \frac{1}{N-1} \right) \simeq \frac{1}{e}. $$

Тогда
$$ V^*(s) = \begin{cases} V^*, 1 \leq s \leq s^*(N), \\ 
s/N, s \geq s^*(N). \end{cases} $$

    \section{Задача о многоруких бандитах}\label{bandit_problem}

В данном разделе будет рассмотрена другая популярная задача на управляемые марковские процессы -- задача о многоруких бандитах~\cite{Bubeck,Cesa-Bianchi,SuttonBarto,banditalgs}. В отличие от задачи о разборчивой невесте, явно решить здесь уравнение Вальда--Беллмана не получается. Тем не менее существуют различные достаточно эффективные способы приближенного решения данной задачи, которые хотелось бы продемонстрировать.

\subsection{Формулировка задачи} Имеется $n$ ручек (игровых автоматов). Дергая ручку $i = 1,\dotsc,n$, мы каждый раз (независимо ни от чего) с вероятностью $p_i$ получаем один рубль, а с вероятностью $1 - p_i$ ничего не получаем.  Разрешает\eg{с}я сделать $N \gg 1$ шагов, на каждом шаге можно дергать только одну ручку. Величины $p_i$ априорно не известны. Однако известно, что $p_i$ <<были приготовлены>> следующим образом: $p_i\in\mathrm{R}(0,1)$, $i = 1,\dotsc,n$. Требуется найти оптимальную (максимизирующую ожидаемый доход) стратегию выбора ручек на шагах.

\subsection{Уравнение Вальда--Беллмана}


Основной задачей данного раздела является сопоставление описанному процессу выбора ручек управляемого марковского процесса и получение соответствующего уравнения Вальда--Беллмана. 

Прежде всего, определим пространство состояний ($w$ -- win, $l$ -- lose):  $s = \left(w_1,l_1; \dotsc; w_n,l_n\right)$, где $\sum_{i=1}^n (w_i + l_i) = k \le N$, $k$ -- номер шага, $w_i$ -- сколько выигрышей было связано с ручкой $i$ к шагу $k$ (т.е. сколько рублей нам принесла $i$-я ручка к шагу $k$), 
$l_i$ -- сколько неудач было связано с ручкой $i$ к шагу $k$. Стратегией является выбор на каждом шаге одной из ручек $a(s)\in\left\{ 1,\dotsc,n\right\}$, исходя из истории выигрышей, имеющейся к данному шагу. 

В постановке задачи ничего явно не говорится относительно вероятностей переходов из одного состояния в другое:
$$\left(w_1,l_1; \dotsc; w_i,l_i; \dotsc; w_n,l_n\right) \to \left(w_1,l_1; \dotsc; w_i +1,l_i; \dotsc;  w_n,l_n\right),$$
$$\left(w_1,l_1; \dotsc; w_i,l_i; \dotsc; w_n,l_n\right) \to \left(w_1,l_1; \dotsc; w_i,l_i+1; \dotsc;  w_n,l_n\right).$$
Для того чтобы определить вероятности таких переходов (обозначим их соответственно $\rho^w_i(w_i,l_i)$ и $\rho^l_i(w_i,l_i) = 1 - \rho^w_i(w_i,l_i)$), заметим, что если априорно $p_i\in\mathrm{B}(w,l)$, т.е. плотность распределения $p_i$ имеет вид
\begin{equation*}
   \frac{(w+l+1)!}{w! l!}x^{w}(1-x)^{l},  x\in[0,1],
\end{equation*}
то (по формуле Байеса) апостериорное распределение $$p_i\,|\,(w_i,l_i) \in \mathrm{B}(w+w_i,l+l_i).$$ Заметим, что $\mathrm{B}\gav{(0,0)} = \mathrm{R}(0,1)$, поэтому $$\rho^w_i(w_i,l_i)\,|\,(w_i,l_i)\in \mathrm{B}(w_i,l_i).$$ В этой связи говорят, что бета-распределение является сопряженным к схеме испытаний Бернулли. 

Функция вознаграждения также определяется по введенным вероятностям: если была выбрана ручка $i$ (с историей $(w_i,l_i)$), то вознаграждение равно 1 рубль с вероятностью $\rho^w_i(w_i,l_i)$ и 0 рублей с вероятностью $\rho^l_i(w_i,l_i)$. 

Подчеркнем, что появление в постановке задачи априорного распределения на $p_i$ привело к тому, что вероятности $\rho^w_i(w_i,l_i)$, $\rho^l_i(w_i,l_i)$, в свою очередь, являются случайными величинами, и при выводе уравнения Вальда--Беллмана следует это учитывать, беря дополнительное (условное/апостериорное) усреднение.

Итак, уравнение Вальда--Беллмана в данном случае будет иметь следующий вид (по постановке задачи $\gamma = 1$):
$$V^*\left(w_1,l_1; \dotsc; w_n,l_n\right) = $$
$$=\max_{i=1,\dotsc,n}\EE_{\rho^w_i,\rho^l_i}\bigg(\rho^w_i(w_i,l_i)\cdot\left(1 + \gamma V^*(w_1,l_1; \dotsc; w_i +1,l_i; \dotsc;  w_n,l_n)\right)+ $$
$$+\rho^l_i(w_i,l_i) \gamma V^*(w_1,l_1; \dotsc; w_i,l_i+1; \dotsc;  w_n,l_n)\,\big|\,(w_i,l_i)\bigg)=$$
$$=\max_{i=1,\dotsc,n}\bigg(\frac{w_i+1}{w_i+l_i+2}\left( 1+\gamma V^*(w_1,l_1; \dotsc; w_i +1,l_i; \dotsc;  w_n,l_n)\right)+$$
$$+ \frac{l_i+1}{w_i+l_i+2}\gamma V^*(w_1,l_1; \dotsc; w_i,l_i+1; \dotsc;  w_n,l_n)\bigg).$$
Полагая 
\begin{center}
$V^*\left(w_1,l_1; \dotsc; w_n,l_n\right) = 0$ при $\sum_{i=1}^n (w_i + l_i) > N,$ 
\end{center}
можно попробовать получить решение этого уравнения. Однако для этого придется проделать экспоненциально много (по $N$) вычислений и использовать экспоненциально большие ресурсы памяти, что типично для динамического программирования~\cite{Bertsekas}. 

\subsection{Индексы Гиттинса} Возможным способом решения отмеченной выше проблемы является использование специфики задачи, связанное с допустимостью представления решения в виде индексной стратегии. 

Для простоты будем считать в этом и следующем разделах, что $\gamma < 1$ и $N \to \infty$. Рассмотрим сначала случай, когда всего две ручки, вероятность на одной из которых известна и равна $p$. Обозначим функцию выигрыша в этом случае $V^*(w,l;p)$. Уравнение Вальда--Беллмана для такой функции будет иметь вид 
\begin{eqnarray*}
V^*(w,l;p) &=& \max\bigg(\frac{p}{1-\gamma}, \frac{w+1}{w+l+2}\left[1+\gamma V^*(w+1,l;p)\right]+\\
&& + \ \frac{l+1}{w+l+2}\gamma V^*(w,l+1;p)\bigg).
\end{eqnarray*}
Заметим, что при $w+l \gg 1$ значение этой функции с хорошей точностью можно определять, как $V^*(w,l;p) = \left(1-\gamma\right)^{-1}\max\left(p,w/(w+l)\right)$. Заметим также, что неточность, допущенная в этой формуле, при каждом последующем уменьшении номера шага $k = w+l$ на 1 согласно уравнению Вальда--Беллмана будет уменьшаться в $\gamma$ раз, т.е. нет необходимости следовать условию $N \to \infty$, достаточно задать $V^*(w,l;p)$ лишь при больших значениях $N = w + l$, как указано выше. Таким образом, можно определить при каждом $p$ функцию $V^*(w,l;p)$ и сделать это намного эффективнее (с полиномиальной сложностью), чем в предыдущем пункте. 

Далее определим \textit{индекс Гиттинса} ручки с историей $(w,l)$, как $p_{\gamma}(w,l)$ из решения уравнения (относительно $p$):
$$V^*(w,l;p)=\frac{p}{1-\gamma}.$$ При $w+l \to \infty$ имеем $p_{\gamma}(w,l)\to \left(1-\gamma\right)^{-1}w/(w+l)$. Заметим, что индекс $p_{\gamma}(w,l)$ можно считать в целом известной функцией, значения которой хорошо протабулированы. Используя эту функцию (индекс) и принцип Вальда--Беллмана, несложно заметить, что максимум в правой части уравнения Вальда--Беллмана из предыдущего пункта будет достигаться на той ручке, у которой наибольший индекс $p_{\gamma}(w_i,l_i)$~\cite{Gittins}. Таким образом, решение задачи о многоруких бандитах можно получить в предположении наличия таблицы значений $p_{\gamma}(w,l)$.  

\subsection{$Q$-обучение} В предыдущих разделах проблема отсутствия информации об управляемом марковском процессе, отвечающем задаче о многоруких бандитах, была решена в итоге с помощью индексов Гиттинса за счет байесовского подхода -- задания априорного распределения. В общем случае возможность использовать байесовский подход в исследовании управляемых марковских процессов предполагает наличие вероятностной модели для переходных вероятностей и вознаграждений, зависящей от неизвестных параметров. Сам факт наличия такой параметрической модели далеко не всегда имеет место. Но еще более специальное предположение, которое мы существенным образом использовали, -- это некоторая симметричность постановки задачи, позволившая искать решение в виде индексной стратегии. Естественно, возникает вопрос: что же делать в общем случае, когда ничего не известно и все, что можно делать, это только наблюдать за процессом и обучаться?

Для ответа на поставленный вопрос вернемся к достаточно общей модели управляемого марковского процесса с конечным числом состояний и стратегий (действий). Будем использовать общие обозначения, введенные в разделе~\ref{MDP}. Только сделаем для удобства обозначений, одно небольшое уточнение: будем считать, что функция вознаграждения всецело определяется текущим состоянием, выбранной стратегией и состоянием, в которое перейдет процесс на следующем шаге. Таким образом, предполагается, что случайность в функции вознаграждения всецело определяется состоянием, в которое переходит процесс. Заметим, что и примеры с разборчивой невестой и с многорукими бандитами подходят под это предположение. В сделанных предположениях уравнение Вальда--Беллмана будет иметь вид
$$V^*(s) = \max\limits_{a \in A}  \sum\limits_{s' \in S} p(s,a;s') \left(r(s,a;s') + \gamma V^*(s') \right).$$
Введем $Q$-функцию
$$Q(s,a) =   \sum\limits_{s' \in S} p(s,a;s') \left(r(s,a;s') + \gamma V^*(s') \right).$$
Несложно заметить, что
$$V^*(s) = \max\limits_{a \in A}Q(s,a).$$
Следовательно, $Q$-функция должна удовлетворять \gav{\textit{$Q$-уравнению}}: 
$$Q(s,a) =   \sum\limits_{s' \in S} p(s,a;s') \left(r(s,a;s') + \gamma  \max\limits_{a' \in A}Q(s',a') \right).$$
Данное уравнение может быть решено методом последовательных (простых) итераций. Действительно, если смотреть на $Q = \{Q(s,a)\}_{s\in S, a \in A}$, как на вектор, то $Q$-уравнение можно записать в операторном виде $Q = H(Q)$ (метод простых итераций будет иметь вид $Q_{t+1} = H(Q_t)$), где по определению можно показать, что оператор в правой части $H$ является сжимающим с коэффициентом $\gamma$ в бесконечной норме: $$\max_{s\in S, a \in A} \left| H(\tilde{Q}(s,a)) - H(Q(s,a)) \right|\le \gamma \max_{s\in S, a \in A} \left| \tilde{Q}(s,a) - Q(s,a) \right|.$$ Однако это не решает отмеченные проблемы с отсутствием какой-либо информации о функциях $r(s,a;s')$ и $p(s,a;s')$. Основная идея $Q$-обуче\-ния заключается в замене невычислимой правой части в уравнении $Q_{t+1} = H(Q_t)$ на ее вычислимую несмещенную оценку: 
$$
Q_{t+1}(s,a) =   Q_{t}(s,a) +$$
\begin{equation}\label{Q}
+\alpha_t(s,a)\left(r(s,a;s'(s,a)) + \gamma  \max\limits_{a' \in A}Q_t(s'(s,a),a') - Q_{t}(s,a)\right),
\end{equation}
где $s'(s,a)$ -- положение процесса на шаге $t+1$, если на шаге $t$ процесс был в состоянии $s$ и было выбрано действие $a$. Если на шаге $t$ процесс находился в состоянии $s$ и было выбрано действие $a$, то $0<$\linebreak $<\alpha_t(s,a)\le 1$, иначе $\alpha_t(s,a) = 0$. При сделанных предположениях правая часть~\eqref{Q} может быть вычислена просто путем наблюдения того, куда перешел процесс $s'(s,a)$. Действительно, набор $\{Q_t(s,a)\}_{s\in S, a \in A}$ уже известен с прошлой итерации (следовательно, можно посчитать и $\max_{a' \in A}Q_t(s'(s,a),a')$), а вознаграждение $r(s,a;s'(s,a))$ мы также можем всегда наблюдать по условию, если из состояния $s$ при действии $a$ перешли в состоянии $s'(s,a)$. Определять $r(s,a;s'(s,a))$ в ненаблюдаемых состояниях и/или при неиспользуемых действиях нет необходимости ввиду условия $\alpha_t(s,a) = 0$ в этих случаях. Оказывается, что если  используемая стратегия $a(s)$ приводит к тому, что с вероятностью 1 каждая пара $(s,a)$ будет бесконечное число раз встречаться на~бесконечном горизонте наблюдения, то из отмеченного выше условия сжимаемости при
\begin{center}
$\sum\limits_{t=0}^{\infty} \alpha_t(s,a) = \infty$, $\sum\limits_{t=0}^{\infty} \alpha_t^2(s,a) < \infty$
\end{center}
будет следовать сходимость (с вероятностью 1) также и процесса~\eqref{Q} \cite{Jaakkola}: $\lim_{t\to\infty}Q(s,a) = Q(s,a)$, $V^*(s) = \max_{a \in A}Q(s,a)$. Таким образом, после достаточно большого числа шагов, даже в отсутствие какой-либо информации об управляемом марковском процессе, можно определить оптимальную стратегию $a(s) = \text{arg}\max_{a \in A}Q(s,a)$.

Однако приведенный результат ничего не говорит о скорости обучения и о том, когда можно заканчивать обучение, т.е. о том, в какой момент можно уже переходить на стратегию 
$$a_t(s) = \text{arg}\max\limits_{a \in A}Q_t(s,a)$$ 
со стратегии, которой придерживались сначала и которая обеспечивала максимально быструю сходимость процесса~\eqref{Q}. Собственно, наиболее ярко это все можно продемонстрировать как раз на примере задачи о многоруких бандитах, в которой до какого-то момента идет обучение и исследуются все ручки, и только после того, как все ручки были проверены достаточное число раз, выбирается наилучшая из них (\textit{exploration} vs \textit{exploitation}). На самом деле асимптотически оптимальные стратегии выглядят немного по-другому, см. следующий пункт. В~заключение  упомянем, что полученные недавно оценки скорости сходимости процедур, построенных на базе~\eqref{Q} для широкого класса задач, во многом отвечают разобранному здесь примеру задачи о многоруких бандитах~\cite{Jin}. О каких именно оценках идет речь, мы постараемся пояснить далее.  

\subsection[Асимптотически оптимальные оценки\vspace{-1mm}]{Асимптотически оптимальные оценки} Снова вернемся к задаче о многоруких бандитах. Будем считать $\gamma = 1$. Если число шагов $N$, то, зная оптимальную ручку (с наибольшей вероятностью успеха $p_{\max}$), можно получить ожидаемо\gav{е} вознаграждение $p_{\max}N$. Оказывается, что, не имея никакой информации об $n$ ручках, \gav{в общем случае} невозможно получить ожидаемое вознаграждение больше, чем~\cite{Bubeck,Cesa-Bianchi}: 
$$p_{\max}N - 0.05\sqrt{Nn}.$$ 
Этот результат можно объяснить, опираясь на следующее наблюдение, см. раздел~\ref{Fisher}. Пусть имеются всего две ручки. Одной соответствует вероятность $p = 1/2$, второй $p = 1/2 + \epsilon$. Но неизвестно, какой ручке, какая вероятность соответствует. Тогда, для того чтобы определить (скажем, с вероятностью 0.95), какой ручке соответствует большая вероятность успеха, необходимо дернуть каждую ручку не менее $\sim 1/\varepsilon^2$ раз. Возвращаясь к исходной постановке, припишем $n-1$ ручке одинаковую вероятность $p = 1/2$, а одной оставшейся ручке вероятность $p = 1/2 + \epsilon$, где $\epsilon = \sqrt{n/N}$. Поскольку всего шагов $N$, то хотя бы одна ручка выбиралась не более чем на $N/n = 1/\epsilon^2$ шагах. Таким образом, нельзя гарантировать, что на сделанных $N$ шагах можно достоверно определить лучшую ручку. Использование не оптимальной ручки приводит на каждом шаге к средним потерям $\epsilon$. Суммарные потери за $N$ шагов будут $\epsilon N = \sqrt{Nn}$, что и требовалось объяснить. 

Можно ли предложить такой алгоритм выбора ручек, который бы позволял приблизиться к приведенной оценке? Оказывается это можно сделать, не используя описанную в предыдущих пунктах технику,\footnote{Собственно говоря, не очень и понятно, что можно было бы использовать из описанного в предыдущих пунктах, кроме $Q$-обучения. Впрочем, для $Q$-обучения желательно (но не обязательно, см., например,~\cite{Jin}), чтобы $\gamma < 1$. В остальных подходах требуется делать априорные предположения о распределении $p_i$. От чего в данном пункте мы отказались.} а рассматривая задачу как задачу стохастической онлайн оптимизации~\cite{GasnikovOnline,Bubeck}. Мы не будем приводить здесь наилучший известный сейчас алгоритм (детали см. в~\cite{Bubeck,banditalgs}). Тем не менее описываемый далее алгоритм Exp3 гарантирует ожидаемое вознаграждение не меньше, чем $$p_{\max}N - 2\sqrt{Nn\ln n},$$ 
что достаточно близко к нижней оценке, приведенной выше.

Алгоритм Exp3 (Exponential weights for Exploration and Exploitation) предписывает выбирать на шаге $k$ ручку $i$ с вероятностью
 $$p^k_i = \frac{\exp\left(\eta_N R^k_i \right)}{\sum_{j=1}^n \exp\left(\eta_N R^k_j \right)},$$
 где $$\eta_N = \sqrt{\frac{2\ln n}{Nn}}.$$ 
 Если на шаге $t$ была выбрана ручка номер $i$, то
 $$R^{t+1}_i = R^{t}_i + \frac{1}{p^t_i},$$
 если был успех (с вероятностью $p_i$): 
  $$R^{t+1}_i = R^{t}_i,$$
  если была неудача (с вероятностью $1 - p_i$) и 
  $$R^{t+1}_j = R^{t}_j,$$
  для $j\neq i$. Причем $R^1_i = 0$, $i=1,\dotsc,n$.
  
  В некотором смысле аналогичные результаты недавно были получены и для достаточно большого и важного подкласса управляемых марковских процессов (MDP): \textit{episodic MDP}~\cite{Jin}. А именно, было показано, что нижние оценки тут можно получать на базе задачи о многоруких бандитах, а верхние оценки, которые в данном случае все же получаются немного похуже нижних, обеспечивает $Q$-обучение, в варианте очень близком к приведенному выше~\eqref{Q}, с естественной стратегией поведения $a_t(s) = \text{arg}\max_{a \in A}Q_t(s,a)$. 
  
 \ag{Отметим также работы \cite{JinSidford,Wainwright}, в которых показано, что если для MDP разрешается произвольно выбирать следующее состояние $s'$ (и действие), при которых можно наблюдать вознаграждение $r$ или его реализацию (т.е. не обязательно следовать матрице переходных вероятностей $p$ MDP при выборе нового состояния $s'$, в котором следует наблюдать $r$), то для равномерной аппроксимации функции $Q$ (вектора $Q$ в бесконечной норме) с точностью $\varepsilon$ достаточно (с точностью до логарифмических множителей) $\sim|S||A|\varepsilon^{-2}$ переходов управляемого марковского процесса и вычислений вознаграждений при этих переходах. Данная оценка оптимальная. Примечательно, что работа \cite{Wainwright} также базируется на процедуре типа $Q$-обучения.}
  
  В целом затронутые в данном разделе вопросы оказываются сильно завязаны на стохастическую оптимизацию, онлайн оптимизацию, взвешивание экспертных решений, предсказание последовательностей, теорию игр и, конечно, машинное обучение. Заинтересовавшемуся читателю можно порекомендовать следующую литературу~\ag{\cite{GasnikovOpt,GasnikovOnline,Bubeck,Cesa-Bianchi,Rakhlin,Shalev-Shwartz,Slivkins,banditalgs}}. В частности, около трети студентов \gav{ФУПМ} МФТИ продолжат изучение данной темы в рамках курса~\cite{Vyugin}.
  
    \section*{Заключение}\addcontentsline{toc}{section}{Заключение\vspace{-1mm}}
Далее приводится краткое содержание пособия с выделением основных идей и объяснением связей с другими областями.

В курсе теории вероятностей~\cite{Gnedenko, NGG1, ShiryaevT1} 
было показано, что основной объект изучения -- случайные события и подсчет их вероятностей сводится к изучению  случайных величин $X$ и случайных векторов $\left[X_1,...,X_n\right]^\top$ и свойств их функций распределения, которые полностью их определяют\footnote{Для доказательства предельных теорем удобным средством является аппарат характеристических функций (х.ф.). Грубо говоря, вместо изучения функций распределения предлагается изучать преобразования Фурье их производных (производная функции распределения есть плотность распределения). Связано это с тем, что х.ф. суммы независимых с.в. равна произведению соответствующих х.ф.}: $F_X(x) = \mathbb{P}(X < x)$ и $F_{X_1,...,X_n}(x_1,...,x_n) =$\linebreak $= \mathbb{P}(X_1 < x_1, ..., X_n < x_n)$. 
Были установлены естественные (необходимые и достаточные) условия\footnote{Из наиболее нетривиальных отметим непрерывность слева по каждому аргументу и условие вида ($n=2$):  для любых $a_1<a_2$, $b_1<b_2$ выполняется 
$$F_{X_1,X_2}(a_2,b_2) - F_{X_1,X_2}(a_1,b_2) - F_{X_1,X_2}(a_2,b_1) + F_{X_1,X_2}(a_1,b_1) \ge 0.$$} 
(теорема Колмогорова~\cite{Kolmogorov1974}), при которых $F_{X_1,...,X_n}(x_1,...,x_n)$ будет функцией распределения некоторого случайного вектора. Также в курсе теории вероятностей было введено важное понятие слабой сходимости случайных векторов, которое можно понимать следующим образом: $\left[X_1^N,...,X_n^N\right]^\top$ слабо сходится к~$\left[X_1,...,X_n\right]^\top$ при $N \to \infty$ если $$F_{X_1^N,...,X_n^N}(x_1,...,x_n) \to F_{X_1,...,X_n}(x_1,...,x_n)$$ в точках непрерывности $\left(x_1,...,x_n\right)$ последней функции. 

Если в определении случайного вектора считать $n$ счетной бесконечностью, то получим случайный процесс в дискретном времени. Все основные свойства без изменения наследуются из курса теории вероятностей с небольшими техническими оговорками, которые мы изложим далее в более общем случае (для непрерывного времени). 
Сложнее обстоит дело для случайного процесса в непрерывном времени. 
В этом случае (как и в случае счетного времени) случайный процесс задается семейством своих всевозможных конечномерных распределений $F_{X(t_1),...,X(t_n)}(x_1,...,x_n)$, где $n$ пробегает всевозможные натуральные числа, а набор $t_1,...,t_n$ пробегает всевозможные значения с условием: $0\le t_1<...<t_n$. Причем $$F_{X(t_1),...,X(t_n)}(x_1,...x_{n-1},\infty) = F_{X(t_1),...,X(t_{n-1})}(x_1,...x_{n-1}).$$
Слабая сходимость случайных процессов понимается как слабая сходимость всевозможных конечномерных распределений. Проблема однако возникает при  подсчете вероятностей событий вида\footnote{Тут может возникнуть естественное желание не считать вероятности таких событий и ограничиться рассмотрением только дискретного времени. Однако ситуация здесь приблизительно такая же, как и с дифференциальными уравнениями -- переход от изучения динамики в дискретном времени к непрерывному обогащает и упрощает в ряде случаев изучение интересующих процессов. Далее в курсе это можно пронаблюдать на примере случайного блуждания в дискретном времени и броуновского движения.} $\{\sup_{t\in[0,T]} X(t) <$\linebreak $< x\}$. Событие, вероятность которого необходимо посчитать, в общем случае не принадлежит $\sigma$-алгебре событий, порожденной семейством конечномерных распределений. Причина связана с несчетностью множества $[0,T]$. Однако
$X(t)$ имеет сепарабельную стохастическую модификацию $\tilde{X}(t)$~\cite{BulinskyShiryaev2005, VentselAD}. 
Последнее означает, что для каждого $t\ge 0$ выполняется $\mathbb{P}(X(t)=\tilde X(t)) = 1$ и существует такое счетное множество\footnote{Eсли процесс $X(t)$ стохастически непрерывен $\lim_{\epsilon\to 0}\mathbb{P}(|X(t+\epsilon) - X(t)|>\delta) = 0$ для любого $\delta >0$, то в качестве множества $S$ можно брать любое счетное всюду плотное множество~\cite{BulinskyShiryaev2005, VentselAD}  
.} $S$ в $\mathbb{R}_+$, что для любого $Q$ (в том числе и для $Q = [0,T]$) события
$\sup_{t\in Q} \tilde{X}(t) < x$ и $\sup_{t\in S\cap Q} \tilde{X}(t) < x$ совпадают с вероятностью 1,
т.е. для стохастической модификации отмеченной проблемы нет. Отмеченные выше тонкости необходимо знать для понимания объекта изучения (случайного процесса), однако последующее изложение практически не требует понимания написанного выше. 

Намного более важным разделом курса случайных процессов является специальный класс (стохастически непрерывных) процессов, у которых приращения независимы. Особенно важным представителем этого класса процессов являются однородные процессы (процессы Леви~\cite{Applebaum}). 
Однородность процесса $X(t)$ означает, что распределение $X(t+s) - X(t)$ не зависит от $t$, а независимость приращений, что для любого натурального $n$ и $0\le t_1 <...< t_n$ случайные величины $X(t_n)-X(t_n-1)$, ..., $X(t_1)-X(0)$ -- независимы в совокупности. Из этих двух свойств\footnote{Далее еще используется свойство $X(0)=0$ почти наверное.} следует, что для любого $n$ справедливо представление $$X(t) = \sum_{k=0}^{n-1}\left(X\left(\frac{(k+1)t}{n}\right) - X\left(\frac{kt}{n}\right)\right),$$ 
т.е. $X(t)$ для любого $n$ может быть представлена (имеет такое же распределение) как сумма независимых одинаково распределенных случайных величин. Последнее означает, что $X(t)$ -- безгранично-делимая случайная величина\footnote{Другой способ определения безгранично-делимой случайной величины -- это такая случайная величина, которая может возникать в качестве предела сумм независимых одинаково распределенных случайных величин~\cite{GnedenkoKolmogorov}}. 
Несложно понять, что нормальная случайная величина и пуассоновская случайная величина являются безгранично делимыми.\footnote{Собственно, именно этот факт и определяет привилегированность данных распределений, отражающуюся, например, в том, что часто именно эти два закона возникали в качестве предельных распределений сумм случайная величина в курсе теории вероятностей (теорема Муавра--Лапласа, ц.п.т., теорема Пуассона). Другой факт, свидетельствующий о привилегированности нормального распределения -- теорема Максвелла--Клартага~\cite{Zorich, Klartag} 
-- о том, что проекция вектора, равномерно распределенного на шаре (сфере) радиуса $\sqrt{N}$, на фиксированную гиперплоскость малой размерности стремится при $N\to\infty$ к нормальному вектору, причем все это переносится и на равномерные распределения на изотропных выпуклых телах с диаметром, пропорциональным $\sqrt{N}$~\cite{Milman2008}.} 
Более того, они сами себя <<безгранично делят>>. Например, сумма $n$ независимых одинаково распределенных по закону $\mathrm{N}(m/n,\sigma^2/n^2)$ случайных величин имеет закон распределения $\mathrm{N}(m,\sigma^2)$. Аналогично и сумма $n$ независимых случайных величин с распределением $\Po(\lambda/n)$ имеет распределение $\Po(\lambda)$. Есть ли еще какие-то безгранично-делимые случайные величины? Ответ: да -- сложная пуассоновская случайная величина: $\sum_{k=0}^{K} V_k$, где все случайные величины $K\in \Po(\lambda)$, $\left\{V_k\right\}_{k=0}^{K}$ независимы в совокупности, при этом все случайные величины $\left\{V_k\right\}_{k=0}^{K}$ одинаково распределены (не важно, как именно). Отметим, что нормальным распределением и сложным пуассоновским распределением\footnote{C точностью до небольших оговорок относительно возможности того, что случайные величины $\left\{V_k\right\}_{k=0}^{K}$ несобственные.
Отметим также, что при наиболее естественных условиях в качестве предельных распределений возникает распределение Пуассона (теорема Пуассона), нормальное распределение (ц.п.т.) и вырожденный случай нормального распределения ($\sigma = 0$) -- не случайная величина (з.б.ч.). В последнем случае из слабой сходимости (сходимости по распределению) вытекает сходимость по вероятности.  Однако вполне можно себе представить <<жизненные>> ситуации, когда предельное распределение может быть и не из отмеченного набора, например, гравитационное поле, создаваемое равномерно распределенными во Вселенной звездами, на Земле описывается безгранично-делимым распределением Хольцмарка~\cite{KendallMoran, Chebotarev}.}  
исчерпывается класс безгранично-делимых с.в. Из вышенаписанного <<в первом приближении>> можно заключить, что для процесса Леви $X(t+s) - X(t)$ имеет либо нормальное распределение $\mathrm{N}(ms,\sigma^2s)$ -- в этом случае процесс Леви называют броуновским движением (если $m=0$, $\sigma^2 = 1$ -- винеровским процессом), либо сложное распределение Пуассона (с $K\in\Po(\lambda) \to K(s) \in\Po(\lambda s)$).

Броуновское движение, так же, как и нормальный закон в ц.п.т., естественным образом возникает в задачах, в которых изучают предельное поведение других случайных процессов вида суммы. Рассмотрим один естественный пример (случайное блуждание на прямой). Пусть время течет дискретно (один шаг по времени $1/N$) и точка может двигаться вперед с вероятностью  $1/2$ на $\bar{\sigma}/\sqrt{N}$ и назад с вероятностью  $1/2$ на $\bar{\sigma}/\sqrt{N}$. Если рассматривать предел\footnote{Отметим, что скейлинг с другими степенями $N$ не приводит к интересным результатам -- полученный в результате скейлинга случайный процесс либо будет не отделим от нуля с вероятностью 1, либо сразу же <<взорвется>> и уйдет на бесконечность с вероятностью 1.} (диффузионный скейлинг) при $N\to\infty$, то описанное выше случайное блуждание слабо сходится к броуновскому движению с $m = 0$ и $\sigma = \bar{\sigma}$~\cite{BulinskyShiryaev2005}. 
\footnote{Заметим, что из ц.п.т. следует слабая сходимость в каждом сечении.} Другой важный для курса статистики~\cite{BorovkovMatStat, BulinskyShiryaev2005}  
и статистической теории обучения~\cite{Koltchinskii} (теории эмпирических процессов) 
пример\footnote{В данный курс было решено не включать изложение данного примера, однако мы рекомендуем заинтересованным читателям посмотреть указанную литературу ввиду важности примера для последующего более глубокого понимания курса математической статистики.} -- это эмпирическая функция распределения и слабая сходимость функционалов (статистик) от этой функции к функционалам от винеровского процесса. Заметим, что, как правило, такие функционалы включают в себя супремум по несчетному множеству, поэтому отмеченные выше тонкости (о существовании сепарабельной стохастической модификации у стохастически непрерывного случайного процесса) все же необходимо принимать в расчет.

Хорошим примером возникновения пуассоновского процесса является модель случайного бросания (независимого и равновероятного) $N$ точек на отрезок $[0,N]$ и изучение числа точек 
$K^N(t)$, попавших в отрезок $[0,\lambda t]$. 
Если $t$ считать параметром и при зафиксированных $\lambda$ и $t$ устремить $N \to \infty$ (термодинамический предельный переход), то $K^N(t)$ сходится по распределению к $\Po(\lambda t)$ (этот простой факт следует из теоремы Пуассона) и, более того, $K^N(t)$ слабо сходится к пуассоновскому процессу $K(t)$ с параметром $\lambda$. Такое понимание пуассоновского процесса часто используется в статистической физике и различных геометрических задачах с вероятностью~\cite{KendallMoran, Kingman, KoralovSinai, Malyshev}. 
Другой способ определения пуассоновского процесса $K(t)$ -- число звонков <<пуассоновского будильника>> к моменту времени $t$. Пуассоновский будильник определяется следующим образом. Вероятность того, что он зазвонит в промежутке времени $[t,t+\delta t]$, не зависит от $t$ и от того, сколько раз он уже звонил и равна $\lambda \delta t + o(\delta t)$. Другой способ определения пуассоновского будильника 
$$K(t) = \max\left(k\ge 0: \sum_{i=1}^k \xi_i < t \right),$$
где $\xi_i \in \text{Exp} (\lambda)$ -- независимые случайные величины, базируется на свойстве отсутвия последействия $\mathbb{P}(\xi > t +\tau \,|\, \xi > t) = \mathbb{P}(\xi > \tau) $ у показательной случайной величины: $\xi \in \text{Exp} (\lambda)$ если $\mathbb{P}(\xi > t) = \exp(-\lambda t)$. Отмеченное свойство в классе случайных величин, имеющих плотность, имеет место только у показательных случайных величин. 

Понимание пуассоновского процесса посредством пуассоновского будильника позволяет достаточно легко построить и теорию марковских процессов в непрерывном времени с конечным или счетным числом состояний как обобщение пуассоновского процесса. А именно, в момент звонка будильника частица (независимо от предыстории\footnote{В этом и в \eg{отсутствии} последействия и заключается марковское свойство.}) перемещается по (ориентированному) графу марковской цепи\footnote{Каждая вершина графа отвечает состоянию, а наличие ребра -- возможности перехода; над каждым ребром написана вероятность, сумма вероятностей на ребрах, выходящих из каждой вершины, равна 1.} согласно вероятностям, написанным на выходящих из данной вершины (в другие вершины/состояния) ребрах. Отметим, что ввиду независимости приращений, винеровский процесс (броуновское движение) также будет марковским процессом, т.е. процессом, вероятностное описание будущего которого всецело определяется точно заданным настоящим (текущим) положением, но не тем, как в это положение процесс пришел. Марковское свойство является ключевым для возможности построения аналогии между динамическими системами (дифференциальными и разностными уравнениями) и стохастическими процессами. Так же, как и для динамических систем, для марковских процессов можно задать (инфинитезимальный) оператор перехода (с полугрупповым свойством), который определяет эволюцию процесса. Другими словами, для определения вероятностного закона распределения сечений процесса, достаточно решить некоторую систему дифференциальных уравнений Колмогорова--Феллера (для непрерывного времени и конечного/счетного числа состояний) или уравнение в частных производных Колмогорова--Фоккера--Планка (для непрерывного времени и несчетного числа состояний).\footnote{Полезно заметить, что вывод таких уравнений базируется на формуле полной вероятности и различных предельных переходах. Причем пуассоновский будильник хорошо поясняет, как именно осуществлять такие переходы в части рассмотрения непрерывного времени на базе дискретного времени. В частности, именно таким образом наиболее просто объяснить вывод системы уравнений Колмогорова--Феллера из формулы полной вероятности (уравнения Колмогорова--Чэпмена), описывающей эволюцию конечной однородной дискретной марковской цепи.} Огромная популярность марковских процессов в приложениях как раз и обусловлена тем, что это в некотором смысле наиболее хорошо математически изученный  инструмент, с помощью которого можно аппроксимировать и изучать сложные процессы с непонятной структурой. Различные попытки существенно выйти за пределы этого класса случайных процессов хотя и привели к определенным успехам (стационарные процессы, гауссовские процессы, мартингалы и др.), о чем будет сказано далее, тем не менее основным (наиболее важным) классом процессов, безусловно, являются именно марковские процессы, ввиду богатой теории и возможности их глубокого изучения, а также ввиду вполне естественных предположений (будущее определяется настоящим и только настоящим), которые в хорошем приближении выполняются во многих интересных на практике задачах. 

Наиболее важным классом задач, возникающим при изучении случайных процессов, в том числе марковских, является описание поведения процесса при больших значениях времени. Общая идея поиска асимптотики следующая: если, например, изучаем предел числовой последовательности~\cite{Fichtenholz} 
$x^{k+1} = \sqrt{x^k + 2}$ и считаем, что предел существует, то, чтобы найти такой предел $a = \lim_{k\to\infty} x^k$, нужно найти неподвижную (стационарную, инвариантную) точку данной динамики $a=\sqrt{a+2}$, то есть $a = 2$. Аналогичным образом необходимо действовать и в случае марковских процессов. 

Как уже отмечалось, для марковских процессов так же, как и для однородных (время не входит в уравнения) динамических систем, можно говорить об эволюции меры. Только для марковских процессов эту меру можно не вводить, она естественным образом присутствует в самой задаче. Раскроем написанное выше подробнее. Пусть в начальный момент динамическая система с разной вероятностью могла находиться в разных состояний, тогда под эволюцией меры, описывающей заданную систему, понимается просто описание того, с какой вероятностью в каждый момент времени система будет находиться в том или ином состоянии. В пространстве мер однородная динамическая система порождает унитарный (ортогональный) оператор (грубо говоря, поворот пространства). Стационарная (инвариантная) мера, которую данный оператор оставляет неподвижной, в случае единственности (нет других стационарных мер -- это условие называется условием эргодичности системы) и определяет предельное распределение динамической системы (не зависящее от начального!). Скажем, если динамическая система -- это поворот окружности единичного периметра на иррациональный угол, то единственной инвариантной (стационарной) мерой такой эволюции будет равномерная мера (так же говорят <<лебегова мера>>) на данной окружности. Таким образом, с какими бы вероятностями система не была  <<разбросана>> в начальный момент по этой окружности, со временем вероятность найти систему на некотором измеримом подмножестве этой окружности будет равна лебеговой мере этого подмножества. Можно сказать дополнительно, что доля времени, которую система (откуда бы она не стартовала) провела на выделенном множестве на бесконечном промежутке наблюдения за ней, также равна площади этого множества (среднее по времени равно среднему по пространству, если среднее по пространству понимать согласно инвариантной мере).\footnote{Отметим, что описанный пример демонстрирует порождение последовательности со случайными свойствами чисто детерминированными методами. Такого типа последовательности активно используются в методах Монте-Карло~\cite{Sobol}. 
Например, при вычислении многомерных интегралов. Пример вычисления одномерного интеграла с помощью такой последовательности разобран в пособии.} Все эти результаты естественным образом переносятся на марковские случайные процессы (а также класс стационарных процессов; и даже, с некоторыми оговорками, на очень общий класс -- процессы второго порядка с постоянным средним), для которых мы изначально имеем только одно пространство -- вероятностное и сразу изучаем эволюцию мер, заданную линейным оператором марковской полугруппы (унитарность оператора проявляется в том, что при эволюции меры свойство нормировки на единицу и неотрицательности не изменяется).

Ключевым вопросом в описанной выше схеме является выявление необходимых и достаточных условий, которые гарантируют единственность стационарной (инвариантной) меры\footnote{Существование инвариантной (стационарной) меры для динамических систем, действующих на компакте, следует из теоремы Боголюбова--Крылова~\cite{SinaiLections1996}, а для марковских процессов с конечным числом состояний -- из принципа неподвижной точки Брауэра: непрерывное отображение (в нашем случае линейное) выпуклого компакта (в нашем случае симплекса) в себя имеет неподвижную точку. Пример случайного блуждания на прямой показывает, что уже в случае счетного числа состояний стационарной меры может не существовать. Точнее говоря, существует только тривиальная (нулевая) стационарная мера, что не интересно, ввиду необходимости нормировать стационарную (инвариантную) меру на 1.}, а также вопрос о том, что будет, если такая мера не единственная?

Для марковских процессов с конечным или счетным числом состояний в пособии приводится подробный ответ на поставленный вопрос. А~именно, если изоморфным образом понимать марковский процесс как случайное блуждание на графе согласно вероятностям, написанным на ребрах, то стационарное распределение единственное, если существует такое состояние (вершина), в которую можно попасть, двигаясь по этому графу из любой другой вершины~\cite{BulinskyShiryaev2005}. 
Если это условие не выполняется, то стационарное распределение не единственно, и предельное (финальное) распределение уже зависит от стартового (начального)\footnote{Тут следует сделать оговорку. Для марковских цепей с не конечным числом состояний, например, со счетным числом состояний, в ряде случаев существует возможность с ненулевой вероятностью <<уйти на бесконечность>>. Например, в игре в орлянку, когда игрок играет с казино, имеющим неограниченный бюджет, и выигрывает у казино с большей вероятностью, чем проигрывает, существует единственное стационарное распределение, сосредоточенное в состоянии, когда игрок разорится. Тем не менее в финальном состоянии вероятность игрока оказаться в этом состоянии будет зависеть от его стартового капитала и будет меньше 1. С~оставшей вероятностью игра будет продолжаться бесконечно. Таким образом, для счетных цепей даже в случае единственности стационарного распределения, но наличия бесконечного числа несущественных состояний, финальное распределение может зависеть от начального.}. В~этом случае также можно определять асимптотику. Важным примером, демонстрирующим, как это следует делать, является задача об игре в орлянку, разобранная в пособии. В этой игре двое игроков кидают монетку и делают ставки по одному рублю. Игра идет до разорения одного из игроков. Понятно, что в данной игре существует две стационарные меры. Одна сосредоточена в состоянии, когда все деньги у первого игрока, вторая сосредоточена в состоянии, когда все деньги у другого игрока. Поиск вероятностей, с которыми система <<свалится>> из заданного начального положения (состояния) в одно из этих предельных состояний, сводится к выписыванию рекуррентных соотношений, связывающих (с помощью формулы полной вероятности) вероятности разорения в текущем стартовом состоянии с вероятностями разорения для соседних (на графе) стартовых состояний.

Вопрос о сходимости к стационарному распределению в случае непрерывного времени и компактного пространства состояний не стоит (достаточно единственности стационарного распределения). В случае дискретного времени необходимы оговорки о непериодичности марковской цепи, однако эти технические моменты вполне можно опустить в сухом остатке от курса, поскольку в чезаровском смысле (наиболее важном с точки зрения основных приложений) сходимость будет в любом случае (и для периодических дискретных марковских цепей)\footnote{В дискретном времени марковская динамика задается следующим уравнением Колмогорова--Чэпмена эволюции меры (вектора распределения вероятностей по состояниям $p$): $p^{k+1} = P^\top p^k$, где $P$ -- матрица переходных вероятностей ($P_{ij}$ -- вероятность перейти из состояния $i$ в состояние $j$). Принцип сжимающих отображений в классическом варианте здесь не сработает, поскольку $\lambda_{\max}(P^\top) = 1$, и, таким образом, $P^\top$ не может быть сжимающим отображением. С другой стороны, \ag{у $P^{\eg{\top}}$ есть сжимаемость к} собственному вектору $\pi$, отвечающего максимальному собственному значению (стационарному распределению, вектору Фробениуса--Перрона в экономической литературе~\cite{Nikaydo}). 
То есть можно надеяться, что, например, $\left\|\left(P^\top\right)^n(p - \pi)\right\|_1 \le \alpha(P)^n$, где $\alpha(P)<1$. Как искать это $\alpha(P)$ и при каких условиях можно обеспечить условие $\alpha(P)<1$? Ответ на этот вопрос можно попробовать получить, рассмотрев модельный случай, когда у марковской цепи всего два состояния~\cite{Malyshev}. 
Оказывается, что  достаточным (и необходимым в случае отсутствия несущественных состояний) условием, которое назовем условием (Э), для $\alpha(P)<1$ является существование такой натуральной  степени $r$, что у матрицы $P^r$ все элементы положительны. Это условие равносильно тому, что граф марковской цепи неразложим (из любого состояния можно добраться в любое другое по этому графу) и цепь непериодическая. Замечательно, что все эти результаты переносятся и на более общие линейные динамики с матрицами, элементы которых неотрицательны. Такие динамики возникают в динамических моделях  межотраслевого баланса~\cite{Nikaydo}. 
Заметим также, что можно ввести специальную метрику (Биркгофа) на пространстве лучей неотрицательного ортанта так, что в этой метрике линейный оператор $P^\top$ будет сжимающим при том же самом условии (Э)~\cite{Krasnoselsky}. 
Отметим, что оценка $\alpha(P)$  является важной задачей (о некоторых подходах к ее решению см., например,~\cite{KelbertSukhov2010}). 
Чем меньше значение $\alpha(P)$, тем быстрее цепь выходит на стационарное распределение, что может быть важно, например, в методах Монте-Карло, базирующихся на марковских цепях, см. далее.}.

Стоит отметить, что изложенная выше схема, связывающая эргодичность случайных процессов и эргодичность динамических систем~\cite{Katok1999}, 
не является оригинальной и уже неоднократно использовалась в различных местах при разработке курса теории вероятностей и случайных процессов~\cite{Bulinsky2010, KoralovSinai, Malyshev, Shiryaev2}. 
Однако в данном пособии было решено пойти дальше. А именно, на примере изучения эволюции ансамбля конечных однородной дискретных марковских цепей продемонстрировать концепцию равновесия макросистемы. Более точно, вместо того, чтобы изучать вероятность нахождения блуждающей частицы на графе в разные моменты времени, находящейся в начальный момент с разными вероятностями в разных состояниях, предлагается поместить в эти состояния в начальный момент частицы в пропорции, определяемой начальными вероятностями, а далее следить за эволюцией всех частиц одновременно. Для эргодической марковской цепи каждая частица (независимо от стартового состояния) распределится на графе согласно стационарному распределению. Если частиц достаточно много, то со временем по распределению частиц на графе можно будет восстанавливать стационарное распределение. Таким образом, в пособии интерпретируется вектор PageRank (см. также~\cite{BlumHopcroftKannan}) 
как вектор, который. с одной стороны, является стационарным распределением марковской цепи с графом переходных вероятностей, отвечающих интернету (вершины это web-страницы, а ребра -- гиперссылки), а с другой стороны, как вектор, отражающий популярность web-страниц (компоненты вектора пропорциональны числу посетителей web-страницы в единицу времени при установившемся режиме). Важно также отметить, что описанный способ восстановления стационарного распределения марковской цепи (Markov Chain Monte Carlo~\cite{Diaconis}) 
является хорошим практическим/численным способом (ввиду возможности параллелизации блужданий частиц) поиска (генерирования) стационарного распределения. 

Однако в пособии для закрепления материала и дополнительной демонстрации концепции равновесия макросистемы\footnote{Равновесие макросистемы -- это такое ее макросостояние, около которого концентрируется стационарная мера при стремлении числа агентов к бесконечности. В случае задачи поиска вектора Page Rank агентами были блуждающие по web-графу пользователи.} разбирается парадокс Эренфестов в варианте~\cite{KelbertSukhov2010}. 
Пример замечателен во многих аспектах, в частности, своей связью с простейшими системами массового обслуживания (процессы гибели и размножения), в связи с обыгрыванием теоремы Пуанкаре о возвращении, понятием обратимой динамики, демонстрацией того, что математическое ожидание времени первого возвращения марковской цепи в заданное состояние есть величина, обратная к соответствующей (данному состоянию) компоненте стационарной меры. Наконец, пример наглядно демонстрирует другую важную связь случайных процессов (процессов стохастической химической кинетики~\cite{Gardiner,  MalyshevPirogov, EthierKurtz})\footnote{Эволюция распределений которых всегда задана линейными уравнениями!} и систем (не обязательно линейных!) дифференциальных уравнений. А именно, в результате специального (канонического) скейлинга из первого объекта можно получать второй. Аттрактор полученной системы дифференциальных уравнений будет соответствовать равновесию макросистемы, заданной шкалируемым процессом (детали см., например, в~\cite{StochAn2016}). 
Другой достаточно поучительный пример скейлинга (так же, как и ранее на базе теоремы Куртца~\cite{EthierKurtz}) 
стохастической динамики в детерминированную содержится в пособии при описании роста интернета~\cite{Mitzenmacher,Newman}. 

Продолжая тему марковских процессов, отметим, что в пособие было решено включить, пожалуй, наиболее яркий (из известных нам) пример на управляемые марковские процессы -- <<задачу о разборчивой невесте>>~\cite{Gusein-Zade}, 
как демонстрацию основного принципа стохастического динамического программирования -- принципа Вальда--Беллмана. Здесь мы также не были оригинальными, см.~\cite{Rozanov1,Shiryaev2}. 
Ранее такого типа примеры включали в связи с большой популярностью этого принципа в исследовании операций, в том числе в военном деле~\cite{Ventsel1}. 
Однако в нашем случае мотивация была несколько другой -- а именно, в связи с огромной популярностью обучения с подкреплением~\cite{SuttonBarto} (Reinforcement Learning), 
также базирующегося на упомянутом ранее принципе Вальда--Беллмана.

Завершая тему марковских процессов, снова вернемся к винеровскому процессу  $W(t)$ и, следуя П. Самуэльсону, рассмотрим геометрическое броуновское движение  $S(t) = \exp\left(\ag{(a - \frac{\sigma^2}{2})t} +\sigma W(t)\right)$. Таким процессом на заре эры финансовой математики описывали поведение цены акции, причины см., например,~\cite{StochAn2016, Shiryaev1, Ross}. 
Можно показать, что следующее стохастическое дифференциальное уравнение\footnote{Структура которого хорошо проясняет, почему $S(t)$ может описывать поведение цены акции в случае сложной  процентной ставки.}: $$dS(t) = aS(t)dt + \sigma S(t)dW(t)$$ определяет выписанное ранее $S(t)$. Более корректно это уравнение записывается в интегральном виде:
 $$S(t) = a\int\limits_{0}^{t} S(\tau)d\tau + \sigma \int\limits_{0}^t S(\tau)dW(\tau).$$
 Последний объект уже вполне можно определить стандартным образом. Собственно, марковские процессы (обычно их называют диффузионные процессы или процессы Ито), полученные в результате решения стохастических дифференциальных уравнений, представляют естественное обобщение винеровского процесса, и возникают в различных приложениях, из которых наиболее часто встречается  физика (уравнение Ланжевена) и финансовая математика~\cite{Bulinsky2010, Oksendal2003}. 
 Использование в качестве шума именно винеровского процесса представляется настолько же естественным, насколько естественным представляется использование нормального распределения для описания шума в теории вероятностей. Напомним, что нормальное распределение часто наблюдается на практике благодаря ц.п.т. и ее робастности (центральная предельная теорема часто выполняется при намного более общих условиях, чем условия, при которых она изучалась в курсах теории вероятности). Другими словами, наличие большого количества независимых (или даже слабо зависимых) малых шумов одного масштаба (не обязательно одинаково распределенных) в сумме в хорошем приближении приводит к нормальному шуму. Таким образом, с помощью стохастических дифференциальных уравнений вносится (винеровский) шум в обычные детерминированные динамические системы.
  
В качестве одного из интересных следствий данного раздела отметим, что с помощью диффузионных процессов (в частности, случайных блужданий) можно численно решать краевые задачи для полуэллиптических уравнений в частных производных\gav{,} в частности задачу Дирихле для уравнения Лапласа~\cite{BulinskyShiryaev2005, Oksendal2003}. 
Один простой пример для случайного блуждания будет рассмотрен в пособии, следуя~\cite{Sosinsky}. 
Общая идея -- выпускать из данной точки рассматриваемой области независимые траектории случайного блуждания и ждать момента их первого попадания на границу той области, где решается краевая задача.
Усредняя значение граничного условия в точках выхода различных траекторий по этим траекториям, можно восстанавливать решение соответствующей задачи Дирихле для уравнения Лапласа. Идея обоснования этого наблюдения близка к тому, как исследуется задача об игре в орлянку. Отметим также глубокие связи данного результата с ТФКП~\cite{Oksendal2003, Shabbat}.

Выше мы уже отмечали другие классы процессов (мартингалы, гауссовские процессы), для которых удаётся построить достаточно интересную теорию, приносящую много полезного на практике. 

К сожалению, объем курса и соответствующего ему пособия не позволяет никак коснуться такого важного класса процессов, как мартингалы и мартингал-разности. Изложение таких процессов вошло в ряд современных учебников~\cite{BulinskyShiryaev2005, KoralovSinai, Shiryaev2}. 
В основном данные процессы описываются в связи с приложениями к финансовой математике (основная теорема финансовой математики об отсутствии арбитража и мартингальная мера~\cite{StochAn2016, Shiryaev1, Shiryaev2}) 
и страхованию (теорема Лундберга--Крамера), однако хотелось бы отметить, что во многих современных приложениях в Машинном обучении~\cite{Shalev-Shwartz}, 
Стохастической оптимизации~\cite{GasnikovOpt, Duchi, Shapiro}, 
Онлайн оптимизации~\cite{Cesa-Bianchi, Hazan}, Многоруких бандитах~\cite{Bubeck, banditalgs} 
мартингалы (а точнее мартингал-разности) возникают в связи с изучением скорости сходимости итерационных процедур вида стохастического градиентного спуска $$x^{k+1} = x^k - h \nabla_x f(x^k,\xi^k),$$ где  $\mathbb{E}_{\xi^k}  \nabla_x f(x^k,\xi^k) = \nabla f(x^k)$.\footnote{Используется неравенство концентрации меры~\cite{BoucheronLugosiMassart} 
Азумы--Хеффдинга для оценки сумм мартингал-разностей~\cite{LanNemirovskiShapiro}: 
$$\sum_{k=0}^N \langle \nabla_x f(x^k,\xi^k) - \nabla f(x^k), x^k - x_* \rangle.$$} 

Про гауссовские процессы~\cite{Peterbarg} 
 было решено добавить в пособие только самые необходимые сведения. Прежде всего, полезно заметить, что гауссовским процессом, т.е. процессом, любое конечномерное семейство распределений которого -- нормальный случайный вектор, является броуновское движение. Важной особенностью гауссовских процессов (унаследованной от гауссовских/нормальных векторов) является то, что вероятностные свойства такого процесса $X(t)$ полностью определяются его математическим ожиданием $m(t) = \mathbb{E}X(t)$ и корреляционной функцией\footnote{Корреляционная функция любого процесса Леви (в том числе и броуновского движения) может быть представлена в виде $R(t_1,t_2) = \mathbb{D}\left(\min(t_1,t_2)\right)$, 
 где $\mathbb{D}(t) =$\linebreak $= \mathbb{E}\left(X(t)^2\right) - \left(\mathbb{E}X(t)\right)^2$ -- дисперсия $X(t)$.}
 $R(t_1,t_2) = \mathbb{E}\left((X(t_1) - m(t_1))(X(t_2 - m(t_2))\right)$. 
 В~частности, независимость сечений гауссовского процесса эквивалентна некоррелированности, стационарность в широком и узком смыслах совпадают. Но, пожалуй, главным свойством, являющимся следствием эквивалентности независимости и некоррелированности сечений, является то, что условное математическое ожидание одних сечений гауссовского процесса при фиксированных других всегда есть линейная вектор-функция от этих фиксированных сечений. 
 Данный факт заслуживает особого внимания, поскольку играет очень важную роль при прогнозировании и восстановлении значений гауссовских процессов. Для его понимания полезно в целом напомнить, что такое условное математическое ожидание. Далее ограничимся для простоты пространством случайных величин с конечным математическим ожиданием их квадрата~\cite{Rozanov2}. 
 Введем в этом пространстве скалярное произведение по формуле $\langle X,Y \rangle = \mathbb{E}(XY)$. Несложно проверить, что все свойства скалярного произведения выполняются, если под элементом этого пространства понимать не просто случайную величину, а целый класс эквивалентных ей случайных величин (с точностью до почти наверное). 
 В таком (гильбертовом) пространстве условное математическое ожидание можно определять как  решение задачи $$\mathbb{E}\left(Y\,|\,X\right) = \argmin_{\phi(\cdot) \in \mathcal{B}(\R)}\mathbb{E}\ag{|}Y - \phi(X)\ag{|}^2,$$ является просто проекцией элемента (случайной величины) $Y$ на подпространство всевозможных борелевских функций, зависящих только от $X$. В частности, из такого понимания условного математического ожидания сразу же следует формула полного математического ожидания:\footnote{Эта формула используется в курсе, например, при выводе характеристической функции сложного пуассоновского процесса.}
  $$\mathbb{E}\left(\mathbb{E}(Y\,|\,X)\right) = \langle 1,\mathbb{E}(Y|X)\rangle = \langle1,Y\rangle = \EE Y.$$ 
 Формула справедлива, поскольку тождественная единица 1 принадлежит подпространству всевозможных борелевских функций, зависящих только от $X$, следствием чего является $\langle Y-\mathbb{E}(Y|X),1 \rangle =0$.  
 Так вот в случае, когда $Y$ и $X$ -- это сечения гауссовского процесса, то вместо класса всевозможных борелевских функций достаточно использовать класс линейных функций\footnote{Коэффициенты $a$ и $b$ подбираются из условия ортогональности \ag{и несмещенности} 
\begin{center}
    $\langle Y - (aX+b),X \rangle=0$\ag{,  $\mathbb{E} Y = a \mathbb{E} X + b$.} 
\end{center}
 Отсюда (поскольку процесс гауссовский) будет следовать  независимость  $Y - (aX+$\linebreak $+b)$ и $X$, следовательно, и независимость любых функций от этих с.в., в частности, независимость с.в. $Y - (aX+b)$ и $\phi(X)$, где $\phi(x)$ -- любая борелевская функция. 
 Из последнего следует ортогональность: $Y - (aX+b)$ и $\phi(X)$, где $\phi(x)$ -- любая борелевская функция. Но это и означает, по определению, что  $aX+b$ -- проекция $Y$ на подпространство, порожденное борелевскими функциями от $X$, т.е. условное математическое ожидание: $\mathbb{E}(Y\,|\,X) = aX+b$.}   $\mathbb{E}(Y\,|\,X) = aX + b$ (решение всегда удается подобрать в этом, более узком, классе). Аналогично и в векторном случае.
   
Введенная выше корреляционная функция гауссовского процесса и математическое ожидание, очевидным образом, могут быть введены и для любого другого процесса второго порядка, т.е. процесса, у которого существует математическое ожидание от квадратов сечений. Свойства этих функций (непрерывность, дифференцируемость, интегрируемость) определяют аналогичные свойства (в смысле среднего квадратичного) соответствующего случайного процесса. Заметная часть пособия (по сложившейся при изложении данного курса на ФУПМ \mbox{МФТИ} традиции) посвящена  изучению данных свойств. Хотя важность для практики этой части курса в настоящий момент, пожалуй, наименьшая (по сравнению с другими частями курса), все же было решено оставить эту часть (особо не сокращая) ввиду хорошей возможности продемонстрировать с её помощью   элементы <<классического>> стохастического анализа случайных процессов, а также связи с курсом функционального анализа~\cite{Halmos}. 
В частности, доказательства почти всех фактов данного раздела базируются на одном простом наблюдении из функционального анализа:  скалярное произведение как функция своих аргументов непрерывна относительно нормы (топологии), порожденной этим же самым скалярным произведением. 

Важным свойством корреляционной функции $R(t_1,t_2)$ является ее неотрицательная определенность. Последнее означает, что для любого набора $0\le t_1<\dotsc<t_n$ матрица $\|R(t_i,t_j)\|_{i,j=1}^n$ -- неотрицательно определена\footnote{Если случайный процесс (слабо) стационарен, его корреляционная функция будет иметь вид $R(t_i,t_j) = \tilde{R}(t_i - t_j)$. Определение неотрицательной определенности естественным образом переносится и на такую корреляционную функцию (одного аргумента) $\tilde{R}(\tau)$. Оказывается (далее намеренно немного огрубляем формулировку теоремы Бохнера--Хинчина~\cite{BulinskyShiryaev2005, VentselAD, KoralovSinai},  
чтобы результат был понятнее), для того, чтобы $\tilde{R}(\tau)$ была неотрицательной определенной, необходимо и достаточно, чтобы она была преобразованием Фурье неотрицательной функции, которую называют спектральной плотностью. Эта функция (спектральная плотность) играет важную роль в анализе стационарных случайных процессов.}. Это следует из того, что\footnote{Для простоты считаем процесс $X(t)$ центрированным.} $R(t_1,t_2) = \langle X(t_1),X(t_2) \rangle$. Верно и обратное утверждение: любая неотрицательно определенная функция есть корреляционная функция некоторого случайного процесса второго порядка\footnote{Можно даже сузить класс процессов до гауссовских.}. В~действительности, это общий факт\footnote{Теорема о воспроизводящем ядре гильбертова пространства  (Reproducing Kernel Hilbert Space Theorem).}~\cite{BulinskyShiryaev2005, Halmos}, 
что для любого неотрицательно определенного ядра (функции двух аргументов) существует такое преобразование пространства аргументов этого ядра в некоторое гильбертово пространство\footnote{Это пространство (спрямляющее) играет важную роль в задачах обучения машины опорных векторов~\cite{Shalev-Shwartz} 
(Support Vector Machine (SVM)). Простейший линейный классификатор (в виде разделяющий гиперплоскости)  можно использовать для более сложной классификации. Для этого нужно исходную постановку задачи перевести в должным образом выбранное спрямляющее пространство (как правило, существенно большой размерности) и заметить, что для построения разделяющей гиперплоскости в этом пространстве достаточно уметь считать значения ядра на элементах исходного пространства (что предполагается постановкой задачи). Подчеркнем, что сложность такой задачи (квадратичной оптимизации) определяется объемом выборки, но не размерностью спрямляющего пространства\ag{.} 
}, что в этом новом пространстве значение ядра при любых двух заданных аргументах представляется скалярным произведением их образов. В простейшем случае, когда в некотором  конечномерном векторном пространстве выбран новый 
базис, то матрица Грама (матрица скалярных произведений новых базисных векторов) будет положительно определенной и представлять собой простейший пример возможного ядра с конечным набором значений аргументов (каждый аргумент может принимать число значений, равное размерности пространства). В данном случае по построению понятно, что матрица Грама имеет необходимое представление. В обратную сторону -- если задана положительно определенная матрица (ядро), то существует такой
базис, 
что  эта матрица будет матрицей Грама данного базиса (легко и конструктивно выводится из SVD разложения). Для бесконечномерных пространств последнее утверждение называется теоремой Мерсера.


Несколько примеров в данное пособие  были вставлены из замечательной книги~\cite{Sekey} 
(например, парадокс времени ожидания автобуса), которая уже на протяжении нескольких десятилетий является, пожалуй, основным источником парадоксальных задач к курсам стохастических дисциплин. \ag{Также много тонких контпримеров по теории вероятностей и теории случайных процессов можно найти в книге \cite{Stoyanov}.}


Цель данного заключения --  правильным образом структурировать пройденный материал, выделяя главные идеи и ранжируя результаты по важности.
Авторы надеются, что это может помочь в освоении курса, дополнительно мотивировав изучение ряда разделов.

\end{appendix}

 \addcontentsline{toc}{section}{Литература}

\renewcommand\refname{Литература}
\makeatletter
\renewcommand{\@biblabel}[1]{#1.}
\makeatother

\clearpage

\thispagestyle{empty}

{ \center \large  Учебное издание
	
	\vspace{21mm}

	{\small
		{\bfseries Гасников}~\,Александр\;Владимирович\\
		{\bfseries Горбунов}~\,Эдуард\;Александрович\\
		{\bfseries Гуз}~\,Сергей\;Анатольевич\\
		{\bfseries Черноусова}~\,Елена\;Олеговна\\
		{\bfseries Широбоков}~\,Максим\;Геннадьевич\\
		{\bfseries Шульгин}~\,Егор\;Владимирович\\
		
	}
	\vspace{15mm}
	
	{\large ЛЕКЦИИ\\[1pt] ПО СЛУЧАЙНЫМ\\[1pt]     ПРОЦЕССАМ\\
		[18pt]  Под редакцией А. В. Гасникова 
		
	}
	
}
\vfill

{\parindent=0mm \small
	
	\small {Редакторы: \emph{В.\,А.~Дружинина}, \emph{И.\,А.~Волкова}, \emph{О.\,П.~Котова}}
	
	\small {Корректор \emph{Н.\,Е.~Кобзева}}
	
	\small {Компьютерная верстка \emph{Н.\,Е.~Кобзева}}
	
	\small {Дизайн обложки  \emph{Е.\,А.~Казённова}}

	\vspace*{3mm}
	
	Подписано в печать 00.07.2019. Формат 60$\times$84\,\!$^{1}\!/\!_{16}$. 
	
	Усл. печ. л.  15{,}9. Уч.-изд. л. 13{,}7. Тираж 000~экз. Заказ №\,000.
	
	\vspace*{3mm}
	
Федеральное государственное автономное образовательное учреждение

высшего  образования <<Московский физико-технический институт  

(национальный исследовательский университет)>>

141700, Московская обл., г. Долгопрудный, Институтский пер., 9

Тел. (495) 408-58-22, e-mail: rio@mipt.ru

\rule[2pt]{\textwidth}{0.2pt}

Отпечатано в полном соответствии с предоставленным оригиналом-макетом

ООО <<Печатный салон ШАНС>>

127412, г. Москва, ул. Ижорская, д. 13, стр. 2

Тел. (495) 484-26-55

	
	\newpage
	
	\thispagestyle{empty}
	\begin{center}
		{\largeД\,л\,я\ \ з\,а\,м\,е\,т\,о\,к}
	\end{center}

}

\end{document}